\documentclass[11pt,reqno,oneside]{amsbook}

\usepackage{amsfonts,amssymb,amsmath,gensymb,epic,eepic,float,fullpage}
\usepackage{longtable}
\usepackage[usenames,dvipsnames]{color}
\usepackage{siunitx}
\usepackage{empheq}
\usepackage[mathscr]{euscript}
\usepackage[T1]{fontenc}
\usepackage{mathtools}
\usepackage{comment}
\usepackage{bm}
\usepackage{caption}
\usepackage{hyperref}
\usepackage{enumerate,enumitem}
\usepackage{tikz}
\usetikzlibrary{graphs,patterns,decorations.markings,arrows,matrix}
\usetikzlibrary{calc,decorations.pathmorphing,decorations.pathreplacing,shapes}
\allowdisplaybreaks

\newtheorem{thm}{Theorem}[section]
\newtheorem{lem}[thm]{Lemma}
\newtheorem{prop}[thm]{Proposition}

\newtheorem{cor}[thm]{Corollary}

\theoremstyle{definition}
\newtheorem{defn}[thm]{Definition}

\theoremstyle{remark}
\newtheorem{rmk}[thm]{Remark}

\theoremstyle{remark}
\newtheorem{ex}[thm]{Example}

\def\wh{\widehat}
\def\b{\bar}

\def\leq{\leqslant}
\def\geq{\geqslant}

\def\l{\lambda}
\def\m{\mu}
\def\n{\nu}

\def\ie{{\it i.e.}\/,}
\def\cf{{\it cf.}\/}

\colorlet{lgray}{white!85!black}
\colorlet{lred}{white!85!red}
\colorlet{lgreen}{white!80!green}
\colorlet{dgreen}{black!30!green}
\definecolor{green}{rgb}{0.1,0.8,0.1}
\definecolor{yellow}{rgb}{1.0,0.85,0.25}

\newcommand{\bra}[1]{\left\langle #1\right|}
\newcommand{\ket}[1]{\left|#1\right\rangle}

\renewcommand{\tikz}[2]{
\begin{tikzpicture}[scale=#1,baseline=(current bounding box.center),>=stealth]
#2
\end{tikzpicture}}

\renewcommand{\vert}[7]{
W_{#1,#2} \left(#3; 
\tikz{0.3}{
\draw[lgray,line width=1pt,->] (-1,0) -- (1,0);
\draw[lgray,line width=1pt,->] (0,-1) -- (0,1);
\node[left] at (-1,0) {\tiny $#5$};\node[right] at (1,0) {\tiny $#7$};
\node[below] at (0,-1) {\tiny $#4$};\node[above] at (0,1) {\tiny $#6$};
}\right)}

\def\V{\mathbb{V}(1^n)}
\def\Vs{\mathbb{V}^{*}(1^n)}
\def\B{\mathcal{B}}
\def\C{\mathcal{C}}
\def\F{{\sf F}}
\def\G{{\sf G}}

\def\I{\bm{I}}
\def\J{\bm{J}}
\def\K{\bm{K}}

\def\AA{\bm{A}}
\def\BB{\bm{B}}
\def\CC{\bm{C}}
\def\DD{\bm{D}}
\def\l{{\sf L}}
\def\m{{\sf M}}
\def\n{{\sf N}}
\newcommand{\As}[2]{\AA_{[#1,#2]}}
\newcommand{\Is}[2]{\I_{[#1,#2]}}
\newcommand{\Ks}[2]{\K_{[#1,#2]}}
\newcommand{\rotgr}{\rotatebox[origin=c]{-90}{$<$}}

\renewcommand\le{\leq}
\renewcommand\ge{\geq}

\numberwithin{equation}{section}

\makeindex

\begin{document}

\title{\Huge{Coloured stochastic vertex models \\ and their spectral theory}}

\author{
\LARGE{Alexei Borodin}
\footnote{\large{
Department of Mathematics, MIT, Cambridge, USA, and Institute for Information Transmission Problems, Moscow, Russia. E-mail: \texttt{borodin@math.mit.edu}
\medskip
}} 
and 
\LARGE{Michael Wheeler}
\footnote{\large{
School of Mathematics and Statistics, The University of Melbourne, Parkville, Victoria, Australia. E-mail: \texttt{wheelerm@unimelb.edu.au}
}}
}

\maketitle


{\large{\bf Abstract.}
This work is dedicated to $\mathfrak{sl}_{n+1}$-related integrable stochastic vertex models; we call such models \emph{coloured}. We prove several results about these models, which include the following: 

\medskip
\begin{enumerate}

\item We construct the basis of (rational) eigenfunctions of the coloured transfer-matrices as partition functions of our lattice models with certain boundary conditions. Similarly, we construct a dual basis and prove the corresponding orthogonality relations and Plancherel formulae. 

\medskip

\item We derive a variety of combinatorial properties of those eigenfunctions, such as branching rules, exchange relations under Hecke divided-difference operators, (skew) Cauchy identities of different types, and monomial expansions. 

\medskip

\item We show that our eigenfunctions are certain (non-obvious) reductions of the nested Bethe Ansatz eigenfunctions. 

\medskip

\item For models in a quadrant with domain-wall (or half-Bernoulli) boundary conditions, we prove a matching relation that identifies the distribution of the coloured height function at a point with the distribution of the height function along a line in an associated colour-blind ($\mathfrak{sl}_2$-related) stochastic vertex model. Thanks to a variety of known results about asymptotics of height functions of the colour-blind models, this implies a similar variety of limit theorems 
for the coloured height function of our models. 

\medskip

\item We demonstrate how the coloured-uncoloured match degenerates to the coloured (or multi-species) versions of the ASEP, $q$-PushTASEP, and the $q$-boson model. 

\medskip

\item We show how our eigenfunctions relate to non-symmetric Cherednik--Macdonald theory, and we make use of this connection to prove a probabilistic matching result by applying Cherednik--Dunkl operators to the corresponding non-symmetric Cauchy identity.

\end{enumerate}
}

\numberwithin{section}{chapter}
\numberwithin{subsection}{section}

\setcounter{tocdepth}{1}
\makeatletter
\def\l@subsection{\@tocline{2}{0pt}{2.5pc}{5pc}{}}
\makeatother
\tableofcontents

\chapter{Introduction}

\section{Preface}

Exactly solvable models of Statistical Mechanics is a very well developed 
subject with an illustrious history that spans Mathematics, Physics, and Chemistry. 
Its traditional goals have been analyzing thermodynamic equilibrium in various models of Statistical Mechanics, like in Onsager's 1944 solution of the two-dimensional Ising model \cite{Onsager}, see also Baxter's book \cite{Baxter}; and providing a convenient algebraic formalism for studying integrable systems in Quantum Mechanics in and out of equilibrium, \cf\ Jimbo--Miwa's book \cite{JimboM}.

A novel direction has been added in more recent years (although the pioneering work of Gwa--Spohn \cite{GwaS} goes back to 1992) -- applying the same solvability mechanisms to \emph{Markovian} systems, that can also be often viewed as models of Statistical Mechanics from an out-of-equilibrium perspective. Those include certain classes of interacting particles systems, with the Asymmetric Simple Exclusion Process, or ASEP, as a ubiquitous example, and directed polymers in random media, with the celebrated Kardar--Parisi--Zhang (KPZ) stochastic partial differential equation as a representative example. 

The first wave of these Markovian integrable systems started in late 1990s with the papers of Johansson \cite{Johansson} and Baik--Deift--Johansson \cite{BaikDJ}, and the key to their solvability, or \emph{integrability}, was in (highly non-obvious) reductions to what physicists would call \emph{free-fermion models} -- probabilistic systems, many of whose observables are expressed in terms of determinants and Pfaffians.\footnote{The two-dimensional Ising model mentioned above would also be called `free-fermion' in the physics literature, although it is `less solvable' than the models of \cite{Johansson, BaikDJ} and their relatives. Statistical mechanical models of similar level of free-fermion solvability are dimer models with explicitly known coupling functions.} 

The second wave of integrable Markovian systems started in late 2000s, and their reliance on the methods developed for integrable models of Statistical and Quantum Mechanics was much more apparent. For example, looking at the earlier papers of the second wave we see that: (a) The pioneering work of Tracy--Widom \cite{TracyW1,TracyW2,TracyW3} on the ASEP was based on the famous idea of Bethe \cite{Bethe} of looking for eigenfunctions of a quantum many-body system in the form of superposition of those for noninteracting bodies (coordinate Bethe Ansatz); (b) The work of O'Connell \cite{O'Connell} and Borodin--Corwin \cite{BorodinC} on semi-discrete Brownian polymers utilized properties of eigenfunctions of the Macdonald--Ruijsenaars quantum integrable system -- the celebrated Macdonald polynomials and their degenerations; (c) The physics papers of Dotsenko \cite{Dotsenko} and Calabrese--Le Doussal--Rosso \cite{CalabreseDR}, and a later work of Borodin--Corwin--Sasamoto \cite{BorodinCS} used a duality trick to show that certain observables of infinite-dimensional models solve finite-dimensional quantum many-body systems that are, in their turn, solvable by the coordinate Bethe Ansatz.

It turned out that all the above examples, as well as many others, can be united under a single umbrella -- integrable stochastic vertex models. 

Such a unification was first realized by Corwin--Petrov \cite{CorwinP} on the basis of \cite{Borodin} under the name of {stochastic higher spin six vertex model}, see \cite{BorodinP2} for a lecture style exposition. Its existence was due to the fact that all these models were governed by the same algebraic structure -- the quantum affine group $U_q(\widehat{\mathfrak{sl}_2})$. This was later extended to the level of the elliptic quantum group $E_{\tau,\eta}(\mathfrak{sl}_2)$ in \cite{Borodin3}, \cite{Aggarwal}, which produced \emph{dynamic} stochastic vertex models. 

The natural next step in the ladder is the quantum groups of higher rank, and stochastic vertex models corresponding to those have been introduced by Kuniba--Mangazeev--Maruyama--Okado in \cite{KunibaMMO}. In a certain degeneration, these models reproduce \emph{multi-species} exclusion processes that have been around at least since the 1990s. A dynamic extension was given by Kuniba in \cite{Kuniba}. 

Of course, one does not just want to define more and more general models; one wants to analyze their behaviour in various large time and space limits and put it in the framework of universality classes. 

For $\mathfrak{sl}_2$-related models, a few powerful approaches have been developed. Free-fermionic reductions work well for the models from the first wave and a few of those from the second wave, see \cite{BorodinG} for a survey of the former and \cite{Borodin2}, \cite{BorodinO}, \cite{BarraquandBCW} for the latter. Direct analysis of integral representations of the Markov kernels (a.k.a. transfer-matrices) is sometimes possible, for either the Markovian system itself, like in \cite{TracyW1,TracyW2,TracyW3}, or for certain \emph{duality functionals} (usually the $q$-moments) that evolve in time in a similar fashion, like in \cite{BorodinCS}, \cite{BorodinCG}. Both lead to exact characterization of the large time behaviour in numerous examples.  A Plancherel theory for Fourier-like bases of eigenfunctions of such Markov kernels has been developed in \cite{BorodinCPS}, \cite{BorodinCPS2}, \cite{BorodinP1}, and derivation of the $q$-moments from Cauchy (reproducing kernel) identities for such functions was the 
central topic of \cite{BorodinP1}, \cite{Borodin3}. The $q$-moments can also be obtained from the eigenaction of Macdonald difference operators on Macdonald symmetric polynomials \cite{BorodinC}, \cite{BorodinBW}. 

For the models related to $\mathfrak{sl}_{n+1}$ with $n\ge 2$, the progress has been much more modest. The only asymptotic result that we are aware of is a very recent announcement in \cite{ChenGHS} of a computation of the probability of two groups of particles of different species to completely change their order at large times (the result is also matched to an earlier 
prediction by Spohn \cite{Spohn}). Further, to our knowledge, the only algebraic advances towards possible asymptotics appeared in recent works of Kuan, where for certain models duality functionals have been constructed (see \cite{Kuan} and references therein) and integral representations for transfer-matrices have been derived \cite{Kuan2}; and in a paper by Takeyama \cite{Takeyama}, which contains a combinatorial formula for eigenfunctions of the Markov kernel for a multi-species $q$-boson model.

The primary goal of the present work is to advance the analysis of the $\mathfrak{sl}_{n+1}$-related
stochastic vertex models. We call such models \emph{coloured}, because they consist of paths of different colours (that correspond to different species, in more conventional terminology). We concentrate on the \emph{rainbow sector}, where the colours of all paths are pairwise distinct -- on one hand, it is simpler algebraically, and on the other hand, rainbow stochastic model collapse onto more 
degenerate ones by forgetting some of the colour distinctions. 

Here are our main results. 

\smallskip

\noindent$\bullet$ \quad We construct the basis of (rational) eigenfunctions of the coloured transfer-matrices as partition functions of our lattice models with certain boundary conditions. Similarly, we construct a dual basis and prove the corresponding orthogonality relations and Plancherel formulas. This yields an explicit integral representation of the transfer-matrices that, in particular, sheds some light on the nature of the integral representations obtained in \cite{ChenGHS}. 

\smallskip

\noindent$\bullet$\quad We derive a variety of combinatorial properties of these eigenfunctions, such as branching rules, exchange relations under Hecke 
divided-difference operators, (skew) Cauchy identities of different types, and monomial expansions. At the particular value $s=0$ of the spin parameter $s$, the eigenfunctions turn into non-symmetric Hall--Littlewood polynomials, and, consequently, we call them the \emph{non-symmetric spin Hall--Littlewood functions}. We also show that the non-symmetric spin Hall--Littlewood functions are certain (non-obvious) reductions of the nested Bethe Ansatz eigenfunctions, also known as the \emph{weight functions}. 

\smallskip

\noindent$\bullet$\quad For the coloured stochastic vertex model in a quadrant with a domain-wall (or half-Bernoulli) boundary condition, we prove a matching relation that identifies the distribution of the coloured height function at a point (that encodes the colours of the paths that pass through or below this point) with the distribution of the height function along a line in an associated colour-blind ($\mathfrak{sl}_2$-related) stochastic vertex model. Thanks to a variety of known (proved or conjectural) results about asymptotics of height functions of the colour-blind models, this implies a similar variety of (proved or conjectural) limit theorems 
for the coloured height function of our models. We also demonstrate how this coloured-uncoloured match degenerates to the coloured (or multi-species) versions of the ASEP, $q$-PushTASEP, and the $q$-boson model.

\smallskip

\noindent$\bullet$\quad Another matching relation that we prove identifies the one-point distribution of the coloured height function with the multi-point distribution of zeros of compositions distributed according to an ascending \emph{non-symmetric} Hall--Littlewood process. This is the first appearance of the 
non-symmetric Cherednik--Macdonald theory in a probabilistic setup, and we make use of it by proving the match by applying Cherednik--Dunkl operators to the corresponding non-symmetric Cauchy identity. 
  
Let us now describe our results in more detail. 

\section{The model} The vertex models that we consider in the present work assign weights to finite collections of finite paths drawn on a square grid. Each vertex for which there exists a path that enters and exits it produces a weight that depends on the configuration of all the paths that go through this vertex. 
The total weight for a collection of paths is the product of weights of the vertices that the paths traverse. (Thus, we tacitly assume the normalization in which the weight of an empty vertex is always equal to 1.)

Our paths are going to be \emph{coloured}, \emph{i.e.}, each path carries a colour that is a number between 1 and $n$, where $n\ge 1$ is a fixed parameter. 
We will usually assume that each horizontal edge of the underlying square grid can carry no more than one path, while vertical edges can be occupied by multiple paths. Thus, the states of the horizontal edges can be encoded by an integer between 0 and $n$, with 0 denoting an edge that is not occupied by a path, while the states of the vertical edges can be encoded by $n$-dimensional (nonnegative-valued) vectors which specify the number of times each colour $\{1,\dots,n\}$ appears at that edge. We will also mostly restrict ourselves to the situation when for each colour there is no more than one path of this colour in any path collection; in this case the vectors assigned to vertical edges will be length-$n$ binary strings.

Our paths will always travel upward in the vertical direction, and in the horizontal direction a path can travel rightward or leftward, depending on the region of the grid it is in; this choice will always be explicitly specified. 

Vertex weights in the regions of rightward travel are denoted as 
\begin{align}
\label{generic-L-intro}
L_{x,q,s}(\I,j; \K,\ell)
\equiv
L_x(\I,j; \K,\ell)
\index{L1@$L_x(\I,j; \K,\ell)$; vertex weights}
=
\tikz{0.7}{
	\draw[lgray,line width=1.5pt,->] (-1,0) -- (1,0);
	\draw[lgray,line width=4pt,->] (0,-1) -- (0,1);
	\node[left] at (-1,0) {\tiny $j$};\node[right] at (1,0) {\tiny $\ell$};
	\node[below] at (0,-1) {\tiny $\I$};\node[above] at (0,1) {\tiny $\K$};
},
\quad
0\le j,\ell\le n,
\quad
\I,\K \in \{0,1,2,\dots\}^n,
\end{align}
where the vectors $\I = (I_1,\dots,I_n)$, $\K = (K_1,\dots,K_n)$ are chosen such that $I_i$ (resp., $K_i$) gives the number of paths of colour $i$ present at the bottom (resp., top) edge of the vertex. The explicit values of the weights \eqref{generic-L-intro} are summarized by the table below:
\begin{align}
\label{s-weights-intro}
\begin{tabular}{|c|c|c|}
\hline
\quad
\tikz{0.7}{
	\draw[lgray,line width=1.5pt,->] (-1,0) -- (1,0);
	\draw[lgray,line width=4pt,->] (0,-1) -- (0,1);
	\node[left] at (-1,0) {\tiny $0$};\node[right] at (1,0) {\tiny $0$};
	\node[below] at (0,-1) {\tiny $\I$};\node[above] at (0,1) {\tiny $\I$};
}
\quad
&
\quad
\tikz{0.7}{
	\draw[lgray,line width=1.5pt,->] (-1,0) -- (1,0);
	\draw[lgray,line width=4pt,->] (0,-1) -- (0,1);
	\node[left] at (-1,0) {\tiny $i$};\node[right] at (1,0) {\tiny $i$};
	\node[below] at (0,-1) {\tiny $\I$};\node[above] at (0,1) {\tiny $\I$};
}
\quad
&
\quad
\tikz{0.7}{
	\draw[lgray,line width=1.5pt,->] (-1,0) -- (1,0);
	\draw[lgray,line width=4pt,->] (0,-1) -- (0,1);
	\node[left] at (-1,0) {\tiny $0$};\node[right] at (1,0) {\tiny $i$};
	\node[below] at (0,-1) {\tiny $\I$};\node[above] at (0,1) {\tiny $\I^{-}_i$};
}
\quad
\\[1.3cm]
\quad
$\dfrac{1-s x q^{\Is{1}{n}}}{1-sx}$
\quad
& 
\quad
$\dfrac{(x-sq^{I_i}) q^{\Is{i+1}{n}}}{1-sx}$
\quad
& 
\quad
$\dfrac{x(1-q^{I_i}) q^{\Is{i+1}{n}}}{1-sx}$
\quad
\\[0.7cm]
\hline
\quad
\tikz{0.7}{
	\draw[lgray,line width=1.5pt,->] (-1,0) -- (1,0);
	\draw[lgray,line width=4pt,->] (0,-1) -- (0,1);
	\node[left] at (-1,0) {\tiny $i$};\node[right] at (1,0) {\tiny $0$};
	\node[below] at (0,-1) {\tiny $\I$};\node[above] at (0,1) {\tiny $\I^{+}_i$};
}
\quad
&
\quad
\tikz{0.7}{
	\draw[lgray,line width=1.5pt,->] (-1,0) -- (1,0);
	\draw[lgray,line width=4pt,->] (0,-1) -- (0,1);
	\node[left] at (-1,0) {\tiny $i$};\node[right] at (1,0) {\tiny $j$};
	\node[below] at (0,-1) {\tiny $\I$};\node[above] at (0,1) 
	{\tiny $\I^{+-}_{ij}$};
}
\quad
&
\quad
\tikz{0.7}{
	\draw[lgray,line width=1.5pt,->] (-1,0) -- (1,0);
	\draw[lgray,line width=4pt,->] (0,-1) -- (0,1);
	\node[left] at (-1,0) {\tiny $j$};\node[right] at (1,0) {\tiny $i$};
	\node[below] at (0,-1) {\tiny $\I$};\node[above] at (0,1) {\tiny $\I^{+-}_{ji}$};
}
\quad
\\[1.3cm] 
\quad
$\dfrac{1-s^2 q^{\Is{1}{n}}}{1-sx}$
\quad
& 
\quad
$\dfrac{x(1-q^{I_j}) q^{\Is{j+1}{n}}}{1-sx}$
\quad
&
\quad
$\dfrac{s(1-q^{I_i})q^{\Is{i+1}{n}}}{1-sx}$
\quad
\\[0.7cm]
\hline
\end{tabular} 
\end{align}
where we assume that $1 \leq i<j \leq n$. Here $q$ \index{q@$q$; quantization parameter} is the \emph{quantization parameter}, $s$ \index{s@$s$; spin parameter} is the \emph{spin parameter}, and $x$ \index{x5@$x$; spectral parameter} is the \emph{spectral parameter}; the notation $\Is{k}{n}$ stands for 
$\sum_{a=k}^{n} I_a$. These weights correspond to the image of the universal $R$-matrix for the quantum affine group $U_q(\wh{\mathfrak{sl}_{n+1}})$ in the tensor product of its vector representation (horizontal edges) and a Verma module (vertical edges), with the spin parameter encoding its highest weight. At $s=q^{-\frac N2}$ with $N \in \mathbb{Z}_{\geq1}$, it is easy to see that the above weights will prevent the appearance of vertical edges occupied by more than $N$ paths; in representation theoretic language this corresponds to finite-dimensional irreducible quotients of the Verma modules that are equivalent to symmetric powers of the vector representation. In particular, $N=1$ forbids multiple occupation for vertical edges, and the above weights turn into a version of the $U_q(\wh{\mathfrak{sl}_{n+1}})$ $R$-matrix in its defining, or fundamental, representation.
 
The key property of the above weights is that they satisfy the \emph{Yang--Baxter equation}; we give a detailed account of it in Chapter \ref{sec:models}. 

A simple gauge transformation 
\begin{equation}
\label{gauge-intro}
\tilde{L}_x(\I,j;\K,\ell)
\index{L2@$\tilde{L}_x(\I,j;\K,\ell)$; stochastic weights}
:=
(-s)^{\mathbf{1}_{\ell\ge 1}}\cdot
L_x(\I,j;\K,\ell)
\end{equation}
makes the weights \eqref{generic-L-intro}, \eqref{s-weights-intro}   \emph{stochastic}, in the sense that for any fixed states of the incoming edges, the sum over all possible states of the outgoing edges is always equal to 1. Thus, if the parameters are chosen so that the modified weights are nonnegative, they can be viewed as Markovian transition probabilities. A stochastic normalization of the weights in the $\mathfrak{sl}_{n+1}$ case first appeared in \cite{KunibaMMO}; in Appendix \ref{app:kuan-weights}, we document the precise link between our notation and the one used in \cite[Section 3.5]{Kuan}. 

In the regions of leftward travel, we will use a different set of weights denoted as 
\begin{align}
\label{generic-M-intro}
M_{x,q,s}(\I,j;\K,\ell)
\equiv
M_x(\I,j;\K,\ell)
\index{M1@$M_x(\I,j;\K,\ell)$; dual vertex weights}
=
\tikz{0.7}{
	\draw[lgray,line width=1.5pt,<-] (-1,0) -- (1,0);
	\draw[lgray,line width=4pt,->] (0,-1) -- (0,1);
	\node[left] at (-1,0) {\tiny $\ell$};\node[right] at (1,0) {\tiny $j$};
	\node[below] at (0,-1) {\tiny $\I$};\node[above] at (0,1) {\tiny $\K$};
},
\quad
0\le j,\ell \le n,
\quad
\I,\K \in \{0,1,2,\dots\}^n,
\end{align}
and defined by 
\begin{align}
\label{sym1-intro}
M_{{x}^{-1},{q}^{-1},{s}^{-1}}(\I,j;\K,\ell)
&=
(-s)^{\bm{1}_{\ell \geq 1}-\bm{1}_{j \geq 1}}
\cdot
L_{x,q,s}(\I,j;\K,\ell).
\end{align}
The stochastic modification takes the form $\tilde{M}_x(\I,j;\K,\ell)
:=
(-s)^{-\mathbf{1}_{j\ge 1}}
\cdot M_x(\I,j;\K,\ell)$. \index{M2@$\tilde{M}_x(\I,j;\K,\ell)$; dual stochastic weights}

\section{The transfer-matrix and its eigenfunctions}
Consider the total weight (\emph{i.e.}, the partition function) for paths that start vertically in a prescribed configuration, end vertically in another prescribed configuration one row higher, and can move horizontally (leftward) in between. This can be illustrated pictorially as
\index{G@$G_{\mu/\nu}$; transfer matrix}
\begin{align}
\label{G-pf-intro}
G_{\mu/\nu}(x)
&=
\tikz{0.8}{
		\draw[lgray,line width=1.5pt,<-] (1,1) -- (8,1);
	\foreach\x in {2,...,7}{
		\draw[lgray,line width=4pt,->] (\x,0) -- (\x,2);
	}
	\node[left] at (0.5,1) {$x \leftarrow$};
	\node[above] at (7,2) {$\cdots$};
	\node[above] at (6,2) {$\cdots$};
	\node[above] at (5,2) {$\cdots$};
	\node[above] at (4,2) {\footnotesize$\bm{B}(2)$};
	\node[above] at (3,2) {\footnotesize$\bm{B}(1)$};
	\node[above] at (2,2) {\footnotesize$\bm{B}(0)$};
	\node[below] at (7,0) {$\cdots$};
	\node[below] at (6,0) {$\cdots$};
	\node[below] at (5,0) {$\cdots$};
	\node[below] at (4,0) {\footnotesize$\bm{A}(2)$};
	\node[below] at (3,0) {\footnotesize$\bm{A}(1)$};
	\node[below] at (2,0) {\footnotesize$\bm{A}(0)$};
	\node[right] at (8,1) {$0$};
	\node[left] at (1,1) {$0$};
}
\end{align}
Here $x$ is the spectral parameter used for the weights \eqref{generic-M-intro} in this picture, $\mu=(\mu_1,\dots,\mu_n)$ is a \emph{composition of length $n$}, or a string of nonnegative integers, that gives the starting positions for paths of colours $1,\dots,n$, respectively, $\nu=(\nu_1,\dots,\nu_n)$ is the composition that gives the final positions of the same paths, and the vectors $\bm{A}(k), \bm{B}(k)$ are used to encode $\mu$ and $\nu$ in the grid:
\begin{align}
\label{AB-states}
\bm{A}(k) = \sum_{j=1}^{n} \bm{1}_{\mu_j = k} \bm{e}_j,
\qquad
\bm{B}(k) = \sum_{j=1}^{n} \bm{1}_{\nu_j = k} \bm{e}_j,
\qquad 
\qquad k\in \mathbb{Z}_{\geq0}\,,
\end{align}
with $\bm{e}_j$ denoting the $j$-th Euclidean unit vector.

The partition function $G_{\mu/\nu}(x)$ is the \emph{transfer-matrix}. As \eqref{gauge-intro} and \eqref{sym1-intro} show, it is simply related to the matrix of transition probabilities for a Markov chain (provided that the parameters are chosen in such a way that all entries are nonnegative). 

Let us now describe our spectral representation of the transfer-matrix. 

For a composition $\mu$ and $n$ complex parameters $x_1,\dots, x_n$, consider another partition function $f_\mu(x_1,\dots,x_n)$ corresponding to the picture below:
\begin{align}
\label{f-def-intro}
f_{\mu}(x_1,\dots,x_n)
\index{f1@$f_{\mu}$; non-symmetric spin Hall--Littlewood}
&=
\tikz{0.8}{
	\foreach\y in {1,...,5}{
		\draw[lgray,line width=1.5pt,->] (1,\y) -- (8,\y);
	}
	\foreach\x in {2,...,7}{
		\draw[lgray,line width=4pt,->] (\x,0) -- (\x,6);
	}
	\node[left] at (0.5,1) {$x_1 \rightarrow$};
	\node[left] at (0.5,2) {$x_2 \rightarrow$};
	\node[left] at (0.5,3) {$\vdots$};
	\node[left] at (0.5,4) {$\vdots$};
	\node[left] at (0.5,5) {$x_n \rightarrow$};
	\node[below] at (7,0) {$\cdots$};
	\node[below] at (6,0) {$\cdots$};
	\node[below] at (5,0) {$\cdots$};
	\node[below] at (4,0) {\footnotesize$\bm{0}$};
	\node[below] at (3,0) {\footnotesize$\bm{0}$};
	\node[below] at (2,0) {\footnotesize$\bm{0}$};
	\node[above] at (7,6) {$\cdots$};
	\node[above] at (6,6) {$\cdots$};
	\node[above] at (5,6) {$\cdots$};
	\node[above] at (4,6) {\footnotesize$\bm{A}(2)$};
	\node[above] at (3,6) {\footnotesize$\bm{A}(1)$};
	\node[above] at (2,6) {\footnotesize$\bm{A}(0)$};
	\node[right] at (8,1) {$0$};
	\node[right] at (8,2) {$0$};
	\node[right] at (8,3) {$\vdots$};
	\node[right] at (8,4) {$\vdots$};
	\node[right] at (8,5) {$0$};
	\node[left] at (1,1) {$1$};
	\node[left] at (1,2) {$2$};
	\node[left] at (1,3) {$\vdots$};
	\node[left] at (1,4) {$\vdots$};
	\node[left] at (1,5) {$n$};
}
\end{align}
Here $\bm{A}(k)$'s encode $\mu$ in the same way as in \eqref{AB-states}, and the numbers $1,\dots,n$ next to rightward arrows on the left say that paths of colours $1,\dots,n$ enter through the left boundary in the 1st, \dots, $n$-th row, respectively (all horizontal travel is rightward). Further, the symbols $x_1,\dots,x_n$ say that we use those spectral parameters in the corresponding rows to compute the weights \eqref{generic-L-intro}. The zeros on the bottom and on the right mean that no paths enter and exit there. 

\begin{prop}\label{prop:fG-skew-intro} (Special case of Proposition \ref{prop:fG-skew})
Assume $x_1,\dots,x_n,y\in \mathbb{C}$ are such that
\begin{align}
\label{weight-condition2-intro}
\left|
\frac{x_i-s}{1-sx_i}
\cdot
\frac{y-s}{1-sy}
\right|
<
1,
\qquad
\text{for all}\ 1 \leq i \leq n.
\end{align}
Then for any composition $\nu = (\nu_1,\dots,\nu_n)$ of length $n$ one has the identity
	\begin{align}
	\label{fG-skew-intro}
	\sum_{\mu}
	f_{\mu}(x_1,\dots,x_n)
	G_{\mu / \nu}(y)
	=
	\prod_{i=1}^{n}
	\frac{1-q x_i y}{q(1-x_i y)}
	\cdot
	f_{\nu}(x_1,\dots,x_n),
	\end{align}
	where the summation is taken over all length-$n$ compositions 
	$\mu = (\mu_1,\dots,\mu_n)$.
\end{prop}

Equation \eqref{fG-skew-intro} shows that $f_\mu$'s are algebraic eigenfunctions of the transfer-matrix $G_{\mu/\nu}$. However, since we are in infinite-dimensional space (with basis parameterized by compositions of length $n$), we need some sort of spectral analysis to have any hope of using $f_\mu$'s for
studying the Markov evolution. This is precisely what we explain next. 

\section{Plancherel theory} We need to define functional spaces on which our Fourier-like transform with kernel $f_\mu(x_1,\dots,x_n)$ will act. For the clarity of exposition, we choose the smallest possible spaces; they could be extended by a natural completion procedure but we will not pursue this here. 

Let $\mathcal{C}^n$ \index{C@$\mathcal{C}^n$; finite composition space} denote the space of complex-valued, finitely supported functions on $\mathbb{Z}^n$, and let $\mathcal{L}^n$ \index{L@$\mathcal{L}^n$; Laurent polynomial space} denote the space of all functions $\Phi:\mathbb{C}^n\to \mathbb{C}$ such that (a) $\Phi(x_1,\dots,x_n)$ is a Laurent polynomial in the variables $(x_i-s)/(1-sx_i)$, $1\le i\le n$, and (b)
$\lim_{x_i\to\infty} \Phi(x_1,\dots,x_n)=0$ for all $1\le i\le n$. 

\index{g1@$g_\mu$; dual non-symmetric spin Hall--Littlewood}
We also need the dual functions $g_\mu(x_1,\dots,x_n)$, which can be defined as 
\begin{align}
\label{f-g-sym-intro}
g_{\tilde\mu}(x_n^{-1},\dots,x_1^{-1};q^{-1},s^{-1})
=
c_{\mu}(q,s)
\index{c@$c_{\mu}(q,s)$}
\prod_{i=1}^{n}
x_i
\cdot
f_{\mu}(x_1,\dots,x_n;q,s),\qquad \tilde{\mu} = (\mu_n,\dots,\mu_1),
\end{align}
where the multiplicative constant $c_{\mu}(q,s)$ is given by
\begin{align}
\label{cmu-intro}
c_{\mu}(q,s)
=
\frac{s^n (q-1)^n q^{{\rm inv}(\tilde\mu)}}{\prod_{j \geq 0} (s^2;q)_{m_j(\mu)}},
\qquad
{\rm inv}(\tilde\mu)
=
\#\{i<j : \tilde{\mu}_i \geq \tilde{\mu}_j\}
=
\#\{i<j : \mu_i \leq \mu_j\},
\end{align}
with the standard $q$-Pochhammer \index{$(a;q)_m$; $q$-Pochhammer symbol} definitions
\begin{align*}
(a;q)_m := (1-a)(1-qa) \cdots (1-q^{m-1}a), \quad m \geq 1,
\qquad
(a;q)_0 := 1,
\end{align*}
and with $m_j(\mu):=\#\{1 \leq k \leq n:\mu_k=j\}$ for all $j\ge 0$; as well as
their alternative normalization
\index{g1@$g^{*}_{\mu}$}
\begin{align*}
g^{*}_{\mu}(x_1,\dots,x_n)
:=
{q^{n(n+1)/2}}{(q-1)^{-n}}
\cdot
g_{\mu}(x_1,\dots,x_n).
\end{align*}
The functions $g_\mu(x_1,\dots,x_n)$ can also be constructed as suitable partition functions, similarly to \eqref{f-def-intro}, \cf\ \eqref{g-def} below. 

The functions $f_\mu$ and $g_\mu$ naturally extend to compositions $\mu$ with arbitrary (not necessarily nonnegative) parts via
\begin{align}
f_{\mu+(k,\dots,k)}(x_1,\dots,x_n)
&=
\prod_{i=1}^{n}
\left( \frac{x_i-s}{1-sx_i} \right)^k
f_{\mu}(x_1,\dots,x_n),\qquad k\in \mathbb{Z},
\end{align}
and similarly for $g_\mu$. One checks that $f_\mu,g_\mu\in \mathcal{L}^n$ for any $\mu\in\mathbb{Z}^n$. 

We can now define a forward transform $\mathfrak{G}:\mathcal{C}^n\to\mathcal{L}^n$ and an inverse transform $\mathfrak{F}:\mathcal{C}^n\to\mathcal{L}^n$ as 
\begin{gather*}
\mathfrak{G}[\alpha](x_1,\dots,x_n)
\index{G@$\mathfrak{G}$; forward transform}
=
\sum_{\mu \in \mathbb{Z}^n}
\alpha(\mu)
g^{*}_{\mu}(x_1,\dots,x_n),
\\
\mathfrak{F}[\Phi](\mu)
\index{F@$\mathfrak{F}$; inverse transform}
=
\left( \frac{1}{2\pi\sqrt{-1}} \right)^n
\oint_{C_1}
\frac{dx_1}{x_1}
\cdots 
\oint_{C_n}
\frac{dx_n}{x_n}
\prod_{1 \leq i<j \leq n}
\frac{x_j-x_i}{x_j-q x_i}
f_{\mu}(\b{x}_1,\dots,\b{x}_n)
\Phi(x_1,\dots,x_n),
\end{gather*}
where $\{C_1,\dots,C_n\}$ are closed, pairwise non-intersecting, positively oriented contours in the complex plane such that they all surround the point $s$, and the contours $C_i$ and $q \cdot C_i$ are both contained within contour $C_{i+1}$ for all $1 \leq i \leq n-1$, where $q \cdot C_i$ denotes the image of $C_i$ under multiplication by $q$.	An illustration of such contours is given in Figure \ref{fig:contours} below.

\begin{thm}\label{thm-plancherel-intro} (Theorem \ref{thm-plancherel} below)
	The maps $\mathfrak{F} \circ \mathfrak{G} : \mathcal{C}^n \rightarrow \mathcal{C}^n$ and $\mathfrak{G} \circ \mathfrak{F} : \mathcal{L}^n \rightarrow \mathcal{L}^n$ both act as the identity; we have
	\begin{align}
	\label{compose1-intro}
	\mathfrak{F} \circ \mathfrak{G} 
	=
	{\rm id}
	\in 
	{\rm End}(\mathcal{C}^n), \qquad \mathfrak{G} \circ \mathfrak{F} 
	=
	{\rm id}
	\in 
	{\rm End}(\mathcal{L}^n).
	\end{align}
\end{thm}
Unraveling the first of the relations \eqref{compose1-intro} shows that $\{f_\mu\}$ and $\{g^*_\nu\}$ form biorthonormal bases in $\mathcal{L}^n$, \cf\ Theorem \ref{thm:orthog} below. For versions of Theorem \ref{thm-plancherel-intro}
in the colour-blind ($\mathfrak{sl}_2$-related) case, see \cite{BorodinCPS}, \cite{BorodinCPS2}, \cite{BorodinP1} and references therein. 

As a corollary of Theorem \ref{thm-plancherel-intro}, one obtains a \emph{spectral decomposition} of the transfer-matrix:
\begin{equation}
\label{Gmunu-int-intro}
G_{\mu/\nu}(y)
=
\frac{q^{-n}}{(2\pi\sqrt{-1})^n}
\oint_{C_1}
\frac{dx_1}{x_1}
\cdots 
\oint_{C_n}
\frac{dx_n}{x_n}
\prod_{1 \leq i<j \leq n}
\frac{x_j-x_i}{x_j-q x_i}
\prod_{i=1}^{n}
\frac{x_i - q y}{x_i - y}
f_{\nu}(\b{x}_1,\dots,\b{x}_n)
g^{*}_{\mu}(x_1,\dots,x_n).
\end{equation}
For $\mu = (\mu_1 \geq \cdots \geq \mu_n)$ and $\nu = (\nu_1 \leq \cdots \leq \nu_n)$ this can be substantially simplified; see Remark \ref{rmk:exchange} below, showing a certain connection with the recent work \cite{ChenGHS}. 

While Theorem \ref{thm-plancherel-intro} is easy to state, it was certainly not easy to prove, and understanding structural properties of the functions $f_\mu$ is key. Let us summarize some of these properties. 

\section{Summation identities, recursive relations, monomial expansions} The functions $f_\mu$, $g_\mu$, and their skew variants (defined as partition functions of the form \eqref{f-def-intro}, but with possibly nonempty set of paths entering through the bottom boundary), satisfy a host of summation identities that can be found in Chapter \ref{sec:branching} below. One of those identities is \eqref{fG-skew-intro} above. Let us reproduce another one of them here, because of its importance for \eqref{compose1-intro}, and also because it is strikingly similar to an identity of Mimachi--Noumi for non-symmetric Macdonald polynomials \cite{MimachiN}.

\begin{thm}
	\label{thm:mimachi-intro} (Theorem \ref{thm:mimachi} below)
	Let $(x_1,\dots,x_n)$ and $(y_1,\dots,y_n)$ be two sets of complex parameters such that
	\begin{align}
	\left|
	\frac{x_i-s}{1-sx_i}
	\cdot
	\frac{y_j-s}{1-sy_j}
	\right|
	<
	1,
	\qquad
	\text{for all}\ 1 \leq i,j \leq n.
	\end{align}
	Then
	\begin{align}
	\label{mimachi-id-intro}
	\sum_{\mu}
	f_{\mu}(x_1,\dots,x_n)
	g^{*}_{\mu}(y_1,\dots,y_n)
	=
	\prod_{i=1}^{n}
	\frac{1}{1-x_i y_i}
	\prod_{n \geq i>j \geq 1}
	\frac{1-q x_i y_j}{1-x_i y_j}\,,
	\end{align}
	where the summation is over all compositions $\mu$ (with nonnegative coordinates).
\end{thm}

The similarity with non-symmetric Macdonald polynomials is not a coincidence -- we prove, in Theorem \ref{thm:f-E} below, that at the value $s=0$ of our spin parameter, the functions $f_\mu$ coincide with the non-symmetric Macdonald polynomials in the Hall--Littlewood specialization (Macdonald's $q$-parameter vanishes, and Macdonald's $t$-parameter coincides with our quantization parameter). Furthermore, colour-blindness results of Section \ref{ssec:colour-blind} easily show that if one sums $f_\mu$'s over all possible choices of colours of the top outgoing edges (equivalently, one can sum over all compositions $\mu$ of length $n$ whose entries, when ordered, give the same partition), one recovers the \emph{symmetric} spin Hall--Littlewood functions introduced in \cite{Borodin}; this is Proposition \ref{prop:f-sum-F} in the text.

Because of these coincidences, we call our $f_\mu$'s the \emph{non-symmetric spin Hall--Littlewood functions}. 

The connection to non-symmetric Macdonald theory was originally a surprise to us, largely because, to our best knowledge, the non-symmetric Macdonald polynomials (with any values of their parameters) were not known to have partition function representations as in \eqref{f-def-intro}. We establish the connection through the following two statements. 

The first one is the base for a recursion:
\begin{prop}\label{prop:f-delta-intro} (Proposition \ref{prop:f-delta} below)
	Let $\delta = (\delta_1 \leq \cdots \leq \delta_n)$ be an \emph{anti-dominant} composition. The corresponding non-symmetric spin Hall--Littlewood function $f_{\delta}$ is completely factorized:
	\begin{align}
	\label{f-delta-intro}
	f_{\delta}(x_1,\dots,x_n)
	\index{f3@$f_{\delta}$}
	=
	\frac{\prod_{j \geq 0} (s^2;q)_{m_j(\delta)}}{\prod_{i=1}^{n} (1-s x_i)}
	\prod_{i=1}^{n} \left( \frac{x_i-s}{1-sx_i} \right)^{\delta_i}, \quad  m_j(\delta)=\#\{k\in \{1,\dots,n\}:\delta_k=j\}, \  j\ge 0.
	\end{align}
\end{prop}

The second one is the recursion itself:
\begin{thm}\label{thm:hecke-f-intro} (a part of Theorem \ref{thm:hecke-f} below)
	Let $\mu=(\mu_1,\dots,\mu_n)$ be a length-$n$ composition with $\mu_i<\mu_{i+1}$ for some $1\le i\le n-1$. Then 
	\begin{equation}
	\label{T-f-intro}
	T_i \cdot f_{\mu}(x_1,\dots,x_n) = f_{(\mu_1,\dots,\mu_{i+1},\mu_i,\dots,\mu_n)}(x_1,\dots,x_n),
	\end{equation}
	where 
	 \begin{align}
	 \index{T@$T_i$; Hecke divided-difference operator}
	 T_i \equiv q - \frac{x_i-q x_{i+1}}{x_i-x_{i+1}} (1-\mathfrak{s}_i),
	 \quad
	 1 \leq i \leq n-1,
	 \end{align}
with elementary transpositions \index{s@$\mathfrak{s}_i$; elementary transposition} $\mathfrak{s}_i \cdot h(x_1,\dots,x_n)
:=h(x_1,\dots,x_{i+1},x_i,\dots,x_n)$, are the Demazure--Lusztig operators of the polynomial representation of the Hecke algebra of type $A_{n-1}$. 	 
\end{thm}

The combination of \eqref{f-delta-intro} and \eqref{T-f-intro} provides an algorithm for evaluating $f_\mu$'s, but does not yield a closed formula. 
We offer two rather different formulas for the $f_\mu$'s; both represent them as sums of factorized (monomial) expressions with certain coefficients. The appearance of factorized monomials in both expansions below follows from \eqref{f-delta-intro}. 

The first monomial expression that we give plays a key role in our proof of Theorem \ref{thm-plancherel-intro}.

\begin{thm} (combination of Theorem \ref{thm:fmu-fsigma}, Theorem \ref{thm:fsig-sym}, and Proposition \ref{prop:Z-sigma} below)
	Fix a composition $\mu$, and let $\delta = (\delta_1 \leq \cdots \leq \delta_n)$ and $\sigma \in \mathfrak{S}_n$ be an anti-dominant composition and a minimal-length permutation such that $\mu_i = \delta_{\sigma(i)}$, $1 \leq i \leq n$. Then 
\begin{multline}
f_\mu(x_1^{-1},\dots,x_n^{-1};q^{-1},s^{-1})=
s^{-n} q^{-{\rm inv}(\mu)} 
\prod_{i=1}^{n} x_i\\ \times
\sum_{\kappa \in \mathfrak{S}_n}
\prod_{1 \leq a < b \leq n}
\frac{x_{\kappa(b)} - q x_{\kappa(a)}}{x_{\kappa(b)} - x_{\kappa(a)}}\,
Z_{\kappa}^{\sigma}(x_1,\dots,x_n)
f_{\delta}(x_{\kappa(1)},\dots,x_{\kappa(n)};q,s),
\end{multline}	
where ${\rm inv}(\mu) = \#\{i < j : \mu_i \geq \mu_j\}$, and the coefficients 
$Z_{\kappa}^{\sigma}$ can be determined as follows.
For any permutation $\rho\in\mathfrak{S}_n$ with further notation $\tilde 
\rho(i)=n-\rho(i)+1$, $Z_{\tilde\rho}^{\sigma}(y_n,\dots,y_1)$ equals the 
partition function in a square region of size $n\times n$ with domain wall 
boundary conditions corresponding to the following picture:
	\begin{align}
	\label{Z-sigma-intro}
    Z_{\tilde\rho}^{\sigma}(y_n,\dots,y_1)
	:=
	\tikz{0.7}{
		\foreach\y in {1,...,5}{
			\draw[lgray,line width=1.5pt,->] (0,\y) -- (6,\y);
		}
		\foreach\x in {1,...,5}{
			\draw[lgray,line width=1.5pt,->] (\x,0) -- (\x,6);
		}
		\node[below] at (1,0) {\scriptsize $\sigma(n)$};
		\node[below] at (2,0) {\scriptsize $\cdots$};
		\node[below] at (3,0) {\scriptsize $\cdots$};
		\node[below] at (4,0) {\scriptsize $\sigma(2)$};
		\node[below] at (5,0) {\scriptsize $\sigma(1)$};
		\node[above] at (1,6) {\scriptsize $0$};
		\node[above] at (2,6) {\scriptsize $\cdots$};
		\node[above] at (3,6) {\scriptsize $\cdots$};
		\node[above] at (4,6) {\scriptsize $0$};
		\node[above] at (5,6) {\scriptsize $0$};
		\node[left] at (0,1) {\scriptsize $0$};
		\node[left] at (0,2) {\scriptsize $\vdots$};
		\node[left] at (0,3) {\scriptsize $\vdots$};
		\node[left] at (0,4) {\scriptsize $0$};
		\node[left] at (0,5) {\scriptsize $0$};
		\node[right] at (6,1) {\scriptsize $n$};
		\node[right] at (6,2) {\scriptsize $\vdots$};
		\node[right] at (6,3) {\scriptsize $\vdots$};
		\node[right] at (6,4) {\scriptsize $2$};
		\node[right] at (6,5) {\scriptsize $1$};
  	    \node[left] at (-1,1) {$y_{\rho(n)}$};
	    \node[left] at (-1,2) {$\vdots$};
	    \node[left] at (-1,3) {$\vdots$};
	    \node[left] at (-1,4) {$y_{\rho(2)}$};
	    \node[left] at (-1,5) {$y_{\rho(1)}$};
		\node[below] at (1,-1) {$y_1$};
		\node[below] at (2,-1) {$\cdots$};
		\node[below] at (3,-1) {$\cdots$};
		\node[below] at (4,-1) {$y_{n-1}$};
		\node[below] at (5,-1) {$y_n$};
	}
	\end{align}
	in which the $i$-th horizontal line (counted from the top) carries 
\emph{rapidity} $y_{\rho(i)}$, left external edge state $0$ and right external 
edge state $i$, while the $j$-th vertical line (counted from the left) carries 
\emph{rapidity} $y_j$, bottom external edge state $\sigma(n-j+1)$ and top 
external edge state $0$. Here no edge can be occupied by more than one path, with the vertex weights summarized by the table below:

\begin{align}
\label{col-vert-intro}
\begin{tabular}{|c|c|c|}
\hline
\quad
\tikz{0.6}{
	\draw[lgray,line width=1.5pt,->] (-1,0) -- (1,0);
	\draw[lgray,line width=1.5pt,->] (0,-1) -- (0,1);
	\node[left] at (-1,0) {\tiny $i$};\node[right] at (1,0) {\tiny $i$};
	\node[below] at (0,-1) {\tiny $i$};\node[above] at (0,1) {\tiny $i$};
}
\quad
&
\quad
\tikz{0.6}{
	\draw[lgray,line width=1.5pt,->] (-1,0) -- (1,0);
	\draw[lgray,line width=1.5pt,->] (0,-1) -- (0,1);
	\node[left] at (-1,0) {\tiny $i$};\node[right] at (1,0) {\tiny $i$};
	\node[below] at (0,-1) {\tiny $j$};\node[above] at (0,1) {\tiny $j$};
}
\quad
&
\quad
\tikz{0.6}{
	\draw[lgray,line width=1.5pt,->] (-1,0) -- (1,0);
	\draw[lgray,line width=1.5pt,->] (0,-1) -- (0,1);
	\node[left] at (-1,0) {\tiny $i$};\node[right] at (1,0) {\tiny $j$};
	\node[below] at (0,-1) {\tiny $j$};\node[above] at (0,1) {\tiny $i$};
}
\quad
\\[1.3cm]
\quad
$1$
\quad
& 
\quad
$\dfrac{q(1-z)}{1-qz}$
\quad
& 
\quad
$\dfrac{1-q}{1-qz}$
\quad
\\[0.7cm]
\hline
&
\quad
\tikz{0.6}{
	\draw[lgray,line width=1.5pt,->] (-1,0) -- (1,0);
	\draw[lgray,line width=1.5pt,->] (0,-1) -- (0,1);
	\node[left] at (-1,0) {\tiny $j$};\node[right] at (1,0) {\tiny $j$};
	\node[below] at (0,-1) {\tiny $i$};\node[above] at (0,1) {\tiny $i$};
}
\quad
&
\quad
\tikz{0.6}{
	\draw[lgray,line width=1.5pt,->] (-1,0) -- (1,0);
	\draw[lgray,line width=1.5pt,->] (0,-1) -- (0,1);
	\node[left] at (-1,0) {\tiny $j$};\node[right] at (1,0) {\tiny $i$};
	\node[below] at (0,-1) {\tiny $i$};\node[above] at (0,1) {\tiny $j$};
}
\quad
\\[1.3cm]
& 
\quad
$\dfrac{1-z}{1-qz}$
\quad
&
\quad
$\dfrac{(1-q)z}{1-qz}$
\quad 
\\[0.7cm]
\hline
\end{tabular}
\end{align}

\noindent where $0\le i<j\le n$, and the \emph{spectral parameter} $z$ is equal to the ratio of the rapidities on the vertical and horizontal lines that cross at the corresponding vertex. 
\end{thm} 

Our second monomial expansion has its origin in the nested Bethe Ansatz. We 
prove that the non-symmetric spin Hall--Littlewood functions are appropriate 
specializations of the off-shell nested Bethe vectors also known as the 
\emph{weight functions}. The specializations reduce the number of free variables 
from $n(n+1)/2$ to $n$. Combined with a known symmetrization formula for the 
nested Bethe vectors, see \cite{TarasovV} and references therein, we obtain the 
following statement. 

\begin{thm}\label{thm:f-mu-sym-intro}(Theorem \ref{thm:f-mu-sym} below)
	Fix a composition $\mu = (\mu_1,\dots,\mu_n)$ and let $\delta = (\delta_1 \leq \cdots \leq \delta_n)$ be its anti-dominant reordering. Define a vector $\gamma(\mu)$ via 
	\begin{align*}
	\gamma(\mu) 
	= 
	(\gamma_1,\gamma_2,\dots,\gamma_n)
	= 
	w_{\mu} \cdot (1,2,\dots,n),
	\end{align*}
	where $w_{\mu} \in \mathfrak{S}_n$ is the minimal-length permutation such that $w_{\mu} \cdot \mu = \delta$. Further, define a sequence of vectors of decreasing lengths by 
	\begin{align*}
	p^{(1)} = \gamma(\mu),
	\qquad
	p^{(i+1)} = p^{(i)} \backslash \{i\},
	\ \
	1 \leq i \leq n-1,
	\end{align*}
	as well as $n-1$ strictly increasing integer sequences $a^{(2)},\dots,a^{(n)}$ that, as sets, are given by $a^{(i)} = \{ 1 \leq b \leq n-i+2 : p^{(i-1)}_b \geq i \}$. The non-symmetric spin Hall--Littlewood functions are given by 	
	\begin{multline*}
	f_{\mu}(x_n,\dots,x_1)
	=
	\sum_{\sigma^{(1)} \in \mathfrak{S}_n}
	\cdots
	\sum_{\sigma^{(n-1)} \in \mathfrak{S}_2}
	f_{\delta}\Big( x_{\sigma^{(1)}(1)}, \dots, x_{\sigma^{(1)}(n)} \Big)
	\prod_{1 \leq i < j \leq n}
	\frac{q x_{\sigma^{(1)}(i)}-x_{\sigma^{(1)}(j)}}
	{x_{\sigma^{(1)}(i)}-x_{\sigma^{(1)}(j)}}
	\\ 	\times
	\prod_{b=2}^{n}
	\psi_{\left\{a^{(b)}_1,\dots,a^{(b)}_{n-b+1}\right\}}
	\Big(
	\sigma^{(b)} \cdot (x_1,\dots,x_{n-b+1}); 
	\sigma^{(b-1)} \cdot (x_1,\dots,x_{n-b+2})
	\Big)
	\prod_{1 \leq i < j \leq n-b+1}
	\frac{q x_{\sigma^{(b)}(i)}-x_{\sigma^{(b)}(j)}}
	{x_{\sigma^{(b)}(i)}-x_{\sigma^{(b)}(j)}}\,,
	\end{multline*}
	where by agreement $\sigma^{(n)}$ denotes the trivial permutation $\sigma^{(n)} = (1) \in \mathfrak{S}_1$, and 
	\begin{align*}
	\psi_{\{a_1,\dots,a_m\}}\left( x_1,\dots,x_m; y_1,\dots,y_M \right)
	=
	\prod_{i=1}^{m}
	\left[
	\frac{(1-q) y_{a_i}}{x_i - q y_{a_i}}
	\prod_{j=1}^{a_i-1}
	\frac{x_i-y_j}{x_i-q y_j} 
	\right],
	\qquad
	\forall\ 1 \leq m \leq M.
	\end{align*}
\end{thm}

While Theorem \ref{thm:f-mu-sym-intro} is not used in the rest of the work, we 
view it as an important bridge between what we do and the more traditional 
spectral approach to the higher rank vertex models in finite volume that 
proceeds through the nested Bethe Ansatz. 

\section{Matching distributions} While we fully expect the spectral representation \eqref{Gmunu-int-intro} and the non-symmetric spin Hall--Littlewood functions to be effective for asymptotic analysis of coloured vertex models, we do not attempt to do that in the present work. Instead, we focus on another approach that has been quite successful recently in the case of the colour-blind models, \cf\ \cite{Borodin2}, \cite{BorodinBW}, \cite{BarraquandBCW}. More concretely, we look for distributional matching of observables in different models. Surprisingly, we find a matching that allows one to extract probabilistic and asymptotic information for coloured models from that for the colour-blind ones. As the latter ones are well-studied, one can immediately carry over the known results and conjectures to the coloured situation, \cf\ Section \ref{ss:asymptotics-intro} below. 

Our first matching statement concerns the partition functions associated with the following pictures:

\medskip

\tikz{0.55}{
	\foreach\y in {1,...,5}{
		\draw[lgray,line width=1.5pt] (0,\y) -- (8,\y);
		\draw[line width=1.5pt,->] (0,\y) -- (1,\y);
	}
	\foreach\x in {1,...,7}{
		\draw[lgray,line width=1.5pt] (\x,0) -- (\x,6);
	}
	\foreach\x in {2,5,6}{
		\draw[line width=1.5pt,->] (\x,5) -- (\x,6);
	}
	\foreach\y in {2,5}{
		\draw[line width=1.5pt,->] (7,\y) -- (8,\y);
	}
	\node[above] at (6,6) {$I_k$};
	\node[above] at (5,6) {$<$};
	\node[above] at (4,6) {$\cdots$};
	\node[above] at (3,6) {$<$};
	\node[above] at (2,6) {$I_1$};
	\node[below] at (7,0) {$y_M$};
	\node[below] at (3.5,0) {$\cdots$};
	\node[below] at (4.5,0) {$\cdots$};
	\node[below] at (1,0) {$y_1$};
	\node[left] at (0,1) {$x_N$};
	\node[left] at (0,2.5) {$\vdots$};
	\node[left] at (0,4) {$x_2$};
	\node[left] at (0,5) {$x_1$};
	\node[right] at (8,2) {$J_{\ell}$};
	\node[right] at (8,2.8) {$\rotgr$};
	\node[right] at (8,3.5) {$\vdots$};
	\node[right] at (8,4) {$\rotgr$};
	\node[right] at (8,5) {$J_1$};
}
\qquad\quad 
\tikz{0.55}{
	\foreach\y in {1,...,5}{
		\draw[lgray,line width=1.5pt] (0,\y) -- (8,\y);
	}
	\draw[line width=1.5pt,red,->] (0,5) -- (1,5) node[midway,above] {\tiny $N$};
	\draw[line width=1.5pt,orange,->] (0,4) -- (1,4);
	\draw[line width=1.5pt,yellow,->] (0,3) -- (1,3);
	\draw[line width=1.5pt,green,->] (0,2) -- (1,2) node[midway,above] {\tiny $2$};
	\draw[line width=1.5pt,blue,->] (0,1) -- (1,1) node[midway,above] {\tiny $1$};
	\foreach\x in {1,...,7}{
		\draw[lgray,line width=1.5pt] (\x,0) -- (\x,6);
	}
	\draw[line width=1.5pt,yellow,->] (2,5) -- (2,6);
	\draw[line width=1.5pt,red,->] (6,5) -- (6,6);
	\draw[line width=1.5pt,green,->] (7,2) -- (8,2);
	\draw[line width=1.5pt,blue,->] (7,3) -- (8,3);
	\draw[line width=1.5pt,orange,->] (7,5) -- (8,5);
	\node[above] at (6,6) {$I_k$\}};
	\node[above] at (5,6) {$<$};
	\node[above] at (4,6) {$\cdots$};
	\node[above] at (3,6) {$<$};
	\node[above] at (2,6) {\{$I_1$};
	\node[above,text centered] at (0,6) {positions};
	\node[below] at (7,0) {$y_M$};
	\node[below] at (3.5,0) {$\cdots$};
	\node[below] at (4.5,0) {$\cdots$};
	\node[below] at (1,0) {$y_1$};
	\node[left] at (0,1) {$x_1$};
	\node[left] at (0,2) {$x_2$};
	\node[left] at (0,3.5) {$\vdots$};
	\node[left] at (0,5) {$x_N$};
	\draw [decorate,decoration={brace,amplitude=10pt},xshift=-4pt,yshift=0pt] (8.5,5) -- (8.5,1) 
	node [black,midway,xshift=2cm] {colours\,$\{J_1,\dots,J_{\ell}\}$};
}

\medskip

The picture on the left is for the partition function in the (stochastic) colour-blind model on an $M\times N$ rectangle with no more than one path on each edge, row rapidities $x_1,\dots,x_N$ (numbered top-to-bottom), column rapidities $y_1,\dots,y_M$ (numbered left-to-right), and  
vertex weights given by \eqref{col-vert-intro} with $n=1$ and with the spectral parameter $z$ equal to $x_iy_j$, where $x_i$ and $y_j$ are the rapidities of the vertex's row and column. The boundary conditions are specified by requiring that paths enter through every horizontal edge along the left boundary and exit through positions $\mathcal{I}=\{I_1<\dots<I_k\}$ on the top boundary and positions $\mathcal{J}=\{J_1<\dots<J_{\ell}\}$ on the right boundary. We denote this partition function by $\mathbb{P}_{\rm 6v}(\mathcal{I},\mathcal{J})$. \index{P1@$\mathbb{P}_{\rm 6v}(\mathcal{I},\mathcal{J})$}

The picture on the right is for the partition function in the (stochastic) coloured model on a similar rectangle with $n=N$ colours, no more than one path on each edge, row rapidities $x_1,\dots,x_N$ numbered \emph{bottom-to-top}, column rapidities $y_1,\dots,y_M$ (numbered left-to-right), and  
vertex weights given by \eqref{col-vert-intro} with $n=N$ and with the spectral parameter $z$ equal to $x_iy_j$, where $x_i$ and $y_j$ are the rapidities of the vertex's row and column. The boundary conditions are specified by requiring that a path of colour $i$, $1\le i\le N$, enters through the horizontal edge in row $i$ (with rapidity $x_i$) on the left boundary. We also fix the positions (but not the colours) of the paths that exit through the top boundary to be given by $\mathcal{I}$, exactly as for the colour-blind case, and we fix colours (but not the positions) of the paths that exit through the right boundary to be given by $\mathcal{J}=\{J_1<\dots<J_{\ell}\}$. We denote this partition function by $\mathbb{P}_{\rm col}(\mathcal{I},\mathcal{J})$. \index{P2@$\mathbb{P}_{\rm col}(\mathcal{I},\mathcal{J})$} See Figure \ref{fig:dawn} for a simulation of the coloured model.

\begin{figure}
\includegraphics[scale=0.5]{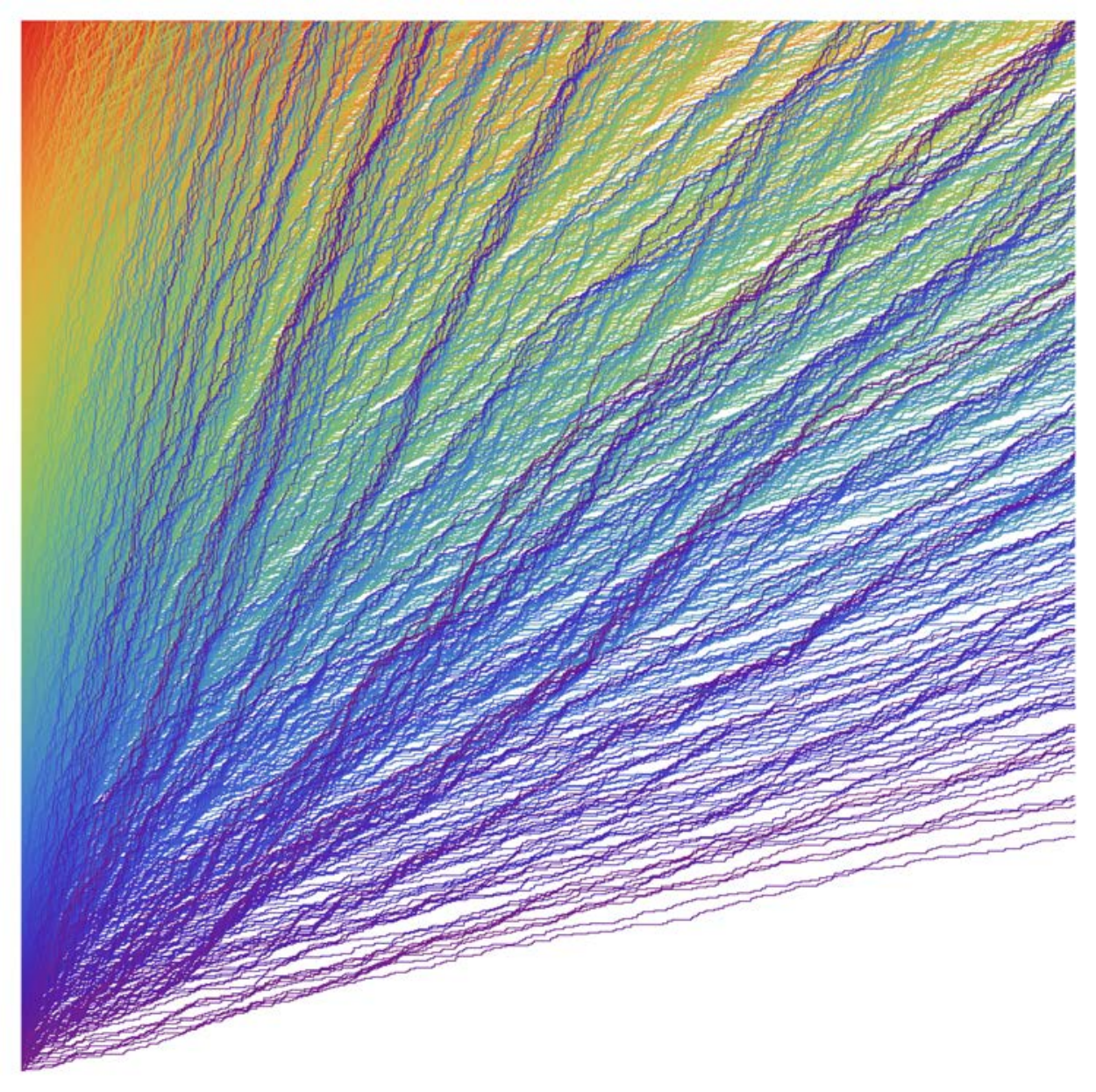}
\caption{A simulation of the stochastic coloured model in a quadrant with domain wall boundary conditions (courtesy of L.~Petrov).}
\label{fig:dawn}
\end{figure}  

\begin{thm}
	\label{thm:2=3-intro} (Theorem \ref{thm:2=3} below)
	For any integers $M,N \geq 1$, and two integer sets $\mathcal{I} = \{1 \leq I_1 < \cdots < I_k \leq M\}$ and 
	$\mathcal{J} = \{1 \leq J_1 < \cdots < J_{\ell} \leq N\}$, 
	the following equality of distributions holds:
	\begin{align}
	\label{2=3-intro }
	\mathbb{P}_{\rm 6v}(\mathcal{I},\mathcal{J})
	=
	\mathbb{P}_{\rm col}(\mathcal{I},\mathcal{J}).
	\end{align}
\end{thm}

This statement also allows for fusion, leading to multiple paths occupying vertical edges; see Theorem \ref{thm:fused2=3}.

A weaker version of the result \eqref{thm:2=3-intro} was previously obtained in \cite{FodaW}. This earlier result applied to the situation in which $M=N$, $\mathcal{I} = \{1,\dots,N\}$, and where the set $\mathcal{J}$ is empty, so that all coloured paths exit the partition function $\mathbb{P}_{\rm col}(\{1,\dots,N\},\varnothing)$ via its top boundary; see Remark \ref{rmk:fw}.

The weights \eqref{col-vert-intro} are \emph{stochastic}, which means that the partition function $\mathbb{P}_{\rm 6v}(\mathcal{I},\mathcal{J})$ can be viewed as the probability, for a Markovian process of propagating paths that entered through the left boundary, to exit the rectangle at prescribed locations. Similarly, the partition function $\mathbb{P}_{\rm col}(\mathcal{I},\mathcal{J})$
is the probability for similarly constructed Markovian \emph{coloured} paths to (a) exit through positions $\mathcal{I}$ on top of the rectangle, and (b) have the set of colours of the paths that exit through the right boundary equal to $\mathcal{J}$. 

The proof of Theorem \ref{thm:2=3} that we give is a non-trivial recursive argument (of the type broadly used in \cite{FodaW,WheelerZ}) that may not look particularly illuminating. As a matter of fact, the way we first encountered this matching went through a correspondence with a third partition function related to \emph{coloured Hall--Littlewood processes}; let us define it. 

Having two positive integers $M,N \geq 1$, a composition $\mu = (\mu_1,\dots,\mu_M)$ of length $M$, and a Gelfand--Tsetlin pattern (sequence of interlacing\footnote{We write \index{aN@$\nu\succ\kappa$; interlacing partitions} 
$\nu\succ\kappa$ or $\kappa\prec\nu$ if and only if the partitions $\nu$ and $\kappa$ interlace, \emph{i.e.}, $\nu_1\ge\kappa_1\ge\nu_2\ge\kappa_2\geq \cdots$.} partitions) 
$\bm{\lambda} = \left\{\lambda^{(1)}\succ\dots\succ\lambda^{(N)}=\varnothing\right\}$ of length $N-1$, associate to this pair of objects the weight 
\begin{align}
\label{asc-HL-intro}
\mathbb{W}_{M,N}
(\mu,\bm{\lambda})
=
E_{\tilde\mu}(y_M,\dots,y_1)
\index{E@$E_{\mu}$; non-symmetric Hall--Littlewood}
\cdot
Q_{\mu^{+}/\lambda^{(1)}}(x_1)
\cdot
\prod_{j=2}^{N}
Q_{\lambda^{(j-1)}/\lambda^{(j)}}(x_j)
\cdot
\prod_{i=1}^{N}
\prod_{j=1}^{M}
\frac{1-x_i y_j}{1-q x_i y_j}\,.
\end{align}
Here $Q$'s \index{Q@$Q_{\lambda/\nu}$; skew Hall--Littlewood} are symmetric (skew) Hall--Littlewood polynomials, $E$'s are non-symmetric Hall--Littlewood polynomials, $\mu^+$ \index{am@$\mu^+$; dominant ordering of $\mu$} is the nondecreasing ordering of $\mu$, and $(\tilde{\mu}_1,\dots,\tilde{\mu}_M)=(\mu_M,\dots,\mu_1)$.  

For $0\le x_i,y_j<1$, this defines a probability measure; the weights \eqref{asc-HL-intro} are nonnegative, and they add up to one:
\begin{align}
\label{Cauchy-summation}
\sum_{\mu}
\sum_{\bm{\lambda}}
\mathbb{W}_{M,N}
(\mu,\bm{\lambda})
\equiv 
1.
\end{align}
This summation follows from the fact that symmetrization of non-symmetric Hall--Littlewood polynomials yields the symmetric ones, and the Cauchy identity for the symmetric Hall--Littlewood polynomials. 

Define the {\it zero set} $z(\mu)$ \index{z6@$z(\mu)$} of the composition $(\mu_1,\dots,\mu_M)$ as
$z(\mu) = \{1 \leq i \leq M : \mu_i = 0\}$,
and a further set
\begin{align*}
\zeta(\mu,\bm{\lambda})
\index{af@$\zeta(\mu,\bm{\lambda})$}
=
\{1 \leq j \leq N: \ell(\lambda^{(j-1)}) - \ell(\lambda^{(j)}) = 0\},
\qquad
\lambda^{(0)} \equiv \mu^{+},
\end{align*}
which measures the differences between lengths of neighbouring partitions in the extended Gelfand--Tsetlin pattern $\mu^{+} \succ \lambda^{(1)} \succ \cdots \succ \lambda^{(N-1)} \succ \lambda^{(N)} = \varnothing$. Further, we define (with notation $\b{\mathcal{I}}=\{1,\dots,M\}\setminus\mathcal{I}$)
\begin{align*}
\index{P3@$\mathbb{P}_{\rm cHL}(\mathcal{I},\mathcal{J})$}
\mathbb{P}_{\rm cHL}(\mathcal{I},\mathcal{J})
=
\sum_{\mu}
\sum_{\bm{\lambda}}
\mathbb{W}_{M,N}
(\mu,\bm{\lambda})
\cdot
\bm{1}_{z(\mu)=\b{\mathcal{I}}}
\cdot
\bm{1}_{\zeta(\mu,\bm{\lambda}) = \mathcal{J}},
\end{align*}
which is the distribution of the random variables $z(\mu)$, $\zeta(\mu,\bm{\lambda})$ in the pair $(\mu,\bm{\lambda})$ sampled with respect to the coloured Hall--Littlewood process \eqref{asc-HL-intro}.

\begin{thm}
	\label{thm:2=1-intro} (Theorem \ref{thm:2=1} below)
	For any integers $M,N \geq 1$, and two integer sets $\mathcal{I} = \{1 \leq I_1 < \cdots < I_k \leq M\}$ and 
	$\mathcal{J} = \{1 \leq J_1 < \cdots < J_{\ell} \leq N\}$, the following equality of distributions holds:
	\begin{align}
	\label{2=1-intro}
	\mathbb{P}_{\rm 6v}(\mathcal{I},\mathcal{J})
	=
	\mathbb{P}_{\rm cHL}(\mathcal{I},\mathcal{J}).
	\end{align}
\end{thm}

Together with Theorem \ref{thm:2=3-intro}, this also implies 
\begin{align}
\label{3=1-intro}
\mathbb{P}_{\rm col}(\mathcal{I},\mathcal{J})
=
\mathbb{P}_{\rm cHL}(\mathcal{I},\mathcal{J}).
\end{align}

We give a proof of \eqref{2=1-intro} by a graphical argument involving the Yang--Baxter equation. There is also a parallel proof of \eqref{3=1-intro}
that uses graphical arguments and exchange relations similar to \eqref{T-f-intro}, but we leave it out of this paper. Instead, we offer a second proof of \eqref{2=1-intro} based on different ideas. 
More exactly, using the results of \cite{KirillovN}, we compute averages of observables of the coloured Hall--Littlewood process by applying Cherednik--Dunkl operators to a version of the Cauchy summation identity \eqref{Cauchy-summation}, match them to the corresponding observables of $\mathbb{P}_{\rm 6v}(\mathcal{I},\mathcal{J})$ (known thanks to \cite{BorodinP1,BorodinBW}), and prove that the resulting sets of observables are rich enough to identify the measures. We hope that this approach might be extendable to understanding joint distributions of colours of paths passing under multiple locations of the lattice (as opposed to the single vertex $(M,N)$ in the definition of $\mathbb{P}_{\rm col}(\mathcal{I},\mathcal{J})$).

\section{Matching for interacting particle systems} The distributional match of Theorems \ref{thm:2=3-intro} and \ref{thm:fused2=3} can be followed through various degenerations of the involved stochastic vertex models. In Chapter \ref{sec:reduction} we give detailed descriptions of how such degenerations work for three continuous time Markov chains -- the ASEP, a system of $q$-bosons, and an interacting particle system with bounded occupancy of each site that generalizes the system known as the PushTASEP (or long-range TASEP). Let us give a brief description of the ASEP result.

Consider a system of particles on the one-dimensional lattice $\mathbb{Z}$ with no more than a single particle per site. Assume that the particles have colours that are positive integers (several particles may have the same colour, and not all colours need to be utilized). The empty sites, that are also commonly called holes, can be naturally identified with particles of colour 0, in which case every site of the lattice is occupied by a particle of a nonnegative colour. 

The coloured ASEP is a continuous time Markov process on the space of such systems of particles. A (somewhat informal) description of the Markovian evolution is as follows: each particle is equipped with a left and a right exponential clock of rates $q$ and 1, respectively; all the clocks in the system are independent. When the left (resp., right) clock of a particle rings, it checks if the site immediately to the left (resp., right) of its current location is occupied by a particle of a smaller colour. If it is, then these two particles are swapped, and if not, then nothing happens. 

If all particles of the system are of a single colour (not counting the holes), then this evolution reduces to the usual uncoloured ASEP.
Note that the uncoloured ASEP is also the projection of the coloured one when the distinctions between the colours (apart from 0) are being ignored. More generally, one can always reduce the number of colours in the coloured ASEP by ignoring distinctions between colours in any interval. If particles of only two nonzero colours are present in the system, then the particles of the smaller colour behave as \emph{second class particles}, in the conventional ASEP terminology. 

We will be concerned with the \emph{(half-)Bernoulli} initial condition (at time $t=0$) defined as follows. Choose $p\in (0,1]$, and for each $i=1,2,\dots$, place a particle of colour $i$ at location $(-i)$ with probability $p$, independently over all $i$. 
For $p=1$ all the sites $-1,-2,\dots$ are going to be occupied by particles of colours $1,2,\dots$, respectively. We refer to the latter case as the \emph{step} initial condition. 

Let us now state the matching result. 

Fix an arbitrary integer $P$ (reference position) and two sets of pairwise distinct integers $\mathcal{I} = \{I_1 < \cdots < I_k \leq P\}$ and $\mathcal{J} = \{1 \leq J_1 < \cdots < J_{\ell} \}$. Consider the following probabilities:
\begin{multline*}
\mathbb{P}^{\rm ASEP}(\mathcal{I},\mathcal{J};P,t)
\index{P4@$\mathbb{P}^{\rm ASEP}(\mathcal{I},\mathcal{J};P,t)$}
\\
=
{\rm Prob}\Big\{\text{there is a particle at each of the locations } I_1,\dots,I_k;P+J_1,\dots,P+J_\ell\Big\},
\end{multline*}
\begin{multline*}
\mathbb{P}^{\rm mASEP}(\mathcal{I},\mathcal{J};P,t)
\index{P5@$\mathbb{P}^{\rm mASEP}(\mathcal{I},\mathcal{J};P,t)$}
\\
={\rm Prob}\Big\{\begin{matrix}\text{there is a particle at each of the locations }I_1,\dots,I_k,\text{ and }\\
\text{ there is a particle of each of the colours } J_1,\dots,J_\ell \text{ to the right of location } P 
 \end{matrix}\Big\}.
\end{multline*}

\begin{thm}\label{thm:ASEP-match-intro} (Theorem \ref{thm:ASEP-match} below)
	Consider the coloured ASEP with the half-Bernoulli initial condition, as defined above. Then
	for all time $t \geq 0$, any integer $P \in \mathbb{Z}$, and arbitrary integer sets $\mathcal{I} = \{I_1 < \cdots < I_k \leq P\}$ and $\mathcal{J} = \{1 \leq J_1 < \cdots < J_{\ell} \}$, we have
	\begin{align*}
	\mathbb{P}^{\rm ASEP}(\mathcal{I},\mathcal{J};P,t)
	=
	\mathbb{P}^{\rm mASEP}(\mathcal{I},\mathcal{J};P,t).
	\end{align*}
\end{thm}

One can show that in the case of step initial condition ($p=1$) and $\mathcal{I}=\varnothing$, Theorem \ref{thm:ASEP-match-intro} also follows from \cite[Theorem 1.4]{AAV}. This does not seem to extend to arbitrary $p$ and $\mathcal{I}$, however.   

Matching results for the coloured $q$-bosons and the PushTASEP-like coloured system that are somewhat similar to Theorem \ref{thm:ASEP-match-intro}, and that are also derived from the vertex model matching of Theorems \ref{thm:2=3-intro} and \ref{thm:fused2=3}, are given in Sections \ref{sec:bosons}--\ref{sec:push} below (see Theorems \ref{thm:12.5} and \ref{thm:push-match}, respectively).

\section{Asymptotics}\label{ss:asymptotics-intro} The matching of Theorem \ref{thm:2=3-intro} and its extensions
gives rise to a host of asymptotic statements for the coloured models via the   corresponding asymptotics of the colour-blind ones. Let us survey the latter ones very quickly. 

It is convenient to speak in terms of the \emph{height functions} that count the number of paths/particles to the right (or left) of a given location in the two-dimensional space time. For coloured models, one can speak of a \emph{coloured} height function that counts the number of paths/particles to the right of a specified location, but only takes into account paths/particles whose colour does not exceed a certain cut-off. This turns the height function into a random function of \emph{three} variables rather than two. The matching statements thus provide identities between
the distributions of the coloured height functions on certain one-dimensional sections of the three-dimensional physical space, and those of the uncoloured height functions on one-dimensional sections of the two-dimensional physical space. 

The large time/space/colour asymptotics for the models considered in the present work can be of (at least) four different types; in terms of the colour-blind models, they correspond to: 

\begin{enumerate}

\item A local limit regime with observed space locations at finite distance from each other and the limiting object typically being a product of Bernoulli measures or their generalization (local equilibrium). 

\item Global Gaussian random field asymptotics with observation points ranging over the whole active region of the physical space, in the case of symmetric models ($q=1$), or weakly asymmetric ones ($q\to 1$) with relatively small, but still diverging scaling of space/time. 

\item In the weakly asymmetric case, intermediate space/time scaling and appropriate mesoscopic distances between observation points yield limiting random polymers in random media, \emph{e.g.}, the KPZ equation. 

\item In the asymmetric case ($q\ne 1$ is fixed), or the weakly asymmetric case at large scales, mesoscopic distances between observation points produce a broadly universal limiting object known as the \emph{KPZ fixed point}. 

\end{enumerate}

Note that in order to use the matching for deriving a non-trivial coloured result from an uncoloured one, one needs to know \emph{process level} convergence of the uncoloured height function as the one-point distributions match trivially. 
This substantially reduces the number of rigorously known results; there are, however, widely believed conjectures and heuristic arguments wherever the one-point convergence is known.

For (1), the very first result is the behaviour of the particle of colour 1 in the coloured TASEP, and our matching result says that it is asymptotically equivalent to the colour-blind TASEP density. This is a celebrated result of Ferrari--Kipnis \cite{FK} about the second class particle in a rarefaction fan. Its extensions to many colours for the ASEP with step initial condition have been investigated by Amir--Angel--Valko \cite{AAV} via a matching result that is not far from (the ASEP degeneration of) ours. For a single second class particle, a recent work by Bal\'azs--Nagy \cite{BN} proves results for more general (continuous time) processes, again by a suitable matching. 

For (2), the asymptotic behaviour of the uncoloured height function is a deterministic law of large numbers (hydrodynamic limit) given by a first order PDE, and Gaussian fluctuations around it are given by a solution of a stochastic PDE that is a classical second order PDE with a white noise inhomogeneity, \cf\ De Masi--Presutti--Scacciatelli \cite{MPS}, Dittrich--G\"artner \cite{DG}, Borodin--Gorin \cite{BG-telegraph}, Shen--Tsai \cite{ShenT} and references therein. No coloured results or conjectures have been previously known. 

For (3), the uncoloured height function converges to the logarithm of the partition function for a finite temperature directed polymer in a Gaussian random medium, see, Bertini--Giacomin \cite{BG} and \cite{ACQ} for the ASEP and the continuum polymer (equivalently, the KPZ equation), Borodin--Corwin \cite{BorodinC} for the $q$-TASEP and the O'Connell--Yor (semi-discrete) polymer, Corwin--Tsai \cite{CT} and Corwin--Ghosal--Shen--Tsai \cite{CGST} for the stochastic six vertex and the continuum polymer. No coloured results or conjectures have been previously known. 

For (4), only for the TASEP and PushTASEP (as well as for a few closely related, \emph{determinantal} models) process level convergence results have been rigorously established, \cf\ Borodin--Ferrari
\cite{BorodinFerrari}, Matetski--Quastel--Remenik \cite{MQR}; conjectures, however, are broadly available, \cf\ Corwin \cite{Corwin-survey}, Spohn \cite{SpohnKPZ}, Corwin--Dimitrov \cite{CD}. No coloured results have been previously known. The only available coloured predictions have been made by Spohn \cite{Spohn}, but it remains unclear whether they are applicable to any of the models covered by our matchings. 

It would be extremely interesting to access any of the outlined asymptotic results without employing the matching, but rather through utilizing the spectral analysis of the transfer-matrices described above. We leave this to a future work. 

\section{A word about extensions} The algebraic formalism of the present work should be readily extendable to: 

\begin{itemize}
	
	\item Inhomogeneous lattice models, similarly to what was done in \cite{BorodinP1} in the colour-blind case;
	
	\item The spin $q$-Whittaker functions, similarly to what was done in 
	\cite{BorodinWheeler} in the colour-blind case;
	
	\item Elliptic and trigonometric IRF (Interaction-Round-a-Face) or SOS (Solid-On-Solid) models, similarly to what was done in \cite{Borodin3}, \cite{Aggarwal} in the colour-blind case. 
	
\end{itemize}

We chose to leave these more general scenarios out of the scope of this paper that, as it is, ended up pretty long.

\section{Acknowledgments} A.~B. was
partially supported by the NSF grant DMS-1607901 and DMS-1664619. M.~W. was partially supported by the ARC grant DE160100958.

\chapter{Rank-$n$ vertex models}
\label{sec:models}

\section{Stochastic $U_q(\wh{\mathfrak{sl}_{n+1}})$ $R$-matrix\index{R@$R$-matrix}}
\label{ssec:fundamental}
 
The higher-spin vertex model that we consider in this paper can be obtained from the $U_q(\wh{\mathfrak{sl}_{n+1}})$ $R$-matrix under the fusion procedure. To begin, let us recall the form of this $R$-matrix \cite{Jimbo}. The $U_q(\wh{\mathfrak{sl}_{n+1}})$ $R$-matrix acts in a tensor product $W_a \otimes W_b$ of two $(n+1)$-dimensional vector spaces, and takes the form
\begin{align}
\label{Rmat}
R_{ab}(z)
&=
\sum_{i=0}^{n}
\left(
R_z(i,i;i,i)
E^{(ii)}_a \otimes E^{(ii)}_b
\right)
\\
\nonumber
&+
\sum_{0 \leq i < j \leq n}
\left(
R_z(j,i;j,i)
E^{(ii)}_a \otimes E^{(jj)}_b
+
R_z(i,j;i,j)
E^{(jj)}_a \otimes E^{(ii)}_b
\right)
\\
\nonumber
&+
\sum_{0 \leq i < j \leq n}
\left(
R_z(j,i;i,j)
E^{(ij)}_a \otimes E^{(ji)}_b
+
R_z(i,j;j,i)
E^{(ji)}_a \otimes E^{(ij)}_b
\right)
\end{align}
where $E^{(ij)}_c \in {\rm End}(W_c)$ denotes the $(n+1) \times (n+1)$ elementary matrix with a $1$ at position $(i,j)$ and $0$ everywhere else, acting in $W_c \cong \mathbb{C}^{n+1}$. The matrix entries are rational functions of the spectral parameter $z$ and the quantization parameter $q$; they are given by
\index{R@$R_z(i,j;k,\ell)$; $R$-matrix entries}
\begin{align}
\label{R-weights-a}
&
\left.
R_z(i,i;i,i)
=
1,
\quad
i \in \{0,1,\dots,n\},
\right.
\\
\nonumber
\\
\label{R-weights-bc}
&
\left.
\begin{array}{ll}
R_z(j,i;j,i)
=
\dfrac{q(1-z)}{1-qz},
&
\quad
R_z(i,j;i,j)
=
\dfrac{1-z}{1-qz}
\\ \\
R_z(j,i;i,j)
=
\dfrac{1-q}{1-qz},
&
\quad
R_z(i,j;j,i)
=
\dfrac{(1-q)z}{1-qz}
\end{array}
\right\}
\quad
i,j \in \{0,1,\dots,n\},
\quad  i<j.
\end{align}
All other matrix entries $R_z(i,j;k,\ell)$ which do not fall into a category listed above are by definition equal to $0$. The model described above differs slightly from the one listed in \cite{Jimbo}, since its entries $R_{z}(j,i;j,i)$ and $R_{z}(i,j;i,j)$ are not symmetric for $i \not= j$. The asymmetric form that we use preserves the integrability of the model, and makes it stochastic: 
\begin{prop}
\label{prop:YB}
The $R$-matrix \eqref{Rmat} satisfies the Yang--Baxter equation and unitarity relations
\begin{align}
\label{YB}
&
R_{ab}(y/x) R_{ac}(z/x) R_{bc}(z/y)
=
R_{bc}(z/y) R_{ac}(z/x) R_{ab}(y/x),
\\
\label{unitarity}
&
R_{ab}(y/x) R_{ba}(x/y) = 1,
\end{align}
which hold as identities in ${\rm End}(W_a \otimes W_b \otimes W_c)$ and ${\rm End}(W_a \otimes W_b)$, respectively.
\end{prop}

\begin{prop}
\label{prop:Rstoch}
For any fixed $i,j \in \{0,1,\dots,n\}$ there holds
\begin{align}
\label{stoch}
\sum_{k=0}^{n}
\sum_{\ell=0}^{n}
R_z(i,j;k,\ell)
=
1.
\end{align}
Equivalently, all rows of the matrix \eqref{Rmat} sum to $1$.
\end{prop}

We shall denote the entries of the $R$-matrix pictorially using vertices. A vertex is the intersection of an oriented horizontal and vertical line, with a state variable $i \in \{0,1,\dots,n\}$ assigned to each of the connected horizontal and vertical line segments. The $R$-matrix entries are identified with such vertices as shown below: 
\begin{align}
\label{R-vert}
R_z(i,j; k,\ell)
=
\tikz{0.7}{
\draw[densely dotted] (0.5,0) arc (0:90:0.5);
\draw[lgray,line width=1.5pt,->] (-1,0) -- (1,0);
\draw[lgray,line width=1.5pt,->] (0,-1) -- (0,1);
\node[left] at (-1,0) {\tiny $j$};\node[right] at (1,0) {\tiny $\ell$};
\node[below] at (0,-1) {\tiny $i$};\node[above] at (0,1) {\tiny $k$};
},
\quad
i,j,k,\ell \in \{0,1,\dots,n\},
\end{align}
where the marked angle assumes the value of the spectral parameter\footnote{We will use both of the terms {\it spectral parameter} and {\it rapidity} in this work, but for slightly different purposes. The variable attached to a lattice line will be termed ``rapidity'' whereas the argument of an $R$-matrix, which is the ratio of the two rapidities passing through that vertex, will be termed ``spectral parameter''.}, \ie\ it is equal to $z$. One can interpret the above figure as the propagation of coloured lattice paths through a vertex: each edge label $i \geq 1$ represents a coloured path superimposed over that edge, while the case $i=0$ indicates that no path is present. The {\it incoming} paths are those situated at the left and bottom edges of the vertex; those at the right and top are called {\it outgoing}. The weight of the vertex, $R_z(i,j; k,\ell)$, vanishes identically unless the total flux of colours through the vertex is preserved, \ie\ unless the ensemble of incoming colours is the same as the ensemble of outgoing colours. This gives rise to five categories of vertices:
\begin{align}
\label{fund-vert}
\begin{tabular}{|c|c|c|}
\hline
\quad
\tikz{0.6}{
	\draw[lgray,line width=1.5pt,->] (-1,0) -- (1,0);
	\draw[lgray,line width=1.5pt,->] (0,-1) -- (0,1);
	\node[left] at (-1,0) {\tiny $i$};\node[right] at (1,0) {\tiny $i$};
	\node[below] at (0,-1) {\tiny $i$};\node[above] at (0,1) {\tiny $i$};
}
\quad
&
\quad
\tikz{0.6}{
	\draw[lgray,line width=1.5pt,->] (-1,0) -- (1,0);
	\draw[lgray,line width=1.5pt,->] (0,-1) -- (0,1);
	\node[left] at (-1,0) {\tiny $i$};\node[right] at (1,0) {\tiny $i$};
	\node[below] at (0,-1) {\tiny $j$};\node[above] at (0,1) {\tiny $j$};
}
\quad
&
\quad
\tikz{0.6}{
	\draw[lgray,line width=1.5pt,->] (-1,0) -- (1,0);
	\draw[lgray,line width=1.5pt,->] (0,-1) -- (0,1);
	\node[left] at (-1,0) {\tiny $i$};\node[right] at (1,0) {\tiny $j$};
	\node[below] at (0,-1) {\tiny $j$};\node[above] at (0,1) {\tiny $i$};
}
\quad
\\[1.3cm]
\quad
$1$
\quad
& 
\quad
$\dfrac{q(1-z)}{1-qz}$
\quad
& 
\quad
$\dfrac{1-q}{1-qz}$
\quad
\\[0.7cm]
\hline
&
\quad
\tikz{0.6}{
	\draw[lgray,line width=1.5pt,->] (-1,0) -- (1,0);
	\draw[lgray,line width=1.5pt,->] (0,-1) -- (0,1);
	\node[left] at (-1,0) {\tiny $j$};\node[right] at (1,0) {\tiny $j$};
	\node[below] at (0,-1) {\tiny $i$};\node[above] at (0,1) {\tiny $i$};
}
\quad
&
\quad
\tikz{0.6}{
	\draw[lgray,line width=1.5pt,->] (-1,0) -- (1,0);
	\draw[lgray,line width=1.5pt,->] (0,-1) -- (0,1);
	\node[left] at (-1,0) {\tiny $j$};\node[right] at (1,0) {\tiny $i$};
	\node[below] at (0,-1) {\tiny $i$};\node[above] at (0,1) {\tiny $j$};
}
\quad
\\[1.3cm]
& 
\quad
$\dfrac{1-z}{1-qz}$
\quad
&
\quad
$\dfrac{(1-q)z}{1-qz}$
\quad 
\\[0.7cm]
\hline
\end{tabular}
\end{align}
where we assume that $0 \leq i < j \leq n$. These are the pictorial representations of the five types of weights in \eqref{R-weights-a}, \eqref{R-weights-bc}.

Having set up these vertex notations, the relations of Propositions \ref{prop:YB} and \ref{prop:Rstoch} then have simple graphical interpretations. The Yang--Baxter equation \eqref{YB} becomes
\begin{align*}
\sum_{0 \leq k_1,k_2,k_3 \leq n}
\tikz{0.8}{
\draw[lgray,line width=1.5pt,->]
(-2,1) node[above,scale=0.6] {\color{black} $i_1$} -- (-1,0) node[below,scale=0.6] {\color{black} $k_1$} -- (1,0) node[right,scale=0.6] {\color{black} $j_1$};
\draw[lgray,line width=1.5pt,->] 
(-2,0) node[below,scale=0.6] {\color{black} $i_2$} -- (-1,1) node[above,scale=0.6] {\color{black} $k_2$} -- (1,1) node[right,scale=0.6] {\color{black} $j_2$};
\draw[lgray,line width=1.5pt,->] 
(0,-1) node[below,scale=0.6] {\color{black} $i_3$} -- (0,0.5) node[scale=0.6] {\color{black} $k_3$} -- (0,2) node[above,scale=0.6] {\color{black} $j_3$};
\node[left] at (-2.2,1) {$x \rightarrow$};
\node[left] at (-2.2,0) {$y \rightarrow$};
\node[below] at (0,-1.4) {$\uparrow$};
\node[below] at (0,-1.9) {$z$};
}
\quad
=
\quad
\sum_{0 \leq k_1,k_2,k_3 \leq n}
\tikz{0.8}{
\draw[lgray,line width=1.5pt,->] 
(-1,1) node[left,scale=0.6] {\color{black} $i_1$} -- (1,1) node[above,scale=0.6] {\color{black} $k_1$} -- (2,0) node[below,scale=0.6] {\color{black} $j_1$};
\draw[lgray,line width=1.5pt,->] 
(-1,0) node[left,scale=0.6] {\color{black} $i_2$} -- (1,0) node[below,scale=0.6] {\color{black} $k_2$} -- (2,1) node[above,scale=0.6] {\color{black} $j_2$};
\draw[lgray,line width=1.5pt,->] 
(0,-1) node[below,scale=0.6] {\color{black} $i_3$} -- (0,0.5) node[scale=0.6] {\color{black} $k_3$} -- (0,2) node[above,scale=0.6] {\color{black} $j_3$};
\node[left] at (-1.5,1) {$x \rightarrow$};
\node[left] at (-1.5,0) {$y \rightarrow$};
\node[below] at (0,-1.4) {$\uparrow$};
\node[below] at (0,-1.9) {$z$};
}
\end{align*}
for all fixed indices $i_1,i_2,i_3,j_1,j_2,j_3 \in \{0,1,\dots,n\}$; the unitarity relation \eqref{unitarity} becomes
\begin{align*}
\sum_{0 \leq k_1,k_2 \leq n}
\tikz{0.7}{
\draw[lgray,line width=1.5pt,->,rounded corners] (-1,0) node[left,scale=0.6] {\color{black} $i_1$} -- (1,0) node[below,scale=0.6] {\color{black} $k_1$} -- (1,2) node[above,scale=0.6] {\color{black} $j_1$};
\draw[lgray,line width=1.5pt,->,rounded corners] (0,-1) node[below,scale=0.6] {\color{black} $i_2$} -- (0,1) node[above,scale=0.6] {\color{black} $k_2$} -- (2,1) node[right,scale=0.6] {\color{black} $j_2$};
\node[left] at (-1.5,0) {$x \rightarrow$};
\node[below] at (0,-1.4) {$\uparrow$};
\node[below] at (0,-2) {$y$};
}
=
1,
\end{align*}
for all fixed indices $i_1,i_2,j_1,j_2 \in \{0,1,\dots,n\}$; and the stochasticity relation \eqref{stoch} reads
\begin{align*}
\sum_{k=0}^{n}
\sum_{\ell=0}^{n}
\tikz{0.7}{
\draw[lgray,line width=1.5pt,->] (-1,0) -- (1,0);
\draw[lgray,line width=1.5pt,->] (0,-1) -- (0,1);
\node[left] at (-1,0) {\tiny $j$};\node[right] at (1,0) {\tiny $\ell$};
\node[below] at (0,-1) {\tiny $i$};\node[above] at (0,1) {\tiny $k$};
}
=
1,
\end{align*}
for all fixed indices $i,j \in \{0,1,\dots,n\}$. We will use these graphical identities frequently throughout the text. 

\section{Higher-spin $L$ and $M$-matrices}
\label{ssec:models}

We now generalize the lattice model discussed in Section \ref{ssec:fundamental}, by lifting it to a higher-spin setting. In this more general model, horizontal edges continue to have a finite-dimensional state-space: the edge is either empty (state $0$), or it is occupied by a path of colour $1 \leq i \leq n$, and accordingly horizontal edges of the lattice continue to be assigned a single index $0 \leq i \leq n$, which encodes their state. 

On the other hand, vertical edges can now be occupied by any number of paths, of any colour. Vertical edges are thus assigned vector labels $\I = (I_1,\dots,I_n)$, where $I_i \geq 0$ denotes the number of paths of colour $i$ situated at that edge. A completely empty vertical edge is encoded by the vector $\I = (0,\dots,0)=\mathbf{0}$. A generic vertex in this model is shown below: 
\begin{align}
\label{generic-L}
L_{x,q,s}(\I,j; \K,\ell)
\equiv
L_x(\I,j; \K,\ell)
\index{L1@$L_x(\I,j; \K,\ell)$; vertex weights}
=
\tikz{0.7}{
\draw[lgray,line width=1.5pt,->] (-1,0) -- (1,0);
\draw[lgray,line width=4pt,->] (0,-1) -- (0,1);
\node[left] at (-1,0) {\tiny $j$};\node[right] at (1,0) {\tiny $\ell$};
\node[below] at (0,-1) {\tiny $\I$};\node[above] at (0,1) {\tiny $\K$};
},
\quad
j,\ell \in \{0,1,\dots,n\},
\quad
\I,\K \in \mathbb{N}^n,
\end{align}
where \index{N@$\mathbb{N}$; natural numbers, including $0$} $\mathbb{N} = \{0,1,2,\dots\}$.\footnote{To avoid having to specify nonnegative integers by the set $\mathbb{Z}_{\geq 0}$, we define the natural numbers $\mathbb{N}$ to include $0$.} The vertex in \eqref{generic-L} is drawn using a thick vertical line; this is to distinguish it from the ordinary $R$-matrix vertex \eqref{R-vert} and to indicate that many lattice paths can now occupy vertical edges.

Before stating the vertex weights explicitly, we develop some useful vector notation. For all $1\leq i \leq n$, let $\bm{e}_i$ denote the $i$-th Euclidean unit vector. For any vector $\I = (I_1,\dots,I_n) \in \mathbb{N}^n$ and indices $i,j \in \{1,\dots,n\}$ we define 
\index{I@$\I^{+}_{i}$}
\index{I@$\I^{-}_{i}$}
\index{I@$\I^{+-}_{ij}$}
\index{I@$\lvert\I\rvert$}
\index{I@$\Is{i}{j}$}
\begin{align*}
\I^{+}_{i}
=
\I + \bm{e}_i,
\quad
\I^{-}_{i}
=
\I - \bm{e}_i,
\quad
\I^{+-}_{ij}
=
\I + \bm{e}_i - \bm{e}_j,
\quad
|\I|
=
\sum_{k=1}^{n} I_k,
\quad
\Is{i}{j}
=
\sum_{k=i}^{j} I_k,
\end{align*}
where in the final case it is assumed that $i \leq j$; by agreement, we choose $\Is{i}{j} = 0$ for $i>j$. In terms of these notations, we tabulate the vertex weights for our higher-spin lattice model below:
\begin{align}
\label{s-weights}
\begin{tabular}{|c|c|c|}
\hline
\quad
\tikz{0.7}{
\draw[lgray,line width=1.5pt,->] (-1,0) -- (1,0);
\draw[lgray,line width=4pt,->] (0,-1) -- (0,1);
\node[left] at (-1,0) {\tiny $0$};\node[right] at (1,0) {\tiny $0$};
\node[below] at (0,-1) {\tiny $\I$};\node[above] at (0,1) {\tiny $\I$};
}
\quad
&
\quad
\tikz{0.7}{
\draw[lgray,line width=1.5pt,->] (-1,0) -- (1,0);
\draw[lgray,line width=4pt,->] (0,-1) -- (0,1);
\node[left] at (-1,0) {\tiny $i$};\node[right] at (1,0) {\tiny $i$};
\node[below] at (0,-1) {\tiny $\I$};\node[above] at (0,1) {\tiny $\I$};
}
\quad
&
\quad
\tikz{0.7}{
\draw[lgray,line width=1.5pt,->] (-1,0) -- (1,0);
\draw[lgray,line width=4pt,->] (0,-1) -- (0,1);
\node[left] at (-1,0) {\tiny $0$};\node[right] at (1,0) {\tiny $i$};
\node[below] at (0,-1) {\tiny $\I$};\node[above] at (0,1) {\tiny $\I^{-}_i$};
}
\quad
\\[1.3cm]
\quad
$\dfrac{1-s x q^{\Is{1}{n}}}{1-sx}$
\quad
& 
\quad
$\dfrac{(x-sq^{I_i}) q^{\Is{i+1}{n}}}{1-sx}$
\quad
& 
\quad
$\dfrac{x(1-q^{I_i}) q^{\Is{i+1}{n}}}{1-sx}$
\quad
\\[0.7cm]
\hline
\quad
\tikz{0.7}{
\draw[lgray,line width=1.5pt,->] (-1,0) -- (1,0);
\draw[lgray,line width=4pt,->] (0,-1) -- (0,1);
\node[left] at (-1,0) {\tiny $i$};\node[right] at (1,0) {\tiny $0$};
\node[below] at (0,-1) {\tiny $\I$};\node[above] at (0,1) {\tiny $\I^{+}_i$};
}
\quad
&
\quad
\tikz{0.7}{
\draw[lgray,line width=1.5pt,->] (-1,0) -- (1,0);
\draw[lgray,line width=4pt,->] (0,-1) -- (0,1);
\node[left] at (-1,0) {\tiny $i$};\node[right] at (1,0) {\tiny $j$};
\node[below] at (0,-1) {\tiny $\I$};\node[above] at (0,1) 
{\tiny $\I^{+-}_{ij}$};
}
\quad
&
\quad
\tikz{0.7}{
\draw[lgray,line width=1.5pt,->] (-1,0) -- (1,0);
\draw[lgray,line width=4pt,->] (0,-1) -- (0,1);
\node[left] at (-1,0) {\tiny $j$};\node[right] at (1,0) {\tiny $i$};
\node[below] at (0,-1) {\tiny $\I$};\node[above] at (0,1) {\tiny $\I^{+-}_{ji}$};
}
\quad
\\[1.3cm] 
\quad
$\dfrac{1-s^2 q^{\Is{1}{n}}}{1-sx}$
\quad
& 
\quad
$\dfrac{x(1-q^{I_j}) q^{\Is{j+1}{n}}}{1-sx}$
\quad
&
\quad
$\dfrac{s(1-q^{I_i})q^{\Is{i+1}{n}}}{1-sx}$
\quad
\\[0.7cm]
\hline
\end{tabular} 
\end{align}
where the indices $i,j$ take any values $i,j \in \{1,\dots,n\}$, and where it is assumed that $i<j$.


\begin{rmk}
At $n=1$, the vector state $\I$ collapses to a single integer state, and the model reduces to the higher-spin vertex model studied in \cite{Borodin,BorodinP1,BorodinP2}. Namely, one recovers a model with weights
\begin{align}
\label{rank1-weights}
\begin{tabular}{|c|c|c|c|}
\hline
\quad
\tikz{0.7}{
\draw[lgray,line width=1.5pt,->] (-1,0) -- (1,0);
\draw[lgray,line width=4pt,->] (0,-1) -- (0,1);
\node[left] at (-1,0) {\tiny $0$};\node[right] at (1,0) {\tiny $0$};
\node[below] at (0,-1) {\tiny $I$};\node[above] at (0,1) {\tiny $I$};
\node at (0,0) {$\bullet$};
}
\quad
&
\quad
\tikz{0.7}{
\draw[lgray,line width=1.5pt,->] (-1,0) -- (1,0);
\draw[lgray,line width=4pt,->] (0,-1) -- (0,1);
\node[left] at (-1,0) {\tiny $1$};\node[right] at (1,0) {\tiny $1$};
\node[below] at (0,-1) {\tiny $I$};\node[above] at (0,1) {\tiny $I$};
\node at (0,0) {$\bullet$};
}
\quad
&
\quad
\tikz{0.7}{
\draw[lgray,line width=1.5pt,->] (-1,0) -- (1,0);
\draw[lgray,line width=4pt,->] (0,-1) -- (0,1);
\node[left] at (-1,0) {\tiny $0$};\node[right] at (1,0) {\tiny $1$};
\node[below] at (0,-1) {\tiny $I$};\node[above] at (0,1) {\tiny $I-1$};
\node at (0,0) {$\bullet$};
}
\quad
&
\quad
\tikz{0.7}{
\draw[lgray,line width=1.5pt,->] (-1,0) -- (1,0);
\draw[lgray,line width=4pt,->] (0,-1) -- (0,1);
\node[left] at (-1,0) {\tiny $1$};\node[right] at (1,0) {\tiny $0$};
\node[below] at (0,-1) {\tiny $I$};\node[above] at (0,1) {\tiny $I+1$};
\node at (0,0) {$\bullet$};
}
\\[1.3cm]
\quad
$\dfrac{1-s x q^{I}}{1-sx}$
\quad
& 
\quad
$\dfrac{(x-sq^{I})}{1-sx}$
\quad
& 
\quad
$\dfrac{x(1-q^{I})}{1-sx}$
\quad
&
\quad
$\dfrac{1-s^2 q^{I}}{1-sx}$
\quad
\\[0.7cm]
\hline
\end{tabular} 
\end{align}
where $I \in \mathbb{N}$ (we no longer employ boldface in the $n=1$ case, since all edge states consist of a single component). To emphasize the difference between the general-rank \eqref{s-weights} and rank-1 vertices \eqref{rank1-weights}, where there is potential for confusion, we include a dot when drawing the latter. We shall also write a generic vertex in the rank-1 model \eqref{rank1-weights} as follows:
\begin{align}
\label{Lbull}
L^{\bullet}_x(I,j;K,\ell)
\index{L3@$L^{\bullet}_x(I,j;K,\ell)$; colour-blind weights}
=
\tikz{0.7}{
\draw[lgray,line width=1.5pt,->] (-1,0) -- (1,0);
\draw[lgray,line width=4pt,->] (0,-1) -- (0,1);
\node[left] at (-1,0) {\tiny $j$};\node[right] at (1,0) {\tiny $\ell$};
\node[below] at (0,-1) {\tiny $I$};\node[above] at (0,1) {\tiny $K$};
\node at (0,0) {$\bullet$};
},
\quad
j,\ell \in \{0,1\},
\quad
I,K \in \mathbb{N}.
\end{align}
The Boltzmann weights \eqref{rank1-weights} are precisely the same as the ``conjugated'' weights $w_u^{\sf c}$ used in \cite[Section 2, Figure 5]{BorodinP2}.
\end{rmk}

\begin{rmk}
For general $n$, the model \eqref{s-weights} was considered in \cite{KunibaMMO,BosnjakM} and \cite{GarbaliGW}, where the vertex weights were obtained under a certain tensor product of representations of the universal $U_q(\wh{\mathfrak{sl}_{n+1}})$ $R$-matrix. It also appeared in \cite{Kuan}; in Appendix \ref{app:kuan-weights}, we document the precise link between the notation used in \cite[Section 3.5]{Kuan} and our own notation.  
\end{rmk}

In our constructions we will also make use of a second set of vertex weights, which we call {\it dual vertices.} A generic dual vertex has the form
\begin{align}
\label{generic-M}
M_{x,q,s}(\I,j;\K,\ell)
\equiv
M_x(\I,j;\K,\ell)
\index{M1@$M_x(\I,j;\K,\ell)$; dual vertex weights}
=
\tikz{0.7}{
\draw[lgray,line width=1.5pt,<-] (-1,0) -- (1,0);
\draw[lgray,line width=4pt,->] (0,-1) -- (0,1);
\node[left] at (-1,0) {\tiny $\ell$};\node[right] at (1,0) {\tiny $j$};
\node[below] at (0,-1) {\tiny $\I$};\node[above] at (0,1) {\tiny $\K$};
},
\quad
j,\ell \in \{0,1,\dots,n\},
\quad
\I,\K \in \mathbb{N}^n,
\end{align}
where the horizontal line is now oriented from right to left. The dual weights are tabulated explicitly below:
\begin{align}
\label{dual-s-weights}
\begin{tabular}{|c|c|c|}
\hline
\quad
\tikz{0.7}{
\draw[lgray,line width=1.5pt,<-] (-1,0) -- (1,0);
\draw[lgray,line width=4pt,->] (0,-1) -- (0,1);
\node[left] at (-1,0) {\tiny $0$};\node[right] at (1,0) {\tiny $0$};
\node[below] at (0,-1) {\tiny $\I$};\node[above] at (0,1) {\tiny $\I$};
}
\quad
&
\quad
\tikz{0.7}{
\draw[lgray,line width=1.5pt,<-] (-1,0) -- (1,0);
\draw[lgray,line width=4pt,->] (0,-1) -- (0,1);
\node[left] at (-1,0) {\tiny $i$};\node[right] at (1,0) {\tiny $i$};
\node[below] at (0,-1) {\tiny $\I$};\node[above] at (0,1) {\tiny $\I$};
}
\quad
&
\quad
\tikz{0.7}{
\draw[lgray,line width=1.5pt,<-] (-1,0) -- (1,0);
\draw[lgray,line width=4pt,->] (0,-1) -- (0,1);
\node[left] at (-1,0) {\tiny $i$};\node[right] at (1,0) {\tiny $0$};
\node[below] at (0,-1) {\tiny $\I$};\node[above] at (0,1) {\tiny $\I^{-}_i$};
}
\quad
\\[1.3cm]
\quad
$\dfrac{q^{-\Is{1}{n}}-s x}{1-s x}$
\quad
& 
\quad
$\dfrac{(xq^{-I_i}-s) q^{-\Is{i+1}{n}}}{1-s x}$
\quad
& 
\quad
$\dfrac{(1-q^{-I_i}) q^{-\Is{i+1}{n}}}{1-s x}$
\quad
\\[0.7cm]
\hline
\quad
\tikz{0.7}{
\draw[lgray,line width=1.5pt,<-] (-1,0) -- (1,0);
\draw[lgray,line width=4pt,->] (0,-1) -- (0,1);
\node[left] at (-1,0) {\tiny $0$};\node[right] at (1,0) {\tiny $i$};
\node[below] at (0,-1) {\tiny $\I$};\node[above] at (0,1) {\tiny $\I^{+}_i$};
}
\quad
&
\quad
\tikz{0.7}{
\draw[lgray,line width=1.5pt,<-] (-1,0) -- (1,0);
\draw[lgray,line width=4pt,->] (0,-1) -- (0,1);
\node[left] at (-1,0) {\tiny $j$};\node[right] at (1,0) {\tiny $i$};
\node[below] at (0,-1) {\tiny $\I$};\node[above] at (0,1) {\tiny $\I^{+-}_{ij}$};
}
\quad
&
\quad
\tikz{0.7}{
\draw[lgray,line width=1.5pt,<-] (-1,0) -- (1,0);
\draw[lgray,line width=4pt,->] (0,-1) -- (0,1);
\node[left] at (-1,0) {\tiny $i$};\node[right] at (1,0) {\tiny $j$};
\node[below] at (0,-1) {\tiny $\I$};\node[above] at (0,1) 
{\tiny $\I^{+-}_{ji}$};
}
\quad
\\[1.3cm]
\quad
$\dfrac{x(s^2-q^{-\Is{1}{n}})}{1-s x}$
\quad
& 
\quad
$\dfrac{s(q^{-I_j}-1) q^{-\Is{j+1}{n}}}{1-s x}$
\quad
&
\quad
$\dfrac{x(q^{-I_i}-1)q^{-\Is{i+1}{n}}}{1-s x}$
\quad
\\[0.7cm]
\hline
\end{tabular} 
\end{align}
In the table above we again assume that $i,j \in \{1,\dots,n\}$ with $i<j$. In the case $n=1$, the weights \eqref{dual-s-weights} reduce to
\begin{align}
\label{rank1-dual-weights}
\begin{tabular}{|c|c|c|c|}
\hline
\quad
\tikz{0.7}{
\draw[lgray,line width=1.5pt,<-] (-1,0) -- (1,0);
\draw[lgray,line width=4pt,->] (0,-1) -- (0,1);
\node[left] at (-1,0) {\tiny $0$};\node[right] at (1,0) {\tiny $0$};
\node[below] at (0,-1) {\tiny $I$};\node[above] at (0,1) {\tiny $I$};
\node at (0,0) {$\bullet$};
}
\quad
&
\quad
\tikz{0.7}{
\draw[lgray,line width=1.5pt,<-] (-1,0) -- (1,0);
\draw[lgray,line width=4pt,->] (0,-1) -- (0,1);
\node[left] at (-1,0) {\tiny $1$};\node[right] at (1,0) {\tiny $1$};
\node[below] at (0,-1) {\tiny $I$};\node[above] at (0,1) {\tiny $I$};
\node at (0,0) {$\bullet$};
}
\quad
&
\quad
\tikz{0.7}{
\draw[lgray,line width=1.5pt,<-] (-1,0) -- (1,0);
\draw[lgray,line width=4pt,->] (0,-1) -- (0,1);
\node[left] at (-1,0) {\tiny $1$};\node[right] at (1,0) {\tiny $0$};
\node[below] at (0,-1) {\tiny $I$};\node[above] at (0,1) {\tiny $I-1$};
\node at (0,0) {$\bullet$};
}
\quad
&
\quad
\tikz{0.7}{
\draw[lgray,line width=1.5pt,<-] (-1,0) -- (1,0);
\draw[lgray,line width=4pt,->] (0,-1) -- (0,1);
\node[left] at (-1,0) {\tiny $0$};\node[right] at (1,0) {\tiny $1$};
\node[below] at (0,-1) {\tiny $I$};\node[above] at (0,1) {\tiny $I+1$};
\node at (0,0) {$\bullet$};
}
\\[1.3cm]
\quad
$\dfrac{q^{-I}-s x}{1-sx}$
\quad
& 
\quad
$\dfrac{(xq^{-I}-s)}{1-sx}$
\quad
& 
\quad
$\dfrac{1-q^{-I}}{1-sx}$
\quad
&
\quad
$\dfrac{x(s^2 - q^{-I})}{1-sx}$
\quad
\\[0.7cm]
\hline
\end{tabular} 
\end{align}
and in analogy with \eqref{Lbull}, we write a generic rank-1 dual vertex as
\begin{align}
M^{\bullet}_x(I,j;K,\ell)
\index{M3@$M^{\bullet}_x(I,j;K,\ell)$; dual colour-blind weights}
=
\tikz{0.7}{
\draw[lgray,line width=1.5pt,<-] (-1,0) -- (1,0);
\draw[lgray,line width=4pt,->] (0,-1) -- (0,1);
\node[left] at (-1,0) {\tiny $\ell$};\node[right] at (1,0) {\tiny $j$};
\node[below] at (0,-1) {\tiny $I$};\node[above] at (0,1) {\tiny $K$};
\node at (0,0) {$\bullet$};
},
\quad
j,\ell \in \{0,1\},
\quad
I,K \in \mathbb{N}.
\end{align}

The dual weights \eqref{dual-s-weights} and those listed in \eqref{s-weights} are not independent of each other; the two sets of vertices are related via certain symmetries which are given in Proposition \ref{prop:sym}. These symmetries have a natural interpretation: \eqref{sym1} describes reflection of the vertex \eqref{generic-M} about its vertical axis, while \eqref{sym2} is related to reflection about the horizontal axis (and reversal of the orientation arrows on both lines).
\begin{prop}
\label{prop:sym}
The vertex weights \eqref{s-weights} and \eqref{dual-s-weights} are related via the following transformations:
\begin{align}
\label{sym1}
M_{\b{x},\b{q},\b{s}}(\I,j;\K,\ell)
&=
(-s)^{\bm{1}_{\ell \geq 1}-\bm{1}_{j \geq 1}}
\cdot
L_{x,q,s}(\I,j;\K,\ell),
\\
\label{sym2}
M_{\b{x},\b{q},\b{s}}(\tilde{\I},\tilde{\jmath};\tilde{\K},\tilde{\ell})
&=
\left(-sxq^{\Is{1}{\ell-1}}\right)^{\bm{1}_{\ell \geq 1}}
\left(-sxq^{\Ks{j+1}{n}}\right)^{-\bm{1}_{j \geq 1}}
\frac{(s^2;q)_{|\K|}}{(s^2;q)_{|\I|}}
\prod_{i=1}^{n}
\frac{(q;q)_{I_i}}{(q;q)_{K_i}}
\cdot
L_{x,q,s}(\K,\ell;\I,j),
\end{align}
where we use bars as a shorthand for reciprocation of variables; namely, $\b{x} := x^{-1}$, $\b{q}:= q^{-1}$ and $\b{s} := s^{-1}$. We have also defined the conjugate $\tilde{\ell}$ of an index $\ell \in \{0,1,\dots,n\}$ by
\begin{align}
\label{conj1}
\tilde{\ell}
:=
\left\{
\begin{array}{ll}
0, & \quad \ell = 0,
\\
n-\ell+1, & \quad \ell \geq 1,
\end{array}
\right.
\end{align}
and the conjugate of the vector $\I = (I_1,\dots,I_n)$ by
\begin{align}
\label{conj2}
\index{I@$\tilde{\I}$}
\tilde{\I} = (\tilde{I}_1,\dots,\tilde{I}_n) := (I_n,\dots,I_1),
\end{align}
which is simply the vector $\I$ read in reverse.
\end{prop}

\begin{proof}
The first symmetry \eqref{sym1} is readily apparent by reciprocating all variables present in \eqref{dual-s-weights}, reflecting each vertex about its central vertical line, then comparing against the entries of the table \eqref{s-weights}. The second symmetry \eqref{sym2} requires a case-by-case check, the details of which we omit.
\end{proof}

\section{Intertwining equations}

We will make use of three extensions of the Yang--Baxter equation \eqref{YB} to the higher-spin setting. There is one equation for the exchange of two $L$-matrices under the action of $R$, another for the exchange of two $M$-matrices, and finally an equation of mixed-type which allows the exchange of $L$ and $M$. Collectively, we refer to these as the {\it intertwining equations.} They are the foundation for the commutation relations between the operators used to build the rows of our two-dimensional lattice, and are responsible in turn for almost all of the properties enjoyed by the rational functions $f_{\mu}$ and $g_{\nu}$ (to be introduced later).

\begin{prop}
\label{prop:3RLL}
For any fixed integers $i_1,i_2,i_3,j_1,j_2,j_3 \in \{0,1,\dots,n\}$ and vectors $\I,\J \in \mathbb{N}^n$, the vertex weights \eqref{R-weights-a}, \eqref{R-weights-bc}, \eqref{s-weights} and \eqref{dual-s-weights} satisfy the relations
\begin{multline}
\label{RLLa}
\sum_{0 \leq k_1,k_2 \leq n}
\
\sum_{\K \in \mathbb{N}^n}
R_{y/x}(i_2,i_1;k_2,k_1)
L_x(\I,k_1;\K,j_1)
L_y(\K,k_2;\J,j_2)
=
\\
\sum_{0 \leq k_1,k_2 \leq n}
\
\sum_{\K \in \mathbb{N}^n}
L_y(\I,i_2;\K,k_2)
L_x(\K,i_1;\J,k_1)
R_{y/x}(k_2,k_1;j_2,j_1),
\end{multline}

\begin{multline}
\label{RLLb}
\sum_{0 \leq k_1,k_3 \leq n}
\
\sum_{\K \in \mathbb{N}^n}
L_x(\I,i_1;\K,k_1)
R_{\b{q}\b{x}\b{z}}(i_3,k_1;k_3,j_1)
M_z(\K,k_3;\J,j_3)
=
\\
\sum_{0 \leq k_1,k_3 \leq n}
\
\sum_{\K \in \mathbb{N}^n}
M_z(\I,i_3;\K,k_3)
R_{\b{q}\b{x}\b{z}}(k_3,i_1;j_3,k_1)
L_x(\K,k_1;\J,j_1),
\end{multline}

\begin{multline}
\label{RLLc}
\sum_{0 \leq k_2,k_3 \leq n}
\
\sum_{\K \in \mathbb{N}^n}
M_y(\I,i_2;\K,k_2)
M_z(\K,i_3;\J,k_3)
R_{y/z}(k_3,k_2;j_3,j_2)
=
\\
\sum_{0 \leq k_2,k_3 \leq n}
\
\sum_{\K \in \mathbb{N}^n}
R_{y/z}(i_3,i_2;k_3,k_2)
M_z(\I,k_3;\K,j_3)
M_y(\K,k_2;\J,j_2).
\end{multline}
Translating the equations \eqref{RLLa}--\eqref{RLLc} into vertex form using the identifications \eqref{R-vert}, \eqref{generic-L} and \eqref{generic-M}, they take the form
\begin{align}
\label{graph-RLLa}
\sum_{0 \leq k_1,k_2 \leq n}
\
\sum_{\K \in \mathbb{N}^n}
\tikz{0.8}{
\draw[densely dotted] (-1.25,0.25) arc (-45:45:{1/(2*sqrt(2))});
\draw[lgray,line width=1.5pt,->]
(-2,1) node[above,scale=0.6] {\color{black} $i_1$} -- (-1,0) node[below,scale=0.6] {\color{black} $k_1$} -- (1,0) node[right,scale=0.6] {\color{black} $j_1$};
\draw[lgray,line width=1.5pt,->] 
(-2,0) node[below,scale=0.6] {\color{black} $i_2$} -- (-1,1) node[above,scale=0.6] {\color{black} $k_2$} -- (1,1) node[right,scale=0.6] {\color{black} $j_2$};
\draw[lgray,line width=4pt,->] 
(0,-1) node[below,scale=0.6] {\color{black} $\I$} -- (0,0.5) node[scale=0.6] {\color{black} $\K$} -- (0,2) node[above,scale=0.6] {\color{black} $\J$};
\node[left] at (-2.2,1) {$x \rightarrow$};
\node[left] at (-2.2,0) {$y \rightarrow$};
}
\quad
=
\quad
\sum_{0 \leq k_1,k_2 \leq n}
\
\sum_{\K \in \mathbb{N}^n}
\tikz{0.8}{
\draw[densely dotted] (1.75,0.25) arc (-45:45:{1/(2*sqrt(2))});
\draw[lgray,line width=1.5pt,->] 
(-1,1) node[left,scale=0.6] {\color{black} $i_1$} -- (1,1) node[above,scale=0.6] {\color{black} $k_1$} -- (2,0) node[below,scale=0.6] {\color{black} $j_1$};
\draw[lgray,line width=1.5pt,->] 
(-1,0) node[left,scale=0.6] {\color{black} $i_2$} -- (1,0) node[below,scale=0.6] {\color{black} $k_2$} -- (2,1) node[above,scale=0.6] {\color{black} $j_2$};
\draw[lgray,line width=4pt,->] 
(0,-1) node[below,scale=0.6] {\color{black} $\I$} -- (0,0.5) node[scale=0.6] {\color{black} $\K$} -- (0,2) node[above,scale=0.6] {\color{black} $\J$};
\node[left] at (-1.5,1) {$x \rightarrow$};
\node[left] at (-1.5,0) {$y \rightarrow$};
}
\end{align}
with marked angle equal to $y/x$,
\begin{align}
\label{graph-RLLb}
\sum_{0 \leq k_1,k_3 \leq n}
\
\sum_{\K \in \mathbb{N}^n}
\tikz{0.8}{
\draw[densely dotted] (1.75,0.75) arc (45:135:{1/(2*sqrt(2))});
\draw[lgray,line width=1.5pt,<-] 
(-1,1) node[left,scale=0.6] {\color{black} $j_3$} -- (1,1) node[above,scale=0.6] {\color{black} $k_3$} -- (2,0) node[below,scale=0.6] {\color{black} $i_3$};
\draw[lgray,line width=1.5pt,->] 
(-1,0) node[left,scale=0.6] {\color{black} $i_1$} -- (1,0) node[below,scale=0.6] {\color{black} $k_1$} -- (2,1) node[above,scale=0.6] {\color{black} $j_1$};
\draw[lgray,line width=4pt,->] 
(0,-1) node[below,scale=0.6] {\color{black} $\I$} -- (0,0.5) node[scale=0.6] {\color{black} $\K$} -- (0,2) node[above,scale=0.6] {\color{black} $\J$};
\node[left] at (-1.5,1) {$z \leftarrow$};
\node[left] at (-1.5,0) {$x \rightarrow$};
}
\quad
=
\quad
\sum_{0 \leq k_1,k_3 \leq n}
\
\sum_{\K \in \mathbb{N}^n}
\tikz{0.8}{
\draw[densely dotted] (-1.25,0.75) arc (45:135:{1/(2*sqrt(2))});
\draw[lgray,line width=1.5pt,<-] 
(-2,1) node[above,scale=0.6] {\color{black} $j_3$} -- (-1,0) node[below,scale=0.6] {\color{black} $k_3$} -- (1,0) node[right,scale=0.6] {\color{black} $i_3$};
\draw[lgray,line width=1.5pt,->] 
(-2,0) node[below,scale=0.6] {\color{black} $i_1$} -- (-1,1) node[above,scale=0.6] {\color{black} $k_1$} -- (1,1) node[right,scale=0.6] {\color{black} $j_1$};
\draw[lgray,line width=4pt,->] 
(0,-1) node[below,scale=0.6] {\color{black} $\I$} -- (0,0.5) node[scale=0.6] {\color{black} $\K$} -- (0,2) node[above,scale=0.6] {\color{black} $\J$};
\node[left] at (-2.2,1) {$z \leftarrow$};
\node[left] at (-2.2,0) {$x \rightarrow$};
}
\end{align}
with marked angle equal to $\b{q}\b{x}\b{z}$,
\begin{align}
\label{graph-RLLc}
\sum_{0 \leq k_2,k_3 \leq n}
\
\sum_{\K \in \mathbb{N}^n}
\tikz{0.8}{
\draw[densely dotted] (-1.75,0.75) arc (135:225:{1/(2*sqrt(2))});
\draw[lgray,line width=1.5pt,<-] 
(-2,1) node[above,scale=0.6] {\color{black} $j_2$} -- (-1,0) node[below,scale=0.6] {\color{black} $k_2$} -- (1,0) node[right,scale=0.6] {\color{black} $i_2$};
\draw[lgray,line width=1.5pt,<-] 
(-2,0) node[below,scale=0.6] {\color{black} $j_3$} -- (-1,1) node[above,scale=0.6] {\color{black} $k_3$} -- (1,1) node[right,scale=0.6] {\color{black} $i_3$};
\draw[lgray,line width=4pt,->] 
(0,-1) node[below,scale=0.6] {\color{black} $\I$} -- (0,0.5) node[scale=0.6] {\color{black} $\K$} -- (0,2) node[above,scale=0.6] {\color{black} $\J$};
\node[left] at (-2.2,1) {$y \leftarrow$};
\node[left] at (-2.2,0) {$z \leftarrow$};
}
\quad
=
\quad
\sum_{0 \leq k_2,k_3 \leq n}
\
\sum_{\K \in \mathbb{N}^n}
\tikz{0.8}{
\draw[densely dotted] (1.25,0.75) arc (135:225:{1/(2*sqrt(2))});
\draw[lgray,line width=1.5pt,<-] 
(-1,1) node[left,scale=0.6] {\color{black} $j_2$} -- (1,1) node[above,scale=0.6] {\color{black} $k_2$} -- (2,0) node[below,scale=0.6] {\color{black} $i_2$};
\draw[lgray,line width=1.5pt,<-] 
(-1,0) node[left,scale=0.6] {\color{black} $j_3$} -- (1,0) node[below,scale=0.6] {\color{black} $k_3$} -- (2,1) node[above,scale=0.6] {\color{black} $i_3$};
\draw[lgray,line width=4pt,->] 
(0,-1) node[below,scale=0.6] {\color{black} $\I$} -- (0,0.5) node[scale=0.6] {\color{black} $\K$} -- (0,2) node[above,scale=0.6] {\color{black} $\J$};
\node[left] at (-1.5,1) {$y \leftarrow$};
\node[left] at (-1.5,0) {$z \leftarrow$};
}
\end{align}
with marked angle equal to $y/z$.
\end{prop}

\begin{proof}
There are several different ways to establish the relations \eqref{RLLa}--\eqref{RLLc}. One proof goes through representations of quantized affine algebra $U_q(\wh{\mathfrak{sl}_{n+1}})$, as in \cite{GarbaliGW}. Alternatively, one can analyze all possible cases of the fixed indices in \eqref{RLLa}--\eqref{RLLc}, and perform a direct check of all identities that arise. A more constructive proof is to start from the Yang--Baxter equation and perform fusion, replacing one of the vector spaces in \eqref{YB} by a symmetric tensor representation of integer weight ${\sf J}$, followed by the analytic continuation $q^{\sf J} \mapsto s^{-2}$ \cite{Borodin}. This leads to each of the three possibilities \eqref{RLLa}--\eqref{RLLc}, depending on which of the spaces $W_a,W_b,W_c$ one chooses to single out for replacement. In Appendix \ref{app:fusion}, we outline how this procedure works in the first case, \eqref{RLLa}.

Yet another approach is outlined in Appendix \ref{app:RLL}. It is based on the fact that all three identities \eqref{RLLa}--\eqref{RLLc} are in fact special cases of a master Yang--Baxter equation obtained in \cite{BosnjakM}.
\end{proof}

\section{Colour-blindness results}
\label{ssec:colour-blind}

In this section we examine a further important feature of the model \eqref{s-weights}. It is a property which has been observed in a number of earlier publications \cite{FodaW,GarbaliGW,Kuan}, and used to different effects within those works. This property allows certain linear combinations of higher-rank partition functions to be computable as rank-1 partition functions; \ie\ the linear combinations end up being equivalent to partition functions in the model \eqref{s-weights} with $n=1$. 
In the present work, partition functions which have this property will be termed {\it colour-blind}. The source of the colour-blindness can be traced back to Proposition \ref{prop:colour-blind}, below.

\begin{defn}
For any integer $k \in \{0,1,\dots,n\}$ define its colour-blind projection
\begin{align*}
\index{ai@$\theta(k)$}
\theta(k) 
:=
\bm{1}_{k \geq 1}
=
\left\{ \begin{array}{ll} 0, & \quad k=0, \\ 1, & \quad k \geq 1. \end{array} \right.
\end{align*}
Similarly, for compositions $\K = (K_1,\dots,K_n) \in \mathbb{N}^n$ we define a projection onto $\mathbb{N}$ via their {\it weight}:
\begin{align}
\label{project}
|\K| := \sum_{i=1}^{n} K_i.
\end{align}
Finally, for any $k \in \mathbb{N}$, write the preimage of $k$ under the map \eqref{project} as follows:
\begin{align*}
\index{W@$\mathcal{W}(k)$}
\mathcal{W}(k)
:=
\{ \K \in \mathbb{N}^n : |\K| = k \},
\end{align*}
which is simply the set of all length-$n$ compositions of weight $k$.
\end{defn}

\begin{prop}
\label{prop:colour-blind}
Let $0 \leq j \leq n$ and $\I = (I_1,\dots,I_n)$ be fixed incoming left and bottom edge states of the vertex \eqref{generic-L}. Fix a further nonnegative integer $k \in \mathbb{N}$. Then for any such $j$, $\I$ and $k$ the following identities hold:
\begin{align}
\label{cb1}
\sum_{\K \in \mathcal{W}(k)}
\left( \tikz{0.7}{
\draw[lgray,line width=1.5pt,->] (-1,0) -- (1,0);
\draw[lgray,line width=4pt,->] (0,-1) -- (0,1);
\node[left] at (-1,0) {\tiny $j$};\node[right] at (1,0) {\tiny $0$};
\node[below] at (0,-1) {\tiny $\I$};\node[above] at (0,1) {\tiny $\K$};
} \right)
&=
\tikz{0.7}{
\draw[lgray,line width=1.5pt,->] (-1,0) -- (1,0);
\draw[lgray,line width=4pt,->] (0,-1) -- (0,1);
\node[left] at (-1,0) {\tiny $\theta(j)$};\node[right] at (1,0) {\tiny $0$};
\node[below] at (0,-1) {\tiny $|\I|$};\node[above] at (0,1) {\tiny $k$};
\node at (0,0) {$\bullet$};
},
\\
\label{cb2}
\sum_{\K \in \mathcal{W}(k)}
\sum_{1 \leq \ell \leq n}
\left( \tikz{0.7}{
\draw[lgray,line width=1.5pt,->] (-1,0) -- (1,0);
\draw[lgray,line width=4pt,->] (0,-1) -- (0,1);
\node[left] at (-1,0) {\tiny $j$};\node[right] at (1,0) {\tiny $\ell$};
\node[below] at (0,-1) {\tiny $\I$};\node[above] at (0,1) {\tiny $\K$};
} \right)
&=
\tikz{0.7}{
\draw[lgray,line width=1.5pt,->] (-1,0) -- (1,0);
\draw[lgray,line width=4pt,->] (0,-1) -- (0,1);
\node[left] at (-1,0) {\tiny $\theta(j)$};\node[right] at (1,0) {\tiny $1$};
\node[below] at (0,-1) {\tiny $|\I|$};\node[above] at (0,1) {\tiny $k$};
\node at (0,0) {$\bullet$};
},
\end{align}
where the vertex weights on the left hand side of \eqref{cb1} and \eqref{cb2} are given by \eqref{s-weights}, while those on the right hand side are given by \eqref{rank1-weights}.
\end{prop}

\begin{proof}
Let us begin with the proof of \eqref{cb1}, and split it into the cases $j=0$ and 
$1 \leq j \leq n$. For $j=0$, the left hand side of \eqref{cb1} is given by
\begin{align}
\label{00-lhs}
\sum_{\K \in \mathcal{W}(k)}
\left( \tikz{0.7}{
\draw[lgray,line width=1.5pt,->] (-1,0) -- (1,0);
\draw[lgray,line width=4pt,->] (0,-1) -- (0,1);
\node[left] at (-1,0) {\tiny $0$};\node[right] at (1,0) {\tiny $0$};
\node[below] at (0,-1) {\tiny $\I$};\node[above] at (0,1) {\tiny $\K$};
} \right)
=
\left(
\tikz{0.7}{
\draw[lgray,line width=1.5pt,->] (-1,0) -- (1,0);
\draw[lgray,line width=4pt,->] (0,-1) -- (0,1);
\node[left] at (-1,0) {\tiny $0$};\node[right] at (1,0) {\tiny $0$};
\node[below] at (0,-1) {\tiny $\I$};\node[above] at (0,1) {\tiny $\I$};
}
\right) \bm{1}_{k=|\I|}
=
\frac{1-sxq^{|\I|}}{1-sx} \cdot \bm{1}_{k=|\I|},
\end{align}
where the trivialization of the sum is due to the fact that necessarily $\K = \I$, or the vertex in the summand of \eqref{00-lhs} vanishes. Meanwhile, for $j=0$, the right hand side of \eqref{cb1} reads
\begin{align}
\label{00-rhs}
\tikz{0.7}{
\draw[lgray,line width=1.5pt,->] (-1,0) -- (1,0);
\draw[lgray,line width=4pt,->] (0,-1) -- (0,1);
\node[left] at (-1,0) {\tiny $0$};\node[right] at (1,0) {\tiny $0$};
\node[below] at (0,-1) {\tiny $|\I|$};\node[above] at (0,1) {\tiny $k$};
\node at (0,0) {$\bullet$};
}
=
\frac{1-sxq^{|\I|}}{1-sx} \cdot \bm{1}_{k=|\I|}.
\end{align}
Equating \eqref{00-lhs} and \eqref{00-rhs} we obtain the $j=0$ case of \eqref{cb1}. The 
$1 \leq j \leq n$ case is no more difficult; for $1 \leq j \leq n$, the left hand side of \eqref{cb1} is given by
\begin{align}
\label{10-lhs}
\sum_{\K \in \mathcal{W}(k)}
\left( \tikz{0.7}{
\draw[lgray,line width=1.5pt,->] (-1,0) -- (1,0);
\draw[lgray,line width=4pt,->] (0,-1) -- (0,1);
\node[left] at (-1,0) {\tiny $j$};\node[right] at (1,0) {\tiny $0$};
\node[below] at (0,-1) {\tiny $\I$};\node[above] at (0,1) {\tiny $\K$};
} \right)
=
\left(
\tikz{0.7}{
\draw[lgray,line width=1.5pt,->] (-1,0) -- (1,0);
\draw[lgray,line width=4pt,->] (0,-1) -- (0,1);
\node[left] at (-1,0) {\tiny $j$};\node[right] at (1,0) {\tiny $0$};
\node[below] at (0,-1) {\tiny $\I$};\node[above] at (0,1) {\tiny $\I^{+}_j$};
}
\right) \bm{1}_{k=|\I|+1}
=
\frac{1-s^2 q^{|\I|}}{1-sx} \cdot \bm{1}_{k=|\I|+1},
\end{align}
with the reduction of the sum again coming from the conservation requirement $\K = \I + \bm{e}_j = \I^{+}_j$. For $1 \leq j \leq n$ we have $\theta(j) = 1$, and the right hand side of \eqref{cb1} is then given by
\begin{align}
\label{10-rhs}
\tikz{0.7}{
\draw[lgray,line width=1.5pt,->] (-1,0) -- (1,0);
\draw[lgray,line width=4pt,->] (0,-1) -- (0,1);
\node[left] at (-1,0) {\tiny $1$};\node[right] at (1,0) {\tiny $0$};
\node[below] at (0,-1) {\tiny $|\I|$};\node[above] at (0,1) {\tiny $k$};
\node at (0,0) {$\bullet$};
}
=
\frac{1-s^2 q^{|\I|}}{1-sx} \cdot \bm{1}_{k=|\I|+1}.
\end{align}
Matching \eqref{10-lhs} and \eqref{10-rhs} we prove the $1 \leq j \leq n$ case of \eqref{cb1}.

The proof of \eqref{cb2} is slightly more involved, since the left hand side features an extra summation; it turns out that this sum telescopes. We again divide the proof into the cases $j=0$ and $1 \leq j \leq n$. For $j=0$, the left hand side of \eqref{cb2} takes the form
\begin{align}
\label{01-lhs}
\sum_{\K \in \mathcal{W}(k)}
\sum_{1 \leq \ell \leq n}
\left( \tikz{0.7}{
\draw[lgray,line width=1.5pt,->] (-1,0) -- (1,0);
\draw[lgray,line width=4pt,->] (0,-1) -- (0,1);
\node[left] at (-1,0) {\tiny $0$};\node[right] at (1,0) {\tiny $\ell$};
\node[below] at (0,-1) {\tiny $\I$};\node[above] at (0,1) {\tiny $\K$};
} \right)
&=
\sum_{1 \leq \ell \leq n}
\left( \tikz{0.7}{
\draw[lgray,line width=1.5pt,->] (-1,0) -- (1,0);
\draw[lgray,line width=4pt,->] (0,-1) -- (0,1);
\node[left] at (-1,0) {\tiny $0$};\node[right] at (1,0) {\tiny $\ell$};
\node[below] at (0,-1) {\tiny $\I$};\node[above] at (0,1) {\tiny $\I^{-}_{\ell}$};
} \right) 
\bm{1}_{k=|\I|-1}
\\
\nonumber
&=
\sum_{1 \leq \ell \leq n}
\frac{x (1-q^{I_\ell}) q^{\Is{\ell+1}{n}}}{1-sx} \cdot
\bm{1}_{k=|\I|-1}
=
\frac{x(1-q^{|\I|})}{1-sx} \cdot
\bm{1}_{k=|\I|-1}
\end{align}
where we eliminated the sum over $\K$ using the conservation requirement 
$\K = \I - \bm{e}_\ell = \I^{-}_\ell$. On the other hand, for $j=0$ the right hand side of \eqref{cb2} is equal to
\begin{align}
\label{01-rhs}
\tikz{0.7}{
\draw[lgray,line width=1.5pt,->] (-1,0) -- (1,0);
\draw[lgray,line width=4pt,->] (0,-1) -- (0,1);
\node[left] at (-1,0) {\tiny $0$};\node[right] at (1,0) {\tiny $1$};
\node[below] at (0,-1) {\tiny $|\I|$};\node[above] at (0,1) {\tiny $k$};
\node at (0,0) {$\bullet$};
}
=
\frac{x(1-q^{|\I|})}{1-sx} \cdot
\bm{1}_{k=|\I|-1},
\end{align}
which agrees with the expression obtained in \eqref{01-lhs}, completing the proof of \eqref{cb2} for $j=0$. Finally, when $1 \leq j \leq n$, the left hand side of \eqref{cb2} reads
\begin{align}
\label{11-lhs}
\sum_{\K \in \mathcal{W}(k)}
\sum_{1 \leq \ell \leq n}
\left( \tikz{0.7}{
\draw[lgray,line width=1.5pt,->] (-1,0) -- (1,0);
\draw[lgray,line width=4pt,->] (0,-1) -- (0,1);
\node[left] at (-1,0) {\tiny $j$};\node[right] at (1,0) {\tiny $\ell$};
\node[below] at (0,-1) {\tiny $\I$};\node[above] at (0,1) {\tiny $\K$};
} \right)
&=
\sum_{1 \leq \ell \leq n}
\left( \tikz{0.7}{
\draw[lgray,line width=1.5pt,->] (-1,0) -- (1,0);
\draw[lgray,line width=4pt,->] (0,-1) -- (0,1);
\node[left] at (-1,0) {\tiny $j$};\node[right] at (1,0) {\tiny $\ell$};
\node[below] at (0,-1) {\tiny $\I$};\node[above] at (0,1) {\tiny $\I^{+-}_{j\ell}$};
} \right) 
\bm{1}_{k=|\I|},
\end{align}
due to the conservation requirement $\K = \I + \bm{e}_j - \bm{e}_{\ell} = \I^{+-}_{j\ell}$. The weights in \eqref{11-lhs} are then naturally grouped into three categories: those with edge states $j < \ell$, $j = \ell$, and $j > \ell$. Splitting the sum \eqref{11-lhs} accordingly, we find that the sub-summations can be telescoped:
\begin{align}
\label{11-lhs2}
&
\sum_{\K \in \mathcal{W}(k)}
\sum_{1 \leq \ell \leq n}
\left( \tikz{0.7}{
\draw[lgray,line width=1.5pt,->] (-1,0) -- (1,0);
\draw[lgray,line width=4pt,->] (0,-1) -- (0,1);
\node[left] at (-1,0) {\tiny $j$};\node[right] at (1,0) {\tiny $\ell$};
\node[below] at (0,-1) {\tiny $\I$};\node[above] at (0,1) {\tiny $\K$};
} \right)
\\
\nonumber
&
=
\left(
\sum_{j < \ell \leq n}
\frac{x(1-q^{I_\ell})q^{\Is{\ell+1}{n}}}{1-sx}
+
\frac{(x-sq^{I_j})q^{\Is{j+1}{n}}}{1-sx}
+
\sum_{1 \leq \ell < j}
\frac{s(1-q^{I_\ell})q^{\Is{\ell+1}{n}}}{1-sx}
\right)
\bm{1}_{k=|\I|}
\\
\nonumber
&=
\left(
\frac{x(1-q^{\Is{j+1}{n}})}{1-sx}
+
\frac{(x-sq^{I_j})q^{\Is{j+1}{n}}}{1-sx}
+
\frac{s(1-q^{\Is{1}{j-1}})q^{\Is{j}{n}}}{1-sx}
\right)
\bm{1}_{k=|\I|}
=
\frac{x-sq^{|\I|}}{1-sx} \cdot
\bm{1}_{k=|\I|}.
\end{align}
Comparing with the right hand side of \eqref{cb2} for $1 \leq j \leq n$,
\begin{align}
\label{11-rhs}
\tikz{0.7}{
\draw[lgray,line width=1.5pt,->] (-1,0) -- (1,0);
\draw[lgray,line width=4pt,->] (0,-1) -- (0,1);
\node[left] at (-1,0) {\tiny $1$};\node[right] at (1,0) {\tiny $1$};
\node[below] at (0,-1) {\tiny $|\I|$};\node[above] at (0,1) {\tiny $k$};
\node at (0,0) {$\bullet$};
}
=
\frac{x-sq^{|\I|}}{1-sx} \cdot
\bm{1}_{k=|\I|},
\end{align}
we observe the required agreement between \eqref{11-lhs2} and \eqref{11-rhs}.
\end{proof}

\begin{rmk}
The colour-blindness results \eqref{cb1} and \eqref{cb2} can be written algebraically as
\begin{align}
\label{cb3}
\sum_{\K \in \mathcal{W}(k)}
L_x(\I,j;\K,0)
=
L^{\bullet}_x(|\I|,\theta(j);k,0),
\\
\label{cb4}
\sum_{\K \in \mathcal{W}(k)}
\sum_{1 \leq \ell \leq n}
L_x(\I,j;\K,\ell)
=
L^{\bullet}_x(|\I|,\theta(j);k,1).
\end{align}
These relations also extend to the dual vertex weights \eqref{dual-s-weights} and \eqref{rank1-dual-weights}; namely, one has
\begin{align}
\label{cb5}
\sum_{\K \in \mathcal{W}(k)}
M_x(\I,j;\K,0)
=
M^{\bullet}_x(|\I|,\theta(j);k,0),
\\
\label{cb6}
\sum_{\K \in \mathcal{W}(k)}
\sum_{1 \leq \ell \leq n}
M_x(\I,j;\K,\ell)
=
M^{\bullet}_x(|\I|,\theta(j);k,1),
\end{align}
which can be easily deduced by applying the symmetry \eqref{sym1} to \eqref{cb3} and \eqref{cb4}.
\end{rmk}

\section{Stochastic weights}
\label{ssec:stoch-wt}

As stated, the vertex models \eqref{s-weights} and \eqref{dual-s-weights} are not stochastic, but they can be brought into stochastic form via a simple gauge transformation:
\begin{prop}
\label{prop:stoch}
Define modified versions of the vertices \eqref{generic-L} and \eqref{generic-M}, given by
\index{L2@$\tilde{L}_x(\I,j;\K,\ell)$; stochastic weights}
\index{M2@$\tilde{M}_x(\I,j;\K,\ell)$; dual stochastic weights}
\begin{align}
\label{stoch-wt}
\tilde{L}_x(\I,j;\K,\ell)
:=
(-s)^{\theta(\ell)}
L_x(\I,j;\K,\ell),
\qquad
\tilde{M}_x(\I,j;\K,\ell)
:=
(-s)^{-\theta(j)}
M_x(\I,j;\K,\ell),
\end{align}
\ie\ the vertex \eqref{generic-L} receives an extra factor of $(-s)$ if $\ell > 0$, and similarly for the vertex \eqref{generic-M}. The modified weights satisfy the stochasticity equations
\begin{align}
\label{LM-stoch}
\sum_{\K \in \mathbb{N}^n}
\sum_{0 \leq \ell \leq n}
\tilde{L}_x(\I,j;\K,\ell)
=
1,
\qquad
\sum_{\K \in \mathbb{N}^n}
\sum_{0 \leq \ell \leq n}
\tilde{M}_x(\I,j;\K,\ell)
=
1,
\end{align}
for all fixed values of $0 \leq j \leq n$ and $\I \in \mathbb{N}^n$.
\end{prop}

\begin{proof}
By virtue of the symmetry \eqref{sym1} and the definitions \eqref{stoch-wt}, it can be seen that the stochasticity relation for $\tilde{M}_x(\I,j;\K,\ell)$ is equivalent to the stochasticity of $\tilde{L}_x(\I,j;\K,\ell)$. Focusing therefore on the first of the equations \eqref{LM-stoch}, we write the left hand side as
\begin{align}
\label{stoch-refine}
\sum_{\K \in \mathbb{N}^n}
\sum_{0 \leq \ell \leq n}
\tilde{L}_x(\I,j;\K,\ell)
=
\sum_{k \in \mathbb{N}}
\left(
\sum_{\K \in \mathcal{W}(k)}
L_{x}(\I,j;\K,0)
-s
\sum_{\K \in \mathcal{W}(k)}
\sum_{1 \leq \ell \leq n}
L_{x}(\I,j;\K,\ell)
\right),
\end{align}
where we have refined the sum over $\K$ by the weight $|\K|$, and split the sum over $\ell$ into the cases $\ell = 0$ and $1 \leq \ell \leq n$. We can now use Proposition \ref{prop:colour-blind} to compute the sums on the right hand side of \eqref{stoch-refine}; they are precisely the identities \eqref{cb1} and \eqref{cb2}. Hence,
\begin{align}
\label{2.8-1}
\sum_{\K \in \mathbb{N}^n}
\sum_{0 \leq \ell \leq n}
\tilde{L}_x(\I,j;\K,\ell)
=
\sum_{k \in \mathbb{N}}
L^{\bullet}_{x}(|\I|,\theta(j);k,0)
-
s 
\sum_{k \in \mathbb{N}}
L^{\bullet}_{x}(|\I|,\theta(j);k,1),
\end{align}
where the right hand side sums are taken over vertices in the rank-1 model \eqref{rank1-weights}. By conservation of paths, the first sum is non-vanishing only at $k=|\I| + \theta(j)$, while the second survives only at $k = |\I| + \theta(j) - 1$. Substituting the weights \eqref{rank1-weights} into \eqref{2.8-1}, we find that
\begin{align*}
\sum_{\K \in \mathbb{N}^n}
\sum_{0 \leq \ell \leq n}
\tilde{L}_x(\I,j;\K,\ell)
=
\left\{
\begin{array}{ll}
\dfrac{1-s x q^{|\I|}}{1-sx}
-
\dfrac{s x(1-q^{|\I|})}{1-sx}
=
1,
&\quad 
\theta(j) = 0,
\\ \\
\dfrac{1-s^2 q^{|\I|}}{1-sx}
-
\dfrac{s (x-sq^{|\I|})}{1-sx}
=
1,
&\quad
\theta(j) = 1.
\end{array}
\right.
\end{align*}

\end{proof}

\begin{rmk}{\rm
It is straightforward to show that the transformed weights \eqref{stoch-wt} continue to obey the intertwining relations \eqref{RLLa}--\eqref{RLLc} if one performs the replacements $L_x \mapsto \tilde{L}_x$ and $M_x \mapsto \tilde{M}_x$, meaning that the integrability of the models is preserved. In fact, it is actually the weights \eqref{stoch-wt} that one obtains by applying the fusion procedure to the starting $R$-matrix \eqref{Rmat}; the overall factors of $(-s)$ and $(-s^{-1})$ can then be dropped to obtain the models \eqref{s-weights} and \eqref{dual-s-weights}. We refer the reader to Appendices \ref{app:fusion} and \ref{app:RLL} for more details. 

The main reason for introducing the non-stochastic versions $L_x$ and $M_x$ of the weights is for consistency with \cite{Borodin,BorodinP1,BorodinP2,BorodinWheeler}; if one preferred, one could work exclusively with the stochastic weights  $\tilde{L}_x$ and $\tilde{M}_x$ in everything that follows.
}
\end{rmk}

\chapter{Row operators and non-symmetric rational functions}

\section{Space of states and row operators}
\label{ssec:row-ops}

Let $V$ be the infinite-dimensional vector space obtained by taking the linear span of all $n$-tuples of nonnegative integers:  
\begin{align*}
V= {\rm Span}\{\ket{\I}\} =
{\rm Span}\{\ket{i_1,\dots,i_n}\}_{i_1,\dots,i_n \in \mathbb{N}}.
\end{align*}
The vector space $V$ is the basic ingredient that allows us to translate between a lattice model description of partition functions and a purely algebraic one. It is convenient to consider an infinite tensor product of such spaces, 
\begin{align*}
\index{V1@$\mathbb{V}$; space of states}
\mathbb{V} = V_0 \otimes V_1 \otimes V_2 \otimes \cdots,
\end{align*}
where each $V_i$ denotes a copy of $V$. Let $\bigotimes_{k=0}^{\infty} \ket{\I_k}_k$ be a finite state in $\mathbb{V}$, \ie\ assume that there exists $N \in \mathbb{N}$ such that $\I_k = \bm{0}$ for all $k > N$; in what follows only such states are considered. We define two families of linear operators acting on the finite states (their action on linear combinations of such states in $\mathbb{V}$ is then deduced by linearity):
\begin{align}
\label{C-row}
\index{C@$\C_i(x)$; row operators}
\C_i(x)
:
\bigotimes_{k=0}^{\infty}
\ket{\I_k}_k
\mapsto
\sum_{\J_0,\J_1,\ldots \in \mathbb{N}^n}
\left(
\tikz{0.6}{
\draw[lgray,line width=1.5pt,->] (0,0) -- (7,0);
\foreach\x in {1,...,6}{
\draw[lgray,line width=4pt,->] (\x,-1) -- (\x,1);
}
\node[left] at (-0.5,0) {$x \rightarrow$};
\node[left] at (0,0) {\tiny $i$};\node[right] at (7,0) {\tiny $0$};
\node[below] at (6,-1) {\tiny $\cdots$};\node[above] at (6,1) {\tiny $\cdots$};
\node[below] at (5,-1) {\tiny $\cdots$};\node[above] at (5,1) {\tiny $\cdots$};
\node[below] at (4,-1) {\tiny $\cdots$};\node[above] at (4,1) {\tiny $\cdots$};
\node[below] at (3,-1) {\tiny $\J_2$};\node[above] at (3,1) {\tiny $\I_2$};
\node[below] at (2,-1) {\tiny $\J_1$};\node[above] at (2,1) {\tiny $\I_1$};
\node[below] at (1,-1) {\tiny $\J_0$};\node[above] at (1,1) {\tiny $\I_0$};
}
\right)
\bigotimes_{k=0}^{\infty}
\ket{\J_k}_k,
\quad
0 \leq i \leq n,
\end{align}

\begin{align}
\label{B-row}
\index{B@$\B_i(x)$; row operators}
\B_i(x)
:
\bigotimes_{k=0}^{\infty}
\ket{\I_k}_k
\mapsto
\sum_{\J_0,\J_1,\ldots \in \mathbb{N}^n}
\left(
\tikz{0.6}{
\draw[lgray,line width=1.5pt,<-] (0,0) -- (7,0);
\foreach\x in {1,...,6}{
\draw[lgray,line width=4pt,->] (\x,-1) -- (\x,1);
}
\node[left] at (-0.5,0) {$x \leftarrow$};
\node[left] at (0,0) {\tiny $i$};\node[right] at (7,0) {\tiny $0$};
\node[below] at (6,-1) {\tiny $\cdots$};\node[above] at (6,1) {\tiny $\cdots$};
\node[below] at (5,-1) {\tiny $\cdots$};\node[above] at (5,1) {\tiny $\cdots$};
\node[below] at (4,-1) {\tiny $\cdots$};\node[above] at (4,1) {\tiny $\cdots$};
\node[below] at (3,-1) {\tiny $\J_2$};\node[above] at (3,1) {\tiny $\I_2$};
\node[below] at (2,-1) {\tiny $\J_1$};\node[above] at (2,1) {\tiny $\I_1$};
\node[below] at (1,-1) {\tiny $\J_0$};\node[above] at (1,1) {\tiny $\I_0$};
}
\right)
\bigotimes_{k=0}^{\infty}
\ket{\J_k}_k,
\quad
0 \leq i \leq n,
\end{align}
where the expansion coefficients in the sums are one-row partition functions in the higher-spin vertex models defined in Section \ref{ssec:models}, on a semi-infinite lattice. It is easy to see that these partition functions are well defined and are completely factorized over the Boltzmann weights \eqref{s-weights} and \eqref{dual-s-weights}, for any fixed finite state $\bigotimes_{k=0}^{\infty} \ket{\J_k}_k$ in the summations \eqref{C-row} and \eqref{B-row}. Indeed, once the state $(\J_0,\J_1,\J_2,\dots)$ at the base of these partition functions is fixed, one finds that all internal horizontal edges assume a unique state (by conservation of lattice paths) and that only finitely many vertices have a Boltzmann weight not equal to $1$ (since sufficiently far to the right, only empty vertices occur).

\section{Commutation relations}

We proceed to derive the necessary commutation relations between the row operators introduced in Section \ref{ssec:row-ops}. They can all be recovered by iteration of the local intertwining equations \eqref{graph-RLLa}--\eqref{graph-RLLc}, as we now show.
\begin{thm}
\label{thm:CC}
Fix two nonnegative integers $i,j$ such that $0 \leq i,j \leq n$. The row operators \eqref{C-row} satisfy the exchange relations
\begin{align}
\label{CC=}
\C_i(x) \C_i(y) &= \C_i(y) \C_i(x),
\\
\label{CC<}
q\, \C_i(x) \C_j(y) &= \frac{x-qy}{x-y} \C_j(y) \C_i(x) - \frac{(1-q)x}{x-y} \C_j(x) \C_i(y),
\quad\quad
i<j,
\\
\label{CC>}
\C_i(x) \C_j(y) &= \frac{x-qy}{x-y} \C_j(y) \C_i(x) - \frac{(1-q)y}{x-y} \C_j(x) \C_i(y),
\quad\quad
i>j.
\end{align}
\end{thm}

\begin{proof}
Start by writing an $(N+1)$-site version of the intertwining equation \eqref{graph-RLLa},
\begin{align}
\label{global-RLLa}
\sum_{0 \leq k_1,k_2 \leq n}
\
\sum_{\K_0,\dots,\K_N \in \mathbb{N}^n}
\tikz{0.6}{
\draw[densely dotted] (-1.25,0.25) arc (-45:45:{1/(2*sqrt(2))});
\node at (1,0.5) {$\cdots$};
\draw[lgray,line width=1.5pt,->]
(-2,1) node[above,scale=0.6] {\color{black} $i_1$} -- (-1,0) node[below,scale=0.6] {\color{black} $k_1$} -- (3,0) node[right,scale=0.6] {\color{black} $j_1$};
\draw[lgray,line width=1.5pt,->] 
(-2,0) node[below,scale=0.6] {\color{black} $i_2$} -- (-1,1) node[above,scale=0.6] {\color{black} $k_2$} -- (3,1) node[right,scale=0.6] {\color{black} $j_2$};
\draw[lgray,line width=4pt,->]
(0,-1) node[below,scale=0.6] {\color{black} $\I_0$} -- (0,0.5) node[scale=0.6] {\color{black} $\K_0$} -- (0,2) node[above,scale=0.6] {\color{black} $\J_0$};
\draw[lgray,line width=4pt,->]
(2,-1) node[below,scale=0.6] {\color{black} $\I_N$} -- (2,0.5) node[scale=0.6] {\color{black} $\K_N$} -- (2,2) node[above,scale=0.6] {\color{black} $\J_N$};
\node[left] at (-2.2,1) {$x \rightarrow$};
\node[left] at (-2.2,0) {$y \rightarrow$};
}
=
\sum_{0 \leq k_1,k_2 \leq n}
\
\sum_{\K_0,\dots,\K_N \in \mathbb{N}^n}
\tikz{0.6}{
\draw[densely dotted] (1.75,0.25) arc (-45:45:{1/(2*sqrt(2))});
\node at (-1,0.5) {$\cdots$};
\draw[lgray,line width=1.5pt,->] 
(-3,1) node[left,scale=0.6] {\color{black} $i_1$} -- (1,1) node[above,scale=0.6] {\color{black} $k_1$} -- (2,0) node[below,scale=0.6] {\color{black} $j_1$};
\draw[lgray,line width=1.5pt,->] 
(-3,0) node[left,scale=0.6] {\color{black} $i_2$} -- (1,0) node[below,scale=0.6] {\color{black} $k_2$} -- (2,1) node[above,scale=0.6] {\color{black} $j_2$};
\draw[lgray,line width=4pt,->] 
(-2,-1) node[below,scale=0.6] {\color{black} $\I_0$} -- (-2,0.5) node[scale=0.6] {\color{black} $\K_0$} -- (-2,2) node[above,scale=0.6] {\color{black} $\J_0$};
\draw[lgray,line width=4pt,->] 
(0,-1) node[below,scale=0.6] {\color{black} $\I_N$} -- (0,0.5) node[scale=0.6] {\color{black} $\K_N$} -- (0,2) node[above,scale=0.6] {\color{black} $\J_N$};
\node[left] at (-3.5,1) {$x \rightarrow$};
\node[left] at (-3.5,0) {$y \rightarrow$};
}
\end{align}
which holds for any $N \geq 0$, and fixed indices $i_1,i_2,j_1,j_2, (\I_0,\dots,\I_N), (\J_0,\dots,\J_N)$. The proof of \eqref{global-RLLa} follows by $N+1$ applications of the local relation \eqref{graph-RLLa} and is standard in quantum integrable models (see, for example, \cite{JimboM}); we will not repeat it here.

The proof of \eqref{CC=} follows by making the choice $i_1 = i_2 = i$ and $j_1 = j_2 = 0$ in \eqref{global-RLLa}. On the left hand side, one finds that the summation over $k_1, k_2$ is trivialized: necessarily, $k_1 = k_2 = i$, since this is the only choice for which the $R$-vertex remains non-vanishing. The resulting $R$-vertex is of the type \eqref{R-weights-a}, with weight $1$, allowing us to suppress it. A similar argument applies to the right hand side of \eqref{global-RLLa}: the summation over $k_1,k_2$ is restricted to $k_1=k_2=0$, with the resulting $R$-vertex again being of the type \eqref{R-weights-a}. We therefore find the relation
\begin{align}
\label{finiteCC=}
\sum_{\K_0,\dots,\K_N \in \mathbb{N}^n}
\tikz{0.6}{
\node at (1,0.5) {$\cdots$};
\draw[lgray,line width=1.5pt,->]
(-1,0) node[left,scale=0.6] {\color{black} $i$} -- (3,0) node[right,scale=0.6] {\color{black} $0$};
\draw[lgray,line width=1.5pt,->] 
(-1,1) node[left,scale=0.6] {\color{black} $i$} -- (3,1) node[right,scale=0.6] {\color{black} $0$};
\draw[lgray,line width=4pt,->]
(0,-1) node[below,scale=0.6] {\color{black} $\I_0$} -- (0,0.5) node[scale=0.6] {\color{black} $\K_0$} -- (0,2) node[above,scale=0.6] {\color{black} $\J_0$};
\draw[lgray,line width=4pt,->]
(2,-1) node[below,scale=0.6] {\color{black} $\I_N$} -- (2,0.5) node[scale=0.6] {\color{black} $\K_N$} -- (2,2) node[above,scale=0.6] {\color{black} $\J_N$};
\node[left] at (-1.5,1) {$y \rightarrow$};
\node[left] at (-1.5,0) {$x \rightarrow$};
}
=
\sum_{\K_0,\dots,\K_N \in \mathbb{N}^n}
\tikz{0.6}{
\node at (-1,0.5) {$\cdots$};
\draw[lgray,line width=1.5pt,->] 
(-3,1) node[left,scale=0.6] {\color{black} $i$} -- (1,1) node[right,scale=0.6] {\color{black} $0$};
\draw[lgray,line width=1.5pt,->] 
(-3,0) node[left,scale=0.6] {\color{black} $i$} -- (1,0) node[right,scale=0.6] {\color{black} $0$};
\draw[lgray,line width=4pt,->] 
(-2,-1) node[below,scale=0.6] {\color{black} $\I_0$} -- (-2,0.5) node[scale=0.6] {\color{black} $\K_0$} -- (-2,2) node[above,scale=0.6] {\color{black} $\J_0$};
\draw[lgray,line width=4pt,->] 
(0,-1) node[below,scale=0.6] {\color{black} $\I_N$} -- (0,0.5) node[scale=0.6] {\color{black} $\K_N$} -- (0,2) node[above,scale=0.6] {\color{black} $\J_N$};
\node[left] at (-3.5,1) {$x \rightarrow$};
\node[left] at (-3.5,0) {$y \rightarrow$};
}
\end{align}
which holds independently of the value of $N$. The rows in the relation \eqref{finiteCC=} are stable under the limit $N \rightarrow \infty$, assuming that $(\I_0,\dots,\I_N), (\J_0,\dots,\J_N)$ converge to finite states in $\mathbb{V}$, at which point we recognise precisely the matrix elements of the commutation relation \eqref{CC=}.

Equations \eqref{CC<} and \eqref{CC>} follow by making a different choice of the free indices in \eqref{global-RLLa}, namely, $i_1 = i$, $i_2 = j$ with $i \not= j$ and $j_1 = j_2 = 0$. Performing this specialization in \eqref{global-RLLa} leads to a more involved relation than in the last case. The sum over $k_1,k_2$ on the right hand side of the equation continues to be restricted to $k_1 = k_2 = 0$, as previously, while the sum on the left hand side now gives rise to two possible terms. We thus obtain the relation
\begin{multline*}
\sum_{\K_0,\dots,\K_N \in \mathbb{N}^n}
\left(
\tikz{0.6}{
\draw[densely dotted] (-1.25,0.25) arc (-45:45:{1/(2*sqrt(2))});
\node at (1,0.5) {$\cdots$};
\draw[lgray,line width=1.5pt,->]
(-2,1) node[above,scale=0.6] {\color{black} $i$} -- (-1,0) node[below,scale=0.6] {\color{black} $i$} -- (3,0) node[right,scale=0.6] {\color{black} $0$};
\draw[lgray,line width=1.5pt,->] 
(-2,0) node[below,scale=0.6] {\color{black} $j$} -- (-1,1) node[above,scale=0.6] {\color{black} $j$} -- (3,1) node[right,scale=0.6] {\color{black} $0$};
\draw[lgray,line width=4pt,->]
(0,-1) node[below,scale=0.6] {\color{black} $\I_0$} -- (0,0.5) node[scale=0.6] {\color{black} $\K_0$} -- (0,2) node[above,scale=0.6] {\color{black} $\J_0$};
\draw[lgray,line width=4pt,->]
(2,-1) node[below,scale=0.6] {\color{black} $\I_N$} -- (2,0.5) node[scale=0.6] {\color{black} $\K_N$} -- (2,2) node[above,scale=0.6] {\color{black} $\J_N$};
\node[left] at (-2.2,1) {$x \rightarrow$};
\node[left] at (-2.2,0) {$y \rightarrow$};
}
+
\tikz{0.6}{
\draw[densely dotted] (-1.25,0.25) arc (-45:45:{1/(2*sqrt(2))});
\node at (1,0.5) {$\cdots$};
\draw[lgray,line width=1.5pt,->]
(-2,1) node[above,scale=0.6] {\color{black} $i$} -- (-1,0) node[below,scale=0.6] {\color{black} $j$} -- (3,0) node[right,scale=0.6] {\color{black} $0$};
\draw[lgray,line width=1.5pt,->] 
(-2,0) node[below,scale=0.6] {\color{black} $j$} -- (-1,1) node[above,scale=0.6] {\color{black} $i$} -- (3,1) node[right,scale=0.6] {\color{black} $0$};
\draw[lgray,line width=4pt,->]
(0,-1) node[below,scale=0.6] {\color{black} $\I_0$} -- (0,0.5) node[scale=0.6] {\color{black} $\K_0$} -- (0,2) node[above,scale=0.6] {\color{black} $\J_0$};
\draw[lgray,line width=4pt,->]
(2,-1) node[below,scale=0.6] {\color{black} $\I_N$} -- (2,0.5) node[scale=0.6] {\color{black} $\K_N$} -- (2,2) node[above,scale=0.6] {\color{black} $\J_N$};
\node[left] at (-2.2,1) {$x \rightarrow$};
\node[left] at (-2.2,0) {$y \rightarrow$};
}
\right)
\\
=
\sum_{\K_0,\dots,\K_N \in \mathbb{N}^n}
\tikz{0.6}{
\node at (-1,0.5) {$\cdots$};
\draw[lgray,line width=1.5pt,->] 
(-3,1) node[left,scale=0.6] {\color{black} $i$} -- (1,1) node[right,scale=0.6] {\color{black} $0$};
\draw[lgray,line width=1.5pt,->] 
(-3,0) node[left,scale=0.6] {\color{black} $j$} -- (1,0) node[right,scale=0.6] {\color{black} $0$};
\draw[lgray,line width=4pt,->] 
(-2,-1) node[below,scale=0.6] {\color{black} $\I_0$} -- (-2,0.5) node[scale=0.6] {\color{black} $\K_0$} -- (-2,2) node[above,scale=0.6] {\color{black} $\J_0$};
\draw[lgray,line width=4pt,->] 
(0,-1) node[below,scale=0.6] {\color{black} $\I_N$} -- (0,0.5) node[scale=0.6] {\color{black} $\K_N$} -- (0,2) node[above,scale=0.6] {\color{black} $\J_N$};
\node[left] at (-3.5,1) {$x \rightarrow$};
\node[left] at (-3.5,0) {$y \rightarrow$};
}
\end{multline*}
in which the two $R$-vertices appearing on the left hand side are given by \eqref{R-weights-bc}. Replacing these vertices by the rational functions in $x,y$ that they represent, and again taking $N \rightarrow \infty$ in a stable way, we recover the matrix elements of the equations
\begin{align*}
\frac{q(1-y/x)}{(1-q y/x)}
\C_i(x) \C_j(y)
+
\frac{(1-q)}{(1-qy/x)}
\C_j(x) \C_i(y)
&=
\C_j(y) \C_i(x),
\qquad
i < j,
\\
\frac{(1-y/x)}{(1-q y/x)}
\C_i(x) \C_j(y)
+
\frac{(1-q)y/x}{(1-qy/x)}
\C_j(x) \C_i(y)
&=
\C_j(y) \C_i(x),
\qquad
i > j.
\end{align*}
After rearrangement, these become precisely \eqref{CC<} and \eqref{CC>}, respectively.

\end{proof}

\begin{rmk}{\rm
The commutation relations of Theorem \ref{thm:CC} are familiar in the algebraic Bethe Ansatz approach to solvable lattice models; see \cite{Faddeev,KorepinBI}. The first relation \eqref{CC=} just expresses the fact that operators of the same type commute, regardless of their rapidity arguments. The second and third relations \eqref{CC<}, \eqref{CC>} are of the same form as identities used to compute the action of a transfer matrix on a Bethe vector (see, for example, \cite[Section 4]{Faddeev}).

We will use these relations to derive properties of the non-symmetric rational functions $f_{\mu}$, such as symmetry under the interchange of two variables (via \eqref{CC=}) or their transformations under the action of Hecke divided-difference operators (via \eqref{CC<} and \eqref{CC>}).
}
\end{rmk}

\begin{thm}
Fix two nonnegative integers $i,j$ such that $0 \leq i,j \leq n$, and complex parameters $x,y$ such that
\begin{align}
\label{weight-condition}
\left|
\frac{x-s}{1-sx}
\cdot
\frac{y-s}{1-sy}
\right|
<
1.
\end{align}
The row operators \eqref{C-row} and \eqref{B-row} obey the following commutation relations:
\begin{align}
\label{CB=}
\C_i(x) \B_i(y) &=
\frac{1-q^{-1}}{1-xy} \sum_{k < i} \B_k(y) \C_k(x)
+
\B_i(y) \C_i(x)
-
\frac{(1-q)xy}{1-xy} \sum_{k > i} \B_k(y) \C_k(x),
\\
\label{CB<}
\C_i(x) \B_j(y) &= \frac{1-qxy}{1-xy} \B_j(y) \C_i(x),
\qquad
i<j,
\\
\label{CB>}
q\, \C_i(x) \B_j(y) &= \frac{1-qxy}{1-xy} \B_j(y) \C_i(x),
\qquad
i>j.
\end{align}
\end{thm}

\begin{proof}
As in the proof of Theorem \ref{thm:CC}, we begin by writing an $(N+1)$-site version of the appropriate intertwining equation, in this case \eqref{graph-RLLb}:
\begin{align}
\label{global-RLLb}
\sum_{0 \leq k_1,k_2 \leq n}
\
\sum_{\K_0,\dots,\K_N \in \mathbb{N}^n}
\tikz{0.6}{
\draw[densely dotted] (1.75,0.75) arc (45:135:{1/(2*sqrt(2))});
\node at (-1,0.5) {$\cdots$};
\draw[lgray,line width=1.5pt,<-] 
(-3,1) node[left,scale=0.6] {\color{black} $j_2$} -- (1,1) node[above,scale=0.6] {\color{black} $k_2$} -- (2,0) node[below,scale=0.6] {\color{black} $i_2$};
\draw[lgray,line width=1.5pt,->] 
(-3,0) node[left,scale=0.6] {\color{black} $i_1$} -- (1,0) node[below,scale=0.6] {\color{black} $k_1$} -- (2,1) node[above,scale=0.6] {\color{black} $j_1$};
\draw[lgray,line width=4pt,->] 
(-2,-1) node[below,scale=0.6] {\color{black} $\I_0$} -- (-2,0.5) node[scale=0.6] {\color{black} $\K_0$} -- (-2,2) node[above,scale=0.6] {\color{black} $\J_0$};
\draw[lgray,line width=4pt,->] 
(0,-1) node[below,scale=0.6] {\color{black} $\I_N$} -- (0,0.5) node[scale=0.6] {\color{black} $\K_N$} -- (0,2) node[above,scale=0.6] {\color{black} $\J_N$};
\node[left] at (-3.5,1) {$y \leftarrow$};
\node[left] at (-3.5,0) {$x \rightarrow$};
}
=
\sum_{0 \leq k_1,k_2 \leq n}
\
\sum_{\K_0,\dots,\K_N \in \mathbb{N}^n}
\tikz{0.6}{
\draw[densely dotted] (-1.25,0.75) arc (45:135:{1/(2*sqrt(2))});
\node at (1,0.5) {$\cdots$};
\draw[lgray,line width=1.5pt,<-]
(-2,1) node[above,scale=0.6] {\color{black} $j_2$} -- (-1,0) node[below,scale=0.6] {\color{black} $k_2$} -- (3,0) node[right,scale=0.6] {\color{black} $i_2$};
\draw[lgray,line width=1.5pt,->] 
(-2,0) node[below,scale=0.6] {\color{black} $i_1$} -- (-1,1) node[above,scale=0.6] {\color{black} $k_1$} -- (3,1) node[right,scale=0.6] {\color{black} $j_1$};
\draw[lgray,line width=4pt,->]
(0,-1) node[below,scale=0.6] {\color{black} $\I_0$} -- (0,0.5) node[scale=0.6] {\color{black} $\K_0$} -- (0,2) node[above,scale=0.6] {\color{black} $\J_0$};
\draw[lgray,line width=4pt,->]
(2,-1) node[below,scale=0.6] {\color{black} $\I_N$} -- (2,0.5) node[scale=0.6] {\color{black} $\K_N$} -- (2,2) node[above,scale=0.6] {\color{black} $\J_N$};
\node[left] at (-2.2,1) {$y \leftarrow$};
\node[left] at (-2.2,0) {$x \rightarrow$};
}
\end{align}
where the $R$-vertex on either side of the equation carries spectral parameter $\b{q}\b{x}\b{y}$.

Let us commence with the proof of \eqref{CB<} and \eqref{CB>}. These identities follow from \eqref{global-RLLb} by choosing the external indices to be $i_1 = i$, $j_2 = j$ with $i \not= j$, and $i_2 = j_1 = 0$. Since $i \not= j$, one finds that the right hand side of \eqref{global-RLLb} reduces to a single term, as $k_1 = i$, $k_2 = j$ are the only possible choices giving a non-vanishing $R$-vertex. The left hand side of \eqref{global-RLLb} does not reduce to a unique term, but it does simplify to a single sum, since we require $k_1 = k_2$ by conservation arguments. Putting all of this together, we arrive at the identity
\begin{align}
\label{pre-CB<>}
\sum_{0 \leq k \leq n}
\
\sum_{\K_0,\dots,\K_N \in \mathbb{N}^n}
\tikz{0.6}{
\draw[densely dotted] (1.75,0.75) arc (45:135:{1/(2*sqrt(2))});
\node at (-1,0.5) {$\cdots$};
\draw[lgray,line width=1.5pt,<-] 
(-3,1) node[left,scale=0.6] {\color{black} $j$} -- (1,1) node[above,scale=0.6] {\color{black} $k$} -- (2,0) node[below,scale=0.6] {\color{black} $0$};
\draw[lgray,line width=1.5pt,->] 
(-3,0) node[left,scale=0.6] {\color{black} $i$} -- (1,0) node[below,scale=0.6] {\color{black} $k$} -- (2,1) node[above,scale=0.6] {\color{black} $0$};
\draw[lgray,line width=4pt,->] 
(-2,-1) node[below,scale=0.6] {\color{black} $\I_0$} -- (-2,0.5) node[scale=0.6] {\color{black} $\K_0$} -- (-2,2) node[above,scale=0.6] {\color{black} $\J_0$};
\draw[lgray,line width=4pt,->] 
(0,-1) node[below,scale=0.6] {\color{black} $\I_N$} -- (0,0.5) node[scale=0.6] {\color{black} $\K_N$} -- (0,2) node[above,scale=0.6] {\color{black} $\J_N$};
\node[left] at (-3.5,1) {$y \leftarrow$};
\node[left] at (-3.5,0) {$x \rightarrow$};
}
=
\sum_{\K_0,\dots,\K_N \in \mathbb{N}^n}
\tikz{0.6}{
\draw[densely dotted] (-1.25,0.75) arc (45:135:{1/(2*sqrt(2))});
\node at (1,0.5) {$\cdots$};
\draw[lgray,line width=1.5pt,<-]
(-2,1) node[above,scale=0.6] {\color{black} $j$} -- (-1,0) node[below,scale=0.6] {\color{black} $j$} -- (3,0) node[right,scale=0.6] {\color{black} $0$};
\draw[lgray,line width=1.5pt,->] 
(-2,0) node[below,scale=0.6] {\color{black} $i$} -- (-1,1) node[above,scale=0.6] {\color{black} $i$} -- (3,1) node[right,scale=0.6] {\color{black} $0$};
\draw[lgray,line width=4pt,->]
(0,-1) node[below,scale=0.6] {\color{black} $\I_0$} -- (0,0.5) node[scale=0.6] {\color{black} $\K_0$} -- (0,2) node[above,scale=0.6] {\color{black} $\J_0$};
\draw[lgray,line width=4pt,->]
(2,-1) node[below,scale=0.6] {\color{black} $\I_N$} -- (2,0.5) node[scale=0.6] {\color{black} $\K_N$} -- (2,2) node[above,scale=0.6] {\color{black} $\J_N$};
\node[left] at (-2.2,1) {$y \leftarrow$};
\node[left] at (-2.2,0) {$x \rightarrow$};
}
\end{align}
valid for any $N \geq 0$. Next we will consider the infinite volume limit of \eqref{pre-CB<>}, when it simplifies further. This simplification is analogous to the infinite volume commutation relations obtained in \cite{Borodin,BorodinP1,WheelerZ}. Namely, as $N \rightarrow \infty$ one finds that all terms corresponding to $1 \leq k \leq n$ on the left hand side of \eqref{pre-CB<>} become vanishingly small, yielding $k=0$ as the only surviving term. To see how this works, consider padding the arbitrary states $(\I_0,\dots,\I_N)$ and $(\J_0,\dots,\J_N)$ to the right by $M$ zero vectors, as a precursor to taking $M \rightarrow \infty$. The left hand side of \eqref{pre-CB<>} becomes
\begin{align}
\label{padded}
\sum_{0 \leq k \leq n}
\
\sum_{\K_0,\dots,\K_{N+M} \in \mathbb{N}^n}
\tikz{0.6}{
\draw[densely dotted] (6.75,0.75) arc (45:135:{1/(2*sqrt(2))});
\node at (-1,0.5) {$\cdots$};
\draw[lgray,line width=1.5pt,<-] 
(-3,1) node[left,scale=0.6] {\color{black} $j$} -- (6,1) node[above,scale=0.6] {\color{black} $k$} -- (7,0) node[below,scale=0.6] {\color{black} $0$};
\draw[lgray,line width=1.5pt,->] 
(-3,0) node[left,scale=0.6] {\color{black} $i$} -- (6,0) node[below,scale=0.6] {\color{black} $k$} -- (7,1) node[above,scale=0.6] {\color{black} $0$};
\draw[lgray,line width=4pt,->] 
(-2,-1) node[below,scale=0.6] {\color{black} $\I_0$} -- (-2,0.5) node[scale=0.6] {\color{black} $\K_0$} -- (-2,2) node[above,scale=0.6] {\color{black} $\J_0$};
\draw[lgray,line width=4pt,->] 
(0,-1) node[below,scale=0.6] {\color{black} $\I_N$} -- (0,0.5) node[scale=0.6] {\color{black} $\K_N$} -- (0,2) node[above,scale=0.6] {\color{black} $\J_N$};
\foreach\x in {2,...,5}{
\draw[lgray,line width=4pt,->] 
(\x,-1) node[below,scale=0.6] {\color{black} $\bm{0}$} -- (\x,0.5) -- (\x,2) node[above,scale=0.6] {\color{black} $\bm{0}$};
}
\node[scale=0.6] at (2,0.5) {$\K_{N+1}$};
\node[scale=0.6] at (3.5,0.5) {$\cdots$};
\node[scale=0.6] at (5,0.5) {$\K_{N+M}$};
\node[scale=0.8] at (1,1) {$\star$};
\node[scale=0.8] at (1,0) {$\star$};
\node[left] at (-3.5,1) {$y \leftarrow$};
\node[left] at (-3.5,0) {$x \rightarrow$};
}
\end{align}
and for $k \geq 1$ we observe a trivialization of the ``padded'' region in \eqref{padded}: since $k \geq 1$ corresponds with a lattice path of colour $k$ being ejected/injected at the right edges of the lattice, these paths must originate somewhere. Given the empty states at the bottom and top of the padded region, we see that the paths of colour $k$ must arise at the positions marked $\star$. This, in turn, completely saturates the horizontal edges in the padded region by paths. As a consequence, all of the vertical edge states $\K_{N+1},\dots,\K_{N+M}$ are forced to assume the value $\bm{0}$. We can then compute the contribution of the padded region using the vertex weights \eqref{s-weights} and \eqref{dual-s-weights}: it produces an overall factor of
\begin{align*}
\left( \frac{x-s}{1-sx} \cdot \frac{y-s}{1-sy} \right)^M,
\end{align*}
which vanishes as $M \rightarrow \infty$, due to the assumption \eqref{weight-condition}. It follows that, as $(\I_0,\dots,\I_N), (\J_0,\dots,\J_N)$ converge to finite states in $\mathbb{V}$, only the term $k=0$ produces a non-zero contribution to the left hand side of \eqref{pre-CB<>}. We can then read off the following relations:
\begin{align*}
\C_i(x) \B_j(y)
&=
\frac{q(1-\b{q}\b{x}\b{y})}{(1-\b{x}\b{y})}
\B_j(y) \C_i(x),
\qquad
i<j,
\\
\C_i(x) \B_j(y)
&=
\frac{(1-\b{q}\b{x}\b{y})}{(1-\b{x}\b{y})}
\B_j(y) \C_i(x),
\qquad
i>j,
\end{align*}
which rearrange to yield \eqref{CB<} and \eqref{CB>}.

The proof of \eqref{CB=} is analogous. One begins by choosing the edge states in \eqref{global-RLLb} to be $i_1 = j_2 = i$ and $i_2 = j_1 = 0$, and it then reads
\begin{align}
\label{pre-CB=}
\sum_{0 \leq k \leq n}
\
\sum_{\K_0,\dots,\K_N \in \mathbb{N}^n}
\tikz{0.6}{
\draw[densely dotted] (1.75,0.75) arc (45:135:{1/(2*sqrt(2))});
\node at (-1,0.5) {$\cdots$};
\draw[lgray,line width=1.5pt,<-] 
(-3,1) node[left,scale=0.6] {\color{black} $i$} -- (1,1) node[above,scale=0.6] {\color{black} $k$} -- (2,0) node[below,scale=0.6] {\color{black} $0$};
\draw[lgray,line width=1.5pt,->] 
(-3,0) node[left,scale=0.6] {\color{black} $i$} -- (1,0) node[below,scale=0.6] {\color{black} $k$} -- (2,1) node[above,scale=0.6] {\color{black} $0$};
\draw[lgray,line width=4pt,->] 
(-2,-1) node[below,scale=0.6] {\color{black} $\I_0$} -- (-2,0.5) node[scale=0.6] {\color{black} $\K_0$} -- (-2,2) node[above,scale=0.6] {\color{black} $\J_0$};
\draw[lgray,line width=4pt,->] 
(0,-1) node[below,scale=0.6] {\color{black} $\I_N$} -- (0,0.5) node[scale=0.6] {\color{black} $\K_N$} -- (0,2) node[above,scale=0.6] {\color{black} $\J_N$};
\node[left] at (-3.5,1) {$y \leftarrow$};
\node[left] at (-3.5,0) {$x \rightarrow$};
}
=
\sum_{0 \leq k \leq n}
\
\sum_{\K_0,\dots,\K_N \in \mathbb{N}^n}
\tikz{0.6}{
\draw[densely dotted] (-1.25,0.75) arc (45:135:{1/(2*sqrt(2))});
\node at (1,0.5) {$\cdots$};
\draw[lgray,line width=1.5pt,<-]
(-2,1) node[above,scale=0.6] {\color{black} $i$} -- (-1,0) node[below,scale=0.6] {\color{black} $k$} -- (3,0) node[right,scale=0.6] {\color{black} $0$};
\draw[lgray,line width=1.5pt,->] 
(-2,0) node[below,scale=0.6] {\color{black} $i$} -- (-1,1) node[above,scale=0.6] {\color{black} $k$} -- (3,1) node[right,scale=0.6] {\color{black} $0$};
\draw[lgray,line width=4pt,->]
(0,-1) node[below,scale=0.6] {\color{black} $\I_0$} -- (0,0.5) node[scale=0.6] {\color{black} $\K_0$} -- (0,2) node[above,scale=0.6] {\color{black} $\J_0$};
\draw[lgray,line width=4pt,->]
(2,-1) node[below,scale=0.6] {\color{black} $\I_N$} -- (2,0.5) node[scale=0.6] {\color{black} $\K_N$} -- (2,2) node[above,scale=0.6] {\color{black} $\J_N$};
\node[left] at (-2.2,1) {$y \leftarrow$};
\node[left] at (-2.2,0) {$x \rightarrow$};
}
\end{align}
In contrast with the proof of \eqref{CB<} and \eqref{CB>}, we are forced to maintain a sum over $k$ on the right hand side of \eqref{pre-CB=}: this is due to the fact that now $i_1 = j_2 = i$, meaning that the path of colour $i$ is already conserved through the $R$-vertex on the right hand side, allowing its remaining (bottom and right) edges to be summed freely. One then transitions to infinite volume, in the same way as outlined above. The left hand side of \eqref{pre-CB=} again reduces to a single non-vanishing term; conversely, the right hand side remains a sum over $n+1$ terms, each coming with a factor to be determined from the Boltzmann weights \eqref{R-weights-bc}. Computing each of these terms, one finds the matrix elements of the equation
\begin{align*}
\C_i(x) \B_i(y)
=
\frac{(1-q) \b{q}\b{x}\b{y}}{(1-\b{x}\b{y})}
\sum_{k=0}^{i-1}
\B_k(y) \C_k(x)
+
\B_i(y) \C_i(x)
+
\frac{(1-q)}{(1-\b{x}\b{y})}
\sum_{k=i+1}^{n}
\B_k(y) \C_k(x),
\end{align*}
which is just \eqref{CB=}.

\end{proof}

\begin{rmk}
Commutation relations \eqref{CB<} and \eqref{CB>} play a key role in this work; we shall use them to derive summation identities for the rational functions $f_{\mu}$, 
$g_{\mu}$. We have included the relation \eqref{CB=} mainly for completeness; although we do not make use of it in subsequent calculations, we feel that it is nevertheless important to note that the commutation between $\C_i(x)$ and $\B_i(y)$ is not simple (in the case of matching indices, $i$). Indeed, it is this relation alone which ensures that Cauchy-type identities with factorized kernels are somewhat rare in the higher-rank setting.
\end{rmk}

\begin{thm}
Fix two nonnegative integers $i,j$ such that $0 \leq i,j \leq n$. The row operators \eqref{B-row} satisfy the exchange relations
\begin{align}
\label{BB=}
\B_i(x) \B_i(y) &= \B_i(y) \B_i(x),
\\
\label{BB<}
q\, \B_i(x) \B_j(y) &= \frac{y-qx}{y-x} \B_j(y) \B_i(x) - \frac{(1-q)x}{y-x} \B_j(x) \B_i(y),
\qquad
i<j,
\\
\label{BB>}
\B_i(x) \B_j(y) &= \frac{y-qx}{y-x} \B_j(y) \B_i(x) - \frac{(1-q)y}{y-x} \B_j(x) \B_i(y),
\qquad
i>j.
\end{align}
\end{thm}

\begin{proof}
The proof proceeds in essentially the same way as the proof of Theorem \ref{thm:CC}, so we will omit most of its details. The starting point is an $(N+1)$-site version of the intertwining equation \eqref{graph-RLLc}:
\begin{align}
\label{global-RLLc}
\sum_{0 \leq k_1,k_2 \leq n}
\
\sum_{\K_0,\dots,\K_N \in \mathbb{N}^n}
\tikz{0.6}{
\draw[densely dotted] (-1.75,0.75) arc (135:225:{1/(2*sqrt(2))});
\node at (1,0.5) {$\cdots$};
\draw[lgray,line width=1.5pt,<-]
(-2,1) node[above,scale=0.6] {\color{black} $j_1$} -- (-1,0) node[below,scale=0.6] {\color{black} $k_1$} -- (3,0) node[right,scale=0.6] {\color{black} $i_1$};
\draw[lgray,line width=1.5pt,<-] 
(-2,0) node[below,scale=0.6] {\color{black} $j_2$} -- (-1,1) node[above,scale=0.6] {\color{black} $k_2$} -- (3,1) node[right,scale=0.6] {\color{black} $i_2$};
\draw[lgray,line width=4pt,->]
(0,-1) node[below,scale=0.6] {\color{black} $\I_0$} -- (0,0.5) node[scale=0.6] {\color{black} $\K_0$} -- (0,2) node[above,scale=0.6] {\color{black} $\J_0$};
\draw[lgray,line width=4pt,->]
(2,-1) node[below,scale=0.6] {\color{black} $\I_N$} -- (2,0.5) node[scale=0.6] {\color{black} $\K_N$} -- (2,2) node[above,scale=0.6] {\color{black} $\J_N$};
\node[left] at (-2.2,1) {$x \leftarrow$};
\node[left] at (-2.2,0) {$y \leftarrow$};
}
=
\sum_{0 \leq k_1,k_2 \leq n}
\
\sum_{\K_0,\dots,\K_N \in \mathbb{N}^n}
\tikz{0.6}{
\draw[densely dotted] (1.25,0.75) arc (135:225:{1/(2*sqrt(2))});
\node at (-1,0.5) {$\cdots$};
\draw[lgray,line width=1.5pt,<-] 
(-3,1) node[left,scale=0.6] {\color{black} $j_1$} -- (1,1) node[above,scale=0.6] {\color{black} $k_1$} -- (2,0) node[below,scale=0.6] {\color{black} $i_1$};
\draw[lgray,line width=1.5pt,<-] 
(-3,0) node[left,scale=0.6] {\color{black} $j_2$} -- (1,0) node[below,scale=0.6] {\color{black} $k_2$} -- (2,1) node[above,scale=0.6] {\color{black} $i_2$};
\draw[lgray,line width=4pt,->] 
(-2,-1) node[below,scale=0.6] {\color{black} $\I_0$} -- (-2,0.5) node[scale=0.6] {\color{black} $\K_0$} -- (-2,2) node[above,scale=0.6] {\color{black} $\J_0$};
\draw[lgray,line width=4pt,->] 
(0,-1) node[below,scale=0.6] {\color{black} $\I_N$} -- (0,0.5) node[scale=0.6] {\color{black} $\K_N$} -- (0,2) node[above,scale=0.6] {\color{black} $\J_N$};
\node[left] at (-3.5,1) {$x \leftarrow$};
\node[left] at (-3.5,0) {$y \leftarrow$};
}
\end{align}
which gets specialized to $j_1 = j_2 = i$, $i_1 = i_2 = 0$ for the proof of \eqref{BB=}, and to $j_1 = i$, $j_2 = j$, $i_1 = i_2 = 0$ for the proof of \eqref{BB<} and \eqref{BB>}.

\end{proof}

\section{Coloured compositions} 

\begin{defn}
Let $\lambda = (\lambda_1,\dots,\lambda_n)$ be a composition of length $n$ and weight $m$, \ie\ such that $\sum_{i=1}^{n} \lambda_i = m$. Denote the partial sums of $\lambda$ by $\sum_{i=1}^{k} \lambda_i = \ell_k$. We introduce the set $\mathcal{S}_{\lambda}$ of $\lambda$-coloured compositions as follows:
\begin{align}
\label{lambda-col}
\index{S@$\mathcal{S}_{\lambda}$; set of $\lambda$-coloured compositions}
\mathcal{S}_{\lambda}
=
\Big\{ 
\mu 
= 
(\mu_1 \geq \cdots \geq \mu_{\ell_1} |
\mu_{\ell_1+1} \geq \cdots \geq \mu_{\ell_2} |
\cdots |
\mu_{\ell_{n-1}+1} \geq \cdots \geq \mu_{\ell_n})
\Big\}.
\end{align}
That is, the elements of $\mathcal{S}_{\lambda}$ are length-$m$ compositions $\mu$, which have been subdivided into blocks of length $\lambda_k$, $1 \leq k \leq n$. These blocks demarcate the colouring of $\mu$. Within any given block, the parts of $\mu$ have the same colouring and are weakly decreasing.
\end{defn}

\begin{ex}
Let $n=5$ and $\lambda = (1,0,3,0,2)$. Then 
\begin{align*}
\mathcal{S}_{\lambda}
=
\Big\{ 
\mu 
=
(\mu_1 | \cdot | \mu_2 \geq \mu_3 \geq \mu_4 | \cdot | \mu_5 \geq \mu_6 )
\Big\},
\end{align*}
where blocks of length zero are indicated by a dot.

\end{ex}

\begin{rmk}
\label{rmk:2cases}{\rm 
Two special cases of $\lambda$-coloured compositions will be particularly important for our purposes. The first is when $\lambda = (n,0,\dots,0)$. In this case, compositions $\mu \in \mathcal{S}_{\lambda}$ consist of a single block whose parts are weakly decreasing; \ie\ one simply recovers partitions. As we shall see, reducing to this case will allow us to recover the symmetric rational functions $\F_{\mu}$, $\G_{\mu}$ previously studied in \cite{Borodin,BorodinP1,BorodinP2}.

The second of these is when $\lambda = (1,1,\dots,1) = (1^n)$. In this situation, compositions $\mu \in \mathcal{S}_{\lambda}$ consist of $n$ blocks, each of a different colour. Since all blocks have unit length, the parts of $\mu$ are not bound by any inequalities; accordingly, one recovers the set of all length-$n$ compositions. We will sometimes refer to these as {\it rainbow compositions}, and in writing them we will omit the bars used in \eqref{lambda-col} to separate colours.
}
\end{rmk}

Let $\mu \in \mathcal{S}_{\lambda}$ be a $\lambda$-coloured composition, with $\ell_k$ denoting the partial sums of $\lambda$, as above. We associate to each such $\mu$ a vector \index{am@$\ket{\mu}_{\lambda}$} $\ket{\mu}_{\lambda} \in \mathbb{V}$, defined as follows:
\begin{align}
\label{A(k)}
\ket{\mu}_{\lambda}
:=
\bigotimes_{k=0}^{\infty}
\ket{\bm{A}(k)}_k,
\qquad
\bm{A}(k) = \sum_{j=1}^{n} A_j(k) \bm{e}_j,
\qquad
A_j(k)
=
\#\{ i : \mu_i = k,\ \ell_{j-1} +1 \leq i \leq \ell_j \},
\end{align}
where by agreement $\ell_{0} = 0$. In other words, the component $A_j(k)$ enumerates the number of parts in the $j$-th block of $\mu$ which are equal to $k$. From this one can define vector subspaces $\mathbb{V}(\lambda)$ \index{V2@$\mathbb{V}(\lambda)$; sectors of $\mathbb{V}$} which provide a natural grading of $\mathbb{V}$:
\begin{align}
\label{filter}
\mathbb{V} 
= 
\bigoplus_{m=0}^{\infty}
\bigoplus_{\lambda \in \mathcal{W}(m)}
\mathbb{V}(\lambda),
\quad\quad
\mathbb{V}(\lambda) 
:=
{\rm Span}_{\mathbb{C}} 
\Big\{ \ket{\mu}_{\lambda} \Big\}_{\mu \in \mathcal{S}_{\lambda}},
\end{align}
where the direct sum is taken over all compositions $\lambda$ of length $n$. The grading \eqref{filter} splits $\mathbb{V}$ into subspaces with fixed particle content: $\mathbb{V}(\lambda)$ is the linear span of all states consisting of $\lambda_i$ particles of type $i$, for all $1 \leq i \leq n$. We refer to these subspaces as {\it sectors} of $\mathbb{V}$.

\begin{rmk}{\rm
Let us return to the two special cases considered in Remark \ref{rmk:2cases} and study them in the state-vector language introduced above. When $\lambda = (n,0,\dots,0)$, $\mu$ is a partition and the components $A_j(k)$ in \eqref{A(k)} are given by
\begin{align*}
A_1(k) = \#\{ i : \mu_i = k \} \equiv m_k(\mu),
\qquad
A_j(k) = 0, \quad \forall\ 2 \leq j \leq n,
\end{align*}
where $m_k(\mu)$ are the multiplicities of the partition $\mu$. The resulting vectors $\ket{\mu}_{(n,0,\dots,0)}$ are precisely the partition states used in the algebraic construction of the Hall--Littlewood polynomials \cite{Tsilevich,WheelerZ} and the higher-spin rational symmetric functions which generalize them \cite{Borodin,BorodinP1,BorodinP2}.

When $\lambda = (1^n)$, $\mu$ is a generic composition, and the components $A_j(k)$ in \eqref{A(k)} are given by
\begin{align}
\label{statesA}
A_j(k) 
=
\left\{
\begin{array}{ll}
1, \quad & \mu_j = k,
\\
0, \quad & {\rm otherwise}.
\end{array}
\right.
\end{align}

\smallskip
\noindent
The resulting vectors $\ket{\mu}_{(1^n)}$ span the {\it rainbow sector} \index{V3@$\mathbb{V}(1^n)$; rainbow sector} $\mathbb{V}(1^n) \subset \mathbb{V}$; we will use it almost exclusively in this work. We often write the basis elements of $\mathbb{V}(1^n)$ as \index{am@$\ket{\mu}$; composition states} $\ket{\mu}_{(1^n)} \equiv \ket{\mu}$, dropping the unnecessary subscript.
}
\end{rmk}

%

\begin{ex}
Let $n=4$, $\lambda = (1,1,1,1)$ and $\mu = (2,0,2,1)$. We find that
\begin{align*}
\ket{\mu}
=
\bigotimes_{k=0}^{\infty}
\ket{\bm{A}(k)}_k,
\quad
\bm{A}(0) = \bm{e}_2,
\quad
\bm{A}(1) = \bm{e}_4,
\quad
\bm{A}(2) = \bm{e}_1 + \bm{e}_3,
\quad
\bm{A}(k) = \bm{0},\ \forall\ k \geq 3,
\end{align*}
or more explicitly,
\begin{align*}
\ket{\mu} = \ket{0,1,0,0}_0 \otimes \ket{0,0,0,1}_1 \otimes \ket{1,0,1,0}_2 \otimes \ket{0,0,0,0}_3 \otimes \cdots.
\end{align*}
\end{ex}

\section{The rational non-symmetric functions $f_{\mu}$ and $g_{\mu}$}

We now come to the definition of the non-symmetric functions which are the principal subject of this work. They depend on an indexing composition $\mu = (\mu_1,\dots,\mu_m)$ of length $m$, and are rational functions in the parameters $q$, $s$, as well as the alphabet $(x_1,\dots,x_m)$.  

\begin{defn}[Generic-sector rational functions]
\label{defn:gen-sector}
Let $\lambda = (\lambda_1,\dots,\lambda_n)$ be a composition of weight $m$, and fix a $\lambda$-coloured composition $\mu = (\mu_1,\dots,\mu_m) \in \mathcal{S}_{\lambda}$. Write $\ell_k = \sum_{i=1}^{k} \lambda_i$ for the $k$-th partial sum of $\lambda$. We define a family of non-symmetric rational functions as matrix elements of products of the row operators \eqref{C-row}:
\begin{align}
\label{generic-f}
\index{f@$f_{\mu}(\lambda;x_1,\dots,x_m)$}
f_{\mu}(\lambda; x_1,\dots,x_m)
:=
\bra{\varnothing}
\left(
\prod_{i=1}^{\ell_1}
\C_1(x_i)
\right)
\left(
\prod_{i=\ell_1+1}^{\ell_2}
\C_2(x_i)
\right)
\cdots
\left(
\prod_{i=\ell_{n-1}+1}^{\ell_n}
\C_n(x_i)
\right)
\ket{\mu}_{\lambda},
\end{align}
where $\ket{\mu}_{\lambda} \in \mathbb{V}(\lambda)$ is given by \eqref{A(k)} and $\bra{\varnothing} \in \mathbb{V}^{*}$ denotes the (dual) vacuum state
\begin{align}
\label{dual-vac}
\bra{\varnothing}
=
\bigotimes_{k=0}^{\infty}
\bra{\bm{0}}_k,
\end{align}
which is completely devoid of particles.
\end{defn}

Apart from specifying the sector to which $\ket{\mu}_{\lambda}$ belongs, $\lambda$ plays a further role in the expectation value \eqref{generic-f}: $\lambda_k$ counts the number of $\C_k$ operators present. In view of the commutativity \eqref{CC=} of operators with a common subscript, we see that $\lambda$ then encodes the partial symmetries of the functions \eqref{generic-f} in the alphabet $(x_1,\dots,x_m)$. Namely, $f_{\mu}(\lambda;x_1,\dots,x_m)$ is individually symmetric in each of the subsets of variables $(x_{\ell_{k-1}+1},\dots,x_{\ell_k})$, $1 \leq k \leq n$. 

\begin{rmk}
When $\lambda = (n,0,\dots,0)$, only $\mathcal{C}_1$ operators appear in \eqref{generic-f} and it becomes a purely rank-1 quantity, which is symmetric in the full alphabet $(x_1,\dots,x_n)$. In fact, one finds that
\begin{align*}
f_{\mu}(\lambda; x_1,\dots,x_n)
=
\F^{\sf c}_{\mu}(x_1,\dots,x_n),
\index{F@$\F^{\sf c}_{\mu}$; symmetric spin Hall--Littlewood}
\qquad
\lambda = (n,0,\dots,0),
\end{align*}
where $\mu$ is a partition and $\F^{\sf c}_{\mu}(x_1,\dots,x_n)$ denotes a symmetric spin Hall--Littlewood function, as defined in \cite[Section 4]{BorodinP2}:
\begin{align}
\label{s-HL}
\F^{\sf c}_{\mu}(x_1,\dots,x_n)
=
\tikz{0.8}{
\foreach\y in {1,...,5}{
\draw[lgray,line width=1.5pt,->] (1,\y) -- (8,\y);
}
\foreach\x in {2,...,7}{
\draw[lgray,line width=4pt,->] (\x,0) -- (\x,6);
}
\foreach\y in {1,...,5}{
\foreach\x in {2,...,7}{
\node at (\x,\y) {$\bullet$};
}}
\node[left] at (0.5,1) {$x_1 \rightarrow$};
\node[left] at (0.5,2) {$x_2 \rightarrow$};
\node[left] at (0.5,3) {$\vdots$};
\node[left] at (0.5,4) {$\vdots$};
\node[left] at (0.5,5) {$x_n \rightarrow$};
\node[below] at (7,0) {$\cdots$};
\node[below] at (6,0) {$\cdots$};
\node[below] at (5,0) {$\cdots$};
\node[below] at (4,0) {$0$};
\node[below] at (3,0) {$0$};
\node[below] at (2,0) {$0$};
\node[above] at (7,6) {$\cdots$};
\node[above] at (6,6) {$\cdots$};
\node[above] at (5,6) {$\cdots$};
\node[above] at (4,6) {$m_2$};
\node[above] at (3,6) {$m_1$};
\node[above] at (2,6) {$m_0$};
\node[right] at (8,1) {$0$};
\node[right] at (8,2) {$0$};
\node[right] at (8,3) {$\vdots$};
\node[right] at (8,4) {$\vdots$};
\node[right] at (8,5) {$0$};
\node[left] at (1,1) {$1$};
\node[left] at (1,2) {$1$};
\node[left] at (1,3) {$\vdots$};
\node[left] at (1,4) {$\vdots$};
\node[left] at (1,5) {$1$};
},
\qquad\qquad
m_i \equiv m_i(\mu).
\end{align}
\end{rmk}

\begin{defn}[Non-symmetric spin Hall--Littlewood functions]
Fix a rainbow composition $\mu = (\mu_1,\dots,\mu_n)$. The non-symmetric spin Hall--Littlewood function $f_{\mu}(x_1,\dots,x_n)$ is defined as the $\lambda = (1^n)$ specialization of \eqref{generic-f}:
\begin{align}
\label{f-HL}
\index{f1@$f_{\mu}$; non-symmetric spin Hall--Littlewood}
f_{\mu}(x_1,\dots,x_n)
:=
\bra{\varnothing}
\C_1(x_1)
\C_2(x_2)
\cdots
\C_n(x_n)
\ket{\mu},
\end{align}
where $\ket{\mu} \in \V$ and $\bra{\varnothing} \in \mathbb{V}^{*}$ is as defined in \eqref{dual-vac}. Notice that we drop $\lambda$ from the arguments of $f_{\mu}$ as its value is now determined; this also distinguishes the non-symmetric spin Hall--Littlewood functions notationally from \eqref{generic-f}.
\end{defn}

As we will show in Chapter \ref{sec:properties}, the functions \eqref{f-HL} generalize the {\it non-symmetric Hall--Littlewood polynomials} \cite{DescouensL} via the inclusion of the parameter $s$; accordingly, we refer to them as {\it non-symmetric spin Hall--Littlewood functions}. They are the ``maximally asymmetric'' members of the general family \eqref{generic-f}: since no $\C_k$ operator appears twice in \eqref{f-HL}, they are not in general preserved under permutations of the variables $(x_1,\dots,x_n)$, although partial symmetries may still exist. In this sense, the symmetric rational functions of \cite{Borodin,BorodinP1,BorodinP2} live at one end of the spectrum of functions introduced in Definition \ref{defn:gen-sector}, while the non-symmetric functions \eqref{f-HL} are situated at the other extreme.

The functions $f_{\mu}(\lambda;x_1,\dots,x_m)$ admit a convenient description as partition functions in the vertex model \eqref{s-weights}. We will focus on this graphical description in the case of the rainbow sector $\lambda = (1^n)$. To recover the desired partition function, one simply replaces each $\mathcal{C}_i(x_i)$ operator in \eqref{f-HL} by its interpretation as a row of vertices, with left-to-right ordering in \eqref{f-HL} corresponding to bottom-to-top ordering of rows in the partition function. We obtain the expression
\begin{align}
\label{f-def}
f_{\mu}(x_1,\dots,x_n)
&=
\tikz{0.8}{
\foreach\y in {1,...,5}{
\draw[lgray,line width=1.5pt,->] (1,\y) -- (8,\y);
}
\foreach\x in {2,...,7}{
\draw[lgray,line width=4pt,->] (\x,0) -- (\x,6);
}
\node[left] at (0.5,1) {$x_1 \rightarrow$};
\node[left] at (0.5,2) {$x_2 \rightarrow$};
\node[left] at (0.5,3) {$\vdots$};
\node[left] at (0.5,4) {$\vdots$};
\node[left] at (0.5,5) {$x_n \rightarrow$};
\node[below] at (7,0) {$\cdots$};
\node[below] at (6,0) {$\cdots$};
\node[below] at (5,0) {$\cdots$};
\node[below] at (4,0) {\footnotesize$\bm{0}$};
\node[below] at (3,0) {\footnotesize$\bm{0}$};
\node[below] at (2,0) {\footnotesize$\bm{0}$};
\node[above] at (7,6) {$\cdots$};
\node[above] at (6,6) {$\cdots$};
\node[above] at (5,6) {$\cdots$};
\node[above] at (4,6) {\footnotesize$\bm{A}(2)$};
\node[above] at (3,6) {\footnotesize$\bm{A}(1)$};
\node[above] at (2,6) {\footnotesize$\bm{A}(0)$};
\node[right] at (8,1) {$0$};
\node[right] at (8,2) {$0$};
\node[right] at (8,3) {$\vdots$};
\node[right] at (8,4) {$\vdots$};
\node[right] at (8,5) {$0$};
\node[left] at (1,1) {$1$};
\node[left] at (1,2) {$2$};
\node[left] at (1,3) {$\vdots$};
\node[left] at (1,4) {$\vdots$};
\node[left] at (1,5) {$n$};
}
\end{align}
where the state $\bm{A}(k)$ at the top of the $k$-th column is the $k$-th vector in the infinite tensor product formula for $\ket{\mu}_{\lambda}$, as given by \eqref{A(k)}, \eqref{statesA}. The vertices in the $i$-th row of the lattice (counted from the bottom) carry their own spectral parameter $x_i$, as indicated. 

\begin{prop}
\label{prop:f-sum-F}
Fix a partition $\nu$, let $\F^{\sf c}_{\nu}(x_1,\dots,x_n)$ denote a symmetric spin Hall--Littlewood function, as given by \eqref{s-HL}, and $f_{\mu}(x_1,\dots,x_n)$ its non-symmetric analogue \eqref{f-def}. We then have the symmetrization identity
\begin{align}
\label{f-sum-F}
\sum_{\mu: \mu^{+} = \nu}
f_{\mu}(x_1,\dots,x_n)
=
\F^{\sf c}_{\nu}(x_1,\dots,x_n),
\end{align}
with the sum taken over all compositions $\mu$ with dominant reordering $\mu^{+} = \nu$.
\end{prop}

\begin{proof}
The summation $\sum_{\mu: \mu^{+} = \nu}$ effectively sums over all ways of colouring the $n$ paths that leave the partition function \eqref{f-def} via its top edges, while keeping fixed the positions where paths leave. Such a summation is precisely the sort which gives rise to the colour-blindness phenomenon of Proposition \ref{prop:colour-blind}. The result \eqref{f-sum-F} then follows by successive application of the relations \eqref{cb1}, \eqref{cb2}.  
\end{proof}

\begin{defn}[Dual generic-sector rational functions]
Let $\lambda = (\lambda_1,\dots,\lambda_n)$ be a composition of weight $m$, and fix a $\lambda$-coloured composition $\mu = (\mu_1,\dots,\mu_m) \in \mathcal{S}_{\lambda}$. Write $\ell_k = \sum_{i=1}^{k} \lambda_i$ for the $k$-th partial sum of $\lambda$. Define a further family of non-symmetric rational functions, this time as matrix elements of products of the row operators \eqref{B-row}:
\begin{align}
\label{generic-g}
\index{g@$g_{\mu}(\lambda; x_1,\dots,x_m)$}
g_{\mu}(\lambda; x_1,\dots,x_m)
:=
\bra{\mu}_{\lambda}
\left(
\prod_{i=1}^{\ell_1}
\B_1(x_i)
\right)
\left(
\prod_{i=\ell_1+1}^{\ell_2}
\B_2(x_i)
\right)
\cdots
\left(
\prod_{i=\ell_{n-1}+1}^{\ell_n}
\B_n(x_i)
\right)
\ket{\varnothing},
\end{align}
where $\bra{\mu}_{\lambda} \in \mathbb{V}^{*}(\lambda)$ is the dual of the vector \eqref{A(k)}, and \index{am@$\ket{\varnothing}$; vacuum state} $\ket{\varnothing} \in \mathbb{V}$ denotes the vacuum state
\begin{align}
\label{vac}
\ket{\varnothing}
=
\bigotimes_{k=0}^{\infty}
\ket{\bm{0}}_k.
\end{align}

\end{defn}

\begin{defn}[Dual non-symmetric spin Hall--Littlewood functions]
Fix a rainbow composition $\mu = (\mu_1,\dots,\mu_n)$. The dual non-symmetric spin Hall--Littlewood function $g_{\mu}(x_1,\dots,x_n)$ is defined as the $\lambda = (1^n)$ specialization of \eqref{generic-g}:
\begin{align}
\label{g-HL}
g_{\mu}(x_1,\dots,x_n)
:=
\bra{\mu}
\B_1(x_1)
\B_2(x_2)
\cdots
\B_n(x_n)
\ket{\varnothing},
\end{align}
where $\bra{\mu} \in \Vs$ and $\ket{\varnothing} \in \mathbb{V}$ is as defined in \eqref{vac}.
\end{defn}

Similarly to above, the functions $g_{\mu}(\lambda;x_1,\dots,x_m)$ also admit a description as partition functions, now in the vertex model \eqref{dual-s-weights}. Let us again focus on this graphical description in the case of the rainbow sector, $\lambda = (1^n)$. Replacing each $\B_i(x_i)$ operator in \eqref{g-HL} by the appropriate row of vertices, as dictated by \eqref{B-row}, we find that
\index{g1@$g_\mu$; dual non-symmetric spin Hall--Littlewood}
\begin{align}
\label{g-def}
g_{\mu}(x_1,\dots,x_n)
&=
\tikz{0.8}{
\foreach\y in {1,...,5}{
\draw[lgray,line width=1.5pt,<-] (1,\y) -- (8,\y);
}
\foreach\x in {2,...,7}{
\draw[lgray,line width=4pt,->] (\x,0) -- (\x,6);
}
\node[left] at (0.5,1) {$x_1 \leftarrow$};
\node[left] at (0.5,2) {$x_2 \leftarrow$};
\node[left] at (0.5,3) {$\vdots$};
\node[left] at (0.5,4) {$\vdots$};
\node[left] at (0.5,5) {$x_n \leftarrow$};
\node[above] at (7,6) {$\cdots$};
\node[above] at (6,6) {$\cdots$};
\node[above] at (5,6) {$\cdots$};
\node[above] at (4,6) {\footnotesize$\bm{0}$};
\node[above] at (3,6) {\footnotesize$\bm{0}$};
\node[above] at (2,6) {\footnotesize$\bm{0}$};
\node[below] at (7,0) {$\cdots$};
\node[below] at (6,0) {$\cdots$};
\node[below] at (5,0) {$\cdots$};
\node[below] at (4,0) {\footnotesize$\bm{A}(2)$};
\node[below] at (3,0) {\footnotesize$\bm{A}(1)$};
\node[below] at (2,0) {\footnotesize$\bm{A}(0)$};
\node[right] at (8,1) {$0$};
\node[right] at (8,2) {$0$};
\node[right] at (8,3) {$\vdots$};
\node[right] at (8,4) {$\vdots$};
\node[right] at (8,5) {$0$};
\node[left] at (1,1) {$1$};
\node[left] at (1,2) {$2$};
\node[left] at (1,3) {$\vdots$};
\node[left] at (1,4) {$\vdots$};
\node[left] at (1,5) {$n$};
}
\end{align}
where the state $\bm{A}(k)$ at the bottom of the $k$-th column is given by \eqref{A(k)}, \eqref{statesA}.

\begin{rmk}{\rm
Both of the partition functions \eqref{f-def} and \eqref{g-def} are manifestly sums of non-vanishing terms. One can view the states $\{1,\dots,n\}$ on the left edges of \eqref{f-def} as incoming lattice paths of each of the $n$ possible colours, while the composition state at the top of \eqref{f-def} represents the departure of these paths from the lattice. Since the total flux of paths through the lattice is conserved, each global configuration in \eqref{f-def} can be decomposed into a union of single vertex configurations, with path conservation being respected at every individual vertex. This ensures that the weight of any such global configuration will be non-zero. A similar argument applies to \eqref{g-def}. 

Both partition functions are meaningful, despite being defined on a semi-infinite lattice, since in any configuration of the lattice paths in \eqref{f-def} and \eqref{g-def}, the only vertices which occur infinitely many times are 
$
\tikz{0.3}{
\draw[lgray,line width=0.5pt,->] (-1,0) -- (1,0);
\draw[lgray,line width=2pt,->] (0,-1) -- (0,1);
\node[left] at (-1,0) {\tiny $0$};\node[right] at (1,0) {\tiny $0$};
\node[below] at (0,-1) {\tiny $\bm{0}$};\node[above] at (0,1) {\tiny $\bm{0}$};
}$ 
and 
$
\tikz{0.3}{
\draw[lgray,line width=0.5pt,<-] (-1,0) -- (1,0);
\draw[lgray,line width=2pt,->] (0,-1) -- (0,1);
\node[left] at (-1,0) {\tiny $0$};\node[right] at (1,0) {\tiny $0$};
\node[below] at (0,-1) {\tiny $\bm{0}$};\node[above] at (0,1) {\tiny $\bm{0}$};
}$, 
both having weight $1$.}
\end{rmk}


\begin{ex}
Let $n=2$ and $\mu = (2,1)$. The vector $\ket{\mu}$ is given by
\begin{align*}
\ket{\mu}
=
\bigotimes_{k=0}^{\infty}
\ket{\bm{A}(k)}_k,
\quad
\bm{A}(0) = \bm{0},
\quad
\bm{A}(1) = \bm{e}_2,
\quad
\bm{A}(2) = \bm{e}_1,
\quad
\bm{A}(k) = \bm{0},\ \forall\ k \geq 3,
\end{align*}
or more explicitly, 
$
\ket{\mu} = \ket{0,0}_0 \otimes \ket{0,1}_1 \otimes \ket{1,0}_2 \otimes \ket{0,0}_3 \otimes \cdots
$. Calculating $f_{\mu}(x_1,x_2)$ from its representation \eqref{f-def} as a partition function consisting of two rows, we find two possible lattice configurations:
\begin{align*}
\tikz{0.8}{
\foreach\y in {4,5}{
\draw[lgray,line width=1.5pt] (2,\y) -- (8,\y);
}
\foreach\x in {3,...,7}{
\draw[lgray,line width=4pt] (\x,3) -- (\x,6);
}
\node[left] at (1.5,4) {$x_1 \rightarrow$};
\node[left] at (1.5,5) {$x_2 \rightarrow$};
\node[above] at (7,6) {$\cdots$};
\node[above] at (6,6) {\footnotesize$(0,0)$};
\node[above] at (5,6) {\footnotesize$(1,0)$};
\node[above] at (4,6) {\footnotesize$(0,1)$};
\node[above] at (3,6) {\footnotesize$(0,0)$};
\node[below] at (7,3) {$\cdots$};
\node[below] at (6,3) {\footnotesize$(0,0)$};
\node[below] at (5,3) {\footnotesize$(0,0)$};
\node[below] at (4,3) {\footnotesize$(0,0)$};
\node[below] at (3,3) {\footnotesize$(0,0)$};
\node[left] at (2,4) {$1$};
\node[left] at (2,5) {$2$};
\node[right] at (8,4) {$0$};
\node[right] at (8,5) {$0$};
\draw[ultra thick,blue,->] (2,4) -- (5,4) -- (5,6);
\draw[ultra thick,green,->] (2,5) -- (4,5) -- (4,6);
}
\quad\quad
\tikz{0.8}{
\foreach\y in {4,5}{
\draw[lgray,line width=1.5pt] (2,\y) -- (8,\y);
}
\foreach\x in {3,...,7}{
\draw[lgray,line width=4pt] (\x,3) -- (\x,6);
}
\node[left] at (1.5,4) {$x_1 \rightarrow$};
\node[left] at (1.5,5) {$x_2 \rightarrow$};
\node[above] at (7,6) {$\cdots$};
\node[above] at (6,6) {\footnotesize$(0,0)$};
\node[above] at (5,6) {\footnotesize$(1,0)$};
\node[above] at (4,6) {\footnotesize$(0,1)$};
\node[above] at (3,6) {\footnotesize$(0,0)$};
\node[below] at (7,3) {$\cdots$};
\node[below] at (6,3) {\footnotesize$(0,0)$};
\node[below] at (5,3) {\footnotesize$(0,0)$};
\node[below] at (4,3) {\footnotesize$(0,0)$};
\node[below] at (3,3) {\footnotesize$(0,0)$};
\node[left] at (2,4) {$1$};
\node[left] at (2,5) {$2$};
\node[right] at (8,4) {$0$};
\node[right] at (8,5) {$0$};
\draw[ultra thick,blue,->] (2,4) -- (4,4) -- (4,5) -- (5,5) -- (5,6);
\draw[ultra thick,green,->] (2,5) -- (4,5) -- (4,6);
}
\end{align*}
where the blue (darker) line represents a path of colour 1 and the green (lighter) line represents a path of colour 2. Summing the weights of the two configurations, one obtains $f_{\mu}(x_1,x_2)$ explicitly:
\begin{align*}
f_{\mu}(x_1,x_2)
&=
\left(\frac{x_1-s}{1-s x_1} \right)^2 
\left(\frac{1-s^2}{1-s x_1} \right)
\left(\frac{x_2-s}{1-s x_2} \right)
\left(\frac{1-s^2}{1-s x_2} \right)
\left(\frac{1-s q x_2}{1-s x_2} \right)
\\
&+
\left(\frac{x_1-s}{1-s x_1} \right) 
\left(\frac{1-s^2}{1-s x_1} \right)
\left(\frac{x_2-s}{1-s x_2} \right)
\left(\frac{s(1-q)}{1-s x_2} \right)
\left(\frac{1-s^2}{1-s x_2} \right).
\end{align*}
\end{ex}

\section{Permuted boundary conditions}

As we have defined them, the non-symmetric spin Hall--Littlewood functions \eqref{f-HL} feature a product of $\C_i$ operators with strictly increasing index $i$, as one reads from left to right. At the level of the partition function \eqref{f-def}, this means that the paths of colour $1 \leq i \leq n$ come in sequentially as one reads the left boundary edges from bottom to top. 

It will sometimes be useful to relax this ordering, allowing the colours to enter the lattice \eqref{f-def} via any permutation of $\{1,\dots,n\}$.  To this effect, we make the following definition:
\begin{defn}
Let $\mu = (\mu_1,\dots,\mu_n)$ be a rainbow composition and fix a permutation $\sigma \in \mathfrak{S}_n$. The $\sigma$-permuted non-symmetric spin Hall--Littlewood functions $f^{\sigma}_{\mu}(x_1,\dots,x_n)$ and $g^{\sigma}_{\mu}(x_1,\dots,x_n)$ \index{f2@$f^{\sigma}_{\mu}$} \index{g2@$g^{\sigma}_{\mu}$} are defined as follows:
\begin{align}
\label{sigma-f}
f^{\sigma}_{\mu}(x_1,\dots,x_n)
&:=
\bra{\varnothing}
\C_{\sigma(1)}(x_1)
\C_{\sigma(2)}(x_2)
\cdots
\C_{\sigma(n)}(x_n)
\ket{\mu},
\\
\label{sigma-g}
g^{\sigma}_{\mu}(x_1,\dots,x_n)
&:=
\bra{\mu}
\B_{\sigma(1)}(x_1)
\B_{\sigma(2)}(x_2)
\cdots
\B_{\sigma(n)}(x_n)
\ket{\varnothing},
\end{align}
where $\ket{\mu} \in \V$, $\bra{\mu} \in \Vs$ and $\ket{\varnothing} \in \mathbb{V}$, 
$\bra{\varnothing} \in \mathbb{V}^{*}$ are as defined in \eqref{vac}, \eqref{dual-vac}.
\end{defn}

\begin{rmk}{\rm
The permuted functions \eqref{sigma-f} and \eqref{sigma-g} are in the same vein as the ``permuted basement Macdonald polynomials'' recently introduced in \cite[Section 6]{Alexandersson}.
}
\end{rmk}

\section{Pre-fused functions}
\label{ssec:prefused}

As detailed in Appendix \ref{app:fusion}, the higher-spin model \eqref{s-weights} can be obtained from the fundamental vertex model \eqref{fund-vert} via the fusion procedure. It therefore turns out to be possible to define a partition function in the model \eqref{fund-vert} which, under appropriate specializations and analytic continuation, becomes proportional to the partition function \eqref{f-def} which defines $f_{\mu}$. This will be an important definition, since it expedites the proof of many of the subsequent properties of $f_{\mu}$; we refer to this as a ``pre-fused'' version of $f_{\mu}$.

Consider an infinite set $\{J_0,J_1,J_2,\dots\}$ of positive integers. The precise value of these integers is unimportant; they only need to be sufficiently large, and in what follows we essentially treat them as variables. We work on a semi-infinite lattice consisting of $n$ rows, where the $i$-th row (counted from the bottom of the lattice) carries the rapidity $x_i$. The columns of the lattice are arranged into ``bundles'' of cardinality $J_k$, $k \geq 0$. Within the $k$-th bundle, the $j$-th column (counted from the left) carries the rapidity $y^{(k)}_j$, $1 \leq j \leq J_k$. 

Let us now fix the boundary conditions of this lattice. As in the construction of \eqref{f-def}, we place an incoming path of colour $i$ on the left external edge of the $i$-th row, while all right external edges (situated infinitely far from the origin) are chosen to be unoccupied. The bottom external edges of the lattice are also chosen to be empty. By conservation of lattice paths, the colours $\{1,\dots,n\}$ must therefore emerge somewhere from the top external edges of the lattice. We will allow each colour to exit the lattice via any bundle; the only restriction that we impose is that within each bundle the outgoing colours are ordered increasingly. In other words, if the outgoing states of the $k$-th bundle are denoted by the set $\mathcal{A}^{(k)} = \{a^{(k)}_1,\dots,a^{(k)}_{J_k}\}$, we must have $0 \leq a^{(k)}_1 \leq \cdots \leq a^{(k)}_{J_k} \leq n$. See Figure \ref{fig:pre-fuse} for an example of this construction.

\begin{figure}
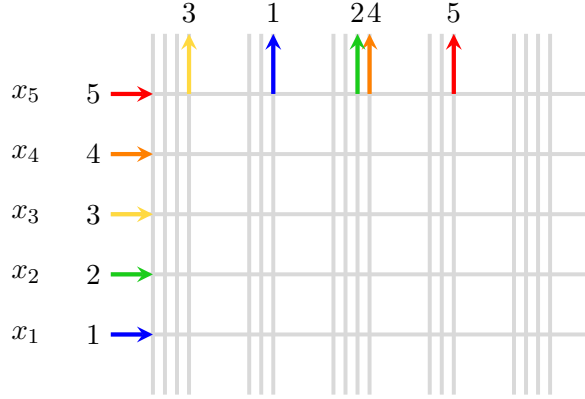

\tikz{0.8}{
\foreach\y in {1,...,5}{
\node[left] at (-1,\y) {$x_\y$};
\draw[lgray,line width=1.5pt] (0,\y) -- (8,\y);
}
\foreach\x in {0.7,0.9,1.1,1.3}{
\draw[lgray,line width=1.5pt] (\x,0) -- (\x,6);
}
\foreach\x in {2.3,2.5,2.7}{
\draw[lgray,line width=1.5pt] (\x,0) -- (\x,6);
}
\foreach\x in {3.7,3.9,4.1,4.3}{
\draw[lgray,line width=1.5pt] (\x,0) -- (\x,6);
}
\foreach\x in {5.3,5.5,5.7}{
\draw[lgray,line width=1.5pt] (\x,0) -- (\x,6);
}
\foreach\x in {6.7,6.9,7.1,7.3}{
\draw[lgray,line width=1.5pt] (\x,0) -- (\x,6);
}
\draw[ultra thick,red,->] (0,5) -- (0.7,5);
\draw[ultra thick,orange,->] (0,4) -- (0.7,4);
\draw[ultra thick,yellow,->] (0,3) -- (0.7,3);
\draw[ultra thick,green,->] (0,2) -- (0.7,2); 
\draw[ultra thick,blue,->] (0,1) -- (0.7,1);
\draw[ultra thick,red,->] (5.7,5) -- (5.7,6) node[above,black] {$5$};
\draw[ultra thick,orange,->] (4.3,5) -- (4.3,6) node[above,black] {\ $4$};
\draw[ultra thick,yellow,->] (1.3,5) -- (1.3,6) node[above,black] {$3$};
\draw[ultra thick,green,->] (4.1,5) -- (4.1,6) node[above,black] {$2$\ };
\draw[ultra thick,blue,->] (2.7,5) -- (2.7,6) node[above,black] {$1$};
\node[left] at (0,1) {$\tiny 1$};
\node[left] at (0,2) {$\tiny 2$};
\node[left] at (0,3) {$\tiny 3$};
\node[left] at (0,4) {$\tiny 4$};
\node[left] at (0,5) {$\tiny 5$};
}
\caption{A partition function in the model \eqref{fund-vert} for $n=5$, with columns grouped into bundles of cardinality $J_k$, $k \geq 0$. On the example, $\{J_0,J_1,J_2,J_3,J_4,\dots\} = \{4,3,4,3,4,\dots\}$. The outgoing states at the top of the lattice can be read off as $\mathcal{A}^{(0)} = \{0,0,0,3\}$, $\mathcal{A}^{(1)} = \{0,0,1\}$, $\mathcal{A}^{(2)} = \{0,0,2,4\}$, $\mathcal{A}^{(3)} = \{0,0,5\}$, while all subsequent bundles are unoccupied. Notice that the outgoing colours in each bundle satisfy the required increasing property.}
\label{fig:pre-fuse}
\end{figure}

\begin{defn}
\label{def:prefus}
Fix a composition $\mu = (\mu_1,\dots,\mu_n)$ and an infinite set of positive integers $\{J_0,J_1,J_2,\dots\}$ with $J_k \geq m_k(\mu)$, for all $k \geq 0$. A configuration $\mathcal{C}$ mapping $(1,\dots,n)$ to $(\mu_1,\dots,\mu_n)$, written $\mathcal{C} : (1,\dots,n) \mapsto (\mu_1,\dots,\mu_n)$, is an ensemble of up-right paths in the model \eqref{fund-vert}, with a path of colour $i$ beginning at the coordinate $(-1,i)$ and such that $i \in \mathcal{A}^{(\mu_i)}$. By specifying via which bundle colour $i$ leaves the lattice, we determine its precise outgoing coordinate uniquely, in view of the ordering property mentioned above.

The weight of a configuration $\mathcal{C}$, written $W_{\mathcal{C}}(x_1,\dots,x_n;Y)$, is the product of Boltzmann weights \eqref{fund-vert} assigned to each vertex in $\mathcal{C}$. We define a non-symmetric function 
$\mathcal{F}_{\mu}(x_1,\dots,x_n;Y)$ as follows:
\begin{align}
\label{f-fund}
\index{Fa@$\mathcal{F}_{\mu}$; pre-fused version of $f_{\mu}$}
\mathcal{F}_{\mu}\left(x_1,\dots,x_n;Y\right)
=
\sum_{\mathcal{C} : (1,\dots,n) \mapsto (\mu_1,\dots,\mu_n)}
W_{\mathcal{C}}\left(x_1,\dots,x_n;Y\right),
\end{align}
where 
\begin{align*}
Y = \left(y^{(0)},y^{(1)},y^{(2)},\dots \right)
\end{align*}
denotes the union of all sets of vertical rapidities. We refer to \eqref{f-fund} as the pre-fused version of the non-symmetric spin Hall--Littlewood function $f_{\mu}$.

More generally, for $\sigma \in \mathfrak{S}_n$, we define
\begin{align}
\label{fsig-fund}
\index{Fb@$\mathcal{F}^{\sigma}_{\mu}$}
\mathcal{F}^{\sigma}_{\mu}\left(x_1,\dots,x_n;Y\right)
=
\sum_{\mathcal{C} : (\sigma^{-1}(1),\dots,\sigma^{-1}(n)) 
\mapsto (\mu_1,\dots,\mu_n)}
W_{\mathcal{C}}\left(x_1,\dots,x_n;Y\right),
\end{align}
summed over configurations mapping $(\sigma^{-1}(1),\dots,\sigma^{-1}(n))$ to 
$(\mu_1,\dots,\mu_n)$; this refers to ensembles of up-right paths with a path of colour $i$ beginning at the coordinate $(-1,\sigma^{-1}(i))$ and such that 
$i \in \mathcal{A}^{(\mu_i)}$.
\end{defn}

\begin{thm}
\label{thm:fusion-F}
Let $\mathcal{F}_{\mu}(x_1,\dots,x_n;Y)$ be as defined in \eqref{f-fund}. We specialize the sets of vertical rapidities to geometric progressions in $q$ with base $s$:
\begin{align}
\label{geom-spec}
y^{(k)}_j = s q^{J_k-j},
\quad\quad
1 \leq j \leq J_k,
\quad\quad
k \geq 0.
\end{align}
Let the image of $\mathcal{F}_{\mu}$ under these specializations be denoted $\mathcal{F}^{*}_{\mu}$. $\mathcal{F}^{*}_{\mu}$ depends rationally on the set of parameters $\{q^{J_0},q^{J_1},q^{J_2},\dots\}$, but is otherwise independent of the value of $\{J_0,J_1,J_2,\dots\}$. We may analytically continue in these variables, performing the replacements $q^{J_k} \mapsto s^{-2}$, leading to the relation
\begin{align}
\label{prefused-f}
\mathcal{F}^{*}_{\mu}\left(x_1,\dots,x_n; q^{J_0},q^{J_1},\dots\right) 
\Big|_{q^{J_0},q^{J_1},\ldots = s^{-2}}
=
\frac{(-s)^{|\mu|}(1-q)^n}{\prod_{k \geq 0} (s^{-2};q^{-1})_{m_k}}
f_{\mu}(x_1,\dots,x_n),
\end{align}
where $m_k \equiv m_k(\mu)$, the multiplicity of part $k$ in $\mu$.

\end{thm}

\begin{proof}
The proof combines several ingredients from Appendix \ref{app:fusion}. The $k$-th bundle of the lattice should be viewed as a stack of $n$ row-vertices of width $J_k$. The vertical spectral parameters within this bundle are specialized as in \eqref{geom-spec}, which are the correct values for the fusion procedure. Furthermore, note that the bottom edges of the bundle are unoccupied by particles (which can be viewed as a trivial sum over all states in the empty sector), while the top edges are occupied by particles whose colours are ordered increasingly. We may therefore apply equation \eqref{stack-Nrow} to replace each row-vertex in the bundle by a $J_k$-fused vertex, as given by \eqref{fus-wt} with $M=J_k$; this replacement is accompanied by an overall multiplicative factor of $(1-q)^{m_k}/(q^{J_k};q^{-1})_{m_k}$, which can be found as an artefact of the normalization in \eqref{stack-Nrow}.

The final step is the analytic continuation $q^{J_k} \mapsto s^{-2}$, which sends each $J_k$-fused vertex to a stochastic $\tilde{L}_x$ matrix; see \eqref{an-cont}. These can be replaced by their non-stochastic counterparts $L_x$ \eqref{s-weights}, at the expense of the factor $(-s)^{(\#\ \text{horizontal unit steps by paths})}$. Using the fact that the total number of unit horizontal steps in each lattice configuration is precisely $|\mu|$, while $\sum_{k \geq 0} m_k = n$, the result \eqref{prefused-f} is then immediate.
\end{proof}

\begin{rmk}
\label{rmk:general-spins}
It is possible to construct a more general version of the functions $f_{\mu}$, by choosing geometric progressions
\begin{align}
\label{geom-spec2}
y^{(k)}_j = s_k y_k q^{J_k-j},
\quad\quad
1 \leq j \leq J_k,
\quad\quad
k \geq 0,
\end{align}
in place of \eqref{geom-spec}. This leads to a version of the model \eqref{s-weights} with inhomogeneous spins $s_k$ and vertical rapidities $y_k$; see Sections \ref{ssec:inhom-model}--\ref{ssec:inhom-model2}.

\end{rmk}

One can also define pre-fused versions of $g_{\mu}$ and $g_{\mu}^{\sigma}$ but we will not do it here, since it turns out to be unnecessary for our purposes.

\chapter{Branching rules and summation identities}
\label{sec:branching}

\section{Skew functions $f_{\mu/\nu}$ and $g_{\mu/\nu}$}

By allowing slightly more general classes of expectation values than in the definitions \eqref{generic-f} and \eqref{generic-g}, one can define {\it skew versions} of the previous non-symmetric rational functions.

\begin{defn}[Generic-sector skew rational functions]
Let $m,m'$ be two integers such that $m > m' \geq 0$. Fix two compositions $\lambda = (\lambda_1,\dots,\lambda_n)$ of weight $m$ and $\kappa = (\kappa_1,\dots,\kappa_n)$ of weight $m'$ with $\lambda_i \geq \kappa_i$ for all $1 \leq i \leq n$. Denote the partial sums of their difference $\lambda-\kappa$ by
\begin{align}
\label{part-sums}
\Delta_j
=
\sum_{i=1}^{j} (\lambda_i - \kappa_i),
\qquad
1 \leq j \leq n.
\end{align}
From these, we fix a $\lambda$-coloured composition $\mu = (\mu_1,\dots,\mu_m)$ and a $\kappa$-coloured composition $\nu = (\nu_1,\dots,\nu_{m'})$. We define skew versions of the functions \eqref{generic-f} as follows:
\begin{align}
\label{sk-generic-f}
f_{\mu/\nu}(\lambda-\kappa;x_1,\dots,x_{m-m'})
:=
\bra{\nu}_{\kappa}
\left(
\prod_{i=1}^{\Delta_1}
\C_1(x_i)
\right)
\cdots
\left(
\prod_{i=\Delta_{n-1}+1}^{\Delta_n}
\C_n(x_i)
\right)
\ket{\mu}_{\lambda},
\end{align}
where $\ket{\mu}_{\lambda} \in \mathbb{V}(\lambda)$ and $\bra{\nu}_{\kappa} \in \mathbb{V}^{*}(\kappa)$. The reduction from the skew functions \eqref{sk-generic-f} to their non-skew counterparts \eqref{generic-f} is obtained by taking $m' = 0$ (so that $\kappa = (0^n)$ and $\nu =\varnothing$).
\end{defn}

\begin{defn}[Rainbow-sector skew rational functions]
Select a pair of rainbow compositions $\mu = (\mu_1,\dots,\mu_m)$ and $\nu = (\nu_1,\dots,\nu_{m'})$, where $m > m' \geq 0$. We introduce a skew-version of the non-symmetric spin Hall--Littlewood function \eqref{f-HL} as follows:
\begin{align}
\label{skew-f-HL}
\index{f@$f_{\mu/\nu}$; skew non-symmetric spin Hall--Littlewood}
f_{\mu/\nu}(x_1,\dots,x_{m-m'})
:=
\bra{\nu} \C_{m'+1}(x_1) \C_{m'+2}(x_2) \dots \C_{m}(x_{m-m'}) \ket{\mu},
\end{align}
where $\ket{\mu} \in \mathbb{V}(1^{m}0^{n-m})$ and $\bra{\nu} \in \mathbb{V}^{*}(1^{m'}0^{n-m'})$. These functions can be considered as the special case $\lambda = 1^{m} 0^{n-m}$, $\kappa = 1^{m'}0^{n-m'}$ of the more general functions \eqref{sk-generic-f}; in other words, $\lambda$ and $\kappa$ are both chosen to be single-column partitions, with $\lambda \supset \kappa$.
\end{defn}

Translating the expectation value \eqref{skew-f-HL} into graphical language, we obtain the following partition function:
\begin{align}
\label{f-sk-pf}
f_{\mu/\nu}(x_1,\dots,x_{m-m'})
&=
\tikz{0.8}{
\foreach\y in {1,...,5}{
\draw[lgray,line width=1.5pt,->] (1,\y) -- (8,\y);
}
\foreach\x in {2,...,7}{
\draw[lgray,line width=4pt,->] (\x,0) -- (\x,6);
}
\node[left] at (-1,1) {$x_1 \rightarrow$};
\node[left] at (-1,2) {$x_2 \rightarrow$};
\node[left] at (-1,3) {$\vdots$};
\node[left] at (-1,4) {$\vdots$};
\node[left] at (-1,5) {$x_{m-m'} \rightarrow$};
\node[below] at (7,0) {$\cdots$};
\node[below] at (6,0) {$\cdots$};
\node[below] at (5,0) {$\cdots$};
\node[below] at (4,0) {\footnotesize$\bm{B}(2)$};
\node[below] at (3,0) {\footnotesize$\bm{B}(1)$};
\node[below] at (2,0) {\footnotesize$\bm{B}(0)$};
\node[above] at (7,6) {$\cdots$};
\node[above] at (6,6) {$\cdots$};
\node[above] at (5,6) {$\cdots$};
\node[above] at (4,6) {\footnotesize$\bm{A}(2)$};
\node[above] at (3,6) {\footnotesize$\bm{A}(1)$};
\node[above] at (2,6) {\footnotesize$\bm{A}(0)$};
\node[right] at (8,1) {$0$};
\node[right] at (8,2) {$0$};
\node[right] at (8,3) {$\vdots$};
\node[right] at (8,4) {$\vdots$};
\node[right] at (8,5) {$0$};
\node[left] at (1,1) {$m'+1$};
\node[left] at (1,2) {$m'+2$};
\node[left] at (1,3) {$\vdots$};
\node[left] at (1,4) {$\vdots$};
\node[left] at (1,5) {$m$};
}
\end{align}
where the states at the top and bottom of the lattice are given by
\begin{align}
\label{a-b-states}
\bm{A}(k) = \sum_{j=1}^{n} A_j(k) \bm{e}_j,
\quad
A_j(k) = \bm{1}_{\mu_j = k},
\qquad
\bm{B}(k) = \sum_{j=1}^{n} B_j(k) \bm{e}_j,
\quad
B_j(k) = \bm{1}_{\nu_j = k}.
\end{align}
In the special case $m'=0$ (and $m=n$) the composition $\nu$ has length zero, \ie\ 
$\nu = \varnothing$, and all $B_j(k) \equiv 0$. We recover in this case the partition function expression \eqref{f-def} for $f_{\mu}(x_1,\dots,x_n)$.

\begin{defn}[Dual generic-sector skew rational functions]
Let $m,m'$ be two integers such that $m > m' \geq 0$. Fix two compositions $\lambda = (\lambda_1,\dots,\lambda_n)$ of weight $m$ and $\kappa = (\kappa_1,\dots,\kappa_n)$ of weight $m'$ with $\lambda_i \geq \kappa_i$ for all $1 \leq i \leq n$, and let the partial sums of their difference be given by \eqref{part-sums}. Fix a $\lambda$-coloured composition $\mu = (\mu_1,\dots,\mu_m)$ and a $\kappa$-coloured composition $\nu = (\nu_1,\dots,\nu_{m'})$. We define
\begin{align}
\label{sk-generic-g}
g_{\mu/\nu}(\lambda-\kappa; x_1,\dots,x_{m-m'})
:=
\bra{\mu}_{\lambda}
\left(
\prod_{i=1}^{\Delta_1}
\B_1(x_i)
\right)
\cdots
\left(
\prod_{i=\Delta_{n-1}+1}^{\Delta_n}
\B_n(x_i)
\right)
\ket{\nu}_{\kappa},
\end{align}
where $\ket{\nu}_{\kappa} \in \mathbb{V}(\kappa)$ and $\bra{\mu}_{\lambda} \in \mathbb{V}^{*}(\lambda)$. 
\end{defn}

\begin{defn}[Dual rainbow-sector skew rational functions]
Fix two rainbow compositions $\mu = (\mu_1,\dots,\mu_m)$ and $\nu = (\nu_1,\dots,\nu_{m'})$, where $m > m' \geq 0$. We introduce a skew-version of the dual non-symmetric spin Hall--Littlewood function \eqref{g-HL} as follows:
\begin{align}
\label{skew-g-HL}
\index{g@$g_{\mu/\nu}$}
g_{\mu/\nu}(x_1,\dots,x_{m-m'})
:=
\bra{\mu} \B_{m'+1}(x_1) \B_{m'+2}(x_2) \dots \B_{m}(x_{m-m'}) \ket{\nu},
\end{align}
where $\ket{\nu} \in \mathbb{V}(0^{m}1^{n-m})$ and $\bra{\mu} \in \mathbb{V}^{*}(0^{m'} 1^{n-m'})$. These functions can be considered as the special case $\lambda = 0^{m'}1^{n-m'}$, $\kappa = 0^{m} 1^{n-m}$ of the more general functions \eqref{sk-generic-g}.
\end{defn}

Converting \eqref{skew-g-HL} into vertex-model formalism, we obtain the partition function
\begin{align}
\label{g-sk-pf}
g_{\mu/\nu}(x_1,\dots,x_{m-m'})
&=
\tikz{0.8}{
\foreach\y in {1,...,5}{
\draw[lgray,line width=1.5pt,<-] (1,\y) -- (8,\y);
}
\foreach\x in {2,...,7}{
\draw[lgray,line width=4pt,->] (\x,0) -- (\x,6);
}
\node[left] at (-1,1) {$x_1 \leftarrow$};
\node[left] at (-1,2) {$x_2 \leftarrow$};
\node[left] at (-1,3) {$\vdots$};
\node[left] at (-1,4) {$\vdots$};
\node[left] at (-1,5) {$x_{m-m'} \leftarrow$};
\node[above] at (7,6) {$\cdots$};
\node[above] at (6,6) {$\cdots$};
\node[above] at (5,6) {$\cdots$};
\node[above] at (4,6) {\footnotesize$\bm{B}(2)$};
\node[above] at (3,6) {\footnotesize$\bm{B}(1)$};
\node[above] at (2,6) {\footnotesize$\bm{B}(0)$};
\node[below] at (7,0) {$\cdots$};
\node[below] at (6,0) {$\cdots$};
\node[below] at (5,0) {$\cdots$};
\node[below] at (4,0) {\footnotesize$\bm{A}(2)$};
\node[below] at (3,0) {\footnotesize$\bm{A}(1)$};
\node[below] at (2,0) {\footnotesize$\bm{A}(0)$};
\node[right] at (8,1) {$0$};
\node[right] at (8,2) {$0$};
\node[right] at (8,3) {$\vdots$};
\node[right] at (8,4) {$\vdots$};
\node[right] at (8,5) {$0$};
\node[left] at (1,1) {$m'+1$};
\node[left] at (1,2) {$m'+2$};
\node[left] at (1,3) {$\vdots$};
\node[left] at (1,4) {$\vdots$};
\node[left] at (1,5) {$m$};
}
\end{align}
where the states at the bottom and top of the lattice are given by \eqref{a-b-states}. Similarly to above, setting $m' = 0$ and $m = n$ reproduces the partition function representation \eqref{g-def} of $g_{\mu}(x_1,\dots,x_n)$.

\section{Branching rules}
\label{ssec:branch}

\begin{prop}
Fix two integers $m > m' \geq 0$ and let $\mu \in \mathcal{S}_{(1^m 0^{n-m})}$, $\nu \in \mathcal{S}_{(1^{m'} 0^{n-m'})}$ be rainbow compositions. Then for any integer $p$ such that $m' < p < m$, one has the branching identity
\begin{align}
\label{branch-f}
f_{\mu/\nu}(x_1,\dots,x_{m-m'})
&=
\sum_{\alpha}
f_{\alpha/\nu}(x_1,\dots,x_{p-m'})
f_{\mu/\alpha}(x_{p-m'+1},\dots,x_{m-m'}),
\end{align}
where the sum is taken over all rainbow compositions $\alpha \in \mathcal{S}_{(1^p 0^{n-p})}$. Similarly, taking $\mu \in \mathcal{S}_{(0^{m'}1^{n-m'})}$, $\nu \in \mathcal{S}_{(0^{m} 1^{n-m})}$ and $p$ such that $m' < p < m$, one has
\begin{align}
\label{branch-g}
g_{\mu/\nu}(x_1,\dots,x_{m-m'})
&=
\sum_{\beta}
g_{\mu/\beta}(x_1,\dots,x_{p-m'})
g_{\beta/\nu}(x_{p-m'+1},\dots,x_{m-m'}),
\end{align}
where the sum is taken over all rainbow compositions 
$\beta \in \mathcal{S}_{(0^{p} 1^{n-p})}$. 
\end{prop}

\begin{proof}
Starting from the expectation value \eqref{skew-f-HL}, we insert a resolution of the identity between the operators $\mathcal{C}_p(x_{p-m'})$ and $\mathcal{C}_{p+1}(x_{p-m'+1})$. In doing so, we may restrict the summation to the sector $\mathcal{S}_{(1^p 0^{n-p})}$. We obtain
\begin{multline*}
f_{\mu/\nu}(x_1,\dots,x_{m-m'})
\\
=
\sum_{\alpha \in \mathcal{S}_{(1^p 0^{n-p})}}
\bra{\nu} \C_{m'+1}(x_1) \dots \C_{p}(x_{p-m'}) \ket{\alpha}
\bra{\alpha} \C_{p+1}(x_{p-m'+1}) \dots \C_{m}(x_{m-m'}) \ket{\mu},
\end{multline*}
and the result \eqref{branch-f} is now immediate by converting the expectation values back into their functional form. The proof of \eqref{branch-g} follows by similar considerations applied to \eqref{skew-g-HL}.


\end{proof}

\begin{cor}
Fix a rainbow composition $\mu \in \mathcal{S}_{(1^n)}$. For any $1 \leq p < n$ one has the branching relations
\begin{align*}
f_{\mu}(x_1,\dots,x_n)
&=
\sum_{\alpha}
f_{\alpha}(x_1,\dots,x_p)
f_{\mu/\alpha}(x_{p+1},\dots,x_n),
\\
g_{\mu}(x_1,\dots,x_n)
&=
\sum_{\beta}
g_{\mu/\beta}(x_1,\dots,x_p)
g_{\beta}(x_{p+1},\dots,x_n),
\end{align*}
where the first sum is over rainbow compositions $\alpha \in \mathcal{S}_{(1^p 0^{n-p})}$, and the second is over rainbow compositions $\beta \in \mathcal{S}_{(0^{p} 1^{n-p})}$.
\end{cor}

\begin{proof}
These relations follow as the $m' = 0$, $m=n$ reduction of \eqref{branch-f} and 
\eqref{branch-g}, when $\nu$ is replaced by the empty composition $\varnothing$.
\end{proof}


\section{Summation identities of Mimachi--Noumi type}
\label{ssec:mimachi}

Throughout the rest of the paper, the following alteration of the normalization of $g_{\mu}$ will be convenient:
\index{g1@$g^{*}_{\mu}$}
\begin{align}
\label{g-star}
g^{*}_{\mu}(x_1,\dots,x_n)
:=
\frac{q^{n(n+1)/2}}{(q-1)^n}
\cdot
g_{\mu}(x_1,\dots,x_n).
\end{align}

\begin{thm}
\label{thm:mimachi}
Let $(x_1,\dots,x_n)$ and $(y_1,\dots,y_n)$ be two sets of complex parameters such that
\begin{align}
\label{weight-condition2}
\left|
\frac{x_i-s}{1-sx_i}
\cdot
\frac{y_j-s}{1-sy_j}
\right|
<
1,
\qquad
\text{for all}\ 1 \leq i,j \leq n.
\end{align}
The non-symmetric spin Hall--Littlewood functions satisfy the identity
\begin{align}
\label{mimachi-id}
\sum_{\mu}
f_{\mu}(x_1,\dots,x_n)
g^{*}_{\mu}(y_1,\dots,y_n)
=
\prod_{i=1}^{n}
\frac{1}{1-x_i y_i}
\prod_{n \geq i>j \geq 1}
\frac{1-q x_i y_j}{1-x_i y_j},
\end{align}
where the summation is over all compositions $\mu$.
\end{thm}

\begin{proof}
We start by defining the expectation value
\begin{align*}
\mathcal{E}_n(x_1,\dots,x_n;y_1,\dots,y_n)
=
\bra{\varnothing}
\C_1(x_1) \dots \C_n(x_n)
\B_1(y_1) \dots \B_n(y_n)
\ket{\varnothing}.
\end{align*}
Inserting a complete set of states $\sum_{\mu} \ket{\mu} \bra{\mu}$ between the final $\C$ operator and the first $\B$ operator, we see that
\begin{align}
\nonumber
\mathcal{E}_n(x_1,\dots,x_n;y_1,\dots,y_n)
&=
\sum_{\mu}
\bra{\varnothing}
\C_1(x_1) \dots \C_n(x_n)
\ket{\mu}
\bra{\mu}
\B_1(y_1) \dots \B_n(y_n)
\ket{\varnothing}
\\
&=
\sum_{\mu}
f_{\mu}(x_1,\dots,x_n)
g_{\mu}(y_1,\dots,y_n).
\label{lhs-mn}
\end{align}
On the other hand, making repeated use of the commutation relation \eqref{CB>} (which we are able to do in view of the assumption \eqref{weight-condition2} on the rapidities, \cf\ \eqref{weight-condition}), we can transfer the operator $\B_1(y_1)$ to the left through $n-1$ of the $\C$ operators:
\begin{multline}
\label{thm4.7-0}
\mathcal{E}_n(x_1,\dots,x_n;y_1,\dots,y_n)
\\
=
q^{-n+1}
\prod_{i=2}^{n}
\left(
\frac{1-q x_i y_1}{1-x_i y_1}
\right)
\bra{\varnothing}
\C_1(x_1) \B_1(y_1) 
\C_2(x_2) \dots \C_n(x_n)
\B_2(y_2) \dots \B_n(y_n)
\ket{\varnothing}.
\end{multline}
We turn our attention to the first piece of this expectation value, namely  the quantity $\bra{\varnothing} \C_1(x_1) \B_1(y_1)$. Employing another commutation relation, \eqref{CB=}, we have
\begin{align}
\label{thm4.7-1}
\bra{\varnothing} \C_1(x_1) \B_1(y_1)
=
\frac{1-q^{-1}}{1-x_1 y_1}
\bra{\varnothing} \B_0(y_1) \C_0(x_1)
+
\bra{\varnothing} \B_1(y_1) \C_1(x_1)
-
\frac{(1-q)x_1 y_1}{1-x_1 y_1}
\sum_{k=2}^{n}
\bra{\varnothing} \B_k(y_1) \C_k(x_1).
\end{align}
Now we observe the following properties of the $\B$ and $\C$ operators when acting on the dual vacuum state $\bra{\varnothing}$:
\begin{align}
\label{simple-action}
\bra{\varnothing} \B_0(y) = \bra{\varnothing} \C_0(x) = \bra{\varnothing},
\qquad
\bra{\varnothing} \B_k(y) = 0, \quad \forall\ k \geq 1.
\end{align}
These relations follow trivially from the graphical interpretation of the row operators, and conservation of lattice paths through a vertex. Making use of the actions \eqref{simple-action} on the right hand side of \eqref{thm4.7-1}, we recover the simplified relation 
$\bra{\varnothing} \C_1(x_1) \B_1(y_1) = (1-q^{-1})/(1-x_1 y_1) \bra{\varnothing}$, which can now be substituted into \eqref{thm4.7-0} to yield
\begin{multline}
\label{thm4.7-2}
\mathcal{E}_n(x_1,\dots,x_n;y_1,\dots,y_n)
\\
=
q^{-n}
\left(
\frac{q-1}{1-x_1 y_1}
\right)
\prod_{i=2}^{n}
\left(
\frac{1-q x_i y_1}{1-x_i y_1}
\right)
\bra{\varnothing} 
\C_2(x_2) \dots \C_n(x_n)
\B_2(y_2) \dots \B_n(y_n)
\ket{\varnothing}.
\end{multline}
The final expectation value in \eqref{thm4.7-2} depends on the colours $\{2,\dots,n\}$, with colour $1$ never making an appearance. Shifting the value of all colours down by $1$ leaves the expectation value invariant; we thus recover the recurrence
\begin{align}
\label{E-recur}
\mathcal{E}_n(x_1,\dots,x_n;y_1,\dots,y_n)
=
q^{-n}
\left(
\frac{q-1}{1-x_1 y_1}
\right)
\prod_{i=2}^{n}
\left(
\frac{1-q x_i y_1}{1-x_i y_1}
\right)
\mathcal{E}_{n-1}(x_2,\dots,x_n;y_2,\dots,y_n),
\end{align}
whose solution is 
\begin{align}
\label{rhs-mn}
\mathcal{E}_n(x_1,\dots,x_n;y_1,\dots,y_n)
=
q^{-n(n+1)/2} (q-1)^n
\prod_{i=1}^{n}
\frac{1}{1-x_i y_i}
\prod_{n \geq i>j \geq 1}
\frac{1-q x_i y_j}{1-x_i y_j}.
\end{align}
Matching \eqref{lhs-mn} and \eqref{rhs-mn} yields the claim, \eqref{mimachi-id}.

\end{proof}


\section{Symmetric rational function $G_{\mu/\nu}$}
\label{ssec:G}

We will also consider a family of {\it symmetric} rational functions $G_{\mu/\nu}$, indexed by a pair of rainbow compositions $\mu,\nu$. These can be considered as coloured refinements of certain symmetric functions considered in \cite{Borodin,BorodinP2}, as we explain below.

\begin{defn}
Let $\mu = (\mu_1,\dots,\mu_n)$ and $\nu = (\nu_1,\dots,\nu_n)$ be a pair of rainbow compositions, and fix any positive integer $p$. We introduce the symmetric rational function\footnote{The function \eqref{G-def} is symmetric in the alphabet $(x_1,\dots,x_p)$, in view of the commutativity \eqref{BB=} of the $\mathcal{B}_0$ operators.}
\index{G@$G_{\mu/\nu}$; transfer matrix}
\begin{align}
\label{G-def}
G_{\mu/\nu}(x_1,\dots,x_p)
:=
\bra{\mu}
\mathcal{B}_0(x_1)
\dots
\mathcal{B}_0(x_p)
\ket{\nu},
\end{align}
where $\ket{\nu} \in \mathbb{V}(1^n)$ and $\bra{\mu} \in \mathbb{V}^{*}(1^n)$. In the special case $\nu = 0^n$, we shall write
\index{G@$G_{\mu}$}
\begin{align*}
G_{\mu/0^n}(x_1,\dots,x_p)
\equiv
G_{\mu}(x_1,\dots,x_p).
\end{align*}
\end{defn}

Graphically speaking, the function $G_{\mu/\nu}$ is given by the partition function
\begin{align}
\label{G-pf}
G_{\mu/\nu}(x_1,\dots,x_p)
&=
\tikz{0.8}{
\foreach\y in {1,...,5}{
\draw[lgray,line width=1.5pt,<-] (1,\y) -- (8,\y);
}
\foreach\x in {2,...,7}{
\draw[lgray,line width=4pt,->] (\x,0) -- (\x,6);
}
\node[left] at (0.5,1) {$x_1 \leftarrow$};
\node[left] at (0.5,2) {$x_2 \leftarrow$};
\node[left] at (0.5,3) {$\vdots$};
\node[left] at (0.5,4) {$\vdots$};
\node[left] at (0.5,5) {$x_p \leftarrow$};
\node[above] at (7,6) {$\cdots$};
\node[above] at (6,6) {$\cdots$};
\node[above] at (5,6) {$\cdots$};
\node[above] at (4,6) {\footnotesize$\bm{B}(2)$};
\node[above] at (3,6) {\footnotesize$\bm{B}(1)$};
\node[above] at (2,6) {\footnotesize$\bm{B}(0)$};
\node[below] at (7,0) {$\cdots$};
\node[below] at (6,0) {$\cdots$};
\node[below] at (5,0) {$\cdots$};
\node[below] at (4,0) {\footnotesize$\bm{A}(2)$};
\node[below] at (3,0) {\footnotesize$\bm{A}(1)$};
\node[below] at (2,0) {\footnotesize$\bm{A}(0)$};
\node[right] at (8,1) {$0$};
\node[right] at (8,2) {$0$};
\node[right] at (8,3) {$\vdots$};
\node[right] at (8,4) {$\vdots$};
\node[right] at (8,5) {$0$};
\node[left] at (1,1) {$0$};
\node[left] at (1,2) {$0$};
\node[left] at (1,3) {$\vdots$};
\node[left] at (1,4) {$\vdots$};
\node[left] at (1,5) {$0$};
}
\end{align}
where the states at the bottom and the top of the lattice are given by \eqref{a-b-states}. The difference between this partition function, and that of \eqref{g-sk-pf}, is that the particles which enter the lattice from its base must all exit from the top boundary, and not via the left boundary as can occur in \eqref{g-sk-pf}.


\begin{rmk}
\label{rmk:4.10}
One can consider colour-blind projections of \eqref{G-pf}, obtained by treating all paths as having the same colour. In that situation, the coloured compositions $\mu$, $\nu$ become partitions $\mu = (\mu_1 \geq \cdots \geq \mu_n \geq 0)$ and $\nu = (\nu_1 \geq \cdots \geq \nu_n \geq 0)$, and one can define
\index{G@$G^{\bullet}_{\mu/\nu}$; colour-blind transfer matrix}
\begin{align}
\label{G}
G^{\bullet}_{\mu/\nu}(x_1,\dots,x_p)
&:=
\tikz{0.85}{
\foreach\y in {1,...,5}{
\draw[lgray,line width=1.5pt,<-] (1,\y) -- (8,\y);
}
\foreach\x in {2,...,7}{
\draw[lgray,line width=4pt,->] (\x,0) -- (\x,6);
}
\foreach\y in {1,...,5}{
\foreach\x in {2,...,7}{
\node at (\x,\y) {$\bullet$};
}}
\node[left] at (0.5,1) {$x_1 \leftarrow$};
\node[left] at (0.5,2) {$x_2 \leftarrow$};
\node[left] at (0.5,3) {$\vdots$};
\node[left] at (0.5,4) {$\vdots$};
\node[left] at (0.5,5) {$x_p \leftarrow$};
\node[above] at (7,6) {$\cdots$};
\node[above] at (6,6) {$\cdots$};
\node[above] at (5,6) {$\cdots$};
\node[above] at (4,6) {\tiny$m_2(\nu)$};
\node[above] at (3,6) {\tiny$m_1(\nu)$};
\node[above] at (2,6) {\tiny$m_0(\nu)$};
\node[below] at (7,0) {$\cdots$};
\node[below] at (6,0) {$\cdots$};
\node[below] at (5,0) {$\cdots$};
\node[below] at (4,0) {\tiny$m_2(\mu)$};
\node[below] at (3,0) {\tiny$m_1(\mu)$};
\node[below] at (2,0) {\tiny$m_0(\mu)$};
\node[right] at (8,1) {$0$};
\node[right] at (8,2) {$0$};
\node[right] at (8,3) {$\vdots$};
\node[right] at (8,4) {$\vdots$};
\node[right] at (8,5) {$0$};
\node[left] at (1,1) {$0$};
\node[left] at (1,2) {$0$};
\node[left] at (1,3) {$\vdots$};
\node[left] at (1,4) {$\vdots$};
\node[left] at (1,5) {$0$};
}
\end{align}
where the states at the bottom and top of the lattice in \eqref{G-pf} are replaced by 
$\bm{A}(k) \mapsto m_k(\mu) \bm{e}_1$ and $\bm{B}(k) \mapsto m_k(\nu) \bm{e}_1$, for all $k \in \mathbb{N}$. The vertices within the lattice \eqref{G} are the rank-1 counterparts of the dual weights \eqref{dual-s-weights}, given by \eqref{rank1-dual-weights}. In the special case $\nu = 0^n$ (all particles leave the lattice via the top of the 0-th column), we write
\begin{align*}
G^{\bullet}_{\mu/0^n}(x_1,\dots,x_p)
\equiv
G^{\bullet}_{\mu}(x_1,\dots,x_p).
\end{align*}
\index{G@$G^{\bullet}_{\mu}$}
The partition function \eqref{G} appeared previously in \cite{Borodin,BorodinP2}. In the notation of those works, one has
\begin{align}
\label{G-sf}
G^{\bullet}_{\mu/\nu}(x_1,\dots,x_p)
=
q^{-np}
\cdot
\G_{\mu/\nu}(x_1,\dots,x_p).
\end{align}
\end{rmk}
\index{G@$\G_{\mu/\nu}$}

\begin{prop}
Let $\mu = (\mu_1,\dots,\mu_n)$ be a length-$n$ composition and let $\mu^{+}$ denote its partition ordering. The coloured partition function \eqref{G-pf} and its uncoloured counterpart \eqref{G} are equal when $\nu = 0^n$, \ie\ one has
\begin{align}
\label{GGbull}
G_{\mu}(x_1,\dots,x_p)
=
G^{\bullet}_{\mu^+}(x_1,\dots,x_p),
\quad
\forall\ p \geq 1.
\end{align}
More generally, summing over all length-$n$ compositions 
$\nu = (\nu_1,\dots,\nu_n)$ with fixed partition ordering $\nu^{+} = \kappa$, one has
\begin{align}
\label{GGbull2}
\sum_{\nu: \nu^{+} = \kappa}
G_{\mu/\nu}(x_1,\dots,x_p)
=
G^{\bullet}_{\mu^+/\kappa}(x_1,\dots,x_p),
\quad
\forall\ p \geq 1.
\end{align}
\end{prop}

\begin{proof}
These identities follow by repeated application of the colour-blindness relations \eqref{cb5} and \eqref{cb6}. Let us sketch how the proof works in the case of \eqref{GGbull}. One begins by isolating the top-left vertex of the lattice \eqref{G-pf}, in the case $\nu = 0^n$. In any given configuration of the lattice, it has the form 
\begin{align}
M_{x_p}(\I,j;1^n,0)
=
\tikz{0.7}{
\draw[lgray,line width=1.5pt,<-] (-1,0) -- (1,0);
\draw[lgray,line width=4pt,->] (0,-1) -- (0,1);
\node[left] at (-1,0) {\tiny $0$};\node[right] at (1,0) {\tiny $j$};
\node[below] at (0,-1) {\tiny $\I$};\node[above] at (0,1) {\tiny $1^n$};
},
\quad
j \in \{0,1,\dots,n\},
\quad
\I \in \mathbb{N}^n.
\end{align}
By direct comparison with the dual rank-1 weights \eqref{rank1-dual-weights}, one sees that
\begin{align*}
M_{x_p}(\I,j;1^n,0)
=
M^{\bullet}_{x_p}(|\I|,\theta(j);n,0),
\end{align*}
for any $j \in \{0,1,\dots,n\}$ and $\I \in \mathbb{N}^n$. This relation serves as the trigger for deducing the colour-blindness of the whole lattice.
\end{proof}

\section{Cauchy identities}
\label{ssec:cauch}

Finally we turn to summation identities which arise by pairing the non-symmetric spin Hall--Littlewood functions with the symmetric functions \eqref{G-def}. They are close relatives of Cauchy-type relations for the spin Hall--Littlewood functions obtained in \cite{Borodin,BorodinP2}.

\begin{prop}
\label{prop:fG-skew}
Let $(x_1,\dots,x_n)$ and $(y_1,\dots,y_p)$ be two sets of complex parameters satisfying the constraints \eqref{weight-condition2}, and fix a composition $\nu = (\nu_1,\dots,\nu_n)$ of length $n$. Then one has the identity
\begin{align}
\label{fG-skew}
\sum_{\mu}
f_{\mu}(x_1,\dots,x_n)
G_{\mu / \nu}(y_1,\dots,y_p)
=
\prod_{i=1}^{n}
\prod_{j=1}^{p}
\frac{1-q x_i y_j}{q(1-x_i y_j)}
\cdot
f_{\nu}(x_1,\dots,x_n),
\end{align}
where the summation is taken over all length-$n$ compositions 
$\mu = (\mu_1,\dots,\mu_n)$.
\end{prop}

\begin{proof}
The proof proceeds along similar lines to that of Theorem \ref{thm:mimachi}. We define the expectation value
\begin{align}
\label{Enu}
\mathcal{E}_{\nu}(x_1,\dots,x_n; y_1,\dots,y_p)
=
\bra{\varnothing}
\C_1(x_1) \dots \C_n(x_n)
\B_0(y_1) \dots \B_0(y_p)
\ket{\nu},
\end{align}
which after insertion of the identity $\sum_{\mu} \ket{\mu} \bra{\mu}$ between $\C_n(x_n)$ and $\B_0(y_1)$ is given by
\begin{align*}
\mathcal{E}_{\nu}(x_1,\dots,x_n; y_1,\dots,y_p)
=
\sum_{\mu}
f_{\mu}(x_1,\dots,x_n)
G_{\mu / \nu}(y_1,\dots,y_p).
\end{align*}
Alternatively, by repeated use of the commutation relation \eqref{CB>} and the action \eqref{simple-action} of $\B_0$ on the dual vacuum $\bra{\varnothing}$, one clearly has
\begin{align*}
\mathcal{E}_{\nu}(x_1,\dots,x_n; y_1,\dots,y_p)
&=
q^{-np}
\prod_{i=1}^{n}
\prod_{j=1}^{p}
\frac{1-q x_i y_j}{1-x_i y_j}
\cdot
\bra{\varnothing}
\B_0(y_1) \dots \B_0(y_p)
\C_1(x_1) \dots \C_n(x_n)
\ket{\nu}
\\
&=
q^{-np}
\prod_{i=1}^{n}
\prod_{j=1}^{p}
\frac{1-q x_i y_j}{1-x_i y_j}
\cdot
\bra{\varnothing}
\C_1(x_1) \dots \C_n(x_n)
\ket{\nu}
\\
&=
q^{-np}
\prod_{i=1}^{n}
\prod_{j=1}^{p}
\frac{1-q x_i y_j}{1-x_i y_j}
\cdot
f_{\nu}(x_1,\dots,x_n).
\end{align*}
\end{proof}

\begin{rmk}
If one sums equation \eqref{fG-skew} over all compositions 
$\nu = (\nu_1,\dots,\nu_n)$ with fixed partition ordering $\nu^{+} = \kappa$, then by a combination of \eqref{GGbull2}, \eqref{G-sf} and \eqref{f-sum-F} one recovers
\begin{align*}
\sum_{\lambda}
\F^{\sf c}_{\lambda}(x_1,\dots,x_n)
\G_{\lambda/\kappa}(y_1,\dots,y_p)
=
\prod_{i=1}^{n}
\prod_{j=1}^{p}
\frac{1-q x_i y_j}{1-x_i y_j}
\cdot
\F^{\sf c}_{\kappa}(x_1,\dots,x_n),
\end{align*}
with the sum ranging over all partitions $\lambda$. This is then equivalent to the skew Cauchy identity of \cite[Equation (4.7)]{Borodin} \cite[Equation (4.16)]{BorodinP2}.
\end{rmk}

\begin{cor}
The $\nu = 0^n$ case of equation \eqref{fG-skew} reads
\begin{align}
\label{fG-cauchy}
\sum_{\mu}
f_{\mu}(x_1,\dots,x_n)
G_{\mu}(y_1,\dots,y_p)
&=
\sum_{\mu}
f_{\mu}(x_1,\dots,x_n)
G^{\bullet}_{\mu^{+}}(y_1,\dots,y_p)
\\
\nonumber
&=
\frac{(s^2;q)_n}{\prod_{i=1}^{n}(1-s x_i)}
\prod_{i=1}^{n}
\prod_{j=1}^{p}
\frac{1-q x_i y_j}{q(1-x_i y_j)},
\end{align}
with the summation taken over all compositions $\mu = (\mu_1,\dots,\mu_n)$.
\end{cor}

\begin{proof}
The first line of \eqref{fG-cauchy} follows from the colour-blindness property \eqref{GGbull}; the second uses \eqref{fG-skew} and the direct evaluation of $f_{0^n}(x_1,\dots,x_n)$. Indeed, in the case $\nu = 0^n$, $f_{\nu}(x_1,\dots,x_n)$ is given by a partition function consisting of a single non-trivial column, and is therefore factorized. One can read off its value as $f_{0^n}(x_1,\dots,x_n) = (s^2;q)_n \prod_{i=1}^{n} (1-sx_i)^{-1}$, and the statement then follows.
\end{proof}

\chapter{Recursive properties and symmetries}
\label{sec:properties}

The rational functions $f_{\mu}$ and $g_{\mu}$, as well as their permuted versions $f^{\sigma}_{\mu}$ and $g^{\sigma}_{\mu}$, have a rich interrelationship. The goal of this chapter is to document their various properties, and to prove them directly from the lattice definitions \eqref{f-def} and \eqref{g-def} of the functions. The heart of the chapter will be to examine the transformations of $f_{\mu}$ and $g_{\mu}$ under Hecke divided-difference operators. The exchange relations obtained (once supplemented by a suitable initial condition) allow us to conclude, in Section \ref{ssec:hl-reduce}, that $f_{\mu}$ and $g_{\mu}$ degenerate to non-symmetric Hall--Littlewood polynomials when $s=0$.

\section{Factorization of $f_{\delta}$}

The non-symmetric spin Hall--Littlewood functions are rather complicated combinatorial objects, expressible as a sum over many lattice configurations \eqref{f-def}. There is, however, one case in which they are particularly simple; namely, when the indexing composition $\mu$ is chosen to be {\it anti-dominant}, meaning that its parts are weakly increasing.
\begin{prop}
\label{prop:f-delta}
Let $\delta = (\delta_1 \leq \cdots \leq \delta_n)$ be an anti-dominant composition. The corresponding non-symmetric spin Hall--Littlewood function $f_{\delta}$ is completely factorized:
\begin{align}
\label{f-delta}
f_{\delta}(x_1,\dots,x_n)
\index{f3@$f_{\delta}$}
=
\frac{\prod_{j \geq 0} (s^2;q)_{m_j(\delta)}}{\prod_{i=1}^{n} (1-s x_i)}
\prod_{i=1}^{n} \left( \frac{x_i-s}{1-sx_i} \right)^{\delta_i},
\end{align}
where $m_j(\delta)$ is the number of parts of $\delta$ equal to $j$.
\end{prop}

\begin{proof}
This is most easily seen from the partition function representation \eqref{f-def} of $f_{\delta}$. Since $\delta$ has increasing parts, the colours $\{1,\dots,n\}$ exit the top of the partition function \eqref{f-def} in an ordered fashion: there will be a path of colour $i$ outgoing in column $\delta_i$ of the lattice, where the horizontal coordinates of the outgoing paths have the ordering $\delta_1 \leq \cdots \leq \delta_n$ (see Figure \ref{fig:factor}).

\begin{figure}
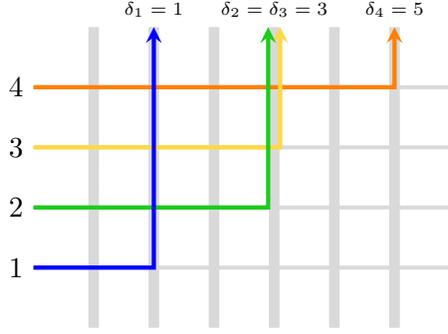

\tikz{0.8}{
\foreach\y in {2,...,5}{
\draw[lgray,line width=1.5pt] (1,\y) -- (8,\y);
}
\foreach\x in {2,...,7}{
\draw[lgray,line width=4pt] (\x,1) -- (\x,6);
}
\draw[ultra thick,orange,->] (1,5) -- (7,5) -- (7,6);
\draw[ultra thick,yellow,->] (1,4) -- (5.1,4) -- (5.1,6);
\draw[ultra thick,green,->] (1,3) -- (4.9,3) -- (4.9,6);
\draw[ultra thick,blue,->] (1,2) -- (3,2) -- (3,6);
\node[above] at (7,6) {\tiny $\delta_4 = 5$};
\node[above] at (5,6) {\tiny $\delta_2 = \delta_3 = 3$};
\node[above] at (3,6) {\tiny $\delta_1 = 1$};
\node[left] at (1,2) {$\tiny 1$};
\node[left] at (1,3) {$\tiny 2$};
\node[left] at (1,4) {$\tiny 3$};
\node[left] at (1,5) {$\tiny 4$};
}
\caption{In the case of an anti-dominant composition $\delta$, the colours $\{1,\dots,n\}$ leave the top of the lattice \eqref{f-def} in the same order (read left to right) as they entered it (read bottom to top). This leads to a unique lattice configuration and the total factorization of $f_{\delta}$. In the above example, one has $n=4$, $\delta = (1,3,3,5)$.}
\label{fig:factor}
\end{figure}

The path of colour $n$ travels at least as far, horizontally, as any other path to reach its outgoing destination; in so doing, it saturates the horizontal edges of the $n$-th row of the lattice all the way up to the column $\delta_n$. This forbids the other paths of colours $\{1,\dots,n-1\}$ from taking a horizontal step in the $n$-th row of the lattice, meaning that this row is frozen to a unique configuration. Reading off the Boltzmann weights of this frozen row, we obtain the recurrence
\begin{align*}
f_{(\delta_1,\dots,\delta_n)}(x_1,\dots,x_n)
=
\frac{1-s^2 q^{m_{\delta_n}(\delta)-1}}{1-sx_n}
\left( \frac{x_n-s}{1-s x_n} \right)^{\delta_n}
\cdot
f_{(\delta_1,\dots,\delta_{n-1})}(x_1,\dots,x_{n-1}),
\end{align*}
which is easily solved to yield \eqref{f-delta}.

\end{proof}

\section{Hecke algebra and its polynomial representation}
\label{ssec:hecke}

Let us recall the definition of the Hecke algebra of type $A_{n-1}$. It is the algebra generated by a family $T_1,\dots,T_{n-1}$, modulo the relations
\begin{align}
\label{hecke1}
(T_i - q)(T_i + 1) = 0,
\quad 
1 \leq i \leq n-1,
\qquad
T_i T_{i+1} T_i = T_{i+1} T_i T_{i+1},
\quad
1 \leq i \leq n-2,
\end{align}
as well as the commutativity property
\begin{align}
\label{hecke2}
[T_i,T_j] = 0,
\quad \forall\ i,j\ \text{such that}\ |i-j| > 1.
\end{align}
A well known representation of this algebra is its {\it polynomial representation}, which makes use of {\it Demazure--Lusztig operators}. Introduce firstly the simple transpositions $\mathfrak{s}_i$, acting on arbitrary functions $h$ of the alphabet 
$(x_1,\dots,x_n)$:
\index{s@$\mathfrak{s}_i$; elementary transposition}
\begin{align*}
\mathfrak{s}_i \cdot h(x_1,\dots,x_n)
:=
h(x_1,\dots,x_{i+1},x_i,\dots,x_n),
\quad
1 \leq i \leq n-1.
\end{align*}
Making use of these, one finds \cite[and references therein]{DescouensL} that the identification
\begin{align}
\label{hecke-poly}
\index{T@$T_i$; Hecke divided-difference operator}
T_i \mapsto q - \frac{x_i-q x_{i+1}}{x_i-x_{i+1}} (1-\mathfrak{s}_i),
\quad
1 \leq i \leq n-1,
\end{align}
provides a faithful representation of the relations \eqref{hecke1}, \eqref{hecke2}.\footnote{Note that the operators \eqref{hecke-poly} act to the {\it right}. This is in contradistinction with some authors, who prefer that Hecke generators act on the other side.} This representation is termed ``polynomial'' due to the fact that the operators \eqref{hecke-poly} act stably on polynomials, although in our setting we will be acting on rational functions in $(x_1,\dots,x_n)$. At all times, in writing $T_i$ we shall mean its operator form \eqref{hecke-poly}, rather than a generator of the abstract algebra.

It is easy to construct the inverses of the Hecke generators \eqref{hecke-poly}. To do that we define
\index{T@$\widehat{T}_i$; inverse Hecke}
\begin{align}
\label{hecke-inv}
\widehat{T}_i := T_i - q + 1 = 1 - \frac{x_i-q x_{i+1}}{x_i-x_{i+1}} (1-\mathfrak{s}_i),
\quad
1 \leq i \leq n-1,
\end{align}
and rewrite the first (quadratic) relation in \eqref{hecke1} in terms of $\widehat{T}_i$. We find that
\begin{align*}
(\widehat{T}_i - 1)(T_i + 1) = 0
\implies
\widehat{T}_i T_i - T_i + \widehat{T}_i - 1 = 0
\implies
\widehat{T}_i T_i = q 
\implies
\widehat{T}_i = q T_i^{-1}.
\end{align*}
In the coming sections we examine the action of the operators \eqref{hecke-poly} and \eqref{hecke-inv} on $f_{\mu}$ and $g_{\mu}$.

\section{Exchange relations for $f_{\mu}$}
\label{ssec:f-ex}

The goal of this section is to prove the following theorem:
\begin{thm}\label{thm:hecke-f}
Fix a composition $\mu = (\mu_1,\dots,\mu_n)$ and consider the corresponding non-symmetric spin Hall--Littlewood function $f_{\mu}(x_1,\dots,x_n)$. This function transforms under the action of \eqref{hecke-poly} and \eqref{hecke-inv} according to the rules
\begin{align}
\label{T-f}
T_i \cdot f_{\mu}(x_1,\dots,x_n) &= f_{\mathfrak{s}_i \cdot \mu}(x_1,\dots,x_n),
\quad
\text{if}\ \
\mu_i < \mu_{i+1},
\\
\label{That-f}
\widehat{T}_i \cdot f_{\mu}(x_1,\dots,x_n) &= f^{(i,i+1)}_{\mu}(x_1,\dots,x_n).
\end{align}
The first relation \eqref{T-f} holds for integers $i$ such that $\mu_i < \mu_{i+1}$, with $\mathfrak{s}_i \cdot \mu$ denoting the permutation $(\mu_1,\dots,\mu_{i+1},\mu_i,\dots,\mu_n)$ of the parts of $\mu$. The second relation \eqref{That-f} is valid for all $1 \leq i \leq n$, with $f^{(i,i+1)}_{\mu}$ denoting a permuted non-symmetric spin Hall--Littlewood function \eqref{sigma-f} corresponding to the permutation $\sigma =(i,i+1)$. 
\end{thm}

In spite of their visual similarity, the relations \eqref{T-f} and \eqref{That-f} express rather different transformative properties of the functions $f_{\mu}$. The second identity \eqref{That-f} involves the permutation of the colours $i$ and $i+1$ at the left boundary of the partition function \eqref{f-def}, and because of its local nature, is easily proved. Conversely, the first identity \eqref{T-f} involves the permutation of the colours $i$ and $i+1$ at the {\it top} boundary of \eqref{f-def}, which is (in general) a non-local change. The proof of \eqref{T-f} is, therefore, much more involved.

\begin{proof}
We begin by calculating the left hand side of \eqref{T-f} more explicitly:
\begin{align}
\nonumber
T_i \cdot f_{\mu}(x_1,\dots,x_n)
&=
q f_{\mu}(x_1,\dots,x_n)
-
\frac{x_i-qx_{i+1}}{x_i-x_{i+1}}
(1-\mathfrak{s}_i) f_{\mu}(x_1,\dots,x_n)
\\
\label{4.1-1}
&=
\frac{(q-1)x_i}{x_i-x_{i+1}} f_{\mu}(x_1,\dots,x_n)
+
\frac{x_i-qx_{i+1}}{x_i-x_{i+1}} f_{\mu}(x_1,\dots,x_{i+1},x_i,\dots,x_n).
\end{align}
Replacing the functions $f_{\mu}$ by their corresponding expectation values \eqref{f-HL}, equation \eqref{4.1-1} becomes
\begin{align}
\nonumber
T_i \cdot f_{\mu}(x_1,\dots,x_n)
&=
\frac{(q-1)x_i}{x_i-x_{i+1}}
\bra{\varnothing} \cdots \C_i(x_i) \C_{i+1}(x_{i+1}) \cdots \ket{\mu}
\\
\label{4.1-2}
&+
\frac{x_i-qx_{i+1}}{x_i-x_{i+1}}
\bra{\varnothing} \cdots \C_i(x_{i+1}) \C_{i+1}(x_i) \cdots \ket{\mu},
\end{align}
where we refrain from writing the $\mathcal{C}_j$ operators for $j \not= i,i+1$, since they play no role in the subsequent argument. In the final expectation value in \eqref{4.1-2}, the order of the variables $x_i,x_{i+1}$ is reversed; we can now use the commutation relation \eqref{CC>} with $i \mapsto i+1$, $j \mapsto i$, $y=x_{i+1}$, $x=x_{i}$ to restore the sequential ordering. Doing so, we find that
\begin{align*}
T_i \cdot f_{\mu}(x_1,\dots,x_n)
&=
\frac{(q-1)x_i}{x_i-x_{i+1}}
\bra{\varnothing} \cdots \C_i(x_i) \C_{i+1}(x_{i+1}) \cdots \ket{\mu}
+
\bra{\varnothing} \cdots \C_{i+1}(x_i) \C_{i}(x_{i+1}) \cdots \ket{\mu}
\\
&+
\frac{(1-q)x_{i+1}}{x_i-x_{i+1}}
\bra{\varnothing} \cdots \C_i(x_i) \C_{i+1}(x_{i+1}) \cdots \ket{\mu},
\end{align*}
which after combining the first and third terms reduces to
\begin{align}
\nonumber
T_i \cdot f_{\mu}(x_1,\dots,x_n)
&=
(q-1) \bra{\varnothing} \cdots \C_i(x_i) \C_{i+1}(x_{i+1}) \cdots \ket{\mu}
+
\bra{\varnothing} \cdots \C_{i+1}(x_i) \C_{i}(x_{i+1}) \cdots \ket{\mu}
\\
\label{key1}
&=
(q-1) f_{\mu}(x_1,\dots,x_n) + f^{(i,i+1)}_{\mu}(x_1,\dots,x_n).
\end{align}
This is the crucial identity behind the proof of Theorem \ref{thm:hecke-f}. Using the fact that $\widehat{T}_i = T_i - q + 1$, we rearrange \eqref{key1} to obtain an immediate proof of \eqref{That-f}. On the other hand, to complete the proof of \eqref{T-f}, we need to show that
\begin{align}
\label{key2}
(q-1) f_{\mu}(x_1,\dots,x_n) + f^{(i,i+1)}_{\mu}(x_1,\dots,x_n)
=
f_{\mathfrak{s}_i \cdot \mu}(x_1,\dots,x_n),
\qquad
\mu_i < \mu_{i+1}.
\end{align}
Our strategy will be to prove this statement at the level of the pre-fused non-symmetric spin Hall--Littlewood functions introduced in Definition \ref{def:prefus}; this turns out to be quite a bit easier than a direct proof of \eqref{key2}.\footnote{The reason why the proof is easier in the fundamental model \eqref{R-vert}, as opposed to the higher-spin version \eqref{generic-L}, is purely technical: our subsequent arguments are simplified by allowing at most one path to occupy horizontal {\it and vertical} lattice edges, which is not a valid assumption in the model \eqref{generic-L}.} Namely, we shall prove that
\begin{align}
\label{key3}
(q-1) \mathcal{F}_{\mu}\left(x_1,\dots,x_n; Y \right) 
+ 
\mathcal{F}^{(i,i+1)}_{\mu}\left(x_1,\dots,x_n; Y \right)
=
\mathcal{F}_{\mathfrak{s}_i \cdot \mu} \left(x_1,\dots,x_n; Y \right),
\qquad
\mu_i < \mu_{i+1},
\end{align}
where $\mathcal{F}_{\mu}(x_1,\dots,x_n;Y)$ denotes the partition function \eqref{f-fund} and $\mathcal{F}^{(i,i+1)}_{\mu}(x_1,\dots,x_n;Y)$ its permuted version \eqref{fsig-fund}, in the case of the permutation $\sigma = (i,i+1)$. 

We proceed towards the proof of \eqref{key3} with an important auxiliary definition and result. Let us focus on the four types of vertices in the model \eqref{R-vert} in which paths of colour $i$ and $i+1$ are the incoming states of a vertex (and its outgoing states, by conservation). We recall their weights below:
\begin{align}
\label{4vertices}
\begin{tabular}{|c|c|c|c|}
\hline
\tikz{0.7}{
\draw[blue,line width=1.5pt,->] (-1,0) -- (1,0);
\draw[green,line width=1.5pt,->] (0,-1) -- (0,1);
\node[left] at (-1,0) {\tiny $i$};\node[right] at (1,0) {\tiny $i$};
\node[below] at (0,-1) {\tiny $i+1$};\node[above] at (0,1) {\tiny $i+1$};
}
&
\tikz{0.7}{
\draw[green,line width=1.5pt,->] (-1,0) -- (1,0);
\draw[blue,line width=1.5pt,->] (0,-1) -- (0,1);
\node[left] at (-1,0) {\tiny $i+1$};\node[right] at (1,0) {\tiny $i+1$};
\node[below] at (0,-1) {\tiny $i$};\node[above] at (0,1) {\tiny $i$};
}
&
\tikz{0.7}{
\draw[blue,line width=1.5pt,->] (-1,0) -- (0,0) -- (0,1);
\draw[green,line width=1.5pt,->] (0,-1) -- (0,0) -- (1,0);
\node[left] at (-1,0) {\tiny $i$};\node[right] at (1,0) {\tiny $i+1$};
\node[below] at (0,-1) {\tiny $i+1$};\node[above] at (0,1) {\tiny $i$};
}
&
\tikz{0.7}{
\draw[green,line width=1.5pt,->] (-1,0) -- (0,0) -- (0,1);
\draw[blue,line width=1.5pt,->](0,-1) -- (0,0) -- (1,0);
\node[left] at (-1,0) {\tiny $i+1$};\node[right] at (1,0) {\tiny $i$};
\node[below] at (0,-1) {\tiny $i$};\node[above] at (0,1) {\tiny $i+1$};
}
\\
\hline
\rule{0pt}{4ex}
$ \dfrac{q(1-z)}{1-qz} $ & $ \dfrac{1-z}{1-qz} $ & $ \dfrac{1-q}{1-qz} $ & 
$ \dfrac{(1-q) z}{1-qz} $
\\
\hline
\end{tabular}
\end{align}

\begin{defn}
Let $(z_0,\dots,z_k)$ be $k+1$ arbitrary (complex) parameters. We introduce the sum \index{p99@$p_k(a,b;c,d)$}
\begin{align}
\label{pk-def}
p_k(a,b;c,d)
:=
\sum_{i_1,\dots,i_k \in \{a,b\}}\ 
\sum_{j_1,\dots,j_k \in \{a,b\}}
R_{z_0}(a,b;i_1,j_1)
R_{z_1}(j_1,i_1;i_2,j_2)
\dots
R_{z_k}(j_k,i_k;c,d).
\end{align}
The function $p_k(a,b;c,d)$ admits the following graphical description, and can be interpreted as a multi-vertex extension of the original model:
\begin{align}
\label{multivertex}
p_k(a,b;c,d)
=
\sum_{i_1,\dots,i_k \in \{a,b\}}\ 
\sum_{j_1,\dots,j_k \in \{a,b\}}
\tikz{0.7}{
\draw[densely dotted] (0.5,0) arc (0:90:0.5); \node at (0.2,0.2) {\tiny $z_0$};
\draw[densely dotted] (1.5,1) arc (0:90:0.5); \node at (1.2,1.2) {\tiny $z_1$};
\node[rotate=90] at (2,2) {\tiny $\ddots$};
\draw[densely dotted] (3.5,3) arc (0:90:0.5); \node at (3.2,3.2) {\tiny $z_k$};
\draw[lgray,line width=1.5pt,->,rounded corners] (-1,0) -- (1,0) -- (1,2);
\draw[lgray,line width=1.5pt,->,rounded corners] (0,-1) -- (0,1) -- (2,1);
\draw[lgray,line width=1.5pt,->,rounded corners] (2,3) -- (4,3);
\draw[lgray,line width=1.5pt,->,rounded corners] (3,2) -- (3,4);
\node[left] at (-1,0) {\tiny $b$};\node[right] at (1,0) {\tiny $j_1$};
\node[below] at (0,-1) {\tiny $a$};\node[above] at (0,1) {\tiny $i_1$};
\node[above] at (1,2) {\tiny $i_2$};\node[right] at (2,1) {\tiny $j_2$};
\node[above] at (2,3) {\tiny $i_k$};\node[right] at (3,2) {\tiny $j_k$};
\node[above] at (3,4) {\tiny $c$};\node[right] at (4,3) {\tiny $d$};
}
\end{align}
in which the indices on all internal edges are summed over the values $\{a,b\}$.
\end{defn}

\begin{prop}
The function $p_k(a,b;c,d)$ vanishes unless either $(a,b) = (c,d)$ or $(a,b) = (d,c)$. In the case $a=b=c=d$, one has $p_k(a,a;a,a) = 1$. Further, for $a<b$, one has the relations
\begin{gather}
\label{pk-1}
(q-1) \cdot p_k(a,b;a,b) + p_k(b,a;a,b) = p_k(a,b;b,a),
\\
\label{pk-2}
p_k(b,a;b,a) = q \cdot p_k(a,b;a,b),
\end{gather}
which hold for all $k \geq 0$.
\end{prop}

\begin{proof}
Since the states $(a,b)$ are the sole incoming colours in the partition function $p_k(a,b;c,d)$, and the states $(c,d)$ are correspondingly the sole outgoing colours, it follows that $p_k(a,b;c,d)$ vanishes unless $(a,b) = (c,d)$ or $(a,b) = (d,c)$. Similarly, the partition function $p_k(a,a;a,a)$ has all incoming and outgoing colours equal to the value $a$; this freezes each vertex in the sum \eqref{multivertex} to a vertex of the form \eqref{R-weights-a}, with weight $1$.

Let us now proceed to \eqref{pk-1} and \eqref{pk-2}, which we prove by induction on $k$. Beginning with the case $k=0$, \ie\ the vertex model \eqref{R-vert} itself, one checks that indeed
\begin{gather*}
(q-1) \cdot p_0(a,b;a,b) + p_0(b,a;a,b) = (q-1) \frac{1-z_0}{1-q z_0} + \frac{1-q}{1-q z_0}
=
\frac{z_0(1-q)}{1-q z_0}
=
p_0(a,b;b,a),
\\
p_0(b,a;b,a) = \frac{q (1-z_0)}{1-q z_0} = q \cdot p_0(a,b;a,b),
\end{gather*}
for $a<b$. This proves the relations \eqref{pk-1} and \eqref{pk-2} for $k=0$. Now assume that they hold when $k=m$, for some $m \geq 0$. Using the definition \eqref{pk-def} of $p_{m+1}(a,b;a,b)$ and $p_{m+1}(b,a;a,b)$, we decompose explicitly the summations over $i_{m+1}$ and $j_{m+1}$. This produces the following recurrences:
\begin{align}
\label{p-rec1}
p_{m+1}(a,b;a,b) = 
\left( \frac{1-q}{1-q z_{m+1}} \right) p_m(a,b;a,b) 
+ 
\left( \frac{1-z_{m+1}}{1-q z_{m+1}} \right)p_m(a,b;b,a),
\\
\label{p-rec2}
p_{m+1}(b,a;a,b) =
\left( \frac{1-q}{1-q z_{m+1}} \right) p_m(b,a;a,b) 
+ 
\left( \frac{1-z_{m+1}}{1-q z_{m+1}} \right)p_m(b,a;b,a).
\end{align}
We compute $(q-1) \cdot p_{m+1}(a,b;a,b) + p_{m+1}(b,a;a,b)$, as a linear combination of the right hand sides of \eqref{p-rec1} and \eqref{p-rec2}. From there, we can replace $p_m(b,a;a,b)$ by $p_m(a,b;b,a) + (1-q)\cdot p_m(a,b;a,b)$ (by virtue of \eqref{pk-1} at $k=m$), and $p_m(b,a;b,a)$ by $q \cdot p_m(a,b;a,b)$ (by virtue of \eqref{pk-2} at $k=m$). After simplifying, one obtains
\begin{multline*}
(q-1) \cdot p_{m+1}(a,b;a,b) + p_{m+1}(b,a;a,b)
\\
=
q \left( \frac{1-z_{m+1}}{1-q z_{m+1}} \right) p_m(a,b;a,b) 
+ 
z_{m+1} \left( \frac{1-q}{1-q z_{m+1}} \right) p_m(a,b;b,a)
=
p_{m+1} (a,b;b,a),
\end{multline*}
with the final equality being a valid recurrence for $p_{m+1}(a,b;b,a)$, that can also be deduced by decomposing the sums over $i_{m+1}$ and $j_{m+1}$ in its definition \eqref{pk-def}. This proves that \eqref{pk-1} holds for $k=m+1$.

To complete the proof we need to show, further, that \eqref{pk-2} holds for $k=m+1$. We begin by writing down the recurrence 
\begin{align*}
p_{m+1}(b,a;b,a)
=
q \left( \frac{1-z_{m+1}}{1-q z_{m+1}} \right) p_m(b,a;a,b) 
+ 
z_{m+1} \left( \frac{1-q}{1-q z_{m+1}} \right) p_m(b,a;b,a),
\end{align*}
which holds by similar considerations to those used to derive the recurrences above. Performing the replacements $p_m(b,a;a,b) \mapsto p_m(a,b;b,a) + (1-q)\cdot p_m(a,b;a,b)$ and $p_m(b,a;b,a) \mapsto q \cdot p_m(a,b;a,b)$, as previously, after simplification we read the relation
\begin{align*}
p_{m+1}(b,a;b,a)
=
q  \left( \frac{1-q}{1-q z_{m+1}} \right) p_m(a,b;a,b) 
+ 
q \left( \frac{1-z_{m+1}}{1-q z_{m+1}} \right)p_m(a,b;b,a)
=
q \cdot p_{m+1}(a,b;a,b),
\end{align*}
with the final equality being precisely that of \eqref{p-rec1}. This proves that \eqref{pk-2} holds for $k=m+1$, and both identities \eqref{pk-1}, \eqref{pk-2} now hold for arbitrary $k$ values, by induction.

\end{proof}

Let us return to the partition function $\mathcal{F}_{\mu}(x_1,\dots,x_n;Y)$ and apply the result \eqref{pk-1}. We will seek a refinement of \eqref{f-fund}, where we specify a collection of vertices in the lattice where the paths of colours $i$ and $i+1$ must cross or touch. 
\begin{defn}
\label{def:F-refine}
Fix an integer $k \geq 0$ and let $\{(j_1,\ell_1),\dots,(j_k,\ell_k)\}$ be a set of lattice coordinates, such that $0 \leq j_1 < \cdots < j_k$ and $1 \leq \ell_1 < \cdots < \ell_k \leq n$. We introduce the set $S_i((j_1,\ell_1),\dots,(j_k,\ell_k))$ of all lattice configurations such that one of the four vertices \eqref{4vertices} is present at each of the points $\{(j_1,\ell_1),\dots,(j_k,\ell_k)\}$, and at no other places in the lattice. By agreement, when $k=0$, $S_i(\varnothing)$ is the set of all configurations such that the four vertices \eqref{4vertices} are not present anywhere in the lattice. From this we define
\begin{align}
\label{F-refine}
\mathcal{F}_{\mu,i}\Big(x_1,\dots,x_n;Y;(j_1,\ell_1),\dots,(j_k,\ell_k)\Big)
&:=
\sum_{\substack{
\mathcal{C} : (1,\dots,n) \mapsto (\mu_1,\dots,\mu_n)
\\ \\
\mathcal{C} \in S_i((j_1,\ell_1),\dots,(j_k,\ell_k))
}}
W_{\mathcal{C}}\left(x_1,\dots,x_n;Y\right),
\\
\label{Fsig-refine}
\mathcal{F}_{\mu,i}^{\sigma}\Big(x_1,\dots,x_n;Y;(j_1,\ell_1),\dots,(j_k,\ell_k)\Big)
&:=
\sum_{\substack{
\mathcal{C} : (\sigma^{-1}(1),\dots,\sigma^{-1}(n)) \mapsto (\mu_1,\dots,\mu_n)
\\ \\
\mathcal{C} \in S_i((j_1,\ell_1),\dots,(j_k,\ell_k))
}}
W_{\mathcal{C}}\left(x_1,\dots,x_n;Y\right).
\end{align}
\end{defn}
The functions \eqref{F-refine} and \eqref{Fsig-refine} are clear refinements of \eqref{f-fund} and \eqref{fsig-fund}, in the sense that
\begin{align}
\label{refine1}
\mathcal{F}_{\mu} (x_1,\dots,x_n;Y)
&=
\sum_{k=0}^{n}
\sum_{\substack{j_1 < \cdots < j_k \\ \ell_1 < \cdots < \ell_k}}
\mathcal{F}_{\mu,i}\Big(x_1,\dots,x_n;Y;(j_1,\ell_1),\dots,(j_k,\ell_k)\Big),
\\
\label{refine2}
\mathcal{F}_{\mu}^{\sigma}(x_1,\dots,x_n;Y)
&=
\sum_{k=0}^{n}
\sum_{\substack{j_1 < \cdots < j_k \\ \ell_1 < \cdots < \ell_k}}
\mathcal{F}^{\sigma}_{\mu,i}\Big(x_1,\dots,x_n;Y;(j_1,\ell_1),\dots,(j_k,\ell_k)\Big),
\end{align}
where it is obviously sufficient to restrict the number $k$ of crossings/touchings of the paths $i$ and $i+1$ to at most $n$. Now we come to the key observation:
\begin{prop}
\label{prop:trio}
Fix two integers $1 \leq i \leq n-1$ and $k \geq 1$, a set of lattice coordinates $\{(j_1,\ell_1),\dots,(j_k,\ell_k)\}$ as in Definition \ref{def:F-refine}, and a composition 
$\mu$ such that $\mu_i < \mu_{i+1}$. Then we have the following trio of equal quantities:
\begin{multline}
\label{trio}
\frac{
\mathcal{F}_{\mu,i}\Big(x_1,\dots,x_n;Y;(j_1,\ell_1),\dots,(j_k,\ell_k)\Big)
}
{p_{k-1}(i,i+1;i,i+1)}
\\
=
\frac{
\mathcal{F}_{\mu,i}^{(i,i+1)}\Big(x_1,\dots,x_n;Y;(j_1,\ell_1),\dots,(j_k,\ell_k)\Big)
}
{p_{k-1}(i+1,i;i,i+1)}
\\
=
\frac{
\mathcal{F}_{\mathfrak{s}_i \cdot \mu,i}\Big(x_1,\dots,x_n;Y;(j_1,\ell_1),\dots,(j_k,\ell_k)\Big)
}
{p_{k-1}(i,i+1;i+1,i)},
\end{multline}
where the functions in the numerators are as given by Definition \ref{def:F-refine}, while the functions in the denominators are given by \eqref{pk-def} with $z_0,\dots,z_{k-1}$ identified with the spectral parameters of the vertices $(j_1,\ell_1),\dots,(j_k,\ell_k)$. In the case $k=0$ one has, more simply,
\begin{align}
\label{no-cross1}
\mathcal{F}_{\mu,i}\Big(x_1,\dots,x_n;Y;\varnothing\Big) &= 0,
\\
\label{no-cross2}
\mathcal{F}_{\mu,i}^{(i,i+1)}\Big(x_1,\dots,x_n;Y;\varnothing\Big)
&=
\mathcal{F}_{\mathfrak{s}_i \cdot \mu,i}
\Big(x_1,\dots,x_n;Y;\varnothing\Big).
\end{align}
\end{prop}

\begin{proof}
Our proof is inspired by the typical approach to the Lindstr\"om--Gessel--Viennot Lemma (see, for example, \cite[Chapter 32]{Aigner-book}). We refine the function $\mathcal{F}_{\mu,i}(x_1,\dots,x_n;Y;(j_1,\ell_1),\dots,(j_k,\ell_k))$ even further. Namely we fix the configurations of the paths of colours $\{1,\dots,i-1,i+2,\dots,n\}$ completely, to some given choice; we also specify the full collection of edges within the lattice which must be covered by paths of colours $i$ or $i+1$. Collectively, we refer to all of these choices as {\it edge data} $\mathcal{D}$. Specifying $\mathcal{D}$ exhausts all extraneous degrees of freedom and leaves us with a sum over $2^{k-1}$ terms; they arise from the choice of two possible configurations at the intersection points $(j_1,\ell_1),\dots,(j_{k-1},\ell_{k-1})$ of paths $i$ and $i+1$.\footnote{Note that there is no choice at the final intersection point $(j_k,\ell_k)$, since the outgoing coordinates of the paths are fixed, which uniquely determines what happens at the final crossing/touching of paths $i$ and $i+1$.} See Figure \ref{fig:refined-configs} for an illustration.

\begin{figure}
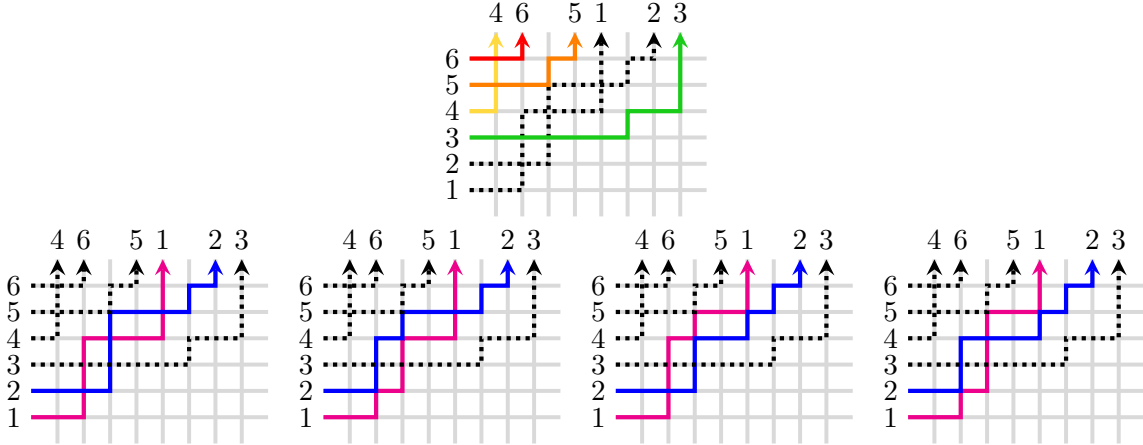

\tikz{0.35}{
\foreach\y in {0,...,5}{
\draw[lgray,line width=1.5pt] (0,\y) -- (9,\y);
}
\foreach\x in {1,...,8}{
\draw[lgray,line width=1.5pt] (\x,-1) -- (\x,6);
}
\draw[ultra thick,dotted,->] (0,0) -- (2,0) -- (2,3) -- (5,3) -- (5,6);
\draw[ultra thick,dotted,->] (0,1) -- (3,1) -- (3,4) -- (6,4) -- (6,5) -- (7,5) -- (7,6);
\draw[ultra thick,green,->] (0,2) -- (6,2) -- (6,3) -- (8,3) -- (8,6);
\draw[ultra thick,yellow,->] (0,3) -- (1,3) -- (1,6);
\draw[ultra thick,orange,->] (0,4) -- (3,4) -- (3,5) -- (4,5) -- (4,6);
\draw[ultra thick,red,->] (0,5) -- (2,5) -- (2,6);
\node[above] at (1,6) {$\tiny 4$};
\node[above] at (2,6) {$\tiny 6$};
\node[above] at (4,6) {$\tiny 5$};
\node[above] at (5,6) {$\tiny 1$};
\node[above] at (7,6) {$\tiny 2$};
\node[above] at (8,6) {$\tiny 3$};
\node[left] at (0,0) {$\tiny 1$};
\node[left] at (0,1) {$\tiny 2$};
\node[left] at (0,2) {$\tiny 3$};
\node[left] at (0,3) {$\tiny 4$};
\node[left] at (0,4) {$\tiny 5$};
\node[left] at (0,5) {$\tiny 6$};
}
\\
\tikz{0.35}{
\foreach\y in {0,...,5}{
\draw[lgray,line width=1.5pt] (0,\y) -- (9,\y);
}
\foreach\x in {1,...,8}{
\draw[lgray,line width=1.5pt] (\x,-1) -- (\x,6);
}
\draw[ultra thick,magenta,->] (0,0) -- (2,0) -- (2,3) -- (5,3) -- (5,6);
\draw[ultra thick,blue,->] (0,1) -- (3,1) -- (3,4) -- (6,4) -- (6,5) -- (7,5) -- (7,6);
\draw[ultra thick,dotted,->] (0,2) -- (6,2) -- (6,3) -- (8,3) -- (8,6);
\draw[ultra thick,dotted,->] (0,3) -- (1,3) -- (1,6);
\draw[ultra thick,dotted,->] (0,4) -- (3,4) -- (3,5) -- (4,5) -- (4,6);
\draw[ultra thick,dotted,->] (0,5) -- (2,5) -- (2,6);
\node[above] at (1,6) {$\tiny 4$};
\node[above] at (2,6) {$\tiny 6$};
\node[above] at (4,6) {$\tiny 5$};
\node[above] at (5,6) {$\tiny 1$};
\node[above] at (7,6) {$\tiny 2$};
\node[above] at (8,6) {$\tiny 3$};
\node[left] at (0,0) {$\tiny 1$};
\node[left] at (0,1) {$\tiny 2$};
\node[left] at (0,2) {$\tiny 3$};
\node[left] at (0,3) {$\tiny 4$};
\node[left] at (0,4) {$\tiny 5$};
\node[left] at (0,5) {$\tiny 6$};
}
\tikz{0.35}{
\foreach\y in {0,...,5}{
\draw[lgray,line width=1.5pt] (0,\y) -- (9,\y);
}
\foreach\x in {1,...,8}{
\draw[lgray,line width=1.5pt] (\x,-1) -- (\x,6);
}
\draw[ultra thick,magenta,->] (0,0) -- (2,0) -- (2,1) -- (3,1) -- (3,3) -- (5,3) -- (5,6);
\draw[ultra thick,blue,->] (0,1) -- (2,1) -- (2,3) -- (3,3) -- (3,4) -- (6,4) -- (6,5) -- (7,5) -- (7,6);
\draw[ultra thick,dotted,->] (0,2) -- (6,2) -- (6,3) -- (8,3) -- (8,6);
\draw[ultra thick,dotted,->] (0,3) -- (1,3) -- (1,6);
\draw[ultra thick,dotted,->] (0,4) -- (3,4) -- (3,5) -- (4,5) -- (4,6);
\draw[ultra thick,dotted,->] (0,5) -- (2,5) -- (2,6);
\node[above] at (1,6) {$\tiny 4$};
\node[above] at (2,6) {$\tiny 6$};
\node[above] at (4,6) {$\tiny 5$};
\node[above] at (5,6) {$\tiny 1$};
\node[above] at (7,6) {$\tiny 2$};
\node[above] at (8,6) {$\tiny 3$};

\node[left] at (0,0) {$\tiny 1$};
\node[left] at (0,1) {$\tiny 2$};
\node[left] at (0,2) {$\tiny 3$};
\node[left] at (0,3) {$\tiny 4$};
\node[left] at (0,4) {$\tiny 5$};
\node[left] at (0,5) {$\tiny 6$};
}
\tikz{0.35}{
\foreach\y in {0,...,5}{
\draw[lgray,line width=1.5pt] (0,\y) -- (9,\y);
}
\foreach\x in {1,...,8}{
\draw[lgray,line width=1.5pt] (\x,-1) -- (\x,6);
}
\draw[ultra thick,magenta,->] (0,0) -- (2,0) -- (2,3) -- (3,3) -- (3,4) -- (5,4) -- (5,6);
\draw[ultra thick,blue,->] (0,1) -- (3,1) -- (3,3) -- (5,3) -- (5,4) -- (6,4) -- (6,5) -- (7,5) -- (7,6);
\draw[ultra thick,dotted,->] (0,2) -- (6,2) -- (6,3) -- (8,3) -- (8,6);
\draw[ultra thick,dotted,->] (0,3) -- (1,3) -- (1,6);
\draw[ultra thick,dotted,->] (0,4) -- (3,4) -- (3,5) -- (4,5) -- (4,6);
\draw[ultra thick,dotted,->] (0,5) -- (2,5) -- (2,6);
\node[above] at (1,6) {$\tiny 4$};
\node[above] at (2,6) {$\tiny 6$};
\node[above] at (4,6) {$\tiny 5$};
\node[above] at (5,6) {$\tiny 1$};
\node[above] at (7,6) {$\tiny 2$};
\node[above] at (8,6) {$\tiny 3$};
\node[left] at (0,0) {$\tiny 1$};
\node[left] at (0,1) {$\tiny 2$};
\node[left] at (0,2) {$\tiny 3$};
\node[left] at (0,3) {$\tiny 4$};
\node[left] at (0,4) {$\tiny 5$};
\node[left] at (0,5) {$\tiny 6$};
}
\tikz{0.35}{
\foreach\y in {0,...,5}{
\draw[lgray,line width=1.5pt] (0,\y) -- (9,\y);
}
\foreach\x in {1,...,8}{
\draw[lgray,line width=1.5pt] (\x,-1) -- (\x,6);
}
\draw[ultra thick,magenta,->] (0,0) -- (2,0) -- (2,1) -- (3,1) -- (3,4) -- (5,4) -- (5,6);
\draw[ultra thick,blue,->] (0,1) -- (2,1) -- (2,3) -- (5,3) -- (5,4) -- (6,4) -- (6,5) -- (7,5) -- (7,6);
\draw[ultra thick,dotted,->] (0,2) -- (6,2) -- (6,3) -- (8,3) -- (8,6);
\draw[ultra thick,dotted,->] (0,3) -- (1,3) -- (1,6);
\draw[ultra thick,dotted,->] (0,4) -- (3,4) -- (3,5) -- (4,5) -- (4,6);
\draw[ultra thick,dotted,->] (0,5) -- (2,5) -- (2,6);
\node[above] at (1,6) {$\tiny 4$};
\node[above] at (2,6) {$\tiny 6$};
\node[above] at (4,6) {$\tiny 5$};
\node[above] at (5,6) {$\tiny 1$};
\node[above] at (7,6) {$\tiny 2$};
\node[above] at (8,6) {$\tiny 3$};
\node[left] at (0,0) {$\tiny 1$};
\node[left] at (0,1) {$\tiny 2$};
\node[left] at (0,2) {$\tiny 3$};
\node[left] at (0,3) {$\tiny 4$};
\node[left] at (0,4) {$\tiny 5$};
\node[left] at (0,5) {$\tiny 6$};
}
\caption{Top panel: An example for $n=6$, in which the configurations of paths $3,4,5,6$ are pre-determined, as is the set of edges covered by paths $1,2$; collectively we refer to these choices as the data $\mathcal{D}$. Bottom panel: With these restrictions, just four lattice configurations are possible; adding the weights of the four configurations, one obtains $p_2(1,2;1,2)$ times a factor which depends only on $\mathcal{D}$.}
\label{fig:refined-configs}
\end{figure}

Computing the resulting sum of $2^{k-1}$ terms, we find that it is equal to 
$p_{k-1}(i,i+1;i,i+1) W_{\mathcal{D}}$, where the factor $W_{\mathcal{D}}$ depends on our choice of $\mathcal{D}$. Indeed, the sum almost precisely replicates that of \eqref{multivertex}, the only difference being that the intersection points between paths $i$ and $i+1$ can now be spread through the lattice.

One can impose a similar refinement of the functions $\mathcal{F}_{\mu,i}^{(i,i+1)}(x_1,\dots,x_n;Y;(j_1,\ell_1),\dots,(j_k,\ell_k))$ and $\mathcal{F}_{\mathfrak{s}_i \cdot \mu,i}(x_1,\dots,x_n;Y;(j_1,\ell_1),\dots,(j_k,\ell_k))$. We again fix the configurations of the paths of colours $\{1,\dots,i-1,i+2,\dots,n\}$ completely and the full collection of edges within the lattice which must be covered by paths of colours $i$ or $i+1$, choosing precisely the same data $\mathcal{D}$ as above.\footnote{We are able to choose exactly the same data $\mathcal{D}$ because both of these partition functions differ from $\mathcal{F}_{\mu,i}(x_1,\dots,x_n;Y;(j_1,\ell_1),\dots,(j_k,\ell_k))$ only by permuting the incoming/outgoing positions of the paths $i$ and $i+1$.} This reduces both to a sum of $2^{k-1}$ terms; they can be computed as $p_{k-1}(i+1,i;i,i+1)W_{\mathcal{D}}$ and $p_{k-1}(i,i+1;i+1,i)W_{\mathcal{D}}$, respectively. Crucially, the overall factor $W_{\mathcal{D}}$ is the same in all three cases.

This establishes the result \eqref{trio} at the fully refined level; to prove \eqref{trio} as stated, one only needs to sum over all possible edge data $\mathcal{D}$.

The final two statements \eqref{no-cross1} and \eqref{no-cross2} are very easily proved. If $\mu_i < \mu_{i+1}$, there is no possible configuration in $\mathcal{F}_{\mu,i}(x_1,\dots,x_n;Y;\varnothing)$, as the paths $i$ and $i+1$ need to cross somewhere in the lattice to reach their final destinations; this yields \eqref{no-cross1}. For the proof of \eqref{no-cross2}, we define an involution on configurations in $\mathcal{F}_{\mu,i}^{(i,i+1)}(x_1,\dots,x_n;Y;\varnothing)$ that maps them to configurations in $\mathcal{F}_{\mathfrak{s}_i \cdot \mu,i}(x_1,\dots,x_n;Y;\varnothing)$; the involution is achieved by swapping the colours of the paths $i$ and $i+1$. Since the paths $i$ and $i+1$ are by assumption non-intersecting, it is straightforward to conclude that performing this involution on any configuration $\mathcal{C}$ preserves its Boltzmann weight. The statement \eqref{no-cross2} follows. 
\end{proof}

We are now in the position to prove \eqref{key3}. Using the refined formulae \eqref{refine1} and \eqref{refine2}, we compute
\begin{multline}
\label{eq-something}
(q-1) \mathcal{F}_{\mu}\left(x_1,\dots,x_n; Y \right) 
+ 
\mathcal{F}^{(i,i+1)}_{\mu}\left(x_1,\dots,x_n; Y \right)
=
\mathcal{F}_{\mu,i}^{(i,i+1)}\Big(x_1,\dots,x_n;Y;\varnothing\Big)
\\
+
\sum_{k=1}^{n}
\sum_{\substack{j_1 < \cdots < j_k \\ \ell_1 < \cdots < \ell_k}}
\mathcal{F}_{\mu,i}\Big(x_1,\dots,x_n;Y;(j_1,\ell_1),\dots,(j_k,\ell_k)\Big)
\left[
q-1 + \frac{p_{k-1}(i+1,i;i,i+1)}{p_{k-1}(i,i+1;i,i+1)}
\right],
\end{multline}
where we have split the sum into its $k=0$ and $k \geq 1$ constituents, used \eqref{no-cross1} to eliminate $(q-1) \mathcal{F}_{\mu,i}(x_1,\dots,x_n;Y;\varnothing)$, and combined terms inside the final summation using the first equality in \eqref{trio}. Now we employ the identity \eqref{no-cross2} to the first term in \eqref{eq-something}, and the relation \eqref{pk-1} to the summand; the result is
\begin{multline}
(q-1) \mathcal{F}_{\mu}\left(x_1,\dots,x_n; Y \right) 
+ 
\mathcal{F}^{(i,i+1)}_{\mu}\left(x_1,\dots,x_n; Y \right)
=
\mathcal{F}_{\mathfrak{s}_i \cdot \mu,i}\Big(x_1,\dots,x_n;Y;\varnothing\Big)
\\
+
\sum_{k=1}^{n}
\sum_{\substack{j_1 < \cdots < j_k \\ \ell_1 < \cdots < \ell_k}}
\mathcal{F}_{\mu,i}\Big(x_1,\dots,x_n;Y;(j_1,\ell_1),\dots,(j_k,\ell_k)\Big)
\frac{p_{k-1}(i,i+1;i+1,i)}{p_{k-1}(i,i+1;i,i+1)}.
\end{multline}
Finally, we use the second equality in \eqref{trio} to rewrite the above summation, leading to
\begin{multline}
(q-1) \mathcal{F}_{\mu}\left(x_1,\dots,x_n; Y \right) 
+ 
\mathcal{F}^{(i,i+1)}_{\mu}\left(x_1,\dots,x_n; Y \right)
\\
=
\mathcal{F}_{\mathfrak{s}_i \cdot \mu,i}\Big(x_1,\dots,x_n;Y;\varnothing\Big)
+
\sum_{k=1}^{n}
\sum_{\substack{j_1 < \cdots < j_k \\ \ell_1 < \cdots < \ell_k}}
\mathcal{F}_{\mathfrak{s}_i\cdot\mu,i}\Big(x_1,\dots,x_n;Y;(j_1,\ell_1),\dots,(j_k,\ell_k)\Big),
\end{multline}
where the final expression is nothing but the refined version of 
$\mathcal{F}_{\mathfrak{s}_i\cdot\mu}\left(x_1,\dots,x_n; Y \right)$, \cf\ \eqref{refine1}. This completes the proof of \eqref{key3}, and consequently (via an application of Theorem \ref{thm:fusion-F}) that of \eqref{key2}, concluding the proof of Theorem \ref{thm:hecke-f}.

\end{proof}

\begin{ex}
It is easy to check, directly, that \eqref{key2} holds in the $i=1$, $n=2$ case. Since $\mu_1 < \mu_2$, one finds that the partition function representation of $f_{(\mu_1,\mu_2)}(x_1,x_2)$ consists of a unique configuration:
\begin{multline*}
f_{(\mu_1,\mu_2)}(x_1,x_2)
=
\tikz{0.7}{
\foreach\y in {4,5}{
\draw[lgray,line width=1.5pt] (2,\y) -- (8,\y);
}
\foreach\x in {3,...,7}{
\draw[lgray,line width=4pt] (\x,3) -- (\x,6);
}
\node[left] at (1.5,4) {$x_1 \rightarrow$};
\node[left] at (1.5,5) {$x_2 \rightarrow$};
\node[above] at (7,6) {\footnotesize$\mu_2$};
\node[above] at (5.5,6) {$<$};
\node[above] at (4,6) {\footnotesize$\mu_1$};
\node[left] at (2,4) {$1$};
\node[left] at (2,5) {$2$};
\node[right] at (8,4) {$0$};
\node[right] at (8,5) {$0$};
\draw[ultra thick,blue,->] (2,4) -- (4,4) -- (4,6);
\draw[ultra thick,green,->] (2,5) -- (7,5) -- (7,6);
}
\\
=
\left( \frac{x_1-s}{1-s x_1} \right)^{\mu_1}
\left( \frac{1-s^2}{1-s x_1} \right)
\left( \frac{x_2-s}{1-s x_2} \right)^{\mu_2}
\left( \frac{1-s^2}{1-s x_2} \right),
\end{multline*}
where we have used the Boltzmann weights \eqref{s-weights} to write the final expression. On the other hand, the order of the left incoming colours is permuted in $f^{(1,2)}_{(\mu_1,\mu_2)}(x_1,x_2)$. Its partition function admits a total of $\mu_2 - \mu_1+1$ configurations:
\begin{multline*}
f^{(1,2)}_{(\mu_1,\mu_2)}(x_1,x_2)
=
\\
\tikz{0.7}{
\foreach\y in {4,5}{
\draw[lgray,line width=1.5pt] (2,\y) -- (8,\y);
}
\foreach\x in {3,...,7}{
\draw[lgray,line width=4pt] (\x,3) -- (\x,6);
}
\node[left] at (1.5,4) {$x_1 \rightarrow$};
\node[left] at (1.5,5) {$x_2 \rightarrow$};
\node[above] at (7,6) {\footnotesize$\mu_2$};
\node[above] at (5.5,6) {$<$};
\node[above] at (4,6) {\footnotesize$\mu_1$};
\node[left] at (2,4) {$2$};
\node[left] at (2,5) {$1$};
\node[right] at (8,4) {$0$};
\node[right] at (8,5) {$0$};
\draw[ultra thick,green,->] (2,4) -- (4,4) -- (4,5) -- (7,5) -- (7,6);
\draw[ultra thick,blue,->] (2,5) -- (4,5) -- (4,6);
}
+
\sum_{\mu_1 < k \leq \mu_2}
\tikz{0.7}{
\foreach\y in {4,5}{
\draw[lgray,line width=1.5pt] (2,\y) -- (8,\y);
}
\foreach\x in {3,...,7}{
\draw[lgray,line width=4pt] (\x,3) -- (\x,6);
}
\node[left] at (1.5,4) {$x_1 \rightarrow$};
\node[left] at (1.5,5) {$x_2 \rightarrow$};
\node[above] at (7,6) {\footnotesize$\ \ \ \mu_2$};
\node[above] at (5,6) {$<$};
\node[above] at (4,6) {\footnotesize$\mu_1$};
\node[above] at (6,6) {\footnotesize$k$};
\node[above] at (6.5,5.93) {\ $\leq$};
\node[left] at (2,4) {$2$};
\node[left] at (2,5) {$1$};
\node[right] at (8,4) {$0$};
\node[right] at (8,5) {$0$};
\draw[ultra thick,green,->] (2,4) -- (6,4) -- (6,5) -- (7,5) -- (7,6);
\draw[ultra thick,blue,->] (2,5) -- (4,5) -- (4,6);
}
\\
=
\left( \frac{x_1-s}{1-s x_1} \right)^{\mu_1}
\left( \frac{1-s^2}{1-s x_1} \right)
\left( \frac{x_2-s}{1-s x_2} \right)^{\mu_2-1}
\left( \frac{1-s^2}{1-s x_2} \right)
\frac{x_2 (1-q)}{1-sx_2}
+
\sum_{\mu_1 < k \leq \mu_2}
W_k(\mu_1,\mu_2),
\end{multline*}
where we do not write the weights in the sum over $k$ explicitly, but abbreviate the weight of each configuration by $W_k(\mu_1,\mu_2)$. Summing the two terms on the left hand side of \eqref{key2} using the above formulae, after some simplification we obtain
\begin{multline}
\label{key2-check1}
(q-1) f_{(\mu_1,\mu_2)}(x_1,x_2) + f^{(1,2)}_{(\mu_1,\mu_2)}(x_1,x_2)
=
\\
\left( \frac{x_1-s}{1-s x_1} \right)^{\mu_1}
\left( \frac{1-s^2}{1-s x_1} \right)
\left( \frac{x_2-s}{1-s x_2} \right)^{\mu_2-1}
\left( \frac{1-s^2}{1-s x_2} \right)
\frac{s (1-q)}{1-sx_2}
+
\sum_{\mu_1 < k \leq \mu_2}
W_k(\mu_1,\mu_2).
\end{multline}
Turning to the right hand side of \eqref{key2}, we again see that its partition function representation gives rise to $\mu_2-\mu_1+1$ configurations:
\begin{multline}
\label{key2-check2}
f_{(\mu_2,\mu_1)}(x_1,x_2)
=
\\
\tikz{0.7}{
\foreach\y in {4,5}{
\draw[lgray,line width=1.5pt] (2,\y) -- (8,\y);
}
\foreach\x in {3,...,7}{
\draw[lgray,line width=4pt] (\x,3) -- (\x,6);
}
\node[left] at (1.5,4) {$x_1 \rightarrow$};
\node[left] at (1.5,5) {$x_2 \rightarrow$};
\node[above] at (7,6) {\footnotesize$\mu_2$};
\node[above] at (5.5,6) {$<$};
\node[above] at (4,6) {\footnotesize$\mu_1$};
\node[left] at (2,4) {$1$};
\node[left] at (2,5) {$2$};
\node[right] at (8,4) {$0$};
\node[right] at (8,5) {$0$};
\draw[ultra thick,blue,->] (2,4) -- (4,4) -- (4,5) -- (7,5) -- (7,6);
\draw[ultra thick,green,->] (2,5) -- (4,5) -- (4,6);
}
+
\sum_{\mu_1 < k \leq \mu_2}
\tikz{0.7}{
\foreach\y in {4,5}{
\draw[lgray,line width=1.5pt] (2,\y) -- (8,\y);
}
\foreach\x in {3,...,7}{
\draw[lgray,line width=4pt] (\x,3) -- (\x,6);
}
\node[left] at (1.5,4) {$x_1 \rightarrow$};
\node[left] at (1.5,5) {$x_2 \rightarrow$};
\node[above] at (7,6) {\footnotesize$\ \ \ \mu_2$};
\node[above] at (5,6) {$<$};
\node[above] at (4,6) {\footnotesize$\mu_1$};
\node[above] at (6,6) {\footnotesize$k$};
\node[above] at (6.5,5.93) {\ $\leq$};
\node[left] at (2,4) {$1$};
\node[left] at (2,5) {$2$};
\node[right] at (8,4) {$0$};
\node[right] at (8,5) {$0$};
\draw[ultra thick,blue,->] (2,4) -- (6,4) -- (6,5) -- (7,5) -- (7,6);
\draw[ultra thick,green,->] (2,5) -- (4,5) -- (4,6);
}
\\
=
\left( \frac{x_1-s}{1-s x_1} \right)^{\mu_1}
\left( \frac{1-s^2}{1-s x_1} \right)
\left( \frac{x_2-s}{1-s x_2} \right)^{\mu_2-1}
\left( \frac{1-s^2}{1-s x_2} \right)
\frac{s (1-q)}{1-sx_2}
+
\sum_{\mu_1 < k \leq \mu_2}
\widetilde{W}_k(\mu_1,\mu_2),
\end{multline}
where we abbreviate the weights of the configurations in the sum over $k$, this time by $\widetilde{W}_k(\mu_1,\mu_2)$. Now one can easily see that $W_k(\mu_1,\mu_2)=\widetilde{W}_k(\mu_1,\mu_2)$ for all $k$, since the only difference between such configurations is switching the colours of the paths. Given that the paths in these configurations do not cross each other or touch, the colour of the paths is irrelevant, and all weights can be calculated using rank-1 vertices. We find that \eqref{key2-check1} and \eqref{key2-check2} are equal, completing the check of \eqref{key2}.
\end{ex}

\section{Symmetry in $(x_i,x_{i+1})$ for $\mu_i = \mu_{i+1}$}
\label{ssec:xx-sym}

\begin{thm}
Let $\mu = (\mu_1,\dots,\mu_n)$ be a composition such that $\mu_i = \mu_{i+1}$, for some $1 \leq i \leq n-1$. Then one has the symmetry
\begin{align}
\label{xx-sym}
f_{\mu}(x_1,\dots,x_n)
=
f_{\mu}(x_1,\dots,x_{i+1},x_i,\dots,x_n).
\end{align}
\end{thm}

\begin{proof}
We need to show that when $\mu_i = \mu_{i+1}$, one has the equality of expectation values
\begin{align}
\label{5.8-1}
\bra{\varnothing} \cdots \C_i(x_i) 
\C_{i+1}(x_{i+1}) \cdots \ket{\mu}
=
\bra{\varnothing} \cdots \C_i(x_{i+1}) 
\C_{i+1}(x_i) \cdots \ket{\mu},
\end{align}
where we omit writing the operators $\mathcal{C}_j$ for $j \not= i,i+1$. Using the commutation relation \eqref{CC>} with $j \mapsto i, i \mapsto i+1$ and $x \mapsto x_i, y \mapsto x_{i+1}$, we find that \eqref{5.8-1} becomes
\begin{align*}
\bra{\varnothing} \cdots \C_i(x_i) 
\C_{i+1}(x_{i+1}) \cdots \ket{\mu}
&=
\frac{x_i-x_{i+1}}{x_i-qx_{i+1}}
\bra{\varnothing} \cdots \C_{i+1}(x_i) 
\C_i(x_{i+1}) \cdots \ket{\mu}
\\
&+
\frac{(1-q)x_{i+1}}{x_i-qx_{i+1}}
\bra{\varnothing} \cdots \C_i(x_i) 
\C_{i+1}(x_{i+1}) \cdots \ket{\mu},
\end{align*}
or after further rearrangement,
\begin{align*}
\bra{\varnothing} \cdots \C_i(x_i) 
\C_{i+1}(x_{i+1}) \cdots \ket{\mu}
=
\bra{\varnothing} \cdots \C_{i+1}(x_i) 
\C_i(x_{i+1}) \cdots \ket{\mu}.
\end{align*}
Hence the symmetry \eqref{xx-sym} is equivalent to the statement
\begin{align}
\label{xx-sym2}
f_{\mu}(x_1,\dots,x_n)
=
f_{\mu}^{(i,i+1)}(x_1,\dots,x_n),
\qquad
\text{when}\
\mu_i = \mu_{i+1},
\end{align}
which is what we will seek to prove. As in the proof of Theorem \ref{thm:hecke-f}, we will work at the pre-fused level. According to Definition \ref{def:prefus} of the function $\mathcal{F}_{\mu}(x_1,\dots,x_n;Y)$, when $\mu_i = \mu_{i+1}$, both of the paths $i$ and $i+1$ exit the lattice via the same bundle $\mathcal{A}^{(\mu_i)}$. By the ordering assumption on the paths within the bundle $\mathcal{A}^{(\mu_i)}$, path $i+1$ is positioned one step to the right of $i$ at the top of the lattice. We conclude that the paths $i$ and $i+1$ have to cross at least once within the partition function 
$\mathcal{F}_{\mu}(x_1,\dots,x_n;Y)$, and accordingly we can write it in refined form
\begin{align}
\label{refine3}
\mathcal{F}_{\mu} (x_1,\dots,x_n;Y)
=
\sum_{k=1}^{n}
\sum_{\substack{j_1 < \cdots < j_k \\ \ell_1 < \cdots < \ell_k}}
\mathcal{F}_{\mu,i}\Big(x_1,\dots,x_n;Y;(j_1,\ell_1),\dots,(j_k,\ell_k)\Big),
\qquad
\mu_i = \mu_{i+1},
\end{align}
where the sum is taken over $k \geq 1$.

\begin{prop}
Fix two integers $1 \leq i \leq n-1$ and $k \geq 1$, and a set of lattice coordinates $\{(j_1,\ell_1),\dots,(j_k,\ell_k)\}$ as in Definition \ref{def:F-refine}. One has the relation
\begin{align}
\label{quartet}
\frac{
\mathcal{F}_{\mu,i}\Big(x_1,\dots,x_n;Y;(j_1,\ell_1),\dots,(j_k,\ell_k)\Big)
}
{p_{k-1}(i,i+1;i,i+1)}
=
\frac{
\mathcal{F}_{\mathfrak{s}_i \cdot \mu,i}^{(i,i+1)}
\Big(x_1,\dots,x_n;Y;(j_1,\ell_1),\dots,(j_k,\ell_k)\Big)
}
{p_{k-1}(i+1,i;i+1,i)}.
\end{align}
\end{prop}

\begin{proof}
This is proved by the same reasoning as in the proof of Proposition \ref{prop:trio}; indeed, the right hand side of \eqref{quartet} can be considered as the fourth natural object in the string of equalities \eqref{trio}.
\end{proof}

Using the result \eqref{quartet} in \eqref{refine3}, and simplifying by employing the identity \eqref{pk-2}, we find that
\begin{multline*}
\mathcal{F}_{\mu} (x_1,\dots,x_n;Y)
=
\sum_{k=1}^{n}
\sum_{\substack{j_1 < \cdots < j_k \\ \ell_1 < \cdots < \ell_k}}
\mathcal{F}_{\mathfrak{s}_i \cdot \mu,i}^{(i,i+1)}
\Big(x_1,\dots,x_n;Y;(j_1,\ell_1),\dots,(j_k,\ell_k)\Big)
\frac{p_{k-1}(i,i+1;i,i+1)}{p_{k-1}(i+1,i;i+1,i)}
\\
=
q^{-1}
\sum_{k=1}^{n}
\sum_{\substack{j_1 < \cdots < j_k \\ \ell_1 < \cdots < \ell_k}}
\mathcal{F}_{\mathfrak{s}_i \cdot \mu,i}^{(i,i+1)}
\Big(x_1,\dots,x_n;Y;(j_1,\ell_1),\dots,(j_k,\ell_k)\Big)
=
q^{-1}
\mathcal{F}_{\mathfrak{s}_i \cdot \mu}^{(i,i+1)}(x_1,\dots,x_n;Y).
\end{multline*}
We have established that
\begin{align*}
\mathcal{F}_{\mu} (x_1,\dots,x_n;Y)
=
q^{-1}
\mathcal{F}_{\mathfrak{s}_i \cdot \mu}^{(i,i+1)}(x_1,\dots,x_n;Y),
\end{align*}
where the partition function on the right hand side is the same as that on the left hand side, up to the switching of colours $i$ and $i+1$. This is the pre-fused version of \eqref{xx-sym2}. Applying the fusion procedure of Theorem \ref{thm:fusion-F} to $\mathcal{F}_{\mu} (x_1,\dots,x_n;Y)$ results in $f_{\mu}(x_1,\dots,x_n)$, while applying the same procedure to $\mathcal{F}_{\mathfrak{s}_i \cdot \mu}^{(i,i+1)}(x_1,\dots,x_n;Y)$ gives $q \cdot f^{(i,i+1)}_{\mu}(x_1,\dots,x_n)$. The extra factor of $q$ arises from the fact that the colours $i$ and $i+1$ are no longer ordered sequentially at the top of the lattice in the partition function $\mathcal{F}_{\mathfrak{s}_i \cdot \mu}^{(i,i+1)}(x_1,\dots,x_n;Y)$; the sequential ordering can be restored at the expense of a factor of $q$, as seen in the $q$-exchangeability result of Proposition \ref{prop:q-exch}.

\end{proof}

\section{Relationship between $f_{\mu}$ and $f^{\sigma}_{\delta}$}
\label{ssec:ff-sym}

\begin{thm}
\label{thm:fmu-fsigma}
Let $\delta = (\delta_1 \leq \cdots \leq \delta_n)$ be an anti-dominant composition, fix a composition $\mu = (\mu_1,\dots,\mu_n)$ with $\mu^{+} = \delta^{+}$, and let 
$\sigma \in \mathfrak{S}_n$ be the minimal-length permutation such that
\begin{align}
\label{mu-delta}
\mu_i = \delta_{\sigma(i)},
\quad
\forall\ 1 \leq i \leq n.
\end{align}
We have the following relationship between the functions \eqref{f-HL} and their permuted analogues \eqref{sigma-f}:
\begin{align}
\label{fmu-fsigma}
f_{\mu}(x_1,\dots,x_n;q,s)
=
s^n q^{\ell(\sigma)} \prod_{j \geq 0} q^{m_j(\delta)(m_j(\delta)-1)/2} 
\prod_{i=1}^{n} x_i^{-1} \cdot
f^{\sigma}_{\delta}(x_1^{-1},\dots,x_n^{-1};q^{-1},s^{-1}),
\end{align}
where $\ell(\sigma)$ \index{l99@$\ell(\sigma)$; length of a permutation} denotes the length of $\sigma$,  $\ell(\sigma)= \#\{ i<j: \sigma(i) > \sigma(j) \}$, and $m_j(\delta) = \#\{ i : \delta_i = j \}$ is the multiplicity of the part $j$ in $\delta$.
\end{thm}

\begin{proof}
Choose any reduced-word decomposition of the permutation $\sigma$:
\begin{align}
\label{reduced}
\sigma =  \mathfrak{s}_{i_{\ell(\sigma)}} \cdots \mathfrak{s}_{i_1},
\end{align}
where $i_1,\dots,i_{\ell(\sigma)} \in \{1,\dots,n-1\}$ and $\mathfrak{s}_i = (i,i+1)$ is a simple transposition. From this, define the following product of the Hecke generators \eqref{hecke-poly}:
\begin{align}
\label{Tsigma}
T_{\sigma} := 
T_{i_{\ell(\sigma)}}
\cdots 
T_{i_1}.
\end{align}
Note that the definition \eqref{Tsigma} is meaningful without specifying which reduced-word decomposition \eqref{reduced} one uses for the permutation $\sigma$. Indeed, the relation $T_i T_{i+1} T_i = T_{i+1} T_i T_{i+1}$ in the Hecke algebra \eqref{hecke1} ensures that the product \eqref{Tsigma} depends only on the permutation $\sigma$, and not the decomposition into transpositions that one chooses.\footnote{Any two reduced-word decompositions \eqref{reduced} of $\sigma$ are related by the Coxeter relations $\mathfrak{s}_i \mathfrak{s}_{i+1} \mathfrak{s}_i = \mathfrak{s}_{i+1} \mathfrak{s}_i \mathfrak{s}_{i+1}$ and $\mathfrak{s}_i \mathfrak{s}_j = \mathfrak{s}_j \mathfrak{s}_i$, $|i-j| >1$, without the need of the relation $\mathfrak{s}_i^2 = {\rm id}$; this is a non-trivial result known as Tits' Theorem \cite[Section 5.13]{Humphreys}.} 

\begin{rmk}
In accordance with our conventions, both \eqref{reduced} and \eqref{Tsigma} act towards the right. Hence, the first transposition in $\sigma$ corresponds with the rightmost operator in these products. For example, for $\sigma = (2,3,1) = (2,3)(1,2)$ one has $T_{\sigma} = T_2 T_1$.
\end{rmk}

Now by repeated application of equation \eqref{T-f}, we are able to write
\begin{align}
\label{T-fdelta}
f_{\mu}(x_1,\dots,x_n;q,s)
=
T_{\sigma} \cdot f_{\delta}(x_1,\dots,x_n;q,s)
\end{align}
with $\sigma$ the minimal-length permutation such that \eqref{mu-delta} holds.
We now transform the right hand side of this equation. First, notice the following symmetry of the Hecke generators \eqref{hecke-poly} and \eqref{hecke-inv}:
\begin{align}
\label{hecke-sym}
T_i(x,q)
= 
q - \frac{x_i - q x_{i+1}}{x_i - x_{i+1}} (1-\mathfrak{s}_i)
=
q \left(1 - \frac{x_i^{-1} - q^{-1} x_{i+1}^{-1}}{x_i^{-1} - x_{i+1}^{-1}} (1-\mathfrak{s}_i) \right)
=
q\cdot \widehat{T}_i(x^{-1},q^{-1})
\end{align}
where we write the starting generator as $T_i(x,q)$ and the final generator as 
$\widehat{T}_i(x^{-1},q^{-1})$, to emphasize that the dependence on the parameters 
$(x_i,x_{i+1};q)$ is reciprocated in the final operator. It follows that we can write
\begin{align}
\label{4.2-1}
f_{\mu}(x_1,\dots,x_n;q,s)
=
q^{\ell(\sigma)} \cdot
\widehat{T}_{\sigma}(x^{-1},q^{-1}) \cdot f_{\delta}(x_1,\dots,x_n;q,s),
\end{align}
where
\begin{align*}
\widehat{T}_{\sigma}(x^{-1},q^{-1}) 
:= 
\widehat{T}_{i_{\ell(\sigma)}}(x^{-1},q^{-1})
\cdots
\widehat{T}_{i_1}(x^{-1},q^{-1}),
\end{align*} 
with $i_1,\dots,i_{\ell(\sigma)} \in \{1,\dots,n-1\}$ labelling the reduced-word decomposition \eqref{reduced}. Now using the explicit factorization \eqref{f-delta} of $f_{\delta}$, we deduce the symmetry
\begin{align}
\label{delta-sym}
f_{\delta}(x_1,\dots,x_n;q,s)
=
s^n \prod_{j \geq 0} q^{m_j(\delta)(m_j(\delta)-1)/2}
\prod_{i=1}^{n} x_i^{-1}
\cdot
f_{\delta}(x_1^{-1},\dots,x_n^{-1};q^{-1},s^{-1}),
\end{align}
which is valid for all anti-dominant compositions $\delta = (\delta_1 \leq \cdots \leq \delta_n)$. Substituting \eqref{delta-sym} into \eqref{4.2-1}, the operator $\widehat{T}_{\sigma}(x^{-1},q^{-1})$ now acts on a function depending on the reciprocated variables $(x_1^{-1},\dots,x_n^{-1};q^{-1})$. Furthermore, the product $\prod_{i=1}^{n} x_i^{-1}$ is symmetric and commutes with the action of $\widehat{T}_{\sigma}(x^{-1},q^{-1})$. The result \eqref{fmu-fsigma} follows immediately, by repeated application of \eqref{That-f}.
\end{proof}

\begin{rmk}
Equation \eqref{fmu-fsigma} can also be stated as
\begin{align}
\label{fmu-fsigma-2}
f_{\mu}(x_1,\dots,x_n;q,s)
=
s^n q^{{\rm inv}(\mu)} 
\prod_{i=1}^{n} x_i^{-1} \cdot
f^{\sigma}_{\delta}(x_1^{-1},\dots,x_n^{-1};q^{-1},s^{-1})
\end{align}
by grouping powers of $q$, where we have introduced the composition inversion number ${\rm inv}(\mu) := \#\{i < j : \mu_i \geq \mu_j\}$.
\end{rmk}

\section{Relationship between $f_{\mu}$ and $g_{\mu}$}
\label{ssec:fg-sym}

\begin{prop} 
Let $\mu = (\mu_1,\dots,\mu_n)$ be a composition, and let \index{am@$\tilde{\mu}$; reverse of a composition} $\tilde{\mu} = (\mu_n,\dots,\mu_1)$ be the same composition read in reverse. We have the following relationship between the functions $f_{\mu}$ and $g_{\tilde{\mu}}$:
\begin{align}
\label{f-g-sym}
g_{\tilde\mu}(x_n^{-1},\dots,x_1^{-1};q^{-1},s^{-1})
=
\index{c@$c_{\mu}(q,s)$}
c_{\mu}(q,s)
\prod_{i=1}^{n}
x_i
\cdot
f_{\mu}(x_1,\dots,x_n;q,s),
\end{align}
where the multiplicative constant $c_{\mu}(q,s)$ is given by
\begin{align}
\label{cmu}
c_{\mu}(q,s)
=
\frac{s^n (q-1)^n q^{{\rm inv}(\tilde\mu)}}{\prod_{j \geq 0} (s^2;q)_{m_j(\mu)}},
\qquad
{\rm inv}(\tilde\mu)
=
\#\{i<j : \tilde{\mu}_i \geq \tilde{\mu}_j\}
=
\#\{i<j : \mu_i \leq \mu_j\}.
\end{align}
\end{prop}

\begin{proof}
Consider the partition function expression for $g_{\tilde\mu}(x_n^{-1},\dots,x_1^{-1};q^{-1},s^{-1})$, as given by \eqref{g-def}. A path of colour $i$ leaves via the left edge of the $i$-th row (counted from the bottom), and enters via the bottom edge of column $\tilde{\mu}_i = \mu_{n-i+1}$, for all $1 \leq i \leq n$. Equivalently, one could say that a path of colour $\tilde\imath = n-i+1$ leaves via the left edge of the $i$-th row (counted from the top), and enters via the bottom edge of column $\tilde{\mu}_{n-i+1} = \mu_i$. Recalling the definitions \eqref{conj1} and \eqref{conj2} of conjugation, we conclude that
\begin{align}
\label{tilde-g}
g_{\tilde\mu}(x_n^{-1},\dots,x_1^{-1};q^{-1},s^{-1})
&=
\tikz{0.8}{
\foreach\y in {1,...,5}{
\draw[lgray,line width=1.5pt,<-] (1,\y) -- (8,\y);
}
\foreach\x in {2,...,7}{
\draw[lgray,line width=4pt,->] (\x,0) -- (\x,6);
}
\node[left] at (0.5,1) {$x^{-1}_n \leftarrow$};
\node[left] at (0.5,2) {$\vdots$};
\node[left] at (0.5,3) {$\vdots$};
\node[left] at (0.5,4) {$x^{-1}_2 \leftarrow$};
\node[left] at (0.5,5) {$x^{-1}_1 \leftarrow$};
\node[above] at (7,6) {$\cdots$};
\node[above] at (6,6) {$\cdots$};
\node[above] at (5,6) {$\cdots$};
\node[above] at (4,6) {\footnotesize$\tilde{\bm{0}}$};
\node[above] at (3,6) {\footnotesize$\tilde{\bm{0}}$};
\node[above] at (2,6) {\footnotesize$\tilde{\bm{0}}$};
\node[below] at (7,0) {$\cdots$};
\node[below] at (6,0) {$\cdots$};
\node[below] at (5,0) {$\cdots$};
\node[below] at (4,0) {\footnotesize$\tilde{\bm{A}}(2)$};
\node[below] at (3,0) {\footnotesize$\tilde{\bm{A}}(1)$};
\node[below] at (2,0) {\footnotesize$\tilde{\bm{A}}(0)$};
\node[right] at (8,1) {$\tilde{0}$};
\node[right] at (8,2) {$\vdots$};
\node[right] at (8,3) {$\vdots$};
\node[right] at (8,4) {$\tilde{0}$};
\node[right] at (8,5) {$\tilde{0}$};
\node[left] at (1,1) {$\tilde{n}$};
\node[left] at (1,2) {$\vdots$};
\node[left] at (1,3) {$\vdots$};
\node[left] at (1,4) {$\tilde{2}$};
\node[left] at (1,5) {$\tilde{1}$};
}
\end{align}
where each vertex within the lattice is given by \eqref{dual-s-weights} with parameters $(x,q,s) \mapsto (x^{-1},q^{-1},s^{-1})$, and the vector states at the bottom are given by
$\tilde{\bm{A}}(k)
=
\sum_{j=1}^{n}
\left( {\bm 1}_{\mu_j = k} \right)
\bm{e}_{n-j+1}
$.

Next, we will make use of the symmetry \eqref{sym2}. We start by translating \eqref{sym2} into graphical form, absorbing each of the factors present on its right hand side into the definition of a new type of vertex:
\begin{align}
\label{sym3}
\tikz{0.7}{
\draw[lgray,line width=1.5pt,<-] (-1,0) -- (1,0);
\draw[lgray,line width=4pt,->] (0,-1) -- (0,1);
\node[left] at (-1,0) {\tiny $\tilde\ell$};\node[right] at (1,0) {\tiny $\tilde\jmath$};
\node[below] at (0,-1) {\tiny $\tilde\I$};\node[above] at (0,1) {\tiny $\tilde\K$};
\node[below,text centered] at (0,-1.7) {$(x^{-1},q^{-1},s^{-1})$};
}
\quad
=
\quad
\tikz{0.7}{
\draw[lred,line width=1.5pt,->] (-1,0) -- (1,0);
\draw[lred,line width=4pt,->] (0,-1) -- (0,1);
\node[left] at (-1,0) {\tiny $\ell$};\node[right] at (1,0) {\tiny $j$};
\node[below] at (0,-1) {\tiny $\K$};\node[above] at (0,1) {\tiny $\I$};
\node[below,text centered] at (0,-1.7) {$(x,q,s)$};
}
\end{align}
where the vertex on the left hand side is given by \eqref{dual-s-weights} with parameters $(x,q,s) \mapsto (x^{-1},q^{-1},s^{-1})$, while the non-vanishing vertex weights on the right hand side are given explicitly in the table below (they are not the same as in \eqref{s-weights}, even though we use the same symbols to denote the vertices):
\begin{align}
\label{green-s-weights}
\begin{tabular}{|c|c|c|}
\hline
\quad
\tikz{0.7}{
\draw[lred,line width=1.5pt,->] (-1,0) -- (1,0);
\draw[lred,line width=4pt,->] (0,-1) -- (0,1);
\node[left] at (-1,0) {\tiny $0$};\node[right] at (1,0) {\tiny $0$};
\node[below] at (0,-1) {\tiny $\I$};\node[above] at (0,1) {\tiny $\I$};
}
\quad
&
\quad
\tikz{0.7}{
\draw[lred,line width=1.5pt,->] (-1,0) -- (1,0);
\draw[lred,line width=4pt,->] (0,-1) -- (0,1);
\node[left] at (-1,0) {\tiny $i$};\node[right] at (1,0) {\tiny $i$};
\node[below] at (0,-1) {\tiny $\I$};\node[above] at (0,1) {\tiny $\I$};
}
\quad
&
\quad
\tikz{0.7}{
\draw[lred,line width=1.5pt,->] (-1,0) -- (1,0);
\draw[lred,line width=4pt,->] (0,-1) -- (0,1);
\node[left] at (-1,0) {\tiny $0$};\node[right] at (1,0) {\tiny $i$};
\node[below] at (0,-1) {\tiny $\I^{+}_i$};\node[above] at (0,1) {\tiny $\I$};
}
\quad
\\[1.3cm]
\quad
$\dfrac{1-s x q^{\Is{1}{n}}}{1-sx}$
\quad
& 
\quad
$\dfrac{(x-sq^{I_i}) q^{\Is{1}{i-1}}}{1-sx}$
\quad
& 
\quad
$\dfrac{-s^{-1}(1- s^{2} q^{\Is{1}{n}})}{1-sx}$
\quad
\\[0.7cm]
\hline
\quad
\tikz{0.7}{
\draw[lred,line width=1.5pt,->] (-1,0) -- (1,0);
\draw[lred,line width=4pt,->] (0,-1) -- (0,1);
\node[left] at (-1,0) {\tiny $i$};\node[right] at (1,0) {\tiny $0$};
\node[below] at (0,-1) {\tiny $\I^{-}_i$};\node[above] at (0,1) {\tiny $\I$};
}
\quad
&
\quad
\tikz{0.7}{
\draw[lred,line width=1.5pt,->] (-1,0) -- (1,0);
\draw[lred,line width=4pt,->] (0,-1) -- (0,1);
\node[left] at (-1,0) {\tiny $i$};\node[right] at (1,0) {\tiny $j$};
\node[below] at (0,-1) {\tiny $\I^{+-}_{ji}$};\node[above] at (0,1) 
{\tiny $\I$};
}
\quad
&
\quad
\tikz{0.7}{
\draw[lred,line width=1.5pt,->] (-1,0) -- (1,0);
\draw[lred,line width=4pt,->] (0,-1) -- (0,1);
\node[left] at (-1,0) {\tiny $j$};\node[right] at (1,0) {\tiny $i$};
\node[below] at (0,-1) {\tiny $\I^{+-}_{ij}$};\node[above] at (0,1) {\tiny $\I$};
}
\quad
\\[1.3cm] 
\quad
$\dfrac{-sx(1-q^{I_i})q^{\Is{1}{i-1}}}{1-sx}$
\quad
& 
\quad
$\dfrac{x(1-q^{I_i}) q^{\Is{1}{i-1}}}{1-sx}$
\quad
&
\quad
$\dfrac{s(1-q^{I_j})q^{\Is{1}{j-1}}}{1-sx}$
\quad
\\[0.7cm]
\hline
\end{tabular} 
\end{align}
with $1 \leq i<j \leq n$. Applying the relation \eqref{sym3} to the partition function \eqref{tilde-g}, we immediately find that
\begin{align}
\label{tilde-g-2}
g_{\tilde\mu}(x_n^{-1},\dots,x_1^{-1};q^{-1},s^{-1})
&=
\tikz{0.73}{
\foreach\y in {1,...,5}{
\draw[lred,line width=1.5pt,->] (1,\y) -- (8,\y);
}
\foreach\x in {2,...,7}{
\draw[lred,line width=4pt,->] (\x,0) -- (\x,6);
}
\node[left] at (0.5,1) {$x_1 \rightarrow$};
\node[left] at (0.5,2) {$x_2 \rightarrow$};
\node[left] at (0.5,3) {$\vdots$};
\node[left] at (0.5,4) {$\vdots$};
\node[left] at (0.5,5) {$x_n \rightarrow$};
\node[below] at (7,0) {$\cdots$};
\node[below] at (6,0) {$\cdots$};
\node[below] at (5,0) {$\cdots$};
\node[below] at (4,0) {\footnotesize$\bm{0}$};
\node[below] at (3,0) {\footnotesize$\bm{0}$};
\node[below] at (2,0) {\footnotesize$\bm{0}$};
\node[above] at (7,6) {$\cdots$};
\node[above] at (6,6) {$\cdots$};
\node[above] at (5,6) {$\cdots$};
\node[above] at (4,6) {\footnotesize$\bm{A}(2)$};
\node[above] at (3,6) {\footnotesize$\bm{A}(1)$};
\node[above] at (2,6) {\footnotesize$\bm{A}(0)$};
\node[right] at (8,1) {$0$};
\node[right] at (8,2) {$0$};
\node[right] at (8,3) {$\vdots$};
\node[right] at (8,4) {$\vdots$};
\node[right] at (8,5) {$0$};
\node[left] at (1,1) {$1$};
\node[left] at (1,2) {$2$};
\node[left] at (1,3) {$\vdots$};
\node[left] at (1,4) {$\vdots$};
\node[left] at (1,5) {$n$};
}
\end{align}
where the vertices within the lattice take the form \eqref{green-s-weights}, and the vector states at the top are given by 
$\bm{A}(k)
=
\sum_{j=1}^{n}
\left( {\bm 1}_{\mu_j = k} \right)
\bm{e}_{j}
$. The partition function \eqref{tilde-g-2} has exactly the same form as that used to define $f_{\mu}(x_1,\dots,x_n;q,s)$; the only difference is that it is comprised of the vertices \eqref{green-s-weights}, rather than those tabulated in \eqref{s-weights}. 

To complete the proof, one needs to show that each configuration in the partition function \eqref{tilde-g-2} receives the same Boltzmann weight as the corresponding configuration in \eqref{f-def}, modulo the overall multiplicative factor $c_{\mu}(q,s) \prod_{i=1}^{n} x_i$. This can be achieved by applying gauge transformations to the lattice \eqref{f-def}, which effectively multiply it by the factor $c_{\mu}(q,s) \prod_{i=1}^{n} x_i$, and then factorizing those gauge transformations locally over the individual vertices to produce the weights \eqref{green-s-weights}. The key relation underpinning this procedure is the following:
\begin{align*}
\left(
s x (q-1) q^{{\rm inv}(\tilde\mu)-{\rm inv}(\tilde\nu)}
\prod_{j \geq 0}
\frac{(s^2;q)_{m_j(\nu)}}{(s^2;q)_{m_j(\mu)}}
\right)
\cdot
&
\tikz{0.8}{
\draw[lgray,line width=1.5pt,->] (1,1) -- (8,1);
\foreach\x in {2,...,7}{
\draw[lgray,line width=4pt,->] (\x,0) -- (\x,2);
}
\node[above] at (7,2) {$\cdots$};
\node[above] at (6,2) {$\cdots$};
\node[above] at (5,2) {$\cdots$};
\node[above] at (4,2) {\footnotesize$\bm{A}(2)$};
\node[above] at (3,2) {\footnotesize$\bm{A}(1)$};
\node[above] at (2,2) {\footnotesize$\bm{A}(0)$};
\node[below] at (7,0) {$\cdots$};
\node[below] at (6,0) {$\cdots$};
\node[below] at (5,0) {$\cdots$};
\node[below] at (4,0) {\footnotesize$\bm{B}(2)$};
\node[below] at (3,0) {\footnotesize$\bm{B}(1)$};
\node[below] at (2,0) {\footnotesize$\bm{B}(0)$};
\node[right] at (8,1) {$0$};
\node[left] at (1,1) {$p$};
}
\\
=
&
\tikz{0.8}{
\draw[lred,line width=1.5pt,->] (1,1) -- (8,1);
\foreach\x in {2,...,7}{
\draw[lred,line width=4pt,->] (\x,0) -- (\x,2);
}
\node[above] at (7,2) {$\cdots$};
\node[above] at (6,2) {$\cdots$};
\node[above] at (5,2) {$\cdots$};
\node[above] at (4,2) {\footnotesize$\bm{A}(2)$};
\node[above] at (3,2) {\footnotesize$\bm{A}(1)$};
\node[above] at (2,2) {\footnotesize$\bm{A}(0)$};
\node[below] at (7,0) {$\cdots$};
\node[below] at (6,0) {$\cdots$};
\node[below] at (5,0) {$\cdots$};
\node[below] at (4,0) {\footnotesize$\bm{B}(2)$};
\node[below] at (3,0) {\footnotesize$\bm{B}(1)$};
\node[below] at (2,0) {\footnotesize$\bm{B}(0)$};
\node[right] at (8,1) {$0$};
\node[left] at (1,1) {$p$};
}
\end{align*}
valid for any integer $p \geq 1$ and pair of compositions $\nu = (\nu_1,\dots,\nu_{p-1})$, $\mu = (\mu_1,\dots,\mu_p)$, where we have defined the states
$\bm{B}(k) =\sum_{j=1}^{p-1} \left( {\bm 1}_{\nu_j = k} \right) \bm{e}_{j}$ and
$\bm{A}(k) =\sum_{j=1}^{p} \left( {\bm 1}_{\mu_j = k} \right) \bm{e}_{j}$. This identity is readily checked using the form of the weights \eqref{s-weights} and \eqref{green-s-weights}; applying it over the $n$ rows of the partition function \eqref{f-def} yields the result \eqref{f-g-sym}.

\end{proof}

\section{Relationship between $f^{\sigma}_{\delta}$ and $g_{\mu}$}

\begin{prop}
Let $\mu = (\mu_1,\dots,\mu_n)$ be a composition and $\tilde{\mu} = (\mu_n,\dots,\mu_1)$ its reverse ordering. Denote the anti-dominant ordering of $\mu$ by $\delta = (\delta_1 \leq \cdots \leq \delta_n)$, and let $\sigma \in \mathfrak{S}_n$ be the minimal-length permutation such that $\mu_i = \delta_{\sigma(i)}$ for all $1 \leq i \leq n$. We have the following relationship between the functions $f^{\sigma}_{\delta}$ and $g_{\tilde\mu}$:
\begin{align}
\label{g-fsigma}
g_{\tilde\mu}(x_n,\dots,x_1) 
=
\frac{(q-1)^n q^{-n(n+1)/2}}{\prod_{j \geq 0} (s^2;q)_{m_j(\mu)}}
\cdot
f^{\sigma}_{\delta}(x_1,\dots,x_n). 
\end{align}
\end{prop}

\begin{proof}
From \eqref{f-g-sym} we know that
\begin{align}
f_{\mu}(x_1,\dots,x_n;q,s)
=
\frac{q^{-{\rm inv}(\tilde\mu)} \prod_{j \geq 0} (s^2;q)_{m_j(\mu)}}
{s^n (q-1)^n}
\prod_{i=1}^{n} x_i^{-1} 
\cdot
g_{\tilde\mu}(x_n^{-1},\dots,x_1^{-1};q^{-1},s^{-1}),
\end{align}
and comparing this against \eqref{fmu-fsigma-2}, after some rearrangement we obtain
\begin{multline*}
g_{\tilde\mu}(x_n^{-1},\dots,x_1^{-1};q^{-1},s^{-1})
\\
=
\frac{s^{2n} (q-1)^n}{\prod_{j \geq 0} (s^2;q)_{m_j(\mu)}}
q^{n(n-1)/2} \prod_{j \geq 0} q^{m_j(\mu)(m_j(\mu)-1)/2}
\cdot
f^{\sigma}_{\delta}(x_1^{-1},\dots,x_n^{-1};q^{-1},s^{-1}).
\end{multline*}
The result \eqref{g-fsigma} follows by reciprocating everywhere $x_i$, $q$ and $s$ and massaging the factors.

\end{proof}

\section{Exchange relations for $g_{\mu}$}
\label{ssec:g-ex}

\begin{thm}
Fix a composition $\mu = (\mu_1,\dots,\mu_n)$ and consider the corresponding dual non-symmetric spin Hall--Littlewood function $g_{\mu}(x_n,\dots,x_1)$ (with a reversed alphabet). This function transforms under the action of \eqref{hecke-poly} and \eqref{hecke-inv} according to the rules
\begin{align}
\label{T-g}
T_{n-i} \cdot g_{\mu}(x_n,\dots,x_1) &= 
g^{(i,i+1)}_{\mu}(x_n,\dots,x_1),
\\
\label{That-g}
\widehat{T}_{n-i} \cdot g_{\mu}(x_n,\dots,x_1) &= 
g_{\mathfrak{s}_i \cdot \mu}(x_n,\dots,x_1),
\quad
\mu_i > \mu_{i+1},
\end{align}
where $g^{(i,i+1)}_{\mu}$ denotes a permuted dual non-symmetric spin Hall--Littlewood function \eqref{sigma-g} in the case of the permutation 
$\sigma = (i,i+1)$.
\end{thm}

\begin{proof}
We begin with the proof of \eqref{That-g}, which is a simple consequence of the relation \eqref{f-g-sym} between the non-symmetric spin Hall--Littlewood functions and their duals, together with \eqref{T-f}. Using \eqref{f-g-sym} (with all variables reciprocated), we see that
\begin{align*}
\widehat{T}_{n-i} \cdot g_{\mu}(x_n,\dots,x_1;q,s)
=
c_{\tilde\mu}(q^{-1},s^{-1})
\prod_{i=1}^{n}
x_i^{-1}
\cdot
\widehat{T}_{n-i}
\cdot
f_{\tilde\mu}(x_1^{-1},\dots,x_n^{-1};q^{-1},s^{-1}),
\end{align*}
where $\tilde\mu = (\mu_n,\dots,\mu_1)$, and where we have used the fact that the symmetric product $\prod_{i=1}^{n} x_i^{-1}$ commutes with the action of the operator $\widehat{T}_{n-i}$. Now we use \eqref{hecke-sym} to replace $\widehat{T}_{n-i} \equiv \widehat{T}_{n-i}(x,q)$ by $q T_{n-i}(x^{-1},q^{-1})$, leading to
\begin{align*}
\widehat{T}_{n-i} \cdot g_{\mu}(x_n,\dots,x_1;q,s)
&=
q
c_{\tilde\mu}(q^{-1},s^{-1})
\prod_{i=1}^{n}
x_i^{-1}
\cdot
T_{n-i}(x^{-1},q^{-1})
\cdot
f_{\tilde\mu}(x_1^{-1},\dots,x_n^{-1};q^{-1},s^{-1})
\\
&=
q
c_{\tilde\mu}(q^{-1},s^{-1})
\prod_{i=1}^{n}
x_i^{-1}
\cdot
f_{\mathfrak{s}_{n-i} \cdot \tilde\mu}(x_1^{-1},\dots,x_n^{-1};q^{-1},s^{-1}),
\end{align*}
where we have employed \eqref{T-f} to deduce the final equality, using the fact that 
$\tilde{\mu}_{n-i} < \tilde{\mu}_{n-i+1}$, since $\tilde{\mu}_{n-i} = \mu_{i+1}$ and $\tilde{\mu}_{n-i+1} = \mu_i$. Finally, we notice that $q c_{\tilde\mu}(q^{-1},s^{-1}) = c_{\mathfrak{s}_{n-i} \cdot \tilde\mu}(q^{-1},s^{-1})$, in view of the inversion-type factor in \eqref{cmu}. We conclude that
\begin{align*}
\widehat{T}_{n-i} \cdot g_{\mu}(x_n,\dots,x_1;q,s)
&=
c_{\mathfrak{s}_{n-i} \cdot \tilde\mu}(q^{-1},s^{-1})
\prod_{i=1}^{n}
x_i^{-1}
\cdot
f_{\mathfrak{s}_{n-i} \cdot \tilde\mu}(x_1^{-1},\dots,x_n^{-1};q^{-1},s^{-1})
\\
&=
g_{\mathfrak{s}_i \cdot \mu}(x_n,\dots,x_1;q,s),
\end{align*}
using \eqref{f-g-sym} to restore the dual non-symmetric Hall--Littlewood function at the final step.

The proof of \eqref{T-g} goes along very similar lines to that of \eqref{That-f}, this time making use of the algebraic definitions \eqref{g-HL} and \eqref{sigma-g} of $g_{\mu}$ and $g_{\mu}^{(i,i+1)}$, as well as the commutation relation \eqref{BB>}.

\end{proof}

\section{Reduction to non-symmetric Hall--Littlewood polynomials}
\label{ssec:hl-reduce}

To conclude the chapter, we show that the functions $f_{\mu}$ and $g_{\mu}$ are indeed $s$-deformations of non-symmetric Hall--Littlewood polynomials; they reduce to the latter at $s=0$. This is done in two steps: the first is to demonstrate that non-symmetric Hall--Littlewood polynomials satisfy the same recursion relations as those given in \eqref{T-f} and \eqref{That-g}; the second is to establish appropriate initial conditions. 

The non-symmetric Hall--Littlewood polynomials are limiting cases of non-symmetric Macdonald polynomials, whose general theory we briefly recall \cite{CherednikNonsym,Macdonald-Bourbaki,KirillovN,MimachiN,Marshall}.\footnote{Note that we use $(p,q)$ as parameters in our Macdonald polynomials, in place of the traditional $(q,t)$; this is because $q$ is already in use throughout the paper, and it is Macdonald's $t$ parameter which matches our $q$.} Let us begin by defining slightly different versions of Hecke generators, obtained by switching $x_i \leftrightarrow x_{i+1}$ in \eqref{hecke-poly}:
\begin{align}
\label{rev-Hecke}
\index{T@$\tilde{T}_i$; reversed-alphabet Hecke generator}
\tilde{T}_i := q - \frac{q x_i-x_{i+1}}{x_i-x_{i+1}} (1-\mathfrak{s}_i),
\quad
\tilde{T}_i^{-1} = q^{-1}\left(1 - \frac{q x_i-x_{i+1}}{x_i-x_{i+1}} (1-\mathfrak{s}_i) \right),
\quad
1 \leq i \leq n-1.
\end{align}
Extend the Hecke algebra generated by $\{\tilde{T}_1,\dots,\tilde{T}_{n-1}\}$ by a generator $\omega$, defined as follows:
\begin{align}
\label{eq:omega}
\index{az@$\omega$}
\omega := \mathfrak{s}_{n-1} \dots \mathfrak{s}_1 \tau_1,
\end{align}
where $\tau_1$ \index{t@$\tau_i$; $p$-shift operator} denotes a $p$-shift operator with action $\tau_1 \cdot h(x_1,\dots,x_n) = h(p x_1, x_2,\dots,x_n)$ on arbitrary functions $h$. The resulting extended Hecke algebra has an Abelian subalgebra generated by the Cherednik--Dunkl operators $Y_i$, which are given by
\begin{align}
\label{eq:Yi}
\index{Y@$Y_i$; Cherednik--Dunkl operators}
Y_i := \tilde{T}_i\cdots \tilde{T}_{n-1} \omega \tilde{T}_{1}^{-1}
\cdots 
\tilde{T}_{i-1}^{-1},
\quad\quad
[Y_i,Y_j] = 0.
\end{align}
In view of their commutativity, these operators can be jointly diagonalized. For generic values of $p$ and $q$, the {\it non-symmetric Macdonald polynomials} $E_{\mu}(x_1,\dots,x_n;p,q)$ \index{E@$E_{\mu}(x_1,\dots,x_n;p,q)$; non-symmetric Macdonald} are the unique family of polynomials which satisfy the properties
\begin{align}
\label{eq:monic}
E_{\mu}(x_1,\dots,x_n;p,q) &= 
x^{\mu}
+ 
\sum_{\nu < \mu} c_{\mu,\nu}(p,q)
x^{\nu},
\quad
x^{\mu}
:=
\prod_{i=1}^{n} x_i^{\mu_i},
\quad
c_{\mu,\nu}(p,q)
\in
\mathbb{Q}(p,q),
\\
Y_iE_{\mu}(x_1,\dots,x_n;p,q) &= y_i(\mu;p,q)E_{\mu}(x_1,\dots,x_n;p,q),
\quad
\forall\ 1 \leq i \leq n, \quad \mu \in \mathbb{N}^n,
\label{eq:eigYi}
\end{align}
with eigenvalues given by
\begin{align}
\label{rhomu}
\index{y@$y_i(\mu;p,q)$; Cherednik--Dunkl eigenvalues}
y_i(\mu;p,q)= p^{\mu_i} q^{\rho_i(\mu)+n-i},
\quad
\rho_i(\mu)= -\#\{j\not=i : \mu_j > \mu_i\} - \#\{j<i : \mu_j = \mu_i\}.
\end{align}
%

\begin{prop}\cite{Sahi}\cite[Theorem 4.2]{Knop}
Let $\mu$ be any composition such that $\mu_i < \mu_{i+1}$. The non-symmetric Macdonald polynomials have the following recursive property:
\begin{align}
\label{E_exch}
E_{\mathfrak{s}_i\cdot \mu}(x_1,\dots,x_n;p,q)
=
q^{-1}
\left( \tilde{T}_i + \frac{1-q}{1- y_{i+1}(\mu)/y_i(\mu)} \right)
E_{\mu}(x_1,\dots,x_n;p,q),
\end{align}
where we abbreviate the eigenvalues \eqref{rhomu} by $y_i(\mu;p,q) \equiv y_i(\mu)$.
\end{prop}

We shall be interested in two limits of the non-symmetric Macdonald polynomials; namely, when $p \rightarrow 0$ and $p \rightarrow \infty$. Both limits can be freely taken (the coefficients $c_{\mu,\nu}(p,q)$ in equation \eqref{eq:monic} are well-behaved at both of these values of $p$); the resulting polynomials, 
$E_{\mu}(x_1,\dots,x_n;0,q)$ and $E_{\mu}(x_1,\dots,x_n;\infty,q^{-1})$, will be referred to as {\it non-symmetric Hall--Littlewood polynomials} and {\it dual non-symmetric Hall--Littlewood polynomials}, respectively.

\begin{prop}
Let $\mu = (\mu_1,\dots,\mu_n)$ be a composition. The non-symmetric Hall--Littlewood polynomials and their duals satisfy the following recursion relations:
\begin{align}
\label{E-T}
\index{E@$E_{\mu}$; non-symmetric Hall--Littlewood}
T_i E_{(\mu_n,\dots,\mu_1)}(x_n,\dots,x_1;0,q)
&=
E_{(\mu_n,\dots,\mu_i,\mu_{i+1},\dots,\mu_1)}(x_n,\dots,x_1;0,q),
\qquad
\mu_i < \mu_{i+1},
\\
\label{E-That}
\widehat{T}_{n-i} E_{(\mu_n,\dots,\mu_1)}(x_1,\dots,x_n;\infty,q^{-1})
&=
E_{(\mu_n,\dots,\mu_i,\mu_{i+1},\dots,\mu_1)}(x_1,\dots,x_n;\infty,q^{-1}),
\qquad
\mu_i > \mu_{i+1}.
\end{align}
\end{prop}

\begin{proof}
Both equations \eqref{E-T} and \eqref{E-That} are, up to relabellings and changes of variables, limits of the recursion \eqref{E_exch} for the non-symmetric Macdonald polynomials. 

For the proof of \eqref{E-T}, we start from \eqref{E_exch} with $i \mapsto n-i$ and send $p \rightarrow 0$. Since by assumption $\mu_{n-i} < \mu_{n-i+1}$, the ratio of eigenvalues $y_{n-i+1}(\mu) / y_{n-i}(\mu)$ contains a positive power of $p$, and this ratio therefore vanishes when $p \rightarrow 0$. It follows that
\begin{align}
\label{E-T-2}
E_{\mathfrak{s}_{n-i}\cdot \mu}(x_1,\dots,x_n;0,q)
=
q^{-1}
\left( \tilde{T}_{n-i} + 1 - q \right)
E_{\mu}(x_1,\dots,x_n;0,q),
\qquad
\mu_{n-i} < \mu_{n-i+1}.
\end{align}
Now reverse the order of the alphabet $(x_1,\dots,x_n)$ and the composition 
$(\mu_1,\dots,\mu_n)$ in \eqref{E-T-2}, giving
\begin{align}
\label{E-T-3}
E_{(\mu_n,\dots,\mu_i,\mu_{i+1},\dots,\mu_1)}(x_n,\dots,x_1;0,q)
=
q^{-1}
\left( T_i + 1 - q \right)
E_{(\mu_n,\dots,\mu_1)}(x_n,\dots,x_1;0,q),
\qquad
\mu_{i} > \mu_{i+1}.
\end{align}
Finally, noting that $q^{-1}(T_i+1-q) = T^{-1}_i$, the relation \eqref{E-T} follows by interchanging $\mu_i \leftrightarrow \mu_{i+1}$ in \eqref{E-T-3}.

For the proof of \eqref{E-That}, start from \eqref{E_exch} with the replacements $i \mapsto n-i$, $p \mapsto p^{-1}$, $q \mapsto q^{-1}$ and send $p \rightarrow 0$. This time, since $\mu_{n-i} < \mu_{n-i+1}$, the ratio $y_{n-i+1}(\mu;p^{-1},q^{-1}) / y_{n-i}(\mu;p^{-1},q^{-1})$ contains a negative power of $p$, which diverges as $p \rightarrow 0$. We therefore find that
\begin{align}
\label{E-That-2}
E_{\mathfrak{s}_{n-i}\cdot \mu}(x_1,\dots,x_n;\infty,q^{-1})
=
q \tilde{T}_{n-i}(x,q^{-1})
E_{\mu}(x_1,\dots,x_n;\infty,q^{-1}),
\qquad
\mu_{n-i} < \mu_{n-i+1},
\end{align}
where the $q$ dependence of the Hecke generator \eqref{rev-Hecke} has been inverted. Reversing the order of the composition $(\mu_1,\dots,\mu_n)$ in \eqref{E-That-2}, we obtain
\begin{align}
\label{E-That-3}
E_{(\mu_n,\dots,\mu_i,\mu_{i+1},\dots,\mu_1)}(x_1,\dots,x_n;\infty,q^{-1})
=
q \tilde{T}_{n-i}(x,q^{-1})
E_{(\mu_n,\dots,\mu_1)}(x_1,\dots,x_n;\infty,q^{-1}),
\qquad
\mu_{i} > \mu_{i+1},
\end{align}
and the result \eqref{E-That} can now be deduced, in view of the fact that
\begin{align*}
q \tilde{T}_{n-i}(x,q^{-1})
&=
q
\left(
q^{-1} - \frac{q^{-1} x_{n-i} - x_{n-i+1}}{x_{n-i} - x_{n-i+1}}(1-\mathfrak{s}_{n-i}) 
\right)
=
1 - \frac{x_{n-i} - q x_{n-i+1}}{x_{n-i} - x_{n-i+1}}(1-\mathfrak{s}_{n-i}) 
=
\widehat{T}_{n-i}.
\end{align*}

\end{proof}

\begin{prop}
Let $\lambda = (\lambda_1 \geq \cdots \geq \lambda_n \geq 0)$ be a partition and 
$\delta = (0 \leq \delta_1 \leq \cdots \leq \delta_n)$ an anti-dominant composition. The non-symmetric Hall--Littlewood polynomials have the following factorized initial conditions:
\begin{align}
\label{IC}
E_{\lambda}(x_1,\dots,x_n;0,q)
=
\prod_{i=1}^{n} x_i^{\lambda_i},
\qquad
E_{\delta}(x_1,\dots,x_n;\infty,q^{-1})
=
\prod_{i=1}^{n} x_i^{\delta_i}.
\end{align}
\end{prop}

\begin{proof}
Both of these expressions can be readily deduced from the explicit combinatorial formula for the non-symmetric Macdonald polynomials obtained in \cite[Theorem 3.5.1]{HaglundHL}.
\end{proof}

We are now ready to state the relationship between the functions $f_{\mu}$, $g_{\mu}$ and the non-symmetric Hall--Littlewood polynomials:
\begin{thm}
\label{thm:f-E}
For any composition $\mu = (\mu_1,\dots,\mu_n)$, there holds
\begin{align}
\label{f-E}
f_{\mu}(x_1,\dots,x_n;q,s)
\Big|_{s = 0}
&=
E_{\tilde\mu}(x_n,\dots,x_1;0,q),
\\
\label{g-E}
g^{*}_{\mu}(x_1,\dots,x_n;q,s)
\Big|_{s = 0}
&=
E_{\tilde\mu}(x_n,\dots,x_1;\infty,q^{-1}),
\end{align}
where $\tilde\mu = (\mu_n,\dots,\mu_1)$, and $g^{*}_{\mu}$ denotes the renormalized version \eqref{g-star} of $g_{\mu}$.
\end{thm}

\begin{proof}
We begin by checking \eqref{f-E} when $\mu$ is anti-dominant, \ie\ the case $\mu = (\mu_1 \leq \cdots \leq \mu_n)$. Using \eqref{f-delta}, we see that
\begin{align*}
f_{\mu}(x_1,\dots,x_n;q,s)
\Big|_{s = 0}
=
\prod_{i=1}^{n} x_i^{\mu_i},
\qquad
\mu = (\mu_1 \leq \cdots \leq \mu_n),
\end{align*}
which matches precisely with $E_{\tilde\mu}(x_n,\dots,x_1;0,q)$ as given by \eqref{IC}. Now by studying \eqref{T-f} (at $s=0$) and \eqref{E-T}, it is clear that $f_{\mu}(x_1,\dots,x_n;q,0)$ and $E_{\tilde\mu}(x_n,\dots,x_1;0,q)$ satisfy the same recursion, allowing the parts $\mu_i < \mu_{i+1}$ of the composition $\mu$ to be exchanged. Since an arbitrary composition $\mu$ can always be realized as a string of such simple transpositions acting on its anti-dominant ordering, we conclude that \eqref{f-E} holds generally.

The proof of \eqref{g-E} is analogous. First, by combining equations \eqref{f-delta} and \eqref{f-g-sym}, it is easily shown that
\begin{align}
\label{g-lambda}
g^{*}_{\lambda}(x_1,\dots,x_n;q,s)
=
\prod_{i=1}^{n}
\frac{1}{1-sx_i}
\left(
\frac{x_i-s}{1-sx_i}
\right)^{\lambda_i},
\qquad
\lambda = (\lambda_1 \geq \cdots \geq \lambda_n),
\end{align}
which is the analogue of the factorization \eqref{f-delta} in the case of the dual functions $g_{\mu}$. This equation can then be used to check \eqref{g-E} when $\mu$ is dominant, \ie\ the case $\mu = (\mu_1 \geq \cdots \geq \mu_n)$. Comparing \eqref{That-g} (at $s=0$) and \eqref{E-That}, one sees that 
$g^{*}_{\mu}(x_1,\dots,x_n;q,0)$ and $E_{\tilde\mu}(x_n,\dots,x_1;\infty,q^{-1})$ satisfy the same recursion, allowing the parts $\mu_i > \mu_{i+1}$ of $\mu$ to be exchanged. We conclude that \eqref{g-E} holds in general, using the same line of reasoning as above.
\end{proof}

\section{Eigenrelation for the non-symmetric Hall--Littlewood polynomials}

In the previous section we showed that the function $f_{\mu}$ reduces, at $s=0$, to a non-symmetric Hall--Littlewood polynomial. Our goal in this section is to make use of this fact to derive the eigenrelation for the non-symmetric Hall--Littlewood polynomials, when acted upon by the $p=0$ case of the Cherednik--Dunkl operators \eqref{eq:Yi}.

\begin{thm}
For any composition $\mu = (\mu_1,\dots,\mu_n)$, one has the eigenrelation
\begin{align}
\label{HL-eig}
Y_i \Big|_{p=0} 
E_{\mu}(x_1,\dots,x_n;0,q)
=
y_i(\mu;0,q)
E_{\mu}(x_1,\dots,x_n;0,q),
\end{align}
where
\begin{align}
\label{HL-eig2}
y_i(\mu;0,q)
=
\left\{
\begin{array}{ll}
0, & \quad \mu_i \geq 1,
\\ \\
q^{\#\{j > i : \mu_j = 0\}-i+1},
& \quad \mu_i = 0.
\end{array}
\right.
\end{align}
\end{thm}

\begin{proof}
Let us begin by recasting this statement in terms of the non-symmetric spin Hall--Littlewood functions at $s=0$. Making use of \eqref{f-E}, we see that \eqref{HL-eig} translates to the relation
\begin{align}
\label{5.19-1}
\left(
T_{n-i} \cdots T_1 \tilde\omega T_{n-1}^{-1} \cdots T_{n-i+1}^{-1}
\right)
f_{\tilde\mu}(x_1,\dots,x_n) \Big|_{s=0}
=
y_i(\mu;0,q)
f_{\tilde\mu}(x_1,\dots,x_n) \Big|_{s=0},
\end{align}
with $\tilde\mu = (\mu_n,\dots,\mu_1)$, and where we have defined the operator 
$\tilde\omega$ with action
\begin{align*}
\tilde\omega \cdot h(x_1,\dots,x_n)
:=
h(x_2,\dots,x_n,0)
\end{align*}
on arbitrary functions $h$ of the alphabet $(x_1,\dots,x_n)$. Transferring the string of operators $T_{n-i} \cdots T_1$ to the right hand side of \eqref{5.19-1}, and replacing all inverse Hecke generators $T_k^{-1}$ by $q^{-1} \widehat{T}_k$, we then read
\begin{align*}
\tilde\omega 
\widehat{T}_{n-1} \cdots \widehat{T}_{n-i+1}
f_{\tilde\mu}(x_1,\dots,x_n) \Big|_{s=0}
=
\left(
{\bm 1}_{\mu_i = 0}
q^{\#\{j > i : \mu_j = 0\}+i-n}
\right)
\widehat{T}_1 \cdots \widehat{T}_{n-i}
f_{\tilde\mu}(x_1,\dots,x_n) \Big|_{s=0},
\end{align*}
which can be brought into a nicer form by performing the replacement $i \mapsto n-i+1$ and dropping the tilde from $\tilde\mu$, since it ends up being irrelevant in the final expression:
\begin{align}
\nonumber
\tilde\omega 
\widehat{T}_{n-1} \cdots \widehat{T}_{i}
f_{\mu}(x_1,\dots,x_n) \Big|_{s=0}
&=
\left(
{\bm 1}_{\mu_i = 0}
q^{\#\{j < i : \mu_j = 0\}-i+1}
\right)
\widehat{T}_1 \cdots \widehat{T}_{i-1}
f_{\mu}(x_1,\dots,x_n) \Big|_{s=0}
\\
&=
\left(
{\bm 1}_{\mu_i = 0}
q^{-\#\{j < i : \mu_j > 0\}}
\right)
\widehat{T}_1 \cdots \widehat{T}_{i-1}
f_{\mu}(x_1,\dots,x_n) \Big|_{s=0}.
\label{5.19-2}
\end{align}
We will now prove \eqref{5.19-2} directly, for all compositions $\mu$ and $1 \leq i \leq n$.

Making use of the algebraic expression \eqref{f-HL} for $f_{\mu}$, as well as repeated actions \eqref{That-f} of the Hecke generators $\widehat{T}_k$, we see that the left and right hand sides of \eqref{5.19-2} can be computed as follows:
\begin{align}
\nonumber
\tilde\omega 
\widehat{T}_{n-1} \cdots \widehat{T}_{i}
f_{\mu}(x_1,\dots,x_n) \Big|_{s=0}
&=
\tilde\omega
\bra{\varnothing}
\C_1(x_1) \dots \C_{i-1}(x_{i-1})
\C_{i+1}(x_i) \dots \C_{n}(x_{n-1})
\C_i(x_n)
\ket{\mu}
\Big|_{s=0}
\\
&=
\bra{\varnothing}
\C_1(x_2) \dots \C_{i-1}(x_i)
\C_{i+1}(x_{i+1}) \dots \C_{n}(x_n)
\C_i(0)
\ket{\mu}
\Big|_{s=0},
\label{lhs-eig}
\end{align}

\begin{align}
\label{rhs-eig}
\widehat{T}_1 \cdots \widehat{T}_{i-1}
f_{\mu}(x_1,\dots,x_n) \Big|_{s=0}
=
\bra{\varnothing}
\C_i(x_1)
\C_1(x_2) \dots \C_{i-1}(x_i)
\C_{i+1}(x_{i+1}) \dots \C_{n}(x_n)
\ket{\mu}
\Big|_{s=0}.
\end{align}
Let us now compare \eqref{lhs-eig} and \eqref{rhs-eig} as partition functions. Translating the right hand side of equation \eqref{lhs-eig} into graphical language, we find that
\begin{align}
\label{lhs-pic}
\tilde\omega 
\widehat{T}_{n-1} \cdots \widehat{T}_{i}
f_{\mu}(x_1,\dots,x_n) \Big|_{s=0}
=
\tikz{0.8}{
\foreach\y in {2,...,6}{
\draw[lgray,line width=1.5pt] (1,\y) -- (8,\y);
}
\foreach\x in {2,...,7}{
\draw[lgray,line width=4pt] (\x,1) -- (\x,7);
}
\draw[ultra thick,yellow,->] (1,6) -- (1.9,6) -- (1.9,7);
\draw[ultra thick,red,->] (1,5) -- (5.1,5) -- (5.1,7);
\draw[ultra thick,orange,->] (1,4) -- (2.1,4) -- (2.1,7);
\draw[ultra thick,green,->] (1,3) -- (4.9,3) -- (4.9,7);
\draw[ultra thick,blue,->] (1,2) -- (4,2) -- (4,4) -- (6,4) -- (6,5) -- (7,5) -- (7,7);
\node[above] at (2,7) {\tiny ${\bm A}$};
\node[below] at (1.9,6) {\tiny ${\bm A}^{-}_i$};
\node at (2.5,6) {\tiny $0$};
\node[left] at (-0.5,2) {$x_2 \rightarrow$}; \node[left] at (1,2) {$1$};
\node[left] at (-0.5,3) {$x_i \rightarrow$}; \node[left] at (1,3) {$i-1$};
\node[left] at (-0.5,4) {$x_{i+1} \rightarrow$}; \node[left] at (1,4) {$i+1$};
\node[left] at (-0.5,5) {$x_n \rightarrow$}; \node[left] at (1,5) {$n$};
\node[left] at (-0.5,6) {$x_1=0 \rightarrow$}; \node[left] at (1,6) {$i$};
\node at (-1,2.6) {\tiny $\vdots$}; \node at (1.5,2.6) {\tiny $\vdots$};
\node at (-1,4.6) {\tiny $\vdots$}; \node at (1.5,4.6) {\tiny $\vdots$};
\draw[densely dotted,thick] (0,5.4) -- (8.5,5.4);
}
\end{align}
where we have chosen a sample configuration of the lattice, and indicated the rapidity associated to every row, as well as the colour which enters each row via its left boundary. The state ${\bm A}$ shown at the top of the 0-th column is given by ${\bm A} = \sum_{k=1}^{n} {\bm 1}_{\mu_k = 0}{\bm e}_k$. Since both the parameter $s$ and the rapidity of the top row are set equal to $0$, by consulting the weights \eqref{s-weights} one sees that the only allowed vertices in the top row of \eqref{lhs-pic} are those of the form 
\tikz{0.5}{
\draw[lgray,line width=1.5pt,->] (-1,0) -- (1,0);
\draw[lgray,line width=4pt,->] (0,-1) -- (0,1);
\node[left] at (-1,0) {\tiny $i$};\node[right] at (1,0) {\tiny $0$};
\node[below] at (0,-1) {\tiny ${\bm A}^{-}_i$};\node[above] at (0,1) {\tiny ${\bm A}$};
}
and
\tikz{0.5}{
\draw[lgray,line width=1.5pt,->] (-1,0) -- (1,0);
\draw[lgray,line width=4pt,->] (0,-1) -- (0,1);
\node[left] at (-1,0) {\tiny $0$};\node[right] at (1,0) {\tiny $0$};
\node[below] at (0,-1) {\tiny ${\bm B}$};\node[above] at (0,1) {\tiny ${\bm B}$};
}, where the state ${\bm B}$ can be generic. This means that the partition function is non-vanishing only if $\mu_i = 0$. In the case where this holds, all vertices above the dotted line are frozen with weight $1$, and we can delete them to obtain the equation
\begin{align}
\label{lhs-}
\tilde\omega 
\widehat{T}_{n-1} \cdots \widehat{T}_{i}
f_{\mu}(x_1,\dots,x_n) \Big|_{s=0}
=
{\bm 1}_{\mu_i = 0}
\cdot
\tikz{0.8}{
\foreach\y in {2,...,5}{
\draw[lgray,line width=1.5pt] (1,\y) -- (8,\y);
}
\foreach\x in {2,...,7}{
\draw[lgray,line width=4pt] (\x,1) -- (\x,6);
}
\draw[ultra thick,red,->] (1,5) -- (5.1,5) -- (5.1,6);
\draw[ultra thick,orange,->] (1,4) -- (2,4) -- (2,6);
\draw[ultra thick,green,->] (1,3) -- (4.9,3) -- (4.9,6);
\draw[ultra thick,blue,->] (1,2) -- (4,2) -- (4,4) -- (6,4) -- (6,5) -- (7,5) -- (7,6);
\node[above] at (2,6) {\tiny ${\bm A}^{-}_i$};
\node[left] at (-0.5,2) {$x_2 \rightarrow$}; \node[left] at (1,2) {$1$};
\node[left] at (-0.5,3) {$x_i \rightarrow$}; \node[left] at (1,3) {$i-1$};
\node[left] at (-0.5,4) {$x_{i+1} \rightarrow$}; \node[left] at (1,4) {$i+1$};
\node[left] at (-0.5,5) {$x_n \rightarrow$}; \node[left] at (1,5) {$n$};
\node at (-1,2.6) {\tiny $\vdots$}; \node at (1.5,2.6) {\tiny $\vdots$};
\node at (-1,4.6) {\tiny $\vdots$}; \node at (1.5,4.6) {\tiny $\vdots$};
}
\end{align}
where the partition function is now truncated to $n-1$ rows, and 
${\bm A}^{-}_i = \sum_{k\not=i}^{n} {\bm 1}_{\mu_k = 0}{\bm e}_k$. This already proves \eqref{5.19-2} in the situation $\mu_i \geq 1$. 

To complete the proof, we turn to the graphical representation of \eqref{rhs-eig} when $\mu_i = 0$. It takes the form
\begin{align}
\label{rhs-pic}
\widehat{T}_1 \cdots \widehat{T}_{i-1}
f_{\mu}(x_1,\dots,x_n) \Big|_{s=0}
=
\tikz{0.8}{
\foreach\y in {2,...,6}{
\draw[lgray,line width=1.5pt] (1,\y) -- (8,\y);
}
\foreach\x in {2,...,7}{
\draw[lgray,line width=4pt] (\x,1) -- (\x,7);
}
\draw[ultra thick,red,->] (1,6) -- (5.1,6) -- (5.1,7);
\draw[ultra thick,orange,->] (1,5) -- (2.1,5) -- (2.1,7);
\draw[ultra thick,green,->] (1,4) -- (4.9,4) -- (4.9,7);
\draw[ultra thick,blue,->] (1,3) -- (4,3) -- (4,5) -- (6,5) -- (6,6) -- (7,6) -- (7,7);
\draw[ultra thick,yellow,->] (1,2) -- (1.9,2) -- (1.9,7);
\node[above] at (2,7) {\tiny ${\bm A}$};
\node[left] at (-0.5,2) {$x_1 \rightarrow$}; \node[left] at (1,2) {$i$};
\node[left] at (-0.5,3) {$x_2 \rightarrow$}; \node[left] at (1,3) {$1$};
\node[left] at (-0.5,4) {$x_i \rightarrow$}; \node[left] at (1,4) {$i-1$};
\node[left] at (-0.5,5) {$x_{i+1} \rightarrow$}; \node[left] at (1,5) {$i+1$};
\node[left] at (-0.5,6) {$x_n \rightarrow$}; \node[left] at (1,6) {$n$};
\node at (-1,3.6) {\tiny $\vdots$}; \node at (1.5,3.6) {\tiny $\vdots$};
\node at (-1,5.6) {\tiny $\vdots$}; \node at (1.5,5.6) {\tiny $\vdots$};
\draw[densely dotted,thick] (0,2.5) -- (8.5,2.5);
}
\end{align}
where we have drawn the same sample configuration as in \eqref{lhs-pic}, but with the path of colour $i$ now incoming in the lowest row of the lattice, and all other paths shifted upwards by a single unit. Because $\mu_i = 0$ by assumption, the path of colour $i$ is forced to make an upward turn as soon as it enters the lattice, and then propagates straight to the top of the 0-th column. Given that the trajectory of the $i$-th path is completely determined, we can consider the effect of deleting this path altogether, and truncating the lattice above the dotted line in \eqref{rhs-pic}. We are able to perform these operations and the partition function remains invariant, modulo an overall multiplicative factor $q^{\#\{j<i:\mu_j>0\}}$. The origin of this factor is easily explained: any path of colour 
$j < i$ that does {\it not} leave the lattice \eqref{rhs-pic} through ${\bm A}$ must exit the 0-th column via a vertex of the form 
\tikz{0.5}{
\draw[lgray,line width=1.5pt,->] (-1,0) -- (1,0);
\draw[lgray,line width=4pt,->] (0,-1) -- (0,1);
\node[left] at (-1,0) {\tiny $j$};\node[right] at (1,0) {\tiny $j$};
\node[below] at (0,-1) {\tiny ${\bm B}$};\node[above] at (0,1) {\tiny ${\bm B}$};
}, where ${\bm B}$ denotes a generic state positioned along the 0-th column. Since this vertex introduces the weight $x_j q^{{\bm B}_{[j+1,n]}}$ (recalling that $s=0$), it will have an extra factor of $q$ when the path $i$ is retained in the lattice compared with when it is deleted. It follows that, when $\mu_i = 0$, one has
\begin{align}
\label{rhs-}
\widehat{T}_1 \cdots \widehat{T}_{i-1}
f_{\mu}(x_1,\dots,x_n) \Big|_{s=0}
=
q^{\#\{j<i : \mu_j > 0\}}
\cdot
\tikz{0.8}{
\foreach\y in {2,...,5}{
\draw[lgray,line width=1.5pt] (1,\y) -- (8,\y);
}
\foreach\x in {2,...,7}{
\draw[lgray,line width=4pt] (\x,1) -- (\x,6);
}
\draw[ultra thick,red,->] (1,5) -- (5.1,5) -- (5.1,6);
\draw[ultra thick,orange,->] (1,4) -- (2,4) -- (2,6);
\draw[ultra thick,green,->] (1,3) -- (4.9,3) -- (4.9,6);
\draw[ultra thick,blue,->] (1,2) -- (4,2) -- (4,4) -- (6,4) -- (6,5) -- (7,5) -- (7,6);
\node[above] at (2,6) {\tiny ${\bm A}^{-}_i$};
\node[left] at (-0.5,2) {$x_2 \rightarrow$}; \node[left] at (1,2) {$1$};
\node[left] at (-0.5,3) {$x_i \rightarrow$}; \node[left] at (1,3) {$i-1$};
\node[left] at (-0.5,4) {$x_{i+1} \rightarrow$}; \node[left] at (1,4) {$i+1$};
\node[left] at (-0.5,5) {$x_n \rightarrow$}; \node[left] at (1,5) {$n$};
\node at (-1,2.6) {\tiny $\vdots$}; \node at (1.5,2.6) {\tiny $\vdots$};
\node at (-1,4.6) {\tiny $\vdots$}; \node at (1.5,4.6) {\tiny $\vdots$};
}
\end{align}
where we have dropped the path of colour $i$ and truncated to $n-1$ rows, by deleting everything below the dotted line. Matching \eqref{lhs-} and \eqref{rhs-} then yields the desired result \eqref{5.19-2} for $\mu_i = 0$.

\end{proof}

\chapter{Monomial expansions: permutation graphs}

The goal of this chapter is to derive a summation formula for the non-symmetric spin Hall--Littlewood functions $f_{\mu}$ and $g_{\mu}$, expanding them in terms of the monomials
\begin{align}
\label{xi}
\index{an@$\xi_{\mu}$; rational monomials}
\xi_{\mu}(x_1,\dots,x_n)
:=
\prod_{i=1}^{n}
\frac{1}{1-sx_i}
\left(
\frac{x_i-s}{1-sx_i}
\right)^{\mu_i}.
\end{align}
Such expansions are a typical result in the theory of the (symmetric) spin Hall--Littlewood functions, see \cite{Borodin,BorodinP1,BorodinP2}, and are essential in the proof of orthogonality statements. 

The monomial expansions obtained in \cite{Borodin,BorodinP1,BorodinP2} consist of a single sum over the symmetric group, with the coefficient of each monomial \eqref{xi} being completely factorized. In the case of the non-symmetric spin Hall--Littlewood functions, the situation is more complicated: while they can also be expressed as single sum over the symmetric group, the coefficients of the monomials \eqref{xi} are not factorized, and need to be calculated combinatorially in terms of what we call ``permutation graphs''. Our formulas, listed in Corollary \ref{cor:mon-exp}, are directly related to Takeyama's algebraic algorithm for constructing eigenfunctions of the multi-species $q$-boson system \cite{Takeyama}.

\section{Warm-up: $F$-matrices for two-site spin chains}
\label{ssec:F2}

Recall the definition \eqref{Rmat} of the $U_q(\wh{\mathfrak{sl}_{n+1}})$ $R$-matrix, and write it as $R_{12}(z) \in {\rm End}(W_1 \otimes W_2)$ (\ie\ we label the auxiliary spaces numerically). A two-site $F$-matrix, \index{F@$F$-matrix} $F_{12}(z) \in {\rm End}(W_1 \otimes W_2)$, is an invertible $(n+1)^2 \times (n+1)^2$ matrix solution of the equation
\begin{align}
\label{FR-F}
F_{21}(\b{z}) R_{12}(z)
=
F_{12}(z),
\qquad
\b{z} := z^{-1},
\end{align}
where $F_{21}(\b{z}) = P_{12} F_{12}(\b{z}) P_{12}$, with $P_{12} = R_{12}(1)$ \index{P@$P$-matrix} denoting the permutation operator on $W_1 \otimes W_2$. Equivalently, one has
\begin{align}
\label{R-FF}
R_{12}(z)
=
F^{-1}_{21}(\b{z})
F_{12}(z),
\end{align}
thereby providing a factorization of the $R$-matrix. We will not be concerned with classifying all possible solutions of \eqref{FR-F}, but instead focus on a particular solution, given below.

\begin{prop}\cite{SantosM,AlbertBFR}
\label{prop:FR-F}
A solution of \eqref{FR-F} is given by
\begin{align}
\label{Fmat}
F_{12}(z)
=
\left(
\sum_{0 \leq k \leq l \leq n}
E^{(kk)}_1 E^{(ll)}_2
\right)
+
\left(
\sum_{0 \leq k < l \leq n}
E^{(ll)}_1 E^{(kk)}_2
\right)
R_{12}(z),
\end{align}
where $E^{(kk)}_1$ and $E^{(ll)}_2$ denote $(n+1) \times (n+1)$ elementary matrices, as in Section \ref{ssec:fundamental}.
\end{prop}

\begin{proof}
Calculating the left hand side of \eqref{FR-F}, we have
\begin{align}
\nonumber
F_{21}(\b{z}) R_{12}(z)
&=
\left(
\sum_{0 \leq k \leq l \leq n}
E^{(kk)}_2 E^{(ll)}_1
\right)
R_{12}(z)
+
\left(
\sum_{0 \leq k < l \leq n}
E^{(ll)}_2 E^{(kk)}_1
\right)
R_{21}(\b{z})
R_{12}(z)
\\
\label{prop6.1-1}
&=
\left(
\sum_{0 \leq k \leq l \leq n}
E^{(ll)}_1 E^{(kk)}_2
\right)
R_{12}(z)
+
\left(
\sum_{0 \leq k < l \leq n}
E^{(kk)}_1 E^{(ll)}_2
\right),
\end{align}
invoking the unitarity property $R_{21}(\b{z}) R_{12}(z) = 1$ of the $R$-matrices. Finally, we note the relation
\begin{align}
\label{a-vertex}
\left( \sum_{k=0}^{n} E^{(kk)}_1 E^{(kk)}_2 \right)
R_{12}(z)
=
\left( \sum_{k=0}^{n} E^{(kk)}_1 E^{(kk)}_2 \right),
\end{align}
which arises from the fact that if we fix the two incoming states of the vertex \eqref{R-vert} to the same value $k \in \{0,1,\dots,n\}$ then the two outgoing states of the vertex must also assume the value $k$, and the resulting vertex \eqref{R-weights-a} has weight $1$. Employing the relation \eqref{a-vertex} in \eqref{prop6.1-1}, we transfer the $k = l$ terms from one sum to the other; the result is
\begin{align*}
\nonumber
F_{21}(\b{z}) R_{12}(z)
=
\left(
\sum_{0 \leq k < l \leq n}
E^{(ll)}_1 E^{(kk)}_2
\right)
R_{12}(z)
+
\left(
\sum_{0 \leq k \leq l \leq n}
E^{(kk)}_1 E^{(ll)}_2
\right)
=
F_{12}(z).
\end{align*}
\end{proof}
The $F$-matrix \eqref{Fmat} is a {\it Drinfeld twist} \cite{Drinfeld,SantosM}, from the theory of quasi-triangular Hopf algebras. It is invertible, with inverse given explicitly by
\begin{align}
\label{Finv}
F^{-1}_{12}(z)
=
\left[
\left(
\sum_{n \geq k > l \geq 0}
E^{(kk)}_1 E^{(ll)}_2
\right)
+
R_{21}(\b{z})
\left(
\sum_{n \geq k \geq l \geq 0}
E^{(ll)}_1 E^{(kk)}_2
\right)
\right]
\Delta_{12}(z),
\end{align}
where $\Delta_{12}(z)$ is a diagonal $(n+1)^2 \times (n+1)^2$ matrix of the form
\begin{align}
\label{Delta}
\Delta_{12}(z)
=
\sum_{0 \leq k,l \leq n}
b_{k,l}(z)
E^{(kk)}_1 E^{(ll)}_2,
\qquad
b_{k,l}(z) 
:= 
\left\{
\begin{array}{ll}
1, & \quad k = l,
\\ \\
\dfrac{1-q\b{z}}{1-\b{z}}, & \quad k<l,
\\ \\
\dfrac{1-q z}{1-z}, & \quad k>l.
\end{array}
\right.
\end{align}
Combining the matrices \eqref{Fmat} and \eqref{Finv} in equation \eqref{R-FF} yields an explicit $(\text{upper} \cdot \text{diagonal} \cdot \text{lower})$ factorization of the $R$-matrix, whose main value is in the following result:
\begin{prop}
\label{prop:Qsym}
Fix complex parameters $x_1,x_2$ and a linear operator $Q_{12}(x_1,x_2) \in {\rm End}(W_1 \otimes W_2)$ which satisfies the equation
\begin{align}
\label{RQ}
R_{12}(x_2/x_1) Q_{12}(x_1,x_2)
=
Q_{21}(x_2,x_1) R_{12}(x_2/x_1),
\qquad
Q_{21}(x_2,x_1) = P_{12} Q_{12}(x_2,x_1) P_{12}.
\end{align}
Define a ``twisted'' version of $Q_{12}$ as follows:
\begin{align*}
\widetilde{Q}_{12}(x_1,x_2)
:=
F_{12}(x_2/x_1) Q_{12}(x_1,x_2) F^{-1}_{12}(x_2/x_1).
\end{align*}
Then $\widetilde{Q}_{12}(x_1,x_2)$ is symmetric under simultaneous exchange of its variables and spaces; \ie\ one has
\begin{align}
\label{Q-sym}
\widetilde{Q}_{12}(x_1,x_2) = \widetilde{Q}_{21}(x_2,x_1).
\end{align}
\end{prop}

\begin{proof}
Start from \eqref{RQ} and apply the factorization \eqref{R-FF} of the $R$-matrix, which yields
\begin{align*}
F^{-1}_{21}(x_1/x_2) F_{12}(x_2/x_1) Q_{12}(x_1,x_2)
=
Q_{21}(x_2,x_1) F^{-1}_{21}(x_1/x_2) F_{12}(x_2/x_1),
\end{align*}
and after some rearrangement,
\begin{align*}
F_{12}(x_2/x_1) Q_{12}(x_1,x_2) F^{-1}_{12}(x_2/x_1)
=
F_{21}(x_1/x_2) Q_{21}(x_2,x_1) F^{-1}_{21}(x_1/x_2),
\end{align*}
which is precisely the claim \eqref{Q-sym}.
\end{proof}

\section{$N$-site $R$-matrices}

In what follows we require a natural $N$-site analogue of the $R$-matrix \eqref{Rmat} which we shall denote by \index{R@$R^{\sigma}_{\rho}$} $R^{\sigma}_{\rho}$, where $\sigma, \rho \in \mathfrak{S}_N$. It reduces to \eqref{Rmat} at $N=2$ by choosing 
$\sigma = (2,1)$ and $\rho = (1,2)$:  
\begin{align*}
R^{21}_{12}
&:=
R_{12}(x_2/x_1)
\in 
{\rm End}(W_1 \otimes W_2).
\end{align*}
Slightly more generally, we may keep $N$ generic, choose $\sigma = \mathfrak{s}_i = (1,\dots,i+1,i,\dots,N)$ and $\rho = {\rm id} = 
(1,\dots,N)$:
\begin{align*}
R^{\mathfrak{s}_i}_{{\rm id}}
&:=
R_{i(i+1)}(x_{i+1}/x_i)
\in 
{\rm End}(W_1 \otimes \cdots \otimes W_N),
\end{align*}
where the $R$-matrix acts non-trivially only on the $W_i \otimes W_{i+1}$ subspace of $W_1 \otimes \cdots \otimes W_N$. We then extend the construction to generic permutations $\sigma$ and $\rho$ by setting
\begin{align*}
R^{\rm \sigma}_{\sigma} := 1,
\qquad
R^{\mathfrak{s}_i \circ \sigma}_{\sigma}
:=
R_{\sigma(i)\sigma(i+1)}\left(x_{\sigma(i+1)}/x_{\sigma(i)}\right),
\qquad
\ \forall\ \sigma \in \mathfrak{S}_N,
\end{align*}
as well as by making the recursive definition
\begin{align}
\label{Rsigma}
R^{\mathfrak{s}_i \circ \sigma}_{\rho}
:=
R^{\mathfrak{s}_i \circ \sigma}_{\sigma}
R^{\sigma}_{\rho},
\qquad
\forall\ \sigma,\rho \in \mathfrak{S}_N.
\end{align}
The recursion \eqref{Rsigma} allows one to build $R^{\sigma}_{\rho}$ for arbitrary permutations $\sigma,\rho \in \mathfrak{S}_N$, starting from the identity $R^{\rho}_{\rho}$: all that is needed is a word decomposition of $\sigma \circ \rho^{-1}$ into simple transpositions. Note that the definition \eqref{Rsigma} is unambiguous without specifying the precise word decomposition; any word decomposition of $\sigma \circ \rho^{-1}$ (even non-reduced ones) will yield the same result, by virtue of the Yang--Baxter and unitarity relations for the two-site $R$-matrix \eqref{Rmat}. The resulting $N$-site $R$-matrix $R^{\sigma}_{\rho}$ depends, in principle, on the full alphabet $(x_1,\dots,x_N)$.

\begin{ex}
\label{ex:321}
Let $N=3$ and take $\sigma = \mathfrak{s}_2 \circ \mathfrak{s}_1 \circ \mathfrak{s}_2 = (3,2,1)$, $\rho = {\rm id} = (1,2,3)$. By repeated use of \eqref{Rsigma} we have
\begin{align*}
R^{\mathfrak{s}_2 \circ \mathfrak{s}_1 \circ \mathfrak{s}_2}_{\rm id}
=
R^{\mathfrak{s}_2 \circ \mathfrak{s}_1 \circ \mathfrak{s}_2}_{\mathfrak{s}_1 \circ \mathfrak{s}_2}
R^{\mathfrak{s}_1 \circ \mathfrak{s}_2}_{\rm id}
=
R^{\mathfrak{s}_2 \circ \mathfrak{s}_1 \circ \mathfrak{s}_2}_{\mathfrak{s}_1 \circ \mathfrak{s}_2}
R^{\mathfrak{s}_1 \circ \mathfrak{s}_2}_{\mathfrak{s}_2}
R^{\mathfrak{s}_2}_{\rm id},
\end{align*}
or in terms of one-line notation for the permutations,
\begin{align}
\label{321-lhs}
R^{321}_{123}
=
R^{321}_{312}
R^{312}_{123}
=
R^{321}_{312}
R^{312}_{132}
R^{132}_{123}
=
R_{12}(x_2/x_1)
R_{13}(x_3/x_1)
R_{23}(x_3/x_2).
\end{align}
Notice that we could have also used the decomposition $\sigma = \mathfrak{s}_1 \circ \mathfrak{s}_2 \circ \mathfrak{s}_1 = (3,2,1)$, leading instead to the product
\begin{align*}
R^{\mathfrak{s}_1 \circ \mathfrak{s}_2 \circ \mathfrak{s}_1}_{\rm id}
=
R^{\mathfrak{s}_1 \circ \mathfrak{s}_2 \circ \mathfrak{s}_1}_{\mathfrak{s}_2 \circ \mathfrak{s}_1}
R^{\mathfrak{s}_2 \circ \mathfrak{s}_1}_{\rm id}
=
R^{\mathfrak{s}_1 \circ \mathfrak{s}_2 \circ \mathfrak{s}_1}_{\mathfrak{s}_2 \circ \mathfrak{s}_1}
R^{\mathfrak{s}_2 \circ \mathfrak{s}_1}_{\mathfrak{s}_1}
R^{\mathfrak{s}_1}_{\rm id},
\end{align*}
whose one-line form reads
\begin{align}
\label{321-rhs}
R^{321}_{123}
=
R^{321}_{231}
R^{231}_{123}
=
R^{321}_{231}
R^{231}_{213}
R^{213}_{123}
=
R_{23}(x_3/x_2)
R_{13}(x_3/x_1)
R_{12}(x_2/x_1).
\end{align}
The consistency of the two results \eqref{321-lhs} and \eqref{321-rhs} is precisely the Yang--Baxter equation \eqref{YB}.
\end{ex}

Alternatively, one may define
\begin{align*}
R^{\rho}_{\mathfrak{s}_i \circ \rho}
:=
R_{\rho(i+1)\rho(i)}\left(x_{\rho(i)}/x_{\rho(i+1)}\right),
\qquad
\ \forall\ \rho \in \mathfrak{S}_N,
\end{align*}
and then perform the recursion on the lower index instead, leading to the definition
\begin{align}
\label{Rsigma-rev}
R^{\sigma}_{\mathfrak{s}_i \circ \rho}
:=
R^{\sigma}_{\rho}
R^{\rho}_{\mathfrak{s}_i \circ \rho},
\qquad
\forall\ \sigma,\rho \in \mathfrak{S}_N.
\end{align}
It is clear that the two definitions \eqref{Rsigma} and \eqref{Rsigma-rev} are consistent with one another, since they both lead to word decompositions of $R^{\sigma}_{\rho}$, and every such decomposition produces the same answer. 

\begin{ex}
Let $N=4$ and take $\sigma = \mathfrak{s}_1 \circ \mathfrak{s}_3 = (2,1,4,3)$, $\rho = \mathfrak{s}_2 = (1,3,2,4)$. Using the recursive definition \eqref{Rsigma}, we obtain
\begin{align*}
R^{\mathfrak{s}_1 \circ \mathfrak{s}_3}_{\mathfrak{s}_2}
=
R^{\mathfrak{s}_1 \circ \mathfrak{s}_3}_{\mathfrak{s}_3}
R^{\mathfrak{s}_3}_{\mathfrak{s}_2}
=
R^{\mathfrak{s}_1 \circ \mathfrak{s}_3}_{\mathfrak{s}_3}
R^{\mathfrak{s}_3}_{\rm id}
R^{\rm id}_{\mathfrak{s}_2},
\end{align*}
while the alternative definition \eqref{Rsigma-rev} yields
\begin{align*}
R^{\mathfrak{s}_1 \circ \mathfrak{s}_3}_{\mathfrak{s}_2}
=
R^{\mathfrak{s}_1 \circ \mathfrak{s}_3}_{\rm id}
R^{\rm id}_{\mathfrak{s}_2}
=
R^{\mathfrak{s}_1 \circ \mathfrak{s}_3}_{\mathfrak{s}_3}
R^{\mathfrak{s}_3}_{\rm id}
R^{\rm id}_{\mathfrak{s}_2},
\end{align*}
which is the same result.
\end{ex}

\section{Permutation graphs}

In this section we focus on $N$-site $R$-matrices of the form \index{R@$R^{\sigma}_{1\dots N}$} $R^{\sigma}_{\rm id} \equiv R^{\sigma}_{1\dots N}$. We denote the components of $R^{\sigma}_{1\dots N} \in {\rm End}(W_1 \otimes \cdots \otimes W_N)$ by $R^{\sigma}_{1\dots N}(i_1, \dots, i_N; j_1, \dots, j_N)$, where $i_k,j_k \in \{0,1,\dots,n\}$ for all $1 \leq k \leq N$. In other words, if $\ket{j_1,\dots,j_N}$ denotes a canonical basis state in $W_1 \otimes \cdots \otimes W_N$, then
\begin{align}
\label{Rsig-comp}
R^{\sigma}_{1\dots N}
\ket{j_1,\dots,j_N}
=
\sum_{i_1,\dots,i_N}
R^{\sigma}_{1\dots N}(i_1, \dots, i_N; j_1, \dots, j_N)
\ket{i_1,\dots,i_N}.
\end{align}
It is useful to think of the components of $R^{\sigma}_{1\dots N}$ as certain partition functions in the vertex model \eqref{fund-vert}, as follows. Let $\sigma \in \mathfrak{S}_N$ be a permutation. Draw a column of $N$ nodes, attaching to them the labels $\{j_1,\dots,j_N\}$, where the $k$-th node (counted from the {\it bottom}) receives label $j_k$. To its left, draw a second column of $N$ nodes, attaching to them the labels $\{i_1,\dots,i_N\}$, where the $k$-th node (again, counted from the bottom) receives label $i_{\sigma(k)}$. Now connect the node labelled $i_k$ to that labelled $j_k$ via a line oriented from left to right, and carrying rapidity variable $x_k$, doing this for all $1 \leq k \leq N$ in such a way that no three or more lines intersect at a point. The spectral parameter at each crossing is set to the ratio of rapidities of the corresponding lines, with the numerator being the rapidity of the line that is lower before (on the left of) the crossing. In this way one obtains a partition function 
$\mathcal{X}^{\sigma}_{1\dots N}(i_1,\dots,i_N;j_1,\dots,j_N)$ \index{X1@$\mathcal{X}^{\sigma}_{1\dots N}$; permutation graph} in the model \eqref{fund-vert}; its lattice is the collection of line crossings induced by drawing the graph of $\sigma$, and its incoming/outgoing states (read from bottom to top) are 
$\{i_{\sigma(1)},\dots,i_{\sigma(N)}\}$/$\{j_1,\dots,j_N\}$, respectively. One then has the relation
\begin{align}
\label{R=X}
R^{\sigma}_{1\dots N}(i_1,\dots,i_N;j_1,\dots,j_N)
=
\mathcal{X}^{\sigma}_{1\dots N}(i_1,\dots,i_N;j_1,\dots,j_N).
\end{align}
See Figure \ref{fig:Rgraph} (left panel) for an illustration of this construction.
\begin{figure}
\begin{tikzpicture}[scale=0.7,baseline=(mid.base)]
\node (mid) at (1,4) {};
\draw[lgray,line width=1.5pt,->] plot [smooth] coordinates {(2,6.5) (2.5,6.5) (5.5,5.5) (7.5,3.5) (8.5,3.5)};
\draw[lgray,line width=1.5pt,->] plot [smooth] coordinates {(2,5.5) (2.5,5.5) (5.5,4.5) (7.5,2.5) (8.5,2.5)};
\draw[lgray,line width=1.5pt,->] plot [smooth] coordinates {(2,4.5) (2.5,4.5) (5.5,6.5) (7.5,4.5) (8.5,4.5)};
\draw[lgray,line width=1.5pt,->] plot [smooth] coordinates {(2,3.5) (2.5,3.5) (5.5,2.5) (7.5,5.5) (8.5,5.5)};
\draw[lgray,line width=1.5pt,->] plot [smooth] coordinates {(2,2.5) (2.5,2.5) (5.5,3.5) (7.5,6.5) (8.5,6.5)};
\draw[lgray,line width=1.5pt,->] (2,1.5) -- (8.5,1.5);
\node at (1.3,6.5) {$\ss i_{\sigma(6)}\ $}; 
\node at (1.3,5.5) {$\ss i_{\sigma(5)}\ $}; 
\node at (1.3,4.5) {$\ss i_{\sigma(4)}\ $};
\node at (1.3,3.5) {$\ss i_{\sigma(3)}\ $}; 
\node at (1.3,2.5) {$\ss i_{\sigma(2)}\ $};
\node at (1.3,1.5) {$\ss i_{\sigma(1)}\ $};
\node at (9,6.5) {$\ss j_6$};
\node at (9,5.5) {$\ss j_5$};
\node at (9,4.5) {$\ss j_4$};
\node at (9,3.5) {$\ss j_3$};
\node at (9,2.5) {$\ss j_2$};
\node at (9,1.5) {$\ss j_1$};
\end{tikzpicture}
\quad
\begin{tikzpicture}[scale=0.7,baseline=(mid.base)]
\node (mid) at (1,4) {};
\draw[lgray,line width=1.5pt,->] plot [smooth] coordinates {(2,6.5)(8.5,6.5)};
\draw[lgray,line width=1.5pt,->] plot [smooth] coordinates {(2,5.5) (2.5,5.5) (5.5,4.5) (7.5,2.5) (8.5,2.5)};
\draw[lgray,line width=1.5pt,->] plot [smooth] coordinates {(2,4.5) (2.5,4.5) (5.5,4.5) (7.5,4.5) (8.5,4.5)};
\draw[lgray,line width=1.5pt,->] plot [smooth] coordinates {(2,3.5) (2.5,3.5) (5.5,3) (7.5,5.5) (8.5,5.5)};
\draw[lgray,line width=1.5pt,->] plot [smooth] coordinates {(2,2.5) (2.5,2.5) (5.5,1.5)(8.5,1.5)};
\draw[lgray,line width=1.5pt,->] plot [smooth] coordinates {(2,1.5) (3.5,2.5) (5.5,2.5) (8.5,3.5)};
\node at (1.3,6.5) {$\ss i_{6} $}; 
\node at (1.3,5.5) {$\ss i_{5} $}; 
\node at (1.3,4.5) {$\ss i_{4} $};
\node at (1.3,3.5) {$\ss i_{3} $}; 
\node at (1.3,2.5) {$\ss i_{2} $};
\node at (1.3,1.5) {$\ss i_{1} $};
\node at (9,6.5) {$\ss\ \ j_{\sigma(6)}$};
\node at (9,5.5) {$\ss\ \ j_{\sigma(5)}$};
\node at (9,4.5) {$\ss\ \ j_{\sigma(4)}$};
\node at (9,3.5) {$\ss\ \ j_{\sigma(3)}$};
\node at (9,2.5) {$\ss\ \ j_{\sigma(2)}$};
\node at (9,1.5) {$\ss\ \ j_{\sigma(1)}$};
\end{tikzpicture}
\caption{Examples in the case $N=6$. Each crossing of lattice lines is to be interpreted as a vertex \eqref{R-vert}. Left panel: the permutation graph associated to $\sigma = (1,6,5,4,2,3)$, interpreted as the partition function $\mathcal{X}^{\sigma}_{1\dots N}(i_1,\dots,i_N;j_1,\dots,j_N)$. Right panel: the reversed permutation graph for $\sigma = (2,5,1,4,3,6)$, interpreted as the partition function 
$\mathcal{X}_{\sigma}^{1\dots N}(i_1,\dots,i_N;j_1,\dots,j_N)$.}
\label{fig:Rgraph}
\end{figure}
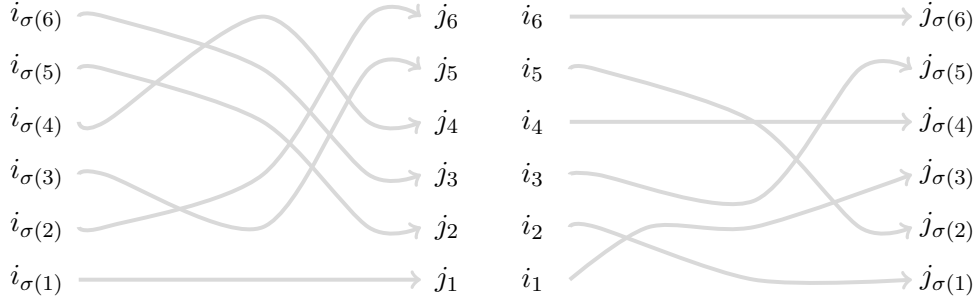

\begin{ex}
Let $N=3$ and $\sigma = (3,2,1)$, as in Example \ref{ex:321}. Following the prescription described above, one has
\begin{align*}
R^{321}_{123}(i_1,i_2,i_3;j_1,j_2,j_3)
=
\begin{tikzpicture}[scale=0.7,baseline=(mid.base)]
\node (mid) at (1,2) {};
\draw[lgray,line width=1.5pt,->] (2,3.5) -- (4,2.5) -- (6,1.5) -- (8,1.5);
\draw[lgray,line width=1.5pt,->] (2,2.5) -- (4,3.5) -- (6,3.5) -- (8,2.5);
\draw[lgray,line width=1.5pt,->] (2,1.5) -- (4,1.5) -- (6,2.5) -- (8,3.5);
\node at (0.7,3.5) {$\ss i_{\sigma(3)} = i_1\ $}; 
\node at (0.7,2.5) {$\ss i_{\sigma(2)} = i_2\ $};
\node at (0.7,1.5) {$\ss i_{\sigma(1)} = i_3\ $};
\node at (8.7,3.5) {$\ss j_3$};
\node at (8.7,2.5) {$\ss j_2$};
\node at (8.7,1.5) {$\ss j_1$};
\end{tikzpicture}
\end{align*}
which is consistent with the decomposition $R^{321}_{123} = R^{321}_{312} R^{312}_{132} R^{132}_{123}$, given in \eqref{321-lhs}.
\end{ex}

\section{Reversed permutation graphs}


In this section we focus on $N$-site $R$-matrices of the form $R^{\rm id}_{\sigma} \equiv R^{1 \dots N}_{\sigma}$. The components of the matrix $R_{\sigma}^{1\dots N}$ can also be realized via a partition function in the model \eqref{fund-vert}. One essentially repeats the graphical construction outlined above, drawing two columns of nodes; the difference is that this time the $k$-th node in the left column receives label $i_k$, while the $k$-th node in the right column is labelled $j_{\sigma(k)}$. Connecting up $i_k$ and $j_k$ for all $1 \leq k \leq N$ produces a partition function \index{X2@$\mathcal{X}_{\sigma}^{1\dots N}$; reversed permutation graph} $\mathcal{X}_{\sigma}^{1\dots N}(i_1,\dots,i_N;j_1,\dots,j_N)$, and one then has 
\begin{align}
\label{R=Xrev}
R_{\sigma}^{1\dots N}(i_1,\dots,i_N;j_1,\dots,j_N)
=
\mathcal{X}_{\sigma}^{1\dots N}(i_1,\dots,i_N;j_1,\dots,j_N),
\end{align}
in direct analogy with \eqref{R=X}. See Figure \ref{fig:Rgraph} (right panel) for an example of this construction.

\section{$F$-matrices for spin chains of generic length}

Following \cite{AlbertBFR}, we now extend the results of Section \ref{ssec:F2} to a tensor product of $N$ auxiliary spaces, \ie\ to $W_1 \otimes \cdots \otimes W_N$. The starting point is to introduce an $N$-site analogue of equation \eqref{FR-F}. Namely, we will seek solutions \index{F@$F_{1\dots N}$} $F_{1\dots N}(x_1,\dots,x_N) \in {\rm End}(W_1 \otimes \cdots \otimes W_N)$ of the equation
\begin{align}
\label{NFR-F}
F_{\sigma(1) \dots \sigma(N)}
\left(x_{\sigma(1)},\dots,x_{\sigma(N)}\right)
R^{\sigma}_{1\dots N}
=
F_{1 \dots N}(x_1,\dots,x_N),
\qquad
\sigma \in \mathfrak{S}_N,
\end{align}
where
\begin{align}
F_{\sigma(1) \dots \sigma(N)}
\left(y_1,\dots,y_N\right)
:=
P^{\sigma}_{1\dots N} 
F_{1\dots N}
\left(y_1,\dots,y_N\right)
P^{\sigma^{-1}}_{1\dots N},
\end{align}
with permutation operators on $W_1 \otimes \cdots \otimes W_N$ prescribed in the standard way,
\begin{align*}
\index{P@$P^{\sigma}_{1 \dots N}$}
P^{\sigma}_{1 \dots N} :
\ket{v_1}_1 \otimes \cdots \otimes \ket{v_N}_N
\mapsto
\ket{v_1}_{\sigma(1)} \otimes \cdots \otimes \ket{v_N}_{\sigma(N)},
\qquad
\forall\ v_1,\dots,v_N \in \{0,1,\dots,n\}.
\end{align*}

\begin{prop}
A solution of \eqref{NFR-F} is given by
\begin{align}
\label{NFmat}
F_{1\dots N}(x_1,\dots,x_N)
=
\sum_{\rho \in \mathfrak{S}_N}
\sum_{(k_1,\dots,k_N) \in \mathcal{I}(\rho)}
\left(
\prod_{i=1}^{N} E^{(k_i k_i)}_{\rho(i)}
\right)
R^{\rho}_{1\dots N},
\end{align}
where the sum is over $N$-tuples in the set $\mathcal{I}(\rho)$, defined as
\begin{align*}
\mathcal{I}(\rho)
=
\Big\{(0 \leq k_1 \leq \cdots \leq k_N \leq n) : k_i < k_{i+1}\ \text{if}\ \rho(i) > \rho(i+1) 
\Big\}.
\end{align*}
\end{prop}

\begin{proof}
We refer the reader to \cite{AlbertBFR,McAteerW} for details.
\end{proof}

\begin{prop}
The $F$-matrix \eqref{NFmat} is lower triangular with non-zero diagonal entries; its inverse is given by
\begin{align}
\label{NFinv}
F_{1\dots N}^{-1}(x_1,\dots,x_N)
=
\left[
\sum_{\rho \in \mathfrak{S}_N}
\sum_{(k_1,\dots,k_N) \in \mathcal{I}'(\rho)}
R^{1\dots N}_{\rho}
\left(
\prod_{i=1}^{N} E^{(k_i k_i)}_{\rho(i)}
\right)
\right]
\Delta_{1\dots N},
\end{align}
where the sum is over $N$-tuples in the set $\mathcal{I}'(\rho)$, defined as
\begin{align*}
\mathcal{I}'(\rho)
=
\Big\{(n \geq k_1 \geq \cdots \geq k_N \geq 0) : 
k_i > k_{i+1}\ \text{if}\ \rho(i) < \rho(i+1) 
\Big\},
\end{align*}
and where $\Delta_{1\dots N}$ is a diagonal matrix with components given by
\begin{align*}
\Delta_{1\dots N}(i_1,\dots,i_N; j_1,\dots,j_N)
=
\prod_{a=1}^{N}
\delta_{i_a,j_a}
\prod_{1 \leq a < b \leq N}
b_{j_a,j_b}(x_b/x_a).
\end{align*}

\end{prop}

\begin{proof}
We refer the reader to \cite{AlbertBFR,McAteerW} for details.
\end{proof}

\begin{rmk}
It is easy to verify that the $N$-site $F$-matrices \eqref{NFmat} and \eqref{NFinv} reduce to \eqref{Fmat} and \eqref{Finv}, respectively, at $N=2$.
\end{rmk}

Although the $F$-matrix \eqref{NFmat} and its inverse \eqref{NFinv} have a seemingly very complicated form, their components can be expressed via appropriate permutation graphs:
\begin{prop}
Let $(i_1,\dots,i_N)$, $(j_1,\dots,j_N)$ be two sets of integers taking values in 
$\{0,1,\dots,n\}$, and denote the components of the matrices \eqref{NFmat} and \eqref{NFinv} by $F_{1\dots N}(i_1,\dots,i_N;j_1,\dots,j_N)$ and 
$F^{-1}_{1\dots N}(i_1,\dots,i_N;j_1,\dots,j_N)$, respectively. Then one has the relations
\begin{align}
\label{Fcomp-1}
F_{1\dots N}(i_1,\dots,i_N;j_1,\dots,j_N)
&=
R_{1\dots N}^{\sigma}(i_1,\dots,i_N;j_1,\dots,j_N),
\\
\label{Fcomp-2}
F^{-1}_{1\dots N}(i_1,\dots,i_N;j_1,\dots,j_N)
&=
R_{\rho}^{1\dots N}(i_1,\dots,i_N;j_1,\dots,j_N)
\prod_{1 \leq a<b \leq N} b_{j_a,j_b}(x_b/x_a),
\end{align}
where $\sigma, \rho \in \mathfrak{S}_N$ are any permutations such that $i_{\sigma(1)} \leq \cdots \leq i_{\sigma(N)}$ and $j_{\rho(1)} \geq \cdots \geq j_{\rho(N)}$.
\end{prop}

\begin{proof}
These relations follow immediately from the explicit form \eqref{NFmat}, \eqref{NFinv} of the $F$-matrices.
\end{proof}

Finally, we come to the $N$-site analogue of Proposition \ref{prop:Qsym}, which will be key in deriving our monomial expansion of the non-symmetric spin Hall--Littlewood functions.

\begin{prop}
\label{prop:Q1N}
Fix complex parameters $x_1,\dots,x_N$ and a linear operator 
$Q_{1\dots N}(x_1,\dots,x_N) \in {\rm End}(W_1 \otimes \cdots \otimes W_N)$ which satisfies the equation
\begin{align}
\label{RQ-QR}
R_{1\dots N}^{\sigma}
Q_{1\dots N}(x_1,\dots,x_N)
=
Q_{\sigma(1)\dots\sigma(N)}\left(x_{\sigma(1)},\dots,x_{\sigma(N)}\right)
R_{1\dots N}^{\sigma},
\quad\quad
\text{for all}\ \sigma \in \mathfrak{S}_N.
\end{align}
We define a twisted version of $Q_{1\dots N}$ by conjugating by the $N$-site $F$-matrix:
\begin{align*}
\widetilde{Q}_{1\dots N}(x_1,\dots,x_N)
:=
F_{1\dots N}(x_1,\dots,x_N)
Q_{1\dots N}(x_1,\dots,x_N)
F^{-1}_{1\dots N}(x_1,\dots,x_N).
\end{align*}
Then the operator $\widetilde{Q}_{1\dots N}(x_1,\dots,x_N)$ is completely symmetric under the simultaneous permutation of its variables and spaces, \ie\ one has
\begin{align*}
\widetilde{Q}_{1\dots N}(x_1,\dots,x_N)
=
\widetilde{Q}_{\sigma(1)\dots\sigma(N)}\left(x_{\sigma(1)},\dots,x_{\sigma(N)}\right),
\quad\quad
\text{for all}\ \sigma \in \mathfrak{S}_N.
\end{align*}
\end{prop}

\begin{proof}
The result follows by replacing $R^{\sigma}_{1\dots N}$ in \eqref{RQ-QR} by its factorized form,
\begin{align*}
R^{\sigma}_{1\dots N} = F^{-1}_{\sigma(1)\dots\sigma(N)} F_{1\dots N},
\end{align*}
then rearrangement of the matrix multiplication, exactly as in the proof of Proposition \ref{prop:Qsym}.
\end{proof}

\section{Column operators}

Previously, we constructed the non-symmetric spin Hall--Littlewood functions by means of the row operators considered in Section \ref{ssec:row-ops}. We now introduce another algebraic formulation, in which we decompose the partition function \eqref{f-def} not by rows, but by columns. 

In what follows, let $\AA \in \mathbb{N}^n$ be a composition. Our key object will be the column operator \index{Q@$Q_{1\dots n}(\AA)$; column operators} $Q_{1\dots n}(\AA) \in {\rm End}(W_1 \otimes \cdots \otimes W_n)$,\footnote{Throughout this section, we shall take $N$ (the number of factors in the tensor product of vector spaces $W_i$) equal to $n$, the rank of the vertex models in Sections \ref{ssec:fundamental} and \ref{ssec:models}.} whose matrix elements are given as one-column partition functions in the model \eqref{s-weights}. More precisely, as in \eqref{Rsig-comp} let $\ket{j_1,\dots,j_n}$  be a canonical basis state in $W_1 \otimes \cdots \otimes W_n$, and write the action of $Q_{1\dots n}(\AA)$ on this state as
\begin{align*}
Q_{1\dots n}(\AA)
\ket{j_1,\dots,j_n}
=
\sum_{i_1,\dots,i_n}
Q_{1\dots n}(i_1, \dots, i_n; \AA; j_1, \dots, j_n)
\ket{i_1,\dots,i_n}.
\end{align*}
Then we define the matrix elements 
$Q_{1\dots n}(i_1, \dots, i_n; \AA; j_1, \dots, j_n)$ to be given by
\begin{align}
Q_{1\dots n}(i_1, \dots, i_n; \AA; j_1, \dots, j_n)
=
\tikz{0.8}{
\foreach\y in {1,...,5}{
\draw[lgray,line width=1.5pt,->] (1,\y) -- (3,\y);
}
\foreach\x in {2}{
\draw[lgray,line width=4pt,->] (\x,0) -- (\x,6);
}
\node[left] at (0,1) {$x_1 \rightarrow$};
\node[left] at (0,2) {$x_2 \rightarrow$};
\node[left] at (0,3) {$\vdots$};
\node[left] at (0,4) {$\vdots$};
\node[left] at (0,5) {$x_n \rightarrow$};
\node[below] at (2,0) {\footnotesize$\bm{0}$};
\node[above] at (2,6) {\footnotesize$\bm{A}$};
\node[right] at (3,1) {$j_1$};
\node[right] at (3,2) {$j_2$};
\node[right] at (3,3) {$\vdots$};
\node[right] at (3,4) {$\vdots$};
\node[right] at (3,5) {$j_n$};
\node[left] at (1,1) {$i_1$};
\node[left] at (1,2) {$i_2$};
\node[left] at (1,3) {$\vdots$};
\node[left] at (1,4) {$\vdots$};
\node[left] at (1,5) {$i_n$};
}
\qquad
i_k, j_k \in \{0,1,\dots,n\},
\quad
\forall\ 1 \leq k \leq n.
\end{align}
In terms of these operators, we have
\begin{align}
\label{f-col}
f_{\mu}(x_1,\dots,x_n)
=
\Big\langle 1,\dots,n \Big|
\prod_{i \geq 0}^{\rightarrow}
Q_{1\dots n}\left(\AA(i)\right)
\Big| 0,\dots,0 \Big\rangle,
\end{align}
and more generally,
\begin{align}
\label{fsig-col}
f^{\sigma}_{\mu}(x_1,\dots,x_n)
=
\Big\langle \sigma(1),\dots,\sigma(n) \Big|
\prod_{i \geq 0}^{\rightarrow}
Q_{1\dots n}\left(\AA(i)\right)
\Big| 0,\dots,0 \Big \rangle,
\quad
\sigma \in \mathfrak{S}_n.
\end{align}
In both cases \eqref{f-col}, \eqref{fsig-col}, the index $i$ increases as one reads the products from left to right, and the compositions $\AA(i)$ are given by $\AA(i) = \sum_{k: \mu_k = i} \bm{e}_k$. These formulae follow immediately from the representation of $f_{\mu}$ and $f_{\mu}^{\sigma}$ as partition functions; for example, reading the partition function \eqref{f-def} from its left boundary edges and proceeding column-by-column to the right, one recovers precisely \eqref{f-col}. Equation \eqref{fsig-col} follows similarly, except that the left boundary states are now a permutation of $(1,\dots,n)$, which is reflected in the use of the covector $\bra{\sigma(1),\dots,\sigma(n)}$. 

\section{Monomial expansions}

We come to the main goal of this chapter, namely expanding the non-symmetric spin Hall--Littlewood functions in terms of the monomials \eqref{xi}.

\begin{thm}
\label{thm:fsig-sym}
Fix a permutation $\sigma \in \mathfrak{S}_n$ and a composition 
$\delta = (\delta_1 \leq \cdots \leq \delta_n)$. One has the expansion
\begin{align}
\label{fsig-sym}
f_{\delta}^{\sigma}(x_1,\dots,x_n)
=
\prod_{i \geq 0} (s^2;q)_{m_i(\delta)}
\sum_{\rho \in \mathfrak{S}_n}
\prod_{1 \leq a < b \leq n}
\left(
\frac{x_{\rho(b)} - q x_{\rho(a)}}{x_{\rho(b)} - x_{\rho(a)}}
\right)
Z_{\rho}^{\sigma}(x_1,\dots,x_n)
\xi_{\delta}\left(x_{\rho(1)},\dots,x_{\rho(n)}\right),
\end{align}
where \index{Z1@$Z_{\rho}^{\sigma}$} $Z_{\rho}^{\sigma}(x_1,\dots,x_n)$ is the partition function (depicted in Figure \ref{fig:Zgraph}) obtained from the following permutation graph:
\begin{itemize}
\item Draw two columns of $n$ nodes, the left column bearing the labels $\sigma(k)$ (with $k$ increasing from from bottom to top) and the right column bearing the labels $k$ (with $k$ increasing from top to bottom);
\item Connect the left node with label $\sigma(\rho(k))$ to the right node with label $k$, via a left-to-right oriented line with rapidity $x_{\rho(k)}$, for all $1 \leq k \leq n$;
\item The spectral parameter at each crossing is the ratio of rapidities of the two intersecting lines, with numerator corresponding to the one entering lower.
\end{itemize}
\end{thm}

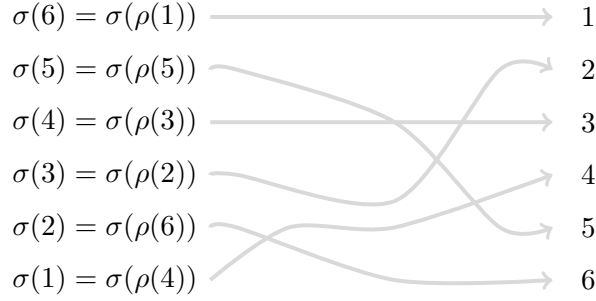
\begin{figure}
\begin{tikzpicture}[scale=0.7,baseline=(mid.base)]
\node (mid) at (1,4) {};
\draw[lgray,line width=1.5pt,->] plot [smooth] coordinates {(2,6.5)(8.5,6.5)};
\draw[lgray,line width=1.5pt,->] plot [smooth] coordinates {(2,5.5) (2.5,5.5) (5.5,4.5) (7.5,2.5) (8.5,2.5)};
\draw[lgray,line width=1.5pt,->] plot [smooth] coordinates {(2,4.5) (2.5,4.5) (5.5,4.5) (7.5,4.5) (8.5,4.5)};
\draw[lgray,line width=1.5pt,->] plot [smooth] coordinates {(2,3.5) (2.5,3.5) (5.5,3) (7.5,5.5) (8.5,5.5)};
\draw[lgray,line width=1.5pt,->] plot [smooth] coordinates {(2,2.5) (2.5,2.5) (5.5,1.5)(8.5,1.5)};
\draw[lgray,line width=1.5pt,->] plot [smooth] coordinates {(2,1.5) (3.5,2.5) (5.5,2.5) (8.5,3.5)};
\node at (0,6.5) {$\ss \sigma(6)=\sigma(\rho(1)) $}; 
\node at (0,5.5) {$\ss \sigma(5)=\sigma(\rho(5)) $}; 
\node at (0,4.5) {$\ss \sigma(4)=\sigma(\rho(3)) $};
\node at (0,3.5) {$\ss \sigma(3)=\sigma(\rho(2)) $}; 
\node at (0,2.5) {$\ss \sigma(2)=\sigma(\rho(6)) $};
\node at (0,1.5) {$\ss \sigma(1)=\sigma(\rho(4)) $};
\node at (9,6.5) {$\ss\ \ 1$};
\node at (9,5.5) {$\ss\ \ 2$};
\node at (9,4.5) {$\ss\ \ 3$};
\node at (9,3.5) {$\ss\ \ 4$};
\node at (9,2.5) {$\ss\ \ 5$};
\node at (9,1.5) {$\ss\ \ 6$};
\end{tikzpicture}
\caption{The partition function $Z_{\rho}^{\sigma}(x_1,\dots,x_6)$ in the case $\rho = (6,3,4,1,5,2)$.}
\label{fig:Zgraph}
\end{figure}

\begin{proof}
We start from the algebraic expression \eqref{fsig-col} for $f_{\delta}^{\sigma}(x_1,\dots,x_n)$ with $\delta = (\delta_1 \leq \cdots \leq \delta_n)$, and perform a series of elementary manipulations on it, which will turn it into an expansion in the monomials \eqref{xi}. Let us firstly insert the identity matrix $1 = F_{1\dots n}^{-1} F_{1\dots n}$ to the left of the first $Q_{1\dots n}$ operator. We obtain
\begin{align}
\nonumber
f^{\sigma}_{\delta}(x_1,\dots,x_n)
&=
\Big\langle \sigma(1),\dots,\sigma(n) \Big|
F^{-1}_{1\dots n}(x_1,\dots,x_n) F_{1\dots n}(x_1,\dots,x_n)
\prod_{i \geq 0}^{\rightarrow}
Q_{1\dots n}\left(\AA(i)\right)
\Big| 0,\dots,0 \Big \rangle,
\\
\nonumber
&=
\sum_{\rho \in \mathfrak{S}_n}
\langle \sigma(1),\dots,\sigma(n) |
F^{-1}_{1\dots n}(x_1,\dots,x_n)
| \rho^{-1}(1),\dots,\rho^{-1}(n) \rangle
\\
\label{second-exp}
&\times
\Big\langle \rho^{-1}(1),\dots,\rho^{-1}(n) \Big| 
F_{1\dots n}(x_1,\dots,x_n)
\prod_{i \geq 0}^{\rightarrow}
Q_{1\dots n}\left(\AA(i)\right)
\Big| 0,\dots,0 \Big \rangle,
\end{align}
where the second equality follows from a further resolution of the identity inserted between the $F$-matrix and its inverse.\footnote{Since the partition function $f_{\delta}^{\sigma}(x_1,\dots,x_n)$ has one copy of each of the colours $\{1,\dots,n\}$ entering via its left boundary, and given that $F^{-1}_{1\dots n}$ conserves colours (as can be seen from the form \eqref{Fcomp-2} of its components), it is clear that we can resolve the identity by summing over all permutations of $\{1,\dots,n\}$.
}
Next, we will simplify the second expectation value in \eqref{second-exp}. Using \eqref{Fcomp-1}, we may write
\begin{align}
\label{F-acts}
\langle \rho^{-1}(1),\dots,\rho^{-1}(n) |_{1\dots n} 
F_{1\dots n}(x_1,\dots,x_n)
=
\langle 1,\dots,n |_{\rho(1)\dots\rho(n)} 
R^{\rho}_{1\dots n},
\end{align}
where the two covectors appearing in this equation are simply rewritings of each other:
\begin{align*}
\langle \rho^{-1}(1),\dots,\rho^{-1}(n) |_{1\dots n}
\equiv
\bigotimes_{i=1}^{n}
\bra{\rho^{-1}(i)}_i
=
\bigotimes_{i=1}^{n}
\bra{i}_{\rho(i)}
\equiv
\langle 1,\dots,n |_{\rho(1)\dots\rho(n)}.
\end{align*}
Making use of equation \eqref{F-acts} in \eqref{second-exp}, we obtain
\begin{align}
\nonumber
f^{\sigma}_{\delta}(x_1,\dots,x_n)
&=
\sum_{\rho \in \mathfrak{S}_n}
\langle \sigma(1),\dots,\sigma(n) |
F^{-1}_{1\dots n}(x_1,\dots,x_n)
| \rho^{-1}(1),\dots,\rho^{-1}(n) \rangle
\\
\label{6.6-1}
&\times
\Big\langle 1,\dots,n \Big|_{\rho(1)\dots\rho(n)} 
\prod_{i \geq 0}^{\rightarrow}
Q_{\rho(1)\dots \rho(n)}\left(\AA(i);x_{\rho(1)},\dots,x_{\rho(n)}\right)
\Big| 0,\dots,0 \Big \rangle_{\rho(1)\dots \rho(n)},
\end{align}
where we have used the relation \eqref{RQ-QR} in Proposition \ref{prop:Q1N} to thread $R^{\rho}_{1\dots n}$ through the infinite product of 
$Q_{1\dots n}(\AA(i))$ operators, converting them into the permuted form $Q_{\rho(1)\dots\rho(n)}(\AA(i))$, and made use of the trivial relation
\begin{align*}
R^{\rho}_{1\dots n}
|0,\dots,0 \rangle_{1\dots n}
=
|0,\dots,0 \rangle_{\rho(1)\dots\rho(n)}
\end{align*}
to ultimately eliminate $R^{\rho}_{1\dots n}$. Now observe that (up to an irrelevant relabelling of the vector spaces which participate) the second expectation value in \eqref{6.6-1} takes the form of the function $f^{\rm id}_{\delta} \equiv f_{\delta}$. We have thus derived the identity
\begin{align}
\label{f-sym}
f^{\sigma}_{\delta}(x_1,\dots,x_n)
&=
\sum_{\rho \in \mathfrak{S}_n}
\langle \sigma(1),\dots,\sigma(n) |
F^{-1}_{1\dots n}(x_1,\dots,x_n)
| \rho^{-1}(1),\dots,\rho^{-1}(n) \rangle
\cdot
f_{\delta}\left(x_{\rho(1)},\dots,x_{\rho(n)}\right),
\end{align}
where $f_{\delta}$ is given by \eqref{f-delta} and has the explicit factorized form
\begin{align*}
f_{\delta}\left(x_{\rho(1)},\dots,x_{\rho(n)}\right)
=
\prod_{i \geq 0} (s^2;q)_{m_i(\delta)}
\cdot
\xi_{\delta}\left(x_{\rho(1)},\dots,x_{\rho(n)}\right).
\end{align*}
It remains to specify the surviving expectation value in \eqref{f-sym} more precisely. We can calculate it, using \eqref{R=Xrev} and \eqref{Fcomp-2}, in terms of a reversed permutation graph; the result is
\begin{multline*}
\langle \sigma(1),\dots,\sigma(n) |
F^{-1}_{1\dots n}(x_1,\dots,x_n)
| \rho^{-1}(1),\dots,\rho^{-1}(n) \rangle
\\
=
\prod_{1 \leq a < b \leq n}
b_{\rho^{-1}(a),\rho^{-1}(b)}(x_b/x_a)
\cdot
\mathcal{X}_{\rho \circ \omega}^{1 \dots n}
(\sigma(1),\dots,\sigma(n);
\rho^{-1}(1),\dots,\rho^{-1}(n)),
\end{multline*}
where $\omega := (n,\dots,1)$ is the longest permutation in $\mathfrak{S}_n$. This can be written in a slightly nicer form by noting that
\begin{align}
\prod_{1 \leq a < b \leq n}
b_{\rho^{-1}(a),\rho^{-1}(b)}(x_b/x_a)
=
\prod_{1 \leq a < b \leq n}
\frac{x_{\rho(b)} - q x_{\rho(a)}}{x_{\rho(b)} - x_{\rho(a)}},
\end{align}
which is an easy consequence of the definition \eqref{Delta}. Combining everything, and using the graphical interpretation of $\mathcal{X}_{\rho \circ \omega}^{1 \dots n}$, we obtain the result \eqref{fsig-sym}.

%
\end{proof}

Using \eqref{fsig-sym} in conjunction with the relation \eqref{g-fsigma}, one can now write down a monomial expansion for the function $g^{*}_{\mu}$. An analogous formula for $f_{\tilde\mu}$ can also be obtained using the relation \eqref{fmu-fsigma-2}. We include both formulae for completeness, although it is the expression for $g^{*}_{\mu}$ which will be most important to us in our subsequent proof of orthogonality statements.
\begin{cor}
\label{cor:mon-exp}
Fix a composition $\mu$, and let $\delta$ be the unique anti-dominant composition and $\sigma \in \mathfrak{S}_n$ the minimal-length permutation such that $\mu_{n-i+1} = \delta_{\sigma(i)}$ for all $1 \leq i \leq n$. We then have
\begin{align}
g^{*}_{\mu}(x_1,\dots,x_n)
=
\sum_{\rho \in \mathfrak{S}_n}
\prod_{1 \leq a < b \leq n}
\left(
\frac{x_{\rho(b)} - q x_{\rho(a)}}{x_{\rho(b)} - x_{\rho(a)}}
\right)
Z_{\tilde\rho}^{\sigma}(x_n,\dots,x_1;q)
\xi_{\delta}\left(x_{\rho(1)},\dots,x_{\rho(n)}\right),
\label{g-monom}
\end{align}
where we have defined the conjugated permutation $\tilde{\rho}$ with parts $\tilde{\rho}(i) := n-\rho(i)+1$. Similarly,
\begin{multline}
\label{f-monom}
f_{\tilde\mu}(x_1,\dots,x_n)
=
q^{{\rm inv}(\mu)}
\prod_{j \geq 0}
(s^2;q)_{m_j(\mu)}
\\
\times
\sum_{\rho \in \mathfrak{S}_n}
\prod_{1 \leq a < b \leq n}
\left(
\frac{x_{\rho(b)} - q x_{\rho(a)}}{x_{\rho(b)} - x_{\rho(a)}}
\right)
Z_{\rho}^{\sigma}(x^{-1}_1,\dots,x^{-1}_n;q^{-1})
\xi_{\delta}\left(x_{\rho(1)},\dots,x_{\rho(n)}\right),
\end{multline}
where $\tilde\mu = (\mu_n,\dots,\mu_1)$ and ${\rm inv}(\mu) = \#\{i < j : \mu_i \geq \mu_j\}$, as usual.

\end{cor}

\begin{rmk}
Note that equations \eqref{g-monom} and \eqref{f-monom} effectively separate the magnitudes of the parts of $\mu$ from their ordering; the magnitudes affect the $\xi$ monomials, while the ordering of the parts affects the partition function $Z$. 
\end{rmk}

Although \eqref{g-monom} and \eqref{f-monom} are explicit monomial expansions of $g^{*}_{\mu}$ and $f_{\tilde\mu}$, they continue to have a combinatorial flavour, due to the partition function present in their summands. Below we give some important results about this partition function.

%
\begin{prop}
Let $\sigma,\rho \in \mathfrak{S}_n$ be two permutations. Then one has
\begin{align}
\label{vanish-Z}
Z_{\rho}^{\sigma}(x_1,\dots,x_n)
=
0,
\qquad
\text{if}\ \ell(\rho) > \ell(\sigma);
\end{align}
\ie\ the partition function $Z_{\rho}^{\sigma}(x_1,\dots,x_n)$ vanishes if $\rho$ has more inversions than $\sigma$.
\end{prop}

\begin{proof}
We prove this by a simple inductive argument. In the case of the longest permutation $\rho = \omega = (n,\dots,1)$, the permutation graph $Z^{\sigma}_{\omega}(x_1,\dots,x_n)$ is trivial: it consists of $n$ non-crossing lines which directly connect left label $\sigma(n-i+1)$ to right label $i$, for $1 \leq i \leq n$. Since the lines do not cross anywhere, colour conservation leads to the constraint $\sigma = \omega$, or else $Z^{\sigma}_{\omega}(x_1,\dots,x_n)$ vanishes. Hence we see that \eqref{vanish-Z} holds in the case $\ell(\rho) = n(n-1)/2$.

Now assume that \eqref{vanish-Z} holds for all permutations $\rho$ such that $\ell(\rho) = L$, where $1 \leq L \leq n(n-1)/2$ is some fixed positive integer. Let $\rho'$ be a permutation of length $\ell(\rho') = L-1$. It is always possible to select integers $1 \leq i \leq n-1$ and $1 \leq j < k \leq n$ such that $\rho'(j) = i, \rho'(k) = i+1$ and $\rho' = \mathfrak{s}_i \circ \rho$, where $\rho$ is a permutation with length $\ell(\rho) = L$. 

We then consider the partition function $Z^{\sigma'}_{\rho'}(x_1,\dots,x_n)$, which can be subdivided as in Figure \ref{fig:subdivide}. The left half of this subdivision consists of a single vertex
\begin{align*}
\tikz{0.4}{
\draw[lgray,line width=1.5pt,->] (-1,0) -- (1,0);
\draw[lgray,line width=1.5pt,->] (0,-1) -- (0,1);
\node[left] at (-1,0) {\tiny $\sigma'(i+1)$};\node[right] at (1,0) {\tiny $\sigma(i)$};
\node[below] at (0,-1) {\tiny $\sigma'(i)$};\node[above] at (0,1) {\tiny $\sigma(i+1)$};
},
\end{align*}
while the right half is equal to $Z^{\sigma}_{\rho}(x_1,\dots,x_n)$, and by assumption vanishes unless $\ell(\sigma) \geq L$. Given that the edge states $\sigma'(1),\dots,\sigma'(n)$ and $\sigma(1),\dots,\sigma(n)$ are separated only by the above vertex, and it admits two possible cases, namely $\sigma'(i) = \sigma(i), \sigma'(i+1) = \sigma(i+1)$ or $\sigma'(i) = \sigma(i+1), \sigma'(i+1) = \sigma(i)$, we clearly have the bounds $\ell(\sigma')-1 \leq \ell(\sigma) \leq \ell(\sigma')+1$. It follows that 
$Z^{\sigma'}_{\rho'}(x_1,\dots,x_n)$ vanishes unless $\ell(\sigma') \geq L-1$, which is the required inductive step. We conclude that \eqref{vanish-Z} holds generally.

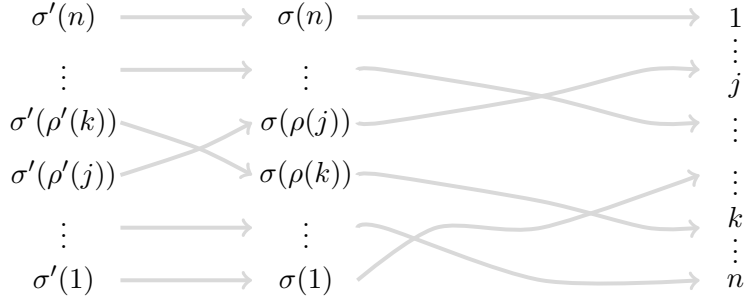
\begin{figure}
\begin{tikzpicture}[scale=0.7,baseline=(mid.base)]
\node (mid) at (1,4) {};
%
\draw[lgray,line width=1.5pt,->] plot [smooth] coordinates {(-2.5,6.5)(0,6.5)};
\draw[lgray,line width=1.5pt,->] plot [smooth] coordinates {(-2.5,5.5) (0,5.5)};
\draw[lgray,line width=1.5pt,->] plot [smooth] coordinates {(-2.5,4.5) (-1,4) (0,3.5)};
\draw[lgray,line width=1.5pt,->] plot [smooth] coordinates {(-2.5,3.5) (-1,4) (0,4.5)};
\draw[lgray,line width=1.5pt,->] plot [smooth] coordinates {(-2.5,2.5) (0,2.5)};
\draw[lgray,line width=1.5pt,->] plot [smooth] coordinates {(-2.5,1.5) (0,1.5)};
\draw[lgray,line width=1.5pt,->] plot [smooth] coordinates {(2,6.5)(8.5,6.5)};
\draw[lgray,line width=1.5pt,->] plot [smooth] coordinates {(2,5.5) (2.5,5.5) (5.5,5) (7.5,4.5) (8.5,4.5)};
\draw[lgray,line width=1.5pt,->] plot [smooth] coordinates {(2,4.5) (2.5,4.5) (5.5,5) (7.5,5.5) (8.5,5.5)};
\draw[lgray,line width=1.5pt,->] plot [smooth] coordinates {(2,3.5) (2.5,3.5) (5.5,3) (7.5,2.5) (8.5,2.5)};
\draw[lgray,line width=1.5pt,->] plot [smooth] coordinates {(2,2.5) (2.5,2.5) (5.5,1.5)(8.5,1.5)};
\draw[lgray,line width=1.5pt,->] plot [smooth] coordinates {(2,1.5) (3.5,2.5) (5.5,2.5) (8.5,3.5)};
\node at (-3.3,6.5) {$\ss \sigma'(n)\ \ \ $}; 
\node at (-3.3,5.5) {$\ss \vdots\ \ \ $}; 
\node at (-3.3,4.5) {$\ss \sigma'(\rho'(k))\ \ \ $};
\node at (-3.3,3.5) {$\ss \sigma'(\rho'(j))\ \ \ $}; 
\node at (-3.3,2.5) {$\ss \vdots\ \ \ $};
\node at (-3.3,1.5) {$\ss \sigma'(1)\ \ \ $};
\node at (1.3,6.5) {$\ss \sigma(n)\ \ \ $}; 
\node at (1.3,5.5) {$\ss \vdots\ \ \ $}; 
\node at (1.3,4.5) {$\ss \sigma(\rho(j))\ \ \ $};
\node at (1.3,3.5) {$\ss \sigma(\rho(k))\ \ \ $}; 
\node at (1.3,2.5) {$\ss \vdots\ \ \ $};
\node at (1.3,1.5) {$\ss \sigma(1)\ \ \ $};
\node at (9,6.5) {$\ss\ \ 1$};
\node at (9,6) {$\ss\ \ \vdots$};
\node at (9,5.2) {$\ss\ \ j$};
\node at (9,4.5) {$\ss\ \ \vdots$};
\node at (9,3.5) {$\ss\ \ \vdots$};
\node at (9,2.7) {$\ss\ \ k$};
\node at (9,2.2) {$\ss\ \ \vdots$};
\node at (9,1.5) {$\ss\ \ n$};
\end{tikzpicture}
\caption{The partition function $Z_{\rho'}^{\sigma'}(x_1,\dots,x_n)$, explicitly showing the decomposition $\rho' = \mathfrak{s}_i \circ \rho$, with summation taken over the intermediate states $\sigma(1),\dots,\sigma(n)$.}
\label{fig:subdivide}
\end{figure}

\end{proof}

\begin{prop}
\label{prop:singular}
All poles of $Z_{\tilde\rho}^{\sigma}(x_n,\dots,x_1)$ are known explicitly; one has
\begin{align}
\label{singular-structure}
Z_{\tilde\rho}^{\sigma}(x_n,\dots,x_1)
=
\prod_{a<b: \rho(a) > \rho(b)}
\frac{1}{x_{\rho(b)} - q x_{\rho(a)}}
\,
\mathcal{P}_{\rho}^{\sigma}(x_1,\dots,x_n),
\end{align}
where $\mathcal{P}_{\rho}^{\sigma}(x_1,\dots,x_n)$ is a homogeneous polynomial of total degree $\ell(\rho)$. The degree of $\mathcal{P}_{\rho}^{\sigma}(x_1,\dots,x_n)$ in the individual variable $x_i$ is sufficiently small such that $\lim_{x_i \rightarrow \infty} Z_{\tilde\rho}^{\sigma}(x_n,\dots,x_1)$ exists, for all $1 \leq i \leq n$.

\end{prop}

\begin{proof}
Given that we have exactly one copy of each of the colours $\{1,\dots,n\}$ entering/leaving the $n$ lines which constitute $Z_{\tilde\rho}^{\sigma}(x_n,\dots,x_1)$, there cannot be any vertices of the form
\tikz{0.3}{
\draw[lgray,line width=1.5pt,->] (-1,0) -- (1,0);
\draw[lgray,line width=1.5pt,->] (0,-1) -- (0,1);
\node[left] at (-1,0) {\tiny $i$};\node[right] at (1,0) {\tiny $i$};
\node[below] at (0,-1) {\tiny $i$};\node[above] at (0,1) {\tiny $i$};
}
for all $0 \leq i \leq n$. All vertices are therefore of the form \eqref{R-weights-bc}, which have a common denominator. It follows that, for each pair of integers $1\leq i<j \leq n$, $Z_{\tilde\rho}^{\sigma}(x_n,\dots,x_1)$ contains a common factor of 
$(1-qx_j/x_i)^{-1}$ if and only if the lines bearing the rapidities $x_i$, $x_j$ cross. The crossings of the lines are easy to determine; for $a<b$ the lines with rapidities $x_{n-\tilde\rho(a)+1}$ and $x_{n-\tilde\rho(b)+1}$ cross if and only if $\tilde\rho(a) < \tilde\rho(b)$, or said another way, the lines with rapidities $x_{\rho(a)}$ and $x_{\rho(b)}$ cross if and only if $\rho(a) > \rho(b)$. The statement \eqref{singular-structure} follows immediately.

The second part of the proposition, namely the degree statement, follows from the fact that individual vertex weights within the lattice satisfy such a limiting criterion; the whole partition function must also, therefore, satisfy it.
\end{proof}

The final result of this chapter gives a representation of the coefficients $Z^{\sigma}_{\tilde\rho}$ as partition functions of ``domain-wall'' type:
\begin{prop}
\label{prop:Z-sigma}
Let $(x_1,\dots,x_n)$, $(y_1,\dots,y_n)$ be two sets of variables, and fix a permutation $\sigma\in \mathfrak{S}_n$. We define the following partition function in the model \eqref{fund-vert}:
\begin{align}
\label{Z-sigma}
Z^{\sigma}(x_n,\dots,x_1;y_1,\dots,y_n)
:=
\tikz{0.7}{
\foreach\y in {1,...,5}{
\draw[lgray,line width=1.5pt,->] (0,\y) -- (6,\y);
}
\foreach\x in {1,...,5}{
\draw[lgray,line width=1.5pt,->] (\x,0) -- (\x,6);
}
\node[below] at (1,0) {\scriptsize $\sigma(n)$};
\node[below] at (2,0) {\scriptsize $\cdots$};
\node[below] at (3,0) {\scriptsize $\cdots$};
\node[below] at (4,0) {\scriptsize $\sigma(2)$};
\node[below] at (5,0) {\scriptsize $\sigma(1)$};
\node[above] at (1,6) {\scriptsize $0$};
\node[above] at (2,6) {\scriptsize $\cdots$};
\node[above] at (3,6) {\scriptsize $\cdots$};
\node[above] at (4,6) {\scriptsize $0$};
\node[above] at (5,6) {\scriptsize $0$};
\node[left] at (0,1) {\scriptsize $0$};
\node[left] at (0,2) {\scriptsize $\vdots$};
\node[left] at (0,3) {\scriptsize $\vdots$};
\node[left] at (0,4) {\scriptsize $0$};
\node[left] at (0,5) {\scriptsize $0$};
\node[right] at (6,1) {\scriptsize $n$};
\node[right] at (6,2) {\scriptsize $\vdots$};
\node[right] at (6,3) {\scriptsize $\vdots$};
\node[right] at (6,4) {\scriptsize $2$};
\node[right] at (6,5) {\scriptsize $1$};
\node[left] at (-1,1) {$y_n$};
\node[left] at (-1,2) {$\vdots$};
\node[left] at (-1,3) {$\vdots$};
\node[left] at (-1,4) {$y_2$};
\node[left] at (-1,5) {$y_1$};
\node[below] at (1,-1) {$x_1$};
\node[below] at (2,-1) {$\cdots$};
\node[below] at (3,-1) {$\cdots$};
\node[below] at (4,-1) {$x_{n-1}$};
\node[below] at (5,-1) {$x_n$};
}
\end{align}
in which the $i$-th horizontal line (counted from the top) carries rapidity $y_i$, left external edge state $0$ and right external edge state $i$, while the $j$-th vertical line (counted from the left) carries rapidity $x_j$, bottom external edge state 
$\sigma(n-j+1)$ and top external edge state $0$; the spectral parameter of each vertex is the ratio of the vertical and horizontal rapidities. Fixing a further permutation $\rho \in \mathfrak{S}_n$, one then has the relation
\begin{align}
\label{square}
Z_{\tilde\rho}^{\sigma}(x_n,\dots,x_1)
=
Z^{\sigma}\left(x_n,\dots,x_1;y_1 = x_{\rho(1)},\dots,y_n = x_{\rho(n)}\right). 
\end{align}
\end{prop}

\begin{proof}
The proof makes use of a basic property of the $R$-matrix \eqref{Rmat}; namely, when its spectral parameter is set to $1$, it reduces to a permutation matrix. More precisely, analyzing the matrix entries \eqref{R-weights-a} and \eqref{R-weights-bc}, we see that both $R_z(j,i;j,i)$ and $R_z(i,j;i,j)$ vanish at $z=1$, while all remaining entries assume the value $1$ at $z=1$. From a graphical point of view, this can be understood as a ``splitting'' of the vertex:
\begin{align}
\label{R-split}
R_{1}(i,j; k,\ell)
=
{\bm 1}_{i=\ell}
\cdot
{\bm 1}_{j=k}
=
\tikz{0.7}{
\draw[lgray,line width=1.5pt,->] plot coordinates {(-1,0)(-0.25,0)(0,0.25)(0,1)};
\draw[lgray,line width=1.5pt,->] plot coordinates {(0,-1)(0,-0.25)(0.25,0)(1,0)};
\node[left] at (-1,0) {\tiny $j$};\node[right] at (1,0) {\tiny $\ell$};
\node[below] at (0,-1) {\tiny $i$};\node[above] at (0,1) {\tiny $k$};
},
\quad
i,j,k,\ell \in \{0,1,\dots,n\}.
\end{align}
Applying the specializations listed in \eqref{square} to the partition function \eqref{Z-sigma}, one sees that they induce a splitting of the vertex at the intersection of the $i$-th horizontal line and the $\rho(i)$-th vertical line, for all $1 \leq i \leq n$:
\begin{align}
\label{Z-sigma-split}
Z^{\sigma}\left(x_n,\dots,x_1;x_{\rho(1)},\dots,x_{\rho(n)}\right)
=
\tikz{0.7}{
\draw[lgray,line width=1.5pt,->] (0,5) -- (3.75,5) -- (4,5.25) -- (4,6);
\draw[lgray,line width=1.5pt,->] (0,4) -- (0.75,4) -- (1,4.25) -- (1,6);
\draw[lgray,line width=1.5pt,->] (0,3) -- (4.75,3) -- (5,3.25) -- (5,6);
\draw[lgray,line width=1.5pt,->] (0,2) -- (1.75,2) -- (2,2.25) -- (2,6);
\draw[lgray,line width=1.5pt,->] (0,1) -- (2.75,1) -- (3,1.25) -- (3,6);
\draw[lgray,line width=1.5pt,->] (1,0) -- (1,3.75) -- (1.25,4) -- (6,4);
\draw[lgray,line width=1.5pt,->] (2,0) -- (2,1.75) -- (2.25,2) -- (6,2);
\draw[lgray,line width=1.5pt,->] (3,0) -- (3,0.75) -- (3.25,1) -- (6,1);
\draw[lgray,line width=1.5pt,->] (4,0) -- (4,4.75) -- (4.25,5) -- (6,5);
\draw[lgray,line width=1.5pt,->] (5,0) -- (5,2.75) -- (5.25,3) -- (6,3);
\node[below] at (1,0) {\scriptsize $\sigma(n)$};
\node[below] at (2,0) {\scriptsize $\cdots$};
\node[below] at (3,0) {\scriptsize $\cdots$};
\node[below] at (4,0) {\scriptsize $\sigma(2)$};
\node[below] at (5,0) {\scriptsize $\sigma(1)$};
\node[above] at (1,6) {\scriptsize $0$};
\node[above] at (2,6) {\scriptsize $\cdots$};
\node[above] at (3,6) {\scriptsize $\cdots$};
\node[above] at (4,6) {\scriptsize $0$};
\node[above] at (5,6) {\scriptsize $0$};
\node[left] at (0,1) {\scriptsize $0$};
\node[left] at (0,2) {\scriptsize $\vdots$};
\node[left] at (0,3) {\scriptsize $\vdots$};
\node[left] at (0,4) {\scriptsize $0$};
\node[left] at (0,5) {\scriptsize $0$};
\node[right] at (6,1) {\scriptsize $n$};
\node[right] at (6,2) {\scriptsize $\vdots$};
\node[right] at (6,3) {\scriptsize $\vdots$};
\node[right] at (6,4) {\scriptsize $2$};
\node[right] at (6,5) {\scriptsize $1$};
\node[left] at (-1,1) {$x_{\rho(n)}$};
\node[left] at (-1,2) {$\vdots$};
\node[left] at (-1,3) {$\vdots$};
\node[left] at (-1,4) {$x_{\rho(2)}$};
\node[left] at (-1,5) {$x_{\rho(1)}$};
\node[below] at (1,-1) {$x_1$};
\node[below] at (2,-1) {$\cdots$};
\node[below] at (3,-1) {$\cdots$};
\node[below] at (4,-1) {$x_{n-1}$};
\node[below] at (5,-1) {$x_n$};
}
\end{align}
The split lattice \eqref{Z-sigma-split} can then be factorized into two disjoint pieces, by making repeated use of the unitarity relation \eqref{unitarity}. Indeed, as long as the beginning and endpoint of each line is held fixed, the braid relations \eqref{YB}, \eqref{unitarity} allow us to freely reposition the lines in \eqref{Z-sigma-split}. We thus read off the factorization
\begin{align}
\label{Z-sigma-split2}
Z^{\sigma}\left(x_n,\dots,x_1;x_{\rho(1)},\dots,x_{\rho(n)}\right)
=
\tikz{0.7}{
\draw[lgray,line width=1.5pt,->] (-6,5) -- (-2.25,5) -- (-2,5.25) -- (-2,6);
\draw[lgray,line width=1.5pt,->] (-6,4) -- (-5.25,4) -- (-5,4.25) -- (-5,6);
\draw[lgray,line width=1.5pt,->] (-6,3) -- (-1.25,3) -- (-1,3.25) -- (-1,6);
\draw[lgray,line width=1.5pt,->] (-6,2) -- (-4.25,2) -- (-4,2.25) -- (-4,6);
\draw[lgray,line width=1.5pt,->] (-6,1) -- (-3.25,1) -- (-3,1.25) -- (-3,6);
\node at (0,3) {$\times$};
\draw[lgray,line width=1.5pt,->] (1,0) -- (1,3.75) -- (1.25,4) -- (6,4);
\draw[lgray,line width=1.5pt,->] (2,0) -- (2,1.75) -- (2.25,2) -- (6,2);
\draw[lgray,line width=1.5pt,->] (3,0) -- (3,0.75) -- (3.25,1) -- (6,1);
\draw[lgray,line width=1.5pt,->] (4,0) -- (4,4.75) -- (4.25,5) -- (6,5);
\draw[lgray,line width=1.5pt,->] (5,0) -- (5,2.75) -- (5.25,3) -- (6,3);
\node[below] at (1,0) {\scriptsize $\sigma(n)$};
\node[below] at (2,0) {\scriptsize $\cdots$};
\node[below] at (3,0) {\scriptsize $\cdots$};
\node[below] at (4,0) {\scriptsize $\sigma(2)$};
\node[below] at (5,0) {\scriptsize $\sigma(1)$};
\node[above] at (-5,6) {\scriptsize $0$};
\node[above] at (-4,6) {\scriptsize $\cdots$};
\node[above] at (-3,6) {\scriptsize $\cdots$};
\node[above] at (-2,6) {\scriptsize $0$};
\node[above] at (-1,6) {\scriptsize $0$};
\node[left] at (-6,1) {\scriptsize $0$};
\node[left] at (-6,2) {\scriptsize $\vdots$};
\node[left] at (-6,3) {\scriptsize $\vdots$};
\node[left] at (-6,4) {\scriptsize $0$};
\node[left] at (-6,5) {\scriptsize $0$};
\node[right] at (6,1) {\scriptsize $n$};
\node[right] at (6,2) {\scriptsize $\vdots$};
\node[right] at (6,3) {\scriptsize $\vdots$};
\node[right] at (6,4) {\scriptsize $2$};
\node[right] at (6,5) {\scriptsize $1$};
\node[left] at (-7,1) {$x_{\rho(n)}$};
\node[left] at (-7,2) {$\vdots$};
\node[left] at (-7,3) {$\vdots$};
\node[left] at (-7,4) {$x_{\rho(2)}$};
\node[left] at (-7,5) {$x_{\rho(1)}$};
\node[below] at (1,-1) {$x_1$};
\node[below] at (2,-1) {$\cdots$};
\node[below] at (3,-1) {$\cdots$};
\node[below] at (4,-1) {$x_{n-1}$};
\node[below] at (5,-1) {$x_n$};
}
\end{align}
The partition function on the left has no colours flowing through it at all, since only $0$ states enter and exit; accordingly, it factorizes into a product of the vertices \eqref{R-weights-a}, and has weight $1$. The remaining partition function is precisely $Z_{\tilde\rho}^{\sigma}(x_n,\dots,x_1)$ (note that the list of variables is reversed).

\end{proof}

\chapter{Monomial expansions: degenerations of nested Bethe vectors}

The nested Bethe Ansatz is a generalization of the standard algebraic Bethe Ansatz to the setting of higher-rank spin-chains. Just as in the case of the XXZ spin-$\frac{1}{2}$ Heisenberg chain (based on $U_q(\wh{\mathfrak{sl}_{2}})$), one of the primary goals in the $U_q(\wh{\mathfrak{sl}_{n+1}})$ Heisenberg spin-chain is to diagonalize the transfer matrix of the model. This is done in three basic steps: {\bf 1.} Identifying a family of monodromy matrix operators which satisfy the Yang--Baxter algebra; {\bf 2.} Constructing states within the physical space of the model, called {\it Bethe vectors}, which are obtained via the action of the monodromy matrix operators on a highest weight vector; {\bf 3.} Proving that the Bethe vectors thus constructed are genuine eigenstates of the transfer matrix. This is done by means of the commutation relations of the Yang--Baxter algebra, and by assuming that certain variables which parametrize the Bethe vectors satisfy the {\it Bethe equations}.

The goal of this chapter is to study the basic combinatorial structure of the nested Bethe vectors, and to give explicit symmetrization formulae for their components, the {\it wavefunctions}. Our motivation is to show that the non-symmetric spin Hall--Littlewood functions \eqref{f-HL} can be recovered as a reduction of certain specific wavefunctions. While interesting in its own right, we also inherit from this correspondence an explicit symmetrization-type identity for the functions $f_{\mu}$. Our approach is somewhat tangential to the usual nested Bethe Ansatz, since we neither make reference to the transfer matrix of the $U_q(\wh{\mathfrak{sl}_{n+1}})$ spin-chain, nor to the Bethe equations.

For more information on the nested Bethe Ansatz and the algebraic structure of the Bethe vectors, we refer the reader to \cite{BelliardR,TarasovV}.

\section{Bethe vector blocks}
\label{ssec:blocks}

We begin by introducing certain partition functions in the model \eqref{fund-vert} which we refer to as {\it blocks}. These partition functions are used in the construction of the components of the nested Bethe vectors. In what follows we fix $n+1$ nonnegative integers $\{N_0,N_1,\dots,N_n\}$ which satisfy $N_0 \geq N_1 \geq \cdots \geq N_n$, and $n+1$ sets of complex variables $\{x^{(0)},x^{(1)},\dots,x^{(n)}\}$ with cardinalities given by $|x^{(i)}| = N_i$. For all $0 \leq i \leq n-1$, the block \index{Z2@$Z^{(i)}$; Bethe vector block} $Z^{(i)}$ is defined as follows:
\begin{align}
\label{Zblock}
Z^{(i)} \left( x^{(i)}/x^{(i+1)} ; k^{(i+1)}, k^{(i)} \right)
:=
\tikz{0.6}{
\foreach\y in {2,...,5}{
\draw[lgray,line width=1.5pt,->] (1,\y) -- (8,\y);
}
\foreach\x in {2,...,7}{
\draw[lgray,line width=1.5pt,->] (\x,1) -- (\x,6);
}
\node[above] at (7,6) {$\tiny k^{(i)}_{N_i}$};
\node[above] at (5,6) {$\tiny \cdots$};
\node[above] at (4,6) {$\tiny \cdots$};
\node[above] at (2,6) {$\tiny k^{(i)}_1$};
\node[left] at (-1,2) {$\tiny x^{(i+1)}_{N_{i+1}} \rightarrow$}; 
\node[left] at (1,2) {$\tiny k^{(i+1)}_{N_{i+1}}$};
\node[left] at (0.5,3.3) {$\tiny \vdots$};
\node[left] at (0.5,4.3) {$\tiny \vdots$};
\node[left] at (-1,5) {$\tiny x^{(i+1)}_{1} \rightarrow$};
\node[left] at (1,5) {$\tiny k^{(i+1)}_{1}$};
\node[below] at (7,-0.5) {$\tiny x^{(i)}_{N_i}$};
\node at (7,-0.3) {$\tiny \uparrow$};
\node[below] at (7,1) {$\tiny i$};
\node[below] at (5,1) {$\tiny \cdots$};
\node[below] at (4,1) {$\tiny \cdots$};
\node[below] at (2,-0.5) {$\tiny x^{(i)}_1$};
\node at (2,-0.3) {$\tiny \uparrow$};
\node[below] at (2,1) {$\tiny i$};
\node[right] at (8,2) {$\tiny i$};
\node[right] at (8,3.3) {$\tiny \vdots$};
\node[right] at (8,4.3) {$\tiny \vdots$};
\node[right] at (8,5) {$\tiny i$};
}
\end{align}
where the vectors $k^{(j)} = \left(k^{(j)}_1,\dots,k^{(j)}_{N_j}\right)$, $j \in \{0,1,\dots,n\}$ label the states at the left and top edges of the lattice, and it is assumed that
\begin{align}
\label{k-restrict}
k^{(j)}_a \in \{j,j+1,\dots,n\},
\qquad
\forall\ 1 \leq a \leq N_j.
\end{align} 
The $a$-th horizontal line (counted from the top) carries the rapidity $x^{(i+1)}_a$, while the $b$-th vertical line (counted from the left) comes with rapidity $x^{(i)}_b$. The spectral parameter assigned to the vertex at the intersection of the $a$-th horizontal and $b$-th vertical line is assumed to be $x^{(i)}_b/x^{(i+1)}_a$.

Let us introduce a further set of indices $m^{(j)}_i$ with $i,j \in \{0,1,\dots,n\}$, which are the colour multiplicities of the vector $k^{(j)}$:
\begin{align*}
m_i^{(j)}
=
\#\{a : k^{(j)}_a = i\}.
\end{align*}
In view of the restriction \eqref{k-restrict} we clearly have $m_i^{(j)} = 0$ for all $i<j$, while colour-conservation through the lattice \eqref{Zblock} imposes the relations
\begin{align}
\label{m-restrict}
m_i^{(i)} = N_i - N_{i+1},
\quad\quad
m_i^{(j)} = m_i^{(j+1)},\ \ i > j.
\end{align}
The conditions \eqref{m-restrict} are necessary to ensure that the partition function \eqref{Zblock} is non-vanishing.

\begin{prop}
\label{prop:block-reduce}
After setting $x^{(i+1)}_a = x^{(i)}_a$ for all $1 \leq a \leq N_{i+1}$, 
the block $Z^{(i)}$ has the following reductive property:
\begin{align}
\label{block-reduce}
Z^{(i)}
\Big|_{x_1^{(i+1)} = x_1^{(i)}}
\cdots
\Big|_{x_{N_{i+1}}^{(i+1)} = x_{N_{i+1}}^{(i)}}
=
\prod_{a=1}^{N_{i+1}}
\left(
{\bm 1}_{k^{(i)}_a = k^{(i+1)}_a}
\right)
\prod_{N_{i+1} < b \leq N_i}
\left(
{\bm 1}_{k^{(i)}_b = i}
\right).
\end{align}
\end{prop}

\begin{proof}
We shall again employ the ``splitting'' property of an $R$-vertex when its spectral parameter is set to $1$. Making use of \eqref{R-split} in the definition \eqref{Zblock} of $Z^{(i)}$, we see that
\begin{align}
\label{Z-split}
Z^{(i)}
\Big|_{x_1^{(i+1)} = x_1^{(i)}}
\cdots
\Big|_{x_{N_{i+1}}^{(i+1)} = x_{N_{i+1}}^{(i)}}
=
\tikz{0.6}{
\foreach\y in {2,...,5}{
\draw[lgray,line width=1.5pt,->] plot coordinates 
{(1,\y) (7-\y-0.25,\y) (7-\y,\y+0.25) (7-\y,6)};
}
\foreach\x in {2,...,5}{
\draw[lgray,line width=1.5pt,->] plot coordinates
{(\x,1) (\x,7-\x-0.25) (\x+0.25,7-\x) (8,7-\x)};
}
\foreach\x in {6,...,7}{
\draw[lgray,line width=1.5pt,->] (\x,1) -- (\x,6);
}
\node[above] at (7,6) {$\tiny k^{(i)}_{N_i}$};
\node[above] at (5,6) {$\tiny \cdots$};
\node[above] at (4,6) {$\tiny \cdots$};
\node[above] at (2,6) {$\tiny k^{(i)}_1$};
\node[left] at (-1,2) {$\tiny x^{(i)}_{N_{i+1}} \rightarrow$}; 
\node[left] at (1,2) {$\tiny k^{(i+1)}_{N_{i+1}}$};
\node[left] at (0.5,3.3) {$\tiny \vdots$};
\node[left] at (0.5,4.3) {$\tiny \vdots$};
\node[left] at (-1,5) {$\tiny x^{(i)}_{1} \rightarrow$};
\node[left] at (1,5) {$\tiny k^{(i+1)}_{1}$};
\node[below] at (7,-0.5) {$\tiny x^{(i)}_{N_i}$};
\node at (7,-0.3) {$\tiny \uparrow$};
\node[below] at (7,1) {$\tiny i$};
\node[below] at (5,1) {$\tiny \cdots$};
\node[below] at (4,1) {$\tiny \cdots$};
\node[below] at (2,-0.5) {$\tiny x^{(i)}_1$};
\node at (2,-0.3) {$\tiny \uparrow$};
\node[below] at (2,1) {$\tiny i$};
\node[right] at (8,2) {$\tiny i$};
\node[right] at (8,3.3) {$\tiny \vdots$};
\node[right] at (8,4.3) {$\tiny \vdots$};
\node[right] at (8,5) {$\tiny i$};
}
\end{align}
where the vertex splitting happens at each of the $N_{i+1}$ locations where horizontal and vertical rapidities coincide. Now by successive applications of the unitarity relation \eqref{unitarity} of the model, it is easy to see that all lines which form the shape $\tikz{0.4}{\draw[lgray,line width=1.5pt] (-1,0) -- (-0.25,0) -- (0,0.25) -- (0,1);}$ can be dissociated entirely from lines with the shape 
$\tikz{0.4}{\draw[lgray,line width=1.5pt] (0,-1) -- (0,-0.25) -- (0.25,0) -- (1,0);}$. The partition function \eqref{Z-split} can thus be factorized into two pieces:
\begin{align*}
Z^{(i)}
\Big|_{x_1^{(i+1)} = x_1^{(i)}}
\cdots
\Big|_{x_{N_{i+1}}^{(i+1)} = x_{N_{i+1}}^{(i)}}
=
\tikz{0.6}{
\foreach\y in {2,...,5}{
\draw[lgray,line width=1.5pt,->] plot coordinates 
{(1,\y) (7-\y-0.25,\y) (7-\y,\y+0.25) (7-\y,6)};
}
\node at (6.5,4) {$\times$};
\foreach\x in {2,...,5}{
\draw[lgray,line width=1.5pt,->] plot coordinates
{(\x+6,1) (\x+6,7-\x-0.25) (\x+6+0.25,7-\x) (8+6,7-\x)};
}
\foreach\x in {6,...,7}{
\draw[lgray,line width=1.5pt,->] (\x+6,1) -- (\x+6,6);
}
\node[above] at (12.5,6) {$\overbrace{k^{(i)}_b}^{N_{i+1} < b \leq N_i}$};
\node[above] at (5,6) {$\tiny k^{(i)}_{N_{i+1}}$};
\node[above] at (3.5,6) {$\tiny \cdots$};
\node[above] at (2,6) {$\tiny k^{(i)}_1$};
\node[left] at (1,2) {$\tiny k^{(i+1)}_{N_{i+1}}$};
\node[left] at (0.5,3.3) {$\tiny \vdots$};
\node[left] at (0.5,4.3) {$\tiny \vdots$};
\node[left] at (1,5) {$\tiny k^{(i+1)}_{1}$};
\node[below] at (7+6,-0.5) {$\tiny x^{(i)}_{N_i}$};
\node at (7+6,-0.3) {$\tiny \uparrow$};
\node[below] at (7+6,1) {$\tiny i$};
\node[below] at (5+6,1) {$\tiny \cdots$};
\node[below] at (4+6,1) {$\tiny \cdots$};
\node[below] at (2+6,-0.5) {$\tiny x^{(i)}_1$};
\node at (2+6,-0.3) {$\tiny \uparrow$};
\node[below] at (2+6,1) {$\tiny i$};
\node[right] at (8+6,2) {$\tiny i$};
\node[right] at (8+6,3.3) {$\tiny \vdots$};
\node[right] at (8+6,4.3) {$\tiny \vdots$};
\node[right] at (8+6,5) {$\tiny i$};
}
\end{align*}
The first of these pieces just enforces the condition that $k^{(i)}_a = k^{(i+1)}_a$ for all $1 \leq a \leq N_{i+1}$, while the second piece is a partition function whose incoming states are all equal to $i$. By colour conservation arguments, this constrains all outgoing states of the second lattice to be equal to $i$, and the resulting partition function is frozen into a product of vertices \eqref{R-weights-a} of weight $1$. Collecting all of these facts together, we recover the right hand side of \eqref{block-reduce}.

\end{proof}

\section{Bethe vector components}

In this section we give the generic form of the 
nested Bethe vectors of the $U_q(\wh{\mathfrak{sl}_{n+1}})$ Heisenberg spin-chain of length $N_0$. They depend on $n$ sets of ``auxiliary'' rapidities $\{x^{(1)},\dots,x^{(n)}\}$, with cardinalities $|x^{(i)}| = N_i$, as well as on a set of ``quantum'' or ``inhomogeneity'' parameters $x^{(0)} \equiv \left( y_1,\dots,y_{N_0} \right)$. Fixing a vector of nonnegative integers $k^{(0)} = \left(k^{(0)}_1,\dots,k^{(0)}_{N_0} \right) \in \{0,1,\dots,n\}^{N_0}$, the components of the nested Bethe vectors are given by
\begin{multline}
\label{bethe-comp}
\index{aY@$\Psi(x^{(1)},\dots,x^{(n)} ; k^{(0)})$; Bethe vector components}
\Psi \left( x^{(1)},\dots,x^{(n)} ; k^{(0)}\right)
=
\sum_{k^{(1)}}
\cdots
\sum_{k^{(n)}}
\\
Z^{(n-1)} \left( x^{(n-1)} / x^{(n)} ; k^{(n)}, k^{(n-1)} \right)
\cdots
Z^{(1)} \left( x^{(1)} / x^{(2)} ; k^{(2)}, k^{(1)} \right)
Z^{(0)} \left( x^{(0)} / x^{(1)} ; k^{(1)}, k^{(0)} \right)
\end{multline}
where the summation is over vectors $k^{(i)} \in \{i,i+1,\dots,n\}^{N_i}$, and each block $Z^{(i)}$ is given by \eqref{Zblock}. Note that the final sum over $k^{(n)}$ is in fact trivial: in view of the restriction \eqref{k-restrict} placed on the elements of the vectors $k^{(i)}$, one sees that $k^{(n)} = (n,\dots,n)$ necessarily. Furthermore, by virtue of the restrictions \eqref{m-restrict} one finds that the colour multiplicities of the vectors $k^{(0)},k^{(1)},\dots,k^{(n-1)}$ are completely determined, and in particular one has
\begin{align}
\label{col-mult}
m^{(0)}_i = N_i-N_{i+1},
\qquad
i \in \{0,1,\dots,n\}.
\end{align}
It is easy to translate the expression \eqref{bethe-comp} into graphical form. Indeed, by replacing each block $Z^{(i)}$ appearing in \eqref{bethe-comp} by its partition function representation \eqref{Zblock}, we arrive at a picture of the type shown in Figure \ref{fig:bethe}.
\begin{figure}
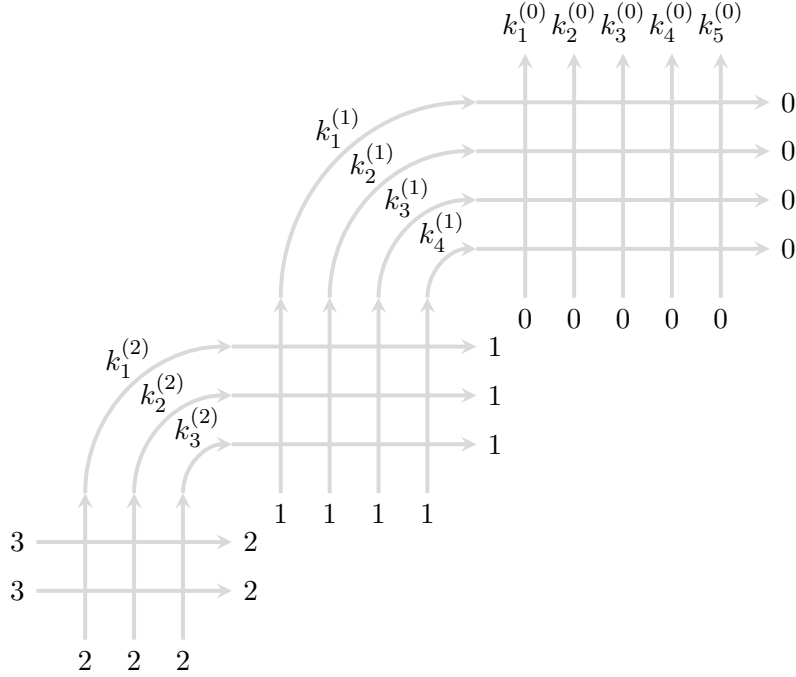

\tikz{0.65}{
\begin{scope}[rotate=180]
\foreach\x in {1,...,4}{
\draw[lgray,line width=1.5pt,<-] (2,\x) node[right] {\color{black} $0$} -- (8,\x);
\draw[lgray,line width=1.5pt,<-] (8,\x) arc (-90:0:5-\x) node[midway,above] {\color{black} $k^{(1)}_{\x}$};
\draw[lgray,line width=1.5pt,<-] (\x+8,5) -- (\x+8,9) node[below] {\color{black} $1$};
}
\foreach\x in {1,...,3}{
\draw[lgray,line width=1.5pt,<-] (8,\x+5) node[right] {\color{black} $1$} -- (13,\x+5);
\draw[lgray,line width=1.5pt,<-] (13,\x+5) arc (-90:0:4-\x) node[midway,above] {\color{black} $k^{(2)}_{\x}$};
\draw[lgray,line width=1.5pt,<-] (\x+13,9) -- (\x+13,12) node[below] {\color{black} $2$};
}
\foreach\x in {1,...,2}{
\draw[lgray,line width=1.5pt,<-] (13,\x+9) node[right] {\color{black} $2$} -- (17,\x+9) node[left] {\color{black} $3$};
}
\foreach\y in {1,...,5}{
\draw[lgray,line width=1.5pt,<-] (8-\y,0) node[above] {\color{black} $k^{(0)}_{\y}$} -- (8-\y,5) node[below] {\color{black} $0$};
}
\end{scope}
}
\caption{The graphical form of the Bethe vectors in the case $n=3$. Here the set cardinalities are chosen to be $(N_0,N_1,N_2,N_3) = (5,4,3,2)$. The vector $k^{(0)}$ selects the components of the Bethe vector, and in this case has colour multiplicities $m^{(0)} = m^{(1)} = m^{(2)} = 1$, $m^{(3)} = 2$. The vectors $k^{(1)}$ and $k^{(2)}$ are situated on internal edges, and are summed over all possible values.}
\label{fig:bethe}
\end{figure}


\section{Symmetrization formula for the Bethe vector components}

In models solvable by the coordinate Bethe Ansatz, it is well known that the components of Bethe vectors can be written explicitly in terms of summations over the symmetric group. The goal of this section will be to write down such a symmetrization formula for the Bethe vector components \eqref{bethe-comp}. We begin with some auxiliary definitions.

\begin{defn}
\label{def:coord}
Let $k^{(0)} = \left( k^{(0)}_1,\dots,k^{(0)}_{N_0} \right)$ be a vector of integers labelling the components of a Bethe vector \eqref{bethe-comp}, with colour multiplicities $m^{(0)}_i = N_i-N_{i+1}$ for all $i \in \{0,1,\dots,n\}$. Define a series of subvectors $p^{(i)} = \left( p^{(i)}_1,\dots,p^{(i)}_{N_i} \right)$, obtained from $k^{(0)}$ by deleting all elements $k^{(0)}_b < i$. By agreement, we take $p^{(0)} \equiv k^{(0)}$. For all $1 \leq i \leq n$, we introduce ordered coordinate sets $a^{(i)} = \left\{ a^{(i)}_1 < \cdots < a^{(i)}_{N_i} \right\}$ given by
\begin{align}
\label{a-sets}
a^{(i)} = \{ b: p^{(i-1)}_b \geq i \}.
\end{align}
\end{defn}

\begin{ex}
Let $(N_0,N_1,N_2,N_3) = (5,4,3,1)$ and $k^{(0)} = (2,1,3,0,2)$. In this case we have
\begin{align*}
p^{(0)} = (2,1,3,0,2),
\quad
p^{(1)} = (2,1,3,2),
\quad
p^{(2)} = (2,3,2),
\end{align*}
and accordingly, the coordinate sets are given by
\begin{align*}
a^{(1)} = \{1,2,3,5\},
\quad
a^{(2)} = \{1,3,4\},
\quad
a^{(3)} = \{2\}.
\end{align*}
\end{ex}

\begin{defn}[Single particle wavefunction]
Let $M \geq 1$ be a positive integer. Fix another integer $a \in \{1,\dots,M\}$ and a set of variables $(y_1,\dots,y_M)$. We introduce a single particle wavefunction as follows:
\begin{align}
\label{single-particle}
\index{ay@$\psi_a( x; y_1,\dots,y_M)$; single particle wavefunction}
\psi_a \left( x; y_1,\dots,y_M \right)
=
\frac{(1-q) y_a}{x - q y_a}
\prod_{j=1}^{a-1}
\left( \frac{x-y_j}{x-q y_j} \right).
\end{align}
We also define a multivariate extension of this. For any set of integers $\{a_1,\dots,a_m\}$ such that $1 \leq a_1 < \cdots < a_m \leq M$, and two sets of variables $(x_1,\dots,x_m)$ and $(y_1,\dots,y_M)$, we write
\begin{align}
\label{m-particle}
\index{ay@$\psi_{\{a_1,\dots,a_m\}}\left( x_1,\dots,x_m; y_1,\dots,y_M \right)$}
\psi_{\{a_1,\dots,a_m\}}\left( x_1,\dots,x_m; y_1,\dots,y_M \right)
=
\prod_{i=1}^{m}
\psi_{a_i} \left( x_i; y_1,\dots,y_M \right)
=
\prod_{i=1}^{m}
\left[
\frac{(1-q) y_{a_i}}{x_i - q y_{a_i}}
\prod_{j=1}^{a_i-1}
\left( \frac{x_i-y_j}{x_i-q y_j} \right)
\right].
\end{align}
\end{defn}

\begin{thm}
\label{thm:generic-bv}
Fix a vector of integers $k^{(0)} = \left( k^{(0)}_1,\dots,k^{(0)}_{N_0} \right) \in \{0,1,\dots,n\}^{N_0}$ and associate to it coordinate sets $a^{(1)},\dots,a^{(n)}$, in the same way as in Definition \ref{def:coord}. The Bethe vector components \eqref{bethe-comp} are given by the explicit formula
\begin{multline}
\label{generic-bv}
\Psi \left( x^{(1)},\dots,x^{(n)} ; k^{(0)}\right)
=
\sum_{\sigma^{(1)} \in \mathfrak{S}_{N_1}}
\cdots
\sum_{\sigma^{(n)} \in \mathfrak{S}_{N_n}}
\\
\sigma^{(1)} \cdots \sigma^{(n)} \cdot
\left(
\prod_{b=1}^{n}
\psi_{\left\{a^{(b)}_1,\dots,a^{(b)}_{N_b}\right\}}
\left(
x^{(b)}_1,\dots,x^{(b)}_{N_b} ; x^{(b-1)}_1,\dots,x^{(b-1)}_{N_{b-1}}
\right)
\prod_{1 \leq i < j \leq N_b}
\frac{q x^{(b)}_i-x^{(b)}_j}{x^{(b)}_i-x^{(b)}_j}
\right),
\end{multline}
where the permutation $\sigma^{(b)}$ is understood to act on the set of variables $x^{(b)}$; \ie\ one has
\begin{align*}
\sigma^{(b)} \cdot h\left( x^{(b)}_1,\dots,x^{(b)}_{N_b} \right)
=
h\left( x^{(b)}_{\sigma^{(b)}(1)},\dots,x^{(b)}_{\sigma^{(b)}(N_b)} \right),
\end{align*}
where $h$ denotes an arbitrary function.
\end{thm}

\begin{proof}
We will not give the full details of this proof, as it is rather technical, and instead sketch the basic method behind it. The three essential ideas are: {\bf 1.} The assumption that the formula holds in the rank-$(n-1)$ case, for some $n \geq 2$; {\bf 2.}  Recursion relations that explicitly relate $\Psi\left(x^{(1)},\dots,x^{(n)};p^{(0)}\right)$ and $\Psi\left(x^{(2)},\dots,x^{(n)};p^{(1)}-1\right)$ via a sequence of interpolating points in the variables $x^{(0)}$, where $p^{(1)}-1 \equiv 
\left( p^{(1)}_1-1,\dots,p^{(1)}_{N_1}-1 \right)$; {\bf 3.} Showing that 
\eqref{generic-bv} satisfies these recursion relations, thus establishing the formula for rank-$n$, and in general, by induction on $n$.

Let us say a few words about each of these. When $n=1$, equation \eqref{generic-bv} becomes
\begin{multline*}
\Psi \left( x^{(1)} ; k^{(0)}\right)
=
\sum_{\sigma^{(1)} \in \mathfrak{S}_{N_1}}
\sigma^{(1)} \cdot
\left(
\psi_{\left\{a^{(1)}_1,\dots,a^{(1)}_{N_1}\right\}}
\left(
x^{(1)}_1,\dots,x^{(1)}_{N_1} ; x^{(0)}_1,\dots,x^{(0)}_{N_{0}}
\right)
\prod_{1 \leq i < j \leq N_1}
\frac{q x^{(1)}_i-x^{(1)}_j}{x^{(1)}_i-x^{(1)}_j}
\right)
\end{multline*}
where $N_0 \geq N_1$ and $a^{(1)}_1 < \cdots < a^{(1)}_{N_1}$ are the coordinates of ones in the vector $k^{(0)} \in \{0,1\}^{N_0}$. This is a classical Bethe Ansatz formula for the rank-$1$ wavefunctions; there are many references to such an expression in the literature, but we note that in our present conventions it can be recovered from \cite[Equations (4.22), (4.23)]{BorodinP1} under the identifications $M = N_1$, $\mu_{M-i+1} = a^{(1)}_i-1$, $u_{M-i+1} = x^{(1)}_i$, $\xi_k = q^{-1/2}/x^{(0)}_{k+1}$, ${\sf s}_k = q^{-1/2}$. Hence the assumption {\bf 1} is indeed valid.

As far as {\bf 2} is concerned, the strategy is to start from \eqref{bethe-comp} and begin setting the variables $x^{(0)}_{a^{(1)}_i}$ to appropriate values. To ease notation in what follows, let us write $u_i := x^{(0)}_{a^{(1)}_i}$ for each $1 \leq i \leq N_1$. We set up the following chain of progressively specialized functions:
\begin{align}
\label{chain}
\Upsilon_0 := \Psi \left( x^{(1)},\dots,x^{(n)} ; k^{(0)}\right),
\qquad
\Upsilon_i := \Upsilon_{i-1} \Big|_{u_i = x^{(1)}_i}
\quad
\text{for all}\ \ 
1 \leq i \leq N_1.
\end{align}
These functions have the following properties, which descend from the definition \eqref{bethe-comp} of the Bethe vector components: 

\medskip
\noindent{\bf (a)} For all $1 \leq i \leq N_1$, the quantity
\begin{align}
\label{upsilon-norm}
\prod_{\ell = i}^{N_1}
\left( x^{(1)}_\ell - q u_i \right)
\cdot
\Upsilon_{i-1}
\end{align}
is a polynomial in $u_i$ of degree $\leq N_1-i+1$, and has symmetric dependence on $\left( x^{(1)}_i,\dots,x^{(1)}_{N_1} \right)$;

\medskip
\noindent{\bf (b)} In addition to the recursion \eqref{chain}, one has the vanishing property
\begin{align}
\label{upsilon-vanish}
\Upsilon_{i-1}
\Big|_{u_i = 0}
=
0,
\qquad
\forall\ 1 \leq i \leq N_1;
\end{align}

\medskip
\noindent{\bf (c)} After specializing all variables $u_i$, the block $Z^{(0)}$ is eliminated from \eqref{bethe-comp} and one recovers a rank-$(n-1)$ quantity. More specifically, there holds
\begin{multline}
\label{bethe-comp2}
\Upsilon_{N_1}
=
\prod_{\substack{j \not \in a^{(1)}}}^{N_0}
\prod_{i = \beta_j+1}^{N_1}
\frac{x^{(1)}_i-x^{(0)}_j}{x^{(1)}_i-q x^{(0)}_j}
\sum_{k^{(2)}}
\cdots
\sum_{k^{(n)}}
Z^{(n-1)} \left( x^{(n-1)} / x^{(n)} ; k^{(n)}, k^{(n-1)} \right)
\\
\times
\cdots
Z^{(1)} \left( x^{(1)} / x^{(2)} ; k^{(2)}, p^{(1)} \right),
\end{multline}
where we have defined $\beta_j = \#\{i : a^{(1)}_i < j \}$ and with the summation taken over all vectors $k^{(i)} \in \{i,i+1,\dots,n\}^{N_i}$, as previously. Now by shifting the entries of all vectors 
$p^{(1)}$, $k^{(i)}$ down by $1$ (which leaves \eqref{bethe-comp2} invariant), we obtain the relation
\begin{align}
\label{bethe-comp3}
\Upsilon_{N_1}
=
\prod_{\substack{j \not \in a^{(1)}}}^{N_0}
\prod_{i = \beta_j+1}^{N_1}
\frac{x^{(1)}_i-x^{(0)}_j}{x^{(1)}_i-q x^{(0)}_j}
\cdot
\Psi \left( x^{(2)},\dots,x^{(n)} ; p^{(1)}-1\right),
\end{align}
where dependence on the alphabet $x^{(1)}$ is now implicit in the right hand side.

The properties {\bf (a)}, {\bf (b)} and {\bf (c)} uniquely characterize $\Psi \left( x^{(1)},\dots,x^{(n)} ; k^{(0)}\right)$, assuming that $\Psi \left( x^{(2)},\dots,x^{(n)} ; p^{(1)}-1\right)$ is already known. This can be seen by induction on the index $i$ in $\Upsilon_i$. Using \eqref{bethe-comp3} as the basis for this induction, we may assume that $\Upsilon_i$ is already uniquely determined for some $1 \leq i \leq N_1$. Then the renormalized version \eqref{upsilon-norm} of $\Upsilon_{i-1}$ is a polynomial in $u_i$ of degree $\leq N_1-i+1$, and invoking its symmetry in $\left( x^{(1)}_i,\dots,x^{(1)}_{N_1} \right)$, together with \eqref{chain} and \eqref{upsilon-vanish}, we know its value at $N_1-i+2$ points. Hence, $\Upsilon_{i-1}$ is also uniquely determined. This argument extends all the way to 
$\Upsilon_0$, by induction on $i$.

Finally, it is not difficult to deal with item {\bf 3}, since the explicit formula \eqref{generic-bv} can be readily degenerated as in \eqref{chain}, while respecting the properties {\bf (a)}, {\bf (b)} and {\bf (c)}. 

\end{proof}

\section{Degenerations}

In this section we will show how the pre-fused functions \eqref{f-fund} introduced in Section \ref{ssec:prefused} can be recovered as reductions of the Bethe vector components \eqref{bethe-comp}. Following the language of Section \ref{ssec:blocks}, we will take the top cardinality $N_0$ to be unbounded (this corresponds with the fact that the lattice used to construct $\mathcal{F}_{\mu}$ is unbounded in the horizontal direction), and group the resulting semi-infinite block $Z^{(0)}$ into bundles with cardinalities $J_i$, $i \geq 0$. The alphabet $x^{(0)}$ will be identified with the union $Y = y^{(0)} \cup y^{(1)} \cup y^{(2)} \cup \cdots$, where $y^{(i)} = \left(y^{(i)}_1,\dots,y^{(i)}_{J_i}\right)$ is the set of variables associated with the $i$-th bundle of the lattice, in exactly the same way as in Section \ref{ssec:prefused}.

For the remaining cardinalities, we shall make the choice $N_i = n-i+1$, 
$1 \leq i \leq n$. By \eqref{col-mult} this constrains the colour multiplicities of $k^{(0)}$ to be $m^{(0)}_i = 1$ for all $1 \leq i \leq n$, and accordingly, each colour $i$ appears exactly once along the top boundary of $Z^{(0)}$. These exiting colours will be assumed to be ordered in increasing fashion within each bundle of the lattice, following the same prescription as in Section \ref{ssec:prefused}.

\begin{prop}
Let $\Psi(x^{(1)},\dots,x^{(n)}; k^{(0)})$ be the function \eqref{bethe-comp} with $N_0$ unbounded, $N_i = n-i+1$ for all $1 \leq i \leq n$ and $x^{(0)} = Y$, as above.
Fix a composition $\mu$ and following the same rules as in Section \ref{ssec:prefused}, associate to it a collection of sets 
$\{\mathcal{A}^{(0)},\mathcal{A}^{(1)},\dots\}$, where $i \in \mathcal{A}^{(\mu_i)}$ for all $1 \leq i \leq n$, $|\mathcal{A}^{(j)}| = J_j$ and each set $\mathcal{A}^{(j)}$ is in increasing order. Let $\mathcal{F}_{\mu}(x_1,\dots,x_n; Y)$ be the pre-fused function \eqref{f-fund}. We will consider the specialization of variables
\begin{align}
\label{x-spec}
x^{(1)}_j = \cdots = x^{(n-j+1)}_j = x_j,
\qquad
1 \leq j \leq n,
\end{align}
which sends each of the $n(n+1)/2$ variables $x^{(i)}_j$ to a member of the alphabet $(x_1,\dots,x_n)$, and denote it by $\tau$. We claim that
\begin{align}
\label{reduce-to-F}
\Psi \left(x^{(1)},\dots,x^{(n)} ; 
k^{(0)} = \mathcal{A}^{(0)} \cup \mathcal{A}^{(1)} \cup \cdots\right) \Big|_{\tau}
=
\mathcal{F}_{\mu}(x_n,\dots,x_1; Y),
\end{align}
where we note that the alphabet appearing in $\mathcal{F}_{\mu}$ is reversed.
\end{prop}

\begin{proof}
The reduction \eqref{reduce-to-F} has a natural graphical interpretation; see Figure \ref{fig:reduce-to-F}.\footnote{Note that the reduction \eqref{reduce-to-F} is different from that described in the proof of Theorem \ref{thm:generic-bv}. The former proceeds by eliminating the blocks $Z^{(n-1)}, \dots, Z^{(1)}$ in that order; the latter eliminates the blocks in the opposite order.}

\begin{figure}
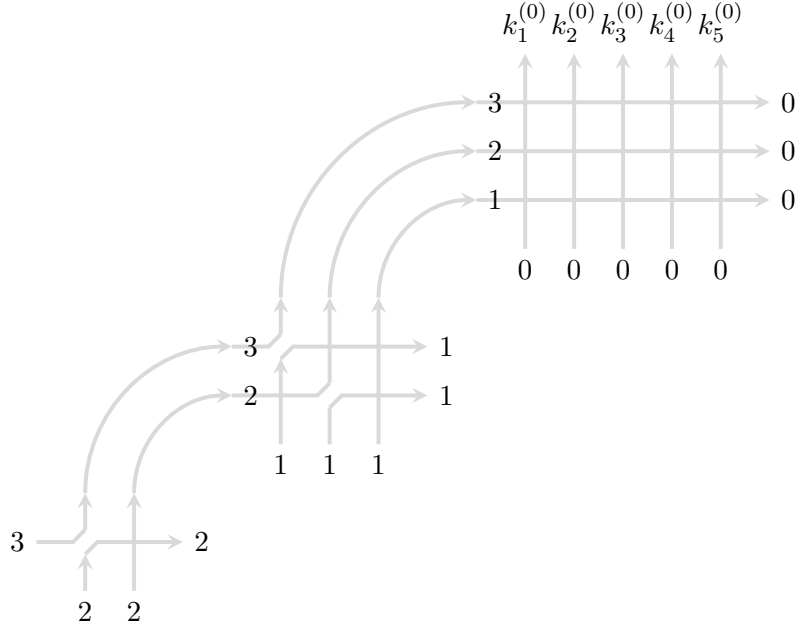

\tikz{0.65}{
\begin{scope}[rotate=180]
\foreach\x in {1,...,1}{
\draw[lgray,line width=1.5pt,<-] (2,\x) node[right] {\color{black} $0$} -- (8,\x) node[right] {\color{black} $3$};
\draw[lgray,line width=1.5pt,<-] (8,\x) arc (-90:0:5-\x);
\draw[lgray,line width=1.5pt,<-] (\x+9,5) -- (\x+9,8) node[below] {\color{black} $1$};
}
\foreach\x in {2,...,2}{
\draw[lgray,line width=1.5pt,<-] (2,\x) node[right] {\color{black} $0$} -- (8,\x) node[right] {\color{black} $2$};
\draw[lgray,line width=1.5pt,<-] (8,\x) arc (-90:0:5-\x);
\draw[lgray,line width=1.5pt,<-] (\x+9,5) -- (\x+9,6.75);
\draw[lgray,line width=1.5pt] (\x+9,7.25) -- (\x+9,8) node[below] {\color{black} $1$};
}
\foreach\x in {3,...,3}{
\draw[lgray,line width=1.5pt,<-] (2,\x) node[right] {\color{black} $0$} -- (8,\x) node[right] {\color{black} $1$};
\draw[lgray,line width=1.5pt,<-] (8,\x) arc (-90:0:5-\x);
\draw[lgray,line width=1.5pt,<-] (\x+9,5) -- (\x+9,5.75); 
\draw[lgray,line width=1.5pt,<-] (\x+9,6.25) -- (\x+9,8) node[below] {\color{black} $1$};
}
\foreach\x in {1,...,1}{
\draw[lgray,line width=1.5pt,<-] (9,\x+5) node[right] {\color{black} $1$} -- (11.75,\x+5) -- (12,\x+5.25);
\draw[lgray,line width=1.5pt] (12,\x+4.75) -- (12.25,\x+5) -- (13,\x+5) node[right] {\color{black} $3$};
\draw[lgray,line width=1.5pt,<-] (13,\x+5) arc (-90:0:4-\x);
\draw[lgray,line width=1.5pt,<-] (\x+14,9) -- (\x+14,11) node[below] {\color{black} $2$};
}
\foreach\x in {2,...,2}{
\draw[lgray,line width=1.5pt,<-] (9,\x+5) node[right] {\color{black} $1$} -- (10.75,\x+5) -- (11,\x+5.25);
\draw[lgray,line width=1.5pt] (11,\x+4.75) -- (11.25,\x+5) -- (13,\x+5) node[right] {\color{black} $2$};
\draw[lgray,line width=1.5pt,<-] (13,\x+5) arc (-90:0:4-\x);
\draw[lgray,line width=1.5pt,<-] (\x+14,9) -- (\x+14,9.75); 
\draw[lgray,line width=1.5pt,<-] (\x+14,10.25) --(\x+14,11) node[below] {\color{black} $2$};
}
\foreach\x in {1,...,1}{
\draw[lgray,line width=1.5pt,<-] (14,\x+9) node[right] {\color{black} $2$} --
(15.75,\x+9) -- (16,\x+9.25); 
\draw[lgray,line width=1.5pt] (16,\x+8.75) -- (16.25,\x+9) -- (17,\x+9) node[left] 
{\color{black} $3$};
}
\foreach\y in {1,...,5}{
\draw[lgray,line width=1.5pt,<-] (8-\y,0) node[above] {\color{black} $k^{(0)}_{\y}$} -- (8-\y,4) node[below] {\color{black} $0$};
}
\end{scope}
}
\caption{The graphical interpretation of the reduction \eqref{reduce-to-F}, in the case $n=3$. Specializing rapidities according to \eqref{x-spec} causes vertex splitting at $n(n-1)/2$ points; these splittings are sufficient to freeze the configurations of all blocks $Z^{(1)}, \dots, Z^{(n-1)}$, and we recover equation \eqref{fully-reduced}.}
\label{fig:reduce-to-F}
\end{figure}

The proof makes repeated use of the reductive property \eqref{block-reduce}, applied to the formula \eqref{bethe-comp}. Indeed, one can see that the specialization \eqref{x-spec} achieves the aim of identifying horizontal and vertical rapidities within each of the blocks $Z^{(1)},\dots,Z^{(n-1)}$ present in 
\eqref{bethe-comp}, meaning that each block reduces according to \eqref{block-reduce}. We thus read the equation
\begin{align}
\label{reduced-sum}
\Psi \left( x^{(1)},\dots,x^{(n)} ; k^{(0)}\right)
\Big|_{\tau}
=
\sum_{k^{(1)}}
\cdots
\sum_{k^{(n)}}
\prod_{i=1}^{n-1}
\left( {\bm 1}_{k^{(i)}_{n-i+1} = i} \right)
\prod_{a=1}^{n-i}
\left( {\bm 1}_{k^{(i)}_a = k^{(i+1)}_a} \right)
Z^{(0)} \left( Y / x^{(1)} ; k^{(1)}, k^{(0)} \right),
\end{align}
which takes an even simpler form than the reduction \eqref{block-reduce} in view of the fact that $N_i = n-i+1$ for all $1 \leq i \leq n$, and where we have used the fact that $x^{(0)} = Y$. All of the sums present in \eqref{reduced-sum} now trivialize: $k^{(n)} = \left(k^{(n)}_1\right) = (n)$ by agreement, and the indicator functions ensure that the remaining vectors $k^{(i)}$ are also uniquely determined as $k^{(i)} = \left(k^{(i)}_1,k^{(i)}_2,\dots,k^{(i)}_{n-i+1}\right) = (n,n-1,\dots,i)$, for all $1 \leq i \leq n$. One therefore recovers the result
\begin{align}
\label{fully-reduced}
\Psi \left( x^{(1)},\dots,x^{(n)} ; k^{(0)}\right)
\Big|_{\tau}
=
Z^{(0)} \left( Y / x^{(1)} ; (n,\dots,1), k^{(0)} \right).
\end{align}
Finally, noting that $x^{(1)}_j = x_j$ for all $1 \leq j \leq n$, after choosing $k^{(0)} = \mathcal{A}^{(0)} \cup \mathcal{A}^{(1)} \cup \cdots$ the block $Z^{(0)}$ on the right hand side of \eqref{fully-reduced} has precisely the form of $\mathcal{F}_{\mu}(x_n,\dots,x_1;Y)$. 
\end{proof}

\section{Symmetrization formula for $f_{\mu}$}

Equipped with the results \eqref{generic-bv} and \eqref{reduce-to-F}, we are now in the position to obtain a symmetrization formula for the non-symmetric spin Hall--Littlewood functions. All that is needed for this is to perform the required principal specializations and analytic continuation, as described in Theorem \ref{thm:fusion-F}, of equation \eqref{reduce-to-F}.

\begin{thm}
\label{thm:f-mu-sym}
Fix a composition $\mu = (\mu_1,\dots,\mu_n)$ and let $\delta = (\delta_1 \leq \cdots \leq \delta_n)$ be its anti-dominant reordering. Associate to this a vector $\gamma(\mu)$, where 
\begin{align*}
\gamma(\mu) 
= 
(\gamma_1,\gamma_2,\dots,\gamma_n)
= 
w_{\mu} \cdot (1,2,\dots,n),
\end{align*}
with $w_{\mu} \in \mathfrak{S}_n$ the minimal-length permutation such that $w_{\mu} \cdot \mu = \delta$. Let us define, further, the vectors
\begin{align*}
p^{(1)} = \gamma(\mu),
\qquad
p^{(i+1)} = p^{(i)} \backslash \{i\},
\ \
1 \leq i \leq n-1,
\end{align*}
as well as $n-1$ strictly increasing integer sequences $a^{(2)},\dots,a^{(n)}$, given by \eqref{a-sets}. Then the non-symmetric spin Hall--Littlewood functions are given by the sum formula
\begin{multline}
\label{f-mu-sym}
f_{\mu}(x_n,\dots,x_1)
=
\sum_{\sigma^{(1)} \in \mathfrak{S}_n}
\cdots
\sum_{\sigma^{(n-1)} \in \mathfrak{S}_2}
f_{\delta}\Big( x_{\sigma^{(1)}(1)}, \dots, x_{\sigma^{(1)}(n)} \Big)
\prod_{1 \leq i < j \leq n}
\frac{q x_{\sigma^{(1)}(i)}-x_{\sigma^{(1)}(j)}}
{x_{\sigma^{(1)}(i)}-x_{\sigma^{(1)}(j)}}
\\
\times
\prod_{b=2}^{n}
\psi_{\left\{a^{(b)}_1,\dots,a^{(b)}_{n-b+1}\right\}}
\Big(
\sigma^{(b)} \cdot (x_1,\dots,x_{n-b+1}); 
\sigma^{(b-1)} \cdot (x_1,\dots,x_{n-b+2})
\Big)
\prod_{1 \leq i < j \leq n-b+1}
\frac{q x_{\sigma^{(b)}(i)}-x_{\sigma^{(b)}(j)}}
{x_{\sigma^{(b)}(i)}-x_{\sigma^{(b)}(j)}},
\end{multline}
where by agreement $\sigma^{(n)}$ denotes the trivial permutation $\sigma^{(n)} = (1) \in \mathfrak{S}_1$.
\end{thm}

\begin{proof}
We will make the specializations $y^{(i)}_j = s q^{J_i-j}$, together with analytic continuation $q^{J_i} \mapsto s^{-2}$, of equation \eqref{reduce-to-F}. By virtue of \eqref{prefused-f}, we know that this converts the right hand side of \eqref{reduce-to-F} to
\begin{align*}
\frac{(-s)^{|\mu|}(1-q)^n}{\prod_{k \geq 0} (s^{-2};q^{-1})_{m_k}}
f_{\mu}(x_n,\dots,x_1),
\qquad
m_k \equiv m_k(\mu),
\end{align*}
which is our desired final form, up to overall normalization. 

Let us now investigate what happens to the left hand side of \eqref{reduce-to-F}. We begin with the generic expression \eqref{generic-bv} with $N_i = n-i+1$ and separate the $b=1$ case of the product in the summand from the $2 \leq b \leq n$ cases, yielding
\begin{multline}
\label{b=1-sep}
\Psi \left(x^{(1)},\dots,x^{(n)} ; 
\mathcal{A}^{(0)} \cup \mathcal{A}^{(1)} \cup \cdots\right) \Big|_{\tau}
\\
=
\sum_{\sigma^{(1)} \in \mathfrak{S}_{n}}
\cdots
\sum_{\sigma^{(n-1)} \in \mathfrak{S}_{2}}
\psi_{\left\{a^{(1)}_1,\dots,a^{(1)}_{n}\right\}}
\Big(
\sigma^{(1)} \cdot (x_1,\dots,x_n); Y \Big)
\prod_{1 \leq i < j \leq n}
\frac{q x_{\sigma^{(1)}(i)}-x_{\sigma^{(1)}(j)}}
{x_{\sigma^{(1)}(i)}-x_{\sigma^{(1)}(j)}}
\\
\times
\prod_{b=2}^{n}
\psi_{\left\{a^{(b)}_1,\dots,a^{(b)}_{n-b+1}\right\}}
\Big(
\sigma^{(b)} \cdot (x_1,\dots,x_{n-b+1}); 
\sigma^{(b-1)} \cdot (x_1,\dots,x_{n-b+2})
\Big)
\prod_{1 \leq i < j \leq n-b+1}
\frac{q x_{\sigma^{(b)}(i)}-x_{\sigma^{(b)}(j)}}
{x_{\sigma^{(b)}(i)}-x_{\sigma^{(b)}(j)}}.
\end{multline}
All dependence on the alphabet $Y = \left(y^{(0)}_1,\dots,y^{(0)}_{J_0}\right) \cup \left(y^{(1)}_1,\dots,y^{(1)}_{J_1}\right) \cup \cdots$ enters via the function 
$\psi_{\{a^{(1)}_1,\dots,a^{(1)}_n\}}$, and accordingly it is this function that we need to study in order to perform the required principal specializations and analytic continuation. We lighten the notation slightly, by writing $\{a_1^{(1)},\dots,a_n^{(1)}\} \equiv \{a_1,\dots,a_n\}$ in everything that follows below.

Analyzing the single particle wavefunction \eqref{single-particle}, assuming that $a = J_0 + J_1 + \cdots + J_{k-1} + K$ with $1 \leq K \leq J_k$ (so that the coordinate $a$ corresponds with the $K$-th element of the $k$-th bundle), we have
\begin{align}
\psi_a(x; Y)
=
\prod_{i=0}^{k-1}
\prod_{j=1}^{J_i}
\left( \frac{x-y^{(i)}_j}{x-q y^{(i)}_j} \right)
\cdot
\prod_{j=1}^{K-1}
\left( \frac{x-y^{(k)}_j}{x-q y^{(k)}_j} \right)
\frac{(1-q)y^{(k)}_K}{x-q y^{(k)}_K}.
\end{align}
Taking the principal specializations \eqref{geom-spec}, this turns into
\begin{align*}
\psi_a(x; Y)
\Big|_{y^{(i)}_j \mapsto s q^{J_i-j}}
&=
\prod_{i=0}^{k-1}
\prod_{j=1}^{J_i}
\left( \frac{x-s q^{J_i-j}}{x-s q^{J_i-j+1}} \right)
\cdot
\prod_{j=1}^{K-1}
\left( \frac{x-s q^{J_k-j}}{x-s q^{J_k-j+1}} \right)
\frac{(1-q) s q^{J_k-K}}{x-s q^{J_k-K+1}},
\\
&=
\prod_{i=0}^{k-1}
\left( \frac{x-s}{x-s q^{J_i}} \right)
\cdot
\frac{(1-q) s q^{J_k-K}}{x-s q^{J_k}}.
\end{align*}
More generally, in the case of the $n$-variable extension \eqref{m-particle} of \eqref{single-particle}, we have
\begin{align*}
\psi_{\{a_1,\dots,a_n\}}(x_1,\dots,x_n; Y)
\Big|_{y^{(i)}_j \mapsto s q^{J_i-j}}
&=
\prod_{j=1}^{n}
\left(
\prod_{i=0}^{\delta_j-1}
\left( \frac{x_j-s}{x_j-s q^{J_i}} \right)
\cdot
\frac{(1-q) s q^{J_{\delta_j}-K_j}}{x_j-s q^{J_{\delta_j}}}
\right),
\end{align*}
where we assume that $a_j = J_0 + J_1 + \cdots + J_{\delta_j-1} + K_j$ with $1 \leq K_j \leq J_{\delta_j}$, for each $1 \leq j \leq n$. If we assume that when several of the coordinates $a_j$ are members of the same bundle they must be situated as far right as possible within that bundle, we can simplify the previous expression as follows: 
\begin{multline*}
\psi_{\{a_1,\dots,a_n\}}(x_1,\dots,x_n; Y)
\Big|_{y^{(i)}_j \mapsto s q^{J_i-j}}
\\
=
(1-q)^n s^n
\prod_{k \geq 0}
q^{m_k(\delta)(m_k(\delta)-1)/2}
\prod_{j=1}^{n}
\left(
\prod_{i=0}^{\delta_j-1}
\left( \frac{x_j-s}{x_j-s q^{J_i}} \right)
\cdot
\frac{1}{x_j-s q^{J_{\delta_j}}}
\right),
\end{multline*}
where $m_k(\delta)$ denotes, as usual, the multiplicity of part $k$ within the anti-dominant composition $(\delta_1 \leq \cdots \leq \delta_n)$. Finally we analytically continue in the variables $q^{J_i}$, sending $q^{J_i} \mapsto s^{-2}$ for all $i \geq 0$, which results in the equation
\begin{multline*}
\psi_{\{a_1,\dots,a_n\}}(x_1,\dots,x_n; Y)
\Big|_{y^{(i)}_j \mapsto s q^{J_i-j}}
\
\Big|_{q^{J_i} \mapsto s^{-2}}
\\
=
(-1)^{|\delta|+n}
(1-q)^n
s^{|\delta|+2n}
\prod_{k \geq 0}
q^{m_k(\delta)(m_k(\delta)-1)/2}
\prod_{j=1}^{n}
\left( 
\frac{x_j-s}{1-s x_j} 
\right)^{\delta_j}
\frac{1}{1-sx_j},
\end{multline*}
or in terms of the function \eqref{f-delta},
\begin{align*}
&
\psi_{\{a_1,\dots,a_n\}}(x_1,\dots,x_n; Y)
\Big|_{y^{(i)}_j \mapsto s q^{J_i-j}}
\
\Big|_{q^{J_i} \mapsto s^{-2}}
\\
&=
(-1)^{|\delta|+n}
(1-q)^n
s^{|\delta|+2n}
\prod_{k \geq 0}
q^{m_k(\delta)(m_k(\delta)-1)/2}
\frac{f_{\delta}(x_1,\dots,x_n)}{\prod_{k \geq 0}(s^2;q)_{m_k(\delta)}}
\\
&=
(1-q)^n
(-s)^{|\delta|}
\frac{f_{\delta}(x_1,\dots,x_n)}{\prod_{k \geq 0}(s^{-2};q^{-1})_{m_k(\delta)}}.
\end{align*}
Substituting this result into \eqref{b=1-sep} leads to the desired formula \eqref{f-mu-sym}, after cancelling off spurious normalization factors.

\end{proof}

\chapter{Orthogonality}

\section{Scalar product}

We begin by introducing the scalar product used in the statement of our orthogonality results.
\begin{defn}[Admissible contours]
\label{def:admiss}
Let $\{C_1,\dots,C_n\}$ be a collection of contours in the complex plane, and fix two complex parameters $q,s \in \mathbb{C}$. We say that the set $\{C_1,\dots,C_n\}$ is admissible with respect to $(q,s)$ if the following conditions are met:
\begin{itemize}
\item The contours $\{C_1,\dots,C_n\}$ are closed, positively oriented and pairwise non-intersecting.
\item The contours $C_i$ and $q \cdot C_i$ are both contained within contour $C_{i+1}$ for all $1 \leq i \leq n-1$, where $q \cdot C_i$ denotes the image of $C_i$ under multiplication by $q$.
\item All contours surround the point $s$.
\end{itemize}
An illustration of such admissible contours is given in Figure \ref{fig:contours}.
\end{defn}

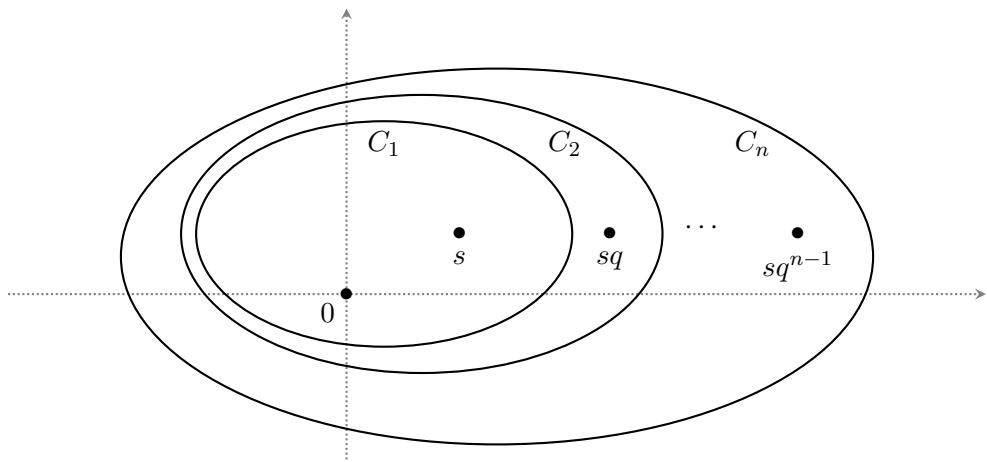
\begin{figure}
\begin{tikzpicture}[>=stealth]
\draw[gray,densely dotted,thick,->] (-1,2.2) -- (12,2.2);
\draw[gray,densely dotted,thick,->] (3.5,0) -- (3.5,6);
\draw[thick] (5.5,2.7) ellipse [x radius=5cm, y radius=2.5cm];
\draw[thick] (4.5,3) ellipse [x radius=3.2cm, y radius=1.85cm];
\draw[thick] (4,3) ellipse [x radius=2.5cm, y radius=1.5cm];
\node at (3.25,1.95) {$0$};
\node at (8.9,4.2) {$C_n$};
\node at (6.4,4.2) {$C_2$};
\node at (4,4.2) {$C_1$};
\foreach\x in {5,7,9.5}{
\node at (\x,3) {$\bullet$};
}
\node at (3.5,2.2) {$\bullet$};
\node[below,text centered] at (5,2.9) {$s$};
\node[below,text centered] at (7,2.9) {$sq$};
\node[below,text centered] at (8.25,3.3) {$\cdots$};
\node[below,text centered] at (9.5,2.9) {$sq^{n-1}$};
\end{tikzpicture}
\caption{Admissible contours $\{C_1,\dots,C_n\}$ with respect to $(q,s)$.}
\label{fig:contours}
\end{figure}
In all the cases below where we integrate rational functions over $\{C_1,\dots,C_n\}$, the integrals can also be computed as sums of residues of the integrand inside the contours. Such sums also make sense for values of parameters that prevent admissible contours existing, and thus the integrals could also be defined via the residue sums. Therefore, in what follows we tacitly assume that we perform such a replacement should the admissible contours not exist.

\index{1@$\langle \cdot,\cdot \rangle_{q,s}$; scalar product}
\begin{defn}[Scalar product]
\label{def:scalar}
Let $\Psi : \mathbb{C}^n \rightarrow \mathbb{C}$ and $\Phi : \mathbb{C}^n \rightarrow \mathbb{C}$ be arbitrary functions in $n$ complex variables $(x_1,\dots,x_n)$, and let $q,s \in \mathbb{C}$ be two fixed complex parameters. We introduce a scalar product (with the usual notation $\b{x}_i = x_i^{-1}$)
\begin{align}
\label{scalar}
\langle \Psi, \Phi \rangle_{q,s} 
:=
\left( \frac{1}{2\pi\sqrt{-1}} \right)^n
\oint_{C_1}
\frac{dx_1}{x_1}
\cdots 
\oint_{C_n}
\frac{dx_n}{x_n}
\prod_{1 \leq i<j \leq n}
\left(
\frac{x_j-x_i}{x_j-q x_i}
\right)
\Psi(\b{x}_1,\dots,\b{x}_n)
\Phi(x_1,\dots,x_n),
\end{align}
where the set of integration contours $\{C_1,\dots,C_n\}$ is admissible with respect to $(q,s)$.
\end{defn}

One of the key properties of the scalar product \eqref{scalar} is that it renders the Hecke generators \eqref{hecke-poly} and their reversed-variable versions \eqref{rev-Hecke} adjoint to each other:
\begin{prop}
Let $\Psi(x_1,\dots,x_n)$ and $\Phi(x_1,\dots,x_n)$ be two rational functions which are non-singular at the values $q x_i = x_j$ for all $i \not= j$. There holds  
\begin{align}
\label{adjoint}
\langle T_k \cdot \Psi, \Phi \rangle_{q,s}
=
\langle \Psi, \tilde{T}_k \cdot \Phi \rangle_{q,s},
\qquad
\forall\ 1 \leq k \leq n-1,
\end{align}
where $T_k$ and $\tilde{T}_k$ are Hecke generators given by \eqref{hecke-poly} and \eqref{rev-Hecke}, respectively.
\end{prop}

\begin{proof}
Begin with the left hand side of \eqref{adjoint}, writing explicitly the action of $T_k$:
\begin{multline}
\langle T_k \cdot \Psi, \Phi \rangle_{q,s}
=
\left( \frac{1}{2\pi\sqrt{-1}} \right)^n
\oint_{C_1}
\frac{dx_1}{x_1}
\cdots 
\oint_{C_n}
\frac{dx_n}{x_n}
\prod_{1 \leq i<j \leq n}
\left(
\frac{x_j-x_i}{x_j-q x_i}
\right)
\\
\times
\left[
\frac{(q-1)\b{x}_k}{\b{x}_k-\b{x}_{k+1}} \Psi(\b{x}_1,\dots,\b{x}_n)
+
\frac{\b{x}_k-q\b{x}_{k+1}}{\b{x}_k-\b{x}_{k+1}} \Psi(\b{x}_1,\dots,\b{x}_{k+1},\b{x}_k,\dots,\b{x}_n)
\right]
\Phi(x_1,\dots,x_n).
\end{multline}
Opening the brackets inside the integrand, we write this as a sum of two integrals:
\begin{align*}
\langle T_k \cdot \Psi, \Phi \rangle_{q,s}
=
I_1(k) + I_2(k),
\end{align*}
where we have defined
\begin{multline*}
I_1(k)
=
\left( \frac{1}{2\pi\sqrt{-1}} \right)^n
\oint_{C_1}
\frac{dx_1}{x_1}
\cdots 
\oint_{C_n}
\frac{dx_n}{x_n}
\\
\times
\prod_{1 \leq i<j \leq n}
\left(
\frac{x_j-x_i}{x_j-q x_i}
\right)
\Psi(\b{x}_1,\dots,\b{x}_n)
\left[
\frac{(q-1) x_{k+1}}{x_{k+1}-x_k}
\right]
\Phi(x_1,\dots,x_n),
\end{multline*}
\begin{multline*}
I_2(k)
=
\left( \frac{1}{2\pi\sqrt{-1}} \right)^n
\oint_{C_1}
\frac{dx_1}{x_1}
\cdots 
\oint_{C_n}
\frac{dx_n}{x_n}
\\
\times
\prod_{1 \leq i<j \leq n}
\left(
\frac{x_j-x_i}{x_j-q x_i}
\right)
\Psi(\b{x}_1,\dots,\b{x}_{k+1},\b{x}_k,\dots,\b{x}_n)
\left[
\frac{x_{k+1}-q x_k}{x_{k+1}-x_k}
\right]
\Phi(x_1,\dots,x_n).
\end{multline*}
Because of the presence of its extra rational factor, the integrand of $I_2(k)$ is clearly non-singular at $q x_k = x_{k+1}$. This means that we can freely move the contours $C_k$ and $C_{k+1}$ past each other, without crossing any singularities. Interchanging these contours and relabelling $C_{k} \leftrightarrow C_{k+1}$, as well as $x_k \leftrightarrow x_{k+1}$, we find that
\begin{multline*}
I_2(k)
=
\left( \frac{1}{2\pi\sqrt{-1}} \right)^n
\oint_{C_1}
\frac{dx_1}{x_1}
\cdots 
\oint_{C_n}
\frac{dx_n}{x_n}
\\
\times
\prod_{1 \leq i<j \leq n}
\left(
\frac{x_j-x_i}{x_j-q x_i}
\right)
\Psi(\b{x}_1,\dots,\b{x}_n)
\left[
\frac{x_{k+1}-q x_k}{x_{k+1}-x_k}
\right]
\Phi(x_1,\dots,x_{k+1},x_k,\dots,x_n).
\end{multline*}
Now we recombine $I_1(k)$ and $I_2(k)$ to a single integral, yielding
\begin{multline}
\label{eq151}
\langle T_k \cdot \Psi, \Phi \rangle_{q,s}
=
\left( \frac{1}{2\pi\sqrt{-1}} \right)^n
\oint_{C_1}
\frac{dx_1}{x_1}
\cdots 
\oint_{C_n}
\frac{dx_n}{x_n}
\prod_{1 \leq i<j \leq n}
\left(
\frac{x_j-x_i}{x_j-q x_i}
\right)
\Psi(\b{x}_1,\dots,\b{x}_n)
\\
\times
\left[
\frac{(q-1) x_{k+1}}{x_{k+1}-x_k} \Phi(x_1,\dots,x_n)
+
\frac{x_{k+1}-qx_k}{x_{k+1}-x_k} \Phi(x_1,\dots,x_{k+1},x_k,\dots,x_n)
\right].
\end{multline}
The term in brackets within the integrand is nothing but the expanded form of $\tilde{T}_k \cdot \Phi(x_1,\dots,x_n)$; it follows that the right hand side of \eqref{eq151} is equal to $\langle \Psi, \tilde{T}_k \cdot \Phi \rangle_{q,s}$, completing the proof.
\end{proof}

\section{Orthogonality of $f_{\mu}$ and $g^{*}_{\nu}$}

The main result of this chapter follows below. We prove that the functions $f_{\mu}$ and $g^{*}_{\nu}$ are orthonormal with respect to the scalar product of Definition \ref{def:scalar}. This result extends a previously known orthogonality statement on the non-symmetric Hall--Littlewood polynomials \cite{Cherednik2,MimachiN} to arbitrary values of the spin parameter $s$.

\begin{thm}
\label{thm:orthog}
Let $\mu = (\mu_1,\dots,\mu_n)$ and $\nu = (\nu_1,\dots,\nu_n)$ be two compositions. We have the following orthonormality of non-symmetric spin Hall--Littlewood functions:
\begin{align}
\langle f_{\mu}, g^{*}_{\nu} \rangle_{q,s} 
=
{\bm 1}_{\mu = \nu},
\end{align}
or more explicitly,
\begin{multline}
\label{f-g-orthog}
\left( \frac{1}{2\pi\sqrt{-1}} \right)^n
\oint_{C_1}
\frac{dx_1}{x_1}
\cdots 
\oint_{C_n}
\frac{dx_n}{x_n}
\prod_{1 \leq i<j \leq n}
\left(
\frac{x_j-x_i}{x_j-q x_i}
\right)
f_{\mu}(\b{x}_1,\dots,\b{x}_n)
g^{*}_{\nu}(x_1,\dots,x_n)
=
{\bm 1}_{\mu = \nu}.
\end{multline}
\end{thm}


\begin{proof}
The proof of \eqref{f-g-orthog} is in three stages. The first (and most difficult) stage is a direct proof of \eqref{f-g-orthog} in the case 
$\mu = \delta = (\delta_1 \leq \cdots \leq \delta_n)$, with $\nu \not= \delta$ but otherwise arbitrary, when the integral must vanish. The second stage is the proof of \eqref{f-g-orthog} in the case $\mu = \nu = \delta = (\delta_1 \leq \cdots \leq \delta_n)$, when the integral must produce the correct normalization. In the third stage, we use the results of the first two steps to infer that \eqref{f-g-orthog} holds for general $\mu,\nu$.

\medskip
\underline{\bf Step 1: The case $\mu = \delta$, $\nu \not= \delta$.}

We will attempt to compute $\langle f_{\delta}, g^{*}_{\nu} \rangle_{q,s}$ directly. We do this making use of two of our earlier results; the explicit factorized form \eqref{f-delta} of $f_{\delta}$, and the expansion \eqref{g-monom} of $g^{*}_{\nu}$ in terms of the monomials \eqref{xi}. Substituting \eqref{f-delta} and \eqref{g-monom} into the scalar product \eqref{f-g-orthog}, we obtain
\begin{multline}
\label{7.2-1}
K_{\delta}
\cdot
\langle f_{\delta}, g^{*}_{\nu} \rangle_{q,s}
=
\sum_{\rho \in \mathfrak{S}_n}
\oint_{C_1}
\frac{dx_1}{x_1}
\cdots 
\oint_{C_n}
\frac{dx_n}{x_n}
\prod_{1 \leq i<j \leq n}
\left(
\frac{x_j-x_i}{x_j-q x_i}
\right)
\left(
\frac{x_{\rho(j)} - q x_{\rho(i)}}{x_{\rho(j)} - x_{\rho(i)}}
\right)
\\
\times
Z_{\tilde\rho}^{\sigma}(x_n,\dots,x_1)
\xi_{\delta}(\b{x}_1,\dots,\b{x}_n)
\xi_{\epsilon}(x_{\rho(1)},\dots,x_{\rho(n)}),
\end{multline}
with $\tilde{\rho}(i) := n-\rho(i)+1$ for all $1 \leq i \leq n$. Here we have collected all irrelevant multiplicative factors into the constant
\begin{align*}
K_{\delta}
:=
\frac{(2\pi\sqrt{-1})^n}{\prod_{i \geq 0} (s^2;q)_{m_i(\delta)}},
\end{align*}
and $\nu$ is expressed in terms of the unique anti-dominant composition $\epsilon = (\epsilon_1 \leq \cdots \leq \epsilon_n)$ and minimal-length permutation $\sigma \in \mathfrak{S}_n$ such that
\begin{align*}
\nu_{n-i+1} = \epsilon_{\sigma(i)},
\qquad
\forall\ 1 \leq i \leq n.
\end{align*} 
Our strategy will be to show that every term in the sum over $\rho \in \mathfrak{S}_n$ in \eqref{7.2-1} vanishes in its own right, unless $\nu = \delta$. In order to achieve this, it is necessary to identify all potential poles in the integrand of \eqref{7.2-1}. Recalling the structure \eqref{singular-structure} of the singularities in $Z^{\sigma}_{\tilde\rho}$ and using the vanishing property \eqref{vanish-Z}, we see that
\begin{multline*}
K_{\delta}
\cdot
\langle f_{\delta}, g^{*}_{\nu} \rangle_{q,s}
=
\sum_{\rho \in \mathfrak{S}_n : \ell(\tilde\rho) \leq \ell(\sigma)}
\oint_{C_1}
\frac{dx_1}{x_1}
\cdots 
\oint_{C_n}
\frac{dx_n}{x_n}
\prod_{1 \leq i<j \leq n}
\left(
\frac{x_j-x_i}{x_j-q x_i}
\right)
\left(
\frac{x_{\rho(j)} - q x_{\rho(i)}}{x_{\rho(j)} - x_{\rho(i)}}
\right)
\\
\times
\prod_{i<j : \rho(i) > \rho(j)}
\left(
\frac{1}{x_{\rho(j)}-q x_{\rho(i)}}
\right)
\mathcal{P}_{\rho}^{\sigma}(x_1,\dots,x_n)
\xi_{\delta}(\b{x}_1,\dots,\b{x}_n)
\xi_{\epsilon}(x_{\rho(1)},\dots,x_{\rho(n)}),
\end{multline*}
where $\mathcal{P}_{\rho}^{\sigma}(x_1,\dots,x_n)$ is the polynomial specified at the end of Proposition \ref{prop:singular}, or after simple rearrangement of the factors in the integrand,
\begin{multline}
\label{7.2-2}
K_{\delta}
\cdot
\langle f_{\delta}, g^{*}_{\nu} \rangle_{q,s}
=
\sum_{\rho \in \mathfrak{S}_n : \ell(\tilde\rho) \leq \ell(\sigma)}
{\rm sgn}(\rho)
\oint_{C_1}
\frac{dx_1}{x_1}
\cdots 
\oint_{C_n}
\frac{dx_n}{x_n}
\prod_{i<j: \rho(i) > \rho(j)}
\left(
\frac{1}{x_{\rho(i)}-q x_{\rho(j)}}
\right)
\\
\times
\mathcal{P}_{\rho}^{\sigma}(x_1,\dots,x_n)
\xi_{\delta}(\b{x}_1,\dots,\b{x}_n)
\xi_{\epsilon}(x_{\rho(1)},\dots,x_{\rho(n)}).
\end{multline}
To proceed further, we need the following result:
\begin{lem}
\label{lem:xixi}
Let $\{C_1,\dots,C_n\}$ be a collection of integration contours as in Definition \ref{def:admiss}, and fix two anti-dominant compositions $\delta, \epsilon$, as well as a permutation $\rho \in \mathfrak{S}_n$. In addition, let $\mathcal{P}(x_1,\dots,x_n)$ be a polynomial in $(x_1,\dots,x_n)$ of sufficiently small degree such that
\begin{align}
\label{deg-constr}
\lim_{x_a \rightarrow \infty}
\left[
\frac{\mathcal{P}(x_1,\dots,x_n)}
{
\prod_{i<j: \rho(i) > \rho(j)} \left(x_{\rho(i)}-q x_{\rho(j)}\right)
}
\right]
\quad
\text{exists for all}\ 1 \leq a \leq n.
\end{align}
We then have
\begin{align}
\label{Pxixi}
\oint_{C_1}
\frac{dx_1}{x_1}
\cdots 
\oint_{C_n}
\frac{dx_n}{x_n}
\frac{\mathcal{P}(x_1,\dots,x_n)}
{
\prod_{i<j: \rho(i) > \rho(j)}
\left(x_{\rho(i)}-q x_{\rho(j)}\right)
}
\xi_{\delta}(\b{x}_1,\dots,\b{x}_n)
\xi_{\epsilon}(x_{\rho(1)},\dots,x_{\rho(n)})
=
0,
\end{align}
unless $\delta = \epsilon$ and $\rho$ is chosen such that $\delta_{\rho(i)} = \delta_i$ for all $1 \leq i \leq n$.
\end{lem}

\begin{proof}
This lemma has a similar flavour to analogous statements in \cite{BorodinCPS,BorodinP2}. A full proof requires a non-trivial combinatorial argument related to permutations, which can be found in \cite{BorodinCPS}; for our purposes it will be sufficient to quote from that previous work.

To illustrate, let us begin by considering what happens to the integral \eqref{Pxixi} when $\rho = {\rm id} = (1,\dots,n)$. In that case the product over $i<j : \rho(i) > \rho(j)$ is empty and the polynomial $\mathcal{P}$ is constrained, by its degree restrictions \eqref{deg-constr}, to be a constant, so the integral reduces to the form
\begin{align}
\label{triv-integral}
\oint_{C_1}
\frac{dx_1}{x_1}
\cdots 
\oint_{C_n}
\frac{dx_n}{x_n}
\xi_{\delta}(\b{x}_1,\dots,\b{x}_n)
\xi_{\epsilon}(x_1,\dots,x_n)
=
\oint_{C_1}
dx_1
\cdots 
\oint_{C_n}
dx_n
\prod_{i=1}^{n}
\frac{(x_i-s)^{\epsilon_i-\delta_i-1}}{(1-s x_i)^{\epsilon_i-\delta_i+1}}.
\end{align}
This can be considered as $n$ independent integrations; for $i \not= j$, we can freely deform the contours $C_i$ and $C_j$ past each other, since there are no longer any cross-terms which are singular under such exchanges. If for some $1 \leq i \leq n$ one has $\epsilon_i > \delta_i$, the integrand of \eqref{triv-integral} is non-singular at $x_i = s$, and shrinking $C_i$ to that point, the integral vanishes. On the other hand if for the same $i$ one has $\epsilon_i < \delta_i$, the integrand of \eqref{triv-integral} is non-singular at $x_i = s^{-1}$, and shrinking $C_i$ to that point (noting that the pole at infinity has zero residue\footnote{Shrinking takes place on $\mathbb{C} \cup \{\infty\}$, \ie\ on the Riemann sphere.}), the integral is again null. We conclude that \eqref{triv-integral} is non-vanishing only if $\delta = \epsilon$.

The argument is similar for generic $\rho$: once again we would like to shrink each contour $C_i$ to surround either $x_i = s$ or $x_i = s^{-1}$, but this time we must be mindful of potential singularities arising at the points 
$x_{\rho(a)} = q x_{\rho(b)}$, $a < b$. 

\begin{itemize}
\item Let us suppose that for some $1 \leq a \leq n$ we have $\rho(a) < {\rm min}(\rho(a+1),\dots,\rho(n))$. Then one can easily see that none of the factors 
$(x_{\rho(a)} - q x_1), \dots, (x_{\rho(a)} - q x_{\rho(a)-1})$ are present in the product $\prod_{i<j : \rho(i) > \rho(j)} (x_{\rho(i)} - q x_{\rho(j)})$. This means that the contour $C_{\rho(a)}$ can be moved past the contours $C_1,\dots,C_{\rho(a)-1}$ without crossing any singularities, and shrunk to the point $x_{\rho(a)} = s$. By similar reasoning as above, we then see that \eqref{Pxixi} vanishes unless $\epsilon_a \leq \delta_{\rho(a)}$.

\medskip
\item Alternatively, let us suppose that for some $1 \leq b \leq n$ one has $\rho(b) > {\rm max}(\rho(1),\dots,\rho(b-1))$. Then none of the factors $(x_n - q x_{\rho(b)}),\dots,(x_{\rho(b)+1} - q x_{\rho(b)})$ are present in the product 
$\prod_{i<j : \rho(i) > \rho(j)} (x_{\rho(i)} - q x_{\rho(j)})$. This allows the contour $C_{\rho(b)}$ to be moved past the contours $C_{\rho(b)+1},\dots,C_n$ without crossing any singularities\footnote{It is straightforward to check that the pole at infinity again has zero residue, which is ensured by the existence of the limits \eqref{deg-constr}.}, and shrunk to the point $x_{\rho(b)} = s^{-1}$. The integration over $x_{\rho(b)}$ can then be directly performed, and we conclude that \eqref{Pxixi} vanishes unless $\epsilon_b \geq \delta_{\rho(b)}$.
\end{itemize}
To summarize, we have shown that the integral in \eqref{Pxixi} vanishes unless 
$\delta, \epsilon$ and  $\rho \in \mathfrak{S}_n$ are chosen such that
\begin{align}
\label{perm-constraints}
\rho(a) < {\rm min}(\rho(a+1),\dots,\rho(n)) \implies \epsilon_a \leq \delta_{\rho(a)},
\quad
\rho(b) > {\rm max}(\rho(1),\dots,\rho(b-1)) \implies \epsilon_b \geq \delta_{\rho(b)}.
\end{align}
This boils the non-vanishing of \eqref{Pxixi} down to the purely combinatorial question of classifying solutions of \eqref{perm-constraints}. It turns out that the only possible solutions of the constraints \eqref{perm-constraints} are to choose $\delta = \epsilon$ and $\rho$ such that $\delta_{\rho(i)} = \delta_i$, \ie\ $\rho$ must live in the following subgroup of $\mathfrak{S}_n$: 
\begin{align}
\label{subgroup}
\rho 
\in 
\mathfrak{S}_{m_0(\delta)} 
\times 
\mathfrak{S}_{m_1(\delta)} 
\times \cdots \times  
\mathfrak{S}_{m_{\delta_n}(\delta)}.
\end{align}
The statement of this fact also appeared in \cite[Section 7]{BorodinP2}; for its proof, we refer the reader to \cite[Section 3]{BorodinCPS}. This completes the proof of equation \eqref{Pxixi}.

\end{proof}

Applying the result of Lemma \ref{lem:xixi} to equation \eqref{7.2-2}, we see that the right hand side vanishes unless $\delta = \epsilon$. This tells us that $\nu$ must at the least be a permutation of the parts of $\delta$. Further to this, even when $\delta = \epsilon$, we require that $\delta_{\rho(i)} = \delta_i$ in order for the permutation $\rho$ to contribute to the sum \eqref{7.2-2}. We now show that if $\nu \not= \delta$, there are no permutations $\rho$ in the sum \eqref{7.2-2} which satisfy this criterion. To do that, note that any permutation $\rho$ in the subgroup \eqref{subgroup} has the length constraint
\begin{align}
\ell(\rho) \leq \sum_{i=0}^{\delta_n} m_i(\delta)(m_i(\delta)-1)/2,
\end{align}
and accordingly, the length of the corresponding conjugated permutation $\tilde\rho$ must satisfy
\begin{align}
\label{ineq-rho}
\ell(\tilde\rho)
\geq
n(n-1)/2
-
\sum_{i=0}^{\delta_n} m_i(\delta)(m_i(\delta)-1)/2
=
\sum_{0 \leq i < j \leq \delta_n} m_i(\delta) m_j(\delta).
\end{align}
On the other hand, given that $\sigma$ is the minimal permutation such that $\nu_{n-i+1} = \delta_{\sigma(i)}$, an upper bound (worst case scenario) on the length of $\sigma$ is determined by considering the case $\nu = \delta$, when $\delta_{n-i+1} = \delta_{\sigma(i)}$ for all $1\leq i \leq n$. From this, we see that
\begin{align}
\label{ineq-sig}
\ell(\sigma) \leq \sum_{0 \leq i < j \leq \delta_n} m_i(\delta) m_j(\delta).
\end{align}
Using the inequalities \eqref{ineq-rho} and \eqref{ineq-sig}, we see that the only permutations $\rho$ and $\sigma$ which respect the constraint $\ell(\tilde\rho) \leq \ell(\sigma)$ of the sum \eqref{7.2-2} are the maximal possible ones; $\rho$ being the longest element in \eqref{subgroup} and $\sigma$ the shortest permutation which arranges $\delta$ in decreasing order. However, this case coincides precisely with the case $\nu = \delta$. We conclude, finally, that
\begin{align}
\label{delta-nu}
\langle f_{\delta}, g^{*}_{\nu} \rangle_{q,s}
=
0,
\qquad
\delta = (\delta_1 \leq \cdots \leq \delta_n),
\qquad
\nu \not= \delta.
\end{align}

\medskip
\underline{\bf Step 2: The case $\mu = \nu = \delta$.}

The next step of the proof is to fix the normalization of our scalar product \eqref{f-g-orthog}, which we do by computing $\langle f_{\delta}, g^{*}_{\delta} \rangle_{q,s}$. For this purpose we do not resort to the monomial expansion \eqref{g-monom} of $g^{*}_{\delta}$, but instead make use of the summation identity \eqref{mimachi-id} of Mimachi--Noumi type. From \eqref{delta-nu}, we know that
\begin{align}
\label{delta-nu-explicit}
\left( \frac{1}{2\pi\sqrt{-1}} \right)^n
\oint_{C_1}
\frac{dx_1}{x_1}
\cdots 
\oint_{C_n}
\frac{dx_n}{x_n}
\prod_{1 \leq i<j \leq n}
\left(
\frac{x_j-x_i}{x_j-q x_i}
\right)
f_{\delta}(\b{x}_1,\dots,\b{x}_n)
g^{*}_{\nu}(x_1,\dots,x_n)
=
\left( {\bm 1}_{\delta = \nu} \right)
\kappa_{\delta}(q,s),
\end{align}
for all compositions $\nu$, with the constant $\kappa_{\delta}(q,s)$ as yet undetermined. Let us multiply \eqref{delta-nu-explicit} by $f_{\nu}(y_1,\dots,y_n)$, where the variables $(\b{y}_1,\dots,\b{y}_n)$ are assumed to lie outside of the integration contours, and sum over all $\nu$. Computing this sum using \eqref{mimachi-id}\footnote{The assumption that $(\b{y}_1,\dots,\b{y}_n)$ lie outside of the integration contours is consistent with the convergence criteria \eqref{weight-condition2}; hence the use of \eqref{mimachi-id} is justified.}, we obtain the equation
\begin{multline}
\left( \frac{1}{2\pi\sqrt{-1}} \right)^n
\oint_{C_1}
\frac{dx_1}{x_1}
\cdots 
\oint_{C_n}
\frac{dx_n}{x_n}
\prod_{1 \leq i<j \leq n}
\left(
\frac{x_j-x_i}{x_j-q x_i}
\right)
\left(
\frac{1-q x_i y_j}{1-x_i y_j}
\right)
\prod_{i=1}^{n}
\left(
\frac{1}{1-x_i y_i}
\right)
f_{\delta}(\b{x}_1,\dots,\b{x}_n)
\\
=
\kappa_{\delta}(q,s)
f_{\delta}(y_1,\dots,y_n),
\end{multline}
or replacing $f_{\delta}$ by its factorized expression \eqref{f-delta}, one has more explicitly
\begin{multline}
\label{kappa-eq}
\kappa_{\delta}(q,s)
=
\prod_{i=1}^{n}
\frac{(1-sy_i)^{\delta_i+1}}{(y_i-s)^{\delta_i}}
\\
\times
\left( \frac{1}{2\pi\sqrt{-1}} \right)^n
\oint_{C_1}
dx_1
\cdots 
\oint_{C_n}
dx_n
\prod_{1 \leq i<j \leq n}
\left(
\frac{x_j-x_i}{x_j-q x_i}
\right)
\left(
\frac{1-q x_i y_j}{1-x_i y_j}
\right)
\prod_{i=1}^{n}
\frac{(1-sx_i)^{\delta_i}}{(1-x_i y_i)(x_i-s)^{\delta_i+1}}.
\end{multline}
Equation \eqref{kappa-eq} can now be used to determine $\kappa_{\delta}(q,s)$. The integrals on the right hand side of \eqref{kappa-eq} can be easily computed, starting with the contour $C_n$ and working towards $C_1$. Indeed, we can shrink the contour $C_n$ (in the Riemann sphere) to surround the single pole at $x_n = \b{y}_n$ (with negative orientation); this works because the residue of the integrand at $x_n = \infty$ is zero, and no other potential singularities are crossed during the shrinking. Computing the residue at $x_n = \b{y}_n$ leads to a nice telescoping of the factors in the integrand; we obtain
\begin{multline*}
\kappa_{\delta}(q,s)
=
\prod_{i=1}^{n-1}
\frac{(1-sy_i)^{\delta_i+1}}{(y_i-s)^{\delta_i}}
\cdot
\left( \frac{1}{2\pi\sqrt{-1}} \right)^{n-1}
\oint_{C_1}
dx_1
\cdots 
\oint_{C_{n-1}}
dx_{n-1}
\\
\times
\prod_{1 \leq i<j \leq n-1}
\left(
\frac{x_j-x_i}{x_j-q x_i}
\right)
\left(
\frac{1-q x_i y_j}{1-x_i y_j}
\right)
\prod_{i=1}^{n-1}
\frac{(1-sx_i)^{\delta_i}}{(1-x_i y_i)(x_i-s)^{\delta_i+1}}.
\end{multline*}
The remaining integrals can then be computed sequentially in exactly the same way, leading to the final result
\begin{align*}
\kappa_{\delta}(q,s)
=
1,
\qquad
\text{for all}\ \ \delta = (\delta_1 \leq \cdots \leq \delta_n).
\end{align*}
We have thus shown that
\begin{align}
\label{delta-delta}
\langle f_{\delta}, g^{*}_{\delta} \rangle_{q,s}
=
1,
\qquad
\delta = (\delta_1 \leq \cdots \leq \delta_n).
\end{align}

\medskip
\underline{\bf Step 3: Generic $\mu$, $\nu$.} 

In the last stage of the proof, we combine the previously established facts \eqref{delta-nu} and \eqref{delta-delta}, together with the adjointness property \eqref{adjoint} of the Hecke generators with respect to the scalar product. Turning to the case of generic compositions $\mu$, $\nu$ we write, as usual,
\begin{align}
\label{delta-epsilon}
\mu_i = \delta_{\sigma(i)},
\qquad
\forall\ 1 \leq i \leq n,
\end{align}
where $\delta$ is an anti-dominant composition and $\sigma \in \mathfrak{S}_n$ is the minimal-length permutation for which these relations are obeyed. Next, we use the recursive property \eqref{T-f} of the non-symmetric spin Hall--Littlewood functions to write
\begin{align}
\label{eq156}
\langle f_{\mu}, g^{*}_{\nu} \rangle_{q,s}
=
\langle T_{\sigma} \cdot f_{\delta}, g^{*}_{\nu} \rangle_{q,s},
\end{align}
where $\delta$ and $\sigma$ are as given by \eqref{delta-epsilon}, and $T_{\sigma} = T_{i_{\ell(\sigma)}} \cdots T_{i_1}$ is a product of Hecke generators corresponding to a reduced-word decomposition of $\sigma$; see equations \eqref{reduced} and \eqref{Tsigma}. Making use of the adjointness relation \eqref{adjoint}, \eqref{eq156} becomes
\begin{align}
\label{expr1}
\langle f_{\mu}, g^{*}_{\nu} \rangle_{q,s}
=
\langle f_{\delta}, \tilde{T}_{\sigma} \cdot g^{*}_{\nu} \rangle_{q,s},
\qquad
\text{with}\ \ 
\tilde{T}_{\sigma} =  \tilde{T}_{i_1} \cdots \tilde{T}_{i_{\ell(\sigma)}}.
\end{align}
Hence we return to the situation where the first argument of the scalar product is $f_{\delta}$, when we know how to compute it. 

Now let $\mathfrak{P}_{\ell}$ denote the proposition that 
\begin{align}
\label{P-ell}
\langle f_{\delta}, \tilde{T}_{\sigma} \cdot g^{*}_{\nu} \rangle_{q,s}
=
\bm{1}_{\sigma\cdot\delta=\nu},
\end{align}
where $\ell \equiv \ell(\sigma)$ is the length of the permutation appearing in \eqref{delta-epsilon}. We have already shown that $\mathfrak{P}_{0}$ is true, by virtue of \eqref{delta-nu} and \eqref{delta-delta}. When $\ell=1$, we have $\sigma = \mathfrak{s}_i$ for some $1 \leq i \leq n-1$, and $\mathfrak{P}_{1}$ becomes the statement that
\begin{align}
\label{P-ell=1}
\langle f_{\delta}, \tilde{T}_{i} \cdot g^{*}_{\nu} \rangle_{q,s}
=
\bm{1}_{\mathfrak{s}_i\cdot\delta=\nu},
\qquad
\delta_i < \delta_{i+1},
\end{align}
where we can eliminate the possibility that $\delta_i = \delta_{i+1}$, since in that case $\delta_{\sigma(i)} = \delta_i$, contradicting the minimality requirement of 
$\sigma$. 

In order to prove \eqref{P-ell=1}, we derive three relations involving the action of Hecke generators $\tilde{T}_i$ on $g^{*}_{\alpha}$, $\alpha \in \mathbb{N}^n$. Consider the recursion \eqref{That-g} with $x_j \mapsto x_{n-j+1}$ for all $1 \leq j \leq n$, when it explicitly reads
\begin{align}
\label{7.2-3}
\left(
1 - \frac{x_{i+1}-q x_i}{x_{i+1}-x_i} (1-\mathfrak{s}_i)
\right)
g^{*}_{\alpha}(x_1,\dots,x_n)
=
g^{*}_{\mathfrak{s}_i \cdot \alpha}(x_1,\dots,x_n),
\quad
\alpha_i > \alpha_{i+1}.
\end{align}
Recalling the definition \eqref{rev-Hecke} of $\tilde{T}_i$ and its inverse, we recognise \eqref{7.2-3} in the form
\begin{align*}
q \cdot \tilde{T}_i^{-1} \cdot g^{*}_{\alpha}(x_1,\dots,x_n)
=
g^{*}_{\mathfrak{s}_i \cdot \alpha}(x_1,\dots,x_n),
\qquad
\alpha_i > \alpha_{i+1},
\end{align*}
or equivalently,
\begin{align}
\label{tildeT-g<}
\tilde{T}_i \cdot g^{*}_{\alpha}(x_1,\dots,x_n)
&=
q\cdot g^{*}_{\mathfrak{s}_i \cdot \alpha}(x_1,\dots,x_n),
\qquad
\alpha_i < \alpha_{i+1}.
\end{align}
Equation \eqref{tildeT-g<} can be used to evaluate the action of $\tilde{T}_i$ on $g^{*}_{\alpha}$ in the case $\alpha_i < \alpha_{i+1}$, however we also need to know how to act in the cases $\alpha_i = \alpha_{i+1}$ and $\alpha_i > \alpha_{i+1}$. In the case $\alpha_i = \alpha_{i+1}$, $g^{*}_{\alpha}$ is symmetric under the interchange of $x_i$ and $x_{i+1}$. Acting by $(1-\mathfrak{s}_i)$ will thus annihilate $g^{*}_{\alpha}$, and it follows that
\begin{align}
\label{tildeT-g=}
\tilde{T}_i \cdot g^{*}_{\alpha}(x_1,\dots,x_n)
=
q\cdot g^{*}_{\alpha}(x_1,\dots,x_n),
\quad
\alpha_i = \alpha_{i+1}.
\end{align}
To deduce what happens in the case $\alpha_i > \alpha_{i+1}$, we act on \eqref{tildeT-g<} by a further application of $\tilde{T}_i$. The reversed Hecke generators \eqref{rev-Hecke} satisfy the same quadratic relation as the one stated in \eqref{hecke1}: namely, $(\tilde{T}_i-q)(\tilde{T}_i+1) = 0$. This allows us to compute
\begin{align*}
\tilde{T}_i^2 \cdot g^{*}_{\alpha}(x_1,\dots,x_n)
=
(q-1) \tilde{T}_i \cdot g^{*}_{\alpha}(x_1,\dots,x_n)
+ 
q \cdot g^{*}_{\alpha}(x_1,\dots,x_n)
=
q\cdot \tilde{T}_i \cdot g^{*}_{\mathfrak{s}_i \cdot \alpha}(x_1,\dots,x_n).
\end{align*}
where $\alpha_i < \alpha_{i+1}$. Using once again \eqref{tildeT-g<} to calculate $\tilde{T}_i \cdot g^{*}_{\alpha}$, and relabelling $\alpha_i \leftrightarrow \alpha_{i+1}$, after further simplification we obtain
\begin{align}
\label{tildeT-g>}
\tilde{T}_i \cdot g^{*}_{\alpha}(x_1,\dots,x_n)
=
(q-1) g^{*}_{\alpha}(x_1,\dots,x_n)
+ 
g^{*}_{\mathfrak{s}_i \cdot \alpha}(x_1,\dots,x_n),
\quad
\alpha_i > \alpha_{i+1}.
\end{align}

Coming back to the proof of $\mathfrak{P}_1$, we combine \eqref{tildeT-g<}, \eqref{tildeT-g=} and \eqref{tildeT-g>} to write 
\begin{align}
\label{outcomes}
\tilde{T}_i \cdot g^{*}_{\nu}
=
\left\{
\begin{array}{ll}
q \cdot g^{*}_{\mathfrak{s}_i \cdot \nu},
\qquad\quad &
\nu_i \leq \nu_{i+1},
\\
\\
(q-1) \cdot g^{*}_{\nu} + g^{*}_{\mathfrak{s}_i \cdot \nu},
\qquad\quad &
\nu_i > \nu_{i+1}.
\end{array}
\right.
\end{align}
Thanks to the validity of $\mathfrak{P}_0$, the only function which pairs non-trivially with $f_{\delta}$ in the scalar product \eqref{P-ell=1} is $g^{*}_{\delta}$. Accordingly, we use \eqref{outcomes} to compute
\begin{align}
\label{outcomes2}
\langle f_{\delta}, \tilde{T}_{i} \cdot g^{*}_{\nu} \rangle_{q,s}
=
q \cdot \bm{1}_{\delta = \mathfrak{s}_i \cdot \nu} \cdot \bm{1}_{\nu_i \leq \nu_{i+1}}
+
(q-1) \cdot \bm{1}_{\delta = \nu} \cdot \bm{1}_{\nu_i > \nu_{i+1}}
+
\bm{1}_{\delta =  \mathfrak{s}_i \cdot \nu}
\cdot
\bm{1}_{\nu_i > \nu_{i+1}}.
\end{align}
In the case $\nu_i \leq \nu_{i+1}$, $\mathfrak{s}_i \cdot \nu$ can never equal $\delta$, since the former is either not anti-dominant or has its $i$-th and $(i+1)$-th parts equal, both of which are contradictory to our assumptions on $\delta$. Similarly, when $\nu_i > \nu_{i+1}$, $\nu$ cannot be equal to $\delta$. The only possible non-zero outcome on the right hand side of \eqref{outcomes2} is for 
$\nu_i > \nu_{i+1}$ and $\mathfrak{s}_i \cdot \nu = \delta$, when we recover precisely \eqref{P-ell=1}.

The general case $\mathfrak{P}_{\ell}$ can be proved by induction on $\ell$, where the inductive step takes a very similar form to the calculation laid out in \eqref{outcomes2}.

\end{proof}

\chapter{Plancherel isomorphisms}

The goal of this chapter is to develop a Plancherel theory based on the properties of the non-symmetric spin Hall--Littlewood functions. Much of the material extends previous results \cite{BorodinCPS,Borodin,BorodinP1,BorodinP2}, obtained in the context of the symmetric spin Hall--Littlewood functions, to the non-symmetric setting. 

Setting up this Plancherel theory involves finding a transform $\mathfrak{G}$ that maps functions valued on the discrete set $\mathbb{Z}^n$ to functions of the continuous variables $(x_1,\dots,x_n)$, as well as an inverse transform $\mathfrak{F}$ which pulls functions on $(x_1,\dots,x_n)$ back to functions on $\mathbb{Z}^n$; we require that the two maps compose as the identity. The tools which allow us to define such transforms are the summation identity \eqref{mimachi-id}, the orthogonality relation \eqref{f-g-orthog}, as well as a way of extending the definition of $f_{\mu}$ and $g_{\mu}$ to compositions with negative parts (see Section \ref{ssec:negative-mu}).

\section{Extending to compositions with negative parts}
\label{ssec:negative-mu}

\begin{prop}
\label{prop:f-shift}
Fix a positive integer $k$, and let $k^n = (k,\dots,k)$ be the composition with $n$ equal parts of size $k$. Then for any composition $\mu \in \mathbb{N}^n$, the non-symmetric spin Hall--Littlewood functions satisfy the shift property
\begin{align}
\label{f-shift}
f_{\mu+k^n}(x_1,\dots,x_n)
&=
\prod_{i=1}^{n}
\left( \frac{x_i-s}{1-sx_i} \right)^k
f_{\mu}(x_1,\dots,x_n),
\\
\label{g-shift}
g^{*}_{\mu+k^n}(x_1,\dots,x_n)
&=
\prod_{i=1}^{n}
\left( \frac{x_i-s}{1-sx_i} \right)^k
g^{*}_{\mu}(x_1,\dots,x_n).
\end{align}
\end{prop}

\begin{proof}
Let us take the partition function representation \eqref{f-def} of $f_{\mu}(x_1,\dots,x_n)$, and shift all parts of the composition $\mu$ by $k$. At the level of the partition function, this shifting corresponds to translating the outgoing coordinates of all lattice paths right by $k$ units. After performing this operation, one sees that the $k$ leftmost columns of the partition function are frozen to vertices of the form 
$\tikz{0.4}{
\draw[lgray,line width=1.5pt,->] (-1,0) -- (1,0);
\draw[lgray,line width=3pt,->] (0,-1) -- (0,1);
\node[left] at (-1,0) {\tiny $i$};\node[right] at (1,0) {\tiny $i$};
\node[below] at (0,-1) {\tiny $\bm{0}$};\node[above] at (0,1) {\tiny $\bm{0}$};
}$. Since these vertices have the Boltzmann weight $(x_i-s)/(1-sx_i)$ (they occur in the $i$-th row of the lattice), it follows that the $k$ leftmost columns produce the overall common factor $\prod_{i=1}^{n} (x_i-s)^k / (1-s x_i)^k$. The remaining columns reproduce $f_{\mu}(x_1,\dots,x_n)$ as previously; we thus deduce the relation \eqref{f-shift}.

The proof of \eqref{g-shift} follows in an analogous way by shifting the paths in the partition function representation \eqref{g-def} of $g_{\mu}(x_1,\dots,x_n)$ (see also \eqref{g-star}).
\end{proof}

The relations \eqref{f-shift} and \eqref{g-shift} allow us to extend the definition of the non-symmetric spin Hall--Littlewood functions to compositions with negative parts. To do so, one only needs to relax the constraint that $\mu \in \mathbb{N}^n$ to the situation $\mu \in \mathbb{Z}^n$. Proposition \ref{prop:f-shift} shows that the following definition is unambiguous.
\begin{defn}
Let $\mu$ be a composition in $\mathbb{Z}^n$, and choose a positive integer $k$ such that $\mu + k^n \in \mathbb{N}^n$. We let
\begin{align}
\label{f-shift2}
f_{\mu}(x_1,\dots,x_n)
:=
\prod_{i=1}^{n}
\left( \frac{1-sx_i}{x_i-s} \right)^k
f_{\mu + k^n}(x_1,\dots,x_n),
\\
\label{g-shift2}
g^{*}_{\mu}(x_1,\dots,x_n)
:=
\prod_{i=1}^{n}
\left( \frac{1-sx_i}{x_i-s} \right)^k
g^{*}_{\mu + k^n}(x_1,\dots,x_n).
\end{align}
\end{defn}

\section{Function spaces}

Before giving the Fourier-like transforms, let us firstly define the spaces on which they will act. It would be possible to relax the constraints on these spaces, to allow our transforms to act on a wider class of functions, but we shall refrain from making the most general statements available.

\index{C@$\mathcal{C}^n$; finite composition space}
\begin{defn}[Finitely supported functions on compositions]
We say that a function $\alpha : \mathbb{Z}^n \rightarrow \mathbb{C}$ is finitely supported if there exists a positive integer $N$ such that $\alpha(\mu) = 0$ for all compositions $\mu \in \mathbb{Z}^n$ with $\max(\mu) > N$ or $\min(\mu) < -N$. We denote the space of such functions by $\mathcal{C}^n$.
\end{defn}

\index{L@$\mathcal{L}^n$; Laurent polynomial space}
\begin{defn}[Laurent polynomial space]
We let $\mathcal{L}^n$ denote the space of all functions $\Phi : \mathbb{C}^n \rightarrow \mathbb{C}$ such that
\begin{itemize}
\item $\Phi(x_1,\dots,x_n)$ is a Laurent polynomial in each of the variables
\begin{align}
\label{laur-pol}
\frac{x_i-s}{1-sx_i},
\qquad
1 \leq i \leq n;
\end{align}

\item $\Phi(x_1,\dots,x_n)$ satisfies the vanishing constraints
\begin{align}
\label{vanish}
\lim_{x_i \rightarrow \infty} \Phi(x_1,\dots,x_n) = 0,
\qquad 
\text{for all}
\ \  
1 \leq i \leq n.
\end{align}
\end{itemize}
\end{defn}

As functions of $\mu$, neither of the non-symmetric spin Hall--Littlewood functions $f_{\mu}$ or $g^{*}_{\mu}$ live in the space $\mathcal{C}^n$, since the composition $\mu$ can be made to have arbitrarily large parts, and neither $f_{\mu}$ or 
$g^{*}_{\mu}$ will vanish. However, as functions of $(x_1,\dots,x_n)$, they both live in $\mathcal{L}^n$, as we now show.

\begin{prop}
\label{prop:laurent}
Fixing a composition $\mu \in \mathbb{Z}^n$, we have $f_{\mu}(x_1,\dots,x_n), g^{*}_{\mu}(x_1,\dots,x_n) \in \mathcal{L}^n$.
\end{prop}

\begin{proof}
Let us focus on proving this statement for $f_{\mu}(x_1,\dots,x_n)$; the proof is similar in the case of $g^{*}_{\mu}(x_1,\dots,x_n)$. For the first of the properties \eqref{laur-pol}, it suffices to demonstrate this at the level of the Boltzmann weights \eqref{s-weights}. Using the identity
\begin{align*}
\frac{1}{1-sx}
=
\frac{1}{1-s^2}
\left(
1+s\cdot \frac{x-s}{1-sx}
\right),
\end{align*}
we find that the six weights tabulated in \eqref{s-weights} can be rewritten as
\begin{gather*}
\tikz{0.6}{
\draw[lgray,line width=1.5pt,->] (-1,0) -- (1,0);
\draw[lgray,line width=4pt,->] (0,-1) -- (0,1);
\node[left] at (-1,0) {\tiny $0$};\node[right] at (1,0) {\tiny $0$};
\node[below] at (0,-1) {\tiny $\I$};\node[above] at (0,1) {\tiny $\I$};
}
=
\frac{1}{1-s^2}
\left(
1-s^2 q^{\Is{1}{n}}
+
s(1-q^{\Is{1}{n}})
\cdot
\frac{x-s}{1-sx}
\right),
\\
\tikz{0.6}{
\draw[lgray,line width=1.5pt,->] (-1,0) -- (1,0);
\draw[lgray,line width=4pt,->] (0,-1) -- (0,1);
\node[left] at (-1,0) {\tiny $i$};\node[right] at (1,0) {\tiny $i$};
\node[below] at (0,-1) {\tiny $\I$};\node[above] at (0,1) {\tiny $\I$};
}
=
\frac{q^{\Is{i+1}{n}}}{1-s^2}
\left(
s-sq^{I_i}
+
(1-s^2 q^{I_i})
\cdot
\frac{x-s}{1-sx}
\right),
\\
\tikz{0.6}{
\draw[lgray,line width=1.5pt,->] (-1,0) -- (1,0);
\draw[lgray,line width=4pt,->] (0,-1) -- (0,1);
\node[left] at (-1,0) {\tiny $0$};\node[right] at (1,0) {\tiny $i$};
\node[below] at (0,-1) {\tiny $\I$};\node[above] at (0,1) {\tiny $\I^{-}_i$};
}
=
\frac{(1-q^{I_i})q^{\Is{i+1}{n}}}{1-s^2}
\left(
s
+
\frac{x-s}{1-sx}
\right),
\quad
\tikz{0.6}{
\draw[lgray,line width=1.5pt,->] (-1,0) -- (1,0);
\draw[lgray,line width=4pt,->] (0,-1) -- (0,1);
\node[left] at (-1,0) {\tiny $i$};\node[right] at (1,0) {\tiny $0$};
\node[below] at (0,-1) {\tiny $\I$};\node[above] at (0,1) {\tiny $\I^{+}_i$};
}
=
\frac{1-s^2 q^{\Is{1}{n}}}{1-s^2}
\left(
1
+
s\cdot
\frac{x-s}{1-sx}
\right),
\\
\tikz{0.6}{
\draw[lgray,line width=1.5pt,->] (-1,0) -- (1,0);
\draw[lgray,line width=4pt,->] (0,-1) -- (0,1);
\node[left] at (-1,0) {\tiny $i$};\node[right] at (1,0) {\tiny $j$};
\node[below] at (0,-1) {\tiny $\I$};\node[above] at (0,1) {\tiny $\I^{+-}_{ij}$};
}
=
\frac{(1-q^{I_j})q^{\Is{j+1}{n}}}{1-s^2}
\left(
s
+
\frac{x-s}{1-sx}
\right),
\quad
\tikz{0.6}{
\draw[lgray,line width=1.5pt,->] (-1,0) -- (1,0);
\draw[lgray,line width=4pt,->] (0,-1) -- (0,1);
\node[left] at (-1,0) {\tiny $j$};\node[right] at (1,0) {\tiny $i$};
\node[below] at (0,-1) {\tiny $\I$};\node[above] at (0,1) {\tiny $\I^{+-}_{ji}$};
}
=
\frac{s(1-q^{I_i})q^{\Is{i+1}{n}}}{1-s^2}
\left(
1
+
s\cdot
\frac{x-s}{1-sx}
\right).
\end{gather*}
It follows that for $\mu \in \mathbb{N}^n$ the function $f_{\mu}(x_1,\dots,x_n)$, which is built directly out of these weights, is a polynomial in $(x_i-s)/(1-sx_i)$ for each $1 \leq i \leq n$. Relaxing this to the case $\mu \in \mathbb{Z}^n$ using \eqref{f-shift2}, we clearly obtain the Laurent polynomial statement \eqref{laur-pol}.

The second statement, \eqref{vanish}, also follows from examination of the weights \eqref{s-weights}. Each of the six weights \eqref{s-weights} has a well defined limit as $x \rightarrow \infty$, and in particular, 
$\tikz{0.4}{
\draw[lgray,line width=1.5pt,->] (-1,0) -- (1,0);
\draw[lgray,line width=3pt,->] (0,-1) -- (0,1);
\node[left] at (-1,0) {\tiny $i$};\node[right] at (1,0) {\tiny $0$};
\node[below] at (0,-1) {\tiny $\bm{I}$};\node[above] at (0,1) {\tiny $\bm{I}^{+}_i$};
}$ 
tends to zero in this limit. Using the partition function representation \eqref{f-def} of $f_{\mu}(x_1,\dots,x_n)$, one readily sees that every row of the lattice must contain at least one vertex of the form $\tikz{0.4}{
\draw[lgray,line width=1.5pt,->] (-1,0) -- (1,0);
\draw[lgray,line width=3pt,->] (0,-1) -- (0,1);
\node[left] at (-1,0) {\tiny $i$};\node[right] at (1,0) {\tiny $0$};
\node[below] at (0,-1) {\tiny $\bm{I}$};\node[above] at (0,1) {\tiny $\bm{I}^{+}_i$};
}$, and accordingly the partition function will vanish under any of the limits $x_i \rightarrow \infty$, $1 \leq i \leq n$. This proves the property \eqref{vanish}.

\end{proof}

\section{Forward transform $\mathfrak{G}$ and inverse transform $\mathfrak{F}$}

\begin{defn}
Fix a function $\alpha \in \mathcal{C}^n$ and let $(x_1,\dots,x_n)$ be a collection of $n$ complex parameters. We define a forward transform 
$\mathfrak{G} : \mathcal{C}^n \rightarrow \mathcal{L}^n$ as follows:
\begin{align}
\label{forward-trans}
\index{G@$\mathfrak{G}$; forward transform}
\mathfrak{G}[\alpha](x_1,\dots,x_n)
:=
\sum_{\mu \in \mathbb{Z}^n}
\alpha(\mu)
g^{*}_{\mu}(x_1,\dots,x_n),
\end{align}
where the right hand side is in fact a finite sum, given the finiteness of the support of $\alpha$. The fact that the right hand side lives in $\mathcal{L}^n$ is manifest, given Proposition \ref{prop:laurent}.
\end{defn}

\begin{defn}
\label{defn:inverse-trans}
Fix a function $\Phi \in \mathcal{L}^n$, and let $\mu \in \mathbb{Z}^n$ be an integer composition. We define an inverse transform $\mathfrak{F} : \mathcal{L}^n \rightarrow \mathcal{C}^n$ as follows: 
\begin{align}
\label{inverse-trans}
\index{F@$\mathfrak{F}$; inverse transform}
\mathfrak{F}[\Phi](\mu)
&:=
\langle f_{\mu}, \Phi \rangle_{q,s}
\\
&=
\left( \frac{1}{2\pi\sqrt{-1}} \right)^n
\oint_{C_1}
\frac{dx_1}{x_1}
\cdots 
\oint_{C_n}
\frac{dx_n}{x_n}
\prod_{1 \leq i<j \leq n}
\left(
\frac{x_j-x_i}{x_j-q x_i}
\right)
f_{\mu}(\b{x}_1,\dots,\b{x}_n)
\Phi(x_1,\dots,x_n),
\nonumber
\end{align}
where $\{C_1,\dots,C_n\}$ is an admissible set of contours, in the sense of Definition \ref{def:admiss}. The fact that the right hand side of \eqref{inverse-trans} lives in $\mathcal{C}^n$ needs some further explanation:
\end{defn}

\begin{prop}
\label{prop:finite-mu}
With the same assumptions as in Definition \ref{defn:inverse-trans}, the right hand side of equation \eqref{inverse-trans} is a finitely supported function on compositions $\mu \in \mathbb{Z}^n$.
\end{prop}

\begin{proof}
Let us begin by proving this in the case of anti-dominant compositions; namely, we restrict to the situation $\mu = \delta = (\delta_1 \leq \cdots \leq \delta_n) \in \mathbb{Z}^n$. Recall the explicit factorized form \eqref{f-delta} of $f_{\delta}$, and observe that this expression continues to hold for anti-dominant compositions $\delta \in \mathbb{Z}^n$, which is easily seen from the shift formula \eqref{f-shift2}. We thus have
\begin{multline}
\label{delta-finite}
\mathfrak{F}[\Phi](\delta)
=
\prod_{j \geq 0} (s^2;q)_{m_j(\delta)}
\left( \frac{1}{2\pi\sqrt{-1}} \right)^n
\\
\times
\oint_{C_1}
dx_1
\cdots 
\oint_{C_n}
dx_n
\prod_{1 \leq i<j \leq n}
\left(
\frac{x_j-x_i}{x_j-q x_i}
\right)
\prod_{i=1}^{n}
\frac{(1-sx_i)^{\delta_i}}{(x_i-s)^{\delta_i+1}}\,
\Phi(x_1,\dots,x_n).
\end{multline}
Given that $\Phi \in \mathcal{L}^n$, there exists a sufficiently large positive integer $M$ such that for all $k\geq M$ the expression $\prod_{i=1}^{n} (x_i-s)^k/(1-sx_i)^k \Phi(x_1,\dots,x_n)$ is polynomial in the variables $(x_i-s)/(1-sx_i)$, 
$1 \leq i \leq n$. Then for any anti-dominant composition $\delta$ with smallest part $\delta_1 < -M$, we see that the integrand of \eqref{delta-finite} has no pole at $x_1 = s$, and therefore the integration over contour $C_1$ vanishes.

Similarly, by virtue of the fact that $\Phi \in \mathcal{L}^n$, there exists a sufficiently large positive integer $N$ such that for all $k\geq N$ the expression $\prod_{i=1}^{n} (1-sx_i)^k/(x_i-s)^k \Phi(x_1,\dots,x_n)$ is polynomial in the variables $(1-sx_i)/(x_i-s)$, $1 \leq i \leq n$. Then for any composition $\delta$ with largest part $\delta_n \geq N$, the integrand of \eqref{delta-finite} has no pole at $x_n = s^{-1}$. Shrinking the contour $C_n$ to the point $x_n = s^{-1}$, noting that there is no residue at infinity (this follows from the fact that $\Phi \rightarrow 0$ as $x_n \rightarrow \infty$), we find that the integral again vanishes.

We conclude that that $\mathfrak{F}[\Phi](\delta)$ is finitely supported on anti-dominant compositions $\delta$. To pass to generic compositions $\mu$, let us recall the formulae \eqref{reduced}--\eqref{T-fdelta}. They continue to apply in the case of compositions $\mu \in \mathbb{Z}^n$, given that the product 
$\prod_{i=1}^{n} (1-s x_i)^k/(x_i-s)^k$ used to effect the shift \eqref{f-shift2} is symmetric in $(x_1,\dots,x_n)$, and therefore commutes with the action of Hecke generators. Using \eqref{adjoint} we can then write
\begin{align}
\label{8.3-1}
\mathfrak{F}[\Phi](\mu)
=
\langle f_{\mu}, \Phi \rangle_{q,s}
=
\langle T_{\sigma} \cdot f_{\delta}, \Phi \rangle_{q,s}
=
\langle f_{\delta}, \tilde{T}_{\sigma} \cdot \Phi \rangle_{q,s},
\end{align}
where $T_{\sigma} = T_{i_{\ell(\sigma)}} \cdots T_{i_1}$ and $\tilde{T}_{\sigma} = \tilde{T}_{i_1} \cdots \tilde{T}_{i_{\ell(\sigma)}}$, $\sigma$ is the minimal-length permutation such that $\mu_i = \delta_{\sigma(i)}$ for all $1 \leq i \leq n$, and where $\sigma$ has the reduced-word decomposition \eqref{reduced}. Now one can verify that the action of any generator $\tilde{T}_i$ preserves the property of being in $\mathcal{L}^n$; it follows that $\tilde{T}_{\sigma} \cdot \Phi \in \mathcal{L}^n$, and accordingly, \eqref{8.3-1} vanishes unless $\delta$ is within a finite interval. This proves that $\mathfrak{F}[\Phi](\mu) = 0$ unless $\mu$ has bounded parts. 
\end{proof}

\section{Plancherel isomorphisms}

\begin{thm}
\label{thm-plancherel}
The maps $\mathfrak{F} \circ \mathfrak{G} : \mathcal{C}^n \rightarrow \mathcal{C}^n$ and $\mathfrak{G} \circ \mathfrak{F} : \mathcal{L}^n \rightarrow \mathcal{L}^n$ both act as the identity; \ie\ there holds
\begin{align}
\label{compose1}
\mathfrak{F} \circ \mathfrak{G} 
=
{\rm id}
\in 
{\rm End}(\mathcal{C}^n),
\end{align}
where the identity acts on the space $\mathcal{C}^n$, and similarly,
\begin{align}
\label{compose2}
\mathfrak{G} \circ \mathfrak{F} 
=
{\rm id}
\in 
{\rm End}(\mathcal{L}^n),
\end{align}
where the identity acts on the space $\mathcal{L}^n$.
\end{thm}

\begin{proof}
We begin by proving the first of these statements, \eqref{compose1}. Taking a finitely supported function $\alpha \in \mathcal{C}^n$, we act on it with $\mathfrak{F} \circ \mathfrak{G}$, producing
\begin{multline*}
\mathfrak{F} \circ \mathfrak{G} [\alpha](\mu)
=
\mathfrak{F}\left[  
\sum_{\nu \in \mathbb{Z}^n}
\alpha(\nu)
g^{*}_{\nu}(x_1,\dots,x_n) \right](\mu)
\\
=
\left( \frac{1}{2\pi\sqrt{-1}} \right)^n
\oint_{C_1}
\frac{dx_1}{x_1}
\cdots 
\oint_{C_n}
\frac{dx_n}{x_n}
\prod_{1 \leq i<j \leq n}
\left(
\frac{x_j-x_i}{x_j-q x_i}
\right)
f_{\mu}(\b{x}_1,\dots,\b{x}_n)
\sum_{\nu \in \mathbb{Z}^n}
\alpha(\nu)
g^{*}_{\nu}(x_1,\dots,x_n).
\end{multline*}
We may interchange the (finite) summation and integration, leading to
\begin{multline}
\label{8.7-1}
\mathfrak{F} \circ \mathfrak{G} [\alpha](\mu)
\\
=
\sum_{\nu \in \mathbb{Z}^n}
\alpha(\nu)
\left( \frac{1}{2\pi\sqrt{-1}} \right)^n
\oint_{C_1}
\frac{dx_1}{x_1}
\cdots 
\oint_{C_n}
\frac{dx_n}{x_n}
\prod_{1 \leq i<j \leq n}
\left(
\frac{x_j-x_i}{x_j-q x_i}
\right)
f_{\mu}(\b{x}_1,\dots,\b{x}_n)
g^{*}_{\nu}(x_1,\dots,x_n).
\end{multline}
To conclude, we would like to use the orthogonality relation \eqref{f-g-orthog} to collapse the summation in \eqref{8.7-1} to a single term, namely $\nu = \mu$. The only potential point of subtlety is that the compositions $\mu,\nu$ in \eqref{8.7-1} now live in $\mathbb{Z}^n$, whereas \eqref{f-g-orthog} was proved only for the case $\mu, \nu \in \mathbb{N}^n$. To resolve this issue, note that there exists a positive integer $k$ such that both $\mu + k^n, \nu + k^n \in \mathbb{N}^n$ for all $\nu$ in \eqref{8.7-1}, due to the finiteness of the support of $\alpha$. We may then apply the shift formulae \eqref{f-shift2} and \eqref{g-shift2} to write
\begin{multline}
\label{8.7-2}
\mathfrak{F} \circ \mathfrak{G} [\alpha](\mu)
\\
=
\sum_{\nu \in \mathbb{Z}^n}
\alpha(\nu)
\left( \frac{1}{2\pi\sqrt{-1}} \right)^n
\oint_{C_1}
\frac{dx_1}{x_1}
\cdots 
\oint_{C_n}
\frac{dx_n}{x_n}
\prod_{1 \leq i<j \leq n}
\left(
\frac{x_j-x_i}{x_j-q x_i}
\right)
f_{\mu+k^n}(\b{x}_1,\dots,\b{x}_n)
g^{*}_{\nu+k^n}(x_1,\dots,x_n),
\end{multline}
where we have made use of the cancellation
\begin{align*}
\prod_{i=1}^{n}
\left(
\frac{1-s\b{x}_i}{\b{x}_i-s}
\right)^k
\left(
\frac{1-sx_i}{x_i-s}
\right)^k
=1.
\end{align*}
We may then freely apply \eqref{f-g-orthog} to \eqref{8.7-2}, yielding
\begin{align*}
\mathfrak{F} \circ \mathfrak{G} [\alpha](\mu)
=
\sum_{\nu \in \mathbb{Z}^n}
\alpha(\nu)
{\bm 1}_{\mu = \nu}
=
\alpha(\mu),
\end{align*}
which proves \eqref{compose1}. 

The proof of the second statement, \eqref{compose2}, is slightly more intricate. Taking a function $\Phi \in \mathcal{L}^n$, we act on it with $\mathfrak{G} \circ \mathfrak{F}$, which yields
\begin{multline*}
\mathfrak{G} \circ \mathfrak{F} [\Phi]
(y_1,\dots,y_n)
\\
=
\mathfrak{G} \left[
\left( \frac{1}{2\pi\sqrt{-1}} \right)^n
\oint_{C_1}
\frac{dx_1}{x_1}
\cdots 
\oint_{C_n}
\frac{dx_n}{x_n}
\prod_{1 \leq i<j \leq n}
\left(
\frac{x_j-x_i}{x_j-q x_i}
\right)
f_{\mu}(\b{x}_1,\dots,\b{x}_n)
\Phi(x_1,\dots,x_n)
\right]
(y_1,\dots,y_n)
\\
=
\left( \frac{1}{2\pi\sqrt{-1}} \right)^n
\sum_{\mu \in \mathbb{Z}^n}
g^{*}_{\mu}(y_1,\dots,y_n)
\oint_{C_1}
\frac{dx_1}{x_1}
\cdots 
\oint_{C_n}
\frac{dx_n}{x_n}
\prod_{1 \leq i<j \leq n}
\left(
\frac{x_j-x_i}{x_j-q x_i}
\right)
f_{\mu}(\b{x}_1,\dots,\b{x}_n)
\Phi(x_1,\dots,x_n).
\end{multline*}
Given that the $n$ integrations return a finitely-supported function
\begin{align*}
\alpha(\mu)
=
\left( \frac{1}{2\pi\sqrt{-1}} \right)^n
\oint_{C_1}
\frac{dx_1}{x_1}
\cdots 
\oint_{C_n}
\frac{dx_n}{x_n}
\prod_{1 \leq i<j \leq n}
\left(
\frac{x_j-x_i}{x_j-q x_i}
\right)
f_{\mu}(\b{x}_1,\dots,\b{x}_n)
\Phi(x_1,\dots,x_n)
\in
\mathcal{C}^n,
\end{align*}
we can truncate the summation over $\mu \in \mathbb{Z}^n$, and instead sum each part $\mu_i$ over $-k \leq \mu_i < \infty$, for an appropriately chosen positive integer $k$. Making this truncation, and switching the order of the summation and integration, we read
\begin{multline*}
\mathfrak{G} \circ \mathfrak{F} [\Phi]
(y_1,\dots,y_n)
=
\left( \frac{1}{2\pi\sqrt{-1}} \right)^n
\\
\times
\oint_{C_1}
\frac{dx_1}{x_1}
\cdots 
\oint_{C_n}
\frac{dx_n}{x_n}
\prod_{1 \leq i<j \leq n}
\left(
\frac{x_j-x_i}{x_j-q x_i}
\right)
\Phi(x_1,\dots,x_n)
\sum_{\mu: \mu + k^n \in \mathbb{N}^n}
f_{\mu}(\b{x}_1,\dots,\b{x}_n)
g^{*}_{\mu}(y_1,\dots,y_n).
\end{multline*}
Assuming that the variables $(y_1,\dots,y_n)$ are all enclosed within 
$C_1$ (and hence all other contours), we can ensure that
\begin{align*}
\left|
\frac{1-sx_i}{x_i-s}
\cdot
\frac{y_j-s}{1-sy_j}
\right|
<
1,
\qquad
\forall\ 
1 \leq i,j \leq n,
\end{align*}
which are the necessary convergence requirements for the sum inside the integrand.\footnote{Other values of $(y_1,\dots,y_n)$ can be reached by analytic continuation, once the result is attained.} Applying the shifts \eqref{f-shift2}, \eqref{g-shift2} and the summation identity \eqref{mimachi-id}, we thus conclude that
\begin{multline*}
\mathfrak{G} \circ \mathfrak{F} [\Phi]
(y_1,\dots,y_n)
=
\left( \frac{1}{2\pi\sqrt{-1}} \right)^n
\oint_{C_1}
dx_1
\cdots 
\oint_{C_n}
dx_n
\\
\times
\prod_{1 \leq i<j \leq n}
\left(
\frac{x_j-x_i}{x_j-q x_i}
\right)
\Phi(x_1,\dots,x_n)
\prod_{i=1}^{n}
\left( \frac{x_i-s}{1-sx_i} \right)^k
\left( \frac{1-sy_i}{y_i-s} \right)^k
\frac{1}{x_i-y_i}
\prod_{n \geq i > j \geq 1}
\frac{x_i - q y_j}{x_i - y_j}.
\end{multline*}
We are finally able to directly compute the integrals, following a very similar procedure to the computation in Step 2 of the proof of Theorem \ref{thm:orthog}. Beginning with the integration over $C_1$, because of the fact that $\Phi(x_1,\dots,x_n) \in \mathcal{L}^n$, if $k$ is large enough we easily see that the only pole enclosed by $C_1$ is at $x_1 = y_1$. Computing its residue, we see that
\begin{multline*}
\mathfrak{G} \circ \mathfrak{F} [\Phi]
(y_1,\dots,y_n)
=
\left( \frac{1}{2\pi\sqrt{-1}} \right)^{n-1}
\oint_{C_2}
dx_2
\cdots 
\oint_{C_n}
dx_n
\\
\times
\prod_{2 \leq i<j \leq n}
\left(
\frac{x_j-x_i}{x_j-q x_i}
\right)
\Phi(y_1,x_2,\dots,x_n)
\prod_{i=2}^{n}
\left( \frac{x_i-s}{1-sx_i} \right)^k
\left( \frac{1-sy_i}{y_i-s} \right)^k
\frac{1}{x_i-y_i}
\prod_{n \geq i > j \geq 2}
\frac{x_i - q y_j}{x_i - y_j},
\end{multline*}
after cancelling out some factors in the integrand. This computation can clearly be iterated over the remaining contours; at each stage we only encounter a pole at $x_i = y_i$, whose residue can be taken in the same way. The final answer is
\begin{align*}
\mathfrak{G} \circ \mathfrak{F} [\Phi]
(y_1,\dots,y_n)
=
\Phi(y_1,\dots,y_n),
\end{align*}
which completes the proof of \eqref{compose2}.

\end{proof}

\section{An integral formula for $G_{\mu/\nu}$}

The symmetric functions $G_{\mu/\nu}$ introduced in Section \ref{ssec:G} also have a natural extension to integer compositions $\mu,\nu \in \mathbb{Z}^n$: it can be obtained simply by padding the partition function \eqref{G-pf} on its left with infinitely many empty columns. With this understanding of $G_{\mu/\nu}$, we observe that for all $\mu \in \mathbb{Z}^n$ one has
\begin{align*}
\sum_{\nu}
(-s)^{|\nu|-|\mu|}
G_{\mu/\nu}(x_1,\dots,x_p)
=1,
\end{align*}
where the sum is taken over all compositions $\nu \in \mathbb{Z}^n$ such that $\nu_i \leq \mu_i$, for all $1 \leq i \leq n$. Indeed, this readily follows from repeated application of the stochasticity property \eqref{stoch-wt}, \eqref{LM-stoch} of $\tilde{M}_x$ (where the introduced factor $(-s)^{|\nu|-|\mu|}$ implements the required stochastic gauge transformation of $G_{\mu/\nu}$). Proceeding along similar lines to \cite[Section 6]{BorodinP2}, it is thus possible to view $(-s)^{|\nu|-|\mu|} G_{\mu/\nu}$ as a multivariate Markov kernel, giving the transition probability from $\mu \rightarrow \nu$.

Now as an application of the preceding Plancherel theory, let us write down an integral formula for the matrix entries $G_{\mu/\nu}$; namely, we give its spectral decomposition. Our starting point is the skew Cauchy identity \eqref{fG-skew} (extended to integer compositions), written with a reciprocated alphabet $(\b{x}_1,\dots,\b{x}_n)$:
\begin{align}
\label{7.3-1}
\sum_{\kappa}
f_{\kappa}(\b{x}_1,\dots,\b{x}_n)
G_{\kappa/\nu}(y_1,\dots,y_p)
=
q^{-np}
\prod_{i=1}^{n}
\prod_{j=1}^{p}
\frac{x_i - q y_j}{x_i - y_j}
f_{\nu}(\b{x}_1,\dots,\b{x}_n),
\end{align} 
where $\nu$ is an arbitrary composition. Fixing a further composition $\mu$, we multiply both sides of \eqref{7.3-1} by $g^{*}_{\mu}(x_1,\dots,x_n)$ before integrating against the same measure and over the same contours as in the definition of the scalar product \eqref{scalar}.\footnote{Note that one is able to choose the contours $C_1,\dots,C_n$ such that the left hand side of \eqref{7.3-1} continues to converge; a sufficient choice would be that of Figure \ref{fig:contours}, with all points $(y_1,\dots,y_p)$ inside $C_1$ and close to the point $s$, when one has \begin{align*} \left| \frac{1-sx_i}{x_i-s} \cdot \frac{y_j-s}{1-sy_j} \right| < 1,
\qquad \forall\ 1 \leq i \leq n, \ \ 1 \leq j \leq p. \end{align*}}  The result of the calculation is
\begin{multline*}
\sum_{\kappa}
G_{\kappa/\nu}(y_1,\dots,y_p)
\langle f_{\kappa},g^{*}_{\mu} \rangle_{q,s}
\\
=
\frac{q^{-np}}{(2\pi\sqrt{-1})^n}
\oint_{C_1}
\frac{dx_1}{x_1}
\cdots 
\oint_{C_n}
\frac{dx_n}{x_n}
\prod_{1 \leq i<j \leq n}
\left(
\frac{x_j-x_i}{x_j-q x_i}
\right)
\prod_{i=1}^{n}
\prod_{j=1}^{p}
\frac{x_i - q y_j}{x_i - y_j}
f_{\nu}(\b{x}_1,\dots,\b{x}_n)
g^{*}_{\mu}(x_1,\dots,x_n),
\end{multline*}
and using the orthogonality result \eqref{f-g-orthog} to collapse the sum on the left hand side,
\begin{multline}
\label{Gmunu-int}
G_{\mu/\nu}(y_1,\dots,y_p)
\\
=
\frac{q^{-np}}{(2\pi\sqrt{-1})^n}
\oint_{C_1}
\frac{dx_1}{x_1}
\cdots 
\oint_{C_n}
\frac{dx_n}{x_n}
\prod_{1 \leq i<j \leq n}
\left(
\frac{x_j-x_i}{x_j-q x_i}
\right)
\prod_{i=1}^{n}
\prod_{j=1}^{p}
\frac{x_i - q y_j}{x_i - y_j}
f_{\nu}(\b{x}_1,\dots,\b{x}_n)
g^{*}_{\mu}(x_1,\dots,x_n),
\end{multline}
where the integration contours $C_1,\dots,C_n$ are assumed to be admissible. Note that the steps we have just followed can be viewed as nothing other than computing the action of $\mathfrak{F} \circ \mathfrak{G}$ on $G_{\mu/\nu}$.

Appropriate limits of the (multivariate) transfer matrix $(-s)^{|\nu|-|\mu|} G_{\mu/\nu}$ lead to the generators for a variety of continuous time Markov processes, such as the multi-species asymmetric simple exclusion process (mASEP). We expect that the formula \eqref{Gmunu-int} would be very helpful to extract expressions for averages of suitable observables in such processes, given appropriate initial data, such as those recently obtained in \cite{Kuan2}. This is certainly an interesting topic for further investigation, but we will not pursue it here.

\begin{rmk}
If we choose $\mu = (\mu_1 \geq \dots \geq \mu_n)$ and $\nu = 0^n$, then by virtue of \eqref{f-delta} and \eqref{g-lambda} the integral formula \eqref{Gmunu-int} simplifies to
\begin{multline}
G_{\mu}(y_1,\dots,y_p)
=
\frac{q^{-np} (s^2;q)_n}{(2\pi\sqrt{-1})^n}
\\
\times
\oint_{C_1}
dx_1
\cdots 
\oint_{C_n}
dx_n
\prod_{1 \leq i<j \leq n}
\left(
\frac{x_j-x_i}{x_j-q x_i}
\right)
\prod_{i=1}^{n}
\prod_{j=1}^{p}
\frac{x_i - q y_j}{x_i - y_j}
\prod_{i=1}^{n}
\frac{1}{(x_i-s)(1-sx_i)}
\left(
\frac{x_i-s}{1-sx_i}
\right)^{\mu_i},
\end{multline}
which matches the expression for ${\sf G}_{\mu}(y_1,\dots,y_p)$ obtained in \cite[Corollary 7.14]{BorodinP2}, up to inversion of integration variables and contours.

\end{rmk}

\begin{rmk}
\label{rmk:exchange}
Another interesting, slightly more general reduction of \eqref{Gmunu-int} is obtained by choosing $\mu = (\mu_1 \geq \cdots \geq \mu_n)$ and $\nu = (\nu_1 \leq \cdots \leq \nu_n)$, when we have
\begin{multline}
G_{\mu/\nu}(y_1,\dots,y_p)
=
\frac{q^{-np}}{(2\pi\sqrt{-1})^n}
\prod_{j \geq 0} (s^2;q)_{m_j(\nu)}
\\
\times
\oint_{C_1}
dx_1
\cdots 
\oint_{C_n}
dx_n
\prod_{1 \leq i<j \leq n}
\left(
\frac{x_j-x_i}{x_j-q x_i}
\right)
\prod_{i=1}^{n}
\prod_{j=1}^{p}
\frac{x_i - q y_j}{x_i - y_j}
\prod_{i=1}^{n}
\frac{1}{(x_i-s)(1-sx_i)}
\left(
\frac{x_i-s}{1-sx_i}
\right)^{\mu_i-\nu_i}.
\end{multline}
Physically, $(-s)^{|\nu|-|\mu|}G_{\mu/\nu}$ gives the probability that particles $\{1,\dots,n\}$ which enter the base of the lattice \eqref{G-pf} with coordinates $\mu_1 \geq \cdots \geq \mu_n$ exit it with coordinates $\nu_1 \leq \cdots \leq \nu_n$, \ie\ it refers to the event that the particles totally reverse their ordering. Similar examples of ``total exchange'' probabilities in a two-species model were studied very recently in \cite{ChenGHS}.

\end{rmk}

\chapter{Matching distributions}

In this chapter we will begin by considering two different partition functions, in unfused vertex models (meaning that there is at most one lattice path per horizontal and vertical edge). Our aim is to prove a direct match between these quantities. To our best knowledge this correspondence is new, and its conceptual origin remains mysterious to us.

Having established this result, it is straightforward to extend it to the higher-spin setting (by fusion), and to then take degenerations to various one-dimensional interacting particle systems. This will be the subject of Chapter \ref{sec:reduction}.

\section{Path distributions in the stochastic six-vertex model, $\mathbb{P}_{{\rm 6v}}(\mathcal{I},\mathcal{J})$}
\label{ssec:Z_MN}

Our first random object arises in the stochastic six-vertex model. The stochastic six-vertex model is recovered as the $n=1$ version of the model \eqref{fund-vert}, when the configuration space of an edge becomes two-dimensional (the set of states is $\{{\rm unoccupied,occupied}\}$ or, equivalently, $\{0,1\}$). In this case, we obtain the weights tabulated below: 
\begin{align}
\label{six-vert}
\begin{tabular}{|c|c|c|}
\hline
\quad
\tikz{0.6}{
\draw[lgray,line width=1.5pt,->] (-1,0) -- (1,0);
\draw[lgray,line width=1.5pt,->] (0,-1) -- (0,1);
\node[left] at (-1,0) {\tiny $0$};\node[right] at (1,0) {\tiny $0$};
\node[below] at (0,-1) {\tiny $0$};\node[above] at (0,1) {\tiny $0$};
}
\quad
&
\quad
\tikz{0.6}{
\draw[lgray,line width=1.5pt,->] (-1,0) -- (1,0);
\draw[lgray,line width=1.5pt,->] (0,-1) -- (0,1);
\node[left] at (-1,0) {\tiny $0$};\node[right] at (1,0) {\tiny $0$};
\node[below] at (0,-1) {\tiny $1$};\node[above] at (0,1) {\tiny $1$};
}
\quad
&
\quad
\tikz{0.6}{
\draw[lgray,line width=1.5pt,->] (-1,0) -- (1,0);
\draw[lgray,line width=1.5pt,->] (0,-1) -- (0,1);
\node[left] at (-1,0) {\tiny $0$};\node[right] at (1,0) {\tiny $1$};
\node[below] at (0,-1) {\tiny $1$};\node[above] at (0,1) {\tiny $0$};
}
\quad
\\[1.3cm]
\quad
$1$
\quad
& 
\quad
$\dfrac{q(1-xy)}{1-qxy}$
\quad
& 
\quad
$\dfrac{1-q}{1-qxy}$
\quad
\\[0.7cm]
\hline
\quad
\tikz{0.6}{
\draw[lgray,line width=1.5pt,->] (-1,0) -- (1,0);
\draw[lgray,line width=1.5pt,->] (0,-1) -- (0,1);
\node[left] at (-1,0) {\tiny $1$};\node[right] at (1,0) {\tiny $1$};
\node[below] at (0,-1) {\tiny $1$};\node[above] at (0,1) {\tiny $1$};
}
\quad
&
\quad
\tikz{0.6}{
\draw[lgray,line width=1.5pt,->] (-1,0) -- (1,0);
\draw[lgray,line width=1.5pt,->] (0,-1) -- (0,1);
\node[left] at (-1,0) {\tiny $1$};\node[right] at (1,0) {\tiny $1$};
\node[below] at (0,-1) {\tiny $0$};\node[above] at (0,1) {\tiny $0$};
}
\quad
&
\quad
\tikz{0.6}{
\draw[lgray,line width=1.5pt,->] (-1,0) -- (1,0);
\draw[lgray,line width=1.5pt,->] (0,-1) -- (0,1);
\node[left] at (-1,0) {\tiny $1$};\node[right] at (1,0) {\tiny $0$};
\node[below] at (0,-1) {\tiny $0$};\node[above] at (0,1) {\tiny $1$};
}
\quad
\\[1.3cm]
\quad
$1$
\quad
& 
\quad
$\dfrac{1-xy}{1-qxy}$
\quad
&
\quad
$\dfrac{(1-q)xy}{1-qxy}$
\quad 
\\[0.7cm]
\hline
\end{tabular}
\end{align}
where $x^{-1}$ is the rapidity associated to the horizontal line, and $y$ is the rapidity associated to the vertical line. Notice that in this chapter we invert horizontal rapidities, to produce weights \eqref{six-vert} in which $x$ and $y$ appear on a symmetric footing. 

We will consider the stochastic six-vertex model in a finite rectangular region of the first quadrant; namely, inside of the $M \times N$ lattice whose vertices are labelled by coordinate pairs $(i,j)$ with $1 \leq i \leq M$ (horizontal coordinate) and $1 \leq j \leq N$ (vertical coordinate). Edges are represented by an arrow linking two vertices; $(i,j) \rightarrow (i+1,j)$ indicates a horizontal edge, while $(i,j) \rightarrow (i,j+1)$ denotes a vertical edge. 

In view of the stochastic property \eqref{stoch} of the vertex weights \eqref{six-vert}, following \cite[Section 2]{BorodinCG} one can construct a discrete-time Markov process in the quadrant. We firstly specify the initial conditions for the process. Fix the left external horizontal edges of the quadrant, $(0,j) \rightarrow (1,j)$ for all $1 \leq j \leq N$, to be occupied by incoming paths, and similarly, fix the bottom external vertical edges, $(i,0) \rightarrow (i,1)$ for all $1 \leq i \leq M$, to be unoccupied; these are the so-called {\it domain wall boundary conditions.} Next, assume that we already have a probability distribution on path configurations restricted to the vertices $(a,b)$ such that $a+b < k$, for some $k \geq 2$. For each vertex $(i,j)$ such that $i+j = k$, a configuration on the vertices $\{(a,b)\}_{a+b < k}$ (together with the domain wall boundary conditions) specifies the states on the left and bottom edges of the vertex $(i,j)$. That information allows us to fill out the right and top edges of the vertex, sampling from the Bernoulli distribution induced by the weights \eqref{six-vert}. In this way we extend our distribution to configurations on the vertices $\{(a,b)\}_{a+b \leq k}$, and to the whole rectangular region, by induction on $k$.

Let us now describe the discrete random variables that will be our focus.  Fix two ordered sets $\mathcal{I} = \{1 \leq I_1 < \cdots < I_k \leq M\}$ and $\mathcal{J} = \{1 \leq J_1 < \cdots < J_{\ell} \leq N\}$, whose cardinalities satisfy 
$|\mathcal{I}| + |\mathcal{J}| = k+\ell = N$. We will be interested in the probability, \index{P1@$\mathbb{P}_{\rm 6v}(\mathcal{I},\mathcal{J})$} $\mathbb{P}_{{\rm 6v}}(\mathcal{I},\mathcal{J})$, that a configuration of the model on the $M \times N$ lattice has an outgoing vertical path situated at the horizontal coordinates $\mathcal{I}$ (counted from left to right) and an outgoing horizontal path situated at the vertical coordinates $\mathcal{J}$ (counted from top to bottom).

This probability can be evaluated in terms of a partition function \index{Z3@$Z_{M,N}(\mathcal{I},\mathcal{J})$} $Z_{M,N}(\mathcal{I},\mathcal{J})$ in the $M \times N$ quadrant, whose boundary conditions are chosen as follows: {\bf 1.} There is an incoming horizontal path at the edge $(0,j)\rightarrow (1,j)$ for all $1 \leq j \leq N$; {\bf 2.} The edge $(i,0) \rightarrow(i,1)$ is devoid of a path, for all $1 \leq i \leq M$; {\bf 3.} There is an outgoing horizontal path at the edge $(M,N-J_a+1) \rightarrow (M+1,N-J_a+1)$ for all $1 \leq a \leq \ell$; {\bf 4.} There is an outgoing vertical path at the edge $(I_b,N) \rightarrow(I_b,N+1)$ for all $1 \leq b \leq k$. See Figure \ref{fig:6v-quad} for an illustration of $Z_{M,N}(\mathcal{I},\mathcal{J})$. One has $\mathbb{P}_{{\rm 6v}}(\mathcal{I},\mathcal{J}) = Z_{M,N}(\mathcal{I},\mathcal{J})$, and in what follows we treat the distribution and partition function interchangeably.

\begin{figure}
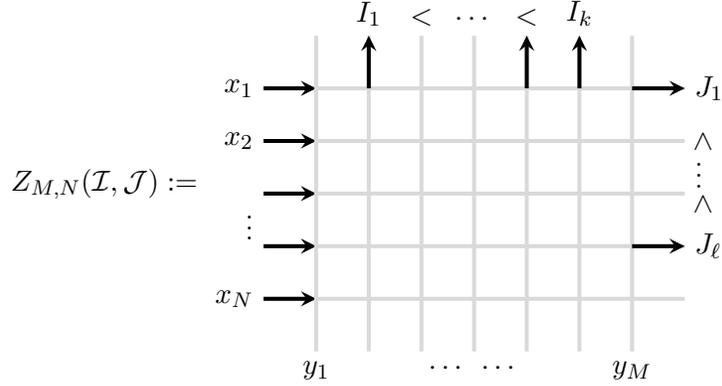

\begin{align*}
Z_{M,N}(\mathcal{I},\mathcal{J})
:=
\tikz{0.7}{
\foreach\y in {1,...,5}{
\draw[lgray,line width=1.5pt] (0,\y) -- (8,\y);
\draw[line width=1.5pt,->] (0,\y) -- (1,\y);
}
\foreach\x in {1,...,7}{
\draw[lgray,line width=1.5pt] (\x,0) -- (\x,6);
}
\foreach\x in {2,5,6}{
\draw[line width=1.5pt,->] (\x,5) -- (\x,6);
}
\foreach\y in {2,5}{
\draw[line width=1.5pt,->] (7,\y) -- (8,\y);
}
\node[above] at (6,6) {$I_k$};
\node[above] at (5,6) {$<$};
\node[above] at (4,6) {$\cdots$};
\node[above] at (3,6) {$<$};
\node[above] at (2,6) {$I_1$};
\node[below] at (7,0) {$y_M$};
\node[below] at (3.5,0) {$\cdots$};
\node[below] at (4.5,0) {$\cdots$};
\node[below] at (1,0) {$y_1$};
\node[left] at (0,1) {$x_N$};
\node[left] at (0,2.5) {$\vdots$};
\node[left] at (0,4) {$x_2$};
\node[left] at (0,5) {$x_1$};
\node[right] at (8,2) {$J_{\ell}$};
\node[right] at (8,2.8) {$\rotgr$};
\node[right] at (8,3.5) {$\vdots$};
\node[right] at (8,4) {$\rotgr$};
\node[right] at (8,5) {$J_1$};
}
\end{align*}
\caption{$Z_{M,N}(\mathcal{I},\mathcal{J})$ is the weighted sum over all configurations in which paths terminate at horizontal coordinates $\mathcal{I} = \{I_1,\dots,I_k\}$ and vertical coordinates $\mathcal{J} = \{J_1,\dots,J_{\ell}\}$. Note that the index of rapidities assigned to horizontal lines increases as one goes from top to bottom.}
\label{fig:6v-quad}
\end{figure}

\section{Coloured path distributions in the quadrant, $\mathbb{P}_{{\rm col}}(\mathcal{I},\mathcal{J})$}

The second random object to be considered arises in the model \eqref{fund-vert} with $n = N$, where $N$ is (as before) the vertical dimension of the lattice. This is a higher-rank analogue of the stochastic six-vertex model, and we recall that one can group the possible types of vertices into five categories:
\begin{align}
\label{col-vert}
\begin{tabular}{|c|c|c|}
\hline
\quad
\tikz{0.6}{
\draw[lgray,line width=1.5pt,->] (-1,0) -- (1,0);
\draw[lgray,line width=1.5pt,->] (0,-1) -- (0,1);
\node[left] at (-1,0) {\tiny $i$};\node[right] at (1,0) {\tiny $i$};
\node[below] at (0,-1) {\tiny $i$};\node[above] at (0,1) {\tiny $i$};
}
\quad
&
\quad
\tikz{0.6}{
\draw[lgray,line width=1.5pt,->] (-1,0) -- (1,0);
\draw[lgray,line width=1.5pt,->] (0,-1) -- (0,1);
\node[left] at (-1,0) {\tiny $i$};\node[right] at (1,0) {\tiny $i$};
\node[below] at (0,-1) {\tiny $j$};\node[above] at (0,1) {\tiny $j$};
}
\quad
&
\quad
\tikz{0.6}{
\draw[lgray,line width=1.5pt,->] (-1,0) -- (1,0);
\draw[lgray,line width=1.5pt,->] (0,-1) -- (0,1);
\node[left] at (-1,0) {\tiny $i$};\node[right] at (1,0) {\tiny $j$};
\node[below] at (0,-1) {\tiny $j$};\node[above] at (0,1) {\tiny $i$};
}
\quad
\\[1.3cm]
\quad
$1$
\quad
& 
\quad
$\dfrac{q(1-xy)}{1-qxy}$
\quad
& 
\quad
$\dfrac{1-q}{1-qxy}$
\quad
\\[0.7cm]
\hline
&
\quad
\tikz{0.6}{
\draw[lgray,line width=1.5pt,->] (-1,0) -- (1,0);
\draw[lgray,line width=1.5pt,->] (0,-1) -- (0,1);
\node[left] at (-1,0) {\tiny $j$};\node[right] at (1,0) {\tiny $j$};
\node[below] at (0,-1) {\tiny $i$};\node[above] at (0,1) {\tiny $i$};
}
\quad
&
\quad
\tikz{0.6}{
\draw[lgray,line width=1.5pt,->] (-1,0) -- (1,0);
\draw[lgray,line width=1.5pt,->] (0,-1) -- (0,1);
\node[left] at (-1,0) {\tiny $j$};\node[right] at (1,0) {\tiny $i$};
\node[below] at (0,-1) {\tiny $i$};\node[above] at (0,1) {\tiny $j$};
}
\quad
\\[1.3cm]
& 
\quad
$\dfrac{1-xy}{1-qxy}$
\quad
&
\quad
$\dfrac{(1-q)xy}{1-qxy}$
\quad 
\\[0.7cm]
\hline
\end{tabular}
\end{align}
where $0 \leq i < j \leq N$. Notice that the completely unoccupied and completely occupied vertices of \eqref{six-vert} should be considered as a single type of a vertex, which is the reason that the vertices \eqref{col-vert} split into five categories, rather than six.

Similarly to the previous section, we consider the model \eqref{col-vert} inside an $M \times N$ lattice, and define a coloured analogue of domain wall boundary conditions: for all $1 \leq j \leq N$, we choose the left external edge $(0,j) \rightarrow (1,j)$ to be occupied by a path of colour $j$, while the external bottom edge $(i,0) \rightarrow (i,1)$ is unoccupied for all $1 \leq i \leq M$. Since the model \eqref{col-vert} continues to enjoy the stochastic property \eqref{stoch}, one can then define a discrete-time Markov process in the very same way as previously; namely, by filling out the anti-diagonals of the lattice incrementally, sampling from the distribution induced by the weights \eqref{col-vert}. Note, however, that we reverse the order of the row rapidities $(x_1,\dots,x_N)$.

Turning to random variables within the model of coloured paths, we again fix two ordered sets $\mathcal{I} = \{1 \leq I_1 < \cdots < I_k \leq M\}$ and 
$\mathcal{J} = \{1 \leq J_1 < \cdots < J_{\ell} \leq N\}$, whose cardinalities satisfy 
$|\mathcal{I}| + |\mathcal{J}| = k+\ell = N$. This time we will be interested in the probability, \index{P2@$\mathbb{P}_{\rm col}(\mathcal{I},\mathcal{J})$} $\mathbb{P}_{{\rm col}}(\mathcal{I},\mathcal{J})$, that a configuration of coloured paths on the $M \times N$ lattice has outgoing vertical paths {\it of any colour} situated at the horizontal coordinates $\mathcal{I}$, and outgoing horizontal paths of colours $\mathcal{J}$ {\it situated at any vertical coordinates.} More simply, we keep track of positions $\mathcal{I}$ of paths that leave the lattice via its top edge, and the colours $\mathcal{J}$ of paths that leave the lattice via its right edge.

The probability $\mathbb{P}_{{\rm col}}(\mathcal{I},\mathcal{J})$ can be evaluated via a partition function \index{X3@$X_{M,N}(\mathcal{I},\mathcal{J})$} $X_{M,N}(\mathcal{I},\mathcal{J})$, which is defined as the sum over all configurations that satisfy the following constraints: {\bf 1.} There is an incoming horizontal path of colour $j$ at the edge $(0,j) \rightarrow (1,j)$ for all $1 \leq j \leq N$; {\bf 2.} The edge $(i,0) \rightarrow (i,1)$ is devoid of a path, for all $1 \leq i \leq M$; {\bf 3.} The collection of edges $(M,j) \rightarrow (M+1,j)$, $1 \leq j \leq N$, features a total of $\ell$ outgoing horizontal paths, of colours $\{J_1,\dots,J_{\ell}\}$; {\bf 4.} There is an outgoing vertical path, of unspecified colour, at the edge $(I_b,N) \rightarrow (I_b,N+1)$ for all $1 \leq b \leq k$. See Figure \ref{fig:col-quad}.

\begin{figure}[H]
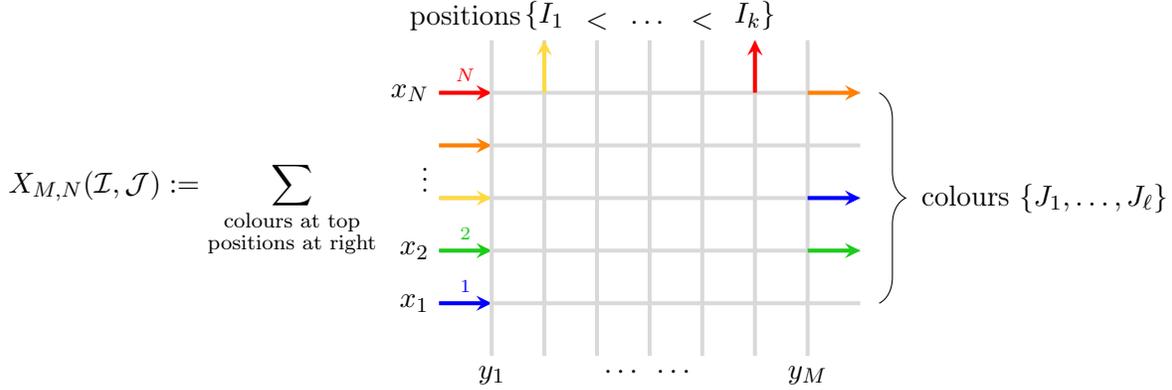

\begin{align*}
X_{M,N}(\mathcal{I},\mathcal{J})
:=
\sum_{\substack{\text{colours at top} \\ \text{positions at right}}}
\tikz{0.7}{
\foreach\y in {1,...,5}{
\draw[lgray,line width=1.5pt] (0,\y) -- (8,\y);
}
\draw[line width=1.5pt,red,->] (0,5) -- (1,5) node[midway,above] {\tiny $N$};
\draw[line width=1.5pt,orange,->] (0,4) -- (1,4);
\draw[line width=1.5pt,yellow,->] (0,3) -- (1,3);
\draw[line width=1.5pt,green,->] (0,2) -- (1,2) node[midway,above] {\tiny $2$};
\draw[line width=1.5pt,blue,->] (0,1) -- (1,1) node[midway,above] {\tiny $1$};
\foreach\x in {1,...,7}{
\draw[lgray,line width=1.5pt] (\x,0) -- (\x,6);
}
\draw[line width=1.5pt,yellow,->] (2,5) -- (2,6);
\draw[line width=1.5pt,red,->] (6,5) -- (6,6);
\draw[line width=1.5pt,green,->] (7,2) -- (8,2);
\draw[line width=1.5pt,blue,->] (7,3) -- (8,3);
\draw[line width=1.5pt,orange,->] (7,5) -- (8,5);
\node[above] at (6,6) {$I_k$\}};
\node[above] at (5,6) {$<$};
\node[above] at (4,6) {$\cdots$};
\node[above] at (3,6) {$<$};
\node[above] at (2,6) {\{$I_1$};
\node[above,text centered] at (0.5,6) {positions};
\node[below] at (7,0) {$y_M$};
\node[below] at (3.5,0) {$\cdots$};
\node[below] at (4.5,0) {$\cdots$};
\node[below] at (1,0) {$y_1$};
\node[left] at (0,1) {$x_1$};
\node[left] at (0,2) {$x_2$};
\node[left] at (0,3.5) {$\vdots$};
\node[left] at (0,5) {$x_N$};
\draw [decorate,decoration={brace,amplitude=10pt},xshift=-4pt,yshift=0pt] (8.5,5) -- (8.5,1) 
node [black,midway,xshift=2.2cm] {colours\ $\{J_1,\dots,J_{\ell}\}$};
}
\end{align*}
\caption{$X_{M,N}(\mathcal{I},\mathcal{J})$ is the weighted sum over all configurations in which paths terminate at horizontal coordinates $\mathcal{I} = \{I_1,\dots,I_k\}$ along the top boundary and the paths of colours $\mathcal{J} = \{J_1,\dots,J_{\ell}\}$ terminate at the right boundary. The index of horizontal rapidities increases from bottom to top, which is the opposite convention to the definition of $Z_{M,N}(\mathcal{I}, \mathcal{J})$.}
\label{fig:col-quad}
\end{figure}
In contrast to the uncoloured partition function $Z_{M,N}(\mathcal{I},\mathcal{J})$, the constraints {\bf 1}--{\bf 4} do not specify the boundaries of the lattice uniquely; the top and right boundaries should be considered as summed over all configurations of paths which respect conditions {\bf 3} and {\bf 4}. More precisely, the right boundary is summed over all $\frac{N!}{(N-\ell)!}$ ways of distributing outgoing colours $\{J_1,\dots,J_{\ell}\}$ along the $N$ horizontal edges, while the top boundary is summed over all $k!$ ways of assigning the $k$ remaining colours to the edges 
$(I_b,N) \rightarrow (I_b,N+1)$, $1 \leq b \leq k$.

Analogously to the uncoloured case, one has $\mathbb{P}_{\rm col}(\mathcal{I},\mathcal{J}) = X_{M,N}(\mathcal{I},\mathcal{J})$, allowing us to think of our probability distribution as a purely combinatorial quantity.

\section{Coloured Hall--Littlewood processes and the distribution $\mathbb{P}_{\rm cHL}(\mathcal{I},\mathcal{J})$}

The third type of random object takes us back to the realm of non-symmetric Hall--Littlewood polynomials. We begin by defining a class of probability measures in this setting, which can be viewed as a refinement of ascending (symmetric) Hall--Littlewood processes. In what follows \index{P@$P_{\lambda/\nu}$; skew Hall--Littlewood} $P_{\lambda/\nu}(x)$ and \index{Q@$Q_{\lambda/\nu}$; skew Hall--Littlewood} $Q_{\lambda/\nu}(x)$ will denote skew Hall--Littlewood polynomials in one variable, given by
\begin{align*}
P_{\lambda/\nu}(x)
=
\left\{
\begin{array}{ll}
x^{|\lambda|-|\nu|}
\prod_{k: m_k(\lambda)+1 = m_k(\nu)}
\left(1-q^{m_k(\nu)}\right),
&\quad
\lambda \succ \nu,
\\ \\
0,
&\quad
{\rm otherwise},
\end{array}
\right.
\end{align*}

\begin{align*}
Q_{\lambda/\nu}(x)
=
\left\{
\begin{array}{ll}
x^{|\lambda|-|\nu|}
\prod_{k: m_k(\lambda) = m_k(\nu)+1}
\left(1-q^{m_k(\lambda)}\right),
&\quad
\lambda \succ \nu,
\\ \\
0,
&\quad
{\rm otherwise},
\end{array}
\right.
\end{align*}
where $\lambda \succ \nu$ indicates the interlacing property $\lambda_1 \geq \nu_1 \geq \lambda_2 \geq \nu_2 \geq \cdots$. As previously, $E_{\mu}(y_1,\dots,y_M)$ denotes the non-symmetric Hall--Littlewood polynomial associated to composition $(\mu_1,\dots,\mu_M)$, \cf\ Section \ref{ssec:hl-reduce}.

\begin{defn}
Fix a positive integer $N \geq 1$. A Gelfand--Tsetlin pattern \index{al@$\bm{\lambda}$; Gelfand--Tsetlin patterns} $\bm{\lambda} = \left\{\lambda^{(1)},\dots,\lambda^{(N)}\right\}$ of length $N-1$ is a sequence of partitions such that
\begin{align*}
\lambda^{(1)} \succ \cdots \succ \lambda^{(N-1)} \succ \lambda^{(N)} \equiv \varnothing.
\end{align*}
Because neighbouring partitions interlace in the above sequence, one has 
\begin{align*}
\ell(\lambda^{(j)}) - \ell(\lambda^{(j+1)}) \in \{0,1\},\quad
\forall\ 1 \leq j \leq N-1,
\end{align*}
where $\ell(\lambda)$ denotes the length of a partition $\lambda$.
\end{defn}

\begin{prop}
Fix two positive integers $M,N \geq 1$, a composition $\mu = (\mu_1,\dots,\mu_M)$ of length $M$ and a Gelfand--Tsetlin pattern $\bm{\lambda} = \left\{\lambda^{(1)},\dots,\lambda^{(N)}\right\}$ of length $N-1$. Associate to this pair of objects the weight 
\begin{align}
\label{asc-HL}
\mathbb{W}_{M,N}
(\mu,\bm{\lambda})
\index{W@$\mathbb{W}_{M,N}(\mu,\bm{\lambda})$; coloured Hall--Littlewood measure}
=
E_{\tilde\mu}(y_M,\dots,y_1)
\cdot
Q_{\mu^{+}/\lambda^{(1)}}(x_1)
\cdot
\prod_{j=2}^{N}
Q_{\lambda^{(j-1)}/\lambda^{(j)}}(x_j)
\cdot
\prod_{i=1}^{N}
\prod_{j=1}^{M}
\frac{1-x_i y_j}{1-q x_i y_j},
\end{align}
where we recall that $\tilde{\mu}_i = \mu_{M-i+1}$, $1 \leq i \leq M$. This defines a probability measure; namely, one has the sum-to-unity relation
\begin{align*}
\sum_{\mu}
\sum_{\bm{\lambda}}
\mathbb{W}_{M,N}
(\mu,\bm{\lambda})
=
1,
\end{align*}
and one can ensure that 
$0 \leq \mathbb{W}_{M,N}(\mu,\bm{\lambda}) \leq 1$ by choosing $0 \leq x_i < 1$, $0 \leq y_j < 1$ for all $1 \leq i \leq N$ and $1 \leq j \leq M$. 
\end{prop}

\begin{proof}
The sum over Gelfand--Tsetlin patterns can be performed using the branching rule for the Hall--Littlewood polynomials \cite[Chapter III]{Macdonald}:
\begin{align*}
Q_{\lambda}(x_1,\dots,x_N)
=
\sum_{\nu \prec \lambda}
Q_{\lambda/\nu}(x_1)
Q_{\nu}(x_2,\dots,x_N).
\end{align*}
Using this identity $N-1$ times, one obtains
\begin{align}
\label{Wsum}
\sum_{\mu}
\sum_{\bm{\lambda}=
\lambda^{(1)} \succ \cdots \succ \lambda^{(N-1)} \succ \varnothing}
\mathbb{W}_{M,N}
(\mu,\bm{\lambda})
=
\sum_{\mu}
E_{\tilde\mu}(y_M,\dots,y_1) Q_{\mu^{+}}(x_1,\dots,x_N)
\cdot
\prod_{i=1}^{N}
\prod_{j=1}^{M}
\frac{1-x_i y_j}{1-q x_i y_j}.
\end{align}
The latter sum can be taken using the fact that
\begin{align*}
\sum_{\mu: \mu^{+} = \nu}
E_{\tilde{\mu}}(y_M,\dots,y_1)
=
P_{\nu}(y_M,\dots,y_1)
=
P_{\nu}(y_1,\dots,y_M),
\end{align*}
where $P_{\nu}(y_1,\dots,y_M)$ is the symmetric Hall--Littlewood polynomial associated to the partition $\nu$. Applying this symmetrization formula to \eqref{Wsum}, one obtains
\begin{align*}
\sum_{\mu}
\sum_{\bm{\lambda}}
\mathbb{W}_{M,N}
(\mu,\bm{\lambda})
=
\sum_{\nu}
P_{\nu}(y_1,\dots,y_M)
Q_{\nu}(x_1,\dots,x_N)
\cdot
\prod_{i=1}^{N}
\prod_{j=1}^{M}
\frac{1-x_i y_j}{1-q x_i y_j}
=
1,
\end{align*}
by the Cauchy identity for Hall--Littlewood polynomials \cite[Chapter III]{Macdonald}. Alternatively one may compute \eqref{Wsum} using the $s=0$ case of the Cauchy identity \eqref{fG-cauchy}; to do so, one should substitute $f_{\mu}(y_1,\dots,y_M)|_{s=0} = E_{\tilde\mu}(y_M,\dots,y_1)$, $G^{\bullet}_{\mu^{+}}(x_1,\dots,x_N) = q^{-MN} Q_{\mu^{+}}(x_1,\dots,x_N)$ into that sum.

\end{proof}

We refer to \eqref{asc-HL} as a {\it coloured} Hall--Littlewood process; it refines ascending Hall--Littlewood processes by assigning a probability weight to compositions paired with Gelfand--Tsetlin patterns, rather than partitions paired with Gelfand--Tsetlin patterns, as in the usual symmetric setting. Replacing $E_{\tilde\mu}(y_M,\dots,y_1)$ by the symmetric Hall--Littlewood polynomial $P_{\mu^{+}}(y_1,\dots,y_M)$ returns \eqref{asc-HL} to the standard ascending Hall--Littlewood process.

In a similar vein to \cite{BorodinBW}, let us now define a class of observables associated to coloured Hall--Littlewood processes. As usual, fix two positive integers 
$M,N \geq 1$ and two ordered sets $\mathcal{I} = \{1 \leq I_1 < \cdots < I_k \leq M\}$ and $\mathcal{J} = \{1 \leq J_1 < \cdots < J_{\ell} \leq N\}$, whose cardinalities satisfy $|\mathcal{I}| + |\mathcal{J}| = k+\ell = N$. We will also require the complement of the set $\mathcal{I}$, defined as 
\begin{align*}
\b{\mathcal{I}} &= \{1 \leq \b{I}_1 < \cdots < \b{I}_{M-k} \leq M\} 
= 
\{1,\dots,M\} \backslash \mathcal{I}.
\end{align*}
The observables that will interest us will be the {\it zero set} $z(\mu)$ \index{z6@$z(\mu)$} of the composition $(\mu_1,\dots,\mu_M)$, defined as
\begin{align}
\label{z-mu}
z(\mu) = \{1 \leq i \leq M : \mu_i = 0\},
\end{align}
and a further set
\begin{align}
\label{zeta-mu}
\zeta(\mu,\bm{\lambda})
\index{af@$\zeta(\mu,\bm{\lambda})$}
=
\{1 \leq j \leq N: \ell(\lambda^{(j-1)}) - \ell(\lambda^{(j)}) = 0\},
\qquad
\lambda^{(0)} \equiv \mu^{+},
\qquad
\lambda^{(N)} \equiv \varnothing,
\end{align}
which records the instances where neighbouring partitions in the extended Gelfand--Tsetlin pattern $\mu^{+} \succ \lambda^{(1)} \succ \cdots \succ \lambda^{(N-1)} \succ \varnothing$ have the same length. In particular, we define
\begin{align}
\label{PcHL}
\mathbb{P}_{\rm cHL}(\mathcal{I},\mathcal{J})
\index{P3@$\mathbb{P}_{\rm cHL}(\mathcal{I},\mathcal{J})$}
=
\sum_{\mu}
\sum_{\bm{\lambda}}
\mathbb{W}_{M,N}
(\mu,\bm{\lambda})
\cdot
\bm{1}_{z(\mu)=\b{\mathcal{I}}}
\cdot
\bm{1}_{\zeta(\mu,\bm{\lambda}) = \mathcal{J}},
\end{align}
which is the joint distribution of the random variables $z(\mu)$, $\zeta(\mu,\bm{\lambda})$ in the pair $(\mu,\bm{\lambda})$ sampled with respect to the coloured Hall--Littlewood process \eqref{asc-HL}.

\section{Equivalence of $\mathbb{P}_{\rm 6v}(\mathcal{I},\mathcal{J})$ and $\mathbb{P}_{\rm col}(\mathcal{I},\mathcal{J})$}
\label{sec:2=3}

\begin{thm}
\label{thm:2=3}
Fix two integers $M,N \geq 1$ and two sets $\mathcal{I} = \{1 \leq I_1 < \cdots < I_k \leq M\}$ and 
$\mathcal{J} = \{1 \leq J_1 < \cdots < J_{\ell} \leq N\}$ whose cardinalities satisfy $k+\ell=N$. The following equality of distributions holds:
\begin{align}
\label{2=3}
\mathbb{P}_{\rm 6v}(\mathcal{I},\mathcal{J})
=
\mathbb{P}_{\rm col}(\mathcal{I},\mathcal{J}).
\end{align}
\end{thm}

\begin{rmk}
\label{rmk:fw}
A weaker version of this theorem, in the special case $M=N$, $\mathcal{I} = \{1,\dots,N\}$, $\mathcal{J} = \varnothing$, was previously obtained in \cite{FodaW}. In that situation the proof goes through without great difficulty, in view of the fact that one can invoke colour-blindness of the partition function $\mathbb{P}_{\rm col}(\{1,\dots,N\},\varnothing)$ that arises for this choice of boundary conditions.
\end{rmk}

\begin{proof}
The proof of this result is long; we will split it into several steps.

\medskip

\underline{A partial result.}

\medskip
 
We will begin by proving that, when $k=M$ and $\ell = N-M \geq 0$, one has
\begin{align}
\label{PF-match}
Z_{M,N}(\mathcal{J})
=
X_{M,N}(\mathcal{J}),
\end{align}
where we have introduced the notation
\begin{align*}
Z_{M,N}(\mathcal{J})
:=
Z_{M,N}(\{1,\dots,M\},\mathcal{J}),
\qquad
X_{M,N}(\mathcal{J})
:=
X_{M,N}(\{1,\dots,M\},\mathcal{J})
\end{align*}
for the special case $k=M$, $\ell = N-M$, in which necessarily $\mathcal{I} = \{1,\dots,M\}$. The proof of \eqref{PF-match} proceeds using Lagrange interpolation techniques for the analysis of partition functions in solvable lattice models (see, for example, \cite{FodaW,WheelerZ}). Namely, we will show that both partition functions $Z_{M,N}(\mathcal{J})$, $X_{M,N}(\mathcal{J})$ satisfy a certain list of properties. Because these properties turn out to be uniquely-determining, the two partition functions must be equal.

\medskip

\underline{Properties of $Z_{M,N}(\mathcal{J})$.}

\medskip
Let us start with the uncoloured partition function $Z_{M,N}(\mathcal{J})$. It satisfies the following properties:
\begin{enumerate}
\item[1.] $\prod_{j=1}^{M} (1-q x_N y_j) Z_{M,N}(\mathcal{J})$ is a polynomial in $x_N$ of degree $\leq M$.

\medskip

\item[{\it Proof of 1.}] The partition function $Z_{M,N}(\mathcal{J})$ depends on $x_N$ only via the lowest row of the lattice. Multiplying by the factor $\prod_{j=1}^{M} (1-q x_N y_j)$ effectively normalizes the Boltzmann weights in that row of the lattice, so that they become degree $1$ polynomials in $x_N$; the claim is then immediate.

\medskip

\item[2.] $Z_{M,N}(\mathcal{J})$ is symmetric in the variables $(y_1,\dots,y_M)$.

\medskip

\item[{\it Proof of 2.}] This follows by a standard argument, identical to that used to prove the symmetry of the domain wall partition function \cite{Korepin} in its rapidity variables. The bottom edges of the partition function $Z_{M,N}(\mathcal{J})$ are devoid of paths; one can therefore attach an empty $R$-vertex at the base of the lattice, which intertwines the $a$-th and $(a+1)$-th vertical lines, for some $1 \leq a \leq M-1$. Using the Yang--Baxter equation \eqref{YB}, this $R$-vertex can be threaded through the lattice until it emerges from the top boundary. Since the top boundary is completely saturated by paths, one can then delete the emerging completely occupied $R$-vertex. The result of the computation is thus the switching of the $a$-th and $(a+1)$-th vertical lines; symmetry in $(y_1,\dots,y_M)$ follows by compositions of such transpositions.

\medskip

\item[3.] If $J_{N-M} = N$, one has 
\begin{align}
\label{Zrec1}
Z_{M,N}(\mathcal{J})
=
\prod_{j=1}^{M}
\left(
\frac{1-x_N y_j}{1-q x_N y_j}
\right)
Z_{M,N-1}(\mathcal{J} \backslash \{J_{N-M}\}).
\end{align}

\medskip

\item[{\it Proof of 3.}] If $J_{N-M} = N$, it means that a path leaves the lattice via the right external edge of the lowest horizontal line. This path must necessarily enter the lattice via the left external edge of the same line; configurations on this line are, therefore, completely frozen, and result in the overall weight 
$\prod_{j=1}^{M} (1-x_N y_j) / (1-q x_N y_j)$ (see Figure \ref{fig:Zrec1}). The remaining $N-1$ lines give rise to the partition function 
$Z_{M,N-1}(\mathcal{J} \backslash \{J_{N-M}\})$, and the factorization \eqref{Zrec1} follows immediately.

\begin{figure}
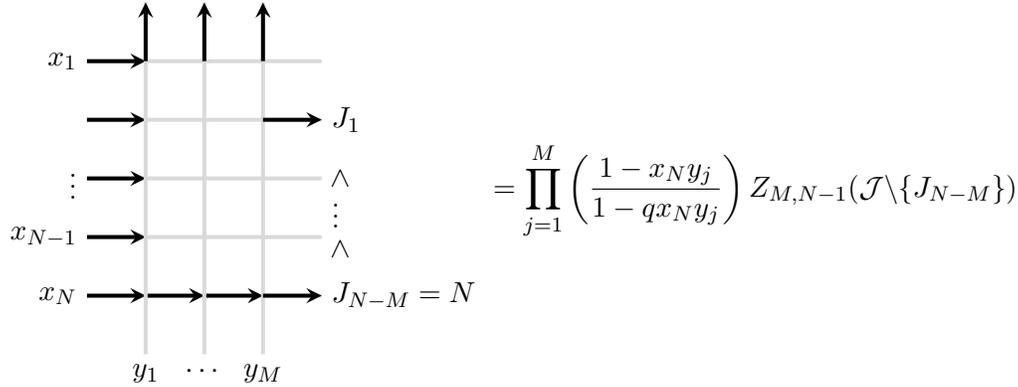

\begin{align*}
\tikz{0.78}{
\foreach\y in {1,...,5}{
\draw[lgray,line width=1.5pt] (0,\y) -- (4,\y);
\draw[line width=1.5pt,->] (0,\y) -- (1,\y);
}
\draw[line width=1.5pt,->] (1,1) -- (2,1);
\draw[line width=1.5pt,->] (2,1) -- (3,1);
\foreach\x in {1,...,3}{
\draw[lgray,line width=1.5pt] (\x,0) -- (\x,6);
}
\foreach\x in {1,2,3}{
\draw[line width=1.5pt,->] (\x,5) -- (\x,6);
}
\foreach\y in {1,4}{
\draw[line width=1.5pt,->] (3,\y) -- (4,\y);
}
\node[below] at (3,0) {$y_M$};
\node[below] at (2,0) {$\cdots$};
\node[below] at (1,0) {$y_1$};
\node[left] at (0,1) {$x_N$};
\node[left] at (0,2) {$x_{N-1}$};
\node[left] at (0,3) {$\vdots$};
\node[left] at (0,5) {$x_1$};
\node[right] at (4,1) {$J_{N-M} = N$};
\node[right] at (4,1.8) {$\rotgr$};
\node[right] at (4,2.5) {$\vdots$};
\node[right] at (4,3) {$\rotgr$};
\node[right] at (4,4) {$J_1$};
}
=
\prod_{j=1}^{M}
\left(
\frac{1-x_N y_j}{1-q x_N y_j}
\right)
Z_{M,N-1}(\mathcal{J} \backslash \{J_{N-M}\})
\end{align*}
\caption{Reduction of $Z_{M,N}(\mathcal{J})$ to $Z_{M,N-1}(\mathcal{J} \backslash \{J_{N-M}\})$ when $J_{N-M} = N$. The bottom row of the lattice is completely frozen, and can be deleted at the expense of an overall multiplicative factor.}
\label{fig:Zrec1}
\end{figure}

\medskip

\item[4.] If $J_{N-M} < N$, $Z_{M,N}(\mathcal{J})$ vanishes at $x_N = 0$ and satisfies the recursion relation
\begin{align}
\label{Zrec2}
Z_{M,N}(\mathcal{J}) \Big|_{x_N = 1/y_M}
=
Z_{M-1,N-1}(\mathcal{J}).
\end{align}

\medskip

\item[{\it Proof of 4.}] If $J_{N-M} < N$, the right external edge of the bottom horizontal line is unoccupied. It follows that in any legal configuration of the bottom horizontal line, the vertex 
\tikz{0.3}{
\draw[lgray,line width=1.5pt,->] (-1,0) -- (1,0);
\draw[lgray,line width=1.5pt,->] (0,-1) -- (0,1);
\node[left] at (-1,0) {\tiny $1$};\node[right] at (1,0) {\tiny $0$};
\node[below] at (0,-1) {\tiny $0$};\node[above] at (0,1) {\tiny $1$};
}
must appear somewhere; since this vertex carries a weight proportional to $x_N$, the vanishing property is apparent. 

To prove the recursion \eqref{Zrec2}, one should first transfer the vertical lattice line carrying the rapidity $y_M$ all the way to the left of the lattice; this transformation leaves $Z_{M,N}(\mathcal{J})$ invariant due to its symmetry in $(y_1,\dots,y_M)$. Next, one studies the effect of setting $x_N = 1/y_M$; it affects the vertex present at the intersection of the bottom horizontal and leftmost vertical lines, preventing it from taking the form
\tikz{0.3}{
\draw[lgray,line width=1.5pt,->] (-1,0) -- (1,0);
\draw[lgray,line width=1.5pt,->] (0,-1) -- (0,1);
\node[left] at (-1,0) {\tiny $1$};\node[right] at (1,0) {\tiny $1$};
\node[below] at (0,-1) {\tiny $0$};\node[above] at (0,1) {\tiny $0$};
}
and thus forcing it to take the form
\tikz{0.3}{
\draw[lgray,line width=1.5pt,->] (-1,0) -- (1,0);
\draw[lgray,line width=1.5pt,->] (0,-1) -- (0,1);
\node[left] at (-1,0) {\tiny $1$};\node[right] at (1,0) {\tiny $0$};
\node[below] at (0,-1) {\tiny $0$};\node[above] at (0,1) {\tiny $1$};
}
with weight $1$. The freezing of this vertex causes the entire bottom horizontal and leftmost vertical lines to freeze to a unique configuration, with weight $1$ (see Figure \ref{fig:Zrec2}). The recursion \eqref{Zrec2} follows immediately.

\begin{figure}
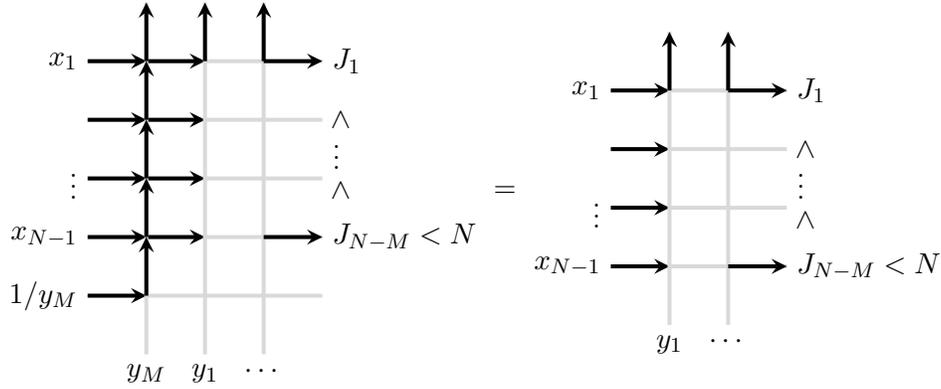

\begin{align*}
\tikz{0.78}{
\foreach\y in {1,...,5}{
\draw[lgray,line width=1.5pt] (0,\y) -- (4,\y);
\draw[line width=1.5pt,->] (0,\y) -- (1,\y);
}
\draw[line width=1.5pt,->] (1,2) -- (2,2);
\draw[line width=1.5pt,->] (1,3) -- (2,3);
\draw[line width=1.5pt,->] (1,4) -- (2,4);
\draw[line width=1.5pt,->] (1,5) -- (2,5);
\foreach\x in {1,...,3}{
\draw[lgray,line width=1.5pt] (\x,0) -- (\x,6);
}
\draw[line width=1.5pt,->] (1,1) -- (1,2);
\draw[line width=1.5pt,->] (1,2) -- (1,3);
\draw[line width=1.5pt,->] (1,3) -- (1,4);
\draw[line width=1.5pt,->] (1,4) -- (1,5);
\foreach\x in {1,2,3}{
\draw[line width=1.5pt,->] (\x,5) -- (\x,6);
}
\foreach\y in {2,5}{
\draw[line width=1.5pt,->] (3,\y) -- (4,\y);
}
\node[below] at (3,0) {$\cdots$};
\node[below] at (2,0) {$y_1$};
\node[below] at (1,0) {$y_M$};
\node[left] at (0,1) {$1/y_M$};
\node[left] at (0,2) {$x_{N-1}$};
\node[left] at (0,3) {$\vdots$};
\node[left] at (0,5) {$x_1$};
\node[right] at (4,2) {$J_{N-M} < N$};
\node[right] at (4,2.8) {$\rotgr$};
\node[right] at (4,3.5) {$\vdots$};
\node[right] at (4,4) {$\rotgr$};
\node[right] at (4,5) {$J_1$};
}
=
\tikz{0.78}{
\foreach\y in {2,...,5}{
\draw[lgray,line width=1.5pt] (1,\y) -- (4,\y);
}
\draw[line width=1.5pt,->] (1,2) -- (2,2);
\draw[line width=1.5pt,->] (1,3) -- (2,3);
\draw[line width=1.5pt,->] (1,4) -- (2,4);
\draw[line width=1.5pt,->] (1,5) -- (2,5);
\foreach\x in {2,...,3}{
\draw[lgray,line width=1.5pt] (\x,1) -- (\x,6);
}
\foreach\x in {2,3}{
\draw[line width=1.5pt,->] (\x,5) -- (\x,6);
}
\foreach\y in {2,5}{
\draw[line width=1.5pt,->] (3,\y) -- (4,\y);
}
\node[below] at (3,1) {$\cdots$};
\node[below] at (2,1) {$y_1$};
\node[left] at (1,2) {$x_{N-1}$};
\node[left] at (1,3) {$\vdots$};
\node[left] at (1,5) {$x_1$};
\node[right] at (4,2) {$J_{N-M} < N$};
\node[right] at (4,2.8) {$\rotgr$};
\node[right] at (4,3.5) {$\vdots$};
\node[right] at (4,4) {$\rotgr$};
\node[right] at (4,5) {$J_1$};
}
\end{align*}
\caption{Recursion relation for $Z_{M,N}(\mathcal{J})$ when $J_{N-M} < N$. After transferring the vertical line carrying parameter $y_M$ to the far left, and setting $x_N = 1/y_M$, the bottom horizontal and leftmost vertical lines freeze. The remaining non-trivial piece computes precisely $Z_{M-1,N-1}(\mathcal{J})$.}
\label{fig:Zrec2}
\end{figure}

\medskip

\item[5.] One has the initial condition
\begin{align}
\label{Zinit}
Z_{1,1}(\varnothing)
=
\frac{(1-q)x_1 y_1}{1-q x_1 y_1}.
\end{align}

\medskip

\item[{\it Proof of 5.}] $Z_{1,1}(\varnothing)$ is nothing but the vertex
\tikz{0.3}{
\draw[lgray,line width=1.5pt,->] (-1,0) -- (1,0);
\draw[lgray,line width=1.5pt,->] (0,-1) -- (0,1);
\node[left] at (-1,0) {\tiny $1$};\node[right] at (1,0) {\tiny $0$};
\node[below] at (0,-1) {\tiny $0$};\node[above] at (0,1) {\tiny $1$};
},
whose weight is simply given by \eqref{Zinit}.

\end{enumerate}

\medskip

\underline{Properties of $X_{M,N}(\mathcal{J})$.}

\medskip

Now let us move on to the coloured partition function $X_{M,N}(\mathcal{J})$. We will show that it satisfies the very same list of properties:
\begin{enumerate}
\item[$1'$.] $\prod_{j=1}^{M} (1-q x_N y_j) X_{M,N}(\mathcal{J})$ is a polynomial in $x_N$ of degree $\leq M$.

\medskip

\item[{\it Proof of $1'$.}] The partition function $X_{M,N}(\mathcal{J})$ depends on $x_N$ via the top row of the lattice. Multiplying by the factor $\prod_{j=1}^{M} (1-q x_N y_j)$ normalizes the Boltzmann weights in that row of the lattice, so that they become degree $1$ polynomials in $x_N$; the claim follows.

\medskip

\item[$2'$.] $X_{M,N}(\mathcal{J})$ is symmetric in the variables $(y_1,\dots,y_M)$.

\medskip

\item[{\it Proof of $2'$.}] This follows by a very similar argument to that used to establish property $2$ of $Z_{M,N}(\mathcal{J})$: namely, the insertion of an $R$-vertex at the base of the $a$-th and $(a+1)$-th columns, followed by its transfer to the top of the lattice using the Yang--Baxter equation \eqref{YB}. The top boundary of $X_{M,N}(\mathcal{J})$ is summed over all possible ways of distributing the colours $\{1,\dots,N\} \backslash \mathcal{J}$ over the $M$ available sites. This means that the $R$-vertex which emerges from the top of the lattice has its two outgoing edges summed over all possible states; by the stochasticity \eqref{stoch} of the $R$-vertex, this sum thus has weight $1$ and can be deleted. The symmetry property then goes through as before.

\medskip

\item[$3'$.] If $J_{N-M} = N$, one has 
\begin{align}
\label{Xrec1}
X_{M,N}(\mathcal{J})
=
\prod_{j=1}^{M}
\left(
\frac{1-x_N y_j}{1-q x_N y_j}
\right)
X_{M,N-1}(\mathcal{J} \backslash \{J_{N-M}\}).
\end{align}

\medskip

\item[{\it Proof of $3'$.}] If $J_{N-M} = N$, the colour $N$ is among those that exit the partition function $X_{M,N}(\mathcal{J})$ via its right edge. To compute $X_{M,N}(\mathcal{J})$, one should sum over all right external edges where the path of colour $N$ can exit the lattice. However, there is only one possible choice: namely, the colour $N$ must exit via the right external edge of the top horizontal line, since it enters the lattice via the left external edge of the same line. This effectively freezes the top row of $X_{M,N}(\mathcal{J})$: the path of colour $N$ saturates all horizontal edges within this row of vertices, so that all paths (of smaller colour) which enter this row from below must pass straight through it vertically (see Figure \ref{fig:Xrec1}). The weight of this row is thus given by a product of vertices of the form
\tikz{0.3}{
\draw[lgray,line width=1.5pt,->] (-1,0) -- (1,0);
\draw[lgray,line width=1.5pt,->] (0,-1) -- (0,1);
\node[left] at (-1,0) {\tiny $N$};\node[right] at (1,0) {\tiny $N$};
\node[below] at (0,-1) {\tiny $i$};\node[above] at (0,1) {\tiny $i$};
},
where $1 \leq i < N$; hence it produces the overall factor $\prod_{j=1}^{M} (1-x_N y_j)/(1-q x_N y_j)$, and the factorization \eqref{Xrec1} can now be deduced.

\begin{figure}
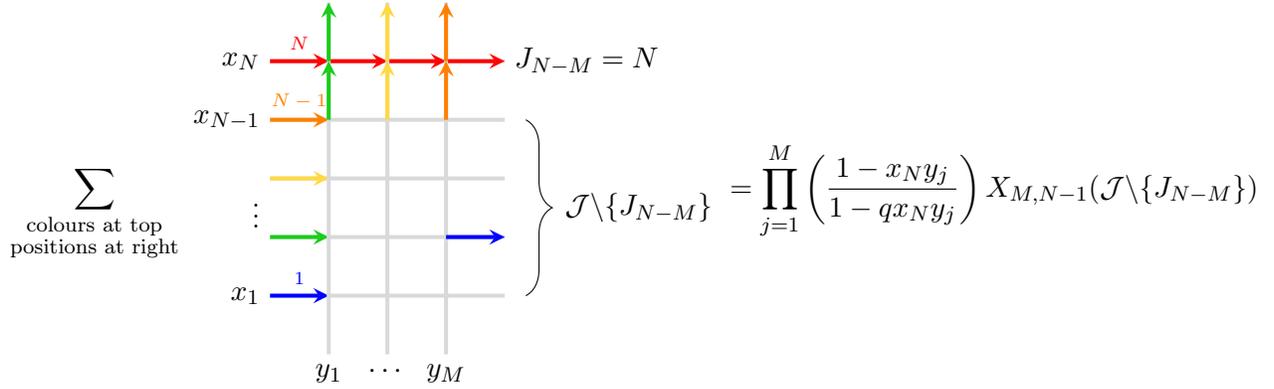

\begin{align*}
\sum_{\substack{\text{colours at top} \\ \text{positions at right}}}
\tikz{0.78}{
\foreach\y in {1,...,5}{
\draw[lgray,line width=1.5pt] (0,\y) -- (4,\y);
}
\draw[line width=1.5pt,red,->] (0,5) -- (1,5) node[midway,above] {\tiny $N$};
\draw[line width=1.5pt,orange,->] (0,4) -- (1,4) node[midway,above] {\tiny $N-1$};
\draw[line width=1.5pt,yellow,->] (0,3) -- (1,3);
\draw[line width=1.5pt,green,->] (0,2) -- (1,2);
\draw[line width=1.5pt,blue,->] (0,1) -- (1,1) node[midway,above] {\tiny $1$};
\draw[line width=1.5pt,red,->] (1,5) -- (2,5);
\draw[line width=1.5pt,red,->] (2,5) -- (3,5);
\draw[line width=1.5pt,red,->] (3,5) -- (4,5);
\foreach\x in {1,...,3}{
\draw[lgray,line width=1.5pt] (\x,0) -- (\x,6);
}
\draw[line width=1.5pt,green,->] (1,5) -- (1,6);
\draw[line width=1.5pt,green,->] (1,4) -- (1,5);
\draw[line width=1.5pt,yellow,->] (2,5) -- (2,6);
\draw[line width=1.5pt,yellow,->] (2,4) -- (2,5);
\draw[line width=1.5pt,orange,->] (3,5) -- (3,6);
\draw[line width=1.5pt,orange,->] (3,4) -- (3,5);
\draw[line width=1.5pt,blue,->] (3,2) -- (4,2);
\node[below] at (3,0) {$y_M$};
\node[below] at (2,0) {$\cdots$};
\node[below] at (1,0) {$y_1$};
\node[left] at (0,1) {$x_1$};
\node[left] at (0,2.5) {$\vdots$};
\node[left] at (0,4) {$x_{N-1}$};
\node[left] at (0,5) {$x_N$};
\node[right] at (4,5) {$J_{N-M} = N$};
\draw [decorate,decoration={brace,amplitude=10pt},xshift=-4pt,yshift=0pt] (4.5,4) -- (4.5,1) 
node [black,midway,xshift=1.5cm] {$\mathcal{J} \backslash \{J_{N-M}\}$};
}
=
\prod_{j=1}^{M}
\left(
\frac{1-x_N y_j}{1-q x_N y_j}
\right)
X_{M,N-1}(\mathcal{J} \backslash \{J_{N-M}\})
\end{align*}
\caption{When $J_{N-M} = N$, the colour $N$ exits via the right edge of the partition function $X_{M,N}(\mathcal{J})$, and it must do so in the top row of the lattice. The top row thus freezes into a product of vertices of the form $R_{x_N y_j}(i,N;i,N)$, with $1 \leq i < N$ and $1 \leq j \leq M$.}
\label{fig:Xrec1}
\end{figure}

\medskip

\item[$4'$.] If $J_{N-M} < N$, $X_{M,N}(\mathcal{J})$ vanishes at $x_N=0$ and satisfies the recursion relation
\begin{align}
\label{Xrec2}
X_{M,N}(\mathcal{J}) \Big|_{x_N = 1/y_M}
=
X_{M-1,N-1}(\mathcal{J}).
\end{align}

\medskip

\item[{\it Proof of $4'$.}] Setting $x_N = 1/y_M$ causes the vertex in the top right corner of the partition function $X_{M,N}(\mathcal{J})$ to split, in the same way as in equation \eqref{R-split}. We thus obtain the partition function shown in Figure \ref{fig:Xrec2}. This partition function can be simplified even further: this can be seen from the fact that the $M$ outgoing edges of its top row are summed over all ways of permuting the colours in the set $\{1,\dots,N\} \backslash \mathcal{J}$ over the $M$ available sites, while the $N$ outgoing edges of its rightmost column are summed over all ways of assigning the $N-M$ colours $\mathcal{J}$ to the $N$ available sites. The stochasticity \eqref{stoch} of the model then allows us to delete the vertices in the top row and rightmost column, and we recover the partition function $X_{M-1,N-1}(\mathcal{J})$.

\begin{figure}
\begin{align*}
\sum_{\substack{\text{colours at top} \\ \text{positions at right}}}
\begin{tikzpicture}[scale=0.78,baseline=2cm,>=stealth]
\foreach\y in {1,...,4}{
\draw[lgray,line width=1.5pt] (0,\y) -- (4,\y);
\node at (3.5,\y) {$\star$};
}
\draw[lgray,line width=1.5pt] (0,5) -- (2.8,5) -- (3,5.2) -- (3,6);
\node at (3,5.5) {$\star$};
\draw[line width=1.5pt,red,->] (0,5) -- (1,5) node[midway,above] {\tiny $N$};
\draw[line width=1.5pt,orange,->] (0,4) -- (1,4) node[midway,above] {\tiny $N-1$};
\draw[line width=1.5pt,yellow,->] (0,3) -- (1,3);
\draw[line width=1.5pt,green,->] (0,2) -- (1,2);
\draw[line width=1.5pt,blue,->] (0,1) -- (1,1) node[midway,above] {\tiny $1$};
\foreach\x in {1,...,2}{
\draw[lgray,line width=1.5pt] (\x,0) -- (\x,6);
\node at (\x,5.5) {$\star$};
}
\draw[lgray, line width=1.5pt] (3,0) -- (3,4.8) -- (3.2,5) -- (4,5);
\node at (3.5,5) {$\star$};
\node[below] at (3,0) {$y_M$};
\node[below] at (2,0) {$\cdots$};
\node[below] at (1,0) {$y_1$};
\node[left] at (0,1) {$x_1$};
\node[left] at (0,2.5) {$\vdots$};
\node[left] at (0,4) {$x_{N-1}$};
\node[left] at (0,5) {$1/y_M$};
\draw [decorate,decoration={brace,amplitude=10pt},xshift=0pt,yshift=-4pt] (1,6.5) -- (3,6.5) 
node [black,midway,yshift=1cm] {$\substack{\text{colours}\\ \\ \{1,\dots,N\} \backslash \mathcal{J}}$};
\draw [decorate,decoration={brace,amplitude=10pt},xshift=-4pt,yshift=0pt] (4.5,5) -- (4.5,1) 
node [black,midway,xshift=1cm] {$\substack{\text{colours} \\ \\ \mathcal{J}}$};
\end{tikzpicture}
=
\sum_{\substack{\text{colours at top} \\ \text{positions at right}}}
\begin{tikzpicture}[scale=0.78,baseline=1.7cm,>=stealth]
\foreach\y in {1,...,4}{
\draw[lgray,line width=1.5pt] (0,\y) -- (3,\y);
\node at (2.5,\y) {$\star$};
}
\draw[line width=1.5pt,orange,->] (0,4) -- (1,4) node[midway,above] {\tiny $N-1$};
\draw[line width=1.5pt,yellow,->] (0,3) -- (1,3);
\draw[line width=1.5pt,green,->] (0,2) -- (1,2);
\draw[line width=1.5pt,blue,->] (0,1) -- (1,1) node[midway,above] {\tiny $1$};
\foreach\x in {1,...,2}{
\draw[lgray,line width=1.5pt] (\x,0) -- (\x,5);
\node at (\x,4.7) {$\star$};
}
\node[below] at (2,0) {$\cdots$};
\node[below] at (1,0) {$y_1$};
\node[left] at (0,1) {$x_1$};
\node[left] at (0,2.5) {$\vdots$};
\node[left] at (0,4) {$x_{N-1}$};
\draw [decorate,decoration={brace,amplitude=10pt},xshift=0pt,yshift=-4pt] (0.5,5.5) -- (2.5,5.5) 
node [black,midway,yshift=1cm] {$\substack{\text{colours}\\ \\ \{1,\dots,N-1\} \backslash \mathcal{J}}$};
\draw [decorate,decoration={brace,amplitude=10pt},xshift=-4pt,yshift=0pt] (3.5,4) -- (3.5,1) 
node [black,midway,xshift=1cm] {$\substack{\text{colours} \\ \\ \mathcal{J}}$};
\end{tikzpicture}
\end{align*}
\caption{Choosing $x_N = 1/y_M$ causes the top right vertex of $X_{M,N}(\mathcal{J})$ to split, and the edges marked by $\star$ are summed over permutations of the colours shown. Making repeated use of the stochasticity property \eqref{stoch}, we may suppress the vertices of the top row and rightmost column, leading directly to the partition function $X_{M-1,N-1}(\mathcal{J})$.}
\label{fig:Xrec2}
\end{figure}
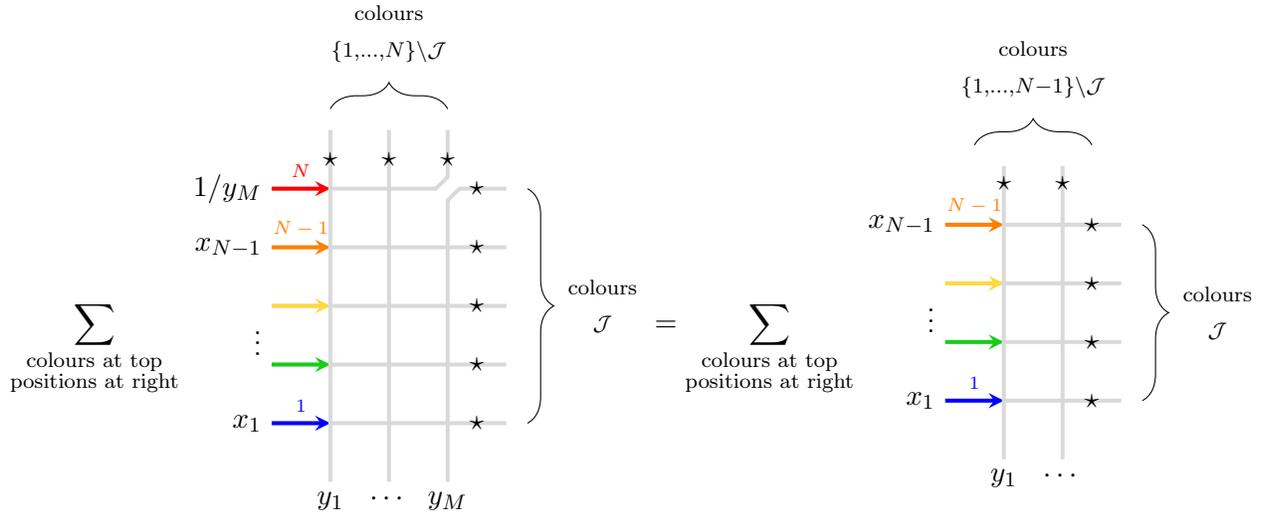

\medskip

\item[$5'$.] One has the initial condition
\begin{align}
\label{Xinit}
X_{1,1}(\varnothing)
=
\frac{(1-q)x_1 y_1}{1-q x_1 y_1}.
\end{align}

\medskip

\item[{\it Proof of $5'$.}] As in the case of $Z_{1,1}(\varnothing)$, $X_{1,1}(\varnothing)$ is equal to the rank-1 vertex
\tikz{0.3}{
\draw[lgray,line width=1.5pt,->] (-1,0) -- (1,0);
\draw[lgray,line width=1.5pt,->] (0,-1) -- (0,1);
\node[left] at (-1,0) {\tiny $1$};\node[right] at (1,0) {\tiny $0$};
\node[below] at (0,-1) {\tiny $0$};\node[above] at (0,1) {\tiny $1$};
},
whose weight is given by \eqref{Xinit}.

\end{enumerate}

\medskip

\underline{Uniqueness.}

\medskip

We have established that $Z_{M,N}(\mathcal{J})$ and $X_{M,N}(\mathcal{J})$ obey the same list of properties; it remains to demonstrate that this necessarily renders them equal for all $1 \leq M \leq N$. To do this, let us define the difference of the previous two quantities:
\begin{align*}
\Delta_{M,N}(\mathcal{J})
=
Z_{M,N}(\mathcal{J}) - X_{M,N}(\mathcal{J}),
\qquad
|\mathcal{J}| = N-M.
\end{align*}
By virtue of properties $5$ and $5'$, we clearly have $\Delta_{1,1}(\varnothing) = 0$. This allows us to assume that there exists an integer $N \geq 2$ such that $\Delta_{m,n}(\mathcal{J}) = 0$ for all $1 \leq m \leq n \leq N-1$ and sets $\mathcal{J}$ with $|\mathcal{J}| = n-m$. 

Now consider the quantity $\prod_{j=1}^{m} (1-q x_N y_j) \Delta_{m,N}(\mathcal{J})$, where $1 \leq m \leq N$ and $|\mathcal{J}| = N-m$. Properties $1$ and $1'$ tell us that this quantity is a polynomial in $x_N$ of degree $\leq m$. If $J_{N-m} = N$, properties $3$ and $3'$ (together with the inductive assumption) imply that $\Delta_{m,N}(\mathcal{J}) = 0$ for all $1 \leq m \leq N-1$.\footnote{Note that $m$ cannot assume the value $m=N$ in this case, since this would imply the set $\mathcal{J}$ is empty.} 

On the other hand, if $J_{N-m} < N$, properties $4$ and $4'$ (as well as the symmetry of $2$ and $2'$) determine $\prod_{j=1}^{m} (1-q x_N y_j) \Delta_{m,N}(\mathcal{J})$ at $m+1$ points; namely the points $x_N = 1/y_j$, $1 \leq j \leq m$ and $x_N = 0$, and for all of these values it vanishes (again by the inductive assumption). Since a non-trivial polynomial of degree $\leq m$ cannot vanish at $m+1$ distinct points\footnote{We have tacitly assumed that the points in the set 
$\{y_1,\dots,y_m\} \cup \{0\}$ are pairwise distinct.}, we conclude that $\Delta_{m,N}(\mathcal{J}) = 0$ for all $1 \leq m \leq N$. 

We have thus shown that $\Delta_{m,n}(\mathcal{J}) = 0$ for all $1 \leq m \leq n \leq N$ and sets $\mathcal{J}$ with $|\mathcal{J}| = n-m$; this statement can be deduced in full generality by induction on $N$. This completes the proof of \eqref{PF-match}.

\medskip

\underline{General result.}

\medskip

So far we have proved that
\begin{align}
\label{Z=X}
Z_{k,N}(\{1,\dots,k\},\mathcal{J})
=
X_{k,N}(\{1,\dots,k\},\mathcal{J}),
\quad
\forall\
\mathcal{J} = \{1 \leq J_1 < \cdots < J_{N-k} \leq N\},
\quad 
1 \leq k \leq N,
\end{align}
where we have relabelled $M \mapsto k$ compared with \eqref{PF-match}. Now let us generalize the result \eqref{Z=X} slightly, by extending the horizontal size of the lattice used in the construction of $Z_{k,N}$ and $X_{k,N}$. 

\begin{prop}
For all $k \leq M \leq N$ and $\mathcal{J} = \{1 \leq J_1 < \cdots < J_{N-k} \leq N\}$ there holds
\begin{align}
\label{Z=X-2}
Z_{M,N}(\{M-k+1,\dots,M\},\mathcal{J})
=
X_{M,N}(\{M-k+1,\dots,M\},\mathcal{J}).
\end{align}
\end{prop}

\begin{proof}
Examining the top edges of the partition function $Z_{M,N}(\{M-k+1,\dots,M\},\mathcal{J})$, we see that the first $M-k$ of them are unoccupied by paths. This implies that the $N$ paths which are incoming from the left edges of this partition function must propagate horizontally for at least their first $M-k$ steps, and accordingly, the leftmost $M-k$ columns of $Z_{M,N}(\{M-k+1,\dots,M\},\mathcal{J})$ are frozen to a product of the vertices 
\tikz{0.3}{
\draw[lgray,line width=1.5pt,->] (-1,0) -- (1,0);
\draw[lgray,line width=1.5pt,->] (0,-1) -- (0,1);
\node[left] at (-1,0) {\tiny $1$};\node[right] at (1,0) {\tiny $1$};
\node[below] at (0,-1) {\tiny $0$};\node[above] at (0,1) {\tiny $0$};
}. 
We thus infer the relation
\begin{multline}
\label{Z-triv-col}
Z_{M,N}(\{M-k+1,\dots,M\},\mathcal{J};y_1,\dots,y_M)
\\
=
\prod_{i=1}^{N}
\prod_{j=1}^{M-k}
\frac{1-x_i y_j}{1-q x_i y_j}
\cdot
Z_{k,N}(\{1,\dots,k\},\mathcal{J};y_{M-k+1},\dots,y_M),
\end{multline}
where we have explicitly shown the dependence of both partition functions on their set of vertical rapidities, since this alphabet is shifted by $M-k$ on the right hand side. 

By exactly the same reasoning, one can easily see that this relation also applies in the case of the coloured partition function:
\begin{multline}
\label{X-triv-col}
X_{M,N}(\{M-k+1,\dots,M\},\mathcal{J};y_1,\dots,y_M)
\\
=
\prod_{i=1}^{N}
\prod_{j=1}^{M-k}
\frac{1-x_i y_j}{1-q x_i y_j}
\cdot
X_{k,N}(\{1,\dots,k\},\mathcal{J};y_{M-k+1},\dots,y_M).
\end{multline}
Now the right hand sides of \eqref{Z-triv-col} and \eqref{X-triv-col} can be matched, by virtue of \eqref{Z=X}, which yields the result \eqref{Z=X-2}.
\end{proof}

Equation \eqref{Z=X-2} now forms the basis for proving the general result \eqref{2=3}. Indeed, we may write \eqref{Z=X-2} in the form
\begin{align}
\label{Z=X-init}
Z_{M,N}(\mathcal{I},\mathcal{J})
=
X_{M,N}(\mathcal{I},\mathcal{J}),
\quad
\mathcal{I} = \{I_1,\dots,I_k\} = \{M-k+1,\dots,M\},
\end{align}
in which the elements of $\mathcal{I}$ are constrained to assume their maximal possible values. We would like to devise means to relax this constraint, and allow generic values for the set $\mathcal{I}$.

\begin{prop}
Let us denote by $T_j(y)$ the Hecke generator \eqref{hecke-poly} with $(x_j,x_{j+1}) \mapsto (y_j,y_{j+1})$:
\begin{align*}
T_j(y) = q - \frac{y_j-q y_{j+1}}{y_j-y_{j+1}} (1-\mathfrak{s}_j),
\quad
1 \leq j \leq M-1.
\end{align*}
Fix the set $\mathcal{I} = \{1 \leq I_1 < \cdots < I_k \leq M\}$ and let $1 \leq i \leq k$ be a positive integer such that $I_{i-1} < I_i-1$, where $I_0 = 0$ by agreement. We then have the pair of relations
\begin{align}
\label{hecke-1}
T_{I_i-1}(y) \cdot Z_{M,N}(\mathcal{I},\mathcal{J})
&=
Z_{M,N}(\{I_1,\dots,I_{i-1},I_{i}-1,I_{i+1},\dots,I_k\},\mathcal{J}),
\\
\label{hecke-2}
T_{I_i-1}(y) \cdot X_{M,N}(\mathcal{I},\mathcal{J})
&=
X_{M,N}(\{I_1,\dots,I_{i-1},I_{i}-1,I_{i+1},\dots,I_k\},\mathcal{J}).
\end{align}
\end{prop}

\begin{proof}
Let us begin with the proof of \eqref{hecke-1}. As with the proofs of other relations involving Hecke generators (discussed throughout Chapter \ref{sec:properties}), the key step is to establish a suitable exchange relation for operators used in the algebraic construction of $Z_{M,N}(\mathcal{I},\mathcal{J})$. In particular, let us note the following relation:
\begin{align}
\label{yb-column}
\tikz{0.7}{
\foreach\y in {1,...,4}{
\draw[lgray,line width=1.5pt] (0,\y) -- (3,\y);
\node at (0.5,\y) {$\star$};
\node at (2.5,\y) {$\star$};
}
\draw[lgray,line width=1.5pt] (1,0) -- (1,5) -- (2,6);
\draw[lgray,line width=1.5pt] (2,0) -- (2,5) -- (1,6);
\node[above] at (1,6) {\tiny $0$};
\node[above] at (2,6) {\tiny $1$};
\node[below] at (2,-0.5) {$y_{j+1}$};
\node[below] at (1,-0.5) {$y_j$};
\node[below] at (1,0) {\tiny $0$};
\node[below] at (2,0) {\tiny $0$};
\node[left] at (0,1) {$x_N$};
\node[left] at (0,2.7) {$\vdots$};
\node[left] at (0,4) {$x_1$};
\draw[line width=1.5pt,->] (1,4) -- (1,5) -- (2,6);
}
+
\tikz{0.7}{
\foreach\y in {1,...,4}{
\draw[lgray,line width=1.5pt] (0,\y) -- (3,\y);
\node at (0.5,\y) {$\star$};
\node at (2.5,\y) {$\star$};
}
\draw[lgray,line width=1.5pt] (1,0) -- (1,5) -- (2,6);
\draw[lgray,line width=1.5pt] (2,0) -- (2,5) -- (1,6);
\node[above] at (1,6) {\tiny $0$};
\node[above] at (2,6) {\tiny $1$};
\node[below] at (2,-0.5) {$y_{j+1}$};
\node[below] at (1,-0.5) {$y_j$};
\node[below] at (1,0) {\tiny $0$};
\node[below] at (2,0) {\tiny $0$};
\node[left] at (0,1) {$x_N$};
\node[left] at (0,2.7) {$\vdots$};
\node[left] at (0,4) {$x_1$};
\draw[line width=1.5pt,->] (2,4) -- (2,5) -- (1.5,5.5) -- (2,6);
}
=
\tikz{0.7}{
\foreach\y in {1,...,4}{
\draw[lgray,line width=1.5pt] (0,\y) -- (3,\y);
\node at (0.5,\y) {$\star$};
\node at (2.5,\y) {$\star$};
}
\draw[lgray,line width=1.5pt] (2,-1) -- (1,0) -- (1,5);
\draw[lgray,line width=1.5pt] (1,-1) -- (2,0) -- (2,5);
\node[above] at (1,5) {\tiny $0$};
\node[above] at (2,5) {\tiny $1$};
\node[below] at (2,-1.5) {$y_{j+1}$};
\node[below] at (1,-1.5) {$y_j$};
\node[below] at (1,-1) {\tiny $0$};
\node[below] at (2,-1) {\tiny $0$};
\node[left] at (0,1) {$x_N$};
\node[left] at (0,2.7) {$\vdots$};
\node[left] at (0,4) {$x_1$};
\draw[line width=1.5pt,->] (2,4) -- (2,5);
}
\end{align}
This identity is a direct consequence of the Yang--Baxter equation \eqref{YB}, and holds for any fixed values of the external horizontal edges, marked by $\star$ on both sides of \eqref{yb-column}. Observe that the assumption $I_{i-1} < I_i-1$ ensures that the $I_i$-th top external edge of $Z_{M,N}(\mathcal{I},\mathcal{J})$ is occupied by a path, while the $(I_i-1)$-th top external edge is unoccupied, allowing us to apply the relation \eqref{yb-column} to the $(I_i-1)$-th and $I_i$-th columns of $Z_{M,N}(\mathcal{I},\mathcal{J})$. Reading off the Boltzmann weight of the diagonally inserted $R$-vertices in \eqref{yb-column}, we are thus led to the equation
\begin{multline}
\label{hecke1-proof}
\left(
\frac{1-y_{j+1}/y_j}{1-q y_{j+1}/y_j}
\right)
\cdot
Z_{M,N}(\{I_1,\dots,I_{i-1},I_{i}-1,I_{i+1},\dots,I_k\},\mathcal{J})
\\
+
\left(
\frac{1-q}{1-q y_{j+1}/y_j}
\right)
\cdot
Z_{M,N}(\mathcal{I},\mathcal{J})
=
\mathfrak{s}_j \cdot Z_{M,N}(\mathcal{I},\mathcal{J}),
\end{multline}
where we have identified $j \equiv I_i-1$, and where $\mathfrak{s}_j$ interchanges $y_j$ and $y_{j+1}$ in $Z_{M,N}(\mathcal{I},\mathcal{J})$. After rearrangement, one finds that \eqref{hecke1-proof} is equivalent to the desired identity \eqref{hecke-1}.

The proof of \eqref{hecke-2} is almost identical: once again the assumption $I_{i-1} < I_i-1$ ensures that the $I_i$-th top external edge of $X_{M,N}(\mathcal{I},\mathcal{J})$ is occupied by a path, while the $(I_i-1)$-th top external edge is unoccupied; the difference being that in this case the $I_i$-th edge should be summed over all possible colourings $\{1,\dots,N\} \backslash \mathcal{J}$. This summation effects little practical change of the previous arguments. Indeed, one can write another version of the exchange relation \eqref{yb-column} in the coloured vertex model \eqref{col-vert}, in which the outgoing path at the top of the lattice is summed over all possible colourings $\{1,\dots,N\} \backslash \mathcal{J}$. Since the Boltzmann weights of the diagonally inserted $R$-vertices are invariant under such re-colourings of the occupied edges, we arrive at the identity
\begin{multline}
\label{hecke2-proof}
\left(
\frac{1-y_{j+1}/y_j}{1-q y_{j+1}/y_j}
\right)
\cdot
X_{M,N}(\{I_1,\dots,I_{i-1},I_{i}-1,I_{i+1},\dots,I_k\},\mathcal{J})
\\
+
\left(
\frac{1-q}{1-q y_{j+1}/y_j}
\right)
\cdot
X_{M,N}(\mathcal{I},\mathcal{J})
=
\mathfrak{s}_j \cdot X_{M,N}(\mathcal{I},\mathcal{J}),
\end{multline}
with $j \equiv I_i-1$. This can be rearranged to yield \eqref{hecke-2}.
\end{proof}

Equations \eqref{hecke-1} and \eqref{hecke-2} allow us to inductively complete the proof of \eqref{2=3}. We have already demonstrated the equality \eqref{Z=X-init} in the case of ``maximal'' $\mathcal{I}$, and now assume that $Z_{M,N}(\mathcal{I},\mathcal{J}) = X_{M,N}(\mathcal{I},\mathcal{J})$ for some $\mathcal{I} = \{1 \leq I_1 < \cdots < I_k \leq M\}$ such that $\exists\ 1 \leq i \leq k$ with $I_{i-1} < I_i-1$. The left hand sides of \eqref{hecke-1} and \eqref{hecke-2} then match, and we obtain as a result
\begin{align}
Z_{M,N}(\{I_1,\dots,I_{i-1},I_{i}-1,I_{i+1},\dots,I_k\},\mathcal{J})
=
X_{M,N}(\{I_1,\dots,I_{i-1},I_{i}-1,I_{i+1},\dots,I_k\},\mathcal{J}).
\end{align}
We may now iterate this argument to yield the equality $Z_{M,N}(\mathcal{I},\mathcal{J}) = X_{M,N}(\mathcal{I},\mathcal{J})$ for all $\mathcal{I}$. This achieves the proof of \eqref{2=3}.

\end{proof}

\section{Equivalence of $\mathbb{P}_{\rm 6v}(\mathcal{I},\mathcal{J})$ and $\mathbb{P}_{\rm cHL}(\mathcal{I},\mathcal{J})$}
\label{ssec:bbw-type-equiv}

\begin{thm}
\label{thm:2=1}
Fix two integers $M,N \geq 1$ and two sets $\mathcal{I} = \{1 \leq I_1 < \cdots < I_k \leq M\}$ and $\mathcal{J} = \{1 \leq J_1 < \cdots < J_{\ell} \leq N\}$ whose cardinalities satisfy $k+\ell=N$. The following equality of distributions holds:
\begin{align}
\label{2=1}
\mathbb{P}_{\rm 6v}(\mathcal{I},\mathcal{J})
=
\mathbb{P}_{\rm cHL}(\mathcal{I},\mathcal{J}).
\end{align}
\end{thm}

\begin{proof}
Our starting point will be the Cauchy summation formula \eqref{fG-cauchy}, using slightly different conventions for the alphabets and their cardinalities. We write
\begin{align}
\label{eta}
\eta(s)
&:=
q^{MN}
\sum_{\mu} 
f_{\mu}(y_1,\dots,y_M)
G^{\bullet}_{\mu^{+}}(x_1,\dots,x_N)
=
\frac{(s^2;q)_M}{\prod_{j=1}^{M} (1-sy_j)}
\prod_{i=1}^{N}
\prod_{j=1}^{M}
\frac{1-q x_i y_j}{1- x_i y_j}.
\end{align}
Recalling the definitions \eqref{z-mu}, \eqref{zeta-mu} of $z(\mu)$ and 
$\zeta(\mu,\bm{\lambda})$, we now refine the summation \eqref{eta} by defining
\begin{align}
\label{eta'}
\eta(\mathcal{I},\mathcal{J};s)
:=
q^{MN}
\sum_{\mu}
\sum_{\bm{\lambda}}
f_{\mu}(y_1,\dots,y_M)
G^{\bullet}_{\mu^{+}/\lambda^{(1)}}(x_1)
\cdot
\prod_{j=2}^{N}
G^{\bullet}_{\lambda^{(j-1)}/\lambda^{(j)}}(x_j)
\cdot
\bm{1}_{z(\mu)=\b{\mathcal{I}}}
\cdot
\bm{1}_{\zeta(\mu,\bm{\lambda}) = \mathcal{J}},
\end{align}
where the second sum is over all Gelfand--Tsetlin patterns $\bm{\lambda} = \lambda^{(1)} \succ \cdots \succ \lambda^{(N-1)} \succ \varnothing$, $f_{\mu}$ is given by \eqref{f-def} as usual, and the one-variable functions $G^{\bullet}_{\lambda^{(j-1)}/\lambda^{(j)}}$ are given by \eqref{G}. Our interest in the quantities \eqref{eta} and \eqref{eta'} comes from the fact that, when $s=0$, one has
\begin{align}
\label{eta=PcHL}
\frac{\eta(\mathcal{I},\mathcal{J};0)}{\eta(0)}
=
\mathbb{P}_{\rm cHL}(\mathcal{I},\mathcal{J}),
\end{align}
which is manifest from the definition \eqref{asc-HL}, \eqref{PcHL} of $\mathbb{P}_{\rm cHL}(\mathcal{I},\mathcal{J})$ and the reductive properties of $f_{\mu}$ and $G^{\bullet}_{\lambda^{(j-1)}/\lambda^{(j)}}$ at $s=0$; \cf\ Theorem \ref{thm:f-E} for the former, and Remark \ref{rmk:4.10} and \cite[Section 8.1]{Borodin} for the latter.

Let us now separately analyze the two sums appearing in \eqref{eta'}. In what follows, we take $\kappa = (\kappa_1 \geq \cdots \geq \kappa_k \geq 1)$ to be a fixed partition of length $k$. We have the relation
\begin{align}
\label{f'1}
\sum_{\mu: \mu^{+} = \kappa}
\left(
\bm{1}_{z(\mu)=\b{\mathcal{I}}}
\right)
f_{\mu}(y_1,\dots,y_M)
=
\tikz{0.7}{
\foreach\y in {1,...,5}{
\draw[lgray,line width=1.5pt,->] (1,\y) -- (8,\y);
}
\foreach\x in {2,...,7}{
\draw[lgray,line width=4pt,->] (\x,0) -- (\x,6);
}
\node[left] at (0.3,1) {$y_1 \rightarrow$};
\node[left] at (0.3,2) {$y_2 \rightarrow$};
\node[left] at (0.3,3) {$\vdots$};
\node[left] at (0.3,4) {$\vdots$};
\node[left] at (0.3,5) {$y_M \rightarrow$};
\node[below] at (7,0) {$\cdots$};
\node[below] at (6,0) {$\cdots$};
\node[below] at (5,0) {$\cdots$};
\node[below] at (4,0) {\footnotesize$\bm{0}$};
\node[below] at (3,0) {\footnotesize$\bm{0}$};
\node[below] at (2,0) {\footnotesize$\bm{0}$};
\node[above] at (5,6) {$\overbrace{\phantom{........................}}^{\substack{
\text{colours}\ \mathcal{I} 
\\ 
\text{at coordinates}\ \kappa}}$};
\node[above] at (2,6) {$\overbrace{}^{\text{colours}\ \b{\mathcal{I}}}$};
\node[right] at (8,1) {\tiny $0$};
\node[right] at (8,2) {\tiny $0$};
\node[right] at (8,3) {\tiny $\vdots$};
\node[right] at (8,4) {\tiny $\vdots$};
\node[right] at (8,5) {\tiny $0$};
\node[left] at (1,1) {\tiny $1$};
\node[left] at (1,2) {\tiny $2$};
\node[left] at (1,3) {\tiny $\vdots$};
\node[left] at (1,4) {\tiny $\vdots$};
\node[left] at (1,5) {\tiny $M$};
}
\end{align}
in which the colours $\b{\mathcal{I}} = \{\b{I}_1,\dots,\b{I}_{M-k}\}$ leave via the top of the 0-th column, while the remaining $k$ colours $\mathcal{I} = \{I_1,\dots,I_k\}$ exit the lattice with collective horizontal coordinates $(\kappa_1 \geq \cdots \geq \kappa_k)$, but with a sum taken over all possible ways of distributing the colours over those sites. If we now set $s=0$, we find that the configuration of the $0$-th column in \eqref{f'1} becomes completely frozen: to see this, note that any colour $i \in \mathcal{I}$ (which enters the $0$-th column via the $i$-th left edge of the column) does not exit the column via its top edge. It must therefore leave the column via one of the $i$-th, $(i+1)$-th, \dots, or $M$-th right edges. However, we can also rule out the possibility that it leaves the column via the $j$-th right edge, for $i+1 \leq j \leq M$, since this would give rise to the vertex $\tikz{0.4}{
\draw[lgray,line width=1.5pt,->] (-1,0) -- (1,0);
\draw[lgray,line width=3pt,->] (0,-1) -- (0,1);
\node[left] at (-1,0) {\tiny $j$};\node[right] at (1,0) {\tiny $i$};
\node[below] at (0,-1) {\tiny $\bm{A}$};\node[above] at (0,1) 
{\tiny $\bm{A}^{+-}_{ji}$};
}$
for some state $\bm{A} \in \mathbb{N}^n$, which has a vanishing Boltzmann weight at $s=0$. We conclude that the colours $\mathcal{I}$ must each propagate horizontally straight across the $0$-th column, while the colours $\b{\mathcal{I}}$ exit via the top of the column; the weight of the whole column (recall that $s$ is set to $0$) can be computed as $\prod_{i \in \mathcal{I}} y_i$. Suppressing the frozen $0$-th column and drawing the partition function from its 1st column onwards, we conclude that
\begin{align}
\label{f'2}
\sum_{\mu: \mu^{+} = \kappa}
\left(
\bm{1}_{z(\mu)=\b{\mathcal{I}}}
\right)
f_{\mu}(y_1,\dots,y_M)
\Big|_{s=0}
=
\prod_{i \in \mathcal{I}} y_i
\cdot
\tikz{0.7}{
\foreach\y in {1,...,5}{
\draw[lgray,line width=1.5pt,->] (2,\y) -- (8,\y);
}
\foreach\x in {3,...,7}{
\draw[lgray,line width=4pt,->] (\x,0) -- (\x,6);
}
\node[left] at (0.1,1) {$y_1 \rightarrow$};
\node[left] at (0.1,2) {$y_2 \rightarrow$};
\node[left] at (0.1,3) {$\vdots$};
\node[left] at (0.1,4) {$\vdots$};
\node[left] at (0.1,5) {$y_M \rightarrow$};
\node[below] at (7,0) {$\cdots$};
\node[below] at (6,0) {$\cdots$};
\node[below] at (5,0) {\footnotesize$\bm{0}$};
\node[below] at (4,0) {\footnotesize$\bm{0}$};
\node[below] at (3,0) {\footnotesize$\bm{0}$};
\node[above] at (5,6) {$\overbrace{\phantom{........................}}^{\substack{
\text{colours}\ \mathcal{I} 
\\ 
\text{at coordinates}\ \kappa}}$};
\node[right] at (8,1) {\tiny $0$};
\node[right] at (8,2) {\tiny $0$};
\node[right] at (8,3) {\tiny $\vdots$};
\node[right] at (8,4) {\tiny $\vdots$};
\node[right] at (8,5) {\tiny $0$};
\node[left] at (2,1) {\tiny $1 \cdot \bm{1}_{1 \in \mathcal{I}}$};
\node[left] at (2,2) {\tiny $2 \cdot \bm{1}_{2 \in \mathcal{I}}$};
\node[left] at (2,3) {$\vdots$};
\node[left] at (2,4) {$\vdots$};
\node[left] at (2,5) {\tiny $M \cdot \bm{1}_{M \in \mathcal{I}}$};
}
\end{align}
where the state that enters the left edge of the $i$-th row is now $i$ if $i \in \mathcal{I}$, and $0$ otherwise.
Because the states at the top of the lattice \eqref{f'2} are summed over all ways of distributing the colours $\mathcal{I}$ over the sites $(\kappa_1 \geq \cdots \geq \kappa_k)$, we can now invoke the colour-blindness results outlined in Section \ref{ssec:colour-blind}; namely, we can replace each of the vertices in the lattice \eqref{f'2} by their rank-1 counterparts \eqref{rank1-weights} evaluated at $s=0$. This yields the expression
\begin{align*}
\sum_{\mu: \mu^{+} = \kappa}
\left(
\bm{1}_{z(\mu)=\b{\mathcal{I}}}
\right)
f_{\mu}(y_1,\dots,y_M)
\Big|_{s=0}
=
\prod_{i \in \mathcal{I}} y_i
\cdot
\tikz{0.7}{
\foreach\y in {1,...,5}{
\draw[lgray,line width=1.5pt,->] (2,\y) -- (8,\y);
}
\foreach\x in {3,...,7}{
\draw[lgray,line width=4pt,->] (\x,0) -- (\x,6);
}
\foreach\y in {1,...,5}{
\foreach\x in {3,...,7}{
\node at (\x,\y) {$\bullet$};
}}
\node[left] at (0.9,1) {$y_1 \rightarrow$};
\node[left] at (0.9,2) {$y_2 \rightarrow$};
\node[left] at (0.9,3) {$\vdots$};
\node[left] at (0.9,4) {$\vdots$};
\node[left] at (0.9,5) {$y_M \rightarrow$};
\node[below] at (7,0) {$\cdots$};
\node[below] at (6,0) {$\cdots$};
\node[below] at (5,0) {\tiny $0$};
\node[below] at (4,0) {\tiny $0$};
\node[below] at (3,0) {\tiny $0$};
\node[above] at (7,6) {$\cdots$};
\node[above] at (6,6) {$\cdots$};
\node[above] at (5,6) {\tiny $m_3$};
\node[above] at (4,6) {\tiny $m_2$};
\node[above] at (3,6) {\tiny $m_1$};
\node[right] at (8,1) {\tiny $0$};
\node[right] at (8,2) {\tiny $0$};
\node[right] at (8,3) {\tiny $\vdots$};
\node[right] at (8,4) {\tiny $\vdots$};
\node[right] at (8,5) {\tiny $0$};
\node[left] at (2,1) {\tiny $\bm{1}_{1 \in \mathcal{I}}$};
\node[left] at (2,2) {\tiny $\bm{1}_{2 \in \mathcal{I}}$};
\node[left] at (2,3) {$\vdots$};
\node[left] at (2,4) {$\vdots$};
\node[left] at (2,5) {\tiny $\bm{1}_{M \in \mathcal{I}}$};
}
\end{align*}
where the left incoming state of the $i$-th row is either 1 or 0, depending on whether $i \in \mathcal{I}$ or not, and $m_j \equiv m_j(\kappa)$ are the multiplicities of $\kappa$. The partition function that we have arrived at has a known form; it can be expressed as a sum over a product of one-variable skew Hall--Littlewood polynomials, \cf\ \cite[Section 8.1]{Borodin}:
\begin{align}
\label{skewPP}
\sum_{\mu: \mu^{+} = \kappa}
\left(
\bm{1}_{z(\mu)=\b{\mathcal{I}}}
\right)
f_{\mu}(y_1,\dots,y_M)
\Big|_{s=0}
=
\sum_{\bm{\nu}}
\prod_{i=1}^{M-1}
P_{\nu^{(i)} / \nu^{(i-1)}}(y_i)
\cdot
P_{\kappa / \nu^{(M-1)}}(y_M)
\cdot
\bm{1}_{\zeta(\kappa,\bm{\nu}) = \b{\mathcal{I}}},
\end{align}
where the sum is taken over all Gelfand--Tsetlin patterns 
$\bm{\nu} = \nu^{(M-1)} \succ \cdots \succ \nu^{(1)} \succ \varnothing$ of length $M-1$, with
\begin{align}
\zeta(\kappa,\bm{\nu})
=
\{ 1 \leq i \leq M:
\ell(\nu^{(i)}) - \ell(\nu^{(i-1)})
=
0 \},
\qquad
\nu^{(M)} \equiv \kappa,
\qquad
\nu^{(0)} \equiv \varnothing.
\end{align}
On the other hand, the $s=0$ case of the second sum in \eqref{eta'} reads
\begin{multline}
\label{skewQQ}
q^{MN}\
\sum_{\bm{\lambda}}
G^{\bullet}_{\kappa/\lambda^{(1)}}(x_1)
\cdot
\prod_{j=2}^{N}
G^{\bullet}_{\lambda^{(j-1)}/\lambda^{(j)}}(x_j)
\cdot
\bm{1}_{\zeta(\kappa,\bm{\lambda}) = \mathcal{J}}
\Big|_{s=0}
\\
=
\sum_{\bm{\lambda}}
Q_{\kappa/\lambda^{(1)}}(x_1)
\cdot
\prod_{j=2}^{N}
Q_{\lambda^{(j-1)}/\lambda^{(j)}}(x_j)
\cdot
\bm{1}_{\zeta(\kappa,\bm{\lambda}) = \mathcal{J}}.
\end{multline}
Combining equations \eqref{skewPP} and \eqref{skewQQ}, we have shown that
\begin{multline}
\label{BBW-match}
\frac{\eta(\mathcal{I},\mathcal{J};0)}{\eta(0)}
=
\prod_{i=1}^{N}
\prod_{j=1}^{M}
\frac{1-x_i y_j}{1-q x_i y_j}
\cdot
\sum_{\kappa}
\sum_{\bm{\nu}}
\sum_{\bm{\lambda}} 
\left(
\bm{1}_{\zeta(\kappa,\bm{\nu}) = \b{\mathcal{I}}}
\right)
\cdot
\left(
\bm{1}_{\zeta(\kappa,\bm{\lambda}) = \mathcal{J}}
\right)
\\
\times
\prod_{i=1}^{M-1}
P_{\nu^{(i)} / \nu^{(i-1)}}(y_i)
\cdot
P_{\kappa / \nu^{(M-1)}}(y_M)
\cdot
Q_{\kappa/\lambda^{(1)}}(x_1)
\cdot
\prod_{j=2}^{N}
Q_{\lambda^{(j-1)}/\lambda^{(j)}}(x_j),
\end{multline}
with the sums taken over $\bm{\nu} = \nu^{(M-1)} \succ \cdots \succ \nu^{(1)} \succ \varnothing$ and $\lambda^{(1)} \succ \cdots \succ \lambda^{(N-1)} \succ \varnothing$. The refined summation thus obtained has been studied in \cite{BorodinBW}, where it was shown that such quantities are given by partition functions of the six-vertex model in the quadrant. The technique for proving this correspondence comes from the work of \cite{WheelerZ}, and exploits the Yang--Baxter integrability of the $q$-boson model. We will not repeat these arguments here, but instead quote Theorem 5.5 of \cite{BorodinBW}, which tells us immediately that \eqref{BBW-match} translates into
\begin{align}
\label{eta=Z}
\frac{\eta(\mathcal{I},\mathcal{J};0)}{\eta(0)}
=
Z_{M,N}(\mathcal{I},\mathcal{J}).
\end{align}
Matching the right hand sides of \eqref{eta=PcHL} and \eqref{eta=Z} now yields the claim, \eqref{2=1}.
\end{proof}

\chapter{Alternative proof of Theorem \ref{thm:2=1} via Cherednik--Dunkl operators}
\label{sec:simpler-match}

The purpose of this chapter is to outline another way of proving Theorem \ref{thm:2=1}. We do this in several stages, firstly introducing less refined versions of the distributions $\mathbb{P}_{\rm 6v}$ and $\mathbb{P}_{\rm cHL}$ that depend on only one random subset $\mathcal{I}$, and then showing (via the theory of Cherednik--Dunkl operators) that these distributions match. Finally, we illustrate how to restore the second random subset $\mathcal{J}$ via the action of first-order Macdonald difference operators, thus recovering the result of Theorem \ref{thm:2=1}. 

This alternative proof is technically more involved than the previous one, but we include it for two reasons: {\bf 1.} This is how we discovered the matching \eqref{2=1} initially; and {\bf 2.} We expect that this approach will be useful for extracting information for the joint distribution of colours under several specified lattice locations (as opposed to the one-point situation, $(M,N)$, that we deal with in the current work).

\section{The distribution $\mathbb{P}_{\rm 6v}(\mathcal{I})$}
\index{P11@$\mathbb{P}_{\rm 6v}(\mathcal{I})$}

Consider the stochastic six-vertex model in the $M \times N$ quadrant as previously, with a path entering via each of the left external edges $(0,j) \rightarrow (1,j)$, $1 \leq j \leq N$, and no paths entering via the bottom external edges $(i,0) \rightarrow (i,1)$, $1 \leq i \leq M$. The $j$-th horizontal line (counted from the top) carries rapidity $x_j$ and the $i$-th vertical line carries rapidity $y_i$, in the same way as in Section \ref{ssec:Z_MN}. Let $\mathcal{I} \subset \{1,\dots,M\}$ be any subset of $\{1,\dots,M\}$, whose cardinality we do not need to specify, and write $\b{\mathcal{I}} = \{1,\dots,M\} \backslash \mathcal{I}$ for its complement. We define the following probability distribution:
\begin{align}
\label{6v-I}
\mathbb{P}_{\rm 6v}(\mathcal{I})
=
{\rm Prob}
\Big\{
\text{there is no path at the outgoing vertical edge}\ (i,N) \rightarrow (i,N+1),\ \forall\ i \in \b{\mathcal{I}}
\Big\},
\end{align}
where a random configuration in the quadrant is sampled with probability given by the product of its vertex weights. Unlike $\mathbb{P}_{\rm 6v}(\mathcal{I},\mathcal{J})$, the distribution $\mathbb{P}_{\rm 6v}(\mathcal{I})$ is not concerned with paths that leave the lattice via right external edges $(M,j) \rightarrow (M+1,j)$, $1 \leq j \leq N$. It can thus be obtained from the more general distribution $\mathbb{P}_{\rm 6v}(\mathcal{I},\mathcal{J})$ as follows:
\begin{align}
\label{6v-I.2}
\mathbb{P}_{\rm 6v}(\mathcal{I})
=
\sum_{
\substack{
\mathcal{I}' \subset \mathcal{I},
\mathcal{J} \subset \{1,\dots,N\}
\\ \vspace{-0.05cm} \\
|\mathcal{I}'| +|\mathcal{J}| = N 
}}
\mathbb{P}_{\rm 6v}(\mathcal{I}',\mathcal{J}),
\end{align}
where the sum is taken over all $\mathcal{J} \subset \{1,\dots,N\}$ and $\mathcal{I}' \subset \{1,\dots,M\}$, such that $|\mathcal{I}'|+|\mathcal{J}| = N$ and $\mathcal{I}' \subset \mathcal{I}$. Therefore, any previous results concerning $\mathbb{P}_{\rm 6v}(\mathcal{I},\mathcal{J})$ can be translated to results for $\mathbb{P}_{\rm 6v}(\mathcal{I})$, by summing appropriately.

Note that, by virtue of the correspondence \eqref{BBW-match}, \eqref{eta=Z} described at the end of Section \ref{ssec:bbw-type-equiv}, \eqref{6v-I.2} also has an interpretation purely at the level of (skew) Hall--Littlewood polynomials. Substituting $\mathbb{P}_{\rm 6v}(\mathcal{I}',\mathcal{J}) = \eta(\mathcal{I}',\mathcal{J};0)/\eta(0)$ into \eqref{6v-I.2} and computing directly the sum over $\mathcal{I}'$ and $\mathcal{J}$, one has
\begin{align}
\label{P-6v-sym}
\mathbb{P}_{\rm 6v}(\mathcal{I})
=
\prod_{i=1}^{N}
\prod_{j=1}^{M}
\frac{1-x_i y_j}{1-q x_i y_j}
\cdot
\sum_{\kappa}
\sum_{\bm{\nu}}
\bm{1}_{\zeta(\kappa,\bm{\nu}) \supset \b{\mathcal{I}}}
\cdot
\prod_{i=1}^{M-1}
P_{\nu^{(i)} / \nu^{(i-1)}}(y_i)
\cdot
P_{\kappa / \nu^{(M-1)}}(y_M)
\cdot
Q_{\kappa}(x_1,\dots,x_N),
\end{align}
where the two sums are taken over all partitions $\kappa$ and all length-$(M-1)$ Gelfand--Tsetlin patterns $\bm{\nu} = \nu^{(M-1)} \succ \cdots \succ \nu^{(1)} \succ \nu^{(0)} \equiv \varnothing$, respectively.

\section{The distribution $\mathbb{P}_{\rm cHL}(\mathcal{I})$}
\index{P33@$\mathbb{P}_{\rm cHL}(\mathcal{I})$}

We shall also require a less refined version of the coloured Hall--Littlewood measure \eqref{asc-HL}, associated to compositions $\mu$ (without an attached Gelfand--Tsetlin pattern $\bm{\lambda}$):
\begin{align}
\label{W-MN-mu}
\mathbb{W}_{M,N}(\mu)
\index{W@$\mathbb{W}_{M,N}(\mu)$; coloured Hall--Littlewood measure}
=
E_{\tilde{\mu}}(y_M,\dots,y_1)
Q_{\mu^{+}}(x_1,\dots,x_N)
\cdot
\prod_{i=1}^{N}
\prod_{j=1}^{M}
\frac{1-x_i y_j}{1-q x_i y_j},
\qquad
\tilde{\mu} = (\mu_M,\dots,\mu_1),
\end{align}
which satisfies the sum-to-unity condition $\sum_{\mu} \mathbb{W}_{M,N}(\mu) = 1$. From this measure we define
\begin{align}
\label{cHL-I}
\mathbb{P}_{\rm cHL}(\mathcal{I})
=
{\rm Prob}\{
\text{the random composition}\ \mu\ \text{satisfies}\ \mu_i=0,\ \forall\ i \in \b{\mathcal{I}}
\},
\end{align}
where a random composition $\mu$ is chosen with probability $\mathbb{W}_{M,N}(\mu)$. In contrast to $\mathbb{P}_{\rm cHL}(\mathcal{I},\mathcal{J})$, the distribution $\mathbb{P}_{\rm cHL}(\mathcal{I})$ contains no information about lengths of partitions in an underlying Gelfand--Tsetlin pattern, since the latter object has effectively been summed away. Indeed, one has
\begin{align}
\label{P-cHL-sym}
\mathbb{P}_{\rm cHL}(\mathcal{I})
=
\prod_{i=1}^{N}
\prod_{j=1}^{M}
\frac{1-x_i y_j}{1-q x_i y_j}
\cdot
\sum_{\mu}
E_{\tilde{\mu}}(y_M,\dots,y_1)
Q_{\mu^{+}}(x_1,\dots,x_N)
\cdot
\bm{1}_{z(\mu) \supset \b{\mathcal{I}}},
\end{align}
and writing $Q_{\mu^{+}}(x_1,\dots,x_N)$ as a sum of products of skew one-variable Hall--Littlewood polynomials, one finds that
\begin{align}
\mathbb{P}_{\rm cHL}(\mathcal{I})
&=
\sum_{
\substack{
\mathcal{I}' \subset \mathcal{I},
\mathcal{J} \subset \{1,\dots,N\}
\\ \vspace{-0.05cm} \\
|\mathcal{I}'| +|\mathcal{J}| = N 
}}
\left(
\sum_{\mu}
\sum_{\bm{\lambda}}
\mathbb{W}_{M,N}
(\mu,\bm{\lambda})
\cdot
\bm{1}_{z(\mu)=\{1,\dots,M\}  \backslash \mathcal{I}'}
\cdot
\bm{1}_{\zeta(\mu,\bm{\lambda}) = \mathcal{J}}
\right)
\nonumber
\\
&=
\sum_{
\substack{
\mathcal{I}' \subset \mathcal{I},
\mathcal{J} \subset \{1,\dots,N\}
\\ \vspace{-0.05cm} \\
|\mathcal{I}'| +|\mathcal{J}| = N 
}}
\mathbb{P}_{\rm cHL}(\mathcal{I}',\mathcal{J}),
\label{cHL-I.2}
\end{align}
where the sum is again taken over all $\mathcal{J} \subset \{1,\dots,N\}$ and $\mathcal{I}' \subset \{1,\dots,M\}$, such that $|\mathcal{I}'|+|\mathcal{J}| = N$ and $\mathcal{I}' \subset \mathcal{I}$.

\section{A simpler match of distributions}

Our aim, in what follows, will be to show that
\begin{align}
\label{simpler-match}
\mathbb{P}_{\rm 6v}(\mathcal{I})
=
\mathbb{P}_{\rm cHL}(\mathcal{I}),
\qquad
\forall\ \mathcal{I} \subset \{1,\dots,M\}.
\end{align}
Note that \eqref{simpler-match} is already implied by Theorem \ref{thm:2=1} and the equations \eqref{6v-I.2}, \eqref{cHL-I.2}; however we will seek a different proof that takes us through the general theory of Cherednik--Dunkl operators \eqref{eq:Yi}.

\section{Setting up notations}

The main ingredient of our alternative proof will be the action of the Cherednik--Dunkl operators on the Cauchy reproducing kernel for Macdonald polynomials. Results in this direction were previously obtained by Kirillov--Noumi \cite{KirillovN}, and we will refer to that paper for a number of formulae that we require. Ultimately we will be interested in the Hall--Littlewood specialization $p=0$, but the results of \cite{KirillovN} apply for generic $p$, so we will preserve this parameter for as long as possible.

We will make frequent use of the Macdonald reproducing kernel  
\begin{align}
\label{pi-xy}
\Pi(x;y)
\index{aP@$\Pi(x;y)$; Cauchy reproducing kernel}
:=
\prod_{i=1}^{N}
\prod_{j=1}^{M}
\frac{(qx_i y_j;p)}{(x_i y_j;p)},
\qquad
(x;p)
:=
\prod_{k=0}^{\infty}
(1-p^k x),
\end{align}
as well as Cherednik--Dunkl operators which act on the alphabet $(y_1,\dots,y_M)$:
\begin{align}
\label{Y-tilde}
\tilde{Y}_i
\index{Y@$\tilde{Y}_i$; Cherednik--Dunkl operators}
:=
q^{i-1} \tilde{T}_i\cdots \tilde{T}_{M-1} \omega 
\tilde{T}_{1}^{-1} \cdots \tilde{T}_{i-1}^{-1},
\qquad
1 \leq i \leq M,
\end{align}
where
\begin{align*}
\tilde{T}_i
\index{T@$\tilde{T}_i$; reversed-alphabet Hecke generator}
=
q - \frac{qy_i-y_{i+1}}{y_i-y_{i+1}} (1-\mathfrak{s}_i),
\qquad
\omega
=
\mathfrak{s}_{M-1} \dots \mathfrak{s}_1 \tau_1.
\end{align*}
All operators $\mathfrak{s}_i$ and $\tau_1$ are defined as previously, but are now assumed to act on the variables $(y_1,\dots,y_M)$. Note that the operators $\tilde{Y}_i$ also differ from their previous incarnations \eqref{eq:Yi} by an overall factor of $q^{i-1}$. They have the following eigenaction on the non-symmetric Macdonald polynomials:
\index{am@$\langle \mu \rangle_i$; Cherednik--Dunkl eigenvalues}
\begin{align*}
\tilde{Y}_i
E_{\mu}(y_1,\dots,y_M;p,q)
=
\langle \mu \rangle_i
E_{\mu}(y_1,\dots,y_M;p,q),
\qquad
\langle \mu \rangle_i
=
p^{\mu_i} q^{\rho_i(\mu)+M-1},
\end{align*}
where $\mu = (\mu_1,\dots,\mu_M)$ is a composition of length $M$ and 
$\rho_i(\mu)$ is given by equation \eqref{rhomu}.

\section{Expectations from the action of $\tilde{Y}_i$ on the Cauchy kernel}

Let us consider a non-symmetric Macdonald measure
on compositions $\mu \in \mathbb{N}^M$, defined by
\begin{align}
\label{nsM}
\mathbb{P}_{\rm nsM}(\mu)
\index{P@$\mathbb{P}_{\rm nsM}(\mu)$; non-symmetric Macdonald measure}
:=
\frac{1}{\Pi(x;y)}
\cdot
d_{\mu}(p,q)
E_{\mu}(y_1,\dots,y_M;p,q)
Q_{\mu^{+}}(x_1,\dots,x_N;p,q),
\end{align}
where $Q_{\mu^{+}}(x_1,\dots,x_N;p,q)$ denotes a symmetric Macdonald $Q$-polynomial (we follow the same conventions as in \cite[Chapter VI]{Macdonald}) and $d_{\mu}(p,q)$ is a simple, combinatorially defined constant (it can be found, for example, in \cite[Lemma 2.4]{Marshall}).
The precise form of $d_{\mu}(p,q)$ will not be relevant to us, but let us note that it satisfies $d_{\mu}(0,q) = 1$ for all compositions $\mu$. For a fixed partition $\lambda = (\lambda_1 \geq \cdots \geq \lambda_M \geq 0)$ one has
\begin{align*}
\sum_{\mu: \mu^{+} = \lambda} 
d_{\mu}(p,q) E_{\mu}(y_1,\dots,y_M;p,q) 
= 
P_{\lambda}(y_1,\dots,y_M;p,q),
\end{align*} 
where $P_{\lambda}(y_1,\dots,y_M;p,q)$ denotes a symmetric Macdonald $P$-polynomial, and we conclude that 
\begin{align}
\label{nsM-sum}
\sum_{\mu} \mathbb{P}_{\rm nsM}(\mu)
=
\frac{1}{\Pi(x;y)}
\sum_{\lambda}
P_{\lambda}(y_1,\dots,y_M;p,q)
Q_{\lambda}(x_1,\dots,x_N;p,q)
=
1,
\end{align}
by virtue of the Cauchy identity for symmetric Macdonald polynomials \cite[Chapter VI]{Macdonald}. Furthermore, \cf\ \cite{BorodinC}, the weights \eqref{nsM} are nonnegative and the series \eqref{nsM-sum} converges if $x_i,y_j \geq 0$ and $x_i y_j < 1$ for all $1 \leq i \leq N$, $1 \leq j \leq M$. Hence, \eqref{nsM} can be viewed as an honest probability measure on length-$M$ compositions.

\begin{prop}
We have the following expression for the expected value of the observable $\langle \mu \rangle_i$ with respect to the probability measure \eqref{nsM}:
\begin{align}
\label{act-Y_i}
\mathbb{E}_{\rm nsM} \left[ \langle \mu \rangle_i \right]
=
\frac{1}{\Pi(x;y)}
\cdot
\tilde{Y}_i\, \Pi(x;y),
\end{align}
and more generally, for any $1 \leq I_1 \leq \cdots \leq I_k \leq M$,
\begin{align}
\label{act-Y_iii}
\mathbb{E}_{\rm nsM}
\left[ 
\langle \mu \rangle_{I_1}
\cdots
\langle \mu \rangle_{I_k}
\right]
=
\frac{1}{\Pi(x;y)}
\cdot
\tilde{Y}_{I_1} \cdots \tilde{Y}_{I_k}\, \Pi(x;y).
\end{align}
Note that since the operators $\tilde{Y}_{I_j}$ commute, their order on the right hand side of \eqref{act-Y_iii} is inessential.

\end{prop}

\begin{proof}
Writing the Cauchy kernel \eqref{pi-xy} in de-symmetrized form
\begin{align*}
\Pi(x;y)
=
\sum_{\mu}
d_{\mu}(p,q) E_{\mu}(y_1,\dots,y_M;p,q)
Q_{\mu^{+}}(x_1,\dots,x_N;p,q)
\end{align*}
and then applying $\tilde{Y}_i$, we obtain
\begin{align}
\tilde{Y}_i \,
\Pi(x;y)
=
\sum_{\mu}
\langle \mu \rangle_i
d_{\mu}(p,q) E_{\mu}(y_1,\dots,y_M;p,q)
Q_{\mu^{+}}(x_1,\dots,x_N;p,q)
=
\sum_{\mu}
\langle \mu \rangle_i
\Pi(x;y)
\mathbb{P}_{\rm nsM}(\mu).
\end{align}
The result \eqref{act-Y_i} follows, while \eqref{act-Y_iii} is just repeated applications of the above reasoning (it is easy to see that convergence of the series causes no issues as long as $|x_i y_j| < 1$ for all $1 \leq i \leq N$ and $1 \leq j \leq M$).
\end{proof}

\section{Integral expression for expectations}

We now use \cite[Lemma 6.2]{KirillovN} to evaluate the right hand side of equation \eqref{act-Y_iii}. With the conventions $\mathcal{I} = \{I_1 < \cdots < I_k\}$, $\b{\mathcal{I}} = \{1,\dots,M\} \backslash \mathcal{I}$, it reads
\begin{align}
\label{HK-sum}
\frac{1}{\Pi(x;y)}
\cdot
\tilde{Y}_{I_1} \cdots \tilde{Y}_{I_k} \Pi(x;y)
=
\sum_{m=0}^{k}\ 
\sum_{\substack{
\mathcal{H} \subset \mathcal{I}, \mathcal{K} \subset \{1,\dots,N\} 
\\ \vspace{-0.06cm} \\
|\mathcal{H}| = |\mathcal{K}| = m
}}
q^{-\ell(\b{\mathcal{I}}; \mathcal{H}) + (M-N)k - ||\mathcal{I}||}
\cdot
a_{\mathcal{K}}(x)
\cdot
h^{\mathcal{H}}_{\mathcal{K}}(x;y),
\end{align}
with the second sum taken over all sets $\mathcal{H} = \{H_1 < \cdots < H_m\}$ and 
$\mathcal{K} = \{K_1< \cdots < K_m\}$ whose cardinalities satisfy $|\mathcal{H}| = |\mathcal{K}| = m$, and where we have defined
\begin{align*}
h^{\mathcal{H}}_{\mathcal{K}}(x;y)
=
\sum_{\sigma \in \mathfrak{S}_m}
\left[
\prod_{i=1}^{m}
\left(
\frac{q-1}{1-q x_{K_{\sigma(i)}} y_{H_i}}
\prod_{j = H_i+1}^{M}
\frac{q(1- x_{K_{\sigma(i)}} y_j)}{1- q x_{K_{\sigma(i)}} y_j}
\right)
\prod_{1\leq i<j \leq m}
\frac{q x_{K_{\sigma(i)}}-x_{K_{\sigma(j)}}}
{x_{K_{\sigma(i)}}-x_{K_{\sigma(j)}}}
\right]
\end{align*}
as in \cite[Equation (4.2)]{KirillovN}, and
\begin{align*}
a_{\mathcal{K}}(x)
=
\prod_{i \in \mathcal{K}, j \not\in \mathcal{K}}
\frac{q x_i - x_j}{x_i - x_j},
\qquad
\text{\cite[Equation (5.4)]{KirillovN}},
\end{align*}
\begin{align*}
\ell(\mathcal{A};\mathcal{B}) = \#\{ (i,j) \in \mathcal{A} \times \mathcal{B}: i>j\},
\qquad
\text{\cite[Equation (6.18)]{KirillovN}},
\end{align*}
as well as $||\mathcal{I}|| = \sum_{i=1}^{k} I_i$.

\begin{prop}
There holds
\begin{multline}
\label{Y_iii-integral}
\frac{1}{\Pi(x;y)}
\cdot
\tilde{Y}_{I_1} \cdots \tilde{Y}_{I_k} \Pi(x;y)
=
\frac{q^{\frac{k(k-1)}{2}+(M-N)k - ||\mathcal{I}||}}{(2\pi \sqrt{-1})^k}
\\
\times
\oint_{C_1} \cdots \oint_{C_k}
\prod_{1\leq \alpha < \beta \leq k}
\frac{w_{\alpha}-w_{\beta}}{w_{\alpha}-qw_{\beta}}
\prod_{\ell=1}^{k}
\left(
\frac{1}{1-q y_{I_{\ell}} w_{\ell}}
\cdot
\prod_{j=I_{\ell}+1}^{M}
\frac{1-y_j w_{\ell}}{1-q y_j w_{\ell}}
\cdot
\prod_{r=1}^{N}
\frac{x_r-qw_{\ell}}{x_r-w_{\ell}}
\right)
\frac{dw_{\ell}}{w_{\ell}}
\end{multline}
where for each $1 \leq i \leq k$, the (positively oriented) contour $C_i$ consists of two disjoint pieces: a small circle $c[x]$ around the set $\{x_r\}_{r=1}^{N}$, and a small circle $c[i]$ around $0$ such that $q \cdot c[i] \subset c[i] \subset q \cdot c[i+1]$ (see Figure \ref{fig:two-pieces}).
\end{prop}

\begin{figure}
\begin{tikzpicture}[scale=1.3,>=stealth]
\draw[gray,densely dotted,thick,->] (-4,0) -- (8,0);
\draw[gray,densely dotted,thick,->] (0,-3.5) -- (0,3.5);
\draw[thick,dashed] (0,0) ellipse [x radius=0.8cm, y radius=0.8cm];
\draw[thick] (0,0) ellipse [x radius=1.3cm, y radius=1.3cm];
\draw[thick,dashed] (0,0) ellipse [x radius=1.8cm, y radius=1.8cm];
\draw[thick] (0,0) ellipse [x radius=2.3cm, y radius=2.3cm];
\draw[thick] (0,0) ellipse [x radius=3.3cm, y radius=3.3cm];
\draw[thick] (6.45,0) ellipse [x radius=1.3cm, y radius=1.3cm];
\node[left] at (0.85,-0.2) {\tiny $q w_1$};
\node[left] at (1.25,0.4) {\tiny $w_1$};
\node[left] at (1.85,-0.2) {\tiny $q w_2$};
\node[left] at (2.35,0.4) {\tiny $w_2$};
\node[text centered] at (2.7,0.4) {$\cdots$};
\node[left] at (3.35,0.4) {\tiny $w_k$};
\node[above] at (6.45,1.3) {\tiny $w_1,\dots,w_k$};
\node at (0,0) {$\bullet$};
\node[below left] at (0,0) {$0$};
\foreach\x in {6,6.3,6.6,6.9}{
\node at (\x,0) {$\bullet$};
}
\node[below] at (6.45,-0.1) {$\{x_r\}_{r=1}^{N}$};
\end{tikzpicture}
\caption{Integration contours for equation \eqref{Y_iii-integral}. Each contour $C_i$ is the union of the small circle $c[x]$ surrounding the points $\{x_1,\dots,x_N\}$, and a small circle $c[i]$ surrounding the origin. The circles surrounding the origin are $q$-nested; one has $q \cdot c[i] \subset c[i] \subset q \cdot c[i+1]$ for all $1 \leq i \leq k-1$.}
\label{fig:two-pieces}
\end{figure}
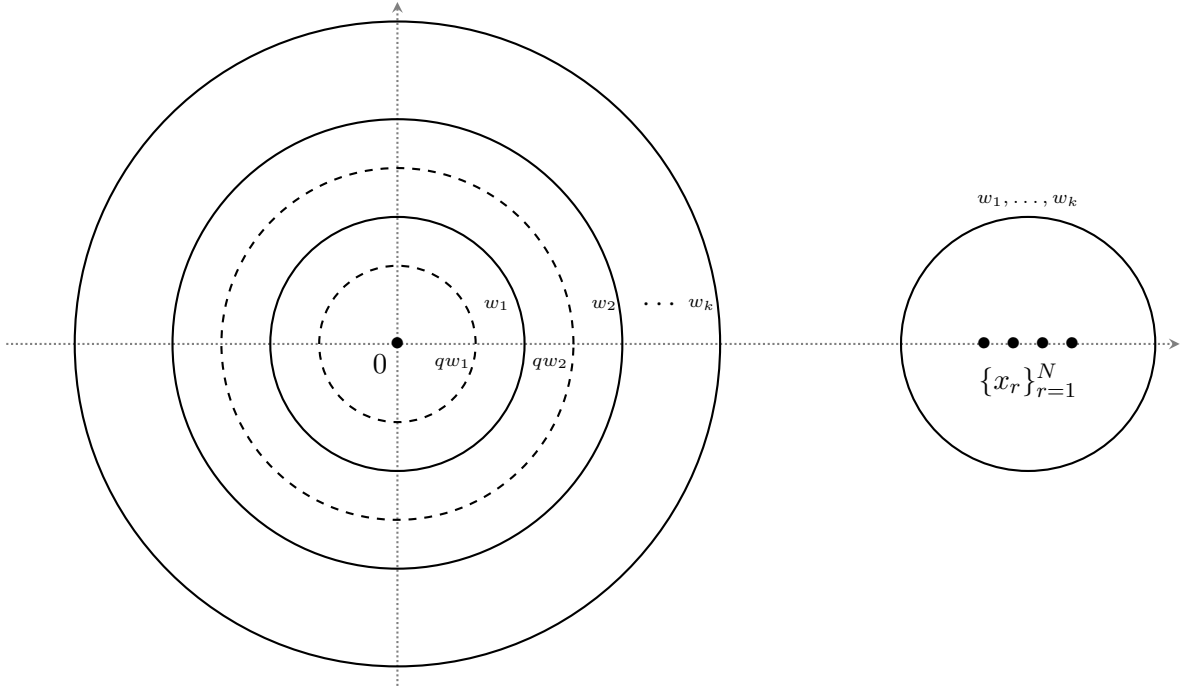

\begin{proof}
We will evaluate the integral \eqref{Y_iii-integral} by computing the sum of all residues and matching it to the expression \eqref{HK-sum}. The foregoing computation closely follows the material after equation (9.8) in \cite{BorodinP1}. Let us consider the following integral, which is equal to the right hand side of \eqref{Y_iii-integral} after suppressing the irrelevant factor $q^{(M-N)k-||\mathcal{I}||}$, that is already common to both \eqref{HK-sum} and \eqref{Y_iii-integral}:
\begin{align}
\label{int-1}
\Upsilon(I_1,\dots,I_k)
:=
\frac{q^{\frac{k(k-1)}{2}}}{(2\pi \sqrt{-1})^k}
\oint_{C_1} \cdots \oint_{C_k}
\prod_{1 \leq \alpha < \beta \leq k}
\frac{w_{\alpha}-w_{\beta}}{w_{\alpha}-qw_{\beta}}
\prod_{\ell=1}^{k}
\frac{f_{I_{\ell}}(w_{\ell})}{w_{\ell}}
dw_{\ell},
\end{align}
with
\begin{align*}
f_i(w)
:=
\frac{1}{1-q y_i w}
\prod_{j=i+1}^{M}
\frac{1- y_j w}{1-q y_j w}
\prod_{r=1}^{N}
\frac{x_r-qw}{x_r-w},
\end{align*}
and observe that $f_i(0) = 1$. Let us now split the integral $\Upsilon$ into a sum of $2^k$ integrals indexed by subsets $\mathcal{L} \subset \{1,\dots,k\}$ of cardinality $|\mathcal{L}| = m$, with the condition that 
$w_{\ell}$ for $\ell \not\in \mathcal{L}$ are integrated around $0$, while $w_{\ell}$ for $\ell \in \mathcal{L}$ are integrated around $\{x_r\}_{r=1}^N$. We shall write $\mathcal{L} = \{L_1 < \cdots < L_m\}$, $||\mathcal{L}|| = \sum_{i=1}^{m} L_i$, $\{1,\dots,k\} \backslash \mathcal{L} = \{P_1 < \cdots < P_{k-m}\}$. 

Because of their $q$-nesting property, the contours around $0$ corresponding to $\{w_{P_1},\dots,w_{P_{k-m}}\}$ can be shrunk to $0$ in the order $w_{P_1},\dots,w_{P_{k-m}}$ without crossing any poles (other than the pole at $0$ itself). Computing residues at $0$, each such integration over $w_{P_j}$ produces the factor $q^{P_j-k}$, coming from the product $\prod_{1 \leq \alpha < \beta \leq k}(w_{\alpha} - w_{\beta})/(w_{\alpha} - q w_{\beta})$. Renaming $w_{L_j} = z_j$, the integral \eqref{int-1} thus becomes
\begin{align}
\label{int-2}
\Upsilon(I_1,\dots,I_k)
=
\sum_{m=0}^{k}
\sum_{\substack{\mathcal{L} \subset \{1,\dots,k\} \\ |\mathcal{L}| = m}}
q^{km - ||\mathcal{L}||}
\cdot
\oint_{c[x]} \cdots \oint_{c[x]}
\prod_{1 \leq \alpha < \beta \leq m}
\frac{z_{\alpha}-z_{\beta}}{z_{\alpha}-qz_{\beta}}
\prod_{i=1}^{m}
\frac{f_{I_{L_i}}(z_i)}{z_i}
\frac{dz_i}{2\pi\sqrt{-1}}.
\end{align}
Now we compute the remaining $m$ integrals by taking residues at $z_i = x_r$ for various $1 \leq r \leq N$. Taking these residues produces a sum over subsets $\mathcal{K} = \{K_1 < \cdots < K_m\} \subset \{1,\dots,N\}$ (which label the poles $\{x_{K_1},\dots,x_{K_m}\}$ that get selected\footnote{Because of the product 
$\prod_{1 \leq \alpha < \beta \leq m} (z_{\alpha} - z_{\beta})$ present in the integrand of \eqref{int-2}, the selected poles are necessarily pairwise distinct.}) and a sum over permutations $\sigma \in \mathfrak{S}_m$ (which specify the order in which those poles are chosen). We thus obtain
\begin{multline}
\label{int-3}
\Upsilon(I_1,\dots,I_k)
\\
=
\sum_{m=0}^{k}
\sum_{\substack{\mathcal{L} \subset \{1,\dots,k\} \\ |\mathcal{L}| = m}}
q^{km - ||\mathcal{L}||}
\sum_{\substack{\mathcal{K} \subset \{1,\dots,N\} \\ |\mathcal{K}| = m}}
\sum_{\sigma \in \mathfrak{S}_m}
\prod_{1\leq \alpha < \beta \leq m}
\frac{x_{K_{\sigma(\alpha)}}-x_{K_{\sigma(\beta)}}}
{x_{K_{\sigma(\alpha)}}-q x_{K_{\sigma(\beta)}}}
\prod_{i=1}^{m}
\frac{{\rm Res}_{z=x_{K_{\sigma(i)}}}\left[ f_{I_{L_i}}(z) \right]}{x_{K_{\sigma(i)}}},
\end{multline}
where the residue present in the summand of \eqref{int-3} is given by
\begin{align*}
\frac{1}{x_K}
{\rm Res}_{z=x_K}
\left[
f_I(z)
\right]
=
\frac{q-1}{1-q x_K y_I}
\cdot
\prod_{j=I+1}^{M}
\frac{1-x_K y_j}{1-q x_K y_j}
\cdot
\prod_{\substack{r=1 \\ r\not= K}}^{N}
\frac{q x_K-x_r}{x_K-x_r}.
\end{align*}
We can now directly match \eqref{int-3} with \eqref{HK-sum}. If we introduce the set $\mathcal{H} \subset \mathcal{I}$ via the correspondence
\begin{align}
\label{HL-id}
\mathcal{H} = \{H_1 < \cdots < H_m\} = \{I_{L_1} < \cdots < I_{L_m}\},
\end{align}
one readily checks that \eqref{int-3} and \eqref{HK-sum} agree in all factors, apart from potential discrepancies in the powers of $q$. For this, one needs to check that
\begin{align}
\label{q-exponent}
\ell(\b{\mathcal{I}}; \mathcal{H})
-
\sum_{i=1}^{m}
(M-H_i)
=
||\mathcal{L}|| - km,
\end{align}
which allows all $q$ exponents to be matched. We notice that $\ell(\b{\mathcal{I}}; \mathcal{H})$ equals $\sum_{i=1}^{m} ((M-I_{L_i})-(k-L_i))$, since $(M-I_{L_i})$ is the number of elements of $\{1,\dots,M\}$ greater than $H_i \equiv I_{L_i}$, and $(k-L_i)$ is the number of elements of $\{I_1,\dots,I_k\}$ greater than $H_i \equiv I_{L_i}$. Substituting this into the left hand side of \eqref{q-exponent}, there is some direct cancellation, and we obtain a match with the right hand side. The proof of \eqref{Y_iii-integral} is complete.
\end{proof}

\section{Matching with the six-vertex model height function}

Let us now compare the contour integral \eqref{Y_iii-integral} with Theorem 9.8 of \cite{BorodinP1}. We begin by recalling the latter result:

\begin{defn}[Height function]
Let $\mathcal{C}$ be a configuration of the stochastic six-vertex model 
\eqref{six-vert} in the positive quadrant $\mathbb{N}^2$. Fixing two positive integers $M,N \geq 1$, we define the height function $\mathfrak{h}(M,N)$ \index{h@$\mathfrak{h}(M,N)$; height function} as the number of paths in $\mathcal{C}$ which pass through one of the horizontal edges $(M,j) \rightarrow (M+1,j)$, $1 \leq j \leq N$. In other words, $\mathfrak{h}(M,N)$ is the number of paths which pass through or below the point $(M,N)$.
\end{defn}
\begin{thm}
Consider the stochastic six-vertex model \eqref{six-vert} with horizontal rapidities $x_i$ (where the index $i$ can assume any order), vertical rapidities $v_j$ (with $j$ increasing from left to right), and let $\{\mathfrak{h}(\theta_1,N),\dots,\mathfrak{h}(\theta_k,N)\} \equiv \{\mathfrak{h}(\theta_1),\dots,\mathfrak{h}(\theta_k)\}$ denote the height function of the model measured at horizontal coordinates $\theta_1 \geq \cdots \geq \theta_k \geq 1$ along the $N$-th horizontal line. One has
\begin{align}
\label{height-exp}
\mathbb{E}_{\rm 6v}
\prod_{i=1}^{k}
q^{\mathfrak{h}(\theta_i)}
=
\frac{q^{\frac{k(k-1)}{2}}}{(2\pi \sqrt{-1})^k}
\oint_{C_1}
\cdots
\oint_{C_k}
\prod_{1\leq \alpha < \beta \leq k}
\frac{w_{\alpha}-w_{\beta}}{w_{\alpha}-q w_{\beta}}
\prod_{i=1}^{k}
\left(
\prod_{j=1}^{\theta_i-1}
\frac{1-v_j w_i}{1- q v_j w_i}
\prod_{r=1}^{N}
\frac{x_r-q w_i}{x_r-w_i}
\right)
\frac{dw_i}{w_i}
\end{align}
with the same integration contours as in Figure \ref{fig:two-pieces}. In the notation of \cite[Theorem 9.8]{BorodinP1}, we have chosen ${\sf s}_j = q^{-1/2}$, $\xi_j = v_j^{-1} q^{-1/2}$ and $u_r = x_r^{-1}$ in order to write down \eqref{height-exp}.

\end{thm}

\begin{prop}
\label{prop:6v-YYY}
In terms of all preceding notation, one has
\begin{align}
\label{prop-d.5}
\mathbb{E}_{\rm 6v}
\left[
\prod_{i=1}^{k}
\frac{q^{\mathfrak{h}(\theta_i)}-q^{\mathfrak{h}(\theta_i+1)+1}}
{1-q}
\right]
=
\frac{q^{(N-M)k+||\mathcal{I}||}}{\Pi(x;y)}
\cdot
\tilde{Y}_{I_1} \cdots \tilde{Y}_{I_k} \Pi(x;y),
\end{align}
provided that
\begin{equation}
\label{ids}
\{\theta_1,\dots,\theta_k\}
=
\{M-I_1+1,\dots,M-I_k+1\},
\qquad
(v_1,\dots,v_M) 
= 
(y_M,\dots,y_1).
\end{equation}
\end{prop}

\begin{proof}
Let us check the effect of taking differences $(q^{\mathfrak{h}(\theta_i)}-q^{\mathfrak{h}(\theta_i+1)+1})/(1-q)$ on the integrand of \eqref{height-exp}. In view of the relation
\begin{align*}
\frac{1}{1-q}
\left(
\prod_{j=1}^{\theta-1}
\frac{1-v_j w}{1-q v_j w}
-
q \cdot
\prod_{j=1}^{\theta}
\frac{1-v_j w}{1-q v_j w}
\right)
&=
\frac{1}{1-q}
\prod_{j=1}^{\theta-1}
\frac{1-v_j w}{1-q v_j w}
\left(
1-q\cdot \frac{1-v_{\theta} w}{1-q v_{\theta} w}
\right)
\\
&=
\prod_{j=1}^{\theta-1}
\frac{1-v_j w}{1-q v_j w}
\cdot
\frac{1}{1-q v_{\theta} w},
\end{align*}
one clearly has
\begin{multline}
\label{height-exp2}
\mathbb{E}_{\rm 6v}
\left[
\prod_{i=1}^{k}
\frac{q^{\mathfrak{h}(\theta_i)}-q^{\mathfrak{h}(\theta_i+1)+1}}
{1-q}
\right]
=
\frac{q^{\frac{k(k-1)}{2}}}{(2\pi \sqrt{-1})^k}
\\
\times
\oint_{C_1}
\cdots
\oint_{C_k}
\prod_{1\leq \alpha < \beta \leq k}
\frac{w_{\alpha}-w_{\beta}}{w_{\alpha}-q w_{\beta}}
\prod_{i=1}^{k}
\left(
\frac{1}{1-q v_{\theta_i} w_i}
\cdot
\prod_{j=1}^{\theta_i-1}
\frac{1-v_j w_i}{1- q v_j w_i}
\prod_{r=1}^{N}
\frac{x_r-q w_i}{x_r-w_i}
\right)
\frac{dw_i}{w_i},
\end{multline}
which can be seen to coincide exactly with \eqref{Y_iii-integral} under the identifications \eqref{ids}.

\end{proof}

\section{Reduction to non-symmetric Hall--Littlewood measures}

In the remainder of the chapter, we shall analyze the result \eqref{prop-d.5} in the case $p=0$. Note that the six-vertex side of \eqref{prop-d.5} is manifestly independent of $p$, so this specialization only affects the interpretation of the right hand side of that equation. Let us begin with the reduction of \eqref{act-Y_iii} to non-symmetric Hall--Littlewood measures:
\begin{cor}
\label{cor:nsHL}
Assuming $p=0$, from \eqref{act-Y_iii} we recover
\begin{align}
\label{p=0-Y_iii}
\mathbb{E}_{\rm nsHL}
\left[
\prod_{i=1}^{k}
\left(
{\bm 1}_{\mu_{I_i} = 0}
\right)
q^{\#\{j > I_i : \mu_j = 0\}}
\right]
=
\left.
\frac{1}{\Pi(x;y)}
\cdot
\tilde{Y}_{I_1} \cdots \tilde{Y}_{I_k} \Pi(x;y)
\right|_{p=0},
\end{align}
where the expectation is taken with respect to the non-symmetric Hall--Littlewood measure
\begin{align}
\label{nsHL}
\mathbb{P}_{\rm nsHL}(\mu)
\index{P@$\mathbb{P}_{\rm nsHL}(\mu)$; non-symmetric Hall--Littlewood measure}
=
\prod_{i=1}^{N}
\prod_{j=1}^{M}
\frac{1-x_i y_j}{1-q x_i y_j}
\cdot
E_{\mu}(y_1,\dots,y_M;0,q)
Q_{\mu^+}(x_1,\dots,x_N;0,q),
\end{align}
which is the $p=0$ specialization of \eqref{nsM}. Note that, in terms of our previously defined measure \eqref{W-MN-mu}, one has $\mathbb{P}_{\rm nsHL}(\mu;y_1,\dots,y_M) = \mathbb{W}_{M,N}(\tilde{\mu};y_M,\dots,y_1)$.
\end{cor}

\begin{rmk}
\label{rmk:11.7}
When $q=1$, the left hand side of \eqref{p=0-Y_iii} is the $k$-point correlation function for the point process of the locations of zeros in $\mu$, \ie\ it becomes 
\begin{align*}
\mathbb{E}_{\rm nsHL} \left[ {\bm 1}_{\mu_{I_1} = \cdots = \mu_{I_k} = 0} \right] = \mathbb{P}_{\rm nsHL} \left[ \mu_{I_1} = \cdots = \mu_{I_k} = 0 \right].
\end{align*} 
\end{rmk}

Now by combining the results of Proposition \ref{prop:6v-YYY} and Corollary \ref{cor:nsHL}, we arrive at the following match of expectations:
\begin{cor}
With the same matching of parameters \eqref{ids}, one has
\begin{align}
\label{cor-d7}
\mathbb{E}_{\rm 6v}
\left[
\prod_{i=1}^{k}
{\bm 1}_{\{\text{no path at}\ (\theta_i,N)\}} 
q^{\mathfrak{h}(\theta_i)}
\right]
=
\mathbb{E}_{\rm nsHL}
\left[
\prod_{i=1}^{k}
\left(
{\bm 1}_{\mu_{I_i}=0} 
\right)
q^{N-M+I_i+\#\{j > I_i : \mu_j = 0\}}
\right].
\end{align}

\end{cor}

\begin{proof}
Indeed, one easily sees that the left hand side of \eqref{cor-d7} is the same as the left hand side of Proposition \ref{prop:6v-YYY}, and the right hand side of \eqref{cor-d7} arises by multiplying \eqref{p=0-Y_iii} by $q^{(N-M)k+||\mathcal{I}||}$.
\end{proof}

\section{Matching the underlying distributions}

We would now like to strengthen \eqref{cor-d7} to a statement at the level of the underlying distributions. Namely, we wish to prove the following result:
\begin{prop}
Fix any subset $\mathcal{I} \subset \{1,\dots,M\}$. Under the identification $(v_1,\dots,v_M) = (y_M,\dots,y_1)$, we have
\begin{align}
\label{prop-d.9}
\mathbb{P}_{\rm 6v}
\left[
\text{no path at edge $(M-i+1,N) \rightarrow (M-i+1,N+1)$},\ i \in \mathcal{I}
\right]
=
\mathbb{P}_{\rm nsHL}
\left[
\mu_i=0\ \text{for each}\ i \in \mathcal{I}
\right].
\end{align}
This statement is equivalent to \eqref{simpler-match}.
\end{prop}

\begin{proof}
We see that the distribution match \eqref{prop-d.9} is a trivial corollary of \eqref{cor-d7} when $q=1$, \cf\ Remark \ref{rmk:11.7}. Now consider \eqref{cor-d7} for generic $q$. We remark that, for $\theta_i=M-I_i+1$, one has
\begin{align*}
\mathfrak{h}(\theta_i)
&=
\#\{\text{paths situated at, or right of}\ (\theta_i,N)\}
\\
&=
N - \#\{\text{paths strictly left of}\ (\theta_i,N)\}
\\
&=
N - \theta_i + 1 + \#\{\text{empty vertices among}\ (1,N),\dots,(\theta_i-1,N)\}
\\
&=
N - M + I_i + \#\{\text{empty vertices among}\ (1,N),\dots,(M-I_i,N)\}.
\end{align*}
Hence, if we pair up the random set $\{\text{empty vertices among}\ (1,N),\dots,(M-I_i,N)\}$ on the side of the six-vertex model with the random set $\{j>I_i:\mu_j=0\}$ on the non-symmetric Hall--Littlewood side, we see that powers of $q$ coincide in the two expectations of \eqref{cor-d7}. It follows that the difference of the two sides of \eqref{cor-d7} can be written as linear combinations of 
\begin{multline}
\Delta(\mathcal{L}) 
:=
\mathbb{P}_{\rm 6v}
\left[
\text{no path at edge $(M-\ell+1,N) \rightarrow (M-\ell+1,N+1)$},\ \ell \in \mathcal{L}
\right]
\\
-
\mathbb{P}_{\rm nsHL}
\left[
\mu_\ell=0\ \text{for each}\ \ell \in \mathcal{L}
\right] 
\end{multline}
for various subsets $\mathcal{L} \subset \{1,\dots,M\}$, with polynomial in $q$ coefficients. The number of the unknowns $\Delta(\mathcal{L})$ is $2^M$, while \eqref{cor-d7} (for various choices of $\mathcal{I} \subset \{1,\dots,M\}$) imposes $2^M$ homogeneous linear equations on them. The resulting linear system admits non-trivial solutions (solutions other than $\Delta(\mathcal{L})=0$, for all $\mathcal{L}$) if and only if its matrix is singular; but it is clear that it is non-singular when $q=1$, given that $\Delta(\mathcal{L})=0$ is the only possible solution in that case (since \eqref{prop-d.9} holds for $q=1$). The matrix of the linear system is, therefore, non-singular for generic $q$, and we conclude that $\Delta(\mathcal{L})=0$ for all $\mathcal{L}$.

The equivalence of \eqref{simpler-match} and \eqref{prop-d.9} is immediate, taking into account the translation between measures $\mathbb{P}_{\rm nsHL}(\mu;y_1,\dots,y_M) = \mathbb{W}_{M,N}(\tilde{\mu};y_M,\dots,y_1)$, which not only implements the alphabet reversal $(v_1,\dots,v_M) = (y_M,\dots,y_1)$ but allows each instance of $M-i+1$ on the left hand side of \eqref{prop-d.9} to be replaced with $i$.
\end{proof}

\section{Hall--Littlewood difference operators}

Let us recall the form of the first-order Macdonald difference operators \cite[Chapter VI]{Macdonald} acting on an alphabet $(x_1,\dots,x_N)$:
\begin{align}
\label{D1N}
D_N^1 \equiv
D_N
\index{D@$D_N$; Macdonald difference operators}
=
\sum_{i=1}^{N}
\Big(
\prod_{\substack{j\not=i \\ j=1}}^{N}
\frac{qx_i-x_j}{x_i-x_j}
\Big)
\tau_i,
\end{align}
where $\tau_i$ \index{t@$\tau_i$; $p$-shift operator} denotes the $p$-shift operator with action
\begin{align*}
\tau_i \cdot h(x_1,\dots,x_N) = h(x_1,\dots,p x_i,\dots,x_N)
\end{align*} 
on arbitrary functions $h$ of our alphabet. They have the following eigenaction on the symmetric Macdonald polynomials $Q_{\lambda}(x_1,\dots,x_N;p,q)$:
\begin{align}
\label{mac-ops}
D_N Q_{\lambda}(x_1,\dots,x_N;p,q)
=
\left(
\sum_{i=1}^{N} p^{\lambda_i} q^{N-i}
\right)
Q_{\lambda}(x_1,\dots,x_N;p,q).
\end{align}
In what follows we will focus on the Hall--Littlewood limit, tacitly assuming that $p=0$. This converts equation \eqref{mac-ops} into
\begin{align}
\label{HL-ops}
D_N Q_{\lambda}(x_1,\dots,x_N)
=
\left(
\frac{1-q^{N-\ell(\lambda)}}{1-q}
\right)
Q_{\lambda}(x_1,\dots,x_N),
\end{align}
where $\ell(\lambda)$ denotes the length of $\lambda$, and where the participating functions are the symmetric Hall--Littlewood polynomials. Now recall the branching rule for Hall--Littlewood polynomials,
\begin{align}
\label{Q-branch}
Q_{\lambda}(x_1,\dots,x_N)
=
\sum_{\nu}
Q_{\nu}(x_1,\dots,x_{N-1})
Q_{\lambda/\nu}(x_N),
\end{align}
and apply to this equation the linear combination 
$q^{-N+1} (D_N-D_{N-1})$ of the difference operators \eqref{D1N} at $p=0$. Then as a result of \eqref{HL-ops}, we see that
\begin{align}
\label{DN-diff}
\frac{D_N-D_{N-1}}{q^{N-1}}
Q_{\lambda}(x_1,\dots,x_N)
=
\sum_{\nu}
\left(
\frac{q^{-\ell(\nu)}-q^{-\ell(\lambda)+1}}{1-q}
\right)
Q_{\nu}(x_1,\dots,x_{N-1})
Q_{\lambda/\nu}(x_N).
\end{align}
Since $Q_{\lambda/\nu}(x_N) = 0$ unless $\lambda \succ \nu$, the sum on the right hand side of \eqref{DN-diff} is taken over partitions $\nu$ such that $\ell(\nu) = \ell(\lambda)$ or $\ell(\nu) = \ell(\lambda)-1$. Due to the factor $(q^{-\ell(\nu)} - q^{-\ell(\lambda)+1})$ present in the summand, all terms such that $\ell(\nu) = \ell(\lambda)-1$ will give a vanishing contribution, and we then find that
\begin{align}
\label{DN-diff2}
\frac{D_N-D_{N-1}}{q^{N-1}}
Q_{\lambda}(x_1,\dots,x_N)
=
q^{-\ell(\lambda)}
\sum_{\nu}
\bm{1}_{\ell(\nu) = \ell(\lambda)}
\cdot
Q_{\nu}(x_1,\dots,x_{N-1})
Q_{\lambda/\nu}(x_N).
\end{align}

\section{Restoring the subset $\mathcal{J}$}

Now assume that $\kappa$ is a partition with length $\ell(\kappa) \leq k$, for some $k \geq 1$, and fix a subset $\mathcal{J} = \{J_1 < \cdots < J_{N-k}\} \subset \{1,\dots,N\}$. Following \cite[Section 4]{BCGS}, by repeated application of \eqref{DN-diff2}, one concludes that
\begin{multline}
\label{multiple-D-action}
\left(
\frac{D_{J_1}-D_{J_1-1}}{q^{J_1-1}}
\right)
\cdots
\left(
\frac{D_{(J_{N-k})}-D_{(J_{N-k})-1}}{q^{(J_{N-k})-1}}
\right)
Q_{\kappa}(x_1,\dots,x_N)
\\
=
\sum_{\bm{\nu}}
\left( \bm{1}_{\zeta(\kappa,\bm{\nu}) \supset \mathcal{J}} \right)
\cdot
\prod_{i \in \mathcal{J}} q^{-\ell\left(\nu^{(i)}\right)}
\cdot
Q_{\kappa/\nu^{(N-1)}}(x_N)
\cdot
\prod_{j=1}^{N-1}
Q_{\nu^{(j)}/\nu^{(j-1)}}(x_j),
\end{multline}
where $\nu^{(N)} = \kappa$ by agreement, and with the sum taken over all length-$(N-1)$ Gelfand--Tsetlin patterns $\bm{\nu} = \nu^{(N-1)} \succ \cdots \succ \nu^{(1)} \succ \nu^{(0)} \equiv \varnothing$. Note that in the extremal case $\ell(\kappa) = k$, the inclusion of $\mathcal{J}$ in $\zeta(\kappa,\bm{\nu})$ is only possible if $\zeta(\kappa,\bm{\nu}) \equiv \mathcal{J}$. Then due to the constraints imposed by the indicator function $\bm{1}_{\zeta(\kappa,\bm{\nu}) = \mathcal{J}}$, each of the lengths $\ell\left(\nu^{(i)}\right)$ is known explicitly; in particular, one has 
\begin{align*}
\ell\left(\nu^{(J_j)}\right) = J_j-j,\qquad \forall\ 1 \leq j \leq N-k.
\end{align*}
The associated powers of $q$ can be factored outside of the sum, and after some cancellation we obtain
\begin{multline}
\label{multiple-D-action2}
\left(
D_{J_1}-D_{J_1-1}
\right)
\cdots
\left(
D_{(J_{N-k})}-D_{(J_{N-k})-1}
\right)
Q_{\kappa}(x_1,\dots,x_N)
\\
=
q^{\frac{(N-k)(N-k-1)}{2}}
\sum_{\bm{\nu}}
\left( \bm{1}_{\zeta(\kappa,\bm{\nu}) = \mathcal{J}} \right)
\cdot
Q_{\kappa/\nu^{(N-1)}}(x_N)
\cdot
\prod_{j=1}^{N-1}
Q_{\nu^{(j)}/\nu^{(j-1)}}(x_j),
\qquad
\ell(\kappa) = k.
\end{multline}
We can now conclude our alternative proof of Theorem \ref{thm:2=1}:

\begin{cor}
The match \eqref{simpler-match} implies \eqref{2=1}.
\end{cor}

\begin{proof}
Comparing \eqref{P-6v-sym} and \eqref{P-cHL-sym}, the equality \eqref{simpler-match} implies that
\begin{multline}
\label{10.16}
\sum_{\kappa}
\sum_{\bm{\nu}}
\bm{1}_{\zeta(\kappa,\bm{\nu}) \supset \b{\mathcal{I}}}
\cdot
\prod_{i=1}^{M-1}
P_{\nu^{(i)} / \nu^{(i-1)}}(y_i)
\cdot
P_{\kappa / \nu^{(M-1)}}(y_M)
\cdot
Q_{\kappa}(x_1,\dots,x_N)
\\
=
\sum_{\kappa}
\sum_{\mu:\mu^{+}=\kappa}
E_{\tilde{\mu}}(y_M,\dots,y_1)
Q_{\kappa}(x_1,\dots,x_N)
\cdot
\bm{1}_{z(\mu) \supset \b{\mathcal{I}}},
\end{multline}
for any $\mathcal{I} = \{I_1,\dots,I_k\} \subset \{1,\dots,M\}$, $\b{\mathcal{I}} = \{1,\dots,M\} \backslash \mathcal{I}$. On either side of this equation, $\kappa$ is being summed over all partitions, but the indicator functions $\bm{1}_{\zeta(\kappa,\bm{\nu}) \supset \b{\mathcal{I}}}$ and $\bm{1}_{z(\mu) \supset \b{\mathcal{I}}}$ readily imply that nonzero terms only come from $\kappa$ with $\ell(\kappa) \leq k$. If we now act on both sides of \eqref{10.16} with $\left(D_{J_1}-D_{J_1-1}\right) \cdots \left( D_{(J_{N-k})}-D_{(J_{N-k})-1} \right)$, perform the relabelling $(x_1,\dots,x_N) \mapsto (x_N,\dots,x_1)$, and restore the reciprocal of the Cauchy kernel to the equation, we ultimately find that
\begin{align}
\label{triang-sys}
\sum_{\substack{
\mathcal{I}' \subset \mathcal{I},\mathcal{J}' \supset \mathcal{J}
\\ \vspace{-0.1cm} \\
|\mathcal{I}'| + |\mathcal{J}'| = N
}}
\prod_{i=1}^{N-k}
q^{-\#\{j: j < J_i,j \not\in \mathcal{J}'\}}
\cdot
\mathbb{P}_{\rm 6v}(\mathcal{I}',\mathcal{J}')
=
\sum_{\substack{
\mathcal{I}' \subset \mathcal{I},\mathcal{J}' \supset \mathcal{J}
\\ \vspace{-0.1cm} \\
|\mathcal{I}'| + |\mathcal{J}'| = N
}}
\prod_{i=1}^{N-k}
q^{-\#\{j: j < J_i,j \not\in \mathcal{J}'\}}
\cdot
\mathbb{P}_{\rm cHL}(\mathcal{I}',\mathcal{J}'),
\end{align}
where we have used \cite[Theorem 5.5]{BorodinBW} and the definition of $\mathbb{P}_{\rm 6v}(\mathcal{I}',\mathcal{J}')$  for the left hand side, and the definition \eqref{PcHL} of $\mathbb{P}_{\rm cHL}(\mathcal{I}',\mathcal{J}')$ for the right hand side. The latter (triangular) system of equations can be diagonalized to yield \eqref{2=1}. Note that the powers of $q$ match on both sides of \eqref{triang-sys}, but their exact values are otherwise irrelevant in reaching this conclusion.
\end{proof}

\section{A match between Cherednik--Dunkl and Macdonald difference operators}

One can also think of the subset $\mathcal{I}$ as arising from the action of suitable operators, and doing so leads to an interesting relation between the Macdonald difference and Cherednik--Dunkl operators. To see this, let us extract the coefficient of the Hall--Littlewood polynomial $Q_{\kappa}(x_1,\dots,x_N)$ from both sides of equation \eqref{10.16}, since the latter comprise a basis for the ring of symmetric functions in $(x_1,\dots,x_N)$. After some cosmetic relabelling we read the identity
\begin{align}
\label{hl-identity1}
\sum_{\bm{\nu}}
\bm{1}_{\zeta(\kappa,\bm{\nu}) \supset \theta}
\cdot
\prod_{i=1}^{M-1}
P_{\nu^{(i)} / \nu^{(i-1)}}(\tilde{y}_i)
\cdot
P_{\kappa / \nu^{(M-1)}}(\tilde{y}_M)
=
\sum_{\mu:\mu^{+}=\kappa}
E_{\mu}(y_1,\dots,y_M)
\cdot
\bm{1}_{z(\mu) \supset \tilde{\theta}},
\end{align}
for any partition $\kappa$ such that $\ell(\kappa) \leq M-k$, any set of pairwise distinct integers $\theta = \{\theta_1,\dots,\theta_{k}\}$, with 
$\tilde{\theta} = \{M-\theta_{1}+1,\dots,M-\theta_{k}+1\}$, and 
$\tilde{y}_i = y_{M-i+1}$ for all $1 \leq i \leq M$. In an analogous way to the diagonalization of the system \eqref{triang-sys}, one can rearrange the set of equations \eqref{hl-identity1} to deduce that
\begin{align}
\label{hl-identity2}
\sum_{\bm{\nu}}
\bm{1}_{\zeta(\kappa,\bm{\nu}) = \theta}
\cdot
\prod_{i=1}^{M-1}
P_{\nu^{(i)} / \nu^{(i-1)}}(\tilde{y}_i)
\cdot
P_{\kappa / \nu^{(M-1)}}(\tilde{y}_M)
=
\sum_{\mu:\mu^{+}=\kappa}
E_{\mu}(y_1,\dots,y_M)
\cdot
\bm{1}_{z(\mu) = \tilde{\theta}},
\end{align}
for any partition $\kappa$ such that $\ell(\kappa) = M-k$, any set of pairwise distinct integers $\theta = \{\theta_1,\dots,\theta_{k}\}$, with $\tilde{\theta} = \{M-\theta_{1}+1,\dots,M-\theta_k+1\}$, and any integer $1 \leq k \leq M$. The difference, compared with \eqref{hl-identity1}, is that the indicator functions in \eqref{hl-identity2} now fully determine $\zeta(\kappa,\bm{\nu})$ and $z(\mu)$.

These identities motivate the following result:
\begin{prop}
Define another version of the difference operators \eqref{D1N}, acting on subsets of the alphabet $(\tilde{y}_1,\dots,\tilde{y}_M) = (y_M,\dots,y_1)$ (with the parameter $p$ set to zero):
\begin{align*}
\tilde{D}_a
\index{D@$\tilde{D}_a$; reversed-alphabet Macdonald difference operators}
=
\sum_{i=1}^{a}
\Big(
\prod_{\substack{j\not=i \\ j=1}}^{a}
\frac{q \tilde{y}_i - \tilde{y}_j}{\tilde{y}_i - \tilde{y}_j}
\Big)
\tilde{\tau}_i,
\qquad
1 \leq a \leq M,
\end{align*}
where $\tilde{\tau}_i \cdot h(\tilde{y}_1,\dots,\tilde{y}_M) = h(\tilde{y}_1,\dots,\tilde{y}_{i-1},0,\tilde{y}_{i+1},\dots,\tilde{y}_M)$. Let us also recall the definition of the Cherednik--Dunkl operators \eqref{Y-tilde}, with $p=0$. Then for any set $\theta = \{1 \leq \theta_1 < \cdots < \theta_k \leq M\}$ one has the equality
\begin{align}
\label{mac-CD}
\left( \tilde{D}_{\theta_1} - \tilde{D}_{\theta_1-1} \right)
\cdots
\left( \tilde{D}_{\theta_{k}} - \tilde{D}_{\theta_{k}-1} \right)
=
\tilde{Y}_{\tilde{\theta}_1}
\dots
\tilde{Y}_{\tilde{\theta}_{k}},
\qquad
\tilde{\theta}_i = M-\theta_{i}+1,
\end{align}
as an operator identity on symmetric functions in $(y_1,\dots,y_M)$.
\end{prop}

\begin{proof}
One checks \eqref{mac-CD} by acting on an arbitrary Hall--Littlewood polynomial $P_{\kappa}(y_1,\dots,y_M) = \sum_{\mu: \mu^{+} = \kappa} E_{\mu}(y_1,\dots,y_M)$. The action of the left hand side can be obtained as follows 
(\cf\ \eqref{multiple-D-action}):
\begin{multline}
\label{lhs-action}
\left( \tilde{D}_{\theta_1} - \tilde{D}_{\theta_1-1} \right)
\cdots
\left( \tilde{D}_{\theta_{k}} - \tilde{D}_{\theta_{k}-1} \right)
P_{\kappa}(y_1,\dots,y_M)
\\
=
\prod_{i=1}^{k} q^{\theta_i-1}
\cdot
\sum_{\bm{\nu}}
\left( \bm{1}_{\zeta(\kappa,\bm{\nu}) \supset \theta} \right)
\cdot
\prod_{i \in \theta} q^{-\ell\left(\nu^{(i)}\right)}
\cdot
P_{\kappa/\nu^{(M-1)}}(\tilde{y}_M)
\cdot
\prod_{j=1}^{M-1}
P_{\nu^{(j)}/\nu^{(j-1)}}(\tilde{y}_j),
\\
=
\prod_{i=1}^{k} q^{\theta_i-1}
\cdot
\sum_{\phi \supset \theta}
\prod_{i=1}^{k} q^{-\#\{j: j < \theta_i, j\not\in \phi \}}
\cdot
\sum_{\bm{\nu}}
\left( \bm{1}_{\zeta(\kappa,\bm{\nu}) = \phi} \right)
\cdot
P_{\kappa/\nu^{(M-1)}}(\tilde{y}_M)
\cdot
\prod_{j=1}^{M-1}
P_{\nu^{(j)}/\nu^{(j-1)}}(\tilde{y}_j),
\end{multline}
where the final expression is summed over all sets $\phi$ which contain $\theta$ as a subset. 

For the action of the right hand side of \eqref{mac-CD}, one uses the de-symmetrized form of the Hall--Littlewood polynomial to calculate
\begin{multline}
\label{rhs-action}
\tilde{Y}_{\tilde{\theta}_1}
\dots
\tilde{Y}_{\tilde{\theta}_{k}}
\sum_{\mu: \mu^{+} = \kappa}
E_{\mu}(y_1,\dots,y_M)
=
\sum_{\mu: \mu^{+} = \kappa}
\bm{1}_{z(\mu) \supset \tilde{\theta}}
\cdot
\prod_{i=1}^{k}
q^{\#\{j: j > \tilde{\theta}_i, j \in z(\mu) \}}
\cdot
E_{\mu}(y_1,\dots,y_M),
\\
=
\sum_{\phi \supset \theta}
\prod_{i=1}^{k}
q^{\#\{j: j > \tilde{\theta}_i, j \in \tilde{\phi} \}}
\cdot
\sum_{\mu: \mu^{+} = \kappa}
\bm{1}_{z(\mu) = \tilde{\phi}}
\cdot
E_{\mu}(y_1,\dots,y_M),
\end{multline}
where $\tilde{\theta}_i = M-\theta_i+1$, $\tilde{\phi}_i = M-\phi_i+1$ as usual, and the final sum is over all sets $\phi$ which contain $\theta$ as a subset. Finally, one checks that 
\begin{align*}
\#\{j: j<\theta_i, j \not\in\phi\} + \#\{j: j>\tilde{\theta}_i, j\in\tilde{\phi}\} = 
\#\{j: j<\theta_i\} = \theta_i-1,
\end{align*}
which allows powers of $q$ to be matched between \eqref{lhs-action} and \eqref{rhs-action}, with their equality then being a direct consequence of \eqref{hl-identity2}.
\end{proof}

\chapter{Reduction to one-dimensional systems of particles}
\label{sec:reduction}

We start this chapter by considering the effect of taking principal specializations of the set of variables $(y_1,\dots,y_M)$ appearing in Theorem \ref{thm:2=3}; this lifts the result \eqref{2=3} to a statement in the stochastic version of the fused model \eqref{s-weights}, with inhomogeneous spins and vertical rapidities (\ie\ a different spin parameter $s_b$ and rapidity $y_b$ associated to each column of the lattice). We subsequently degenerate our matching result to three different interacting particle systems.

\section{Higher spin vertex model with inhomogeneous spins}
\label{ssec:inhom-model}

Throughout the next sections we shall work with the following vertex model: 
\begin{align}
\label{inhom-weights}
\begin{tabular}{|c|c|c|}
\hline
\quad
\tikz{0.7}{
\draw[lgray,line width=1.5pt,->] (-1,0) -- (1,0);
\draw[lgray,line width=4pt,->] (0,-1) -- (0,1);
\node[left] at (-1,0) {\tiny $0$};\node[right] at (1,0) {\tiny $0$};
\node[below] at (0,-1) {\tiny $\AA$};\node[above] at (0,1) {\tiny $\AA$};
}
\quad
&
\quad
\tikz{0.7}{
\draw[lgray,line width=1.5pt,->] (-1,0) -- (1,0);
\draw[lgray,line width=4pt,->] (0,-1) -- (0,1);
\node[left] at (-1,0) {\tiny $i$};\node[right] at (1,0) {\tiny $i$};
\node[below] at (0,-1) {\tiny $\AA$};\node[above] at (0,1) {\tiny $\AA$};
}
\quad
&
\quad
\tikz{0.7}{
\draw[lgray,line width=1.5pt,->] (-1,0) -- (1,0);
\draw[lgray,line width=4pt,->] (0,-1) -- (0,1);
\node[left] at (-1,0) {\tiny $0$};\node[right] at (1,0) {\tiny $i$};
\node[below] at (0,-1) {\tiny $\AA$};\node[above] at (0,1) {\tiny $\AA^{-}_i$};
}
\quad
\\[1.3cm]
\quad
$\dfrac{1-s_b x_a y_b q^{\As{1}{n}}}{1-s_b x_a y_b}$
\quad
& 
\quad
$\dfrac{s_b (s_b q^{A_i}-x_a y_b) q^{\As{i+1}{n}}}{1-s_b x_a y_b}$
\quad
& 
\quad
$\dfrac{s_b x_a y_b (q^{A_i}-1) q^{\As{i+1}{n}}}{1-s_b x_a y_b}$
\quad
\\[0.7cm]
\hline
\quad
\tikz{0.7}{
\draw[lgray,line width=1.5pt,->] (-1,0) -- (1,0);
\draw[lgray,line width=4pt,->] (0,-1) -- (0,1);
\node[left] at (-1,0) {\tiny $i$};\node[right] at (1,0) {\tiny $0$};
\node[below] at (0,-1) {\tiny $\AA$};\node[above] at (0,1) {\tiny $\AA^{+}_i$};
}
\quad
&
\quad
\tikz{0.7}{
\draw[lgray,line width=1.5pt,->] (-1,0) -- (1,0);
\draw[lgray,line width=4pt,->] (0,-1) -- (0,1);
\node[left] at (-1,0) {\tiny $i$};\node[right] at (1,0) {\tiny $j$};
\node[below] at (0,-1) {\tiny $\AA$};\node[above] at (0,1) 
{\tiny $\AA^{+-}_{ij}$};
}
\quad
&
\quad
\tikz{0.7}{
\draw[lgray,line width=1.5pt,->] (-1,0) -- (1,0);
\draw[lgray,line width=4pt,->] (0,-1) -- (0,1);
\node[left] at (-1,0) {\tiny $j$};\node[right] at (1,0) {\tiny $i$};
\node[below] at (0,-1) {\tiny $\AA$};\node[above] at (0,1) {\tiny $\AA^{+-}_{ji}$};
}
\quad
\\[1.3cm] 
\quad
$\dfrac{1-s_b^2 q^{\As{1}{n}}}{1-s_b x_a y_b}$
\quad
& 
\quad
$\dfrac{s_b x_a y_b(q^{A_j}-1) q^{\As{j+1}{n}}}{1-s_b x_a y_b}$
\quad
&
\quad
$\dfrac{s_b^2 (q^{A_i}-1)q^{\As{i+1}{n}}}{1-s_b x_a y_b}$
\quad
\\[0.7cm]
\hline
\end{tabular} 
\end{align}
valid for $1 \leq i < j \leq n$. The weights shown above apply to the vertex at the intersection of the $a$-th horizontal and $b$-th vertical lines. The index $a$ will take values $1 \leq a \leq N$: depending on the partition function under consideration, it will either increase or decrease as horizontal lines are read from bottom to top of the lattice. The index $b$ will take values $0 \leq b \leq M$ (it is more convenient to start the labelling at $0$ rather than $1$). These are the stochastic $\tilde{L}$-weights of Proposition \ref{prop:stoch}, with explicitly shown dependence of parameters on rows and columns.

\section{Higher spin partition functions $\mathbb{Z}_{M,N}(\mathcal{I},\mathcal{J})$ and $\mathbb{X}_{M,N}(\mathcal{I},\mathcal{J})$}
\label{ssec:inhom-model2}

\index{Z4@$\mathbb{Z}_{M,N}(\mathcal{I},\mathcal{J})$}
\index{X4@$\mathbb{X}_{M,N}(\mathcal{I},\mathcal{J})$}

Let us now define fused analogues of the previous quantities $Z_{M,N}(\mathcal{I},\mathcal{J})$ and $X_{M,N}(\mathcal{I},\mathcal{J})$. In what follows $\mathcal{J}$ will continue to denote a set $\{1 \leq J_1 < \cdots < J_{\ell} \leq N\}$ of increasing positive integers, but $\mathcal{I}$ will now take the form $\{0 \leq I_1 \leq \cdots \leq I_k \leq M\}$; that is, we relax the constraint that the parts of $\mathcal{I}$ be pairwise distinct, and allow them to assume the value $0$.

The fused analogue of $Z_{M,N}(\mathcal{I},\mathcal{J})$ will be denoted $\mathbb{Z}_{M,N}(\mathcal{I},\mathcal{J})$. It is defined as the partition function of the model \eqref{inhom-weights} at $n=1$ (colour-blind case) inside the $M \times N$ rectangle, where the $a$-th horizontal line (counted from the {\it top}) has rapidity $x_a$, while the $b$-th vertical line (counted from the left) has spin parameter $s_b$ and rapidity $y_b$. The boundary conditions are chosen as follows: {\bf 1.} There is an incoming horizontal path at the edge $(-1,j) \rightarrow (0,j)$ for all $1 \leq j \leq N$; {\bf 2.} The edge $(i,0) \rightarrow (i,1)$ is devoid of a path, for all $0 \leq i \leq M$; {\bf 3.} There is an outgoing horizontal path at the edge $(M,N-J_a+1) \rightarrow (M+1,N-J_a+1)$ for all $1 \leq a \leq \ell$; {\bf 4.} There is an outgoing vertical path at the edge $(I_b,N) \rightarrow (I_b,N+1)$ for all $1 \leq b \leq k$, which may lead to multiple outgoing paths at any given vertical edge. See the left panel of Figure \ref{fig:fus-6v-quad} for an illustration of $\mathbb{Z}_{M,N}(\mathcal{I},\mathcal{J})$.

\begin{figure}
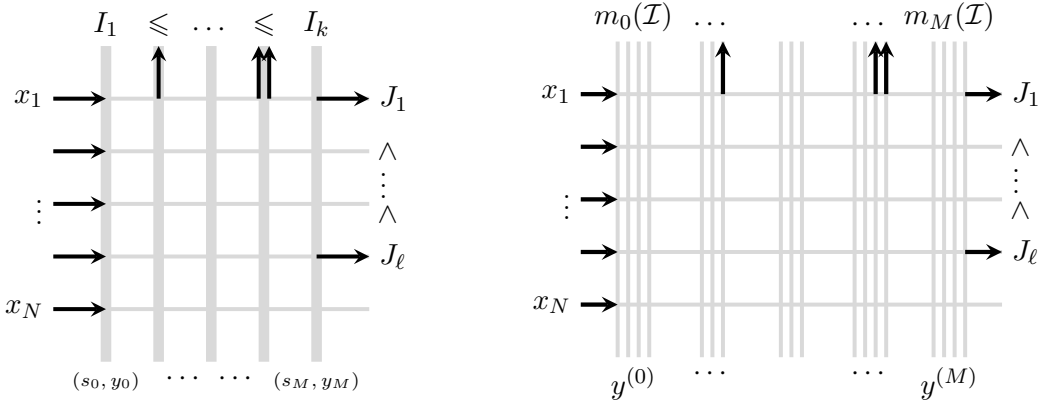

\tikz{0.7}{
\foreach\x in {1,...,5}{
\draw[lgray,line width=4pt] (\x,0) -- (\x,6);
}
\foreach\y in {1,...,5}{
\draw[lgray,line width=1.5pt] (0,\y) -- (6,\y);
\draw[line width=1.5pt,->] (0,\y) -- (1,\y);
}
\foreach\x in {2,3.9,4.1}{
\draw[line width=1.5pt,->] (\x,5) -- (\x,6);
}
\foreach\y in {2,5}{
\draw[line width=1.5pt,->] (5,\y) -- (6,\y);
}
\node[above] at (5,6) {$I_k$};
\node[above] at (4,6) {$\leq$};
\node[above] at (3,6) {$\cdots$};
\node[above] at (2,6) {$\leq$};
\node[above] at (1,6) {$I_1$};
\node[below] at (5,0) {\tiny $(s_M,y_M)$};
\node[below] at (2.5,0) {$\cdots$};
\node[below] at (3.5,0) {$\cdots$};
\node[below] at (1,0) {\tiny $(s_0,y_0)$};
\node[left] at (0,1) {$x_N$};
\node[left] at (0,3) {$\vdots$};
\node[left] at (0,5) {$x_1$};
\node[right] at (6,2) {$J_{\ell}$};
\node[right] at (6,2.8) {$\rotgr$};
\node[right] at (6,3.5) {$\vdots$};
\node[right] at (6,4) {$\rotgr$};
\node[right] at (6,5) {$J_1$};
}
\qquad\quad
\tikz{0.7}{
\foreach\y in {1,...,5}{
\draw[lgray,line width=1.5pt] (0,\y) -- (8,\y);
}
\foreach\x in {0.7,0.9,1.1,1.3}{
\draw[lgray,line width=1.5pt] (\x,0) -- (\x,6);
}
\foreach\x in {2.3,2.5,2.7}{
\draw[lgray,line width=1.5pt] (\x,0) -- (\x,6);
}
\foreach\x in {3.8,4,4.2}{
\draw[lgray,line width=1.5pt] (\x,0) -- (\x,6);
}
\foreach\x in {5.2,5.4,5.6,5.8}{
\draw[lgray,line width=1.5pt] (\x,0) -- (\x,6);
}
\foreach\x in {6.7,6.9,7.1,7.3}{
\draw[lgray,line width=1.5pt] (\x,0) -- (\x,6);
}
\node[left] at (0,1) {$x_N$};
\node[left] at (0,3) {$\vdots$};
\node[left] at (0,5) {$x_1$};
\node[below] at (1,0) {$y^{(0)}$};
\node[below] at (2.5,0) {$\cdots$};
\node[below] at (5.5,0) {$\cdots$};
\node[below] at (7,0) {$y^{(M)}$};
\node[above] at (1,6) {$m_0(\mathcal{I})$};
\node[above] at (2.5,6) {$\cdots$};
\node[above] at (5.5,6) {$\cdots$};
\node[above] at (7,6) {$m_M(\mathcal{I})$};
\node[right] at (8,2) {$J_{\ell}$};
\node[right] at (8,2.8) {$\rotgr$};
\node[right] at (8,3.5) {$\vdots$};
\node[right] at (8,4) {$\rotgr$};
\node[right] at (8,5) {$J_1$};
\draw[ultra thick,->] (0,5) -- (0.7,5);
\draw[ultra thick,->] (0,4) -- (0.7,4);
\draw[ultra thick,->] (0,3) -- (0.7,3);
\draw[ultra thick,->] (0,2) -- (0.7,2); 
\draw[ultra thick,->] (0,1) -- (0.7,1);
\draw[ultra thick,->] (5.8,5) -- (5.8,6);
\draw[ultra thick,->] (5.6,5) -- (5.6,6);
\draw[ultra thick,->] (2.7,5) -- (2.7,6);
\draw[ultra thick,->] (7.3,5) -- (8,5);
\draw[ultra thick,->] (7.3,2) -- (8,2);
}
\caption{Left panel: lattice definition of $\mathbb{Z}_{M,N}(\mathcal{I},\mathcal{J})$. The set $\mathcal{I}$ labels the coordinates of $k$ paths leaving via the top boundary, while $\mathcal{J}$ gives the coordinates of the remaining $\ell$ paths leaving via the right boundary. Right panel: the pre-fused version of $\mathbb{Z}_{M,N}(\mathcal{I},\mathcal{J})$, $\mathcal{Z}_{M,N}(\mathcal{I},\mathcal{J})$. For each $0 \leq i \leq M$, a total of $m_i(\mathcal{I})$ paths leave via the top of the $i$-th bundle, where $m_i(\mathcal{I}) = \#\{a: I_a = i\}$, and they are positioned as far to the right as possible within the bundle.}
\label{fig:fus-6v-quad}
\end{figure}

The fused analogue of $X_{M,N}(\mathcal{I},\mathcal{J})$ will be denoted $\mathbb{X}_{M,N}(\mathcal{I},\mathcal{J})$. It is defined as the partition function of the model \eqref{inhom-weights} at $n=N$ (rainbow case) inside the $M \times N$ rectangle, where the $a$-th horizontal line (counted from the {\it bottom}) has rapidity $x_a$, while the $b$-th vertical line (counted from the left) has spin parameter $s_b$ and rapidity $y_b$. The boundary conditions are chosen as follows: {\bf 1.} There is an incoming horizontal path of colour $j$ at the edge $(-1,j) \rightarrow (0,j)$ for all $1 \leq j \leq N$; {\bf 2.} The edge $(i,0) \rightarrow (i,1)$ is devoid of a path, for all $0 \leq i \leq M$; {\bf 3.} The collection of edges $(M,j) \rightarrow (M+1,j)$, $1 \leq j \leq N$ feature a total of $\ell$ outgoing horizontal paths, of colours $\{J_1,\dots,J_{\ell}\}$. These edges are summed over all choices which obey this criterion; {\bf 4.} There is an outgoing vertical path, of unspecified colour, at the edge $(I_b,N) \rightarrow (I_b,N+1)$ for all $1 \leq b \leq k$. These edges are also summed over all choices which obey this criterion. See the left panel of Figure \ref{fig:fus-col-quad} for an illustration of $\mathbb{X}_{M,N}(\mathcal{I},\mathcal{J})$.
\begin{figure}
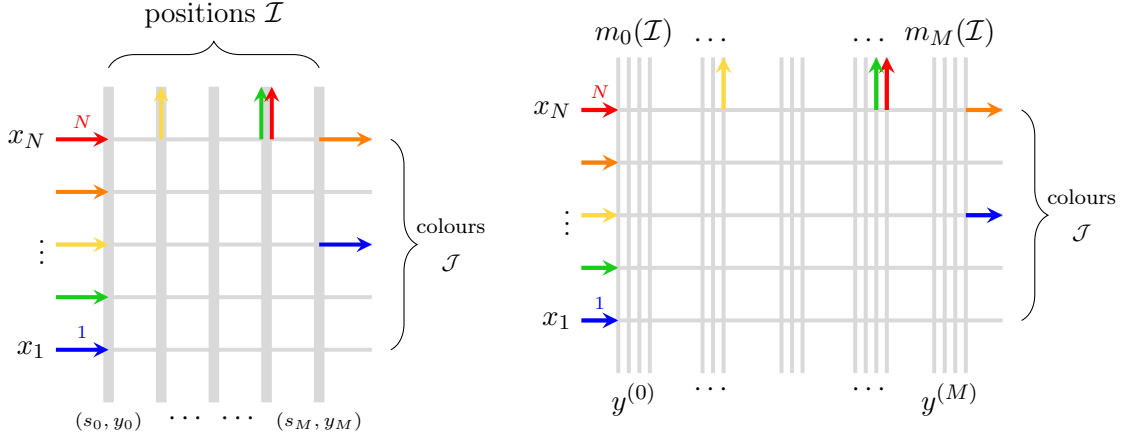

\tikz{0.7}{
\foreach\x in {1,...,5}{
\draw[lgray,line width=4pt] (\x,0) -- (\x,6);
}
\foreach\y in {1,...,5}{
\draw[lgray,line width=1.5pt] (0,\y) -- (6,\y);
}
\draw[line width=1.5pt,red,->] (0,5) -- (1,5) node[midway,above] {\tiny $N$};
\draw[line width=1.5pt,orange,->] (0,4) -- (1,4);
\draw[line width=1.5pt,yellow,->] (0,3) -- (1,3);
\draw[line width=1.5pt,green,->] (0,2) -- (1,2);
\draw[line width=1.5pt,blue,->] (0,1) -- (1,1) node[midway,above] {\tiny $1$};
\draw[line width=1.5pt,yellow,->] (2,5) -- (2,6);
\draw[line width=1.5pt,green,->] (3.9,5) -- (3.9,6);
\draw[line width=1.5pt,red,->] (4.1,5) -- (4.1,6);
\draw[line width=1.5pt,blue,->] (5,3) -- (6,3);
\draw[line width=1.5pt,orange,->] (5,5) -- (6,5);
\draw [decorate,decoration={brace,amplitude=10pt},xshift=0pt,yshift=-4pt] (1,6.5) -- (5,6.5) 
node [black,midway,yshift=0.7cm] {positions\ $\mathcal{I}$};
\node[below] at (5,0) {\tiny $(s_M,y_M)$};
\node[below] at (2.5,0) {$\cdots$};
\node[below] at (3.5,0) {$\cdots$};
\node[below] at (1,0) {\tiny $(s_0,y_0)$};
\node[left] at (0,1) {$x_1$};
\node[left] at (0,3) {$\vdots$};
\node[left] at (0,5) {$x_N$};
\draw [decorate,decoration={brace,amplitude=10pt},xshift=-4pt,yshift=0pt] (6.5,5) -- (6.5,1) 
node [black,midway,xshift=0.8cm] {$\substack{\text{colours}\\ \\ \mathcal{J}}$};
}
\
\tikz{0.7}{
\foreach\y in {1,...,5}{
\draw[lgray,line width=1.5pt] (0,\y) -- (8,\y);
}
\foreach\x in {0.7,0.9,1.1,1.3}{
\draw[lgray,line width=1.5pt] (\x,0) -- (\x,6);
}
\foreach\x in {2.3,2.5,2.7}{
\draw[lgray,line width=1.5pt] (\x,0) -- (\x,6);
}
\foreach\x in {3.8,4,4.2}{
\draw[lgray,line width=1.5pt] (\x,0) -- (\x,6);
}
\foreach\x in {5.2,5.4,5.6,5.8}{
\draw[lgray,line width=1.5pt] (\x,0) -- (\x,6);
}
\foreach\x in {6.7,6.9,7.1,7.3}{
\draw[lgray,line width=1.5pt] (\x,0) -- (\x,6);
}
\node[left] at (0,1) {$x_1$};
\node[left] at (0,3) {$\vdots$};
\node[left] at (0,5) {$x_N$};
\node[below] at (1,0) {$y^{(0)}$};
\node[below] at (2.5,0) {$\cdots$};
\node[below] at (5.5,0) {$\cdots$};
\node[below] at (7,0) {$y^{(M)}$};
\node[above] at (1,6) {$m_0(\mathcal{I})$};
\node[above] at (2.5,6) {$\cdots$};
\node[above] at (5.5,6) {$\cdots$};
\node[above] at (7,6) {$m_M(\mathcal{I})$};
\draw [decorate,decoration={brace,amplitude=10pt},xshift=-4pt,yshift=0pt] (8.5,5) -- (8.5,1) 
node [black,midway,xshift=0.8cm] {$\substack{\text{colours}\\ \\ \mathcal{J}}$};
\draw[ultra thick,->,red] (0,5) -- (0.7,5) node[midway,above] {\tiny $N$};
\draw[ultra thick,->,orange] (0,4) -- (0.7,4);
\draw[ultra thick,->,yellow] (0,3) -- (0.7,3);
\draw[ultra thick,->,green] (0,2) -- (0.7,2); 
\draw[ultra thick,->,blue] (0,1) -- (0.7,1) node[midway,above] {\tiny $1$};
\draw[ultra thick,->,red] (5.8,5) -- (5.8,6);
\draw[ultra thick,->,green] (5.6,5) -- (5.6,6);
\draw[ultra thick,->,yellow] (2.7,5) -- (2.7,6);
\draw[ultra thick,->,orange] (7.3,5) -- (8,5);
\draw[ultra thick,->,blue] (7.3,3) -- (8,3);
}
\caption{Left panel: lattice definition of $\mathbb{X}_{M,N}(\mathcal{I},\mathcal{J})$. The set $\mathcal{I}$ labels the coordinates of $k$ paths leaving via the top boundary, while $\mathcal{J}$ gives the colours of the remaining $\ell$ paths which exit via the right boundary. Right panel: the pre-fused version of $\mathbb{X}_{M,N}(\mathcal{I},\mathcal{J})$, $\mathcal{X}_{M,N}(\mathcal{I},\mathcal{J})$. For each $0 \leq i \leq M$, a total of $m_i(\mathcal{I})$ paths leave via the top of the $i$-th bundle, and they are positioned as far to the right as possible within the bundle. The top boundary is summed over all permutations of the exiting colours, holding the outgoing coordinates fixed; the right boundary is summed over all positions where the colours $\mathcal{J}$ exit.}
\label{fig:fus-col-quad}
\end{figure}

\begin{thm}\label{thm:fused2=3}
Fix two integers $M,N \geq 1$ and two sets $\mathcal{I} = \{0 \leq I_1 \leq \cdots \leq I_k \leq M\}$ and $\mathcal{J} = \{1 \leq J_1 < \cdots < J_{\ell} \leq N\}$ whose cardinalities satisfy $k+\ell=N$. We have the following equality of partition functions:
\begin{align}
\label{fused-pf=}
\mathbb{Z}_{M,N}(\mathcal{I},\mathcal{J})
=
\mathbb{X}_{M,N}(\mathcal{I},\mathcal{J}).
\end{align}
\end{thm}

\begin{proof}
The idea of the proof is to establish a pre-fused version of the equality \eqref{fused-pf=}, which we are able to do by virtue of Theorem \ref{thm:2=3}, and then to apply principal specializations/analytic continuation to map to the desired statement in the fused model \eqref{inhom-weights}.

Let us consider the partition function as shown in the right panel of Figure \ref{fig:fus-6v-quad}, which we denote $\mathcal{Z}_{M,N}(\mathcal{I},\mathcal{J})$. Each vertex in the lattice is of unfused type, and assumes one of the six possible forms \eqref{six-vert}. The boundary conditions assigned to external horizontal edges match completely with the corresponding boundaries in $\mathbb{Z}_{M,N}(\mathcal{I},\mathcal{J})$. Where the two partition functions differ is with respect to the vertical lines; in $\mathcal{Z}_{M,N}(\mathcal{I},\mathcal{J})$ we group these lines into bundles with respective cardinalities $(K_0,K_1,\dots,K_M)$, such that the $i$-th line within the $j$-th bundle (as counted from the left) carries vertical rapidity $y^{(j)}_i$. The bottom external edges of the $j$-th bundle are devoid of paths, while a total of $m_j(\mathcal{I})$ paths exit via the top external edges of this bundle, where $m_j(\mathcal{I}) = \#\{a: I_a = j\}$. The paths that exit from the top of each bundle are positioned as far to the right as possible within the bundle.

We shall also study the coloured pre-fused partition function shown in the right panel of Figure \ref{fig:fus-col-quad}, which we denote $\mathcal{X}_{M,N}(\mathcal{I},\mathcal{J})$. Its vertices are also of unfused type, and assume the possible forms \eqref{col-vert}. As we did above, we assign boundary conditions to external horizontal edges which match with the corresponding boundaries in $\mathbb{X}_{M,N}(\mathcal{I},\mathcal{J})$; namely, there is a path of colour $i$ entering via the $i$-th left external edge (counted from the bottom), and the right external edges of $\mathcal{X}_{M,N}(\mathcal{I},\mathcal{J})$ are summed over all ways of positioning the colours $\mathcal{J}$ over the $N$ available sites. The vertical lines of $\mathbb{X}_{M,N}(\mathcal{I},\mathcal{J})$ are again grouped into bundles with respective cardinalities $(K_0,K_1,\dots,K_M)$, such that the $i$-th line within the $j$-th bundle has rapidity $y^{(j)}_i$. Bottom external edges of the $j$-th bundle are devoid of paths, while a total of $m_j(\mathcal{I})$ paths exit via the top external edges of this bundle, where $m_j(\mathcal{I})$ has the same definition as above. We continue to enforce the constraint that paths leave each bundle via the rightmost external edges, but with an extra degree of freedom in the case of $\mathcal{X}_{M,N}(\mathcal{I},\mathcal{J})$: namely, having fixed the positions where paths must leave the lattice, we sum over all ways of distributing the colours $\{1,\dots,N\}\backslash\mathcal{J}$ over those edges.

Now by application of Theorem \ref{thm:2=3}, we immediately see that
\begin{align}
\label{prefused-pf=}
\mathcal{Z}_{M,N}(\mathcal{I},\mathcal{J})
=
\mathcal{X}_{M,N}(\mathcal{I},\mathcal{J}).
\end{align}
It remains to see what happens to both sides of \eqref{prefused-pf=} after taking the principal specializations $y^{(j)}_i = s_j y_j q^{K_j-i}$ and performing the analytic continuation $q^{K_j} \mapsto s_j^{-2}$. In the case of $\mathcal{Z}_{M,N}(\mathcal{I},\mathcal{J})$, we apply the equations \eqref{stack-Nrow} and \eqref{an-cont} with $n=1$ to calculate the fusion of each bundle of the lattice; the result is
\begin{align}
\nonumber
\mathcal{Z}_{M,N}(\mathcal{I},\mathcal{J})
\Big|_{y^{(j)}_i = s_j y_j q^{K_j-i}}
\Big|_{q^{K_j} \mapsto s_j^{-2}}
&=
\prod_{j=0}^{M}
\left.
\frac{(q;q)_{K_j-m_j(\mathcal{I})} (q;q)_{m_j(\mathcal{I})}}
{(q;q)_{K_j}}
\right|_{q^{K_j} \mapsto s_j^{-2}}
\cdot
\mathbb{Z}_{M,N}(\mathcal{I},\mathcal{J})
\\
\label{aab}
&=
\prod_{j=0}^{M}
\frac{(q;q)_{m_j(\mathcal{I})}}
{(s_j^{-2};q^{-1})_{m_j(\mathcal{I})}}
\cdot
\mathbb{Z}_{M,N}(\mathcal{I},\mathcal{J}).
\end{align}
The prefactor that arises above is the inverse of that on the left hand side of \eqref{stack-Nrow}. The calculation is similar for $\mathcal{X}_{M,N}(\mathcal{I},\mathcal{J})$, except that there is a sum over the colours of the paths which exit via the top of each bundle. By virtue of the $q$-exchangeability property \eqref{q-exch}, the terms in this sum differ from each other only up to powers of $q$. Applying equations \eqref{stack-Nrow} and \eqref{an-cont} to each bundle of $\mathcal{X}_{M,N}(\mathcal{I},\mathcal{J})$ (keeping $n$ generic), we find that
\begin{multline}
\label{aaa}
\mathcal{X}_{M,N}(\mathcal{I},\mathcal{J})
\Big|_{y^{(j)}_i = s_j y_j q^{K_j-i}}
\Big|_{q^{K_j} \mapsto s_j^{-2}}
\\
=
\prod_{j=0}^{M}
\left.
\left(
\sum_{\sigma \in \mathfrak{S}_{m_j(\mathcal{I})}}
q^{{\rm inv}(\sigma)}
\right)
\frac{(q;q)_{K_j-m_j(\mathcal{I})} (1-q)^{m_j(\mathcal{I})}}
{(q;q)_{K_j}}
\right|_{q^{K_j} \mapsto s_j^{-2}}
\cdot
\mathbb{X}_{M,N}(\mathcal{I},\mathcal{J}),
\end{multline}
where we use ${\rm inv}(\sigma) = \#\{a < b: \sigma(a) > \sigma(b)\}$. Now since
\begin{align}
\sum_{\sigma \in \mathfrak{S}_m}
q^{{\rm inv}(\sigma)}
=
\frac{(q;q)_m}{(1-q)^m},
\qquad
\forall\ m \geq 1,
\end{align}
the factors on the right hand side of \eqref{aaa} can be regrouped to yield
\begin{align}
\label{aac}
\mathcal{X}_{M,N}(\mathcal{I},\mathcal{J})
\Big|_{y^{(j)}_i = s_j y_j q^{K_j-i}}
\Big|_{q^{K_j} \mapsto s_j^{-2}}
=
\prod_{j=0}^{M}
\frac{(q;q)_{m_j(\mathcal{I})}}
{(s_j^{-2};q^{-1})_{m_j(\mathcal{I})}}
\cdot
\mathbb{X}_{M,N}(\mathcal{I},\mathcal{J}).
\end{align}
Combining equations \eqref{prefused-pf=}, \eqref{aab} and \eqref{aac} completes the proof of \eqref{fused-pf=}.

\end{proof}

\section{Reduction to ASEP}

We come to the first application of the matching in Theorem \ref{thm:fused2=3}; namely, its reduction to the asymmetric simple exclusion process (ASEP). In order to carry out this reduction one needs the $s_b = q^{-1/2},\ \forall\ b \geq 1$ specialization of \eqref{fused-pf=}, when the underlying model \eqref{inhom-weights} reverts to the fundamental vertex model \eqref{col-vert}. 


\subsection{Multi-species ASEP}
\label{sssec:masep}

Before progressing to these details, let us recall the definition of the multi-species ASEP. It is a system of interacting, coloured particles with the following properties:
\begin{itemize}
\item The system lives on the infinite one-dimensional integral lattice, with sites labelled by $\mathbb{Z}$.
\item Every particle in the system has a colour prescribed to it, which is a positive integer. 
\item Each site can be occupied by at most one particle. Accordingly, at any point in time $t$, a configuration of the multi-species ASEP is given by the list 
\begin{align*}
\index{ah2@$\eta^{\rm mASEP}(t)$; mASEP occupation data}
\eta^{\rm mASEP}(t) = \{\dots,\eta_{-1}(t),\eta_0(t),\eta_1(t),\dots\},
\qquad
\eta_i(t) \in \mathbb{N},
\end{align*}
where $0$ indicates an unoccupied site, and an integer $j \geq 1$ indicates a particle of colour $j$.

\item Let $\{\eta_{i-1},\eta_i,\eta_{i+1}\}$ be the occupation data for the site $i$ and its two neighbours, at some point in time. This site is assigned two exponential clocks: a left clock of rate $q \cdot {\bm 1}_{\eta_i > \eta_{i-1}}$ and a right clock of rate ${\bm 1} _{\eta_i > \eta_{i+1}}$. When the left clock rings, the occupation data gets updated as $\{\eta_{i-1},\eta_i,\eta_{i+1}\} \mapsto \{\eta_i,\eta_{i-1},\eta_{i+1}\}$ (\ie\ the particles at sites $i$ and $i-1$ exchange positions). When the right clock rings, the occupation data is updated as $\{\eta_{i-1},\eta_i,\eta_{i+1}\} \mapsto \{\eta_{i-1},\eta_{i+1},\eta_{i}\}$ (\ie\ the particles at sites $i$ and $i+1$ exchange positions). All clocks in the system are independent. See Figure \ref{fig:ASEP}.
\end{itemize}

\begin{figure}
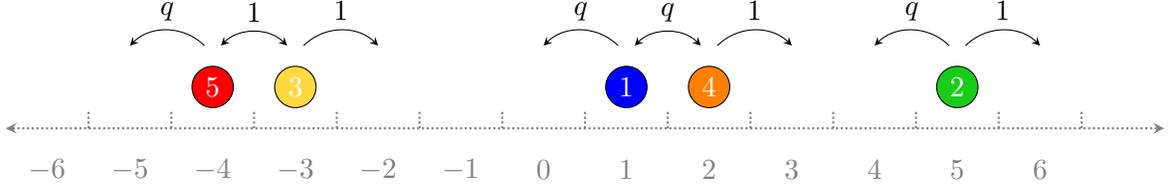

\tikz{1.1}{
\draw[gray,densely dotted,thick,<->] (-7,0) -- (7,0);
\foreach\x in {-6,...,6}{
\draw[gray,densely dotted,thick] (\x,0) -- (\x,0.2);
\node at (\x-0.5,-0.5) {\color{gray} $\x$};
}
\filldraw[fill=red] (-4.5,0.5) circle (0.25);
\filldraw[fill=yellow] (-3.5,0.5) circle (0.25);
\filldraw[fill=blue] (0.5,0.5) circle (0.25); 
\filldraw[fill=orange] (1.5,0.5) circle (0.25);
\filldraw[fill=green] (4.5,0.5) circle (0.25);
\node at (0.5,0.5) {\color{white} $1$};
\node at (4.5,0.5) {\color{white} $2$};
\node at (-3.5,0.5) {\color{white} $3$};
\node at (1.5,0.5) {\color{white} $4$};
\node at (-4.5,0.5) {\color{white} $5$};
\draw[<-,] (-5.5,1) to [in=135, out=45] node[above] {$q$} (-4.6,1);
\draw[<->] (-4.4,1) to [in=135, out=45] node[above] {$1$} (-3.6,1);
\draw[->] (-3.4,1) to [in=135, out=45] node[above] {$1$} (-2.5,1);
\draw[<-,] (-0.5,1) to [in=135, out=45] node[above] {$q$} (0.4,1);
\draw[<->] (0.6,1) to [in=135, out=45] node[above] {$q$} (1.4,1);
\draw[->] (1.6,1) to [in=135, out=45] node[above] {$1$} (2.5,1);
\draw[<-] (3.5,1) to [in=135, out=45] node[above] {$q$} (4.4,1);
\draw[->] (4.6,1) to [in=135, out=45] node[above] {$1$} (5.5,1);
}
\caption{A configuration of the multi-species ASEP and its associated dynamics. The colour of each particle is indicated by its internal number, while possible hops (and their associated rates) are shown by arrows.}
\label{fig:ASEP}
\end{figure}
We will consider two different types of initial data:
\begin{itemize}
\item When referring to {\it step initial data}, we will mean that there is initially a particle of colour $i$ positioned at site $-i$ of the lattice, for all $i \in \mathbb{Z}_{\geq 1}$, while the nonnegative integer sites $\{0,1,2,\dots\}$ are initially empty. In other words,
\begin{align*}
\eta_i(0) = -i,\ \forall\ i < 0,
\qquad
\eta_i(0) = 0,\ \forall\ i \geq 0.
\end{align*}

\item {\it Bernoulli initial data} means that with probability $p$ there is initially a particle of colour $i$ positioned at site $-i$ of the lattice, and with probability $1-p$ this site is unoccupied, for all $i \in \mathbb{Z}_{\geq 1}$. The nonnegative integer sites $\{0,1,2,\dots\}$ are all initially empty.\footnote{Note that not all colours in 
$\{1,2,3,\dots\}$ are present in such a system.}
\end{itemize}

\subsection{Single-species ASEP}

The single-species ASEP is a special case of the system described in Section \ref{sssec:masep}, in which particles become indistinguishable (\ie\ every particle has the same colour, which we can take to be $1$). Therefore, at any point in time $t$, a configuration of the single-species ASEP is given by the list
\begin{align}
\index{ah1@$\eta^{\rm ASEP}(t)$; ASEP occupation data}
\eta^{\rm ASEP}(t)
=
\{\dots,\eta_{-1}(t),\eta_{0}(t),\eta_{1}(t),\dots\},
\qquad
\eta_i(t)
\in \{0,1\}.
\end{align}
We shall again be interested in two different classes of initial data:
\begin{itemize}
\item {\it Step initial data}, which will mean that there is initially a particle positioned at site $-i$ of the lattice, for all $i \in \mathbb{Z}_{\geq 1}$, while the nonnegative integer sites $\{0,1,2,\dots\}$ are initially empty. In other words,
\begin{align*}
\eta_i(0) = 1,\ \forall\ i < 0,
\qquad
\eta_i(0) = 0,\ \forall\ i \geq 0.
\end{align*}

\item {\it Bernoulli initial data}, which will mean that with probability $p$ there is initially a particle positioned at site $-i$ of the lattice, and with probability $1-p$ this site is unoccupied, for all $i \in \mathbb{Z}_{\geq 1}$. The nonnegative integer sites $\{0,1,2,\dots\}$ are all initially empty.
\end{itemize}

\subsection{From the vertex model \eqref{inhom-weights} to multi-species ASEP with step initial data}
\label{sssec:to-step}

Now we describe the reduction of the vertex model \eqref{inhom-weights} to the multi-species ASEP. Our exposition is fairly informal, but it can readily be upgraded to a formal argument following the lines of \cite{Aggarwal2}. 

We will work inside the quadrant $\mathbb{Z}_{\geq 1} \times \mathbb{Z}_{\geq 1}$, so that the indices of all parameters take values in $\mathbb{Z}_{\geq 1}$. The starting point is to study the vertex weights \eqref{inhom-weights} at $s_b = q^{-1/2}$ and to tune the rapidities to the values
\begin{align}
\label{rapidity-choice}
x_a = 1,
\quad
y_b = q^{-1/2} + (1-q) q^{-1/2} \epsilon,
\qquad
\forall\
a \geq 1,\
b \geq 1,
\end{align}
where $\epsilon$ is a small positive real number. The $s_b = q^{-1/2}$ specialization takes us back to the vertex model \eqref{col-vert} (with $x \equiv 1/x_a$, $y \equiv q^{-1/2}/y_b$), and after sending rapidities to the value \eqref{rapidity-choice}, the vertex weights have the following small-$\epsilon$ dependence:
\begin{align}
\label{small-eps}
\begin{tabular}{|c|c|c|}
\hline
\quad
\tikz{0.6}{
\draw[lgray,line width=1.5pt,->] (-1,0) -- (1,0);
\draw[lgray,line width=1.5pt,->] (0,-1) -- (0,1);
\node[left] at (-1,0) {\tiny $i$};\node[right] at (1,0) {\tiny $i$};
\node[below] at (0,-1) {\tiny $i$};\node[above] at (0,1) {\tiny $i$};
}
\quad
&
\quad
\tikz{0.6}{
\draw[lgray,line width=1.5pt,->] (-1,0) -- (1,0);
\draw[lgray,line width=1.5pt,->] (0,-1) -- (0,1);
\node[left] at (-1,0) {\tiny $i$};\node[right] at (1,0) {\tiny $i$};
\node[below] at (0,-1) {\tiny $j$};\node[above] at (0,1) {\tiny $j$};
}
\quad
&
\quad
\tikz{0.6}{
\draw[lgray,line width=1.5pt,->] (-1,0) -- (1,0);
\draw[lgray,line width=1.5pt,->] (0,-1) -- (0,1);
\node[left] at (-1,0) {\tiny $i$};\node[right] at (1,0) {\tiny $j$};
\node[below] at (0,-1) {\tiny $j$};\node[above] at (0,1) {\tiny $i$};
}
\quad
\\[1.3cm]
\quad
$1$
\quad
& 
\quad
$q \epsilon + O(\epsilon^2)$
\quad
& 
\quad
$1-q \epsilon + O(\epsilon^2)$
\quad
\\[0.7cm]
\hline
&
\quad
\tikz{0.6}{
\draw[lgray,line width=1.5pt,->] (-1,0) -- (1,0);
\draw[lgray,line width=1.5pt,->] (0,-1) -- (0,1);
\node[left] at (-1,0) {\tiny $j$};\node[right] at (1,0) {\tiny $j$};
\node[below] at (0,-1) {\tiny $i$};\node[above] at (0,1) {\tiny $i$};
}
\quad
&
\quad
\tikz{0.6}{
\draw[lgray,line width=1.5pt,->] (-1,0) -- (1,0);
\draw[lgray,line width=1.5pt,->] (0,-1) -- (0,1);
\node[left] at (-1,0) {\tiny $j$};\node[right] at (1,0) {\tiny $i$};
\node[below] at (0,-1) {\tiny $i$};\node[above] at (0,1) {\tiny $j$};
}
\quad
\\[1.3cm]
& 
\quad
$\epsilon + O(\epsilon^2)$
\quad
&
\quad
$1-\epsilon + O(\epsilon^2)$
\quad 
\\[0.7cm]
\hline
\end{tabular}
\end{align}
for all $0 \leq i < j \leq n$. As $\epsilon \rightarrow 0$, vertices of the types
\tikz{0.3}{
\draw[lgray,line width=1.5pt,->] (-1,0) -- (1,0);
\draw[lgray,line width=1.5pt,->] (0,-1) -- (0,1);
\node[left] at (-1,0) {\tiny $i$};\node[right] at (1,0) {\tiny $i$};
\node[below] at (0,-1) {\tiny $i$};\node[above] at (0,1) {\tiny $i$};
},
\tikz{0.3}{
\draw[lgray,line width=1.5pt,->] (-1,0) -- (1,0);
\draw[lgray,line width=1.5pt,->] (0,-1) -- (0,1);
\node[left] at (-1,0) {\tiny $i$};\node[right] at (1,0) {\tiny $j$};
\node[below] at (0,-1) {\tiny $j$};\node[above] at (0,1) {\tiny $i$};
} and
\tikz{0.3}{
\draw[lgray,line width=1.5pt,->] (-1,0) -- (1,0);
\draw[lgray,line width=1.5pt,->] (0,-1) -- (0,1);
\node[left] at (-1,0) {\tiny $j$};\node[right] at (1,0) {\tiny $i$};
\node[below] at (0,-1) {\tiny $i$};\node[above] at (0,1) {\tiny $j$};
}
will occur with probability approaching $1$, while vertices of the types
\tikz{0.3}{
\draw[lgray,line width=1.5pt,->] (-1,0) -- (1,0);
\draw[lgray,line width=1.5pt,->] (0,-1) -- (0,1);
\node[left] at (-1,0) {\tiny $i$};\node[right] at (1,0) {\tiny $i$};
\node[below] at (0,-1) {\tiny $j$};\node[above] at (0,1) {\tiny $j$};
} and
\tikz{0.3}{
\draw[lgray,line width=1.5pt,->] (-1,0) -- (1,0);
\draw[lgray,line width=1.5pt,->] (0,-1) -- (0,1);
\node[left] at (-1,0) {\tiny $j$};\node[right] at (1,0) {\tiny $j$};
\node[below] at (0,-1) {\tiny $i$};\node[above] at (0,1) {\tiny $i$};
}
become rare. This means that, in a random configuration of the model in the quadrant $\mathbb{Z}_{\geq 1} \times \mathbb{Z}_{\geq 1}$, paths will tend to propagate diagonally in a zig-zagging fashion, while deviations from these diagonal strata are low-probability events. It is precisely such deviations which will translate into the hopping of particles in the ASEP.

In order to see this more concretely, let us define a family of sets of edges within our lattice. We write
\begin{align}
\mathcal{E}(0)
=
\{ (0,j) \rightarrow (1,j) \}_{j \geq 1} 
\cup
\{ (i,0) \rightarrow (i,1) \}_{i \geq 1}
\end{align}
for the collection of all left and bottom external edges of the quadrant, and define more generally
\begin{align}
\mathcal{E}(k)
=
\{ (k,k+j) \rightarrow (k+1,k+j) \}_{j \geq 1} 
\cup
\{ (k+i,k) \rightarrow (k+i,k+1) \}_{i \geq 1}
\end{align}
by translating all edges in the set $\mathcal{E}(0)$ horizontally and vertically by $k \geq 1$ steps. For any given configuration of the model inside the quadrant $\mathbb{Z}_{\geq 1} \times \mathbb{Z}_{\geq 1}$, we will be interested in the {\it occupation data} associated to the above sets:
\begin{align}
\label{eta-quad}
\eta^{\rm quad}(k)
\index{ah@$\eta^{\rm quad}(k)$; quadrant occupation data}
=
\{\dots,\eta_{-1}(k),\eta_0(k),\eta_{1}(k),\dots\},
\end{align}
where 
\begin{align*}
\eta_{-j}(k) = [\text{colour at edge}\ (k,k+j) \rightarrow (k+1,k+j)],& \ \qquad j \geq 1, 
\\
\eta_{i-1}(k) = [\text{colour at edge}\ (k+i,k) \rightarrow (k+i,k+1)],& \ \qquad i \geq 1.
\end{align*} 
%

\begin{prop}
\label{prop:quad-ASEP-step}
Let $\{\eta^{\rm quad}(k)\}_{k \geq 0}$ denote the occupation data of a random configuration of the model with vertex weights \eqref{small-eps}, inside $\mathbb{Z}_{\geq 1} \times \mathbb{Z}_{\geq 1}$, with coloured domain wall boundary conditions:
\begin{align*}
\eta_{-j}(0) = j, \qquad \forall\ j \geq 1,
\qquad
\eta_i(0) = 0, \qquad \forall\ i \geq 0.
\end{align*}
Similarly, let $\eta^{\rm mASEP}(t)$ denote a configuration of the multi-species ASEP at time $t$, with step initial data at $t=0$. Then by performing the rescaling $k = t/\epsilon$, $t \in \mathbb{R}_{>0}$, we have the following convergence in the sense of finite-dimensional distributions:
\begin{align}
\label{quad-ASEP}
\lim_{\epsilon \rightarrow 0}
\eta^{\rm quad}(t/\epsilon)
=
\eta^{\rm mASEP}(t).
\end{align}
\end{prop}

\begin{proof}
We will not pursue a rigorous proof of this statement, since its rank-1 counterpart has already been explained in a number of previous works \cite{BorodinCG,BorodinP1,Aggarwal2}. 
\end{proof}

\begin{rmk}
\label{rmk:sets}
The convergence result \eqref{quad-ASEP} is in fact rather more general than we have stated. It applies not just to the edge sets \eqref{eta-quad}, but moreover to any family of edge sets which are transversal to the main diagonal of the lattice; any such family will converge to the multi-species ASEP under the scaling that we consider. The convergence can also be extended to multiple time moments.
\end{rmk}

\subsection{From the vertex model \eqref{inhom-weights} to multi-species ASEP with Bernoulli initial data}

The previous arguments can be modified to allow more general boundary conditions in the vertex model \eqref{inhom-weights}, and in turn, in the multi-species ASEP. In particular, we will show how Bernoulli initial data also lies within the scope of our approach. Let us now work inside the quadrant $\mathbb{Z}_{\geq 0} \times \mathbb{Z}_{\geq 1}$, \ie\ with the $0$-th column restored, subject to rainbow boundary conditions: 
\begin{align}
\label{rainbow-bc}
\text{horizontal edge}\ (-1,j) \rightarrow (0,j)\ \text{occupied by path of colour}\ j, & \qquad \forall\ j \geq 1,
\\
\nonumber
\text{vertical edge}\ (i,0) \rightarrow (i,1)\ \text{unoccupied}, & \qquad \forall\ i \geq 0.
\end{align}
We adopt the same choices for the parameters $\{s_b\}_{b \geq 1}$, $\{x_a\}_{a \geq 1}$, $\{y_b\}_{b \geq 1}$ as in Section \ref{sssec:to-step}, and in addition to these, we choose
\begin{align*}
s_0 = \epsilon,
\qquad
y_0 = \frac{p}{\epsilon(1-p)},
\quad
p \in [0,1).
\end{align*}
The vertices in the $0$-th column of the lattice then have the following small-$\epsilon$ dependence:
\begin{align}
\label{0-th}
\begin{tabular}{|c|c|c|}
\hline
\quad
\tikz{0.6}{
\draw[lgray,line width=1.5pt,->] (-1,0) -- (1,0);
\draw[lgray,line width=4pt,->] (0,-1) -- (0,1);
\node[left] at (-1,0) {\tiny $i$};\node[right] at (1,0) {\tiny $0$};
\node[below] at (0,-1) {\tiny $\AA$};\node[above] at (0,1) {\tiny $\AA^{+}_i$};
}
\quad
&
\quad
\tikz{0.6}{
\draw[lgray,line width=1.5pt,->] (-1,0) -- (1,0);
\draw[lgray,line width=4pt,->] (0,-1) -- (0,1);
\node[left] at (-1,0) {\tiny $i$};\node[right] at (1,0) {\tiny $i$};
\node[below] at (0,-1) {\tiny $\AA$};\node[above] at (0,1) {\tiny $\AA$};
}
\quad
&
\quad
\tikz{0.6}{
\draw[lgray,line width=1.5pt,->] (-1,0) -- (1,0);
\draw[lgray,line width=4pt,->] (0,-1) -- (0,1);
\node[left] at (-1,0) {\tiny $j$};\node[right] at (1,0) {\tiny $i$};
\node[below] at (0,-1) {\tiny $\AA$};\node[above] at (0,1) {\tiny $\AA^{+-}_{ji}$};
}
\quad
\\[1.3cm]
\quad
$1-p+O(\epsilon^2)$
\quad
& 
\quad
$p+O(\epsilon^2)$
\quad
& 
\quad
$O(\epsilon^2)$
\quad
\\[0.7cm]
\hline
\end{tabular} 
\end{align}
where we assume that $0 < i < j$. Note that we have only written down the three types of vertices which can actually occur in the $0$-th column, in view of the rainbow domain wall boundary conditions \eqref{rainbow-bc} assigned to its left edges. 

When $\epsilon \rightarrow 0$, and for each $i \geq 1$, we see that as the colour $i$ enters the lattice it either turns into the $0$-th column (and remains there permanently) with probability $1-p$, or takes a horizontal step across to the 1st column with probability $p$. This leads us straight to the following result:
\begin{prop}
\label{prop:quad-ASEP-bern}
Consider the quadrant $\mathbb{Z}_{\geq 0} \times \mathbb{Z}_{\geq 1}$, where the vertices of the $0$-th column are given by \eqref{0-th}, vertices in all other columns have the form \eqref{small-eps}, subject to the boundary conditions \eqref{rainbow-bc}. Let $\{\eta^{\rm quad}(k)\}_{k \geq 0}$ denote the occupation data of a random configuration inside the quadrant, as defined previously \eqref{eta-quad}. 

\smallskip
Similarly, let $\eta^{\rm mASEP, Ber}(t)$ denote a configuration of the multi-species ASEP at time $t$, with Bernoulli initial data at $t=0$. Then by performing the rescaling $k = t/\epsilon$, $t \in \mathbb{R}_{>0}$, we have the following convergence in the sense of finite-dimensional distributions:
\begin{align*}
\lim_{\epsilon \rightarrow 0}
\eta^{\rm quad}(t/\epsilon)
=
\eta^{\rm mASEP, Ber}(t).
\end{align*}
\end{prop}

\begin{rmk}
More elaborate types of random initial data should also be accessible within this framework. For example, performing the specialization $s_b = \epsilon$, $y_b = \frac{p_b}{\epsilon(1-p_b)}$ for all $b \in \{0,1,\dots,k\}$ leads to ``generalized Bernoulli'' initial conditions of the type considered in \cite{AggarwalB}.
\end{rmk}

\subsection{Matching distributions}

Let us now translate the matching statement of Theorem \ref{thm:fused2=3} into the language of the ASEP. In what follows we fix an arbitrary integer $P$ and two sets of pairwise distinct integers $\mathcal{I} = \{I_1 < \cdots < I_k \leq P\}$ and $\mathcal{J} = \{1 \leq J_1 < \cdots < J_{\ell} \}$. We shall be interested in the following pair of probability distributions:
\begin{align*}
\mathbb{P}^{\rm ASEP}(\mathcal{I},\mathcal{J};P,t)
\index{P4@$\mathbb{P}^{\rm ASEP}(\mathcal{I},\mathcal{J};P,t)$}
:=
{\rm Prob}\Big[\eta_{I_i}(t) = 1,\ \forall\ 1 \leq i \leq k,\ \text{and}\ 
\eta_{P+J_j}(t) =1,\ \forall\ 1\leq j \leq \ell \Big],
\end{align*}
describing the probability that in the (uncoloured) ASEP at time $t$, each of the positions $\{I_1,\dots,I_k\}$ and $\{P+J_1,\dots,P+J_{\ell}\}$ are occupied by particles; and
\begin{multline*}
\mathbb{P}^{\rm mASEP}(\mathcal{I},\mathcal{J};P,t)
\index{P5@$\mathbb{P}^{\rm mASEP}(\mathcal{I},\mathcal{J};P,t)$}
\\
:=
{\rm Prob}\Big[\eta_{I_i}(t) \geq 1,\ \forall\ 1 \leq i \leq k,\ \text{and}\ \exists\ p_j > P\ \text{such that}\ 
\eta_{p_j}(t) = J_j, \ \forall\ 1\leq j \leq \ell \Big],
\end{multline*}
describing the probability that in the multi-species ASEP at time $t$, each of the positions $\{I_1,\dots,I_k\}$ are occupied by particles, and the particles of colours $\{J_1,\dots,J_{\ell}\}$ are all situated strictly to the right of position $P$.

\begin{thm}
\label{thm:ASEP-match}
For all time $t \geq 0$, any integer $P \in \mathbb{Z}$, and integer sets $\mathcal{I} = \{I_1 < \cdots < I_k \leq P\}$ and $\mathcal{J} = \{1 \leq J_1 < \cdots < J_{\ell} \}$, one has
\begin{align}
\mathbb{P}^{\rm ASEP}(\mathcal{I},\mathcal{J};P,t)
=
\mathbb{P}^{\rm mASEP}(\mathcal{I},\mathcal{J};P,t).
\end{align}
This distribution match holds for both step and Bernoulli initial data.
\end{thm}

\begin{proof}
We omit the full details of this proof. In the case of step initial data, the proof is a straightforward combination of Theorem \ref{thm:fused2=3}, Proposition \ref{prop:quad-ASEP-step} and Remark \ref{rmk:sets}. For Bernoulli initial data, one combines Theorem \ref{thm:fused2=3}, Proposition \ref{prop:quad-ASEP-bern} and Remark \ref{rmk:sets} to deduce the result.
\end{proof}

\begin{rmk} 
In the case of step initial data and $\mathcal{I}=\varnothing$, Theorem \ref{thm:ASEP-match} is an easy corollary of Theorem 1.4 in \cite{AAV} (in the TASEP case, the latter theorem goes back to \cite{AHR}). This does not extend to non-trivial $\mathcal{I}$ and other initial data, however. 
\end{rmk}

\section{Reduction to $q$-bosons}
\label{sec:bosons}

The next reduction that we consider is to a system of $q$-bosons. We will outline, separately, how the system works when $q=0$ and $q \in (0,1)$, since the $q=0$ case is simpler and worthy of study in its own right.

\subsection{Coloured $q$-bosons}
\label{sssec:coloured-q-bos}

Consider a system of particles with the following features:
\begin{itemize}
\item The system lives on a finite segment of the one-dimensional integral lattice, with sites labelled by the integers $\{1,\dots,M\}$, for some fixed $M \geq 1$.

\item Every particle in the system has a colour prescribed to it, which is a positive real number (colours are no longer restricted to take integral values). 

\item Each site can be occupied by any number of particles, and there are initially no particles in the system. 
\end{itemize}
Let us firstly describe the $q=0$ dynamics, $\mathfrak{D}_{\rm bos}$, of this system:
\begin{enumerate}
\item Particles arrive at site $1$ of the lattice according to a Poisson process of rate $r_0$. The $i$-th particle to enter the system is assigned colour $t_i \in \mathbb{R}_{>0}$, where $t_i$ is its time of entrance.

\item There is an exponential clock of rate $r_i \in \mathbb{R}_{>0}$ attached to the $i$-th site of the lattice, for each $1 \leq i \leq M$ (\ie\ these clocks ring at times that form $M$ independent Poisson processes of intensity $r_i$ at position $i$). All Poisson processes are independent.

\item When the $i$-th clock rings, $1 \leq i \leq M-1$, the particle with maximal colour at position $i$ hops to position $i+1$; all other particles remain stationary. When the $M$-th clock rings, the particle with maximal colour at position $M$ exits the system.
\end{enumerate}
The evolution of the system can be described by world lines of the particles in two dimensions, where the horizontal axis corresponds with the sites $\{1,\dots,M\}$ of the lattice, and the vertical axis is continuous time:
\begin{align}
\label{q-boson-pic}
\tikz{0.8}{
\draw[gray,densely dotted,thick,->] (0,0) -- (11,0);
\draw[gray,densely dotted,thick,->] (0,-0.5) node[below] {entrance} -- (0,6);
\draw[gray,densely dotted,thick] (2,-0.5) node[below] {$1$} -- (2,1);
\draw[gray,densely dotted,thick] (4,-0.5) node[below] {$2$} -- (4,1.2);
\draw[gray,densely dotted,thick] (6,-0.5) node[below] {$\cdots$} -- (6,1.7);
\draw[gray,densely dotted,thick] (8,-0.5) node[below] {$M$} -- (8,3);
\draw[gray,densely dotted,thick] (10,-0.5) node[below] {exit} -- (10,6);
\draw[line width=1.5pt,blue,->,smooth] (0,1) node[left] {$t_1$} -- (2,1) -- (2,1.2) -- (4,1.2) -- (4,1.7) -- (6,1.7) -- (6,3.3) -- (8,3.3) -- (8,3.6) -- (10,3.6);
\draw[line width=1.5pt,green,->,smooth] (0,2) node[left] {$t_2$} -- (2,2) -- (2,4) -- (4,4) -- (4,4.3) -- (6,4.3) -- (6,6);
\draw[line width=1.5pt,yellow,->,smooth] (0,2.4) node[left] {$t_3$} -- (2,2.4) -- (2,2.6) -- (4,2.6) -- (4,2.8) -- (6,2.8) -- (6,2.9) -- (8,2.9) -- (8,3.2) -- (10,3.2);
\draw[line width=1.5pt,orange,->,smooth] (0,3) node[left] {$t_4$} -- (2,3) -- (2,3.6) -- (4,3.6) -- (4,3.8) -- (6,3.8) -- (6,4.7) -- (8,4.7) -- (8,4.9) -- (10,4.9);
\draw[line width=1.5pt,red,->,smooth] (0,5) node[left] {$t_5$} -- (2,5) -- (2,5.5) -- (4,5.5) -- (4,6);
}
\end{align}
The state of the system $\eta^{\rm bos}(t)$ at a given time $t$ can be read by taking a horizontal cross-section of the above picture; namely, one has
\begin{align*}
\index{ah@$\eta^{\rm bos}(t)$; $q$-boson occupation data}
\eta^{\rm bos}(t)
=
\{c_1(t),\dots,c_M(t)\},
\end{align*}
where each $c_i(t)$ is a (possibly empty) set of real numbers which indicates the colours present at position $i$ at time $t$. 

The dynamics for generic $q \in (0,1)$, denoted \index{D2@$\mathfrak{D}_{\rm bos}(q)$; $q$-boson dynamics} $\mathfrak{D}_{\rm bos}(q)$, are obtained by replacing rules $(2)$ and $(3)$ above with the following:
\begin{enumerate}[label={(\arabic*$'$)}]
\setcounter{enumi}{1}
\item At some point in time, suppose there are $m_i$ particles present at the $i$-th site of the lattice. Then we assign to that site $m_i$ clocks with rates $\{r_i (1-q), r_i (1-q) q,\dots,r_i (1-q) q^{m_i-1}\}$, where the clock of rate $r_i (1-q) q^{j-1}$ is associated to the particle with $j$-th highest colour. All clocks are independent.
\item When the $j$-th clock at site $i$ rings, the particle with the $j$-th highest colour hops from position $i$ to $i+1$. The total rate of exit from site $i$ is thus $r_i (1-q^{m_i})$. If $i=M$, the hopping particle is considered to have left the system.
\end{enumerate}
It is clear that rules $(2')$ and $(3')$ reduce to $(2)$ and $(3)$ respectively, when $q=0$. 

\subsection{Projection onto $q$-{\rm TASEP}}

The dynamics $\mathfrak{D}_{\rm bos}$ have a nice colour-blind projection. If one ignores the different colours prescribed to each of the lines in the picture \eqref{q-boson-pic}, it can also be viewed as the time evolution of the (uncoloured) totally asymmetric simple exclusion process (TASEP). Let us recall the TASEP dynamics, 
$\mathfrak{D}_{\rm TASEP}$, on $\mathbb{Z}$ with step initial data:
\begin{itemize}[label={$\star$}]
\item Each site of the lattice can be occupied by at most one particle.

\item Initially there are no particles present at sites labelled by nonnegative integers $\{0,1,2,\dots\}$, and particles situated at all negative positions $\{\dots,-2,-1\}$. We refer to the particle that starts at $-i$ as the $i$-th particle, $i \in \mathbb{Z}_{\geq 1}$. In what follows we focus only on the motion of the first $M+1$ particles, in which case the initial presence of the particles at positions $j <-M-1$ is irrelevant. 

\item We denote TASEP configurations by $\mathcal{P}^{\rm TASEP}(t) = \{p_i(t)\}_{i \in \mathbb{Z}_{\geq 1}}$, where $p_i(t) \in \mathbb{Z}$ is the position of the $i$-th particle at time $t$.

\item The $i$-th particle has an associated exponential clock of rate $r_{i-1} \cdot {\bm 1}_{(p_{i-1}>p_i+1)} \in \mathbb{R}_{\geq 0}$, where $p_0 = \infty$, by agreement. When this clock rings, the $i$-th particle takes one step to the right (by assumption of the clock ringing, the position that it hops to is unoccupied by the $(i-1)$-th particle). All clocks are independent.

\end{itemize}
To read the time evolution of the TASEP from the uncoloured version of the picture \eqref{q-boson-pic}, one should view each arrival of a boson at site $i \in \{1,\dots,M\}$ as corresponding with a jump of the $i$-th particle in the TASEP. In this way, at any time $t$, the number of particles $|c_{i-1}(t)|$ present at site $i-1$ in the picture \eqref{q-boson-pic} is equivalent to the gap $g_i = p_{i-1}-p_{i}-1$ between the $i$-th and $(i-1)$-th particles in the TASEP.

One can also consider colour-blind projections of the dynamics $\mathfrak{D}_{\rm bos}(q)$. In this case, the relevant underlying system is the $q$-TASEP, with dynamics $\mathfrak{D}_{\rm TASEP}(q)$ \index{D3@$\mathfrak{D}_{\rm TASEP}(q)$; $q$-TASEP dynamics} given by replacing the last of the rules above with the following:
\begin{itemize}[label={$\star$}]
\item The $i$-th particle has an associated exponential clock of rate $r_{i-1}
 (1-q^{g_i})$, where $g_i = p_{i-1}-p_i-1$, as before. When this clock rings, the $i$-th particle takes one step to the right; note that this rate is automatically $0$ when the $i$-th and $(i-1)$-th particles are nearest neighbours. All clocks are independent.
\end{itemize}
It is easy to see that the colour-blind projection of the $q$-boson system and the $q$-TASEP coincide exactly; indeed, the total rate of arrival at site $i$ in the former is given by $r_{i-1} (1-q^{|c_{i-1}(t)|})$, which matches with the hopping rate $r_{i-1} (1-q^{g_i})$ of the $i$-th particle in the latter.

\subsection{From the vertex model \eqref{inhom-weights} to coloured $q$-bosons}
\label{sssec:reduce-qb}

Below we present a sketch of the reduction of the vertex model \eqref{inhom-weights} to the system of coloured $q$-bosons; we will avoid a rigorous proof of this result, as it is fairly straightforward. 

Consider the model \eqref{inhom-weights} after making the specializations $s_b = \epsilon$, $y_b = r_b$, $x_a = -1$, where both $\epsilon,r_b \in \mathbb{R}_{>0}$. We shall also remove the restriction on the rank of the model, taking $n \rightarrow \infty$. Assuming that $\epsilon$ is small, and expanding each vertex weight to leading order in $\epsilon$, they take the form
\begin{align}
\label{inhom-weights-eps}
\begin{tabular}{|c|c|c|}
\hline
\quad
\tikz{0.6}{
\draw[lgray,line width=1.5pt,->] (-1,0) -- (1,0);
\draw[lgray,line width=4pt,->] (0,-1) -- (0,1);
\node[left] at (-1,0) {\tiny $0$};\node[right] at (1,0) {\tiny $0$};
\node[below] at (0,-1) {\tiny $\AA$};\node[above] at (0,1) {\tiny $\AA$};
}
\quad
&
\quad
\tikz{0.6}{
\draw[lgray,line width=1.5pt,->] (-1,0) -- (1,0);
\draw[lgray,line width=4pt,->] (0,-1) -- (0,1);
\node[left] at (-1,0) {\tiny $i$};\node[right] at (1,0) {\tiny $i$};
\node[below] at (0,-1) {\tiny $\AA$};\node[above] at (0,1) {\tiny $\AA$};
}
\quad
&
\quad
\tikz{0.6}{
\draw[lgray,line width=1.5pt,->] (-1,0) -- (1,0);
\draw[lgray,line width=4pt,->] (0,-1) -- (0,1);
\node[left] at (-1,0) {\tiny $0$};\node[right] at (1,0) {\tiny $i$};
\node[below] at (0,-1) {\tiny $\AA$};\node[above] at (0,1) {\tiny $\AA^{-}_i$};
}
\quad
\\[1.3cm]
\quad
$1+O(\epsilon)$
\quad
& 
\quad
$\epsilon r_b q^{[\bm{A}]_{i+1}}+O(\epsilon^2)$
\quad
& 
\quad
$\epsilon r_b (1-q^{A_i}) q^{[\bm{A}]_{i+1}}+O(\epsilon^2)$
\quad
\\[0.7cm]
\hline
\quad
\tikz{0.6}{
\draw[lgray,line width=1.5pt,->] (-1,0) -- (1,0);
\draw[lgray,line width=4pt,->] (0,-1) -- (0,1);
\node[left] at (-1,0) {\tiny $i$};\node[right] at (1,0) {\tiny $0$};
\node[below] at (0,-1) {\tiny $\AA$};\node[above] at (0,1) {\tiny $\AA^{+}_i$};
}
\quad
&
\quad
\tikz{0.6}{
\draw[lgray,line width=1.5pt,->] (-1,0) -- (1,0);
\draw[lgray,line width=4pt,->] (0,-1) -- (0,1);
\node[left] at (-1,0) {\tiny $i$};\node[right] at (1,0) {\tiny $j$};
\node[below] at (0,-1) {\tiny $\AA$};\node[above] at (0,1) 
{\tiny $\AA^{+-}_{ij}$};
}
\quad
&
\quad
\tikz{0.6}{
\draw[lgray,line width=1.5pt,->] (-1,0) -- (1,0);
\draw[lgray,line width=4pt,->] (0,-1) -- (0,1);
\node[left] at (-1,0) {\tiny $j$};\node[right] at (1,0) {\tiny $i$};
\node[below] at (0,-1) {\tiny $\AA$};\node[above] at (0,1) {\tiny $\AA^{+-}_{ji}$};
}
\quad
\\[1.3cm] 
\quad
$1+O(\epsilon)$
\quad
& 
\quad
$\epsilon r_b(1-q^{A_j}) q^{[\bm{A}]_{j+1}}+O(\epsilon^2)$
\quad
&
\quad
$O(\epsilon^2)$
\quad
\\[0.7cm]
\hline
\end{tabular} 
\end{align}
valid for $1 \leq i < j$, where we have introduced $[\bm{A}]_i := \sum_{k \geq i} A_k$. As $\epsilon \rightarrow 0$, we see that vertices whose right edge is unoccupied (those in the first column of the table above) become extremely probable, while the remaining types of vertices are low-probability events. 

We will study what happens at each such instance of a rare event. Starting from the rainbow boundary conditions in the quadrant, namely,
\begin{align}
\label{quad}
\tikz{0.7}{
\foreach\x in {1,...,5}{
\draw[lgray,line width=4pt,->] (\x,0) -- (\x,6);
}
\foreach\y in {1,...,5}{
\draw[lgray,line width=1.5pt,->] (0,\y) -- (6,\y);
}
\draw[line width=1.5pt,red,->] (0,5) -- (1,5);
\draw[line width=1.5pt,orange,->] (0,4) -- (1,4);
\draw[line width=1.5pt,yellow,->] (0,3) -- (1,3) node[midway,above] {\tiny $3$};
\draw[line width=1.5pt,green,->] (0,2) -- (1,2) node[midway,above] {\tiny $2$};
\draw[line width=1.5pt,blue,->] (0,1) -- (1,1) node[midway,above] {\tiny $1$};
\node[below] at (5,0) {$r_M$};
\node[below] at (4,0) {$\cdots$};
\node[below] at (3,0) {$\cdots$};
\node[below] at (2,0) {$r_1$};
\node[below] at (1,0) {$r_0$};
\node[below right] at (6,0.2) {(rate parameters)};
}
\end{align}
paths will tend to turn into the $0$-th column and remain in it permanently. However, as the model evolves vertically up the page, at any given moment we can see the occurrence of one of two possible rare events: {\bf 1.} A path of colour $i$, as it enters the lattice, does not immediately turn into the $0$-th column, but instead takes a single horizontal step across to the 1st column. This event corresponds with the vertex $\tikz{0.35}{
\draw[lgray,line width=1pt,->] (-1,0) -- (1,0);
\draw[lgray,line width=2.5pt,->] (0,-1) -- (0,1);
\node[left] at (-1,0) {\tiny $i$};\node[right] at (1,0) {\tiny $i$};
\node[below] at (0,-1) {\tiny $\bm{A}$};\node[above] at (0,1) 
{\tiny $\bm{A}$};
}$ and happens with probability $\epsilon r_0 + O(\epsilon^2)$, where we have noted that the $q^{[\bm{A}]_{i+1}}$ factor which is associated to such a vertex is absent, due to the fact that necessarily $[\bm{A}]_{i+1} = 0$ (all colours greater than $i$ enter the lattice higher up along the vertical axis). This path turns immediately into the 1st column, since any further horizontal steps come with extra powers of $\epsilon$, leading to a combined weight of $O(\epsilon^2)$, which is negligibly small as $\epsilon \rightarrow 0$; {\bf 2.} A path of colour $i$, situated in the $j$-th column of the lattice, makes a turn into the $(j+1)$-th column. This event corresponds with the vertex $\tikz{0.35}{
\draw[lgray,line width=1pt,->] (-1,0) -- (1,0);
\draw[lgray,line width=2.5pt,->] (0,-1) -- (0,1);
\node[left] at (-1,0) {\tiny $0$};\node[right] at (1,0) {\tiny $i$};
\node[below] at (0,-1) {\tiny $\bm{A}$};\node[above] at (0,1) 
{\tiny $\bm{A}^{-}_i$};
}$ and happens with probability $\epsilon r_j (1-q) q^{[\bm{A}]_{i+1}}+O(\epsilon^2)$. Once again, we can rule out the possibility of this path making further horizontal steps beyond the $(j+1)$-th column, since this leads to an $O(\epsilon^2)$ weighting.

One can easily verify that these are the only possible $O(\epsilon)$ events, and that furthermore they coincide exactly with the hopping rules and rates of $\mathfrak{D}_{\rm bos}(q)$. Although the vertical spacing between such events diverges as $\epsilon \rightarrow 0$, we can restore a finite spacing by rescaling the vertical axis by $1/\epsilon$. We now summarise these considerations in the following statement:


\begin{prop}
Consider a random configuration of the model \eqref{inhom-weights-eps} inside the quadrant \eqref{quad}, and for all $1 \leq i \leq M$, $j \geq 0$, let $c^{\rm quad}_i(j)$ denote the set of colours which pass through the vertical edge $(i,j) \rightarrow (i,j+1)$. Then under the rescaling $j = t/\epsilon$ (of both vertical coordinates and colours) we have the following convergence in finite-dimensional distributions:
\begin{align*}
\lim_{\epsilon \rightarrow 0}
\left\{c^{\rm quad}_i(t/\epsilon)\right\}_{1 \leq i \leq M}
=
\eta^{\rm bos}(t).
\end{align*}
\end{prop}


\subsection{Equally distributed random sets}
\label{sssec:equal}

Let us fix some cut-off time $t \in \mathbb{R}_{>0}$ in the coloured $q$-boson system, and define the collection of random sets that will interest us:
\begin{itemize}
\item The set of colours \index{J1@$\mathcal{J}^{\rm bos}(t)$} $\mathcal{J}^{\rm bos}(t) = \{0 < J_1 < \cdots < J_\ell < t\}$ that have exited the system (\ie\ colours that have passed to the right of site $M$) up to time $t$. This is a continuous random variable, since $J_i \in \mathbb{R}_{>0}$ for all $1 \leq i \leq \ell$.

\smallskip

\item The set of times \index{T2@$\mathcal{T}^{\rm bos}(t)$} $\mathcal{T}^{\rm bos}(t) = \{0 < \tau_1 < \cdots < \tau_\ell < t\}$ when the exits happened. We will define, in fact, the complemented set \index{J1@$\tilde{\mathcal{J}}^{\rm bos}(t)$} $\tilde{\mathcal{J}}^{\rm bos}(t) = \{0 < \tilde{J}_1 < \cdots < \tilde{J}_\ell < t\}$, where $\tilde{J}_i = t-\tau_{\ell-i+1} \in \mathbb{R}_{>0}$ for all $1 \leq i \leq \ell$. 

\smallskip
\noindent
In the language of the $q$-TASEP, $\mathcal{T}^{\rm bos}(t)$ translates to the set of times when the $(M+1)$-th particle performed a hop.

\smallskip

\item The set of occupancies \index{O@$\mathcal{O}^{\rm bos}(t)$} $\mathcal{O}^{\rm bos}(t) = \{m_1,\dots,m_M\}$, where $m_i \in \mathbb{N}$ is the number of particles (of any colour) situated at position $i$ at time $t$, for all $1 \leq i \leq M$. Unlike the previous two, this random variable remains discrete. 

\smallskip
\noindent
In the language of the $q$-TASEP, $\mathcal{O}^{\rm bos}(t)$ translates to $\{g_2,\dots,g_{M+1}\}$, where as before $g_i$ measures the gap between the $i$-th and $(i-1)$-th particles at time $t$.
\end{itemize} 
\begin{thm}
\label{thm:12.5}
$(\mathcal{O}^{\rm bos}(t),\tilde{\mathcal{J}}^{\rm bos}(t))$ is equidistributed with $(\mathcal{O}^{\rm bos}(t),\mathcal{J}^{\rm bos}(t))$, for any $t \in \mathbb{R}_{>0}$.\footnote{Note that the inclusion of the random variable $\mathcal{O}^{\rm bos}(t)$ in this statement is non-trivial, since it is not independent of either $\tilde{\mathcal{J}}^{\rm bos}(t)$ or $\mathcal{J}^{\rm bos}(t)$.}
\end{thm}

\begin{proof}
This result follows as a degeneration of Theorem \ref{thm:fused2=3}. In Section \ref{sssec:reduce-qb} we have already seen how a certain specialization of the rapidity and spin parameters, combined with suitable rescaling of the vertical axis, makes the vertex model \eqref{inhom-weights} in the quadrant converge to the dynamics $\mathfrak{D}_{\rm bos}(q)$. By analogous techniques, one can show that under the same specializations and rescaling, the $n=1$ version of the model \eqref{inhom-weights} in the quadrant converges to the colour-blind dynamics 
$\mathfrak{D}_{\rm TASEP}(q)$. 

It follows that Theorem \ref{thm:fused2=3} induces a distribution match between the coloured $q$-boson system and its colour-blind equivalent; it is then only a matter of checking what observables one recovers. It is easy to see that $\mathbb{Z}_{M,N}(\mathcal{I},\mathcal{J})$ converges to the distribution of $(\mathcal{I}^{\rm bos}(t),\tilde{\mathcal{J}}^{\rm bos}(t))$, where $\mathcal{I}^{\rm bos}(t) = \{1 \leq I_1 \leq \cdots \leq I_k \leq M \}$ records the positions of all particles which remain in the system at time $t$; on the other hand, $\mathbb{X}_{M,N}(\mathcal{I},\mathcal{J})$ converges to the distribution of $(\mathcal{I}^{\rm bos}(t),\mathcal{J}^{\rm bos}(t))$. The claim is then immediate, noting that $\mathcal{I}^{\rm bos}(t)$ and $\mathcal{O}^{\rm bos}(t)$ encode the same information about the system.

\end{proof}

\section{Reduction to $q$-PushTASEP}
\label{sec:push}

Our final reduction is to a system which, even under colour-blind projection, is to the best of our knowledge unstudied. Throughout this section we tacitly assume that $q>1$; this is necessary to ensure that events are assigned nonnegative probabilities.

\subsection{Coloured $q$-PushTASEP}

The coloured $q$-PushTASEP has a similar setup to the previous model of coloured $q$-bosons:
\begin{itemize}
\item The system lives on a finite segment of the one-dimensional integral lattice, with sites labelled by the integers $\{1,\dots,M\}$, for some fixed $M \geq 1$. 

\item Every particle in the system has a colour prescribed to it, which is a positive real number. 

\item Each site can be occupied by at most $P$ particles, where $P$ is some arbitrary fixed positive integer, and there are initially no particles in the system. 
\end{itemize}
The dynamics $\mathfrak{D}_{\rm cPush}(q,P)$ \index{D1@$\mathfrak{D}_{\rm cPush}(q,P)$; $q$-PushTASEP dynamics} of the system are composed of two basic ingredients: (a) the ringing of an exponential clock, which ``activates'' a particle somewhere within the system; (b) the activated particle relocating to a new site somewhere to the right of its starting position, subject to certain probabilistic rules, and then ``de-activating''. In the course of this relocation, it may activate other particles. These dynamics are more involved than in the previous ASEP and $q$-boson systems, in the sense that in the latter systems particle motion is deterministic after a clock rings. 

Below we describe $\mathfrak{D}_{\rm cPush}(q,P)$; points (1) and (2) deal with activation by clocks, while (3) describes what happens after a particle gets activated. As usual, all clocks in the system are independent.

\begin{enumerate}
\item A new particle is activated and arrives at site $1$ of the lattice according to a Poisson process of rate $r_0$. The $i$-th particle to arrive in this way is assigned colour $t_i \in \mathbb{R}_{>0}$, where $t_i$ is its time of entrance.

\item At some time $t$ suppose there are $m_i \leq P$ particles present at the $i$-th site of the lattice. We assign to that site $m_i$ clocks with rates $\{r_i (q-1) q^{-P}, r_i (q-1) q^{1-P},\dots,r_i (q-1) q^{m_i-1-P}\}$, and refer to the clock with rate $r_i (q-1) q^{j-1-P}$ as the $j$-th clock. When the $j$-th clock rings, the particle with $j$-th highest colour is activated and removed from this site, and arrives at site $i+1$.

\item When an activated particle of colour $c$ arrives at site $i$, three types of events are possible: {\bf 1.} It remains at this site and de-activates with probability $1-q^{m_i-P}$, where $m_i$ is the number of particles at site $i$ prior to arrival; {\bf 2.} It remains at this site and de-activates, but causes the activation of another particle of colour $d < c$ present at site $i$, which is then transferred to site $i+1$. This event occurs with probability $(q-1) q^{m_i(d)-P}$, where $m_i(d)$ is the number of particles at site $i$ with colours exceeding $d$, prior to arrival; {\bf 3.} The activated particle passes on to site $i+1$ and remains activated, with probability $q^{m_i(c)-P}$. Note that (with probability $1$) no two particles have the same colour, and
\begin{align*}
1-q^{m_i-P} 
+ 
\sum_{\substack{\text{particles of} \\ \text{colours}\ d<c}}
(q-1) q^{m_i(d)-P}
+
q^{m_i(c)-P}
\equiv
1.
\end{align*} 

\item If a particle arrives at site $M+1$ of the lattice, it is considered to have exited the system.
\end{enumerate}
As in the case of the $q$-boson system of Section \ref{sssec:coloured-q-bos}, we can chart the time evolution of the coloured $q$-PushTASEP by figures of the form 
\begin{align}
\label{q-push-pic}
\tikz{0.8}{
\draw[gray,densely dotted,thick,->] (0,0) -- (11,0);
\draw[gray,densely dotted,thick,->] (0,-0.5) node[below] {entrance} -- (0,6);
\draw[gray,densely dotted,thick] (2,-0.5) node[below] {$1$} -- (2,1);
\draw[gray,densely dotted,thick] (4,-0.5) node[below] {$2$} -- (4,1);
\draw[gray,densely dotted,thick] (6,-0.5) node[below] {$\cdots$} -- (6,1.8);
\draw[gray,densely dotted,thick] (8,-0.5) node[below] {$M$} -- (8,2.2);
\draw[gray,densely dotted,thick] (10,-0.5) node[below] {exit} -- (10,6);
\draw[line width=1.5pt,blue,->,smooth] (0,1) node[left] {$t_1$} -- (4,1) -- (4,3.4) -- (8,3.4) -- (8,4.6) -- (10,4.6);
\draw[line width=1.5pt,green,->,smooth] (0,1.8) node[left] {$t_2$} -- (6,1.8) -- (6,2.2) -- (8,2.2) -- (8,3.4) -- (10,3.4);
\draw[line width=1.5pt,yellow,->,smooth] (0,2.4) node[left] {$t_3$} -- (2,2.4) -- (2,4.1) -- (4,4.1) -- (4,6);
\draw[line width=1.5pt,orange,->,smooth] (0,3) node[left] {$t_4$} -- (6,3) -- (6,5.5) -- (10,5.5);
\draw[line width=1.5pt,red,->,smooth] (0,4.1) node[left] {$t_5$} -- (2,4.1) -- (2,6);
\filldraw[fill=blue,draw=blue] (0,1) circle (0.1);
\filldraw[fill=white,draw=blue] (2,1) circle (0.1);
\filldraw[fill=white,draw=blue] (4,1) circle (0.1);
\filldraw[fill=blue,draw=blue] (4,3.4) circle (0.1);
\filldraw[fill=white,draw=blue] (6,3.4) circle (0.1);
\filldraw[fill=white,draw=blue] (8,3.4) circle (0.1);
\filldraw[fill=blue,draw=blue] (8,4.6) circle (0.1);
\filldraw[fill=green,draw=green] (0,1.8) circle (0.1);
\filldraw[fill=white,draw=green] (2,1.8) circle (0.1);
\filldraw[fill=white,draw=green] (4,1.8) circle (0.1);
\filldraw[fill=white,draw=green] (6,1.8) circle (0.1);
\filldraw[fill=green,draw=green] (6,2.2) circle (0.1);
\filldraw[fill=white,draw=green] (8,2.2) circle (0.1);
\filldraw[fill=yellow,draw=yellow] (0,2.4) circle (0.1);
\filldraw[fill=white,draw=yellow] (2,2.4) circle (0.1);
\filldraw[fill=white,draw=yellow] (4,4.1) circle (0.1);
\filldraw[fill=orange,draw=orange] (0,3) circle (0.1);
\filldraw[fill=white,draw=orange] (2,3) circle (0.1);
\filldraw[fill=white,draw=orange] (4,3) circle (0.1);
\filldraw[fill=white,draw=orange] (6,3) circle (0.1);
\filldraw[fill=orange,draw=orange] (6,5.5) circle (0.1);
\filldraw[fill=white,draw=orange] (8,5.5) circle (0.1);
\filldraw[fill=red,draw=red] (0,4.1) circle (0.1);
\filldraw[fill=white,draw=red] (2,4.1) circle (0.1);
}
\end{align}
where in this instance we choose $P=1$ (lines cannot vertically overlap). In this picture, closed dots $\bullet$ indicate times where an alarm clock rang; while open dots $\circ$ indicate subsequent ``choices'' made by activated particles, until they de-activate.

In a similar way as in the coloured $q$-boson system, the state $\eta^{\rm Push}(t)$ of the coloured $q$-PushTASEP at a given time $t$ is denoted
\begin{align*}
\index{ah@$\eta^{\rm Push}(t)$; $q$-PushTASEP occupation data}
\eta^{\rm Push}(t)
=
\{c_1(t),\dots,c_M(t)\},
\qquad
0 \leq |c_i(t)| \leq P,
\quad \forall\ 1 \leq i \leq M,
\end{align*}
where each $c_i(t)$ is a (possibly empty) set of real numbers which indicates the colours present at position $i$ at time $t$.

The dynamics $\mathfrak{D}_{\rm cPush}(q,P)$ admit a well-defined $q \rightarrow \infty$ limit, when they simplify considerably. The rules which govern $\mathfrak{D}_{\rm cPush}(\infty,P)$ are as stated above, up to replacement of (2) and (3) by the following:
\begin{enumerate}[label={(\arabic*$'$)}]
\setcounter{enumi}{1}
\item At some time $t$ suppose there are $m_i \leq P$ particles present at the $i$-th site of the lattice. We associate to that site a single exponential clock with rate $r_i \cdot {\bm 1}_{m_i = P}$. When the clock rings, the particle with the lowest colour at site $i$ is activated and removed from the site, and arrives at site $i+1$.
\item When a particle arrives at site $i$, hosting $m_i$ particles prior to the arrival, it remains there and de-activates if $m_i < P$. If $m_i = P$, the smallest colour at site $i$ (post arrival) gets activated and removed from the site, and arrives at site $i+1$.
\end{enumerate}
Note that the colour of individual particles becomes redundant information in $\mathfrak{D}_{\rm cPush}(\infty,P)$, since it is always the smallest colour at any given site that undergoes activation. Taking $P=1$, $\mathfrak{D}_{\rm cPush}(\infty,1)$ becomes exactly the dynamics of the (space inhomogeneous) PushTASEP.


\subsection{Projection onto $q$-PushTASEP}

If one ignores the different colours assigned to particles in the previous system, it projects onto colour-blind dynamics $\mathfrak{D}_{\rm Push}(q,P)$ with the following rules:
\begin{enumerate}
\item Particles are activated and arrive at site $1$ of the lattice according to a Poisson process of rate $r_0$.

\item At some time $t$ suppose there are $m_i \leq P$ particles present at the $i$-th site of the lattice. We assign to that site a single clock of rate $r_i (q^{m_i}-1) q^{-P}$. When the clock rings, one of the particles is activated and removed from this site, and arrives at site $i+1$.

\item When an activated particle arrives at site $i$, two events are possible: {\bf 1.} It remains at this site and de-activates with probability $1-q^{m_i-P}$, where $m_i$ is the number of particles at site $i$ prior to arrival; {\bf 2.} The activated particle passes on to site $i+1$ and remains activated, with probability $q^{m_i-P}$. One can also view this event as the activated particle settling at site $i$, and kicking out one of the existing particles there; the two points of view are equivalent, since all particles are now indistinguishable.

\item If a particle arrives at site $M+1$ of the lattice, it is considered to have exited the system.
\end{enumerate}

\subsection{From the vertex model \eqref{inhom-weights} to coloured $q$-PushTASEP }

Now we explain how to recover the $q$-PushTASEP from the inhomogeneous, higher spin vertex model \eqref{inhom-weights}. To perform the reduction we treat the $0$-th column of the lattice on a separate footing from the columns with positive index. The reason for this separate treatment is that the $0$-th column will control the entrance of particles into the system (\ie\ it will give rise to a Poisson process of rate $r_0$), while the remaining columns will produce the bulk dynamics of the $q$-PushTASEP.

Starting from the vertex model \eqref{inhom-weights}, we make the following specializations: {\bf 1.} Let $s_0 = \epsilon$, while for all $b \geq 1$ we choose $s_b = q^{-P/2}$, where $P$ is some fixed positive integer. The latter choice bounds the total number of paths occupying a vertical edge (in columns with index $b \geq 1$) to be no greater than $P$; {\bf 2.} Let $x_a = -1$ for all $a \geq 1$, $y_0 = r_0$ and $y_b = -\epsilon r_b q^{-P/2}$ for all $b \geq 1$; {\bf 3.} Remove the bound on the rank, by sending $n \rightarrow \infty$. Vertices in the $0$-th column then have the same small-$\epsilon$ expansion as in \eqref{inhom-weights-eps} (with $b=0$), while the vertices of the remaining columns are tabulated below:
\begin{align}
\label{inhom-weights-eps2}
\begin{tabular}{|c|c|c|}
\hline
\quad
\tikz{0.6}{
\draw[lgray,line width=1.5pt,->] (-1,0) -- (1,0);
\draw[lgray,line width=4pt,->] (0,-1) -- (0,1);
\node[left] at (-1,0) {\tiny $0$};\node[right] at (1,0) {\tiny $0$};
\node[below] at (0,-1) {\tiny $\AA$};\node[above] at (0,1) {\tiny $\AA$};
}
\quad
&
\quad
\tikz{0.6}{
\draw[lgray,line width=1.5pt,->] (-1,0) -- (1,0);
\draw[lgray,line width=4pt,->] (0,-1) -- (0,1);
\node[left] at (-1,0) {\tiny $i$};\node[right] at (1,0) {\tiny $i$};
\node[below] at (0,-1) {\tiny $\AA$};\node[above] at (0,1) {\tiny $\AA$};
}
\quad
&
\quad
\tikz{0.6}{
\draw[lgray,line width=1.5pt,->] (-1,0) -- (1,0);
\draw[lgray,line width=4pt,->] (0,-1) -- (0,1);
\node[left] at (-1,0) {\tiny $0$};\node[right] at (1,0) {\tiny $i$};
\node[below] at (0,-1) {\tiny $\AA$};\node[above] at (0,1) {\tiny $\AA^{-}_i$};
}
\quad
\\[1.3cm]
\quad
$1+O(\epsilon)$
\quad
& 
\quad
$q^{[\bm{A}]_{i}-P}+O(\epsilon)$
\quad
& 
\quad
$\epsilon r_b (q^{A_i}-1) q^{[\bm{A}]_{i+1}-P}+O(\epsilon^2)$
\quad
\\[0.7cm]
\hline
\quad
\tikz{0.6}{
\draw[lgray,line width=1.5pt,->] (-1,0) -- (1,0);
\draw[lgray,line width=4pt,->] (0,-1) -- (0,1);
\node[left] at (-1,0) {\tiny $i$};\node[right] at (1,0) {\tiny $0$};
\node[below] at (0,-1) {\tiny $\AA$};\node[above] at (0,1) {\tiny $\AA^{+}_i$};
}
\quad
&
\quad
\tikz{0.6}{
\draw[lgray,line width=1.5pt,->] (-1,0) -- (1,0);
\draw[lgray,line width=4pt,->] (0,-1) -- (0,1);
\node[left] at (-1,0) {\tiny $i$};\node[right] at (1,0) {\tiny $j$};
\node[below] at (0,-1) {\tiny $\AA$};\node[above] at (0,1) 
{\tiny $\AA^{+-}_{ij}$};
}
\quad
&
\quad
\tikz{0.6}{
\draw[lgray,line width=1.5pt,->] (-1,0) -- (1,0);
\draw[lgray,line width=4pt,->] (0,-1) -- (0,1);
\node[left] at (-1,0) {\tiny $j$};\node[right] at (1,0) {\tiny $i$};
\node[below] at (0,-1) {\tiny $\AA$};\node[above] at (0,1) {\tiny $\AA^{+-}_{ji}$};
}
\quad
\\[1.3cm] 
\quad
$1-q^{[\bm{A}]_1-P} + O(\epsilon)$
\quad
& 
\quad
$\epsilon r_b (q^{A_j}-1) q^{[\bm{A}]_{j+1}-P}+O(\epsilon^2)$
\quad
&
\quad
$(q^{A_i}-1)q^{[\bm{A}]_{i+1}-P}+O(\epsilon)$
\quad
\\[0.7cm]
\hline
\end{tabular} 
\end{align}
valid for $1 \leq i < j$, with $[\bm{A}]_i = \sum_{k \geq i} A_k$, as before. Let us now repeat the type of analysis performed in Section \ref{sssec:reduce-qb}. Starting from rainbow boundary conditions in the quadrant \eqref{quad}, with $0$-th column given by \eqref{inhom-weights-eps} and all other columns given by \eqref{inhom-weights-eps2}, paths will tend to turn into the $0$-th column and remain in it permanently. As the model evolves vertically up the page, we can see the occurrence of two possible rare events: {\bf 1.} A path of colour $i$, as it enters the lattice, does not immediately turn into the $0$-th column, but instead takes a single horizontal step across to the 1st column. This event corresponds with the vertex $\tikz{0.35}{
\draw[lgray,line width=1pt,->] (-1,0) -- (1,0);
\draw[lgray,line width=2.5pt,->] (0,-1) -- (0,1);
\node[left] at (-1,0) {\tiny $i$};\node[right] at (1,0) {\tiny $i$};
\node[below] at (0,-1) {\tiny $\bm{A}$};\node[above] at (0,1) 
{\tiny $\bm{A}$};
}$ in \eqref{inhom-weights-eps} and happens with probability $\epsilon r_0 + O(\epsilon^2)$. This is the same type of injection of a particle into the system as in Section \ref{sssec:reduce-qb}, because both lattices share the same $0$-th column; {\bf 2.} A path of colour $i$, situated in the $j$-th column of the lattice, makes a turn into the $(j+1)$-th column. This event corresponds with the vertex $\tikz{0.35}{
\draw[lgray,line width=1pt,->] (-1,0) -- (1,0);
\draw[lgray,line width=2.5pt,->] (0,-1) -- (0,1);
\node[left] at (-1,0) {\tiny $0$};\node[right] at (1,0) {\tiny $i$};
\node[below] at (0,-1) {\tiny $\bm{A}$};\node[above] at (0,1) 
{\tiny $\bm{A}^{-}_i$};
}$ in \eqref{inhom-weights-eps2} and happens with probability $\epsilon r_j (q-1) q^{[\bm{A}]_{i+1}-P}+O(\epsilon^2)$.

\smallskip
The events {\bf 1} and {\bf 2} coincide with the activation events (1) and (2) listed in $\mathfrak{D}_{\rm cPush}(q,P)$. After either event, one then sees the de-activation of the path which has just moved. This happens via combinations of the vertices $\tikz{0.35}{
\draw[lgray,line width=1pt,->] (-1,0) -- (1,0);
\draw[lgray,line width=2.5pt,->] (0,-1) -- (0,1);
\node[left] at (-1,0) {\tiny $i$};\node[right] at (1,0) {\tiny $0$};
\node[below] at (0,-1) {\tiny $\bm{A}$};\node[above] at (0,1) 
{\tiny $\bm{A}^{+}_i$};
}$, 
$\tikz{0.35}{
\draw[lgray,line width=1pt,->] (-1,0) -- (1,0);
\draw[lgray,line width=2.5pt,->] (0,-1) -- (0,1);
\node[left] at (-1,0) {\tiny $j$};\node[right] at (1,0) {\tiny $i$};
\node[below] at (0,-1) {\tiny $\bm{A}$};\node[above] at (0,1) 
{\tiny $\bm{A}^{+-}_{ji}$};
}$
and
$\tikz{0.35}{
\draw[lgray,line width=1pt,->] (-1,0) -- (1,0);
\draw[lgray,line width=2.5pt,->] (0,-1) -- (0,1);
\node[left] at (-1,0) {\tiny $i$};\node[right] at (1,0) {\tiny $i$};
\node[below] at (0,-1) {\tiny $\bm{A}$};\node[above] at (0,1) 
{\tiny $\bm{A}$};
}$ 
in \eqref{inhom-weights-eps2}, each with a finite weight, and leads exactly to the probabilities listed in point (3) of $\mathfrak{D}_{\rm cPush}(q,P)$.

One can check that all other rare events will occur with weight at best $O(\epsilon^2)$. Scaling the vertical axis by $1/\epsilon$ and sending $\epsilon \rightarrow 0$, we deduce the following result:
\begin{prop}
Consider a random configuration inside the quadrant \eqref{quad}, with $0$-th column given by \eqref{inhom-weights-eps} and all other columns given by \eqref{inhom-weights-eps2}. For all $1 \leq i \leq M$, $j \geq 0$, let $c^{\rm quad}_i(j)$ denote the set of colours which pass through the vertical edge $(i,j) \rightarrow (i,j+1)$. Then under the rescaling $j = t/\epsilon$ (of both vertical coordinates and colours) we have the following convergence in finite-dimensional distributions:
\begin{align*}
\lim_{\epsilon \rightarrow 0}
\left\{c^{\rm quad}_i(t/\epsilon)\right\}_{1 \leq i \leq M}
=
\eta^{\rm Push}(t).
\end{align*}
\end{prop}

\subsection{Equally distributed random sets}

This section directly mirrors Section \ref{sssec:equal}; for the sake of precision, let us briefly restate the definitions therein. Fixing a cut-off time $t \in \mathbb{R}_{>0}$ in the coloured $q$-PushTASEP system, we introduce the following random sets:
\begin{itemize}
\item The set of colours \index{J@$\mathcal{J}^{\rm Push}(t)$} $\mathcal{J}^{\rm Push}(t) = \{0 < J_1 < \cdots < J_\ell < t\}$ that have exited the system up to time $t$.

\smallskip

\item The set of times \index{T1@$\mathcal{T}^{\rm Push}(t)$} $\mathcal{T}^{\rm Push}(t) = \{0 < \tau_1 < \cdots < \tau_\ell < t\}$ when the exits happened. We define the complemented set \index{J@$\tilde{\mathcal{J}}^{\rm Push}(t)$} $\tilde{\mathcal{J}}^{\rm Push}(t) = \{0 < \tilde{J}_1 < \cdots < \tilde{J}_\ell < t\}$, where $\tilde{J}_i = t-\tau_{\ell-i+1}$ for all $1 \leq i \leq \ell$. 

\smallskip

\item The set of occupancies \index{O@$\mathcal{O}^{\rm Push}(t)$} $\mathcal{O}^{\rm Push}(t) = \{m_1,\dots,m_M\}$, where $0 \leq m_i \leq P$ is the number of particles (of any colour) situated at position $i$ at time $t$, for all $1 \leq i \leq M$. 
\end{itemize}
\begin{thm}
\label{thm:push-match}
$(\mathcal{O}^{\rm Push}(t),\tilde{\mathcal{J}}^{\rm Push}(t))$ is equidistributed with $(\mathcal{O}^{\rm Push}(t),\mathcal{J}^{\rm Push}(t))$, for any $t \in \mathbb{R}_{>0}$.
\end{thm}

\begin{proof}
This again follows by appropriate degeneration of Theorem \ref{thm:fused2=3}. The above discussion demonstrates how specializations of the rapidity and spin parameters, combined with suitable rescaling of the vertical axis, makes the vertex model \eqref{inhom-weights} in the quadrant converge to the dynamics $\mathfrak{D}_{\rm cPush}(q,P)$. Under the same specializations and rescaling, the $n=1$ version of the model \eqref{inhom-weights} in the quadrant converges to the colour-blind dynamics $\mathfrak{D}_{\rm Push}(q,P)$. The proof then goes through in the same way as the proof of Theorem \ref{thm:12.5}, since one is dealing with the same observables.
\end{proof}

\begin{rmk} Theorem \ref{thm:push-match} remains non-trivial even for $q=\infty$, $P=1$, when the dynamics reduce to the usual PushTASEP. 
\end{rmk}


\begin{appendix}

\chapter{Matching with \cite{Kuan}}
\label{app:kuan-weights}

The purpose of this appendix is to match the weights of the vertex model \eqref{s-weights} to those used by \cite{Kuan}. Quoting from Section 3.5 of that work, one introduces the weights
\begin{align}
\label{S(z)}
\mathcal{S}(z)_{\epsilon_j,\beta}^{\epsilon_k,\delta}
=
\frac{{\bm 1}_{(\epsilon_j + \beta = \epsilon_k + \delta)}}
{\mu^{-1} Q - z}
\cdot
\left\{
\begin{array}{ll}
Q^{2 \beta_{[1,k]}}
(Q\mu-Q^{-2\beta_k}z),
&
\quad
k=j<n+1,
\\
\\
Q\mu^{-1}-Q^{2 \beta_{[1,n]}} z,
&
\quad
k=j=n+1,
\\
\\
-Q^{2 \beta_{[1,k-1]}} Q \mu (1-Q^{2 \beta_k}),
&
\quad
n+1 > k > j,
\\
\\
Q \mu (-Q^{2 \beta_{[1,n]}} + \mu^{-2}),
&
\quad
n+1 = k > j,
\\
\\
-Q^{2\beta_{[1,k-1]}} z (1-Q^{2\beta_k}),
&
\quad
k<j,
\end{array}
\right.
\end{align}
where $k,j \in \{1,\dots,n+1\}$ and $\beta,\delta \in \mathbb{N}^{n}$, with $\epsilon_k$ denoting the $k$-th canonical basis vector of $\mathbb{R}^{n}$ for $1 \leq k \leq n$, while $\epsilon_{n+1}$ is identified with the zero vector of $\mathbb{R}^{n}$. $\mathcal{S}(z)$ defines a stochastic vertex model with (left, bottom) input states 
$(\epsilon_j,\beta)$ and (right, top) output states $(\epsilon_k,\delta)$.

We proceed over the five cases in \eqref{S(z)}, performing the replacements $z \mapsto Qx$, $\mu \mapsto s$ and $Q^2 \mapsto q$. The result is
\begin{align}
\label{S(z)-2}
\mathcal{S}(z)_{\epsilon_j,\beta}^{\epsilon_k,\delta}
\mapsto
\frac{{\bm 1}_{(\epsilon_j + \beta = \epsilon_k + \delta)}}
{1-sx}
\cdot
\left\{
\begin{array}{ll}
s(sq^{\beta_k}-x)q^{\beta_{[1,k-1]}},
&
\quad
k=j<n+1,
\\
\\
1-s x q^{\beta_{[1,n]}},
&
\quad
k=j=n+1,
\\
\\
s^2 (q^{\beta_k}-1) q^{\beta_{[1,k-1]}},
&
\quad
n+1 > k > j,
\\
\\
1-s^2 q^{\beta_{[1,n]}},
&
\quad
n+1 = k > j,
\\
\\
sx(q^{\beta_k}-1) q^{\beta_{[1,k-1]}},
&
\quad
k<j.
\end{array}
\right.
\end{align}
We then have the following matching:
\begin{align}
\label{weight-match}
\mathcal{S}(z)_{\epsilon_j,\beta}^{\epsilon_k,\delta}
\Big|_{z \mapsto Qx,\ \mu \mapsto s,\ Q^2 \mapsto q}
=
\tilde{L}_x(\tilde{\beta},n-j+1; \tilde{\delta},n-k+1)
\end{align}
where $\tilde{\beta} = (\beta_n,\dots,\beta_1)$, $\tilde{\delta} = 
(\delta_n,\dots,\delta_1)$ and $\tilde{L}_x$ is the stochastic version of our $L$-matrix, defined in \eqref{stoch-wt}. The matching \eqref{weight-match} is easily verified by comparing \eqref{S(z)-2} against \eqref{explicit-weights}, and noting that in \cite{Kuan} the labelling of colours is the reverse of the ordering adopted in the present paper (with, in particular, the state $n+1$ in \cite{Kuan} identified with the empty state $0$ in our work).

\chapter{Fusion}
\label{app:fusion}

At times in the paper it is more convenient to state and prove results pertaining to the fundamental vertex model \eqref{fund-vert}, before lifting them to the more general framework of the higher-spin model \eqref{s-weights}. The procedure which allows us to easily translate results from one setting to the other is the technique of {\it fusion}, that goes back to \cite{KulishRS}. Here we recall some basic facts about it, in particular the derivation of the weights \eqref{s-weights} from the fundamental $R$-matrix \eqref{Rmat}.

\section{Row-vertices}

The key combinatorial object in the fusion procedure is the {\it row-vertex}. It is obtained by horizontally concatenating $M$ of the $R$-vertices \eqref{R-vert}, and specializing the spectral parameters to a geometric progression in $q$ with base $y/x$. More specifically, for two fixed integers $j, \ell \in \{0,1,\dots,n\}$ and two $M$-tuples of integers $(i_1,\dots,i_M), (k_1,\dots,k_M) \in \{0,1,\dots,n\}^M$, we define
\begin{multline}
\label{R-row}
\index{R@$R_{y/x}((i_1,\dots,i_M),j; (k_1,\dots,k_M),\ell)$; row-vertices}
R_{y/x}\Big((i_1,\dots,i_M),j; (k_1,\dots,k_M),\ell \Big)
\\
:=
\sum_{c_1 = 0}^{n}
\cdots
\sum_{c_{M-1} = 0}^{n}
R_{q^{M-1} y/x}(i_1,j;k_1,c_1)
R_{q^{M-2} y/x}(i_2,c_1;k_2,c_2)
\dots
R_{y/x}(i_M,c_{M-1};k_M,\ell),
\end{multline}
or graphically,
\begin{align}
\label{R-row-graph}
R_{y/x}\Big((i_1,\dots,i_M),j; (k_1,\dots,k_M),\ell \Big)
= 
\tikz{0.9}{
\draw[lgray,line width=1.5pt,->] (-1,0) -- (6,0);
\draw[lgray,line width=1.5pt,->] (0,-1) -- (0,1);
\draw[lgray,line width=1.5pt,->] (1,-1) -- (1,1);
\draw[lgray,line width=1.5pt,->] (2,-1) -- (2,1);
\draw[lgray,line width=1.5pt,->] (3,-1) -- (3,1);
\draw[lgray,line width=1.5pt,->] (4,-1) -- (4,1);
\draw[lgray,line width=1.5pt,->] (5,-1) -- (5,1);
\draw[densely dotted] (0.5,0) arc (0:90:0.5);
\node[above right] at (-0.1,-0.1) {\tiny $\alpha_1$};
\draw[densely dotted] (1.5,0) arc (0:90:0.5);
\node[above right] at (0.9,-0.1) {\tiny $\alpha_2$};
\draw[densely dotted] (2.5,0) arc (0:90:0.5);
\draw[densely dotted] (3.5,0) arc (0:90:0.5);
\draw[densely dotted] (4.5,0) arc (0:90:0.5);
\draw[densely dotted] (5.5,0) arc (0:90:0.5);
\node[above right] at (4.85,-0.1) {\tiny $\alpha_M$};
\node[left] at (-1,0) {\tiny $j$};\node[right] at (6,0) {\tiny $\ell$};
\node[below] at (0,-1) {\tiny $i_1$};\node[above] at (0,1) {\tiny $k_1$};
\node[below] at (1,-1) {\tiny $i_2$};\node[above] at (1,1) {\tiny $k_2$};
\node[below] at (3,-1) {\tiny $\cdots$};\node[above] at (3,1) {\tiny $\cdots$};
\node[below] at (5,-1) {\tiny $i_M$};\node[above] at (5,1) {\tiny $k_M$};
},
\end{align}
where the angles (=\ spectral parameters) are given by $\alpha_i = q^{M-i} y/x$, and each internal horizontal edge is summed over all possible values in the set $\{0,1,\dots,n\}$. In fact, one can easily see that the row-vertex \eqref{R-row-graph} is either equal to zero or factorized into a product of weights \eqref{fund-vert}. This is because the value of the summation index $c_1$ is constrained uniquely by colour-conservation, given the states $i_1,j,k_1$ surrounding the first vertex, which then constrains uniquely the value of index $c_2$, and so on. If at the $a$-th vertex along the row no $c_a \in \{0,1,\dots,n\}$ exists such that colour-conservation is respected at that vertex, the weight of the entire row-vertex is zero.

\section{$M$-fused vertices}

\begin{defn}
Let $M \geq 1$ and consider an $M$-tuple of nonnegative integers $(i_1,\dots,i_M) \in \{0,1,\dots,n\}^M$. From this we define another vector,
\begin{align*}
\mathcal{C}(i_1,\dots,i_M)
:=
(I_1,\dots,I_n),
\qquad
I_a = \#\{k : i_k = a \},
\quad
\forall\ 1 \leq a \leq n,
\end{align*}
which keeps track of the multiplicity of each colour $1 \leq a \leq n$ within $(i_1,\dots,i_M)$.
\end{defn}

Let us fix two vectors $\I = (I_1,\dots,I_n)$ and $\K = (K_1,\dots,K_n)$ whose components are nonnegative integers, such that $|\I| \leq M$ and $|\K| \leq M$. We define an {\it $M$-fused vertex} as follows:
\begin{align}
\label{M-fus}
\mathcal{L}^{(M)}_{y/x}(\I,j; \K,\ell)
\index{L4@$\mathcal{L}^{(M)}_{y/x}(\I,j; \K,\ell)$; $M$-fused vertices}
:=
\frac{1}{Z_q(M;\I)}
\sum_{\substack{
\mathcal{C}(i_1,\dots,i_M) = \I
\\
\mathcal{C}(k_1,\dots,k_M) = \K
}}
q^{{\rm inv}(i_1,\dots,i_M)}
R_{y/x}\Big((i_1,\dots,i_M),j; (k_1,\dots,k_M),\ell \Big),
\end{align}
where the summation is over all $M$-tuples of integers
$(i_1,\dots,i_M), (k_1,\dots,k_M) \in \{0,1,\dots,n\}^M$ such that $\mathcal{C}(i_1,\dots,i_M) = \I$ and $\mathcal{C}(k_1,\dots,k_M) = \K$. The exponent appearing in the sum is given by
\begin{align*}
{\rm inv}(i_1,\dots,i_M) = \#\{ a<b : i_a > i_b \},
\end{align*}
while the normalization takes the form of a $q$-multinomial coefficient:
\begin{align*}
\index{Z5@$Z_q(M;\I)$}
Z_q(M;\I)
:=
\sum_{\mathcal{C}(i_1,\dots,i_M) = \I}
q^{{\rm inv}(i_1,\dots,i_M)}
=
\frac{(q;q)_M}{(q;q)_{I_0} (q;q)_{I_1} \dots (q;q)_{I_n}},
\quad
I_0 := M - \sum_{a=1}^{n} I_a.
\end{align*}
\begin{prop}[$q$-exchangeability]
\label{prop:q-exch}
Fix an $M$-tuple of integers $(k_1,\dots,k_M) \in \{0,1,\dots,n\}^M$ such that $k_a > k_{a+1}$ for some $1 \leq a \leq M-1$, and two further integers $j,\ell \in \{0,1,\dots,n\}$. Then one has the exchange relation
\begin{multline}
\label{q-exch}
\sum_{i_1=0}^{n} \cdots \sum_{i_M=0}^{n}
q^{{\rm inv}(i_1,\dots,i_M)}
R_{y/x}\Big((i_1,\dots,i_M),j; (k_1,\dots,k_M),\ell \Big)
\\
=
\sum_{i_1=0}^{n} \cdots \sum_{i_M=0}^{n}
q^{{\rm inv}(i_1,\dots,i_M)+1}
R_{y/x}\Big((i_1,\dots,i_M),j; (k_1,\dots,k_{a+1},k_a,\dots,k_M),\ell \Big).
\end{multline}
\end{prop}

\begin{proof}
This property can be proved by a case-by-case analysis of the two vertices involved in the exchange, such as in \cite{CorwinP,BorodinP1}. Here we give another proof based on the Yang--Baxter equation\footnote{We are very grateful to A.~Garbali for showing us this proof.}. First, note the following analogue of the stochastic property \eqref{stoch}:
\begin{align}
\label{rev-stoch}
\sum_{i=0}^{n}
\sum_{j=0}^{n}
q^{{\rm inv}(j,i)}
R_z(i,j;k,\ell)
=
q^{{\rm inv}(k,\ell)},
\qquad
{\rm inv}(a,b)
\equiv
\bm{1}_{a > b},
\end{align}
obtained by summing over the incoming edges of the vertex \eqref{R-vert}, rather than the outgoing states. We are now able to write the left hand side of \eqref{q-exch} as
\begin{multline}
\label{b.1-1}
\sum_{p_a=0}^{n}\ 
\sum_{p_{a+1}=0}^{n}\ 
\sum_{i_1=0}^{n} \cdots \sum_{i_M=0}^{n}
q^{{\rm inv}(i_1,\dots,i_{a-1},p_{a+1},p_{a},i_{a+2},\dots,i_M)}
R_q(p_a,p_{a+1};i_a,i_{a+1})
\\
\times
R_{y/x}\Big((i_1,\dots,i_M),j; (k_1,\dots,k_M),\ell \Big),
\end{multline}
where we have used \eqref{rev-stoch} to introduce an ``extra'' $R$-matrix into \eqref{b.1-1}, with spectral parameter equal to $q$. By virtue of the Yang--Baxter equation \eqref{YB}, we observe the relation 
\begin{multline}
\label{b.1-2}
\sum_{i_a=0}^{n}\ 
\sum_{i_{a+1}=0}^{n}
R_q(p_a,p_{a+1};i_a,i_{a+1})
R_{y/x}\Big((i_1,\dots,i_M),j; (k_1,\dots,k_M),\ell \Big)
\\
=
\sum_{r_a=0}^{n}\ 
\sum_{r_{a+1}=0}^{n}
R^{(a+1,a)}_{y/x}\Big((i_1,\dots,i_{a-1},p_{a+1},p_a,i_{a+2},\dots,i_M),j; 
(k_1,\dots,k_{a-1},r_{a+1},r_a,k_{a+2},\dots,k_M),\ell \Big)
\\
\times
R_q(r_a,r_{a+1};k_a,k_{a+1}),
\end{multline}
where the row-vertex on the right hand side has the order of its $a$-th and $(a+1)$-th spectral parameters reversed. Pictorially, \eqref{b.1-2} is given by
\begin{align*}
\tikz{0.9}{
\draw[lgray,line width=1.5pt,->] (-1,0) -- (6,0);
\draw[lgray,line width=1.5pt,->] (0,-1) -- (0,1);
\draw[lgray,line width=1.5pt,->] (1,-1) -- (1,1);
\draw[lgray,line width=1.5pt,->] (2,-1) -- (2,1);
\draw[lgray,line width=1.5pt,->] (3,-1) -- (3,1);
\draw[lgray,line width=1.5pt,->] (4,-1) -- (4,1);
\draw[lgray,line width=1.5pt,->] (5,-1) -- (5,1);
\draw[lgray,line width=1.5pt] (3,-2) -- (2,-1);
\draw[lgray,line width=1.5pt] (2,-2) -- (3,-1);
\draw[densely dotted] (0.5,0) arc (0:90:0.5);
\draw[densely dotted] (1.5,0) arc (0:90:0.5);
\draw[densely dotted] (2.5,0) arc (0:90:0.5);
\node[above right] at (1.9,-0.1) {\tiny $\alpha_{a}$};
\draw[densely dotted] (3.5,0) arc (0:90:0.5);
\node[above right] at (2.9,-0.1) {\tiny $\alpha_{a+1}$};
\draw[densely dotted] (4.5,0) arc (0:90:0.5);
\draw[densely dotted] (5.5,0) arc (0:90:0.5);
\node[left] at (-1,0) {\tiny $j$};\node[right] at (6,0) {\tiny $\ell$};
\node[below] at (0,-1) {\tiny $i_1$};\node[above] at (0,1) {\tiny $k_1$};
\node[below] at (1,-1) {\tiny $\cdots$};\node[above] at (1,1) {\tiny $\cdots$};
\node[below] at (2,-1) {\tiny $i_a$}; \node[below] at (3,-1) {\tiny\ \ \ $i_{a+1}$};
\node[below] at (2,-2) {\tiny $p_{a+1}$}; \node[below] at (3,-2) {\tiny $p_a$};
\node[above] at (2,1) {\tiny $k_a$}; \node[above] at (3,1) {\tiny $k_{a+1}$};
\node[below] at (4,-1) {\tiny $\cdots$};\node[above] at (4,1) {\tiny $\cdots$};
\node[below] at (5,-1) {\tiny $i_M$};\node[above] at (5,1) {\tiny $k_M$};
}
=
\tikz{0.9}{
\draw[lgray,line width=1.5pt,->] (-1,0) -- (6,0);
\draw[lgray,line width=1.5pt,->] (0,-1) -- (0,1);
\draw[lgray,line width=1.5pt,->] (1,-1) -- (1,1);
\draw[lgray,line width=1.5pt] (2,-1) -- (2,1);
\draw[lgray,line width=1.5pt] (3,-1) -- (3,1);
\draw[lgray,line width=1.5pt,->] (4,-1) -- (4,1);
\draw[lgray,line width=1.5pt,->] (5,-1) -- (5,1);
\draw[lgray,line width=1.5pt,->] (3,1) -- (2,2);
\draw[lgray,line width=1.5pt,->] (2,1) -- (3,2);
\draw[densely dotted] (0.5,0) arc (0:90:0.5);
\draw[densely dotted] (1.5,0) arc (0:90:0.5);
\draw[densely dotted] (2.5,0) arc (0:90:0.5);
\node[above right] at (1.9,-0.1) {\tiny $\alpha_{a+1}$};
\draw[densely dotted] (3.5,0) arc (0:90:0.5);
\node[above right] at (2.9,-0.1) {\tiny $\alpha_a$};
\draw[densely dotted] (4.5,0) arc (0:90:0.5);
\draw[densely dotted] (5.5,0) arc (0:90:0.5);
\node[left] at (-1,0) {\tiny $j$};\node[right] at (6,0) {\tiny $\ell$};
\node[below] at (0,-1) {\tiny $i_1$};\node[above] at (0,1) {\tiny $k_1$};
\node[below] at (1,-1) {\tiny $\cdots$};\node[above] at (1,1) {\tiny $\cdots$};
\node[below] at (2,-1) {\tiny $p_{a+1}$}; \node[below] at (3,-1) {\tiny $p_a$};
\node[above] at (2,1) {\tiny $r_{a+1}$\ \ }; \node[above] at (3,1) {\tiny $r_a$};
\node[above] at (2,2) {\tiny $k_a$}; \node[above] at (3,2) {\tiny $k_{a+1}$};
\node[below] at (4,-1) {\tiny $\cdots$};\node[above] at (4,1) {\tiny $\cdots$};
\node[below] at (5,-1) {\tiny $i_M$};\node[above] at (5,1) {\tiny $k_M$};
}
\end{align*}
where $\alpha_i = q^{M-i} y/x$ for both values $i \in \{a,a+1\}$. Substituting \eqref{b.1-2} into \eqref{b.1-1} and relabelling summation indices $p_{a+1} \rightarrow i_a$, $p_a \rightarrow i_{a+1}$, one finds that the left hand side of \eqref{q-exch} is equal to
\begin{multline}
\label{b.1-3}
\sum_{i_1=0}^{n} \cdots \sum_{i_M=0}^{n}
\
\sum_{r_a = 0}^{n} \ \sum_{r_{a+1} = 0}^{n}
q^{{\rm inv}(i_1,\dots,i_M)}
\\
\times
R^{(a+1,a)}_{y/x}\Big((i_1,\dots,i_M),j; (k_1,\dots,k_{a-1},r_{a+1},r_a,k_{a+2},\dots,k_M),\ell \Big)
R_q(r_a,r_{a+1};k_a,k_{a+1}).
\end{multline}
Now we note another property of the $R$-vertex \eqref{R-vert}, namely that when its spectral parameter is equal to $q$, one has
\begin{align*}
R_q(i,j;k,\ell) =q\cdot R_q(i,j;\ell,k),
\qquad
\forall\
n \geq k > \ell \geq 0,
\end{align*}
which is verified by a quick inspection of the weights \eqref{R-weights-bc}. Making use of this in \eqref{b.1-3} (since $k_a > k_{a+1}$ by assumption), the left hand side of \eqref{q-exch} is now given by
\begin{multline}
\sum_{i_1=0}^{n} \cdots \sum_{i_M=0}^{n}
\
\sum_{r_a = 0}^{n} \ \sum_{r_{a+1} = 0}^{n}
q^{{\rm inv}(i_1,\dots,i_M)+1}
\\
\times
R^{(a+1,a)}_{y/x}\Big((i_1,\dots,i_M),j; (k_1,\dots,k_{a-1},r_{a+1},r_a,k_{a+2},\dots,k_M),\ell \Big)
R_q(r_a,r_{a+1};k_{a+1},k_a).
\end{multline}
The sole purpose of these calculations has been to switch the order of the indices $k_a,k_{a+1}$; this has been achieved, and one can now perform all of the above steps in reverse. It can be seen that the final result will be precisely the right hand side of \eqref{q-exch}.
\end{proof}

A very useful consequence of Proposition \ref{prop:q-exch} is that one does not need to perform the summation over $(k_1,\dots,k_M)$ in order to compute the $M$-fused vertex \eqref{M-fus}:
\begin{cor}
\label{cor:alt-fus}
The $M$-fused vertex \eqref{M-fus} is given, equivalently, by
\begin{align}
\label{M-fus2}
\mathcal{L}^{(M)}_{y/x}(\I,j; \K,\ell)
=
\frac{Z_q(M;\K)}{Z_q(M;\I)}
\sum_{\substack{
\mathcal{C}(i_1,\dots,i_M) = \I
}}
q^{{\rm inv}(i_1,\dots,i_M)}
R_{y/x}\Big((i_1,\dots,i_M),j; (0^{K_0},1^{K_1},\dots,n^{K_n}),\ell \Big),
\end{align}
where $(0^{K_0},1^{K_1},\dots,n^{K_n})$ denotes the unique weakly-increasing $M$-tuple $(k_1,\dots,k_M)$ such that $\mathcal{C}(k_1,\dots,k_M) = \K$, with $K_0 := M - \sum_{a=1}^{n} K_a$.
\end{cor}

\begin{proof}
Using \eqref{q-exch} repeatedly, all terms in the sum over $(k_1,\dots,k_M)$ in \eqref{M-fus} are proportional: the proportionality factor counts the number of exchanges of the type \eqref{q-exch} required to sort $(k_1,\dots,k_M)$ increasingly, which is precisely ${\rm inv}(k_1,\dots,k_M)$. It follows that
\begin{multline*}
\mathcal{L}^{(M)}_{y/x}(\I,j; \K,\ell)
=
\frac{1}{Z_q(M;\I)}
\\
\times
\sum_{\substack{
\mathcal{C}(i_1,\dots,i_M) = \I
\\
\mathcal{C}(k_1,\dots,k_M) = \K
}}
q^{{\rm inv}(i_1,\dots,i_M)+{\rm inv}(k_1,\dots,k_M)}
R_{y/x}\Big((i_1,\dots,i_M),j; (0^{K_0},1^{K_1},\dots,n^{K_n}),\ell \Big),
\end{multline*}
and the sum over $(k_1,\dots,k_M)$ can now be taken to yield \eqref{M-fus2}.

\end{proof}

\section{Stacking $M$-fused vertices}

The $M$-fused vertices have the following important property, which allows them to be {\it vertically} concatenated in a natural way.
\begin{prop}
\label{prop:double-row}
Introduce a double-row analogue of row-vertices, obtained by vertically stacking two copies of \eqref{R-row-graph}:
\begin{multline}
\label{double-row}
R_{y/\{x_1,x_2\}}
\Big( (i_1,\dots,i_M), (j_1,j_2); (k_1,\dots,k_M), (\ell_1,\ell_2) \Big)
\\
:=
\sum_{a_1 = 0}^{n}
\cdots
\sum_{a_M = 0}^{n}
R_{y/x_1}\Big((i_1,\dots,i_M),j_1; (a_1,\dots,a_M),\ell_1 \Big)
R_{y/x_2}\Big((a_1,\dots,a_M),j_2; (k_1,\dots,k_M),\ell_2 \Big)
\end{multline}
or graphically,
\begin{align*}
R_{y/\{x_1,x_2\}}
\Big( (i_1,\dots,i_M), (j_1,j_2); (k_1,\dots,k_M), (\ell_1,\ell_2) \Big)
= 
\tikz{0.9}{
\draw[lgray,line width=1.5pt,->] (-1,0) -- (6,0);
\draw[lgray,line width=1.5pt,->] (-1,1) -- (6,1);
\draw[lgray,line width=1.5pt,->] (0,-1) -- (0,2);
\draw[lgray,line width=1.5pt,->] (1,-1) -- (1,2);
\draw[lgray,line width=1.5pt,->] (2,-1) -- (2,2);
\draw[lgray,line width=1.5pt,->] (3,-1) -- (3,2);
\draw[lgray,line width=1.5pt,->] (4,-1) -- (4,2);
\draw[lgray,line width=1.5pt,->] (5,-1) -- (5,2);
\draw[densely dotted] (0.5,0) arc (0:90:0.5);
\node[above right] at (-0.1,-0.1) {\tiny $\alpha_1$};
\draw[densely dotted] (1.5,0) arc (0:90:0.5);
\node[above right] at (0.9,-0.1) {\tiny $\alpha_2$};
\draw[densely dotted] (2.5,0) arc (0:90:0.5);
\draw[densely dotted] (3.5,0) arc (0:90:0.5);
\draw[densely dotted] (4.5,0) arc (0:90:0.5);
\draw[densely dotted] (5.5,0) arc (0:90:0.5);
\node[above right] at (4.85,-0.1) {\tiny $\alpha_M$};
\draw[densely dotted] (0.5,1) arc (0:90:0.5);
\node[above right] at (-0.1,0.9) {\tiny $\beta_1$};
\draw[densely dotted] (1.5,1) arc (0:90:0.5);
\node[above right] at (0.9,0.9) {\tiny $\beta_2$};
\draw[densely dotted] (2.5,1) arc (0:90:0.5);
\draw[densely dotted] (3.5,1) arc (0:90:0.5);
\draw[densely dotted] (4.5,1) arc (0:90:0.5);
\draw[densely dotted] (5.5,1) arc (0:90:0.5);
\node[above right] at (4.85,0.9) {\tiny $\beta_M$};
\node[left] at (-1,0) {\tiny $j_1$};\node[right] at (6,0) {\tiny $\ell_1$};
\node[left] at (-1,1) {\tiny $j_2$};\node[right] at (6,1) {\tiny $\ell_2$};
\node[below] at (0,-1) {\tiny $i_1$};\node[above] at (0,2) {\tiny $k_1$};
\node[below] at (1,-1) {\tiny $i_2$};\node[above] at (1,2) {\tiny $k_2$};
\node[below] at (3,-1) {\tiny $\cdots$};\node[above] at (3,2) {\tiny $\cdots$};
\node[below] at (5,-1) {\tiny $i_M$};\node[above] at (5,2) {\tiny $k_M$};
}
\end{align*}
with angles given by $\alpha_i = q^{M-i} y/x_1$ and $\beta_i = q^{M-i} y/x_2$, 
$1 \leq i \leq M$. Then one has the relation
\begin{multline}
\label{stacking}
\frac{Z_q(M;\K)}{Z_q(M;\I)}
\sum_{\substack{
\mathcal{C}(i_1,\dots,i_M) = \I
}}
q^{{\rm inv}(i_1,\dots,i_M)}
R_{y/\{x_1,x_2\}}\Big((i_1,\dots,i_M),(j_1,j_2); (0^{K_0},1^{K_1},\dots,n^{K_n}),(\ell_1,\ell_2) \Big)
\\
=
\sum_{\AA = (A_1,\dots,A_n)}
\mathcal{L}^{(M)}_{y/x_1}(\I, j_1; \AA, \ell_1)
\mathcal{L}^{(M)}_{y/x_2}(\AA, j_2; \K, \ell_2),
\end{multline}
where $(0^{K_0},1^{K_1},\dots,n^{K_n})$ is the weakly increasing $M$-tuple described in Corollary \ref{cor:alt-fus}.
\end{prop}

\begin{proof}
Using the definition \eqref{double-row} of the double-row vertex, we can write the left hand side of \eqref{stacking} as
\begin{multline*}
\frac{Z_q(M;\K)}{Z_q(M;\I)}
\sum_{\AA = (A_1,\dots,A_n)}
\sum_{\substack{
\mathcal{C}(i_1,\dots,i_M) = \I
\\
\mathcal{C}(a_1,\dots,a_M) = \AA
}}
q^{{\rm inv}(i_1,\dots,i_M)}
R_{y/x_1}\Big((i_1,\dots,i_M),j_1; (a_1,\dots,a_M),\ell_1 \Big)
\\
\times
R_{y/x_2}\Big((a_1,\dots,a_M),j_2; (0^{K_0},1^{K_1},\dots,n^{K_n}),\ell_2 \Big).
\end{multline*}
Making repeated use of the $q$-exchangeability relation \eqref{q-exch}, we can order the $M$-tuple $(a_1,\dots,a_M)$ appearing in the first row-vertex and rearrange the summation slightly:
\begin{multline*}
\frac{Z_q(M;\K)}{Z_q(M;\I)}
\sum_{\AA = (A_1,\dots,A_n)}
\sum_{\mathcal{C}(i_1,\dots,i_M) = \I}
q^{{\rm inv}(i_1,\dots,i_M)}
R_{y/x_1}\Big((i_1,\dots,i_M),j_1; (0^{A_0},1^{A_1},\dots,n^{A_n}),\ell_1 \Big)
\\
\times
\sum_{\substack{
\mathcal{C}(a_1,\dots,a_M) = \AA
}}
q^{{\rm inv}(a_1,\dots,a_M)}
R_{y/x_2}\Big((a_1,\dots,a_M),j_2; (0^{K_0},1^{K_1},\dots,n^{K_n}),\ell_2 \Big),
\end{multline*}
where $A_0 := M-\sum_{b=1}^{n} A_b$. Taking the sum over $(i_1,\dots,i_M)$ using \eqref{M-fus2}, we obtain
\begin{multline*}
\sum_{\AA = (A_1,\dots,A_n)}
\mathcal{L}^{(M)}_{y/x_1}(\I,j_1;\AA,\ell_1)
\frac{Z_q(M;\K)}{Z_q(M;\AA)}
\\
\times
\sum_{\substack{
\mathcal{C}(a_1,\dots,a_M) = \AA
}}
q^{{\rm inv}(a_1,\dots,a_M)}
R_{y/x_2}\Big((a_1,\dots,a_M),j_2; (0^{K_0},1^{K_1},\dots,n^{K_n}),\ell_2 \Big),
\end{multline*}
and we recover the right hand side of \eqref{stacking} after again applying \eqref{M-fus2} to take the remaining sum over $(a_1,\dots,a_M)$.
\end{proof}

It is clear that Proposition \ref{prop:double-row} admits a generalization to $N$ rows, for arbitrary $N > 2$. We record it below:
\begin{multline}
\label{stack-Nrow}
\frac{Z_q(M;\K)}{Z_q(M;\I)}
\sum_{\substack{
\mathcal{C}(i_1,\dots,i_M) = \I
}}
q^{{\rm inv}(i_1,\dots,i_M)}
R_{y/\{x_1,\dots,x_N\}}\Big((i_1,\dots,i_M),(j_1,\dots,j_N); 
(0^{K_0},1^{K_1},\dots,n^{K_n}),(\ell_1,\dots,\ell_N) \Big)
\\
=
\sum_{\AA_1,\dots,\AA_{N-1}}
\mathcal{L}^{(M)}_{y/x_1}(\I, j_1; \AA_1, \ell_1)
\mathcal{L}^{(M)}_{y/x_2}(\AA_1, j_2; \AA_2, \ell_2)
\dots
\mathcal{L}^{(M)}_{y/x_N}(\AA_{N-1}, j_N; \K, \ell_N),
\end{multline}
where $\AA_1,\dots,\AA_{N-1}$ are summed over vectors in $\mathbb{N}^n$ and $R_{y/\{x_1,\dots,x_N\}}$ denotes the partition function
\begin{multline}
R_{y/\{x_1,\dots,x_N\}}\Big((i_1,\dots,i_M),(j_1,\dots,j_N); 
(k_1,\dots,k_M),(\ell_1,\dots,\ell_N) \Big)
\\
=
\tikz{0.8}{
\foreach\y in {2,...,5}{
\draw[lgray,line width=1.5pt,->] (1,\y) -- (8,\y);
}
\foreach\x in {2,...,7}{
\draw[lgray,line width=1.5pt,->] (\x,1) -- (\x,6);
}
\foreach\x in {2,...,5}{
\foreach\y in {2,...,7}{
\draw[densely dotted] (\y+0.5,\x) arc (0:90:0.5);
}}
\node[above] at (7,6) {$\tiny k_M$};
\node[above] at (5,6) {$\tiny \cdots$};
\node[above] at (4,6) {$\tiny \cdots$};
\node[above] at (2,6) {$\tiny k_1$};
\node[left] at (1,2) {$\tiny j_1$};
\node[left] at (0.5,3) {$\tiny \vdots$};
\node[left] at (0.5,4) {$\tiny \vdots$};
\node[left] at (1,5) {$\tiny j_N$};
\node[below] at (7,1) {$\tiny i_M$};
\node[below] at (5,1) {$\tiny \cdots$};
\node[below] at (4,1) {$\tiny \cdots$};
\node[below] at (2,1) {$\tiny i_1$};
\node[right] at (8,2) {$\tiny \ell_1$};
\node[right] at (8,3) {$\tiny \vdots$};
\node[right] at (8,4) {$\tiny \vdots$};
\node[right] at (8,5) {$\tiny \ell_N$};
}
\end{multline}
with the spectral parameter of the vertex $(a,b)$ (both indices counted from the bottom-left corner) equal to $q^{M-a}y/x_b$, $1 \leq a \leq M$, $1 \leq b \leq N$. The proof of \eqref{stack-Nrow} proceeds along the same lines as the proof of Proposition \ref{prop:double-row}; we omit it. It is a central result for our purposes, since it allows us to freely transition between partition functions in the fundamental vertex model \eqref{fund-vert} and those in the higher-spin setting \eqref{s-weights}, provided that appropriate boundary conditions and choices of vertical rapidities are imposed in the former.

\section{Evaluation of $M$-fused vertices and analytic continuation}

\begin{thm}
Fix two vectors $\I, \K \in \mathbb{N}^n$ such that $|\I| \leq M$, $|\K| \leq M$, and two integers $j,\ell \in \{0,1,\dots,n\}$. The $M$-fused vertex \eqref{M-fus} has the following explicit evaluation:
\begin{align}
\label{fus-wt}
\mathcal{L}^{(M)}_{y/x}(\I,j;\K,\ell)
=
\bm{1}_{(\I + \bm{e}_j = \K + \bm{e}_{\ell})}
\cdot
\frac{1}
{1-q^M y/x}
\cdot
\left\{
\begin{array}{ll}
(1-q^{I_\ell} y/x) q^{\Is{\ell+1}{n}}, & \quad j = \ell,
\\ \\
(1-q^{I_\ell}) q^{\Is{\ell+1}{n}}, & \quad j < \ell,
\\ \\
y/x (1-q^{I_\ell}) q^{\Is{\ell+1}{n}}, & \quad j > \ell,
\end{array}
\right.
\end{align}
or more compactly,
\begin{align}
\label{fus-compact}
\mathcal{L}^{(M)}_{y/x}(\I,j;\K,\ell)
=
\bm{1}_{(\I + \bm{e}_j = \K + \bm{e}_{\ell})}
\cdot
(y/x)^{\bm{1}_{j >\ell}}
\cdot
\frac{1-q^{I_{\ell}} (y/x)^{\bm{1}_{j=\ell}}}{1-q^M y/x}
\cdot
q^{\Is{\ell+1}{n}},
\end{align}
where by agreement $I_0 = M - \sum_{i=1}^{n} I_i$.
\end{thm}

\begin{proof}
One simple way of proving this is to write a recursion for $\mathcal{L}^{(M)}_{y/x}(\I,j;\K,\ell)$ in $M$, and then show that the weights \eqref{fus-wt} satisfy it. Taking the definition \eqref{M-fus} of the $M$-fused vertex, we decompose the row-vertex that appears into two pieces:
\begin{align*}
R_{y/x}\Big((i_1,\dots,i_M),j; (k_1,\dots,k_M),\ell \Big)
=
\sum_{c=0}^{n}
R_{qy/x}\Big( (i_1,\dots,i_{M-1}),j ; (k_1,\dots,k_{M-1}),c \Big)
R_{y/x}(i_M,c ; k_M,\ell).
\end{align*}
Substituting this into \eqref{M-fus}, after some rearrangement of the summation we find that
\begin{multline*}
Z_q(M;\I)
\mathcal{L}^{(M)}_{y/x}(\I,j; \K,\ell)
=
\sum_{a,b,c=0}^{n}
\sum_{\substack{
\mathcal{C}(i_1,\dots,i_{M-1}) = \I_{a}^{-}
\\
\mathcal{C}(k_1,\dots,k_{M-1}) = \K_{b}^{-}
}}
q^{{\rm inv}(i_1,\dots,i_{M-1})}
q^{\Is{a+1}{n}}
\\
\times
R_{qy/x}\Big( (i_1,\dots,i_{M-1}),j ; (k_1,\dots,k_{M-1}),c \Big)
R_{y/x}(a,c ; b,\ell).
\end{multline*}
Now applying the definition \eqref{M-fus} in the case of an $(M-1)$-fused vertex, we obtain the recursion
\begin{align*}
\nonumber
Z_q(M;\I)
\mathcal{L}^{(M)}_{y/x}(\I,j; \K,\ell)
=
\sum_{a,b,c=0}^{n}
Z_q(M-1;\I_a^{-})
q^{\Is{a+1}{n}}
\mathcal{L}^{(M-1)}_{qy/x}(\I_a^{-},j; \K_b^{-},c)
R_{y/x}(a,c ; b,\ell),
\end{align*}
and cancelling factors common to the normalizations $Z_q(M;\I)$ and $Z_q(M-1;\I_a^{-})$, we have
\begin{align}
\label{fus-rec}
(1-q^M)
\mathcal{L}^{(M)}_{y/x}(\I,j; \K,\ell)
&=
\sum_{a,b,c=0}^{n}
(1-q^{I_a})
q^{\Is{a+1}{n}}
\mathcal{L}^{(M-1)}_{qy/x}(\I_a^{-},j; \K_b^{-},c)
R_{y/x}(a,c ; b,\ell).
\end{align}
The number of summations appearing in \eqref{fus-rec} can be reduced by considering explicitly the sums over $a$ and $b$. In particular, we note that when $a=b$ the only surviving terms will have $c=\ell$ (by colour-conservation of the $R$-vertex). Similarly, for $a\not=b$, colour-conservation through the $R$-vertex imposes that $a=\ell$ and $b=c$. This leads us to a simpler version of the recursion:
\begin{align}
\label{fus-rec2}
(1-q^M)
\mathcal{L}^{(M)}_{y/x}(\I,j; \K,\ell)
&=
\sum_{a=0}^{n}
(1-q^{I_a})
q^{\Is{a+1}{n}}
\mathcal{L}^{(M-1)}_{qy/x}(\I_a^{-},j; \K_a^{-},\ell)
R_{y/x}(a,\ell;a,\ell)
\\
\nonumber
&+
\sum_{\substack{c=0 \\ c\not=\ell}}^{n}
(1-q^{I_{\ell}})
q^{\Is{\ell+1}{n}}
\mathcal{L}^{(M-1)}_{qy/x}(\I_{\ell}^{-},j; \K_c^{-},c)
R_{y/x}(\ell,c;c,\ell).
\end{align}
It is now easy, by case-by-case analysis of the possible values of $j$ and $\ell$, to check that the weights \eqref{fus-wt} obey the required relation \eqref{fus-rec2}.

Finally, one needs to check the base case $M=1$. In this situation, one necessarily has $\I = \bm{e}_i$ and $\K = \bm{e}_k$ for some $i,k \in \{0,1,\dots,n\}$. Using the expression \eqref{M-fus2} for the $1$-fused vertex, it is clear that
\begin{align*}
\mathcal{L}^{(1)}_{y/x}(\bm{e}_i,j;\bm{e}_k,\ell)
=
R_{y/x}(i,j;k,\ell),
\end{align*}
and this is consistent with equation \eqref{fus-compact} at $M=1$, $\I = \bm{e}_i$ and $\K = \bm{e}_k$.
\end{proof}

\begin{prop}
\label{prop:rll-fus}
The weights \eqref{fus-wt} provide a solution of the intertwining equation
\begin{multline}
\label{prop-b.5}
\sum_{0 \leq k_1,k_2 \leq n}
\
\sum_{\K \in \mathbb{N}^n}
R_{y/x}(i_2,i_1;k_2,k_1)
\mathcal{L}^{(M)}_{z/x}(\I,k_1;\K,j_1)
\mathcal{L}^{(M)}_{z/y}(\K,k_2;\J,j_2)
\\
=
\sum_{0 \leq k_1,k_2 \leq n}
\
\sum_{\K \in \mathbb{N}^n}
\mathcal{L}^{(M)}_{z/y}(\I,i_2;\K,k_2)
\mathcal{L}^{(M)}_{z/x}(\K,i_1;\J,k_1)
R_{y/x}(k_2,k_1;j_2,j_1).
\end{multline}
\end{prop}

\begin{proof}
This is a simple corollary of Proposition \ref{prop:double-row}. Using \eqref{stacking}, we see that both of the sums $\sum_{\K \in \mathbb{N}^n}
\mathcal{L}^{(M)}_{z/x}(\I,k_1;\K,j_1)
\mathcal{L}^{(M)}_{z/y}(\K,k_2;\J,j_2)$ 
and
$\sum_{\K \in \mathbb{N}^n}
\mathcal{L}^{(M)}_{z/y}(\I,i_2;\K,k_2)
\mathcal{L}^{(M)}_{z/x}(\K,i_1;\J,k_1)$ 
can be replaced by double-row vertices of the form \eqref{double-row}. The statement then follows by $M$ applications of the Yang--Baxter equation \eqref{YB} in the fundamental vertex model.
\end{proof}

As can be seen from \eqref{fus-wt}, the $M$-fused vertices depend on $M$ only via the combination $q^M$. Analytically continuing in this variable by sending $q^M \mapsto s^{-2}$, and setting $y = s$, after rearrangement one finds that the weights become
\begin{align}
\label{explicit-weights}
\mathcal{L}^{(M)}_{y/x}(\I,j;\K,\ell)
\Big|_{q^M \rightarrow s^{-2}, y \rightarrow s}
=
\frac{\bm{1}_{(\I + \bm{e}_j = \K + \bm{e}_{\ell})}}
{1-s x}
\cdot
\left\{
\begin{array}{ll}
1-s q^{\Is{1}{n}} x, & \quad j = \ell = 0,
\\ \\
-s(x-s q^{I_\ell}) q^{\Is{\ell+1}{n}}, & \quad j = \ell \geq 1,
\\ \\
-sx(1-q^{I_\ell}) q^{\Is{\ell+1}{n}}, & \quad j < \ell,
\\ \\
1-s^2 q^{\Is{1}{n}}, & \quad j > \ell = 0,
\\ \\
-s^2(1-q^{I_\ell}) q^{\Is{\ell+1}{n}}, & \quad j > \ell \geq 1.
\end{array}
\right.
\end{align}
One can now check that
\begin{align}
\label{an-cont}
\mathcal{L}^{(M)}_{y/x}(\I,j;\K,\ell)
\Big|_{q^M \rightarrow s^{-2}, y \rightarrow s}
=
\tilde{L}_x(\I,j;\K,\ell),
\end{align}
where $\tilde{L}_x(\I,j;\K,\ell)$ is the stochastic $L$-matrix as defined in \eqref{stoch-wt}, establishing a direct link between the higher-spin model studied throughout the paper and fusion of the fundamental vertex model \eqref{fund-vert}. Note that, as a consequence of Proposition \ref{prop:rll-fus}, one has the following fact:
\begin{cor}
The stochastic $L$-matrix \eqref{stoch-wt} obeys the intertwining equation
\begin{multline}
\label{cor-b.6}
\sum_{0 \leq k_1,k_2 \leq n}
\
\sum_{\K \in \mathbb{N}^n}
R_{y/x}(i_2,i_1;k_2,k_1)
\tilde{L}_x(\I,k_1;\K,j_1)
\tilde{L}_y(\K,k_2;\J,j_2)
\\
=
\sum_{0 \leq k_1,k_2 \leq n}
\
\sum_{\K \in \mathbb{N}^n}
\tilde{L}_y(\I,i_2;\K,k_2)
\tilde{L}_x(\K,i_1;\J,k_1)
R_{y/x}(k_2,k_1;j_2,j_1).
\end{multline}
\end{cor}

\begin{proof}
Equation \eqref{cor-b.6} is true at each of the points $s = q^{-M/2}$, where $M$ is any sufficiently large positive integer, since at such values it reduces to \eqref{prop-b.5} at $y=s$. Since both sides of \eqref{cor-b.6} are rational in $s$, equality at infinitely many values of $s$ proves the equation for generic $s$.
\end{proof}

\chapter{Three intertwining equations}
\label{app:RLL}

The purpose of this appendix is to explain how the three intertwining equations \eqref{RLLa}--\eqref{RLLc} arise as special cases of a single Yang--Baxter equation. We will quote a number of results within it, but in most cases we omit proofs in view of their lengthy nature.

\section{Bosnjak--Mangazeev solution of the Yang--Baxter equation}

Throughout this section, $n$ denotes the rank of the quantized affine algebra $U_q(\wh{\mathfrak{sl}_{n+1}})$, $\l,\m,\n$ are positive integers which denote the spin of a line, and $x,y,z$ are rapidities which are associated to lines. We consider vertices of the following type:
\begin{align*}
\tikz{0.9}{
\draw[lgray,line width=4pt,->] (-1,0) -- (1,0);
\draw[lgray,line width=4pt,->] (0,-1) -- (0,1);
\node[left] at (-1,0) {\tiny $\BB$};\node[right] at (1,0) {\tiny $\DD$};
\node[below] at (0,-1) {\tiny $\AA$};\node[above] at (0,1) {\tiny $\CC$};
\node[left] at (-1.5,0) {$(x,\l) \rightarrow$};
\node[below] at (0,-1.4) {$\uparrow$};
\node[below] at (0,-1.9) {$(y,\m)$};
}
=
\vert{\l}{\m}{\frac{x}{y};q}{\AA}{\BB}{\CC}{\DD}
\equiv
\index{W@$W_{\l,\m}(x/y;q; \AA,\BB,\CC,\DD)$; Bosnjak--Mangazeev weights}
W_{\l,\m}(x/y;q; \AA,\BB,\CC,\DD),
\end{align*}
where $\AA,\BB,\CC,\DD$ are compositions in $\mathbb{N}^n$, whose weights are constrained by
\begin{align}
\label{wt-constrain}
|\AA|, |\CC| \leq \m,
\qquad
|\BB|, |\DD| \leq \l.
\end{align}
If the conditions \eqref{wt-constrain} are not met, the vertex weight $W_{\l,\m}(x/y;q; \AA,\BB,\CC,\DD)$ is identically zero. The Yang--Baxter equation takes the form
\begin{multline}
\label{master-yb}
\sum_{\CC_1,\CC_2,\CC_3}
\vert{\l}{\m}{\frac{x}{y};q}{\AA_2}{\AA_1}{\CC_2}{\CC_1}
\vert{\l}{\n}{\frac{x}{z};q}{\AA_3}{\CC_1}{\CC_3}{\BB_1}
\vert{\m}{\n}{\frac{y}{z};q}{\CC_3}{\CC_2}{\BB_3}{\BB_2}
\\
=
\sum_{\CC_1,\CC_2,\CC_3}
\vert{\m}{\n}{\frac{y}{z};q}{\AA_3}{\AA_2}{\CC_3}{\CC_2}
\vert{\l}{\n}{\frac{x}{z};q}{\CC_3}{\AA_1}{\BB_3}{\CC_1}
\vert{\l}{\m}{\frac{x}{y};q}{\CC_2}{\CC_1}{\BB_2}{\BB_1},
\end{multline}
where $\AA_1,\AA_2,\AA_3,\BB_1,\BB_2,\BB_3 \in \mathbb{N}^n$ are fixed compositions, and the summation on both sides of the equation is over triples of compositions $\CC_1,\CC_2,\CC_3 \in \mathbb{N}^n$. Equivalently, we can represent the equation \eqref{master-yb} graphically:
\begin{align}
\label{graph-master}
\sum_{\CC_1,\CC_2,\CC_3}
\tikz{0.9}{
\draw[lgray,line width=4pt,->]
(-2,1) node[above,scale=0.6] {\color{black} $\AA_1$} -- (-1,0) node[below,scale=0.6] {\color{black} $\CC_1$} -- (1,0) node[right,scale=0.6] {\color{black} $\BB_1$};
\draw[lgray,line width=4pt,->] 
(-2,0) node[below,scale=0.6] {\color{black} $\AA_2$} -- (-1,1) node[above,scale=0.6] {\color{black} $\CC_2$} -- (1,1) node[right,scale=0.6] {\color{black} $\BB_2$};
\draw[lgray,line width=4pt,->] 
(0,-1) node[below,scale=0.6] {\color{black} $\AA_3$} -- (0,0.5) node[scale=0.6] {\color{black} $\CC_3$} -- (0,2) node[above,scale=0.6] {\color{black} $\BB_3$};
\node[left] at (-2.2,1) {$(x,\l) \rightarrow$};
\node[left] at (-2.2,0) {$(y,\m) \rightarrow$};
\node[below] at (0,-1.4) {$\uparrow$};
\node[below] at (0,-1.9) {$(z,\n)$};
}
\quad
=
\quad
\sum_{\CC_1,\CC_2,\CC_3}
\tikz{0.9}{
\draw[lgray,line width=4pt,->] 
(-1,1) node[left,scale=0.6] {\color{black} $\AA_1$} -- (1,1) node[above,scale=0.6] {\color{black} $\CC_1$} -- (2,0) node[below,scale=0.6] {\color{black} $\BB_1$};
\draw[lgray,line width=4pt,->] 
(-1,0) node[left,scale=0.6] {\color{black} $\AA_2$} -- (1,0) node[below,scale=0.6] {\color{black} $\CC_2$} -- (2,1) node[above,scale=0.6] {\color{black} $\BB_2$};
\draw[lgray,line width=4pt,->] 
(0,-1) node[below,scale=0.6] {\color{black} $\AA_3$} -- (0,0.5) node[scale=0.6] {\color{black} $\CC_3$} -- (0,2) node[above,scale=0.6] {\color{black} $\BB_3$};
\node[left] at (-1.5,1) {$(x,\l) \rightarrow$};
\node[left] at (-1.5,0) {$(y,\m) \rightarrow$};
\node[below] at (0,-1.4) {$\uparrow$};
\node[below] at (0,-1.9) {$(z,\n)$};
}
\end{align}

\begin{thm}[Bosnjak--Mangazeev, \cite{BosnjakM}]
\label{thm:master-yb}
For any two compositions $\lambda, \mu \in \mathbb{N}^n$ such that $\lambda_i \leq \mu_i$ for all $1\leq i \leq n$, we define the function
\begin{align*}
\Phi(\lambda,\mu;x,y)
:=
\frac{(x;q)_{|\lambda|} (y/x;q)_{|\mu-\lambda|}}{(y;q)_{|\mu|}}
(y/x)^{|\lambda|}
\left(q^{\sum_{i<j} (\mu_i-\lambda_i) \lambda_j} \right)
\prod_{i=1}^{n} \binom{\mu_i}{\lambda_i}_q.
\end{align*}
Let $\AA = (A_1,\dots,A_n)$, $\BB = (B_1,\dots,B_n)$, $\CC = (C_1,\dots,C_n)$ and 
$\DD = (D_1,\dots,D_n)$ be compositions and fix two positive integers $\l,\m$. An explicit solution of the Yang--Baxter equation \eqref{master-yb} is obtained by choosing
\begin{multline}
\label{bm-weights}
\vert{\l}{\m}{x;q}{\AA}{\BB}{\CC}{\DD}
=
\left( \bm{1}_{\AA + \BB = \CC + \DD} \right)
x^{|\DD-\BB|} q^{|\AA \l - \DD \m|}
\\
\times
\sum_{\bm{P}}
\Phi(\CC-\bm{P},\CC+\DD-\bm{P}; q^{\l-\m} x, q^{-\m} x)
\Phi(\bm{P},\BB; q^{-\l}/x, q^{-\l}),
\end{multline}
where the sum is over compositions $\bm{P} = (P_1,\dots,P_n)$ such that $0 \leq P_i \leq \min(B_i,C_i)$ for all $1 \leq i \leq n$.
\end{thm}

We refer the reader to \cite{BosnjakM} for a proof of Theorem \ref{thm:master-yb}, using techniques from three-dimensional integrability.

\begin{prop}[Stochasticity]\label{prop:bm-stoch}
Let $\AA = (A_1,\dots,A_n)$ and $\BB = (B_1,\dots,B_n)$ be two fixed compositions in 
$\mathbb{N}^n$. The following identity holds for all such $\AA,\BB$:
\begin{align}
\label{sum-to-1}
\sum_{\CC, \DD}
\vert{\l}{\m}{x;q}{\AA}{\BB}{\CC}{\DD}
=
1,
\end{align}
where the sum is over all compositions $\CC = (C_1,\dots,C_n)$ and $\DD = (D_1,\dots,D_n)$ in $\mathbb{N}^n$.
\end{prop}

\begin{prop}[Symmetry of weights]\label{prop:bm-sym}
The vertex weights \eqref{bm-weights} satisfy the following symmetry relation:
\begin{align}
\label{bm-sym}
\vert{\l}{\m}{x;q}{\AA}{\BB}{\CC}{\DD}
=
\vert{\m}{\l}{\frac{q^{\m-\l}}{x};\frac{1}{q}}{\BB}{\AA}{\DD}{\CC}.
\end{align}
\end{prop}

Propositions \ref{prop:bm-stoch} and \ref{prop:bm-sym} can both be proved directly from the explicit form \eqref{bm-weights} of the weights. These proofs are not difficult but fairly laborious; we omit them. The sum-to-unity property \eqref{sum-to-1} allows one to construct a wide variety of stochastic processes as limiting cases of the model \eqref{bm-weights}. The symmetry property \eqref{bm-sym} is necessary to allow the three relations \eqref{RLLa}--\eqref{RLLc} to be deduced from the single equation \eqref{master-yb}, as we will show below.

\begin{prop}[Reduction to fundamental $R$-matrix]
\label{prop:reduction}
Taking $\l = \m = 1$ in \eqref{bm-weights}, one has
\begin{align*}
\vert{1}{1}{\frac{x}{y};q}{\AA}{\BB}{\CC}{\DD}
=
\left\{
\begin{array}{ll}
R_{y/x}(\AA^*,\BB^*;\CC^*,\DD^*), 
& \quad
|\AA|, |\BB|, |\CC|, |\DD| \leq 1,
\\ \\
0,
& \quad
{\rm otherwise},
\end{array}
\right.
\end{align*}
where $R_{y/x}$ denotes the fundamental $R$-matrix as given by \eqref{R-weights-a}, \eqref{R-weights-bc}, and where we have defined 
\begin{align*}
\I^{*}
=
\left\{
\begin{array}{ll}
0, & \quad \I = \bm{0},
\\
i, & \quad \I = \bm{e}_i,
\end{array}
\right.
\end{align*}
for any composition $\I = (I_1,\dots,I_n)$ such that $|\I| \leq 1$.
\end{prop}

\begin{proof}
Taking $\l = \m = 1$ means that the compositions $\AA,\BB,\CC,\DD$ are constrained, by \eqref{wt-constrain}, to be equal to $\bm{e}_0 \equiv \bm{0}$ or $\bm{e}_i$ for some $1 \leq i \leq n$. The conservation requirement $\AA + \BB = \CC + \DD$ then imposes that the only non-vanishing cases are (i) $\AA = \BB = \CC = \DD = \bm{e}_j$, $0 \leq j \leq n$, (ii) $\AA = \CC = \bm{e}_i$ and $\BB = \DD = \bm{e}_j$ with $i \not= j$, or (iii) $\AA = \DD = \bm{e}_i$ and $\BB = \CC = \bm{e}_j$ with $i \not= j$. One can then check these cases explicitly, with the weights \eqref{R-weights-a}, \eqref{R-weights-bc} being recovered.
\end{proof}

\section{First intertwining equation, \eqref{RLLa}}

Consider equation \eqref{master-yb} for the weights \eqref{bm-weights}. Taking the specialization $\l = \m = 1$, while keeping $\n$ general, we obtain the relation
\begin{multline}
\label{RLL-A}
\sum_{c_1,c_2 = 0}^{n}\
\sum_{\CC \in \mathbb{N}^n}
R_{y/x}(a_2,a_1;c_2,c_1)
\vert{1}{\n}{\frac{x}{z};q}{\AA}{\bm{e}_{c_1}}{\CC}{\bm{e}_{b_1}}
\vert{1}{\n}{\frac{y}{z};q}{\CC}{\bm{e}_{c_2}}{\BB}{\bm{e}_{b_2}}
\\
=
\sum_{c_1,c_2 = 0}^{n}\ 
\sum_{\CC \in \mathbb{N}^n}
\vert{1}{\n}{\frac{y}{z};q}{\AA}{\bm{e}_{a_2}}{\CC}{\bm{e}_{c_2}}
\vert{1}{\n}{\frac{x}{z};q}{\CC}{\bm{e}_{a_1}}{\BB}{\bm{e}_{c_1}}
R_{y/x}(c_2,c_1;b_2,b_1),
\end{multline}
where $a_1,a_2,b_1,b_2 \in \{0,1,\dots,n\}$ and $\AA,\BB \in \mathbb{N}^n$ are fixed to arbitrary values on both sides of \eqref{RLL-A}, and where we have used Proposition \ref{prop:reduction} to replace $W_{\l,\m}$ by the fundamental $R$-matrix $R_{y/x}$. As can be seen from their explicit form \eqref{bm-weights}, the weights $W_{1,\n}$ in \eqref{RLL-A} depend on $\n$ only through $q^{\n}$. After normalizing both sides of \eqref{RLL-A} appropriately, it becomes a polynomial identity in $q^{\n}$, true at infinitely many points $\n \in \{1,2,3,\dots\}$. It is therefore true for arbitrary complex values of $q^{\n}$; with this in mind we make the substitution $q^{\n} \mapsto s^{-2} \in \mathbb{C}$ and define
\begin{align}
\label{def-Ltilde}
\tilde{L}_{x}(\AA,b;\CC,d)
:=
\left.
\vert{1}{\n}{\frac{x}{s};q}{\AA}{\bm{e}_b}{\CC}{\bm{e}_d}
\right|_{q^{\n} \rightarrow s^{-2}}
\end{align}
for all $\AA,\CC \in \mathbb{N}^n$ and $b,d \in \{0,1,\dots,n\}$.
\begin{prop}
\label{prop:c.5}
The function $\tilde{L}_{x}$ as defined in \eqref{def-Ltilde} is also given by \eqref{generic-L}, \eqref{s-weights} and \eqref{stoch-wt}.
\end{prop}

\begin{proof}
This can be checked by a straightforward calculation with the use of Theorem \ref{thm:master-yb}.
\end{proof}

Making the substitution \eqref{def-Ltilde} for both of the $W_{1,\n}$ weights in \eqref{RLL-A}, we obtain the relation
\begin{multline}
\label{RLL-A2}
\sum_{c_1,c_2 = 0}^{n}\
\sum_{\CC \in \mathbb{N}^n}
R_{y/x}(a_2,a_1;c_2,c_1)
\tilde{L}_x(\AA,c_1;\CC,b_1)
\tilde{L}_y(\CC,c_2;\BB,b_2)
\\
=
\sum_{c_1,c_2 = 0}^{n}\
\sum_{\CC \in \mathbb{N}^n}
\tilde{L}_y(\AA,a_2;\CC,c_2)
\tilde{L}_x(\CC,a_1;\BB,c_1)
R_{y/x}(c_2,c_1;b_2,b_1),
\end{multline}
which is just the intertwining equation \eqref{RLLa}, but for the stochastic weights \eqref{stoch-wt}.

\section{Second intertwining equation, \eqref{RLLb}}

Taking the specialization $\l = \n = 1$ of \eqref{master-yb}, while keeping $\m$ general, we obtain
\begin{multline}
\label{RLL-B}
\sum_{c_1,c_3 = 0}^{n}\
\sum_{\CC \in \mathbb{N}^n}
\vert{1}{\m}{\frac{x}{y};q}{\AA}{\bm{e}_{a_1}}{\CC}{\bm{e}_{c_1}}
R_{z/x}(a_3,c_1;c_3,b_1)
\vert{\m}{1}{\frac{y}{z};q}{\bm{e}_{c_3}}{\CC}{\bm{e}_{b_3}}{\BB}
\\
=
\sum_{c_1,c_3 = 0}^{n}\
\sum_{\CC \in \mathbb{N}^n}
\vert{\m}{1}{\frac{y}{z};q}{\bm{e}_{a_3}}{\AA}{\bm{e}_{c_3}}{\CC}
R_{z/x}(c_3,a_1;b_3,c_1)
\vert{1}{\m}{\frac{x}{y};q}{\CC}{\bm{e}_{c_1}}{\BB}{\bm{e}_{b_1}},
\end{multline}
where $a_1,a_3,b_1,b_3 \in \{0,1,\dots,n\}$ and $\AA,\BB \in \mathbb{N}^n$ are fixed to arbitrary values on both sides of the equation. Now we may apply the symmetry \eqref{bm-sym} to \eqref{RLL-B}, performing in addition the change of variables $z \mapsto q^{-1} z^{-1} \equiv \b{q}\b{z}$. This gives rise to the relation
\begin{multline}
\label{RLL-B2}
\sum_{c_1,c_3 = 0}^{n}\
\sum_{\CC \in \mathbb{N}^n}
\vert{1}{\m}{\frac{x}{y};q}{\AA}{\bm{e}_{a_1}}{\CC}{\bm{e}_{c_1}}
R_{\b{q}\b{x}\b{z}}(a_3,c_1;c_3,b_1)
\vert{1}{\m}{\frac{q^{-\m}}{yz};\frac{1}{q}}{\CC}{\bm{e}_{c_3}}{\BB}{\bm{e}_{b_3}}
\\
=
\sum_{c_1,c_3 = 0}^{n}\
\sum_{\CC \in \mathbb{N}^n}
\vert{1}{\m}{\frac{q^{-\m}}{yz};\frac{1}{q}}{\AA}{\bm{e}_{a_3}}{\CC}{\bm{e}_{c_3}}
R_{\b{q}\b{x}\b{z}}(c_3,a_1;b_3,c_1)
\vert{1}{\m}{\frac{x}{y};q}{\CC}{\bm{e}_{c_1}}{\BB}{\bm{e}_{b_1}}.
\end{multline}
In a similar vein to \eqref{def-Ltilde}, we define
\begin{align}
\label{def-Mtilde}
\tilde{M}_{z}(\AA,b;\CC,d)
:=
\left.
\vert{1}{\m}{\frac{s}{z};\frac{1}{q}}{\AA}{\bm{e}_b}{\CC}{\bm{e}_d}
\right|_{q^{\m} \rightarrow s^{-2}}
\end{align}
for all $\AA,\CC \in \mathbb{N}^n$ and $b,d \in \{0,1,\dots,n\}$.

\begin{prop}
The function $\tilde{M}_z$ as defined in \eqref{def-Mtilde} is equivalent to that given by \eqref{generic-M}, \eqref{dual-s-weights} and \eqref{stoch-wt}.
\end{prop}

\begin{proof}
As in the case of Proposition \ref{prop:c.5}, one checks this directly via \eqref{bm-weights}.
\end{proof}

Substituting $y \mapsto s$ in \eqref{RLL-B2} and sending $q^{\m} \mapsto s^{-2}$, one recovers
\begin{multline}
\label{RLL-B3}
\sum_{c_1,c_3 = 0}^{n}\
\sum_{\CC \in \mathbb{N}^n}
\tilde{L}_x(\AA,a_1;\CC,c_1)
R_{\b{q}\b{x}\b{z}}(a_3,c_1;c_3,b_1)
\tilde{M}_z(\CC,c_3;\BB,b_3)
\\
=
\sum_{c_1,c_3 = 0}^{n}\
\sum_{\CC \in \mathbb{N}^n}
\tilde{M}_z(\AA,a_3;\CC,c_3)
R_{\b{q}\b{x}\b{z}}(c_3,a_1;b_3,c_1)
\tilde{L}_x(\CC,c_1;\BB,b_1),
\end{multline}
which can be recognised as the intertwining equation \eqref{RLLb}, but for the stochastic weights \eqref{stoch-wt}.

\section{Third intertwining equation, \eqref{RLLc}}

Finally let us take the specialization $\m= \n = 1$ of \eqref{master-yb}, keeping $\l$ general. This yields
\begin{multline*}
\sum_{c_2,c_3 = 0}^{n}\
\sum_{\CC \in \mathbb{N}^n}
\vert{\l}{1}{\frac{x}{y};q}{\bm{e}_{a_2}}{\AA}{\bm{e}_{c_2}}{\CC}
\vert{\l}{1}{\frac{x}{z};q}{\bm{e}_{a_3}}{\CC}{\bm{e}_{c_3}}{\BB}
R_{z/y}(c_3,c_2;b_3,b_2)
\\
=
\sum_{c_2,c_3 = 0}^{n}\
\sum_{\CC \in \mathbb{N}^n}
R_{z/y}(a_3,a_2;c_3,c_2)
\vert{\l}{1}{\frac{x}{z};q}{\bm{e}_{c_3}}{\AA}{\bm{e}_{b_3}}{\CC}
\vert{\l}{1}{\frac{x}{y};q}{\bm{e}_{c_2}}{\CC}{\bm{e}_{b_2}}{\BB},
\end{multline*}
where $a_2,a_3,b_2,b_3 \in \{0,1,\dots,n\}$ and $\AA,\BB \in \mathbb{N}^n$ are fixed to arbitrary values on both sides of the equation. Applying everywhere the symmetry \eqref{bm-sym}, and changing variables $y \mapsto \b{q} \b{y}$, $z \mapsto \b{q} \b{z}$, we obtain
\begin{multline*}
\sum_{c_2,c_3 = 0}^{n}\
\sum_{\CC \in \mathbb{N}^n}
\vert{1}{\l}{\frac{q^{-\l}}{xy};\frac{1}{q}}{\AA}{\bm{e}_{a_2}}{\CC}{\bm{e}_{c_2}}
\vert{1}{\l}{\frac{q^{-\l}}{xz};\frac{1}{q}}{\CC}{\bm{e}_{a_3}}{\BB}{\bm{e}_{c_3}}
R_{y/z}(c_3,c_2;b_3,b_2)
\\
=
\sum_{c_2,c_3 = 0}^{n}\
\sum_{\CC \in \mathbb{N}^n}
R_{y/z}(a_3,a_2;c_3,c_2)
\vert{1}{\l}{\frac{q^{-\l}}{xz};\frac{1}{q}}{\AA}{\bm{e}_{c_3}}{\CC}{\bm{e}_{b_3}}
\vert{1}{\l}{\frac{q^{-\l}}{xy};\frac{1}{q}}{\CC}{\bm{e}_{c_2}}{\BB}{\bm{e}_{b_2}}.
\end{multline*}
Substituting $x \mapsto s$ and sending $q^{\l} \mapsto s^{-2}$, using \eqref{def-Mtilde} we find that
\begin{multline}
\label{RLL-C}
\sum_{c_2,c_3 = 0}^{n}\
\sum_{\CC \in \mathbb{N}^n}
\tilde{M}_y(\AA,a_2;\CC,c_2)
\tilde{M}_z(\CC,a_3;\BB,c_3)
R_{y/z}(c_3,c_2;b_3,b_2)
\\
=
\sum_{c_2,c_3 = 0}^{n}\
\sum_{\CC \in \mathbb{N}^n}
R_{y/z}(a_3,a_2;c_3,c_2)
\tilde{M}_z(\AA,c_3;\CC,b_3)
\tilde{M}_y(\CC,c_2;\BB,b_2),
\end{multline}
which is the final intertwining equation \eqref{RLLc}, for the stochastic weights \eqref{stoch-wt}.

\section{Gauge transformations}

In the previous sections we have shown how three types of intertwining equations, namely \eqref{RLL-A2}, \eqref{RLL-B3} and \eqref{RLL-C}, arise as specializations of the master higher-spin Yang--Baxter equation \eqref{master-yb}. These relations are almost identically the same as \eqref{RLLa}, \eqref{RLLb} and \eqref{RLLc}, except that the Boltzmann weights which appear in \eqref{RLL-A2}, \eqref{RLL-B3} and \eqref{RLL-C} are stochastic. To go from the stochastic version of the weights to the non-stochastic one, one defines
\begin{align}
\label{transf-wt}
L_x(\I,j;\K,\ell)
:=
(-s)^{-\bm{1}_{\ell \geq 1}}
\tilde{L}_x(\I,j;\K,\ell),
\qquad
M_x(\I,j;\K,\ell)
:=
(-s)^{\bm{1}_{j \geq 1}}
\tilde{M}_x(\I,j;\K,\ell).
\end{align}
It is easy to see that the transformed weights \eqref{transf-wt} are still integrable; for example, substituting $\tilde{L}_x$ and $\tilde{L}_y$ in \eqref{RLL-A2} by their non-stochastic counterparts yields
\begin{multline}
\label{RLL-A3}
(-s)^{\bm{1}_{b_1 \geq 1}+\bm{1}_{b_2 \geq 1}}
\sum_{c_1,c_2 = 0}^{n}\
\sum_{\CC \in \mathbb{N}^n}
R_{y/x}(a_2,a_1;c_2,c_1)
L_x(\AA,c_1;\CC,b_1)
L_y(\CC,c_2;\BB,b_2)
\\
=
\sum_{c_1,c_2 = 0}^{n}\
\sum_{\CC \in \mathbb{N}^n}
(-s)^{\bm{1}_{c_1 \geq 1}+\bm{1}_{c_2 \geq 1}}
L_y(\AA,a_2;\CC,c_2)
L_x(\CC,a_1;\BB,c_1)
R_{y/x}(c_2,c_1;b_2,b_1)
\\
=
(-s)^{\bm{1}_{b_1 \geq 1}+\bm{1}_{b_2 \geq 1}}
\sum_{c_1,c_2 = 0}^{n}\
\sum_{\CC \in \mathbb{N}^n}
L_y(\AA,a_2;\CC,c_2)
L_x(\CC,a_1;\BB,c_1)
R_{y/x}(c_2,c_1;b_2,b_1),
\end{multline}
where the final equality follows from the colour-conservation property of the $R$-matrix \eqref{R-vert}, namely,
\begin{align*}
\prod_{b=0}^{n}
\alpha_b^{\bm{1}_{i = b} + \bm{1}_{j = b}}
R_z(i,j;k,\ell)
=
\prod_{b=0}^{n}
\alpha_b^{\bm{1}_{k = b} + \bm{1}_{\ell = b}}
R_z(i,j;k,\ell),
\end{align*}
for arbitrary parameters $\{\alpha_0,\alpha_1,\dots,\alpha_n\}$ and $i,j,k,\ell \in \{0,1,\dots,n\}$, by choosing $\alpha_0 = 1$ and $\alpha_1,\dots,\alpha_n = -s$. Equation \eqref{RLL-A3} matches \eqref{RLLa} after deleting the spurious multiplicative factors. The derivation of \eqref{RLLb} and \eqref{RLLc} is analogous.

\end{appendix}

\printindex

\bibliographystyle{alpha}
\bibliography{references}

\end{document}